\documentclass[11pt,twoside]{article} 
\usepackage[utf8]{inputenc} 

\usepackage[margin=1.25in]{geometry} 
\usepackage{graphicx} 

\usepackage{booktabs} 
\usepackage{array} 
\usepackage{paralist} 
\usepackage{verbatim} 
\usepackage{mathrsfs} 
\usepackage{amssymb}
\usepackage{amsthm}
\usepackage{amsmath,amsfonts,amssymb}
\usepackage{esint}
\usepackage{graphics}
\usepackage{enumerate}
\usepackage{mathtools}
\usepackage{xfrac}
\usepackage{bbm}
\usepackage{subcaption}
\usepackage{stmaryrd} 
\usepackage[nice]{nicefrac}

\let\oldsquare\square 

\usepackage{mathabx}

\renewcommand{\square}{\oldsquare}

\usepackage{tocloft}
 
\usepackage[usenames,dvipsnames]{xcolor}
\usepackage[colorlinks=true, pdfstartview=FitV, linkcolor=blue, citecolor=blue, urlcolor=blue]{hyperref}
\usepackage[normalem]{ulem}

\numberwithin{equation}{section}

\newtheorem{theorem}{Theorem}[section]
\newtheorem{conjecture}[theorem]{Conjecture}
\newtheorem{corollary}[theorem]{Corollary}
\newtheorem{proposition}[theorem]{Proposition}
\newtheorem{lemma}[theorem]{Lemma}
\theoremstyle{definition}
\newtheorem{definition}[theorem]{Definition}
\newtheorem{exercise}[theorem]{Exercise}
\newtheorem{remark}[theorem]{Remark}
\newtheorem{example}[theorem]{Example}

\newcommand{\dataref}{\hyperref[e.data.def]{\textcolor{blue}{\ensuremath{\mathrm{data}}}}}
\newcommand{\datareff}{\hyperref[e.data.def.f]{\textcolor{blue}{\ensuremath{\mathrm{data}_{\mathbf{f}}}}}}

\let\originalleft\left
\let\originalright\right
\renewcommand{\left}{\mathopen{}\mathclose\bgroup\originalleft}
\renewcommand{\right}{\aftergroup\egroup\originalright}

\newcommand{\vertiii}{\vert\kern-0.3ex\vert\kern-0.25ex\vert}

\newcommand{\norm}[1]{\left\|#1\right\|}

\newcommand*{\Id}{\ensuremath{\mathrm{I}_d}}
\newcommand*{\Itwod}{\ensuremath{\mathrm{I}_{2d}}}

\newcommand*{\N}{\ensuremath{\mathbb{N}}}
\newcommand*{\Z}{\ensuremath{\mathbb{Z}}}
\newcommand*{\R}{\ensuremath{\mathbb{R}}}
\newcommand*{\Zd}{\ensuremath{\mathbb{Z}^d}}
\newcommand*{\Rd}{\ensuremath{\mathbb{R}^d}}
\newcommand{\eps}{\varepsilon}

\renewcommand*{\tilde}{\widetilde}
\renewcommand*{\hat}{\widehat}

\renewcommand{\b}{\ensuremath{\mathbf{b}}}
\newcommand{\qand}{\quad \mbox{and} \quad }
\newcommand{\per}{\mathrm{per}}
\newcommand{\sym}{\mathrm{sym}}
\renewcommand{\skew}{\mathrm{skew}}

\newcommand{\g}{\mathbf{g}}
\newcommand{\f}{\mathbf{f}}
\newcommand{\s}{\mathbf{s}}
\NewDocumentCommand{\bfs}{e{^_}}{{\boldsymbol{\sigma}}\IfValueT{#1}{^{#1}}\IfValueT{#2}{_{\!#2}}}
\newcommand{\ep}{\eps}

\newcommand{\pot}{\mathrm{pot}}
\newcommand{\sol}{\mathrm{sol}}

\newcommand{\inv}{\mathrm{inv}}

\DeclareSymbolFont{boldoperators}{OT1}{cmr}{bx}{n}
\SetSymbolFont{boldoperators}{bold}{OT1}{cmr}{bx}{n}
\renewcommand{\a}{\mathbf{a}}

\renewcommand{\k}{\mathbf{k}}
\newcommand{\ahom}{\bar{\a}}
\newcommand{\bhom}{\bar{\mathbf{b}}}
\newcommand{\shom}{\bar{\mathbf{s}}}

\newcommand{\khom}{\bar{\mathbf{k}}}

\newcommand{\bfzero}{\boldsymbol{0}}

\newcommand{\cu}{\square}

\newcommand{\F}{\mathcal{F}}
\renewcommand{\P}{\mathbb{P}}
\newcommand{\E}{\mathbb{E}}
\newcommand{\X}{\mathcal{X}}
\renewcommand{\S}{\mathcal{S}}

\renewcommand{\O}{\mathcal{O}}

\newcommand{\Y}{\ensuremath{\mathcal{Y}}}

\newcommand{\indc}{\boldsymbol{1}}

\newcommand{\data}{\mathrm{data}}
\newcommand{\minscale}{\chi}

\DeclareMathOperator{\dist}{dist}

\DeclareMathOperator*{\osc}{osc}
\DeclareMathOperator{\var}{var}
\DeclareMathOperator{\cov}{cov}
\DeclareMathOperator{\diam}{diam}

\DeclareMathOperator{\supp}{supp}
\DeclareMathOperator{\spn}{span}

\newcommand{\avsum}{\mathop{\mathpalette\avsuminner\relax}\displaylimits}

\makeatletter
\newcommand\avsuminner[2]{%
  {\sbox0{$\m@th#1\sum$}%
   \vphantom{\usebox0}%
   \ooalign{%
     \hidewidth
     \smash{\,\rule[.23em]{8.8pt}{1.1pt} \relax}%
     \hidewidth\cr
   ~$\m@th#1\sum$\cr
   }%
  }%
}
\makeatother

\newcommand{\avsumtext}{\mathop{\mathpalette\avsuminnerr\relax}\displaylimits}

\makeatletter
\newcommand\avsuminnerr[2]{%
  {\sbox0{$\m@th#1\sum$}%
   \vphantom{\usebox0}%
   \ooalign{%
     \hidewidth
     \smash{\,\rule[.23em]{6pt}{0.7pt} \relax}%
     \hidewidth\cr
   ~$\m@th#1\sum$\cr
   }%
  }%
}
\makeatother

\def\Xint#1{\mathchoice
{\XXint\displaystyle\textstyle{#1}}%
{\XXint\textstyle\scriptstyle{#1}}%
{\XXint\scriptstyle\scriptscriptstyle{#1}}%
{\XXint\scriptscriptstyle\scriptscriptstyle{#1}}%
\!\int}
\def\XXint#1#2#3{{\setbox0=\hbox{$#1{#2#3}{\int}$}
\vcenter{\hbox{$#2#3$}}\kern-.5\wd0}}

\def\fint{\Xint-}

\makeatletter 
\newcommand{\negphantom}{\v@true\h@true\negph@nt} 
\newcommand{\neghphantom}{\v@false\h@true\negph@nt} 
\newcommand{\negph@nt}{\ifmmode\expandafter\mathpalette 
  \expandafter\mathnegph@nt\else\expandafter\makenegph@nt\fi} 
\newcommand{\makenegph@nt}[1]{%
  \setbox\z@\hbox{\color@begingroup#1\color@endgroup}\finnegph@nt} 
\newcommand{\finnegph@nt}{%
  \setbox\tw@\null 
  \ifv@ \ht\tw@\ht\z@\dp\tw@\dp\z@\fi \ifh@\wd\tw@-\wd\z@\fi\box\tw@} 
\newcommand{\mathnegph@nt}[2]{%
  \setbox\z@\hbox{$\m@th #1{#2}$}\finnegph@nt} 
\makeatother

\newcommand{\Hminusul}{\hat{\phantom{H}}\negphantom{H}\underline{H}^{-1}}

\newcommand{\Hminusulk}{\hat{\phantom{H}}\negphantom{H}\underline{H}^{-k}}

\newcommand{\Ent}{\mathsf{Ent}}
\newcommand{\CFS}{\mathsf{CFS}}
\newcommand{\FRD}{\mathsf{FRD}}
\newcommand{\LSI}{\mathsf{LSI}}
\newcommand{\SG}{\mathsf{SG}}
\newcommand{\AFRD}{\mathsf{AFRD}}
\newcommand{\A}{\mathcal{A}}
\newcommand{\Ahom}{\bar{\A}} 

\newcommand{\Lsol}{L^2_{\mathrm{sol}}} 
\newcommand{\Lsolo}{L^2_{\mathrm{sol,0}}} 
\newcommand{\Lpot}{L^2_{\mathrm{pot}}} 
\newcommand{\Lpoto}{L^2_{\mathrm{pot,0}}} 

\newcommand{\bfA}{\mathbf{A}}
\newcommand{\bfAhom}{\overline{\mathbf{A}}}
\NewDocumentCommand{\bfJ}{e{^_}}{{\mathbf{J}}\IfValueT{#1}{^{#1}}\IfValueT{#2}{_{\!#2}}}
\newcommand{\bfF}{\mathbf{F}}
\newcommand{\rota}{\mathbf{R}}

\usepackage{titlesec}

\newcommand{\addperiod}[1]{#1.}
\titleformat{\section}
   {\normalfont\Large}{\thesection.}{0.5em}{}
\titleformat*{\subsection}{\normalfont\large}
\titleformat{\subsubsection}[runin]
  {\bfseries}
  {\thesubsubsection.}
  {0.5em}
  {\addperiod}
\titleformat*{\subsubsection}{\bfseries}
\titleformat*{\paragraph}{\bfseries}
\titleformat*{\subparagraph}{\large\bfseries}

\usepackage{fancyhdr}
\newcommand{\sectiontitle}{}
\newcommand{\subsectiontitle}{}

\title{\bf \Large Elliptic homogenization from qualitative to quantitative}

\author{Scott Armstrong
\thanks{Courant Institute of Mathematical Sciences, New York University.
{\footnotesize \href{mailto:scotta@cims.nyu.edu}{scotta@cims.nyu.edu}.}
}
\and 
Tuomo Kuusi
\thanks{Department of Mathematics and Statistics, University of Helsinki.
{\footnotesize \href{mailto:tuomo.kuusi@helsinki.fi}{tuomo.kuusi@helsinki.fi}.}
}
}
\date{\today}

\usepackage[nottoc,notlot,notlof]{tocbibind}

\begin{document}

\maketitle

\begin{abstract}
We give a self-contained introduction to the theory of elliptic homogenization for random coefficient fields, starting from classical qualitative homogenization. The presentation also contains new results, such as optimal estimates (both in terms of stochastic moments and scaling of the error) for coefficient fields which are local functions of Gaussian random fields.  
\end{abstract}

\setcounter{tocdepth}{2}  
\tableofcontents

\setcounter{section}{-1}

\section{Preface}
  
The purpose of this manuscript is to present the theory of elliptic homogenization in a self-contained manner, starting from the basic qualitative theory before motivating and building the quantitative theory. We hope that graduate students and other non-experts will find an accessible text that can, in finite time, bring them up to speed on the main ideas and techniques used in the subfield, allowing them to read the latest research articles. We assume a background in basic analysis and probability theory; a knowledge of elliptic PDE cannot hurt but is not required. 

\smallskip

Homogenization is a topic at the interface of mathematical analysis, PDEs, probability and statistical physics. It is about inventing analytic tools, often using PDE methods, for studying disordered systems exhibiting behavior on several very different length scales.
Applications of homogenization techniques are potentially vast, but in most cases, what is needed is more than just the soft asymptotic limits given by the classical homogenization theory built in the 1970s and 80s. What is needed is a quantitative theory that is more robust to perturbations in the assumptions. We want to apply our ideas not just to one specific model; rather, we want to develop methods that can be used on many different models that share an underlying structure. 
This motivated the need for a quantitative theory of stochastic homogenization, which has been rapidly developed in the last decade. There is now a large number of long, very technical articles. This literature is difficult for outsiders to penetrate, so we have written this manuscript in an attempt to make the theory more accessible.
 
 \smallskip
 
The beginning starts very gently in Chapter~\ref{s.qualitativetheory}, where we cover the essentials of qualitative homogenization, starting from the periodic case before giving two complete proofs of stochastic homogenization for stationary coefficient fields. We also include complete proofs of the multiparameter ergodic theorems which are used. 
In Chapter~\ref{s.probability}, we review the probabilistic interpretation of homogenization, giving two different proofs of the central limit theorem for symmetric diffusions in random environments and explaining how this is equivalent to homogenization of the corresponding PDEs. 

\smallskip

The rest of the text is about the quantitative theory. We have made a considerable effort to compress the proofs and simplify the presentation as much as possible. We had the same goal in writing our previous book~\cite{AKMBook}, which is only a few years old, but we discovered in writing this text that a great deal of improvements were left to be made. 
In order to make our results more quotable, we formalized our quantitative theory under a new and very general mixing condition we call \emph{concentration for sums}. In Chapter~\ref{s.CFS} we introduce this mixing condition and give many examples of random environments satisfying it (this includes many often-used assumptions like finite range dependence or spectral gap-type concentration inequalities). 
The core of the text consists of Chapters~\ref{s.subadd},~\ref{s.nonsymm},~\ref{s.regularity}, and the first five sections of Chapter~\ref{s.renormalization}, which contains a proof of the optimal quantitative estimates for the first-order correctors for general uniformly elliptic coefficient fields. 

\smallskip

It will come as no surprise that we have chosen to emphasize the coarse-graining and renormalization approach to quantitative homogenization, which has been largely pioneered by ourselves and our collaborators Jean-Christophe Mourrat and Charles Smart. As we show here, these methods give better estimates in terms of stochastic moments, even on the most basic examples like a random checkerboard or a Gaussian random field with power-like decay of correlations. In fact, for a general class of Gaussian fields we obtain quantitative estimates which are sharp both in terms of the scaling of the error and stochastic integrability, a result appearing here for the first time. However, we have also incorporated the ideas of Antoine Gloria \& Felix Otto and their collaborators into this text by showing how nonlinear concentration inequalities may be used as an alternative to getting corrector bounds. 

\smallskip

We will eventually include more material and have specific plans to add chapters on, for example, scaling limits for the correctors and numerical methods for computing homogenized coefficients. In the future, we will treat this like a software project by continually fixing typos, improving the presentation, adding chapters, etc. Readers will always find the latest version on our webpages and in this GitHub repository:
\begin{center}
{\tt https://github.com/scottnarmstrong/Homogenization}
\end{center}

\smallskip

This manuscript began as a set of notes to accompany a course given by the first author at the summer school ``Journ\'ees \'equations aux d\'eriv\'ees partielles" in Obernai, France, in June 2019. It took quite a bit longer than expected to complete. 

\smallskip

The first author was partially supported by NSF grant DMS-2000200. The second author was supported by the Academy of Finland and the European Research Council (ERC) under the European Union's Horizon 2020 research and innovation programme (grant agreement No 818437).

\subsection*{Notation}

Denote~$\N:=\{0,1,2,\ldots\}$. The set of real numbers is written~$\R$. When we write~$\R^m$ we implicitly assume that~$m\in\N\setminus\{0\}$. For each~$x,y\in\R^m$, the scalar product of~$x$ and~$y$ is denoted by~$x\cdot y$ and the Euclidean norm on~$\R^m$ is~$\left| \,\cdot\, \right|$. The canonical basis of~$\R^m$ is written as~$\{e_1,\ldots,e_m\}$. A domain is an open and connected subset of~$\R^m$.  The set of~$m\times n$ matrices with real entries is denoted by~$\R^{m\times n}$. The identity matrix is denoted~$\Id$. The transpose of a matrix~$A \in \R^{m\times n}$ is denoted by~$A^t$. The sets of symmetric and anti-symmetric~$n$-by-$n$ matrices are denoted, respectively, by
\begin{equation}
\R_{\mathrm{sym}}^{n \times n} := \{ A \in \R^{n\times n} \,:\, A = A^t\}
\quad \mbox{and} \quad 
\R_{\mathrm{skew}}^{n \times n} := \{ A \in \R^{n\times n} \,:\, A = - A^t\}
\end{equation}
We use the usual Loewner ordering for symmetric matrices. If~$A,B\in \R_{\mathrm{sym}}^{n \times n}$, the inequality~$A \leq B$ means that~$B-A$ is nonnegative definite; equivalently,~$e \cdot (B-A)e \geq 0$ for every~$e\in\R^n$. 
If~$r,s\in \R$, then we denote~$r\vee s := \max\{ r,s \}$ and~$r\wedge s := \min\{r,s\}$. If~$S$ is a finite set, then~$|S|$ denotes the cardinality of~$S$ and we write
\begin{equation}
\label{e.slashedsum}
\avsum_{i \in S} a_i := \frac{1}{|S|} \sum_{i\in S} a_i 
\,.
\end{equation}

\smallskip

The boundary of a subset~$U\subseteq \R^m$ is denoted by~$\partial U$ and its closure by~$\overline{U}$. A domain is a \emph{Lipschitz} domain if the boundary is locally a graph of a Lipschitz function. A~$C^{k,\alpha}$ domain is defined analogously.  The open ball of radius~$r>0$ centered at~$x\in\R^m$ is denoted by~$B_r(x):= \left\{ y\in\R^m\,:\, |x-y|<r \right\}$. The distance from a point to a set~$V\subseteq \R^m$ is written~$\dist(x,V):= \inf\left\{ |x-y|\,:\, y\in V\right\}$ and between sets,~$\dist(U,V):= \inf\left\{ |x-y|\,:\, x \in U,\, y\in V\right\}$.  For~$t >0$, we set~$t U:= \left\{ t x \,:\, x\in U\right\}$. Triadic cubes are also used throughout the notes. For each~$m\in\N$, we denote 
\begin{equation} 
\label{e.triadiccubes}
\cu_m := \Bigl( -\frac12 3^m, \frac12 3^m \Bigr)^d\subseteq \Rd. 
\end{equation}

If~$U\subseteq\Rd$ and~$f: U \to \R$, the partial derivatives of~$f$ are denoted by~$\partial_{x_i} f$. The gradient of~$f$ is~$\nabla f:= \left( \partial_{x_1} f ,\ldots, \partial_{x_d} f \right)$. The Hessian of~$f$ is denoted by~$\nabla^2f:= ( \partial_{x_i} \partial_{x_j} f )_{i,j\in\{1,\ldots,d\}}$, and higher derivatives are denoted similarly:~$\nabla^k f := ( \partial_{x_{i_1}}\cdots\partial_{x_{i_k}} f )_{i_1,\ldots,i_k\in\{1,\ldots,d\}}$. The divergence of a vector field~$\f$ is~$\nabla \cdot \f = \sum_{k=1}^d \partial_{x_i} f_i$, where~$(f_i)$ are the entries of~$\f$. 

\smallskip

Let~$C^k(U)$,~$k\in\N\cup\{\infty\}$, be the collection of~$k$ times continuously differentiable functions. Let~$C_0^k(U)$ be a collection of~$C^k(U)$ functions vanishing on the boundary. For~$k\in\N$ and~$\alpha \in (0,1]$, we denote by~$C^{k,\alpha}(U)$ classical H\"older spaces, which are the functions~$u\in C^k(U)$ for which the norm 
\begin{equation*} \label{}
\left\| u \right\|_{C^{k,\alpha}(U)} :=
\bigl[\nabla^ku\bigr]_{C^{0,\alpha}(U)}
+
\sum_{n=0}^k \sup_{x\in U} \, \left| \nabla^n u(x) \right| 
\,,
\quad
\left[ f \right]_{C^{0,\alpha}(U)} 
:=  \sup_{x,y\in U,\,x\neq y} \frac{ \left|  f(x) - f(y) \right|}{|x-y|^{\alpha}}\,,
\end{equation*}
is finite.

\smallskip

By~$|U|$, we denote the Lebesgue measure of a measurable set~$U$. We often suppress the variable from the integration, and denote, for an integrable function~$f: U \to \R$,
\begin{equation*} \label{}
\int_U f:=  \int_U f(x)\,dx. 
\end{equation*}
For~$U\subseteq \Rd$ and~$p\in[1,\infty]$,~$L^p(U)$ stands for the standard Lebesgue space on~$U$ with exponent~$p$, with the norm
\begin{equation*}  
\| f \|_{L^p(U)} := \Bigl( \int_U | f |^p \Bigr)^{\!\nicefrac1p} \,.
\end{equation*}
For the~$\R^m$-valued functions, whose components are all in~$L^p(U)$, we denote by~$L^p(U;\R^m)$. 
By the space~$L_{\mathrm{loc}}^p(U)$ we mean that~$f$ belongs to~$L_{\mathrm{loc}}^p(U)$ if and only if~$f \in L^p(V)$ whenever~$V$ is a bounded subset of~$U$ such that the closure of~$V$ belongs to the interior of~$U$. If~$0< |U| < \infty$ and~$f \in L^1(U)$, set
\begin{equation*}  
\left( f \right)_U = \fint_U f := \frac 1 {|U|} \int_U f.
\end{equation*}
For keeping track of volume factors in our computations, it is convenient to use the following volume-normalized~$L^p$ norms:
\begin{equation} 
\label{e.volume.normalized}
\|f\|_{\underline{L}^p(U)} := \biggl( \fint_U |f|^p \biggr) ^{\!\nicefrac 1 p} = | U |^{-\nicefrac 1p} \left\| f \right\|_{L^p(U)}.
\end{equation}
For~$f\in L^p(\Rd)$ and~$g\in L^{p'}(\Rd)$ with~$\frac1p+\frac1{p'}=1$, we denote the convolution of~$f$ and~$g$ by
$
( f\ast g )(x) := \int_{\Rd} f(x-y)g(y)\,dy.
$

\smallskip

Next, for~$x,p\in\Rd$, we denote~$\ell_p(x):= p\cdot x$.  The heat kernel is denoted by
\begin{equation*} \label{}
\Phi(t,x):= \left( 4\pi t\right)^{-\nicefrac d2} \exp\left( -\frac{|x|^2}{4t} \right)
\end{equation*}
and, for each~$z\in\Rd$ and~$r>0$, we denote
\begin{equation*} \label{}
\Phi_{z,r} (x):= \Phi(r^2, x-z) \quad \mbox{and} \quad \Phi_r := \Phi_{0,r}. 
\end{equation*}
In the context of the heat kernels, we use the following notation, for~$f \in L_{\mathrm{loc}}^p(\R^d)$, 
\begin{equation} 
\label{e.L2.Psi}
\| f  \|_{L^p(\Phi_{z,r})} 
:=
\biggl( \int_{\R^d} |f(x)|^p \, \Phi( r^2, x-z ) \, dx \biggr)^{\!\nicefrac1p}
\,.
\end{equation}

We develop the theory using the following notation for Sobolev spaces. For every~$U\subseteq\Rd$ with~$|U|<\infty$,~$k\in\N$ and~$p\in [1,\infty)$, we define the following volume-normalized norms:
\begin{equation}
\label{e.underlinednorms}
\left\{ 
\begin{aligned}
&
\left\| f \right\|_{\underline{W}^{k,p}(U)} 
:=
\sum_{j=0}^k
|U|^{\nicefrac {(j{-}k)}d} \left\| \nabla^j f \right\|_{\underline{L}^p(U)},
\\ & 
\left\| f \right\|_{\underline{W}^{-k,p}(U)} 
:=
\sup\left\{ 
\fint_U fg \,:\, 
g\in W^{k,p}_0(U), \, 
\left\| g \right\|_{\underline{W}^{k,p}(U)} 
\leq 1 \right\}, 
\\ & 
\left\| f \right\|_{\hat{\underline{W}}^{-k,p}(U)} 
:=
\sup\left\{ 
\fint_U fg \,:\, 
g\in W^{k,p}(U), \, 
\left\| g \right\|_{\underline{W}^{k,p}(U)} 
\leq 1 \right\} \,.
\end{aligned}
\right.
\end{equation}
We also denote~$\left\| \cdot \right\|_{\underline{H}^k(U)}:= \left\| \cdot \right\|_{\underline{W}^{k,2}(U)}$,~$\left\| \cdot \right\|_{\underline{H}^{-k}(U)}:= \left\| \cdot \right\|_{\underline{W}^{-k,2}(U)}$ and~$\left\| \cdot \right\|_{\Hminusulk(U)}:= \left\| \cdot \right\|_{\hat{\underline{W}}^{-k,2}(U)}$.
Observe that for every~$t>0$, 
\begin{equation} 
\label{e.Sobolev.scaling} 
\left\| u\left( \tfrac{\cdot}{t} \right) \right\|_{\underline{W}^{k,p}(t U)}  
= t^{-k} \left\| u \right\|_{\underline{W}^{k,p}(U)} 
\qand
\left\| u \left( \tfrac{\cdot}{t} \right) \right\|_{\underline{W}^{-k,p}(t U)} = t^{k} \left\| u  \right\|_{\underline{W}^{-k,p}(U)}
\,.
\end{equation}
The function space~$W^{k,p}(U)$,~$k\in \Z$, is a collection of functions having bounded~$\| \cdot\|_{\underline{W}^{k,p}(U)}$ norm. By the classical theory,~$W^{k,p}(U)$ is a closure of~$C^\infty(U)$ under~$\| \cdot\|_{\underline{W}^{k,p}(U)}$. Analogously, the space~$W_0^{k,p}(U)$ is the closure of~$C_{\mathrm{c}}^\infty(U)$ under~$\| \cdot\|_{\underline{W}^{k,p}(U)}$, where~$C_{\mathrm{c}}^\infty(U)$ is the space of smooth function compactly supported in~$U$. 
The notation~$W^{k,p}_{\mathrm{loc}}(U)$ and~$H^{k}_{\mathrm{loc}}(U)$ is used analogously to the definition of~$L^p_{\mathrm{loc}}(U)$.  

\smallskip

Furthermore, by~$\Lpot(U)$ and~$\Lsol(U)$ we denote the closed subspaces of~$L^2(U;\Rd)$ 
of \emph{potentials} (gradients) and \emph{solenoidal} (divergence-free) vector fields. Set 
\begin{equation}
\label{e.Lpot.not}
\Lpot(U):= \left\{ \nabla u \,:\,  u \in H^1(U) \right\}
\end{equation}
and
\begin{equation}
\label{e.Ls.not}
\Lsol(U):= \left\{ \g \in L^2(U;\Rd) \,:\,  \forall \phi\in H^1_0(U)  \,, \ \int_U \g\cdot \nabla \phi = 0 \right\}.
\end{equation}
The condition in the definition of~$\Lsol(U)$ just means~$\nabla \cdot \g = 0$ in the sense of distributions.  We also define 
\begin{equation} 
\label{e.def.Lp0}
\Lpoto(U):= \left\{ \nabla u \,:\, u\in H^1_0(U) \right\}
\end{equation}
and
\begin{equation}
\label{e.def.Ls0}
\Lsolo(U):= \left\{ \g \in L^2(U;\Rd) \,:\, \forall \phi\in H^1(U), \ \int_U \g\cdot \nabla \phi = 0 \right\}
\,.
\end{equation}
Given a coefficient field~$\a$ with  and an open subset~$U\subseteq \Rd$, the set of weak solutions of the equation
\begin{equation*} \label{}
-\nabla \cdot \left( \a\nabla u \right)  = 0 \quad \mbox{in} \ U
\end{equation*}
is denoted by 
\begin{equation} \label{e.def.A(U)}
\A(U) := \left\{ u\in H^1(U)\,:\,\a\nabla u \in L^2_{\mathrm{sol}}(U) \right\}.
\end{equation}
Likewise, we denote the solutions of the adjoint equation by 
\begin{equation} \label{e.def.A.t(U)}
\A^*(U) := \left\{ u\in H^1(U)\,:\,\a^t\nabla u \in L^2_{\mathrm{sol}}(U) \right\}.
\end{equation}
and those of the homogenized equation by
\begin{equation} \label{e.def.Ahom(U)}
\Ahom(U) := \left\{ u\in H^1(U)\,:\,\ahom\nabla u \in L^2_{\mathrm{sol}}(U) \right\}.
\end{equation}
For~$k \in \N$, we denote by~$\Ahom_{k}$ the polynomials of degree at most~$k$ in~$\Ahom(\Rd)$. 

\smallskip

Throughout the notes, the symbols~$c$ and~$C$ denote positive constants which may change from line to line. When we declare the dependence of~$C$ and~$c$ on the various parameters, we use parentheses~$(\cdots)$.

\section{Qualitative homogenization}
\label{s.qualitativetheory}

We cover the classical theory of qualitative homogenization for uniformly elliptic,~$\Zd$-stationary coefficient fields. 

\smallskip

We consider a coefficient field~$\a(\cdot)$ which is a map from~$\Rd$ into the set of (possibly nonsymmetric) matrices satisfying, for given constants~$0 < \lambda\leq \Lambda <\infty$, the uniform ellipticity condition
\begin{equation}
\label{e.ue}
e\cdot \a(x) e
\geq \lambda |e|^2 
\qquad \mbox{and} \qquad
e \cdot \a^{-1} (x) e 
\geq \Lambda^{-1}  |e|^2 
\,,
\qquad 
\forall x,e\in\R^d. 
\end{equation}
Notice that the second condition in~\eqref{e.ue} implies that 
\begin{align*}
| \a(x) e |^2
\leq\Lambda e \cdot \a(x) e 
\leq 
\Lambda |e| |\a(x) e| \implies | \a(x) e | \leq \Lambda |e|.
\end{align*}
For symmetric matrices, this uniform upper bound on~$\a(\cdot)$ is equivalent to the second condition in~\eqref{e.ue}; for nonsymmetric matrices, it is weaker in the sense that it implies the latter with a larger constant.\footnote{See also~\eqref{e.ellipticity.nonsymm}, below, which is an equivalent way of writing~\eqref{e.ue}.} 
The condition~\eqref{e.ue} is the more natural one to use for homogenization\footnote{In fact, we think that the second condition~\eqref{e.ue} is the most natural ``uniform ellipticity upper bound'' assumption when working with nonsymmetric coefficients, generally speaking. For instance, regularity estimates starting with Caccioppoli's inequality will have the same dependence in~$\Lambda$ as in the symmetric case if~\eqref{e.ue} is used, but not if it is replaced by the condition~$|\a(x)e| \leq \Lambda |e|$.} to because, as we will see it is closed under homogenization---the effective matrix~$\ahom$ will satisfy the same bound, with the same constants.

\smallskip

The program of homogenization theory in the elliptic setting is to study the \emph{large-scale} behavior of the elliptic operator~$-\nabla \cdot \a\nabla$ under appropriate assumptions on the field~$\a(\cdot)$. 
In the next section, we will suppose that
\begin{equation}
\a(\cdot) \ \  \mbox{is} \ \mbox{$\Zd$--periodic,}
\end{equation}
while in Section~\ref{ss.random}, we will consider more general stationary random coefficient fields.  
It is customary to scale the operator by introducing an artificial parameter~$\ep>0$ and denoting~$\a^\ep:=\a\bigl(\tfrac \cdot \ep \bigr)$. This goal of homogenization can be rephrased as that of understanding the behavior of the operator~$-\nabla \cdot \a^\ep \nabla$ for~$0< \ep \ll 1$.

\subsection{Periodic homogenization: a warmup}
\label{s.periodic}

In this section, we consider periodic coefficient fields to get a feel for the problem. We start with a heuristic, informal derivation of the ``correctors,'' which are functions appearing in an asymptotic ansatz for solutions of the equation and play a fundamental role in the theory.

\subsubsection{Heuristic derivation of the correctors}

A basic idea in perturbation theory is to attempt to obtain an expansion in powers of the perturbation parameter, in this case~$\ep>0$. A customary way to proceed is to ``work backward'' by first trying to guess what the terms in the expansion should be by performing some informal computations. If one can identify the correct terms in the expansion, then the expansion can often be justified rigorously. 

\smallskip

In the setting of periodic, elliptic homogenization, it is natural to look for an expansion for a solution~$u^\ep$ of the equation
\begin{equation}
\label{e.pde}
-\nabla \cdot \a^\ep \nabla u^\ep = 0
\end{equation}
having the form 
\begin{equation}
\label{e.ansatz}
u^\ep(x) 
= u_0\bigl(x,\tfrac x\ep\bigr) + \ep u_1\bigl(x,\tfrac x\ep \bigr) + \ep^2 u_2 \bigl(x,\tfrac x\ep \bigr) + \ldots 
= \sum_{i=0}^\infty \ep^i u_i\bigl(x,\tfrac x\ep \bigr)\,,
\end{equation}
where the functions~$u_i(x,y)$ are~$\Zd$--periodic in the second variable~$y$.
This ansatz for~$u^\ep$ can be plugged into the equation for~$u^\ep$ so that a system of equations for the functions~$u_i$ can be discovered by matching powers of~$\ep$. The function~$u_0$ is discovered in this process to be a function of~$x$ only and shown to satisfy an equation of the form
\begin{equation}
-\nabla \cdot \ahom\nabla u_0 = 0,
\end{equation}
where~$\ahom$ is a constant matrix whose definition is given in terms of~$u_1$ and relies on the solvability of the equation for~$u_2$. The function~$u_i$ for~$i\geq 1$ are characterized in terms of~$\nabla^i u_0$ and a periodic function called \emph{the~$i$th order corrector} which is characterized in terms of an equation with periodic boundary conditions. 

\smallskip

In most expositions of homogenization theory in the mathematics literature, 
this informal, heuristic derivation is followed by a discussion of the difficulties arising in justifying the ansatz~\eqref{e.ansatz} mathematically, aptly referred to as~\emph{the problem of justification}. The informal derivation is then thanked for revealing the corrector equations before being promptly discarded, with the problem of justification approached via an entirely different line of thought. There are a handful of methods that are most commonly used, such as \emph{the oscillating test function method}, \emph{two-scale convergence}, \emph{$G$-convergence or~$H$-convergence},  \emph{the periodic unfolding method}, and so on. 
These methods share the characteristic that they are indirect---because they are based on compactness---and, therefore, lose quantitative information about the rate of homogenization. 
We will not give detailed accounts of them here because we want to explore a different approach. Still, the reader can find a good exposition in many other books on homogenization (see, for instance,~\cite{T,Allaire,CD}).

\smallskip

We will proceed by following a more direct approach that Zhikov and the Russian school have called ``the method of first approximation,'' which is more faithful to the derivation of the corrector equations from the ansatz~\eqref{e.ansatz}. It is also quantitative, extends easily to other situations, and is closer to the original informal argument, leading to the guess for the corrector equations. Basically, we plug the ansatz into the equation and check that the error can be made small. 

\smallskip

We begin with an ansatz like~\eqref{e.ansatz}, but we only keep the~$O(\ep)$ term. The twist is that we are not trying to make an ansatz for an exact solution, only a guess for an \emph{approximate} solution. We give ourselves the task of finding a function of the form
\begin{equation}
\label{e.guess}
w^\ep (x) = u_0(x) + \ep u_1\bigl( x,\tfrac x\ep \bigr), 
\end{equation}
where~$u_1=u_1(x,y)$ is periodic in the~$y$ variable, such that~$w^\ep$ is \emph{almost} a solution of the equation~\eqref{e.pde}. What should ``almost'' mean in this context? Elliptic equations like~\eqref{e.pde} should be understood in the sense of~$H^{-1}$. So we should demand that
\begin{equation}
\label{e.demand}
\limsup_{\ep\to 0}
\left\| \nabla\cdot \a^\ep\nabla w^\ep \right\|_{H^{-1}}  
=0.
\end{equation}
This will imply that~$w^\ep$ is close in the~$H^1$-norm to a true, exact solution of the equation. At this stage, we are  vague about the underlying domain on which~$w^\ep$ lives, but this is unimportant (and can be taken to be~$B_1$ for concreteness). 

\smallskip

So let us plug the guess~\eqref{e.guess} into the operator~$\nabla \cdot \a^\ep\nabla$ and try to get the~$H^{-1}$ norm of the result to be small by making good choices of~$u_0(x)$ and~$u_1(x,y)$. We compute:
\begin{equation}
\nabla w^\ep(x) = 
\nabla u_0(x) 
+ \nabla_y u_1\bigl(x,\tfrac x\ep \bigr)
+ \ep \nabla_xu_1\bigl(x,\tfrac x\ep \bigr).
\end{equation}
The quantity we need to be small in~$H^{-1}$ is therefore
\begin{equation}
\label{e.realansatz}
\nabla \cdot \a^\ep \nabla w^\ep 
=
\nabla \cdot \left( 
\a^\ep \left( 
\nabla u_0(x) 
+ \nabla_y u_1\bigl(x,\tfrac x\ep \bigr)
\right)
+ \ep \a^\ep \nabla_xu_1\bigl(x,\tfrac x\ep \bigr)
\right).
\end{equation}
Since
\begin{equation}
\label{e.Oeps}
\left\| \nabla \cdot \left( \ep \a^\ep \nabla_xu_1\bigl(x,\tfrac x\ep \bigr)
\right) \right\|_{H^{-1}} 
\lesssim
\left\|  \ep \a^\ep \nabla_xu_1\bigl(x,\tfrac x\ep \bigr)
\right\|_{L^2} 
\lesssim O(\ep), 
\end{equation}
we can discard the last term inside the divergence in~\eqref{e.realansatz}. 
Since the variable~$x$ is moving much more slowly than~$\frac x\ep$, to keep a shred of hope alive, it seems that we must demand that
\begin{equation}
\nabla_y \cdot 
\left( 
\a(y) \left( 
\nabla u_0(x) 
+ \nabla_y u_1\bigl(x,y \bigr)
\right)\right)
= 0 \,. 
\end{equation}
From the above equation, we see that the dependence of~$u_1(x,y)$ in the variable~$x$ should involve only~$\nabla u_0(x)$, and it should be linear in its dependence on~$\nabla u_0(x)$. Let us therefore substitute~$e:= \nabla u_0(x)$ and write~$u_1(x,y) =: \phi_e(y)$ and then observe that~$\phi_e$ should satisfy the problem
\begin{equation}
\label{e.corr.prob}
\left\{
\begin{aligned}
&-\nabla \cdot \a \left( e + \nabla \phi_e  \right) = 0 \quad \mbox{in}  \ \R^d, \\
& \phi_e  \ \mbox{is~$\Z^d$--periodic and} \ \langle \phi_e  \rangle = 0. 
\end{aligned}
\right.
\end{equation}
The problem~\eqref{e.corr.prob} possesses a unique solution~$\phi_e\in H^1_{\mathrm{loc}}(\Rd)$, with the condition~$\langle \phi_e  \rangle = 0$ being imposed in order to have uniqueness. The solutions~$\phi_e$ are called the \emph{first-order correctors}. The problem~\eqref{e.corr.prob} is sometimes called \emph{the cell problem} since it is posed on a single period of the coefficient field (one  ``cell'').

\smallskip

We are not done yet, as we have just narrowed our guess for~$w^\ep$ to 
\begin{equation}
\label{e.intro.wep}
w^\ep (x) = u_0(x) + \ep \phi_{\nabla u_0(x)}\bigl(\tfrac x\ep \bigr) 
= 
u_0(x) + \ep \sum_{k=1}^d \partial_{x_k}u_0(x) \phi_{e_k}\bigl(\tfrac x\ep \bigr),
\end{equation}
with the second equality in~\eqref{e.intro.wep} due to the evident linearity of~$e\mapsto \phi_e(y)$. We still must make a good choice of~$u_0$ in order to ensure that~$w^\ep$ is an ``almost-solution.'' We return to~\eqref{e.realansatz}, discard the last term of order~$O(\ep)$ using~\eqref{e.Oeps}, plug in the expression for~$u_1$ and then use the equation for~$\phi_e$ to get:
\begin{align}
\label{e.heurthreelines}
\nabla \cdot\a^\ep \nabla w^\ep 
&
=
\sum_{i,j=1}^d
\partial_{x_i} 
\biggl(
\a_{ij}\bigl(\tfrac x\ep \bigr) 
\biggl(
\partial_{x_j} u_0(x) 
+
\sum_{k=1}^d \partial_{x_k}u_0(x) \partial_{x_j} \phi_{e_k} \bigl(\tfrac x\ep \bigr)
\biggr)
\biggr)
+O(\ep)
\notag \\ &
=
\sum_{i,j=1}^d
\a_{ij}\bigl(\tfrac x\ep \bigr)
\biggl( \partial_{x_i}\partial_{x_j} u_0(x) + 
\sum_{k=1}^d \partial_{x_i}\partial_{x_k}u_0(x) \partial_{x_j} \phi_{e_k} \bigl(\tfrac x\ep \bigr) \biggr)
+O(\ep)
\notag \\ &
=
\sum_{i,j,k=1}^d
\a_{ij}\bigl(\tfrac x\ep \bigr) 
\Bigl( \delta_{jk} 
+ 
\partial_{x_j} \phi_{e_k} \bigl(\tfrac x\ep \bigr) \Bigr)
\partial_{x_i}\partial_{x_k} u_0(x) 
+ O(\ep).
\end{align}
As we send~$\ep \to 0$, the first term on the right side of~\eqref{e.heurthreelines} should converge in~$H^{-1}$ to 
\begin{equation}
\label{e.quan}
\sum_{i,j,k=1}^d
\Bigl\langle \!
\a_{ij}\left( \delta_{jk} 
+ 
\partial_{x_j} \phi_{e_k} \right)
\! \Bigr\rangle \,
\partial_{x_i}\partial_{x_k} u_0(x)\,,
\end{equation}
where the brackets~$\langle \cdot \rangle$ denote the mean of a periodic function.
Indeed, it is clear that we should expect this convergence to hold \emph{weakly} in~$L^2$, provided that~$u_0$ is smooth enough (say~$C^2$). But weak convergence in~$L^2$ is equivalent to strong convergence in~$H^{-1}$ for any bounded sequence of~$L^2$ functions. Therefore, we should expect this convergence to hold locally in~$H^{-1}$. 

\smallskip

We define the \emph{homogenized matrix}~$\ahom\in \R^{d\times d}$ by
\begin{equation}
\ahom_{ij} 
:= 
\sum_{k=1}^d 
\left\langle \a_{ik}\left( \delta_{kj} 
+ 
\partial_{x_k} \phi_{e_j} \right)
\right\rangle.
\end{equation}
The need for the expression in~\eqref{e.quan} to be small in~$H^{-1}$ translates to the condition
\begin{equation}
\label{e.pde.homoge}
-\nabla \cdot \ahom \nabla u_0 = 0. 
\end{equation}
This is the \emph{homogenized equation}. 

\smallskip

It is probably better to write the homogenized matrix in a coordinate-free way by recognizing it as the mean of the flux of the correctors (compare to~\eqref{e.corr.prob}): 
\begin{equation}
\label{e.intro.ahom}
\ahom e
:=
\left\langle \a \left( e+ \nabla \phi_e    \right) \right\rangle, \quad e\in\R^d. 
\end{equation}
The matrix~$\ahom$ is uniformly elliptic and satisfies~\eqref{e.ue}, that is, 
\begin{equation}
\label{e.ahom.ue}
e\cdot \ahom e\geq \lambda |e|^2 
\qquad \mbox{and} \qquad 
e\cdot \ahom^{-1} e\geq \Lambda^{-1} |e|^2 \,,
\qquad 
\forall e\in\R^d. 
\end{equation}
The proof of~\eqref{e.ahom.ue} is left for the reader (or see below in~\eqref{e.ahom.boundsme}).

\smallskip

Let us summarize our conclusions from this heuristic derivation:
\begin{itemize}
\item We should  define the correctors~$\phi_e$ and homogenized matrix~$\ahom$ by~\eqref{e.corr.prob} and~\eqref{e.intro.ahom}, respectively.

\item If we select any~$C^2$ function~$u_0$ satisfying the homogenized equation~\eqref{e.pde.homoge} and define~$w^\ep$ by~\eqref{e.intro.wep}, then~$w^\ep$ should be close in the~$H^1$ norm to a true solution of the original equation~\eqref{e.pde}.

\item The function~$w^\ep$ is clearly close to~$u_0$ itself, at least in the~$L^2$ norm, and therefore we should expect that solutions of~\eqref{e.pde} should be close to those of~\eqref{e.pde.homoge}, at least in~$L^2$. 

\end{itemize}

This clearly allows us to build a huge family of approximate solutions. But does it tell us anything about a particular, exact solution~$u^\ep$ of the equation~\eqref{e.pde}? This is, after all, what we set out to study. 
The answer is yes, because nearly every particular solution~$u^\ep$ will be close in~$H^1$ to one of the~$w^\ep$'s constructed above, because~$-\nabla\cdot \a^\ep\nabla (u^\ep - w^\ep) = \nabla\cdot \a^\ep\nabla w^\ep$ is small in~$H^{-1}$. One needs to choose an appropriate function~$u_0$ to build~$w^\ep$ around, which depends on the particular~$u^\ep$ we are working with. The next section will give a more precise and rigorous demonstration.

\subsubsection{Two-scale expansion computations}

Turning to the problem of justification, we should ask ourselves how rigorous the computations in the previous section are. Can they be easily turned into a proof of the bullet points above, or must we introduce new mathematical ideas to rigorously establish the approximation of solutions of~\eqref{e.pde} by those of~\eqref{e.pde.homoge}? As we check carefully below, the above calculation is essentially already a proof of homogenization. Moreover, this argument is \emph{quantitative}---it will immediately give us a convergence rate for the homogenization limit---as well as \emph{robust}---the method we sketched above is applicable and can be straightforwardly modified to homogenize other versions of~\eqref{e.pde} as well as other types of equations.

Before presenting the rigorous version of the argument above, we introduce the \emph{flux correctors}. 
We  take~$\bfs_e$ to be a \emph{stream matrix} (or \emph{matrix potential}) for the divergence-free vector field
\begin{equation}
\g_e := \a\left( e + \nabla\phi_e   \right) - \ahom e\,,
\end{equation}
which is the difference between the flux of the correctors and their average, the homogenized flux. Such stream matrices are unique up to the specification of a \emph{gauge}, but this does not matter for our purposes. We simply take the function~$\bfs_e \in H^1_{\per}(\R^d;\R^{d\times d}_{\skew})$ to be the periodic function valued in the~$d$-by-$d$ skew-symmetric matrices, with the entry~$\bfs_{e,ij}$ defined for each~$i,j\in\{1,\ldots,d\}$ as the solution of  
\begin{equation}
\label{e.fluxcorrect}
\left\{
\begin{aligned}
&
-\Delta \bfs_{e,ij} = 
\partial_{x_i} \g_{e,j} 
-
\partial_{x_j} \g_{e,i}
\quad \mbox{in}  \ \R^d, \\
& \bfs_{e,ij}  \ \mbox{is~$\Z^d$--periodic and} \ \langle \bfs_{e,ij} \rangle = 0. 
\end{aligned}
\right. 
\end{equation}
One can check by a direct computation that~$\nabla \cdot \bfs_e  = \g_e$. where we define the divergence of a matrix-valued function~$\bfA$ by~$(\nabla \cdot \bfA)_i:= \sum_{j=1}^d \partial_{x_j} \bfA_{ij}$. 
Indeed, apply~$\partial_{x_j}$ to both sides of~\eqref{e.fluxcorrect} and sum over~$j \in \{1,\ldots,d\}$ to see that~$\nabla \cdot \bfs_e$ and~$\g_e$ have the same Laplacian in the sense of distributions. Since they are each mean-zero periodic functions, they must be equal. We carefully note that this relies on the fact that~$\g_e$ has zero mean, which is how the homogenized matrix~$\ahom$ will enter the stage. 

\smallskip

The weak convergence of the fluxes was used implicitly in the heuristic derivation above in the passage from the last line of~\eqref{e.heurthreelines} to~\eqref{e.quan}. The elements of the family~$\{ \bfs_e \}_{e\in\R^d}$ are called the \emph{flux correctors} because they capture this weak convergence in a more convenient form, since roughly speaking we have~$\| \g_e\bigl(\tfrac \cdot \ep\bigr) \|_{H^{-1}} \simeq \ep \left\| \bfs_e\bigl(\tfrac \cdot \ep\bigr) \right\|_{L^2} = O(\ep)$. In other words, we can write the~$H^{-1}$ norm as the~$L^2$ norm of something else, and~$L^2$ spaces are more familiar and easier to work with than spaces of negative regularity. We can, of course, already do this for the gradient fields~$\nabla \phi_e$, and the flux correctors allow us to do it for the fluxes~$\g_e$. 

\smallskip

We introduce the rescaled correctors and flux correctors defined for~$\ep >0$ by
\begin{equation} \label{e.scaledcorrectors}
\phi^\ep_e  :=  \ep \phi_e \bigl(\tfrac \cdot\ep \bigr)
\qand
\bfs_e^\ep  := \ep  \bfs_e \bigl(\tfrac \cdot\ep \bigr).
\end{equation}
The main use of the assumption of periodicity in periodic homogenization is found in the above construction of the correctors~$\phi_e$ and~$\bfs_e$ and the fact that an easy energy estimate yields, for a constant~$C(d,\lambda,\Lambda)<\infty$, 
\begin{equation}
\label{e.correctorsbounded}
\left\| \phi_e \right\|_{H^1((0,1)^d)} 
+
\left\| \bfs_e \right\|_{H^1((0,1)^d)}
\leq 
C|e|.
\end{equation}
This estimate is sometimes referred to informally as ``the boundedness of correctors."
It yields that, for any bounded Lipschitz domain~$U\subseteq\Rd$, there exists~$C(U,d,\lambda,\Lambda)<\infty$ such that, for every~$\ep\in (0,1]$, 
\begin{equation}
\label{e.correctorsbounded2}
\sum_{k=1}^d 
\Bigl( \bigl\| \phi_{e_k}^\ep  \bigr\|_{L^2(U)}  + \bigl\| \bfs_{e_k}^\ep  \bigr\|_{L^2(U)} \Bigr)
\leq 
C \ep
\,.
\end{equation}

We turn to the rigorous justification of the heuristic computations given in the previous subsection, which compare the heterogeneous operator~$-\nabla \cdot \a^\ep\nabla$ to the homogenized operator~$-\nabla\cdot \ahom\nabla$ via the function~$w^\ep$. The proof of the following lemma is essentially a repetition of these heuristic computations, written more formally so that they are no longer just heuristic and rearranged somewhat to take advantage of the flux correctors. Moreover, since the following computation is valid provided that the corrector and flux correctors satisfy the corresponding equations, we give a slightly more general statement since the \emph{same} computation is valid also in more general cases such as \emph{stochastic homogenization}.  Notice, however, that in the periodic case, one gets the quantitative estimate by simply taking correctors from~\eqref{e.corr.prob} and~\eqref{e.fluxcorrect} and then combining~\eqref{e.twoscale.explicit},~\eqref{e.twoscale.explicit2} and~\eqref{e.twoscale.strip} with~\eqref{e.correctorsbounded2}.

\begin{lemma}[Basic two-scale expansion computation]
\label{l.twoscale}
Let~$U\subseteq \Rd$ be open,~$\ep \in \bigl(0,\tfrac12\bigr]$ and~$u\in W^{2,\infty}(U)$. 
Define the function~$w^\ep \in H^1_{\mathrm{loc}}(U)$ by
\begin{equation}
\label{e.twoscale.wep}
w^\ep (x) 
:= 
u(x) + \sum_{k=1}^d \partial_{x_k}u(x) \phi_{e_k}^\ep(x)
\,. 
\end{equation}
Then, we have the identities
\begin{equation}
\label{e.twoscale.flux}
\a^\ep \nabla w^\ep - \ahom\nabla u
=
\sum_{k=1}^d \Bigl( 
\nabla \cdot \left( \bfs^\ep_{e_k}  \partial_{x_k} u  \right)
+
\bigl( - \bfs^\ep_{e_k} + \phi_{e_k}^\ep \a^\ep \bigr) \nabla \partial_{x_k} u
\Bigr)
\end{equation}
and
\begin{align}
\label{e.twoscale.error}
\nabla \cdot \a^\ep \nabla w^\ep 
-
\nabla\cdot \ahom \nabla u
= 
\nabla\cdot \left( \sum_{k=1}^d 
\left(
- \bfs_{e_k}^\ep 
+
\phi_{e_k}^\ep \a^\ep 
\right) \nabla \partial_{x_k} u \right).
\end{align}
\end{lemma}
\begin{proof}
We  compute the difference of the fluxes:
\begin{align*}
\left( \a^\ep \nabla w^\ep - \ahom \nabla u \right)_i
& = 
\sum_{j=1}^d
\left(
\a^\ep_{ij} \partial_{x_j} w^\ep 
- \ahom_{ij} \partial_{x_j} u
\right)
\notag \\ &
=
\sum_{j,k=1}^d
\left( \a^\ep_{ij}
\left( \delta_{jk} + \partial_{x_j} \phi^\ep_{e_k} \right)
- \ahom_{ik}\right)
\partial_{x_k}u
+\!
\sum_{j,k=1}^d 
\a^\ep_{ij} \partial_{x_j}\partial_{x_k} u \phi^\ep_{e_k}
\notag \\ & 
= 
\sum_{j,k=1}^d
\partial_{x_j} \bfs_{e_k,ij}^\ep \partial_{x_k}u
+\!
\sum_{j,k=1}^d 
\a^\ep_{ij} \partial_{x_j}\partial_{x_k} u \phi^\ep_{e_k}
\notag \\ & 
=
\sum_{k=1}^d \Bigl( 
\nabla \cdot \left( \bfs^\ep_{e_k}  \partial_{x_k} u  \right)
+
\bigl( - \bfs^\ep_{e_k} + \phi_{e_k}^\ep \a^\ep \bigr) \nabla \partial_{x_k} u
\Bigr)_i
\,.
\end{align*}
This is~\eqref{e.twoscale.flux}.

\smallskip

Next, we take the divergence of~\eqref{e.twoscale.flux}. By the skew-symmetry of~$\bfs_{e_k}$, we have that 
\begin{align*}  
\nabla \cdot  \left( \nabla \cdot \left( \bfs^\ep_{e_k}  \partial_{x_k} u  \right) \right) = \sum_{i,j=1}^d \partial_{x_i} \partial_{x_j} \left( \bfs_{e_k,ij}^\ep\partial_{x_k}u\right) = 0, 
\end{align*}
and thus, we get that
\begin{align*}
\nabla \cdot 
\left( \a^\ep \nabla w^\ep - \ahom \nabla u \right)
&
= 
\nabla\cdot \left( \sum_{k=1}^d 
\left(
- \bfs_{e_k}^\ep 
+
\phi_{e_k}^\ep \a^\ep 
\right) \nabla \partial_{x_k} u \right)  \,.
\end{align*}
This is~\eqref{e.twoscale.error}, finishing the proof. 
\end{proof}

\begin{corollary} \label{c.twoscale}
Suppose~$U\subseteq\Rd$ is a bounded Lipschitz domain,~$u\in W^{2,\infty}(U)$ and~$\ep >0$. Let~$w^\ep$ be defined by~\eqref{e.twoscale.wep}. 
Then there exists a constant~$C(U,d,\lambda,\Lambda)<\infty$ such that
\begin{align}
\label{e.twoscale.explicit}
\left\|
\nabla \cdot \a^\ep \nabla w^\ep 
-
\nabla\cdot \ahom \nabla u
\right\|_{H^{-1}(U)}  
\leq
C 
\sum_{k=1}^d 
\left( \left\| \bfs_{e_k}^\ep  \right\|_{L^2(U)} {+} \left\| \phi_{e_k}^\ep  \right\|_{L^2(U)}   \right)
\left\| \nabla^2 u \right\|_{L^{\infty}(U)}
\end{align}
and, moreover,  
\begin{align} \label{e.twoscale.explicit2}
\left\| 
\a^\ep \nabla w^\ep - \ahom \nabla u \right\|_{H^{-1}(U)} 
\leq 
C \sum_{k=1}^d 
\left( \left\| \bfs_{e_k}^\ep  \right\|_{L^2(U)} {+} \left\| \phi_{e_k}^\ep  \right\|_{L^2(U)}   \right) 
\left(\left\| \nabla u \right\|_{L^{\infty}(U)}  {+} \left\| \nabla^2 u \right\|_{L^{d}(U)} \right)
\end{align}
and 
\begin{align}
\label{e.twoscale.strip}
\left\| 
w^\ep - u 
\right\|_{L^2(U)} 
+
\left\| 
\nabla w^\ep - \nabla u \right\|_{H^{-1}(U)} 
\leq
C  \sum_{k=1}^d 
\left\| \phi_{e_k}^\ep  \right\|_{L^2(U)} 
\left\| \nabla u \right\|_{L^\infty(U)} \,.
\end{align}
\end{corollary}

\begin{proof}
We first compare~$u$ to~$w^\ep$, which is straightforward: by~\eqref{e.twoscale.wep}, 
\begin{equation*}
\| w^\ep - u  \|_{L^2(U)} 
\leq
\sum_{k=1}^d  
\| \phi_{e_k}^\ep\bigl( \tfrac \cdot \ep\bigr) \|_{L^2(U)}
\| \nabla u \|_{L^\infty(U)}
\,.
\end{equation*}
By the Poincar\'e inequality and a duality argument, there exists a constant~$C(U,d)<\infty$ such that~$\| \nabla f \|_{H^{-1}(U)} \leq C \| f  \|_{L^2(U)}$ for every~$f\in L^2(U)$. Therefore,
\begin{equation*}
\left\| 
\nabla w^\ep - \nabla u \right\|_{H^{-1}(U)} 
\leq
C
\sum_{k=1}^d  
\left\| \phi_{e_k}^\ep\bigl( \tfrac \cdot \ep\bigr) \right\|_{L^2(U)}
\left\| \nabla u \right\|_{L^\infty(U)}
\,.
\end{equation*}
Combining the above two displays yields~\eqref{e.twoscale.strip}.

\smallskip

Next, we deduce~\eqref{e.twoscale.explicit2} from~\eqref{e.twoscale.flux}:
\begin{align*}
\left\| 
\a^\ep \nabla w^\ep {-} \ahom \nabla u \right\|_{H^{{-}1}(U)} 
&
\leq
\sum_{k=1}^d 
\bigl\| 
\nabla \cdot \left( \bfs^\ep_{e_k}  \partial_{x_k} u  \right)
+
\bigl( - \bfs^\ep_{e_k} {+} \phi_{e_k}^\ep \a^\ep \bigr) \nabla \partial_{x_k} u
\bigr\|_{H^{-1}(U)} 
\\  &
\leq
C \! \! \sum_{k=1}^d \!
\left( 
\| \bfs_{e_k}^\ep  \|_{L^2(U)} \| \nabla u \|_{L^{\infty}(U)}  
 {+} 
\bigl( 
\| \bfs_{e_k}^\ep  \|_{L^2(U)} {+} \| \phi_{e_k}^\ep  \|_{L^2(U)}
\bigr)  
\left\| \nabla^2 u \right\|_{L^{d}(U)}  
 \right) 
\,.
\end{align*}
Here we also used the embedding~$\| f \|_{H^{-1}} \leq C \| f \|_{L^{2_\ast}}$ with~$2_\ast := \frac{2d}{d+2}$ for~$d>2$ and any number~$(1,2)$ for~$d=2$ (provided by the Sobolev inequality and duality) and H\"older's inequality.
Finally, the estimate~\eqref{e.twoscale.explicit} then immediately follows from~\eqref{e.twoscale.error} and
\begin{align*}
\left\|
\nabla \cdot \a^\ep \nabla w^\ep 
-
\nabla\cdot \ahom \nabla u
\right\|_{H^{-1}(U)} 
&
\leq
\sum_{k=1}^d 
\left\| 
\left(
- \bfs_{e_k}^\ep 
+
\phi_{e_k}^\ep \a^\ep 
\right) \nabla \partial_{x_k} u 
\right\|_{L^2(U)}
\, .
\end{align*}
The proof is complete. 
\end{proof}

\subsubsection{Homogenization of boundary-value problems}

The following lemma presents a quantitative homogenization result for the Dirichlet problem with smooth data. The proof is to compare~$u^\ep$ with~$w^\ep$, using the previous lemma, and then compare~$w^\ep$ to~$u$.  Here, we formalize and quantify the idea that \emph{$w^\ep$ is almost a solution and, therefore, it is close to the true exact solution~$u^\ep$.}

\begin{lemma}[Quantitative homogenization for Dirichlet problem]
\label{l.DP.2infty}
Suppose that~$U\subseteq\Rd$ is a bounded Lipschitz domain,~$u\in W^{2,\infty}(U)$,~$\ep\in (0,\tfrac12]$ and~$w^\ep\in H^1_{\mathrm{loc}}(\Rd)$ is defined by~\eqref{e.twoscale.wep}. Let~$u^\ep \in H^1(U)$ be the solution of the Dirichlet problem 
\begin{align}
\label{e.eq.DP.2infty}
\left\{
\begin{aligned}
& -\nabla \cdot \a^\ep \nabla u^\ep = - \nabla \cdot \ahom\nabla u & \mbox{in} & \ U, 
\\
& u^\ep = u & \mbox{on} & \ \partial U. 
\end{aligned}
\right.
\end{align}
Then there exists~$C(U,d,\lambda,\Lambda)<\infty$ such that 
\begin{equation}
\label{e.ee.easy}
\left\| \nabla u^\ep - \nabla w^\ep \right\|_{{L}^2(U)} 
\leq
C \ep^{\frac12} \left(  \left\| \nabla u \right\|_{L^\infty(U)}
+
\ep^{\frac12} \left\| \nabla^2 u \right\|_{L^\infty(U)}
\right)
\end{equation}
and, consequently, 
\begin{multline}
\label{e.ee.easy2}
\left\| 
u^\ep - u 
\right\|_{L^2(U)} 
+
\left\| 
\nabla u^\ep - \nabla u \right\|_{H^{-1}(U)} 
+
\left\| 
\a^\ep \nabla u^\ep - \ahom \nabla u \right\|_{H^{-1}(U)}
\\
\leq 
C \ep^{\frac12} \left(  \left\| \nabla u \right\|_{L^\infty(U)}
+
\ep^{\frac12} \left\| \nabla^2 u \right\|_{L^\infty(U)}
\right).
\end{multline}
\end{lemma}
\begin{proof}
By Lemma~\ref{l.twoscale}, the function~$u^\ep - w^\ep$ satisfies the equation
\begin{equation}
\label{e.uepwep}
-\nabla \cdot \a^\ep \nabla (u^\ep - w^\ep) = \nabla \cdot \f^\ep \quad \mbox{in} \ U,
\end{equation}
with
\begin{equation}
\label{e.def.fep}
\f^\ep := \sum_{k=1}^d 
\left(
\bfs_{e_k}^\ep 
-
\phi_{e_k}^\ep \a^\ep 
\right)
\nabla \partial_{x_k} u 
\,.
\end{equation}
Notice that 
\begin{align}
\label{e.beforecorrectors}
\left\| \f^\ep \right\|_{{L}^2(U)} 
& 
\leq 
C \left\| \nabla^2 u \right\|_{L^\infty(U)} 
\sup_{|e| =1} \Bigl( \| \phi^\ep_e \|_{L^2(U)} + \| \bfs_e^\ep \|_{L^2(U)} \Bigr)
\,.
\end{align}
%
The function~$u^\ep - w^\ep$ equals~$u-w^\ep$ on~$\partial U$ (in the sense of the trace). To correct the boundary values, we take~$z^\ep\in H^1(U)$ to satisfy 
\begin{equation*}
\left\{
\begin{aligned}
& -\nabla \cdot \a^\ep \nabla z^\ep = 0 & \mbox{in} & \ U,  \\
& z^\ep = w^\ep - u & \mbox{on} & \ \partial U. 
\end{aligned}
\right.
\end{equation*}
Since~$-\nabla \cdot \a^\ep  \nabla (u^\ep - w^\ep + z^\ep) = \nabla \cdot\f^\ep$ and~$u^\ep - w^\ep + z^\ep\in H^1_0(U)$, we have
\begin{equation}
\label{e.uepwepzep.0}
\left\| \nabla (u^\ep - w^\ep + z^\ep) \right\|_{L^2(U)} 
\leq
C \left\| \f^\ep \right\|_{L^2(U)} 
\leq
C \left\| \nabla^2 u \right\|_{L^\infty(U)} 
\sup_{|e| =1} \bigl( \| \phi^\ep_e \|_{L^2(U)} + \| \bfs_e^\ep \|_{L^2(U)} \bigr)
\,.
\end{equation}
Using the bound~\eqref{e.correctorsbounded2} on the correctors, we obtain
\begin{equation}
\label{e.uepwepzep}
\left\| \nabla (u^\ep - w^\ep + z^\ep) \right\|_{L^2(U)} 
\leq
C \ep \left\| \nabla^2 u \right\|_{L^\infty(U)} 
\,.
\end{equation}
We next estimate~$z^\ep$, which represents a boundary layer error. 
We select~$\zeta\in C^\infty_{\mathrm{c}}(U)$ which satisfies~$\zeta \equiv 1$ on~$U_{\ep}:= \left\{ x\in U\,:\, B_\ep(x) \subseteq U \right\}$ and~$\left\| \nabla \zeta \right\|_{L^\infty} \leq 2\ep^{-1}$.
Denote
\begin{equation*}
\tilde{z}^\ep
:= 
(1-\zeta)(w^\ep-u) 
= 
(1-\zeta) \sum_{j=1}^d \phi_{e_j}^\ep  \partial_{x_j}u \,.
\end{equation*}
Using H\"older's inequality,~\eqref{e.correctorsbounded}, the fact that~$1-\zeta$ is supported in~$U \setminus U_\ep$ and~$|U\setminus U_\ep| \leq C\ep$, we estimate
\begin{equation*}
\left\| \tilde{z}^\ep \right\|_{L^2(U)} 
\leq 
\left\| \nabla u \right\|_{L^\infty(U)} 
\sup_{|e| =1} \| \phi^\ep_e \|_{L^2(U\setminus U_\ep)}
\leq
C\ep^{\frac32} \left\| \nabla u \right\|_{L^\infty(U)}
\,.
\end{equation*}
Similarly, we have
\begin{equation*}
\left\| \nabla \tilde{z}^\ep \right\|_{{L}^2(U)} 
\leq 
C\ep^{\frac12} \left( \left\| \nabla u \right\|_{L^\infty(\Rd)} 
+
\ep \left\| \nabla^2 u \right\|_{L^\infty(\Rd)} \right). 
\end{equation*}
Testing the equation for~$z^\ep$ with~$z^\ep - \tilde{z}^\ep \in H^1_0(U)$ yields
$0 = \int_U \nabla (z^\ep - \tilde{z}^\ep) \cdot \a^\ep \nabla z^\ep$, 
which implies 
\begin{equation*}
\left\| \nabla z^\ep \right\|_{{L}^2(U)} 
\leq
C\left\| \nabla \tilde{z}^\ep \right\|_{{L}^2(U)}
\leq 
C\ep^{\frac12} \left( \left\| \nabla u \right\|_{L^\infty(\Rd)} 
+
\ep \left\| \nabla^2 u \right\|_{L^\infty(\Rd)} \right).
\end{equation*}
The triangle inequality,~\eqref{e.uepwepzep} and the previous line yield~\eqref{e.ee.easy}. Notice also that, by the Poincar\'e inequality, 
\begin{align}
\label{e.crude}
\left\| z^\ep\right\|_{L^2(U)}
&
\leq 
\left\| \tilde{z}^\ep \right\|_{L^2(U)} 
+ \left\| z^\ep - \tilde{z}^\ep\right\|_{L^2(U)}
\notag \\ & 
\leq C\ep^{\frac32} \left\| \nabla u \right\|_{L^\infty(U)}
+ C \left\| \nabla (z^\ep - \tilde{z}^\ep)\right\|_{L^2(U)}
\notag \\ &  
\leq
C\ep^{\frac12} \left\| \nabla u \right\|_{L^\infty(U)} 
+
C\ep^{\frac 32} \left\| \nabla^2 u \right\|_{L^\infty(U)}.
\end{align}
The estimates for the last two terms on the left side of~\eqref{e.ee.easy2} follow from~\eqref{e.twoscale.strip},~\eqref{e.ee.easy} and the triangle inequality. To estimate the first term, we use the triangle and Poincar\'e inequalities and~\eqref{e.uepwepzep},~\eqref{e.twoscale.strip}, to obtain
\begin{align*}
\left\| u^\ep - u \right\|_{L^2(U)}
&
\leq
\left\| u^\ep - w^\ep+z^\ep \right\|_{L^2(U)}
+
\left\| z^\ep \right\|_{L^2(U)} 
+ \left\| w^\ep - u\right\|_{L^2(U)}
\notag \\ & 
\leq
C
\left\| \nabla( u^\ep - w^\ep+z^\ep) \right\|_{L^2(U)}
+
\left\| z^\ep \right\|_{L^2(U)} 
+ \left\| w^\ep - u\right\|_{L^2(U)}
\notag \\ & 
\leq
C \ep^{\frac12} \bigl (  \left\| \nabla u \right\|_{L^\infty(U)}
+
\ep^{\frac12} \left\| \nabla^2 u \right\|_{L^\infty(U)}\bigr )
+
\left\| z^\ep\right\|_{L^2(U)}.
\end{align*}
Combining this with~\eqref{e.crude} finishes the proof. 
\end{proof}

The previous lemmas assumed that~$u\in W^{2,\infty}= C^{1,1}$. However, an easy approximation argument yields a general qualitative homogenization result for the Dirichlet problem with arbitrary boundary data in~$H^1(U)$.

\begin{corollary}[Homogenization of the Dirichlet problem]
\label{c.DP}
Assume~$U\subseteq\Rd$ is a bounded Lipschitz domain and~$u\in H^1(U)$. For each~$\ep\in (0,1]$, denote by~$u^\ep \in H^1(U)$ the solution of the Dirichlet problem 
\begin{align}
\label{e.eq.DP}
\left\{
\begin{aligned}
& -\nabla \cdot \a^\ep \nabla u^\ep = - \nabla \cdot \ahom\nabla u & \mbox{in} & \ U, 
\\
& u^\ep = u & \mbox{on} & \ \partial U. 
\end{aligned}
\right.
\end{align}
Then
\begin{equation}
\label{e.Dp.homog}
\limsup_{\ep\to 0} 
\Bigl( 
\left\| u^\ep - u \right\|_{L^2(U)}
+ 
\left\| \nabla u^\ep - \nabla u \right\|_{H^{-1}(U)} 
+
\left\| \a^\ep \nabla u^\ep - \ahom \nabla u \right\|_{H^{-1}(U)} 
\Bigr)
= 0.
\end{equation}
\end{corollary}
\begin{proof}
By the Sobolev extension theorem, we may suppose that~$u\in H^1(\Rd)$. Let~$\delta>0$ and select a function~$\tilde{u} \in W^{2,\infty}(\Rd)$ such that 
$\left\| u - \tilde{u} \right\|_{H^1(\Rd)} \leq \delta$. If we define~$\tilde{u}^\ep$ as the solution of ~\eqref{e.eq.DP} with~$\tilde{u}$ in place of~$u$, then it is easy to check that
\begin{equation}
\left\| \tilde{u}^\ep - u^\ep \right\|_{H^1(U)} 
\leq 
C\left\| \tilde{u} - u \right\|_{H^1(U)} 
\leq C\delta\,. 
\end{equation}
Therefore, by the triangle inequality, 
\begin{align*}
& \limsup_{\ep\to 0} 
\Bigl( 
\left\| u^\ep - u \right\|_{L^2(U)}
+ 
\left\| \nabla u^\ep - \nabla u \right\|_{H^{-1}(U)} 
+
\left\| \a^\ep \nabla u^\ep - \ahom \nabla u \right\|_{H^{-1}(U)} 
\Bigr)
\\ & 
\leq
\limsup_{\ep\to 0} 
\Bigl( 
\left\|  \tilde{u}^\ep -  \tilde{u} \right\|_{L^2(U)}
+ 
\left\| \nabla  \tilde{u}^\ep - \nabla  \tilde{u} \right\|_{H^{-1}(U)} 
+
\left\| \a^\ep \nabla  \tilde{u}^\ep - \ahom \nabla  \tilde{u} \right\|_{H^{-1}(U)} 
\Bigr)
+ C\delta\,. 
\end{align*}
Since~$\delta>0$ is arbitrary, it suffices to prove the proposition with~$\tilde{u}$ in place of~$u$. In other words, we may suppose that~$u \in W^{2,\infty}(\Rd)$. The proposition reduces, therefore, to Lemma~\ref{l.DP.2infty}. 
\end{proof}

Further remarks on the arguments we have just seen are now in order:

\begin{itemize}

\item \emph{Lemma~\ref{l.twoscale} is where homogenization happened.} It represents the point in the string of arguments and computations in which we can see that our definitions of the correctors are good and going to give us an understanding of the operator~$-\nabla \cdot \a^\ep\nabla$ at leading-order. This is not very surprising because it is here that the two-scale ansatz~$w^\ep$ is (quite literally) plugged into the equation. 

\item Lemma~\ref{l.DP.2infty} represents merely an unraveling of the implications of Lemma~\ref{l.twoscale} as well as an estimate of the boundary layer error (the effect of the boundary condition being perturbed by the correctors). The proof of Corollary~\ref{c.DP} is then just a little approximation exercise. 
This is not particular to the Dirichlet problem: essentially, any well-posed problem involving the operator~$-\nabla \cdot \a^\ep \nabla$ can be homogenized routinely using Lemma~\ref{l.twoscale}. We present the case of the Neumann problem below in Lemma~\ref{l.NP.2infty}. Any uniquely defined object associated with the operator~$-\nabla \cdot \a^\ep \nabla$, such as the elliptic or parabolic Green functions, can also be homogenized like this. \emph{If you can prove uniqueness for your problem and stability around this uniqueness, then you can use Lemma~\ref{l.DP.2infty} to prove homogenization.}

\item 
Later, we will be interested in generalizing the arguments above to equations with quasiperiodic or random coefficients. Therefore, it is noteworthy that the \emph{only place where periodicity was used was in the existence and ``boundedness'' of the correctors~$\phi_e$ and flux correctors~$\bfs_e$.} In particular, periodicity was not used in the proofs of Lemma~\ref{l.twoscale} and Corollary~\ref{c.twoscale}, and so these statements generalize to the stationary random setting.

\item
The~$W^{2,\infty}(\Rd)$ regularity assumed in Lemmas~\ref{l.twoscale} and~\ref{l.DP.2infty} is much stronger than needed. The statements of the lemmas can be slightly modified to accommodate~$u\in H^2_{\mathrm{loc}}(\Rd)$ (and even weaker) without any loss in the strength of the conclusions. We need to slightly modify the definition of~$w^\ep$ by mollifying the macroscopic function~$u$. Remember that we can choose the function~$u$ in the formula for~$w^\ep$! This is explored in detail below in Remark~\ref{r.twoscale.H2}.

\item
The estimate~\eqref{e.ee.easy} is sharp. Indeed, we should expect that~$\nabla u^\ep$ and~$\nabla w^\ep$ differ by at least~$O(1)$ in a boundary layer of thickness at least~$O(\ep)$, since~$w^\ep$ is not well-adapted to the boundary condition. This alone suggests that~$\left\| \nabla (u^\ep-w^\ep) \right\|_{L^2(U)} \gtrsim O(\ep^{\frac12})$. We can see this in the proof of Lemma~\ref{l.DP.2infty}, as the function~$z^\ep$ represents the boundary layer error. On the other hand, since we have estimated~$\left\| z^\ep \right\|_{L^2(U)}$ quite crudely---it should be~$O(\ep)$ rather than~$O(\ep^{\frac12})$---the estimate~\eqref{e.ee.easy2} can be improved to~$O(\ep)$. This can be found, for instance, in~\cite[Section 6.4]{AKMBook}.
\end{itemize} 

\begin{remark}[{Relaxing the regularity for~$u\in H^2_{\mathrm{loc}}(\Rd)$}]
\label{r.twoscale.H2}
The reason we assumed that~$u\in W^{2,\infty}(\Rd)$ in Lemmas~\ref{l.twoscale} and~\ref{l.DP.2infty} is because this is needed to ensure that~$w^\ep \in H^1_{\mathrm{loc}}$, since the gradient of the correctors~$\nabla \phi_e^\ep$ have only~$L^2$ integrability on the small scales (smaller than~$\ep$) without further regularity assumptions on the coefficient field~$\a(\cdot)$. However, this is easily fixed by simply altering the definition~$w^\ep$ to
\begin{equation}
\label{e.twoscale.mollify}
w^\ep (x)
:=
u(x) 
+ 
\ep \sum_{j=1}^d \partial_{x_j}( u \ast \eta_\ep) (x) \phi_{e_j}\bigl( \tfrac x\ep\bigr),
\end{equation}
where~$\{ \eta_\ep \}_{\ep>0}$ is the standard mollifier. Notice that, in the case~$u\in W^{2,\infty}$, this~$w^\ep$ agrees with the previous one up to~$O(\ep)$ in~$H^1$, which is a permissible error! But this definition gives us that~$w^\ep\in H^1$ even in the case that~$u$ is merely in~$H^2_{\mathrm{loc}}$. Applying Lemma~\ref{l.twoscale} with~$u \ast \eta_\ep$ in place of~$u$ and rearranging the result yields
\begin{align}
\label{e.twoscale.flux.H2}
\a^\ep \nabla w^\ep - \ahom\nabla u
&
=
\sum_{k=1}^d \left( 
\nabla \cdot \bigl ( \bfs^\ep_{e_k}  \partial_{x_k} (u\ast\eta_\ep)  \bigr) 
{+}
(- \bfs^\ep_{e_k} {+} \phi_{e_k}^\ep \a^\ep ) \nabla \partial_{x_k} (u\ast\eta_\ep)  
\right)
\notag \\ 
& \qquad  
+ (\ahom-\a^\ep) \nabla \bigl( (u\ast\eta_\ep) -u \bigr)
\end{align}
and
\begin{equation}
\label{e.twoscale.error.H2}
\nabla \cdot \a^\ep \nabla w^\ep 
-
\nabla\cdot \ahom \nabla u
= 
\nabla\cdot \f^\ep
\,,
\end{equation}
where 
\begin{equation}
\label{e.def.fep.H2}
\f^\ep
:=
\sum_{k=1}^d 
\bigl(
- \bfs_{e_k}^\ep 
+
\phi_{e_k}^\ep \a^\ep 
\bigr) 
\nabla \partial_{x_k} (u\ast\eta_\ep)
+
(\ahom-\a^\ep) \nabla \bigl( (u\ast\eta_\ep) -u \bigr).
\end{equation}
We therefore obtain, for any bounded Lipschitz domain~$U\subseteq \Rd$,  
\begin{equation}
\label{e.twoscale.explicit.rem}
\left\|
\nabla \cdot \a^\ep \nabla w^\ep 
-
\nabla\cdot \ahom \nabla u
\right\|_{H^{-1}(U)} 
\leq 
\left\| \f^\ep \right\|_{L^2(U)}
\leq
C \ep \left\| \nabla^2 u \right\|_{L^2(U)}
\end{equation}
and
\begin{equation}
\label{e.twoscale.strip.rem}
\left\| 
w^\ep - u 
\right\|_{L^2(U)} 
+
\left\| 
\nabla w^\ep - \nabla u \right\|_{H^{-1}(U)} 
+
\left\| 
\a^\ep \nabla w^\ep - \ahom \nabla u \right\|_{H^{-1}(U)}
\leq 
C \ep \left\| \nabla u \right\|_{H^1(U)}.
\end{equation}
The proof of these estimates requires the additional observation that for any bounded Lipschitz domain~$U\subseteq\Rd$,~$\ep\in(0,1]$,~$\Zd$--periodic function~$p \in L^2_{\mathrm{loc}}(\Rd)$ and~$f\in H^1_{\mathrm{loc}}(\Rd)$, we have, for a constant~$C(U,d)<\infty$,
\begin{align}
\label{e.cubescutout}
\left\| (f\ast \eta_\ep) p\bigl(\tfrac\cdot\ep\bigr)\right\|_{L^2(U)}
\leq
C \left\| f \right\|_{L^2(U^\ep)}
\left\| p \right\|_{L^2(\cu_0)}\,,
\end{align}
where~$U^\ep := \cup \left\{ z+\ep \cu_0 \, :\, z\in \ep\Zd, \, (z+\ep \cu_0) \cap U \neq \emptyset \right\}$
and~$\cu_0 := (-\nicefrac12, \nicefrac12)^d$.  
We leave the proof of~\eqref{e.cubescutout} as well as those of~\eqref{e.twoscale.explicit.rem} and~\eqref{e.twoscale.strip.rem} as exercises for the reader (they can also be found in~\cite[Chapter 6]{AKMBook}). 
\end{remark}

We next present a version of Lemma~\ref{l.DP.2infty} for the Neumann problem. 
We denote the outward-pointing unit normal to~$\partial U$ at a boundary point~$x\in\partial U$ by~$\mathbf{n}_{U}(x)$. 

\begin{lemma}[Neumann problem]
\label{l.NP.2infty}
Let~$\ep\in (0,1]$,~$u\in W^{2,\infty}(\Rd)$ and~$w^\ep\in H^1_{\mathrm{loc}}(\Rd)$ be defined by~\eqref{e.twoscale.wep}. Let~$U\subseteq\Rd$ be a bounded Lipschitz domain and~$u^\ep \in H^1(U)$ satisfy the Neumann problem 
\begin{align}
\label{e.eq.NP.2infty}
\left\{
\begin{aligned}
& -\nabla \cdot \a^\ep \nabla u^\ep = - \nabla \cdot \ahom\nabla u & \mbox{in} & \ U, 
\\
& \mathbf{n}_{U} \cdot \a^\ep \nabla u^\ep = \mathbf{n}_{U} \cdot \ahom\nabla u & \mbox{on} & \ \partial U. 
\end{aligned}
\right.
\end{align}
Then there exists~$C(U,d,\lambda,\Lambda)<\infty$ such that 
\begin{equation}
\label{e.ee.easy.NP}
\left\| \nabla u^\ep - \nabla w^\ep \right\|_{{L}^2(U)} 
\leq
C \ep^{\frac12} \left(  \left\| \nabla u \right\|_{L^\infty(U)}
+
\ep^{\frac12} \left\| \nabla^2 u \right\|_{L^\infty(U)}
\right).
\end{equation}
Consequently, 
\begin{equation}
\label{e.ee.easy2.NP}
\left\| 
\nabla u^\ep {-} \nabla u \right\|_{H^{-1}(U)} 
+
\left\| 
\a^\ep \nabla u^\ep{ -} \ahom \nabla u \right\|_{H^{-1}(U)}
\leq 
C \ep^{\frac12} \left(  \left\| \nabla u \right\|_{L^\infty(U)}
+
\ep^{\frac12} \left\| \nabla^2 u \right\|_{L^\infty(U)}
\right).
\end{equation}
\end{lemma}
\begin{proof}
We give a sketch of the argument, leaving some details as an exercise for the reader.
In view of~\eqref{e.twoscale.flux}, let us define 
\begin{equation}
\mathbf{h}^\ep (x)
:= 
\sum_{k=1}^d 
\left( 
\partial_{x_k} u  \nabla \cdot \left( (1-\zeta) \bfs^\ep_{e_k} \right)
+
 (1-\zeta)\phi^\ep_{e_k} \a^\ep \nabla \partial_{x_k} u 
\right),
\end{equation}
where~$\zeta$ is as in the proof of Lemma~\ref{l.DP.2infty}. 
We define~$z^\ep\in H^1(U)$ to be the solution of 
\begin{align*}
\left\{
\begin{aligned}
& -\nabla \cdot \a^\ep \nabla z^\ep = - \nabla \cdot \mathbf{h}_\ep
& \mbox{in} & \ U,  \\
& \mathbf{n}_{U} \cdot \a^\ep \nabla z^\ep = \mathbf{n}_{U} \cdot \mathbf{h}_\ep 
& \mbox{on} & \ \partial U. 
\end{aligned}
\right. 
\end{align*}
It is straightforward to check that 
\begin{equation}
\left\| \nabla z^\ep \right\|_{L^2(U)}
\leq 
C\left\| \mathbf{h}^\ep \right\|_{L^2(U)}
\leq
C \ep^{\frac12} \left\| \nabla u \right\|_{L^\infty(U)}
+
C \ep^{\frac32} \left\| \nabla^2 u \right\|_{L^\infty(U)}.
\end{equation}
Testing the equations for~$z^\ep$ and~$u^\ep$ and~\eqref{e.twoscale.error} with~$(u^\ep {-} w^\ep {+} z^\ep)$. we find that 
\begin{equation}
\int_U 
\nabla (u^\ep {-} w^\ep {+} z^\ep) 
\cdot \a^\ep 
\nabla (u^\ep {-} w^\ep {+} z^\ep)
=
\int_U 
\left(\mathbf{h}^\ep {-} \f^\ep \right)
\cdot
\nabla (u^\ep {-} w^\ep {+} z^\ep),
\end{equation}
with~$\f^\ep$ as in~\eqref{e.def.fep}. We deduce that 
\begin{align*}
\left\| 
\nabla (u^\ep {-} w^\ep {+} z^\ep)
\right\|_{L^2(U)}
& 
\leq 
C \bigl( 
\left\| \mathbf{h}^\ep \right\|_{L^2(U)}
+\left\| \f^\ep \right\|_{L^2(U)}
\bigr)
\leq
C \ep^{\frac12} \left\| \nabla u \right\|_{L^\infty(U)}
+
C \ep \left\| \nabla^2 u \right\|_{L^\infty(U)}.
\end{align*}
Thus 
\begin{align*}
\left\| 
\nabla (u^\ep {-} w^\ep)
\right\|_{L^2(U)}
&
\leq
\left\| 
\nabla (u^\ep {-} w^\ep {+} z^\ep)
\right\|_{L^2(U)}
+
\left\| 
\nabla z^\ep
\right\|_{L^2(U)}
\notag \\ &
\leq 
C \ep^{\frac12} \left\| \nabla u \right\|_{L^\infty(U)}
+
C \ep \left\| \nabla^2 u \right\|_{L^\infty(U)}.
\end{align*}
This is~\eqref{e.ee.easy.NP} and~\eqref{e.ee.easy2.NP} follows from this and~\eqref{e.twoscale.strip}. 
\end{proof}

\subsection{Ergodic theorems in~\texorpdfstring{$\Rd$}{{Rd}} and~\texorpdfstring{$\Zd$}{{Zd}}}
\label{ss.ergodic}

This section presents versions of the multiparameter ergodic and subadditive ergodic theorems suitable for proving qualitative homogenization in a general stochastic setting. Homogenization involves, at some point in the argument, taking a limit---usually of a large-scale spatial average of a field, such as the gradient or flux of a solution, perhaps a corrector. 
Ergodic theorems are statements saying that if certain conditions are satisfied, then we can take such limits.

\smallskip

To state the ergodic theorems, we require some terminology from the theory of dynamical systems. We let~$(\Omega,\F)$ denote a measurable space endowed with a group action~$\{T_y\}_{y\in\Rd}$ of~$(\Rd,+)$ on~$(\Omega,\F)$. In other words, we assume that, for each~$y\in\Rd$,
\begin{equation}
T_y:\Omega \to \Omega
\end{equation}
is an~$\F$--measurable transformation such that
\begin{equation}
T_y \circ T_{y'} = T_{y+y'}, \qquad \forall y,y'\in\Rd.
\end{equation}
We extend~$T_y$ so that~$T_y :\F \to \F$ by defining~$T_yE := \{ T_y \omega \,:\, \omega \in E \}$. 
An event~$E \in \F$ is called \emph{invariant} if~$T_z E = E$ for every~$z\in\Zd$. 
A \emph{random variable} is an~$\F$--measurable map~$X:\Omega \to \R$. A \emph{random element} is an~$(\F,\mathcal{S})$--measurable map~$\xi : \Omega \to S$ where~$(S,\mathcal{S})$ is some measurable space. 
We extend~$T_y$ to random elements~$\xi$ on~$\Omega$ by defining~$(T_y\xi)(\omega) := \xi(T_y\omega)$ for each~$y\in\Rd$. 
We denote the Borel~$\sigma$-algebra on~$\Rd$ by~$\mathscr{B}$. An \emph{$S$--valued random field} is an~$(\F \times \mathscr{B} , \mathcal{S})$--measurable map~$F: \Omega \times \Rd \to S$. A random field~$F$ is called~$\Zd$--\emph{stationary}, or simply \emph{stationary}, if it satisfies~$F(\omega,x+z) = F(T_z\omega,x)$, for every~$x\in\Rd$,~$z\in \Zd$ and~$\omega\in\Omega$. 
If~$\xi$ is a random element on~$\Omega$, then the \emph{stationary extension of~$\xi$} is the random field~$\overline{\xi}$ defined by~$\overline{\xi}(\omega,x) := \xi(T_x\omega)$. Note that~$\overline{\xi}$ is~$\Zd$--stationary by definition. 

\smallskip

Throughout this section, we assume that~$\P$ is probability measure on~$(\Omega,\F)$ which is~\emph{$\Zd$--stationary}. The latter means that  
\begin{equation}
\label{e.Zdstat}
\P \circ T_z = \P, \qquad \forall z\in\Zd\,.
\end{equation}
We denote the expectation with respect to~$\P$ by~$\E$. 
The probability measure~$\P$ is called \emph{ergodic} with respect to the translation group~$\{T_z : z\in\Zd\}$ if
\begin{equation}
\label{e.Zdergodic}
\forall E \in\F,
\qquad
E = \bigcap_{z\in \Zd} T_z E 
\quad\implies \quad
\P[E] \in \{0,1\}
\,.
\end{equation}
If~$\P$ satisfies~\eqref{e.Zdergodic} in addition to~\eqref{e.Zdstat} and, then, we say that it is \emph{stationary-ergodic}, for short. Most of the results in this section assume~$\P$ is stationary but do not require~$\P$ to be ergodic. 

\smallskip

We denote the \emph{mean} of a real-valued stationary random field~$f$ by
\begin{equation}
\left\langle f \right\rangle
:=
\E \biggl[ \int_{\cu_0} f(x)\,dx \biggr]. 
\end{equation}
To be more precise, we first define~$\langle \cdot \rangle$ for nonnegative~$f$ and then extend the definition to~$f$ satisfying~$\langle |f| \rangle <\infty$. Notice that~$\langle T_y f \rangle$ is independent of~$y\in\Rd$ for any~$\Zd$--stationary random field~$f$. 
If~$\xi$ is extended random variable on~$\Omega$, that is, an~$\F$--measurable random variable valued in~$\R \cup \{\pm \infty\}$, then we write~$\xi \in S_{\mathrm{inv}}$ if~$\P [ T_y \xi = \xi\,, \forall y\in\Rd ] = 1$. Equivalently,~$\xi \in S_{\mathrm{inv}}$ if its stationary extension~$\overline{\xi}$ satisfies~$\overline{\xi}(\omega,x) = \xi(\omega)$ for every~$x\in\Rd$. By a slight abuse of notation, we say that a~$\Zd$--stationary random field belongs to~$S_{\mathrm{inv}}$ if it is the stationary extension of an extended random variable belonging to~$S_{\mathrm{inv}}$. 

\smallskip

For each~$p\in [1,\infty)$, we let~$S^p(\P)$ denote the linear space of real-valued stationary random fields~$f$ satisfying~$\| f \|_{S^p(\P)} := \left\langle |f|^p \right\rangle^{\nicefrac1p} <\infty$. The definition is extended to~$p=\infty$ in the obvious way. It is clear that~$(S^p(\P), \| \cdot\|_{S^p(\P)})$ is a Banach space for every~$1\leq p \leq \infty$. 
Throughout, we let~$p':=p/(p-1)$ denote the H\"older conjugate of~$p\in[1,\infty]$. 

\smallskip

The translation group~$\{ T_y \,:\, y\in \Rd \}$ acts on~$S^p(\P)$ in a natural way. We introduce the closed subspace
\begin{equation}
S^p_{\mathrm{inv}}(\P) := S^p(\P) \cap S_{\mathrm{inv}} \,.
\end{equation}
This consists of the elements of~$S^p(\P)$ which are~$\P$--a.s.~constant in space. Note that these can be identified with the constant (deterministic) fields if and only if~$\P$ is ergodic:
\begin{equation}
\label{e.ergodicityduh}
S^p_{\mathrm{inv}}(\P) = \R \quad \iff \quad 
\text{$\P$ is ergodic.}
\end{equation}
We also define
\begin{equation}
S^p_{0}(\P):=
\left\{ f \in S^p(\P) \,:\,  
\left\langle f g \right\rangle = 0, \quad 
\forall g \in S^{p'}_{\mathrm{inv}}(\P)
\right\}.
\end{equation}
In the ergodic case, these are the mean-zero random fields. Otherwise, a member of~$S^p_{0}(\P)$ must have zero mean on every  element of~$\F$ which is invariant with respect to~$T_z$ for every~$z\in\Zd$. 

\subsubsection{A multiparameter version of the Wiener ergodic theorem}
The following result is analogous to the Lebesgue differentiation theorem. The main difference is that we are zooming out and looking at large scales rather than zooming in and looking at infinitesimal scales. 

\begin{proposition}[Wiener Ergodic theorem, multiparameter version]
\label{p.ergodic}
Suppose that~$\P$ is a~$\Zd$--stationary probability measure on~$(\Omega,\F)$. 
For every~$f\in S^1(\P)$, there exists~$\overline{f}\in S^1_{\mathrm{inv}}(\P)$ such that, for every bounded open set~$U\subseteq \Rd$, 
\begin{equation}
\label{e.ergodic}
\limsup_{t\to\infty} \
\left| \fint_{tU} f - \overline{f} \right| = 0, 
\qquad 
\mbox{$\P$--a.s.}
\end{equation}
\end{proposition}

We begin by introducing a version of the Hardy-Littlewood maximal function. For every~$f \in S^1(\P)$, we define the maximal function (random field) of~$f$ by
\begin{equation}
Mf(x) := \sup_{n\in\N} \ \fint_{x+\cu_n} |f(y)|\,dy. 
\end{equation}
Recall that~$\cu_n$ denotes the cube centered at the origin with side length~$3^n$. 

\smallskip

The following is an ergodic version of the weak and strong-type Hardy-Littlewood maximal inequalities.

\begin{lemma}[Maximal ergodic theorem]
\label{l.maximal.ergodic}
There exists~$C(d)<\infty$ such that, for every~$f\in S^1(\P)$ and~$x\in \Rd$, 
\begin{equation}
\P \left[ Mf(x) > \lambda \right] \leq \frac{C}{\lambda} \bigl \langle |f| \bigr\rangle\,. 
\end{equation}
For each~$p\in (1,\infty)$, there exists~$C(p,d)<\infty$ such that, for every~$f \in S^p(\P)$ and~$x\in\Rd$, 
\begin{equation}
\label{e.maximal.ergodic.p}
\E \left[ |Mf(x)|^p \right] 
\leq 
C \bigl\langle |f|^p \bigr\rangle\,. 
\end{equation}
\end{lemma}
\begin{proof}

\emph{Step 1.}
We show there exists~$C(d)<\infty$ such that, for~~$f\in S^1(\P)$,
\begin{equation}
\label{e.maximal1}
\P \left[ Mf(x) > \lambda \right] \leq \frac{C}{\lambda} \bigl \langle |f| \bigr\rangle, \quad \forall x\in \Rd.
\end{equation}
Define, for~$k\in\N$  and~$f\in S^1(\P)$, the approximate maximal random field of~$f$ by 
\begin{equation}
(M_kf)(x):= 
\sup_{n\in \N \cap [0,k]} 
\fint_{x+\cu_n} |f| \, \quad x\in \Rd.
\end{equation}
Since the standard Hardy-Littlewood maximal function dominates~$M_kf$, we get by the classical weak-type maximal inequality (\cite[Theorem 3.17]{Folland}), that there exists a constant~$C(d)<\infty$ such that 
\begin{equation}
\frac{\left| \left\{ x \in \cu_k \,:\, M_kf(x) > \lambda \right\} \right|}
{|\cu_k|}
\leq 
\frac{C}{\lambda}
\fint_{\cu_{k+1}} |f|\,. 
\end{equation}
Thus 
\begin{equation}
\E\left[ \frac{\left| \left\{ x \in \cu_k \,:\, M_kf(x) > \lambda \right\} \right|}
{|\cu_k|} \right]
\leq 
\frac{C}{\lambda}
\E \biggl[ 
\fint_{\cu_{k+1}} |f| \biggr]\,.
\end{equation}
By stationarity, this inequality is the same as 
\begin{equation}
\E\bigl[ \left| \left\{ x \in \cu_0 \,:\, M_kf(x) > \lambda \right\} \right| \bigr]
\leq
\frac{C}{\lambda} \bigl \langle |f| \bigr\rangle\,. 
\end{equation}
Sending~$k\to \infty$ yields, by the monotone convergence theorem, that
\begin{equation}
\label{e.slideup}
\E\bigl[  \left| \left\{ x \in \cu_0 \,:\, Mf(x) > \lambda \right\} \right| \bigr]
\leq
\frac{C}{\lambda} \bigl \langle |f| \bigr\rangle. 
\end{equation}
From~\eqref{e.slideup} and the simple fact that
\begin{equation}
3^{-d} Mf (x) \leq Mf(0) \leq 3^d Mf(x), \quad \forall x\in \cu_0, 
\end{equation}
we obtain
\begin{equation}
\label{e.maximal11}
\P \left[ Mf(x) > \lambda \right] \leq \frac{C}{\lambda} \bigl \langle |f| \bigr\rangle, \quad \forall x\in \cu_0.
\end{equation}
Stationarity and the above estimate imply the claim~\eqref{e.maximal1}. 

\smallskip

\emph{Step 2.} We improve~\eqref{e.maximal1} to the bound
\begin{equation}
\label{e.maximal2}
\P \left[ Mf(x) > \lambda \right] 
\leq 
\frac{C}{\lambda} \bigl\langle \left| f \right| \indc_{\{ |f| > \frac12\lambda \} } \bigr\rangle, \quad \forall x\in \Rd.
\end{equation}
For~$\lambda>0$, define~$f^\lambda:= f \indc_{\{ |f| > \frac12\lambda\}}$ and~$f_\lambda := f \indc_{\{ |f| \leq \frac12\lambda\}}$ so that~$f= f^\lambda+f_\lambda$. 
By the sublinearity of the maximal operator~$M$, we have 
\begin{equation*}
Mf 
\leq M f^\lambda + Mf_\lambda
\leq Mf^\lambda + \frac12\lambda
\end{equation*}
and thus~$\{ Mf > \lambda \} \subseteq \{ M f^\lambda > \frac12\lambda\}$. Applying Step~1 above to~$f^{\lambda}$, we obtain, for every~$x\in\Rd$, 
\begin{equation*}
\P \bigl[ Mf(x) > \lambda \bigr]
\leq 
\P \bigl[ Mf^{\lambda} (x) > \tfrac12\lambda \bigr]
\leq
\frac{C}{\lambda} \bigl\langle \left| f \right| \indc_{\{ |f| > \frac12\lambda \} } \bigr\rangle\,.
\end{equation*}

\emph{Step 3.} We prove~\eqref{e.maximal.ergodic.p}.
Compute, for~$f\in S^p(\P)$ with~$p\in (1,\infty)$, 
\begin{align*}
\E \bigl[ |Mf(x)|^p \bigr]
&
=
p\int_0^\infty\lambda^{p-1} \P \bigl[ Mf(x) >\lambda \bigr] \,d\lambda
\\ & 
\leq
Cp
\int_0^\infty \lambda^{p-2}
\bigl\langle \left| f \right| \indc_{\{ |f| > \frac12\lambda \} } \bigr\rangle \, d\lambda
\\ & 
= 
Cp \,
\biggl\langle
|f| 
\int_0^{2|f|} \lambda^{p-2} \,d\lambda
\biggr\rangle
=
\frac{Cp 2^{p-1}}{p-1}
\bigl\langle |f|^p \bigr\rangle\,.
\end{align*}
In the second line, we used~\eqref{e.maximal2}. 
The proof is now complete. 
\end{proof}

Next, we will provide the proof of the multiparameter Wiener ergodic theorem. The idea is that, restricted to~$S^2(\P)$, the mapping~$f\mapsto \overline{f}$ should simply be the linear projection onto~$S^2_\inv(\P)$. The ergodic theorem can then be proved by extending this projection operator to~$S^1(\P)$ by an easy density argument. 

\begin{proof}[Proof of Proposition~\ref{p.ergodic}]
A convenient way to represent~$S^2_{0}(\P)$ is via
\begin{equation}
\label{e.V0rep}
S^2_0(\P) = \overline{K},
\end{equation}
where~$\overline{K}$ is the closure with respect to~$S^2(\P)$ of the subspace 
\begin{equation*}
K:= \bigl\{ f - T_y f \,:\, f\in S^\infty(\P), \ y \in\Rd \bigr\}\,.
\end{equation*}
To see this, we use the fact that~$\bigl\langle (g - T_{-y}g) f  \bigr\rangle
=
\bigl\langle g(f-T_{y}f) \bigr\rangle$ for every~$f,g\in S^2(\P)$ to observe that 
\begin{align*}
g \in K^\perp
&
\iff 
\bigl\langle g(f-T_{y}f) \bigr\rangle 
= 0, \quad \forall f\in S^\infty(\P), \, y\in \Rd, 
\\ & 
\iff
\bigl\langle (g - T_{-y}g) f  \bigr\rangle 
= 0, \quad \forall f\in S^\infty(\P), \, y\in \Rd, 
\\ & 
\iff
g = T_{-y}g, \quad \forall y\in\Rd
\\ & 
\iff 
g\in S^2_{\inv}(\P). 
\end{align*}
Thus~$\overline{K} = (K^{\perp})^\perp = S^2_{\inv}(\P)^\perp = S^2_{0}(\P)$, which proves~\eqref{e.V0rep}. 

\smallskip

We next show that, for every~$h\in K$ and bounded open subset~$U\subseteq\Rd$,
\begin{equation}
\label{e.ergodicforS}
\limsup_{t \to \infty} \
\left| \fint_{tU} h \right| = 0, \quad \mbox{$\P$--a.s.}
\end{equation}
To prove~\eqref{e.ergodicforS}, we first use~\eqref{e.V0rep} to write~$h= f-T_yf$ for~$f\in S^\infty(\P)$ and~$y\in\Rd$, and then compute
\begin{align*}
\limsup_{t\to \infty} \
\left| \fint_{tU} h  \right|
&
= 
\limsup_{t\to \infty} 
\left| \fint_{tU} (\indc_{tU}(x) -  \indc_{y+tU}(x)) f(x)  \,dx \right| 
\\ & 
\leq 
\left\| f \right\|_{S^\infty(\P)}
\limsup_{t\to \infty} \
\frac{\left| (tU) \triangle (y+tU) \right|}{|tU|}
\\ &
=
\left\| f \right\|_{S^\infty(\P)}
\limsup_{t\to \infty} \
\frac{\left| U \triangle (t^{-1}y+U) \right| }{|U|}
= 0 \,,
\end{align*}
where~$\triangle$ stands for the symmetric difference of sets. This proves~\eqref{e.ergodicforS}. 

\smallskip

By~\eqref{e.ergodicforS} for~$h\in K$ and a density argument, using the maximal ergodic theorem (Lemma~\ref{l.maximal.ergodic}), we obtain that~\eqref{e.ergodicforS} is valid for every~$h\in S^2_0(\P)$. 
Indeed, given~$h\in S^2_0(\P)$ and~$\ep>0$, we use~\eqref{e.V0rep} to select~$h_\ep \in K$ with~$\| h - h_\ep \|_{S^2(\P)} \leq \ep$. Since~$U$ is bounded, we can select~$m\in\N$ such that~$U\subseteq \cu_m$. By the triangle inequality and~\eqref{e.ergodicforS} for~$h\in K$, we obtain,~$\P$--a.s., 
\begin{align*}
\limsup_{t\to \infty} \
\left| \fint_{tU} h  \right|
&
\leq 
\limsup_{t\to \infty} \
\left| \fint_{tU} h_\ep  \right| 
+
\limsup_{t\to \infty} \
\fint_{tU} | h - h_\ep | 
\\ & 
\leq 
0 + \frac{C|\cu_m|}{|U|} \sup_{n \geq m}\fint_{\cu_n} | h - h_\ep | 
\leq
\frac{C|\cu_m|}{|U|} M(h-h_\ep)(0)
\,.
\end{align*}
Therefore by Lemma~\ref{l.maximal.ergodic}, for every~$\lambda>0$, 
\begin{align*}
\P \left[ 
\limsup_{t\to \infty} \
\left| \fint_{tU} h  \right| > \lambda \right]
\leq 
\P \left[ 
M(h-h_\ep)(0) > \frac{\lambda |U|}{C|\cu_m|}\right]
\leq
\frac{C3^{dm}}{\lambda|U|} \langle | h - h_\ep | \rangle 
\leq 
\frac{C3^{dm}\ep}{\lambda|U|}
\,.
\end{align*}
Sending~$\ep \to 0$ yields~\eqref{e.ergodicforS} for every~$h\in S^2_0(\P)$. 

Let~$P$ be the orthogonal projection of~$S^2(\P)$ onto~$S^2_{\mathrm{inv}}(\P) = S^2_0(\P)^\perp$. For every~$f\in S^2(\P)$, we have that~$f-Pf \in S^2_0(\P)$ and hence
\begin{equation}
\label{e.ergodicforS.P}
\limsup_{t \to \infty} \
\left| \fint_{tU} (f-Pf) \right| = 0, \quad \mbox{$\P$--a.s.}
\end{equation}
Since~$P$ is a projection, it may be extended to a projection from~$S^1(\P)$ to~$S^1_{\mathrm{inv}}(\P)$. We may, therefore, use the density of~$S^2(\P)$ in~$S^1(\P)$ and of~$S^2_{\mathrm{inv}}$ in~$S^1_{\mathrm{inv}}$ and another density argument, as above, with the help of the maximal ergodic theorem, to obtain~\eqref{e.ergodicforS} for~$f\in S^1(\P)$. 
\end{proof}

We next present a convenient form of the Wiener ergodic theorem which is stated in terms of the~$H^{-1}$ norms of elements of~$S^2(\P)$. The~$H^{-1}$ norm is a natural way to quantify the weak topology in~$L^2$.
Note that we work with the volume-normalized and scale-invariant Sobolev norms defined in~\eqref{e.volume.normalized} and~\eqref{e.underlinednorms}. Recall that we use the notation~\eqref{e.slashedsum} for normalizing sums with slashes. 

\begin{corollary}[Wiener ergodic theorem,~$H^{-1}$ version]
\label{c.stat.Hminusone}
For every~$f \in S^2_0(\P)$,
\begin{equation}
\label{e.stat.Hminusone}
\P \biggl[
\limsup_{m\to \infty} 
3^{-m}\left\| f   \right\|_{\Hminusul(\cu_m)} 
= 0
\biggr] = 1
\,.
\end{equation}
\end{corollary}
\begin{proof}
The argument is a simple consequence of Proposition~\ref{p.ergodic} and the 
multiscale Poincar\'e inequality, which is stated below in Proposition~\ref{p.MSP}.
The latter implies,  for every~$m,k\in\N$,
\begin{align*}
3^{-m} \left\| f \right\|_{\Hminusul(\cu_{m})} 
&
\leq
C3^{-k} \left\| f \right\|_{\underline{L}^2(\cu_{m})}
+
C\sum_{n=m-k}^{m}
3^{n-m}
\biggl( \avsum_{z\in 3^n\Zd \cap \cu_{m}} \left| \left( f \right)_{z+\cu_n} \right|^2 \biggr)^{\!\nicefrac12}
\\ & 
\leq 
C3^{-k} \left\| f \right\|_{\underline{L}^2(\cu_{m})}
+
C \max \bigl\{ 
\left| \left( f \right)_{z+\cu_n} \right|  : m-k\leq n \leq m, \, z\in 3^n\Zd\cap \cu_{m} \bigr\}.
\end{align*}
By Proposition~\ref{p.ergodic}, there exists a nonnegative random variable~$F \in S^2_{\mathrm{inv}}(\P)$ such that 
\begin{equation*}
\lim_{m\to \infty} \left\| f \right\|_{\underline{L}^2(\cu_m)} = 
\lim_{m\to \infty} \biggl( \fint_{\cu_m} |f|^2 \biggr)^{\!\!\nicefrac12}
= F \,,
\quad \mbox{$\P$--a.s.,} 
\end{equation*}
and, using also that~$f \in S_0^2(\P)= S^2_{\inv}(\P)^\perp$, we have, for each fixed~$k\in\N$, 
\begin{equation*}
\limsup_{m\to \infty} \, \max\left\{
\left| \left( f \right)_{z+\cu_n} \right| \,:\, m-k\leq n \leq m, \ z\in 3^n\Zd\cap \cu_{m} \right\}
= 0,
\quad \mbox{$\P$--a.s.}
\end{equation*}
We deduce that, for each fixed~$k\in\N$,
\begin{equation}
\limsup_{m\to\infty}
3^{-m}\left\| f \right\|_{\Hminusul(\cu_{m})} 
\leq 
C3^{-k} F\,, \quad \mbox{$\P$--a.s.}
\end{equation}
Sending~$k\to \infty$ yields~\eqref{e.stat.Hminusone}.
\end{proof}

In the proof of Corollary~\ref{c.stat.Hminusone} we used the following estimate, which roughly asserts that the constant in the Poincar\'e inequality (which is  proportional to the size of the cube) need not multiply the full~$L^2$ norm of the gradient, but rather only its weaker spatial average. Every triadic scale~$n$ has its own Poincar\'e constant~$C3^n$ which multiplies only the~$\ell_2$-norm of the spatial averages on that scale. For functions with gradients exhibiting many spatial oscillations, this evidently gives a much better estimate than the usual Poincar\'e inequality, and this is the reason it is a useful tool in deriving Corollary~\ref{c.stat.Hminusone} from Proposition~\ref{p.ergodic}. 

\begin{proposition}[Multiscale Poincar\'e inequality]
\label{p.MSP}
There exists a constant~$C(d)<\infty$ such that, for every~$m \in \N$ and~$f\in L^2(\cu_m)$, 
\begin{equation}
\label{e.MSP}
\left\| f \right\|_{\Hminusul(\cu_m)}
\leq
C 
\sum_{n=-\infty}^{{m}} 3^{n} 
\biggl( \avsum_{z\in 3^n\Zd\cap \cu_{m}} 
\bigl|  ( f  )_{z+\cu_n} \bigr|^2 \biggr)^{\!\!\nicefrac12}. 
\end{equation}
Moreover, for every~$m \in \N$ and~$w \in H^1(\cu_{m})$, we have
\begin{equation} 
\label{e.MSPw}
\left\| w - \left( w \right)_{\cu_{m}} \right\|_{\underline{L}^2(\cu_{m}) }
\leq 
C \sum_{n=-\infty}^{{m}} 3^{n} 
\biggl( 
\avsum_{z\in 3^n\Zd\cap \cu_{m}} 
\bigl|  ( \nabla w  )_{z+\cu_n} \bigr|^2
\biggr)^{\!\!\nicefrac12}
\,.
\end{equation}

\end{proposition}

\begin{proof}
Fix~$m \in \N$ and~$f\in L^2(\cu_m)$. 

\smallskip 

\emph{Step 1.} 
We show that, for every~$s\in(0,1]$ and~$f,g\in L^2(\cu_m)$ with~$(g)_{\cu_m} = 0$, 
\begin{align}
\label{e.general.Besov.pairing}
\biggl| \fint_{\cu_m} f g  \biggr| 
&
\leq 
3^{d+s}
\sup_{k\in\Z,\,k\leq m} 
3^{-sk} 
\biggl( 
\avsum_{z\in 3^{k} \Zd\cap \cu_m} 
\fint_{z+\cu_k} 
\bigl| g -
(g)_{z+\cu_{k}} \bigr|^2 
\biggr)^{\!\nicefrac12}
\notag \\ & \qquad\qquad\qquad \times
\sum_{k=-\infty}^{m-1} 
3^{sk}
\biggl( 
\avsum_{z\in 3^{k} \Zd\cap \cu_m} 
\bigl| (f)_{z+\cu_{k}} \bigr|^2 
\biggr)^{\!\nicefrac12} 
\,.
\end{align}
To prove~\eqref{e.general.Besov.pairing}, 
We next use the fact that, by the Lebesgue differentiation theorem, any~$g\in L^1(\cu_m)$ function can be written (almost everywhere in~$\cu_m)$ as 
\begin{equation*}
g = 
(g)_{\cu_m} 
+
\sum_{k=-\infty}^m 
\sum_{z\in 3^k\Zd \cap \cu_m} 
\sum_{z'\in 3^{k-1} \Zd \cap (z+\cu_k)} 
\bigl( (g)_{z'+\cu_{k-1}} -
(g)_{z+\cu_{k}} \bigr) \indc_{z'+\cu_{k-1}} 
\,.
\end{equation*}
Therefore we have that, if~$(g)_{\cu_m} =0$, then 
\begin{equation*}
\fint_{\cu_m} f g 
=
\sum_{k=-\infty}^m 
\avsum_{z\in 3^k\Zd \cap \cu_m} 
\avsum_{z'\in 3^{k-1} \Zd \cap (z+\cu_k)} 
\bigl( (g)_{z'+\cu_{k-1}} -
(g)_{z+\cu_{k}} \bigr) 
(f)_{z'+\cu_{k-1}} 
\,.
\end{equation*}
Therefore, using H\"older's inequality, we find that, for every~$s \in (0,1]$, 
\begin{align*}
\lefteqn{ 
\biggl| \fint_{\cu_m} f g  \biggr| 
}  \ \  & 
\notag \\ & 
\leq
\sum_{k=-\infty}^m
\avsum_{z\in 3^k\Zd \cap \cu_m} 
\avsum_{z'\in 3^{k-1} \Zd \cap (z+\cu_k)} 
\bigl| (g)_{z'+\cu_{k-1}} -
(g)_{z+\cu_{k}} \bigr|
\bigl| (f)_{z'+\cu_{k-1}} \bigr|
\notag \\ & 
\leq
\sum_{k=-\infty}^m 
\avsum_{z\in 3^k\Zd \cap \cu_m} 
3^d \fint_{z+\cu_k} 
\bigl| g -
(g)_{z+\cu_{k}} \bigr|
\avsum_{z'\in 3^{k-1} \Zd \cap (z+\cu_k)} 
\bigl| (f)_{z'+\cu_{k-1}} \bigr|
\notag \\ & 
\leq
3^d \sum_{k=-\infty}^{m} 
\biggl( 
\avsum_{z\in 3^{k} \Zd\cap \cu_m} 
\fint_{z+\cu_k} 
\bigl| g -
(g)_{z+\cu_{k}} \bigr|^2 
\biggr)^{\!\nicefrac12} 
\biggl( 
\avsum_{z\in 3^{k-1} \Zd\cap \cu_m} 
\bigl| (f)_{z+\cu_{k-1}} \bigr|^2 
\biggr)^{\!\nicefrac12} 
\notag \\ & 
\leq 
3^{d+s} \sup_{k\leq m} 
3^{-sk} 
\biggl( 
\avsum_{z\in 3^{k} \Zd\cap \cu_m} 
\fint_{z+\cu_k} 
\bigl| g -
(g)_{z+\cu_{k}} \bigr|^2 
\biggr)^{\!\nicefrac12}
\sum_{k=-\infty}^{m-1} 
3^{sk}
\biggl( 
\avsum_{z\in 3^{k} \Zd\cap \cu_m} 
\bigl| (f)_{z+\cu_{k}} \bigr|^2 
\biggr)^{\!\nicefrac12} 
\,.
\end{align*}
This completes the proof~\eqref{e.general.Besov.pairing}. 

\smallskip

\emph{Step 2.} 
The proof of~\eqref{e.MSP}.
We fix~$g\in H^1(\cu_m)$ with~$\| g \|_{\underline{H}^1(\cu_m)} \leq1$ and apply~\eqref{e.general.Besov.pairing} with~$s=1$ and~$g - (g)_{\cu_m}$ in place of~$g$. Using the Poincar\'e inequality, we have 
\begin{align*}
\sup_{k\in\Z,\,k\leq m} 
3^{-k} 
\biggl( 
\avsum_{z\in 3^{k} \Zd\cap \cu_m} 
\fint_{z+\cu_k} 
\bigl| g -
(g)_{z+\cu_{k}} \bigr|^2 
\biggr)^{\!\nicefrac12}
&
\leq 
C
\sup_{k\in\Z,\,k\leq n} 
\biggl( 
\avsum_{z\in 3^{k} \Zd\cap \cu_n} 
\fint_{z+\cu_k} 
\bigl| \nabla g\bigr|^2 
\biggr)^{\!\nicefrac12}
\notag \\ & 
= 
C \biggl( 
\fint_{\cu_m} 
\bigl| \nabla g\bigr|^2 
\biggr)^{\!\nicefrac12}
\leq C \,.
\end{align*}
Using also that~$|(g)_{\cu_m}| \leq 3^m$, we obtain
\begin{align*}
\biggl| \fint_{\cu_m} f g  \biggr| 
&
\leq 
|(f)_{\cu_m}| |(g)_{\cu_m}| 
+
C \sum_{k=-\infty}^{m-1} 
3^{k}
\biggl( 
\avsum_{z\in 3^{k} \Zd\cap \cu_m} 
\bigl| (f)_{z+\cu_{k}} \bigr|^2 
\biggr)^{\!\nicefrac12}
\notag \\ & 
\leq 
C \sum_{k=-\infty}^{m} 
3^{k}
\biggl( 
\avsum_{z\in 3^{k} \Zd\cap \cu_m} 
\bigl| (f)_{z+\cu_{k}} \bigr|^2 
\biggr)^{\!\nicefrac12}
\,.
\end{align*}
Since this holds for every~$g\in H^1(\cu_m)$ with~$\| g \|_{\underline{H}^1(\cu_m)} \leq1$, we may take the supremum over such~$g$ to obtain~\eqref{e.MSP}.

\smallskip

\emph{Step 3.} Proof of~\eqref{e.MSPw}.
Let~$w \in H^1(\cu_{m})$. Now~\eqref{e.MSPw} follows from~\eqref{e.MSP} if we show that there exists~$C(d)<\infty$ such that 
\begin{equation} \label{e.L2toHminus}
\left\| w - \left( w \right)_{\cu_{m}} \right\|_{\underline{L}^2(\cu_{m})} \leq C \left\|\nabla w\right\|_{\Hminusul(\cu_m)} 
\,.
\end{equation}
To show this, assume that~$\left( w \right)_{\cu_{m}} = 0$ and let~$v\in H^2(\cu_{m})$ solve the Neumann problem 
\begin{equation}
\left\{
\begin{aligned}
& -\Delta v = w & \mbox{in} & \ \cu_{m}, \\
& \mathbf{n} \cdot \nabla v = 0 & \mbox{on} & \ \partial \cu_{m} .
\end{aligned}
\right.
\end{equation}
Then 
$\left\| w \right\|_{\underline{L}^2(\cu_{m})}^2 = \fint_{\cu_{m}} \nabla w \cdot \nabla v$ and 
$\left\| \nabla v\right\|_{\underline{H}^1(\cu_{m})} \leq C \left\| w \right\|_{\underline{L}^2(\cu_{m})}$. Therefore,
\begin{align*}  
\left\| w \right\|_{\underline{L}^2(\cu_{m})}^2 
\leq 
\left\|\nabla w\right\|_{\Hminusul(\cu_m)} \left\| \nabla v\right\|_{\underline{H}^1(\cu_{m})} 
\leq
C \left\|\nabla w\right\|_{\Hminusul(\cu_m)} \left\| w \right\|_{\underline{L}^2(\cu_{m})} 
\,,
\end{align*}
which yields~\eqref{e.L2toHminus} and completes the proof. 
\end{proof}

\subsubsection{A multiparameter version of the Kingman subadditive ergodic theorem}

We give a second ergodic theorem, in which the integral of a stationary random field is replaced by a general \emph{subadditive} function of the bounded subsets of~$\Rd$. This is a multiparameter version of Kingman's subadditive ergodic theorem. 

\smallskip

Let~$\mathcal{L}$ denote the family of bounded open subsets of~$\Rd$. We say that a function
\begin{align*}
\nu:\Omega \times \mathcal{L} \to \R
\end{align*}
is \emph{subadditive} if, for every~$U, U_1,\ldots,U_N \in \mathcal{L}$ such that 
\begin{equation}
\label{e.U.partition}
U_i \cap U_j = \emptyset,  \quad
\bigcup_{i=1}^N U_i \subseteq U, \quad 
\biggl| U \, \setminus \, 
\bigcup_{i=1}^N U_i 
\biggr| = 0
\,, 
\end{equation}
we have that 
\begin{equation}
\label{e.nu.subadd}
\nu(U) \leq \sum_{i=1}^N \frac{|U_i|}{|U|} \nu(U_i) 
\,.
\end{equation}
If~$-\nu$ is subadditive, then we say that~$\nu$ is \emph{superadditive}. We also say that such a quantity is~$\Zd$-stationary if~$\nu(T_y\omega,U) = \nu(\omega, y+ U)$ for every~$y\in\Zd$,~$\omega\in\Omega$ and~$U\in\mathcal{L}$.

\smallskip

Note that we have suppressed the variable~$\omega \in\Omega$ from the notation by writing~$\nu(U)$ instead of~$\nu(\omega,U)$. 
Our definition of subadditivity differs slightly from the one more commonly found in literature because we have added the normalizing volume factors. That is, we call~$\nu(\cdot)$, as above, subadditive, whereas it would be more typical to call~$U \mapsto  |U| \nu(U)$ subadditive. 

\begin{proposition}[Multiparameter Kingman subadditive ergodic theorem]
\label{p.subadditive.ergodic}
Suppose~$\P$ is a~$\Zd$-stationary probability measure on~$(\Omega,\F)$ and~$\Lambda\in (0,\infty)$. Assume that~$\mu$ is a~$\Zd$-stationary, subadditive quantity such that, for every bounded open~$U\subseteq\Rd$, 
\begin{align}
\label{e.subaddmu.bounded}
\mu(U)  \leq \Lambda, \qquad \mbox{$\P$--a.s.}
\end{align} 
Then there exists an extended random variable~$\overline{\mu} \in S_{\mathrm{inv}}(\P)$ valued in~$(-\infty,\Lambda] \cup\{ -\infty\}$ such that, for every bounded open set~$U\subseteq \Rd$, 
\begin{equation}
\label{e.subadditive.ergodic}
\P \Bigl[
\lim_{t\to\infty} 
\mu(tU) = \overline{\mu} 
\Bigr] = 1\,.
\end{equation}
\end{proposition}

The proof of the subadditive ergodic theorem is based on Proposition~\ref{p.ergodic} and the following weak-type maximal inequality for superadditive quantities, which is a generalization of Lemma~\ref{l.maximal.ergodic}.

\begin{lemma}
\label{l.maximal.superadditive}
There exists~$C(d)<\infty$ such that, for any~$\Zd$--stationary, superadditive quantity~$\nu(\cdot)$ satisfying~$\P [ \nu(\cu_n)\geq 0] = 1$ for~$n\in\N$ we have, for every~$\lambda>0$, 
\begin{align}
\label{e.maximal.superadditive}
\P \left[ \sup_{m\in\N} \nu(\cu_m) > \lambda \right] 
\leq 
\frac{C}{\lambda} 
\sup_{m\in\N} \E \left[ \nu(\cu_m) \right]
\,.
\end{align}
\end{lemma}
\begin{proof}

Let~$m,n \in\N$ with~$n<m$. Consider the set 
\begin{equation*}
A:=
\Bigl\{
z \in \left( -\tfrac12 (3^m-3^n) , \tfrac12 (3^m-3^n) \right)^d \cap \Zd \,:\, \sup_{1\leq k \leq n} \nu (z+\cu_k) \geq \lambda
\Bigr\}
\,.
\end{equation*}
By the Vitali covering lemma, there exists a finite, disjoint collection~$\{ z_i + \cu_{k_i} \,:\, 1\leq i \leq N\}$ such that, for every~$i\in \{1,\ldots,N\}$, we have~$z_i \in A$,~$1\leq k_i \leq n$, 
\begin{equation*}
\nu (z_i+\cu_{k_i}) \geq \lambda
\end{equation*}
and, for some constant~$C(d)<\infty$, 
\begin{equation*}
|A|
\leq C
\sum_{i=1}^N 
|\cu_{k_i}|
\,. 
\end{equation*}
By the superadditivity, 
\begin{equation*}
|\cu_m| \nu(\cu_m) 
\geq 
\sum_{i=1}^N 
|\cu_{k_i}| \nu(z_i + \cu_{k_i}) 
\geq 
\sum_{i=1}^N \lambda |\cu_{k_i}|  
\geq 
c \lambda |A|,\qquad \P\text{--a.s.}
\end{equation*}
By stationary, we therefore obtain
\begin{equation*}
(3^m-3^n)^{d} \,
\P \left[ \sup_{1\leq k \leq n} \nu(\cu_k) \geq \lambda \right]
=
\E\left[ |A| \right] 
\leq 
\frac{C|\cu_m| }{\lambda} \E \left[ \nu(\cu_m) \right]
= 
C 3^{md} \lambda^{-1} \E \left[ \nu(\cu_m) \right]
\,.
\end{equation*}
Dividing by~$3^{md}$ and sending~$m\to \infty$, we get 
\begin{equation*}
\P \left[ \sup_{1\leq k \leq n} \nu(\cu_k) \geq \lambda \right]
\leq
C \lambda^{-1}  \sup_{m\in\N} \E \left[ \nu(\cu_m) \right]
\,.
\end{equation*}
Sending~$n\to \infty$ now yields the lemma. 
\end{proof}

\begin{proof}[{Proof of Proposition~\ref{p.subadditive.ergodic}}]
We may suppose that~$\mu$ is bounded below by a constant,~$\P$--a.s. Otherwise we obtain the result for each of the cutoff quantities~$\max\{ \mu ,  k\}$ with~$k\in\R$ and then send~$k\to -\infty$. We therefore proceed under the assumption that,  for some~$\tilde{\Lambda}\in [0,\infty)$, we have~$-\tilde{\Lambda} \leq \mu(U) \leq \Lambda$ for every bounded open subset~$U\subseteq\Rd$,~$\P$--almost surely.

\smallskip

\emph{Step 1.} 
The construction of~$\overline{\mu}$.
For~$n\in\N$, let~$f_n$ denote the random field 
\begin{equation*}
f_n:= \sum_{z\in 3^n\Zd} \mu(z+\cu_n)  \indc_{z+\cu_n}.
\end{equation*}
Observe that~$f_n$ is stationary with respect to~$3^n\Zd$--translations and~$|f_n(x)|\leq\Lambda \vee \tilde{\Lambda}$, and therefore~$|f_n(x)|$ has expectation bounded uniformly over~$x\in\Rd$. In particular, we have that the random field~$x\mapsto f_n(3^{n}x)$ belongs to~$\tilde{S}^1(\P)$, which is the analog of~$S^1(\P)$ except that stationarity is defined with respect to the group action~$\{ T_{3^n y} \,:\, y\in\Rd\}$ instead of the usual one~$\{ T_y \,:\, y\in\Rd\}$. Note that~$\tilde{S}^1_{\mathrm{inv}}(\P)=S^1_{\mathrm{inv}}(\P)$.
By  subadditivity, for any~$m\geq n$ and~$z\in 3^n\Zd$, 
\begin{align}
\label{e.subbaddtof}
\fint_{z+\cu_m} f_n
= \avsum_{z' \in 3^n\Zd \cap (z+\cu_m)} 
\mu(z'+\cu_n)
\geq 
\mu(z+\cu_m).
\end{align}
In fact, for every~$n\leq m \leq l$ and~$z\in 3^m\Zd$, 
\begin{align}
\label{e.orderfnfm}
\fint_{z+\cu_l} ( f_n - f_m )
\geq 
0
\,.
\end{align}
By Proposition~\ref{p.ergodic}, for each~$n\in\N$, there exists~$\overline{\mu}_n\in S^1_\inv(\P)$ such that, for every bounded open subset~$U\subseteq \Rd$, 
\begin{equation}
\label{e.fnconvergemun}
\limsup_{t\to \infty} \
\left|\fint_{tU} f_n- \overline{\mu}_n \right| 
=0,
\quad\mbox{$\P$--a.s.}
\end{equation}
It is evident from~\eqref{e.orderfnfm} that, for every~$n\leq m$, 
\begin{equation}
\label{e.overlinemumn.mono}
\overline{\mu}_n
\geq 
\overline{\mu}_m 
\quad \mbox{$\P$--a.s.}
\end{equation}
and, also using stationarity, 
\begin{equation*}
\langle \overline{\mu}_n - \overline{\mu}_m \rangle
\leq 
\E \left[ \mu(\cu_n) \right] 
-
\E \left[ \mu(\cu_m) \right].
\end{equation*}
By stationarity and subadditivity, the map~$n\mapsto \E[\mu(\cu_n)]$ is non-increasing in~$n$ and therefore has a limit,
\begin{equation}
\label{e.LEcunlim}
L:= \inf_{n\in\N} \E[\mu(\cu_n)] = \lim_{n\to \infty} \E[\mu(\cu_n)]. 
\end{equation}
We deduce the existence of~$\overline{\mu} \in S^1_\inv(\P)$ such that~$\langle \left| \overline{\mu}_n - \overline{\mu} \right| \rangle \to 0$ as~$n \to \infty$. By~\eqref{e.overlinemumn.mono},
\begin{equation}
\label{e.overlinemulim}
\overline{\mu} = \lim_{n\to \infty} \overline{\mu}_n, \quad \mbox{$\P$--a.s.}
\end{equation}

\emph{Step 2.}
We argue that, for every bounded open subset~$U\subseteq\Rd$, 
\begin{align}
\label{e.limsupdown}
\limsup_{t\to \infty} \mu(t U) \leq \overline{\mu}, 
\quad \mbox{$\P$--a.s.}
\end{align}
Fix~$U$ and define, for each~$\ep>0$ and~$n\in\N$, \begin{align*}
V_{n,\ep} := \bigcup \left\{ z+\ep \cu_n \,:\, z\in \ep3^n \Zd , z+\ep\cu_n \subseteq U \right\}. 
\end{align*}
Then~$|U \setminus V_{n,\ep}| \to 0$ as~$\ep \to 0$ for every fixed~$n\in\N$. 
By subadditivity and~\eqref{e.subaddmu.bounded},  
\begin{equation*}
\mu(tU)
\leq 
\frac{|tV_{n,1}|}{|tU|} \mu(tV_{n,1})+ \frac{|tU \setminus tV_{n,1}|}{|tU|}  \mu(t(U\setminus V_{n,1}))
\leq
\frac{|V_{n,1/t}|}{|U|} \fint_{tV_{n,1}} f_n +  \frac{|U \setminus V_{n,1/t}|}{|U|} \Lambda
\,.
\end{equation*}
Therefore, by~\eqref{e.fnconvergemun},
\begin{equation*}
\limsup_{t\to \infty} \mu(tU) 
\leq 
\limsup_{t\to \infty}
\frac{|V_{n,1/t}|}{|U|} \fint_{tV_{n,1}} f_n 
+ 
\limsup_{t\to \infty} \frac{|U \setminus V_{n,1/t}|}{|U|} \Lambda
=
\overline{\mu}_n
\,,
\qquad 
\mbox{$\P$--a.s.}
\end{equation*}
Sending~$n\to \infty$ yields~\eqref{e.limsupdown}. 

\smallskip

\emph{Step 3.}
We next argue that 
\begin{equation}
\label{e.liminfup}
\overline{\mu}
\leq 
\liminf_{t\to \infty} \mu(t U) , 
\quad \mbox{$\P$--a.s.}
\end{equation}
Let~$\lambda,\ep>0$. Define, for every~$n \in \N$, the superadditive process
\begin{equation*}
\nu_n(U):= \biggl( - \mu(U) + \fint_U f_n \biggr) \vee 0 \,.
\end{equation*}
Then, by the subadditivity and stationarity of~$\mu$, 
\begin{align*}  
\E \left[ \nu_n(\cu_m) \right] = \E \left[ \biggl(  - \mu(\cu_m) + \avsum_{z \in 3^n \Z^d \cap \cu_m} \mu(z+\cu_n)
\biggr) \vee 0 \right] = \E \left[\mu(\cu_n)\right] -  \E \left[\mu(\cu_m)\right]  .
\end{align*}
Thus we may take~$n\in\N$ large enough so that~$\nu_n$ satisfies~$\sup_{m\geq n}\E \left[ \nu_n(\cu_m) \right] \leq \ep$. 
Applying Lemma~\ref{l.maximal.superadditive}, we find that 
\begin{equation*}
\P \left[ 
\sup_{m\geq n} \nu_n (\cu_m) \geq \lambda
\right] 
\leq 
\frac{C \ep}{\lambda}. 
\end{equation*}
Since~$\nu_n$ is superadditive and nonnegative, for any~$V \subseteq U$, 
\begin{align*}
\nu_n(V) \leq \frac{|U|}{|V|} \nu_n(U). 
\end{align*}
We deduce that, for every bounded set~$U$ with~$|U|>0$ and~$U \subset \cu_k$ for some~$k \in \N$, 
\begin{align*}
\P \left[ 
\inf_{t\geq 3^n} \left( \mu(tU) - \fint_{tU} f_n \right) \leq - \lambda
\right] 
& =
\P \left[ 
\sup_{t\geq 3^n} \nu_n(tU) \geq \lambda
\right] 
\\ & \leq 
\P \left[ 
\sup_{m\geq n} \nu_n(\cu_m) \geq \frac{|U|}{|\cu_{k+1}|} \lambda
\right] 
\leq
C \frac{|\cu_k|}{|U|} \frac{\ep}{\lambda}\,.
\end{align*}
This implies that
\begin{equation*}
\liminf_{t\to \infty} 
\left( \mu(tU) - \fint_{tU} f_n \right)
\geq 0\,,\qquad \mbox{$\P$--a.s.}
\end{equation*}
In view of~\eqref{e.fnconvergemun}, we therefore obtain
\begin{equation*}
\liminf_{t\to \infty} \mu(tU) 
\geq \overline{\mu}_n\,,\qquad \mbox{$\P$--a.s.}
\end{equation*}
Sending~$n\to \infty$ yields~\eqref{e.liminfup} and completes the proof. 
\end{proof}

\subsection{Qualitative homogenization with stationary random coefficients}
\label{ss.random}

In this section, we intend to generalize the homogenization theory we initiated in Section~\ref{s.periodic} for periodic coefficients to the case of random coefficients. 
Fortunately, most of the computations and the insights we made in the previous chapter remain intact! 
Compared to the periodic case, the primary additional complexities involve setting up the proper abstract framework and constructing correctors with suitable properties. 

\smallskip

The structure of the correctors lies at the core of homogenization theory. In the periodic analysis in Section~\ref{s.periodic}, a lot of heavy lifting was performed by the very simple assertion ``solutions~$\phi_e$ of~\eqref{e.corr.prob} exist and are bounded (in the sense of~\eqref{e.correctorsbounded2}) by periodicity.'' This is what makes periodic homogenization relatively easy. In the stochastic setting, such a naive statement is no longer true. Indeed, much of the effort in building a good theory of stochastic homogenization involves compensating for its absence. 

\smallskip

\subsubsection{The formal setup}

Before discussing the homogenization of elliptic equations with random coefficients, we must formalize a general notion of the ``random coefficient field.'' We want to imagine that our coefficient field~$\a(\cdot)$ is selected at random from the ``set of all coefficient fields''~$\Omega$ by a probability measure~$\mathbb{P}$ (the law of~$\a(\cdot)$). The field~$\a(\cdot)$ was given to us by this probability measure, but now that we have it, we can study the elliptic operator~$-\nabla \cdot \a\nabla$, compute its solutions, and so forth. All of the solutions and quantities we compute will depend on the realization of~$\a(\cdot)$ itself, and so these will also be random objects, functions of~$\a(\cdot)$. We can compute their expectations with respect to~$\P$ and the probability of certain associated events.

\smallskip

Given a dimension~$d$ and ellipticity constants~$0 < \lambda\leq \Lambda <\infty$, we define~$\Omega(d,\lambda,\Lambda)$ to be the set of all coefficient fields, namely the set of measurable functions~$\a:\Rd \to \R^{d\times d}$ satisfying~\eqref{e.ue}. That is, 
\begin{equation}
\label{e.Omega.dLambda}
\Omega(d,\lambda,\Lambda) : = 
\Bigl\{
\a \in L^\infty(\Rd;\R^{d\times d}) \,:\, \forall e \in\Rd, \ 
\lambda |e|^2 \leq e\cdot \a  e \ \ \mbox{and}
\ 
e \cdot \a^{-1} e 
\geq \Lambda^{-1}  |e|^2 
 \Bigr\}.
\end{equation}
Since we think of~$d$,~$\lambda$ and~$\Lambda$ as being fixed global parameters, we typically keep the dependence on them implicit and write~$\Omega$ in place of~$\Omega(d,\lambda,\Lambda)$. 
The natural~$\sigma$--algebra on~$\Omega$ is the one generated by the random variables 
\begin{equation}
\label{e.rvF}
\Omega \ni \a  \mapsto \int_{\Rd} \varphi(x)\a_{ij}(x)\,dx, \quad \varphi\in C^\infty_{\mathrm{c}}(\Rd), \ i,j\in\{1,\ldots,d\}.
\end{equation}
We denote this~$\sigma$--algebra by~$\F$. 
For each~$y\in\Rd$, we let~$T_y:\Omega \to \Omega$ denote the translation operator~$T_y\a := \a(\cdot+y)$.
We extend~$T_y$ to~$\F$ by defining~$T_y E := \{ T_y\a \,:\,\a\in E\}$ for each~$E\in\F$ and to random elements by~$(T_y\xi)(\a) := \xi(T_y\a)$. Analogously, given an open set~$U$, we denote by~$\mathcal{F}(U)$ the sigma algebra generated by the random variables in~\eqref{e.rvF} with~$\varphi \in C^\infty_{\mathrm{c}}(U)$. 

\smallskip

As in the previous section, we always assume that the probability measure~$\P$ we put on~$(\Omega,\F)$ must be \emph{$\Zd$--stationary} with respect to~$\{ T_z \,:\, z\in \Zd\}$:
that is, 
\begin{equation}
\label{e.Zdstat1}
\qquad
\P \circ T_z  = \P, \quad \forall z\in\Zd.
\end{equation}
The \emph{canonical} element of~$\Omega$ is the random field~$(\a,x) \mapsto \a(x)$. Note that the probability space~$(\Omega,\F,\P)$ satisfies the assumptions of the previous section, and we will continue to employ all of the definitions, notation and terminology we introduced there.
We remark that this section only uses the stationarity of~$\P$ (ergodicity is not required).

\smallskip

All random variables and fields will depend on the underlying coefficient field~$\a$, and we typically do \emph{not} display this dependence, as displaying it would be useless, ugly and cumbersome. 

\smallskip

We will give many examples of probability measures~$\P$ satisfying our assumptions in Chapter~\ref{s.CFS}. Until then, the reader can have in mind a ``random checkerboard'' in which independent but identical coins are tossed to determine each (unit) checkerboard square, and~$\a$ is a constant matrix in each square, say~$\Id$ (in white squares) or~$2 \Id$ (in black squares). The resulting probability measure will be both~$\Zd$--stationary and ergodic; in fact, it will possess a finite range of dependence, a much stronger condition than ergodicity. 

\smallskip

Many statements can be called ``homogenization.'' We next present one such statement, concerning the Dirichlet problem. As usual, for each~$\ep>0$, we denote~$\a^\ep:=\a(\frac\cdot\ep)$.

\begin{theorem}[Qualitative homogenization,~$\Z^d$--stationary random coefficients]
\label{t.qualitative.homogenization}
Suppose that~$\P$ is a~$\Zd$--stationary probability measure on~$(\Omega,\F)$.
Then there exists a random element~$\ahom$ of~$\R^{d\times d}$ which satisfies 
\begin{align*}
e\cdot \ahom e \geq \lambda |e|^2 \quad \mbox{and} \quad 
e \cdot \ahom^{-1} e 
\geq \Lambda^{-1}  |e|^2, \quad \forall e\in\Rd, \ \P\mbox{--a.s.},
\end{align*}
such that if, for each~$\ep>0$, bounded Lipschitz domain~$U\subseteq\Rd$,~$\mathbf{f} \in L^2(U;\Rd)$ and~$g\in H^1(U)$, we denote by~$u^\ep_{U,\f,g}$ and~${u}_{U,\f,g}$, respectively, the solutions of
\begin{align*}
\left\{
\begin{aligned}
& -\nabla \cdot \a^\ep \nabla u^\ep_{U,\f,g} = \nabla \cdot \f & \mbox{in} & \ U, \\
& u^\ep_{U,\f,g} = g & \mbox{on} & \ \partial U,
\end{aligned}
\right.
\quad \mbox{and} \quad 
\left\{
\begin{aligned}
& -\nabla \cdot \ahom\nabla u_{U,\f,g} = \nabla \cdot \f & \mbox{in} & \ U, \\
& u_{U,\f,g} = g & \mbox{on} & \ \partial U,
\end{aligned}
\right.
\end{align*}
then we have  
\begin{equation}
\label{e.qualhomogP}
\P \left[
\sup_{U,\f,g} 
\limsup_{\ep \to 0}\,
\Bigl(
\bigl\| \nabla u^\ep_{U,\f,g} {-} \nabla u_{U,\f,g} \bigr\|_{H^{-1}(U)}
+
\bigl\| \a^\ep \nabla u^\ep_{U,\f,g} {-}\; \ahom \nabla u_{U,\f,g} \bigr\|_{H^{-1}(U)}
\Bigr)
=0
\right] 
= 
1\,.
\end{equation}
If, in addition,~$\P$ is ergodic, then the random matrix~$\ahom$ is deterministic. 
\end{theorem}

\subsubsection{Construction of first-order correctors}

The first-order correctors will be constructed by a variant of the Riesz representation theorem. This requires the identification of a suitable Hilbert space. We let~$V^2(\P) = S^2(\P)^d$ denote the linear space of~$\Rd$--valued stationary random fields~$\f$ satisfying~$\left\langle |\f|^2 \right\rangle < \infty$. Observe that~$V^2(\P)$ is a Hilbert space with respect to the inner product~$(\f,\g) \mapsto \left\langle \f\cdot \g \right\rangle$. We also define the following closed subspaces of~$V^2(\P)$:
\begin{align*}
V^2_{\mathrm{inv}}(\P) := & \ S^2_{\mathrm{inv}}(\P)^d = 
\big\{
\f \in V^2(\P)\,:\, T_y \f = \f, \quad \forall y\in\Rd 
\big\}, 
\\ 
V^2_\pot(\P) 
:= & \
\big\{
\f \in V^2(\P)\,:\, \partial_{x_i} \f_j = \partial_{x_j} \f_i, \quad \forall i,j\in\{1,\ldots,d\}
\big\}, 
\\ 
V^2_\sol(\P) := &  \
\big\{
\f \in V^2(\P)\,:\, \nabla \cdot \f = 0
\big\}.
\end{align*}
The conditions above are to be understood in the~$\P$--almost sure sense, and the spatial derivatives in the sense of distributions; for instance,~$\f \in V^2(\P)$ belongs to~$\f \in V^2_{\mathrm{pot}}$ if 
\begin{equation*}
\P \biggl[
\int_{\Rd} \partial_{x_i} \varphi \f_j \,dx
=
\int_{\Rd} \partial_{x_j} \varphi  \f_i \,dx, \quad
\forall i,j\in \{1,\ldots,d\},\  \varphi \in C^\infty_{\mathrm{c}}(\Rd)
\biggr] = 1. 
\end{equation*}
We call the space~$V^2_{\inv}(\P)$/$V^2_{\pot}(\P)$/$V^2_{\sol}(\P)$ the space of \emph{stationary random shift invariant/potential/solenoidal fields}, respectively. It is clear that 
\begin{equation*}
V^2_{\inv} (\P) \subseteq V^2_\pot(\P) \cap V^2_\sol(\P),
\end{equation*}
as the set~$V^2_{\inv}(\P)$ consists of the set of random fields which are constant in space~$\P$-almost surely; in other words, the elements of~$V^2_{\mathrm{inv}}(\P)$ are random elements of~$\Rd$. If~$\P$ is assumed to be ergodic, then~$V^2_{\inv}(\P)$ can be identified with~$\Rd$.
We next define the following subspaces of~$V^2(\P)$:
\begin{equation*}
\left\{
\begin{aligned}
& V^2_0(\P) := 
(V^2_{\inv}(\P))^\perp 
= 
\big\{
\f \in V^2(\P) \,:\, 
\langle \f\cdot \g \rangle = 0, \ \forall \g\in V^2_{\inv}(\P)
\big\}, 
\\
& V^2_{\pot,0}(\P)
:= 
V^2_{\pot}(\P)
\cap V^2_0(\P),
\\  &
V^2_{\sol,0}(\P)  
:= 
V^2_{\sol}(\P)
\cap V^2_0(\P).
\end{aligned}
\right.
\end{equation*}
We sometimes abuse notation by denoting an element of~$V^2_\pot(\P)$ by~$\nabla f$, with the understanding that the function~$f$ is not necessarily stationary and, indeed, is only defined up to constants (even if~$\nabla f \in V^2_{\pot,0}(\P)$).

\begin{proposition}[Existence of first-order correctors]
\label{p.qual.correctors}
Assume~$\P$ is a~$\Zd$--stationary probability measure on~$(\Omega,\F)$. For each~$e\in \Rd$, there exists a unique potential field~$\nabla \phi_e \in V^2_{\pot,0}(\P)$ such that~$\a(e+\nabla \phi_e)\in V^2_\sol(\P)$. 
\end{proposition}

The proof of Proposition~\ref{p.qual.correctors} relies on the following version of the Helmholtz-Hodge decomposition for stationary random fields. 

\begin{lemma}[Stationary Helmholtz-Hodge decomposition]
\label{l.stat.HH}
Suppose~$\P$ is a~$\Zd$--stationary probability measure on~$(\Omega,\F)$.
Then 
\begin{equation}
V^2(\P)
=
V^2_{\inv}(\P) \oplus
V^2_{\pot,0}(\P)
\oplus 
V^2_{\sol,0}(\P)
\,.
\end{equation}
\end{lemma}
\begin{proof}
The proof comes in two steps. 
We will argue that 
\begin{equation}
\label{e.statHH.yesperp}
V^2_{\pot,0}(\P) \perp V^2_\sol(\P).
\end{equation}
and
\begin{equation}
\label{e.statHH.perpcont}
(V^2_{\pot,0}(\P))^\perp \subseteq V^2_\sol(\P)
\,.
\end{equation}
Together~\eqref{e.statHH.yesperp} and~\eqref{e.statHH.perpcont} imply that~$(V^2_{\pot,0}(\P))^\perp = V^2_\sol(\P)$. Since~$V^2_\sol(\P) = V^2_{\inv}(\P) \oplus V^2_{\sol,0}(\P)$ by definition, this yields the lemma. 

\smallskip

\emph{Step 1.} We prove~\eqref{e.statHH.yesperp}. 
Pick~$\g\in V^2_{\sol}(\P)$ and~$\nabla f\in V^2_{\pot,0}(\P)$ with~$\langle |\nabla f|^2\rangle=\langle |\g|^2\rangle =1$. Let~$m,n\in\N$ and~$\chi \in C^\infty_{\mathrm{c}}(\Rd)$ be such that~$n\leq m$ and  
\begin{align*}
\indc_{\cu_m} \leq \chi \leq \indc_{\cu_m+\cu_n}, 
\qand
\left\| \nabla \chi \right\|_{L^\infty(\Rd)} 
\leq 
C3^{-n}.
\end{align*}
Compute 
\begin{align*}
0
&
=
\int_{\Rd} \nabla (\chi (f- (f)_{\cu_{m+1}} ) ) \cdot \g
\\ & 
=
\int_{\Rd} \chi \nabla f \cdot \g 
+
\int_{\Rd} 
(f- (f)_{\cu_{m+1}} ) (\nabla \chi \cdot \g )
\\ &
=
\int_{\cu_m} \nabla f \cdot \g 
+
\int_{(\cu_m+\cu_n)\setminus\cu_m}
\chi\nabla f \cdot \g
+
\int_{\Rd} 
 (f- (f)_{\cu_{m+1}} ) (\nabla \chi \cdot \g )
 \,.
\end{align*}
Using H\"older's inequality, we estimate
\begin{align*}
\E\biggl[ \,
\biggl| \int_{(\cu_m+\cu_n)\setminus\cu_m}
\chi\nabla f \cdot \g \biggr| \,
\biggr]
&
\leq 
\E\biggl[ 
\int_{(\cu_m+\cu_n)\setminus\cu_m}
\left| \nabla f \right|^2  \biggr]^{\nicefrac12} 
\E \biggl[ 
\int_{(\cu_m+\cu_n)\setminus\cu_m}
\left| \g \right|^2 \biggr]^{\nicefrac12}
\\ & 
\leq 
C 3^{n-m} |\cu_m|
\left\langle |\nabla f |^2\right\rangle^{\nicefrac12}
\left\langle |\g |^2\right\rangle^{\nicefrac12}
=
C3^{n-m} |\cu_m|
\end{align*}
and
\begin{align*}
\E\left[
\left| \int_{\Rd} 
(f- (f)_{\cu_{m+1}} ) (\nabla \chi \cdot \g )
\right| \right]
&
\leq 
C3^{-n} 
\E\left[
\int_{\cu_m+\cu_n}
\left| f- (f)_{\cu_{m+1}} \right|^2 \right]^{\nicefrac12}
\E \left[ 
\int_{\cu_m+\cu_n}
\left| \g \right|^2 \right]^{\nicefrac12}
\\ & 
\leq
C3^{m-n} |\cu_m|
\E\biggl[
3^{-2m}
\fint_{\cu_{m+1}}
\left| f- (f)_{\cu_{m+1}} \right|^2 \biggr]^{\nicefrac12}.
\end{align*}
We deduce that 
\begin{align*}
\left| \left\langle \nabla f \cdot \g \right\rangle \right|
&
=
\biggl| \E\biggl[ \fint_{\cu_m} \nabla f \cdot \g \biggr] \biggr|
\leq 
C
\Bigl( 
3^{m-n}
\E\left[
3^{-2m}
\left\| 
f- (f)_{\cu_{m+1}} \right\|_{\underline{L}^2(\cu_{m+1})}^2 \right]^{\nicefrac12}
+
C 3^{n-m}
\Bigr)\,.
\end{align*}
By Corollary~\ref{c.stat.Hminusone},~\eqref{e.L2toHminus}, 
and the assumption~$\nabla f\in V^2_0(\P)$, we have that 
\begin{align*}
\limsup_{m\to \infty} 
3^{-m}
\left\| 
f- (f)_{\cu_{m+1}} \right\|_{\underline{L}^2(\cu_{m+1})}
=0, \quad
\mbox{$\P$--a.s.}
\end{align*}
Therefore by taking~$n=m-k$ for~$k\in\N$ fixed and sending~$m\to \infty$,
we obtain
\begin{align*}
\left| \left\langle \nabla f \cdot \g \right\rangle \right|
\leq 
C3^{-k}. 
\end{align*}
Sending~$k\to \infty$ yields the claim. 

\smallskip

\emph{Step 2.}
We show that, for every~$\g\in (V^2_{\pot,0}(\P))^\perp$ and~$\chi\in C^\infty_{\mathrm{c}}(\cu_0)$, 
\begin{equation}
\label{e.stat.HH.step1}
\P \left[ \int_{\cu_0} \g \cdot \nabla \chi =0 \right] = 1
\,. 
\end{equation}
Let~$E_\pm$ denote the event~$\{ \pm \int_{\cu_0} \g \cdot \nabla \chi \geq 0\}$ and observe that
\begin{equation}
\f_{\pm}(\a,x)
:=
\sum_{z\in\Zd} \indc_{E_{\pm}}(T_z\a) \indc_{z+\cu_0}(x) \nabla \chi(x+z)
\end{equation}
defines stationary random fields~$\f_+, \f_-\in V^2_{\pot,0}(\P)$. Using that~$\g\in (V^2_{\pot,0}(\P))^\perp$, we obtain 
\begin{equation}
\E \left[ \biggl| \int_{\cu_0} \g\cdot \nabla \chi \biggr| \right]
=
\bigl\langle \g \cdot (\f_{+}- \f_{-}) \bigr\rangle
=
0
\,. 
\end{equation}
This yields the claim~\eqref{e.stat.HH.step1}. 

\smallskip

\emph{Step 3.} We complete the proof of~\eqref{e.statHH.perpcont} by reducing it to Step~2. Select a countable subset~$\{ \chi_n\}_{n\in\N} \subseteq L^2(\Rd)$ of smooth functions such that the closure of~$\spn \chi_n$ in~$L^2(\Rd)$ is equal to~$L^2(\Rd)$ and~$\diam(\supp(\chi_n )) < 1$ for every~$n\in\N$. For every~$y\in\Rd$, the translation operator~$T_y$ is a unitary operator on~$V^2(\P)$. Therefore we deduce from~\eqref{e.stat.HH.step1} that, for every~$\g\in (V^2_{\pot,0}(\P))^\perp$ and~$n\in\N$, 
\begin{equation}
\P \left[ \int_{\Rd} \g \cdot \nabla \chi_n =0 \right] = 1 \,. 
\end{equation}
Using the countable additivity of~$\P$, we obtain, for every~$\g\in (V^2_{\pot,0}(\P))^\perp$,
\begin{equation}
\P \left[  \int_{\Rd} \g \cdot \nabla \chi_n =0, \forall n\in\N \right] = 1 \,. 
\end{equation}
This implies that~$\g \in V^2_\sol(\P)$. 
\end{proof}

\begin{proof}[{Proof of Proposition~\ref{p.qual.correctors}}]
By the Lax-Milgram lemma, for each~$e\in\Rd$, there exists a unique~$\nabla \phi_e \in V^2_{\pot,0}$ satisfying 
\begin{equation}
\label{e.LaxMilg}
\left\langle \g \cdot \a \nabla \phi_e \right\rangle 
= 
-\langle \g \cdot \a e \rangle, \quad \forall \g \in V^2_{\pot,0}.
\end{equation}
The condition~\eqref{e.LaxMilg} is equivalent to~$\a(e+\nabla \phi_e) \in (V^2_{\pot,0}(\Omega))^\perp$.
Lemma~\ref{l.stat.HH} asserts that~$(V^2_{\pot,0}(\Omega))^\perp = V^2_\sol(\Omega)$.
\end{proof}

With the proof of Proposition~\ref{p.qual.correctors} now complete, we have succeeded in extending the well-posedness of the corrector problem~\eqref{e.corr.prob} to the stationary random setting. Let us now define the homogenized matrix.

\begin{definition}[{Homogenized matrix~$\ahom$}]
For each~$e\in\Rd$, we define~$\ahom e$ to be the projection of the flux~$\a(e+\nabla\phi_e) \in V^2_\sol(\P)$ onto~$V^2_{\inv}(\P)$. In other words,~$\ahom e$ is the unique element of~$V^2_{\inv}(\P)$ such that
\begin{equation}
\label{e.ahom.def.sec2}
\a(e+\nabla\phi_e) - \ahom e \in V^2_{\sol,0}(\P).
\end{equation}
Owing to the evident linearity of the map~$e\mapsto \nabla \phi_e$, we may identify~$\ahom$ with an element of~$(V^2_{\inv}(\P))^{d} = S^2_{\inv}(\P)^{d\times d}$ and thus a random element of~$\R^{d\times d}$.
\end{definition}

Observe that, by~\eqref{e.ergodicityduh}, 
\begin{equation}
\mbox{$\P$ is ergodic in the sense of~\eqref{e.Zdergodic}}
\implies
\mbox{$\ahom$ is deterministic.}
\end{equation}
However,~$\ahom$ need not be deterministic, in general, for non-ergodic~$\Zd$-stationary probability measures~$\P$. 

\smallskip

We next check that the homogenized matrix~$\ahom$ satisfies the uniform ellipticity condition~\eqref{e.ue}, that is, 
\begin{equation}
\label{e.ahom.boundsme}
\P \bigl[  
e\cdot \ahom e \geq \lambda |e|^2
\ \ \mbox{and} \ \
e\cdot \ahom^{-1} e 
\geq 
\Lambda^{-1} |e|^2,
\quad \forall e\in\Rd
\bigr]
= 1.
\end{equation}
Observe that~\eqref{e.LaxMilg} and~\eqref{e.ahom.def.sec2} imply that, for every nonnegative~$f \in S^\infty_{\mathrm{inv}}(\P)$, 
\begin{align*}
\lambda |e|^2 \bigl\langle f \bigr\rangle
\leq 
\lambda \,\bigl\langle f |e+\nabla\phi_e|^2\bigr\rangle 
\leq
\bigl\langle f (e+\nabla\phi_e) \cdot \a (e+\nabla \phi_e) \bigr\rangle
=
\bigl\langle f e \cdot \a (e+\nabla \phi_e) \bigr\rangle
=
\bigl\langle f e \cdot \ahom e \bigr\rangle\,.
\end{align*}
Since this is valid for every~$f \in S^\infty_{\mathrm{inv}}(\P)$, we deduce that~$\P [e\cdot \ahom e \geq \lambda |e|^2] = 1$. 
Likewise, if we set~$\g_e := \a(e+\nabla\phi_e) - \ahom e \in V^2_{\sol,0}(\P)$, then we have 
\begin{align*}
\Lambda^{-1} \bigl \langle f | \ahom e|^2 \bigr\rangle
\leq 
\Lambda^{-1} 
\bigl \langle f | \ahom e + \g_e |^2 \bigr\rangle
&
\leq 
\bigl \langle f ( \ahom e + \g_e ) \cdot \a^{-1} ( \ahom e + \g_e ) \bigr\rangle
\\ & 
=
\bigl \langle f ( e+\nabla \phi_e ) \cdot \a ( e + \nabla \phi_e ) \bigr\rangle
= \bigl\langle f e \cdot \ahom e \bigr\rangle\,.
\end{align*}
This is equivalent to~$\Lambda^{-1} \bigl \langle f | e|^2 \bigr\rangle \leq \bigl\langle f e \cdot \ahom^{-1} e \bigr\rangle$.
Since this holds for every~$f \in S^\infty_{\mathrm{inv}}(\P)$ we deduce that~$\P [ e\cdot \ahom^{-1} e 
\geq 
\Lambda^{-1} |e|^2 ] = 1$. 
 
\smallskip 
 
The flux correctors are obtained as a consequence of the following statement on the construction of stream matrices for given stationary solenoidal vector fields. Similar to the case of potential fields, the stream matrix has a stationary gradient but is itself only well-defined modulo the addition of constant skew-symmetric matrices. 

\begin{lemma}[Stream matrices for~$V^2_{\sol,0}(\P)$]
\label{l.stationary.stream}
For every~$\g\in V^2_{\sol,0}(\P)$, there exists
$$\{ \nabla \mathbf{k}_{ij} \}_{i,j\in\{1,\ldots,d\}} \subseteq (V^2_{\pot,0}(\P))$$
such that
\begin{equation}
\label{e.stationary.stream.skew}
\nabla \mathbf{k}_{ij} = -\nabla \mathbf{k}_{ji}, \quad \forall i,j\in\{1,\ldots,d\}
\end{equation}
and 
\begin{equation}
\label{e.divkg}
\nabla \cdot \mathbf{k} = \g. 
\end{equation}
Moreover, the map~$\g \mapsto \nabla \mathbf{k}$ is a bounded linear operator from~$V^2_{\sol,0}(\P)$ into the~$d$-by-$d$ skew symmetric matrices with entries in~$V^2_{\pot,0}(\P)$.  
\end{lemma}
\begin{proof}
Let~$\g = (\g_1,\ldots,\g_d) \in V^2_{\sol,0}(\P)$. The Riesz representation theorem yields, for each~$i,j\in\{1,\ldots,d\}$, a unique~$\nabla \mathbf{k}_{ij}\in V^2_{\pot,0}(\P)$ satisfying
\begin{equation}
\label{e.RRT.stream}
\left\langle \nabla \psi\cdot \nabla \mathbf{k}_{ij}  \right\rangle 
=
\left\langle 
\g_{i} \partial_{x_j} \psi - \g_j \partial_{x_i} \psi
\right\rangle, \quad \forall \psi\in V^2_{\pot,0}(\P).
\end{equation}
It is clear that~\eqref{e.stationary.stream.skew} holds. We have, in the sense of distributions,~$\P$--a.s.,  
\begin{equation*}
-\Delta ( \nabla \cdot \mathbf{k} )_j 
= 
-\Delta\left(\sum_{i=1}^d \partial_{x_i} \mathbf{k}_{ij} \right)
=
\sum_{i=1}^d\partial_{x_i} \left( \partial_{x_j} \g_i - \partial_{x_i} \g_j \right)
=
-\Delta \g_j.
\end{equation*}
Therefore, the entries of~$\nabla \cdot \mathbf{k} - \g$ are harmonic,~$\P$--a.s., and stationarity then implies that they belong to~$S^2_{\inv}(\P)$. But~$\nabla \mathbf{k}_{ij} \in V^2_{\pot,0}(\P) \subseteq V^2_0(\P)$ and by assumption~$\g \in V^2_{\sol,0}(\P) \subseteq V^2_0(\P)$. Hence~$\nabla \cdot \mathbf{k} - \g \in V^2_{0}(\P)$, and so we conclude that~$\nabla \cdot \mathbf{k} - \g \in V^2_{\inv}(\P) \cap V^2_{0}(\P) = \{ 0 \}$. The last assertion in the lemma statement is immediate from~\eqref{e.RRT.stream}. 
\end{proof}

\begin{definition}[Flux correctors]
For each~$e\in\Rd$, we define the \emph{flux corrector}~$\nabla \bfs_{e}$ to be the stream matrix given by Lemma~\ref{l.stationary.stream}, applied to the field
\begin{equation}
\label{e.ge.def}
\g_e := \a(e+\nabla \phi_e) - \ahom e \,.
\end{equation}
It is a skew-symmetric matrix with entries in~$V^2_{\pot,0}(\P)$ satisfying
\begin{equation}
\label{e.fluxcorrect.yes}
\nabla \cdot \bfs_e = \a(e+\nabla \phi_e) - \ahom e
\end{equation}
and, for a constant~$C(d,\lambda,\Lambda)<\infty$, 
\begin{equation}
\big\langle | \nabla \bfs_{e} |^2 \big\rangle 
\leq C|e|^2. 
\end{equation}
We can infer from the proof of Lemma~\ref{l.stationary.stream} that the entries~$\bfs_{e,ij}$ of~$\bfs_e$ are the stationary, mean-zero solutions of the equation
\begin{equation}
\label{e.flux.se.eq}
-\Delta \bfs_{e,ij}
=
\partial_{x_j} \g_{e,i}
-
\partial_{x_i} \g_{e,j}
\quad \mbox{in} \ \Rd
\,.
\end{equation}
\end{definition}
Similar to the case of the corrector field~$\nabla\phi_e$, the flux corrector~$\bfs_e$ is not stationary. It is only~$\nabla \bfs_e$ that is the stationary object. The ``field''~$\bfs_e$ is well-defined only modulo the addition of a constant skew-symmetric matrix. More precisely,~$\bfs_e$ itself is not well-defined, but differences~$\bfs_e(x) - \bfs_e(y)$ are well-defined. 

\smallskip

Having now completed the task of constructing the correctors~$\phi_e$ and flux correctors~$\bfs_e$ in the stationary random setting, we turn next to the generalization of~\eqref{e.correctorsbounded2}. What is needed is a replacement for the rough statement, ``periodic functions are bounded by compactness.'' As it turns out, the statement we need is the generalized Wiener ergodic theorem given in~Corollary~\ref{c.stat.Hminusone}. In view of the fact that~$\nabla \phi_e$ and the entries of~$\nabla \bfs_e$ belong to~$V^2_{\pot,0}(\P) \subseteq V^2_0(\P)$, it implies that
\begin{equation}
\label{e.corrector.qualbound.Hm1}
\limsup_{m\to \infty}
3^{-m} 
\bigl(
\| \nabla \phi_e \|_{\Hminusul(\cu_m)}
+
\| \nabla \bfs_e \|_{\Hminusul(\cu_m)}
\bigr) = 0,
\quad \mbox{$\P$--a.s.}
\end{equation}
By~\eqref{e.L2toHminus}, we also obtain
\begin{equation}
\label{e.corrector.qualbound.L2}
\limsup_{m\to \infty}
3^{-m} 
\bigl(
\left\|  \phi_e - (\phi_e)_{\cu_m} \right\|_{\underline{L}^{2}(\cu_m)}
+
\left\|  \bfs_e - (\bfs_e)_{\cu_m} \right\|_{\underline{L}^{2}(\cu_m)}
\bigr) = 0,
\ \ \mbox{$\P$--a.s.}
\end{equation}
The estimate~\eqref{e.corrector.qualbound.L2} says that the~$L^2$-oscillation of~$\phi_e$ and~$\bfs_e$ in a ball of radius (or cubes) grows ``strictly sublinearly'' as the radius tends to infinity. Notice that this bound is much weaker than the one from the periodic case in~\eqref{e.correctorsbounded}---but~\eqref{e.corrector.qualbound.L2} is still just enough to obtain qualitative homogenization, by an argument which is essentially the same as the proof of Lemma~\ref{l.DP.2infty}.

\begin{exercise}[$L^\infty$-type corrector estimates for scalar equations]
Recall the statement of the De Giorgi-Nash H\"older estimate for uniformly elliptic equations: there exists an exponent~$\alpha(d,\lambda,\Lambda/\lambda)>0$ and constant~$C(d,\lambda,\Lambda)<\infty$ such that,  for every~$r>0$ and solution~$u\in H^1(B_{2r})$ of the equation
\begin{equation*}
-\nabla \cdot \a\nabla u = 0 \quad \mbox{in} \ B_{2r}, 
\end{equation*}
we have that~$u \in C^{0,\alpha}(B_r)$ and the estimate
\begin{equation*}
\| u \|_{L^\infty(B_r)} + r^{\alpha} [ u ]_{C^{0,\alpha}(B_r)} 
\leq 
C \| u \|_{\underline{L}^2 (B_{2r})} 
\,.
\end{equation*}
Using this and an interpolation inequality (to bound the~$L^\infty$ norm by the~$L^2$ and~$C^{0,\alpha}$ norms), show that the~$L^2$ convergence for the first term in~\eqref{e.corrector.qualbound.L2} can be upgraded to~$L^\infty$:
\begin{equation}
\label{e.corrector.qualbound.Linfty}
\limsup_{m\to \infty}
3^{-m} 
\left\|  \phi_e - (\phi_e)_{\cu_m} \right\|_{L^\infty(\cu_m)} = 0,
\ \ \mbox{$\P$--a.s.}
\end{equation}
Note that this trick works only for scalar equations, whereas the rest of the arguments in this section apply verbatim to elliptic systems. 
Can the convergence of the flux correctors be similarly upgraded?
\end{exercise}

\subsubsection{The proof of homogenization}

For each~$m\in\N$, we introduce a random variable 
\begin{equation}
\label{e.Rm.def}
\mathcal{R}(m) :=
3^{-m} 
\bigl(
\left\|  \phi_e - (\phi_e)_{\cu_m} \right\|_{\underline{L}^{2}(\cu_m)}
+
\left\|  \bfs_e - (\bfs_e)_{\cu_m} \right\|_{\underline{L}^{2}(\cu_m)}
\bigr) \,,
\end{equation}
which monitors the speed of convergence of the limit~\eqref{e.corrector.qualbound.L2}. 

\smallskip 

We next present a purely deterministic argument, following the classical two-scale expansion and computations, which reduces the question of homogenization for a general Dirichlet problem to the behavior of~$\mathcal{R}(m)$. This is essentially the same argument as in the periodic case. 
In what follows we use the notation~$U_r := \{x \in U \, : \, \dist(x,\partial U)> r\}$.

\begin{proposition}
[Homogenization of the Dirichlet problem]
\label{p.DP}
Let~$U \subseteq \cu_0$ be a Lipschitz domain. 
There exists~$C(U,d,\lambda,\Lambda)<\infty$ such that, for every~$\ep,r \in (0,1]$ and~$u \in H^1(U)\cap W^{2,\infty}(U_{r})$, if we denote by~$u^\ep\in H^1(U)$ the solution of the Dirichlet problem
\begin{equation*}
\left\{
\begin{aligned}
& -\nabla \cdot \left( \a\left(\tfrac \cdot\ep\right) \nabla u^\ep \right) = -\nabla \cdot \ahom \nabla u &  \mbox{in} & \ U, \\
& u^\ep = u & \mbox{on} & \ \partial U,
\end{aligned}
\right.
\end{equation*}
then we have the estimate
\begin{align}
\label{e.DP}
\lefteqn{
\left\| \nabla u^\ep - \nabla  u  \right\|_{H^{-1}(U)}  
+ \left\| \a\left(\tfrac\cdot\ep\right) \nabla u^\ep -\ahom \nabla  u  \right\|_{H^{-1}(U)}
} \qquad\qquad & 
\notag \\ &
\leq 
C \left\| \nabla u \right\|_{L^{2}(U \setminus U_{2r})} 
+
Cr^{-1} \mathcal{R}\bigl( \lceil |\log_3 \ep| \rceil \bigr) \left\|  \nabla u \right\|_{W^{1,\infty}(U_r)}
\,.
\end{align}
\end{proposition}
\begin{proof}
Throughout the argument, we denote~$m:= \lceil |\log_3 \ep| \rceil$,~$\a^\ep =  \a(\tfrac \cdot\ep)$ and
\begin{align*}
\phi_{e}^\ep (x) : = \ep \bigl(  \phi_{e}(\tfrac x\ep) - ( \phi_{e} )_{\cu_m} \bigr)
\quad \mbox{and} \quad
\bfs_{e}^\ep(x) := \ep \bigl( \bfs_{e}(\tfrac x\ep) - ( \bfs_{e} )_{\cu_m} \bigr)
\,.
\end{align*}
Observe that~$m\in\N$ has been chosen so that~$3^{-m} \leq \ep < 3^{1-m}$. In particular, 
\begin{align*}
\left\| \phi_e^\ep \right\|_{L^2(U)} + \left\| \bfs_e^\ep \right\|_{L^2(U)} 
\leq C\mathcal{R}(m)
\,.
\end{align*}
Select~$r\in (0,1]$ and a cutoff function~$\zeta_r \in C_0^\infty(U)$ satisfying
\begin{equation}
\label{e.zetar.cut.off}
\indc_{U_{2r}} \leq \zeta_r \leq \indc_{U_{r}}
\quad \mbox{and} \quad 
\| \nabla \zeta_r \|_{L^\infty(U)} \leq C r^{-1}
\end{equation}
and define
\begin{equation}
\label{e.wepdef}
w^\ep  :=
u +  \zeta_r \sum_{k=1}^d (\partial_{x_k}u) \phi_{e_k}^\ep .
\end{equation}
The computations here closely follow those of Lemmas~\ref{l.twoscale}, Corollary~\ref{c.twoscale} and Lemma~\ref{l.DP.2infty}. 

\smallskip

\emph{Step 1.} We plug~$w^\ep$ into the operator~$-\nabla \cdot \a^\ep \nabla$ to check that it is almost a solution. The precise claim is that
\begin{align}
\label{e.wep.eqerror}
\left\| 
\nabla \cdot \a^\ep \nabla w^\ep - \nabla \cdot \ahom \nabla u 
\right\|_{H^{-1}(U)}
\leq C \left\| \nabla u \right\|_{L^2(U \setminus U_{2r})}
+ 
C \mathcal{R}( m ) \left\| \nabla (\zeta_r \partial_{x_k} u) \right\|_{L^\infty(U)}  
. 
\end{align}
For this, we derive an expression for the fluxes similar to~\eqref{e.twoscale.flux}:
\begin{align}  \label{e.wep.eqerror.zero}
\a^\ep \nabla w^\ep - \ahom \nabla u
& =
\sum_{k=1}^d 
\phi_{e_k}^\ep \a^\ep \nabla (\zeta_r \partial_{x_k} u) 
+
(1-\zeta_r)(\a^\ep - \ahom) \nabla u 
\notag \\  &
\qquad + \sum_{k=1}^d \left( \a^\ep (e_k + \nabla \phi_{e_k}^\ep) - \ahom e_k \right) \zeta_r \partial_{x_k} u
\,.
\end{align}
Using~\eqref{e.fluxcorrect.yes}, for every~$\eta \in H^1(U)$, 
\begin{equation}
\label{e.wep.eqerror.three}
\bigl( \a(e+\nabla\phi_{e}^\ep) - \ahom e \bigr) \eta 
=
 \nabla \cdot \left( \bfs_{e}^\ep \eta \right) - 
\bfs_{e}^\ep \nabla \eta
\,.
\end{equation}
We use this identity to rewrite the last term on the right side of~\eqref{e.wep.eqerror.zero} in terms of the flux correctors, which yields 
\begin{align}
\label{e.wep.eqerror.zero.2}
\a^\ep \nabla w^\ep {-} \ahom \nabla u
& =
\sum_{k=1}^d 
\bigl( \phi_{e_k}^\ep \a^\ep  {-} \bfs_{e_k}^\ep \bigr) \nabla (\zeta_r \partial_{x_k} u) 
{+}
(1{-}\zeta_r)(\a^\ep {-} \ahom) \nabla u 
 {+} \nabla \cdot \sum_{k=1}^d \bfs_{e_k}^\ep \zeta_r \partial_{x_k} u
\,.
\end{align}
Taking the divergence of this expression and using the skew-symmetry of~$\bfs_e$, we get
\begin{equation}
\label{e.wep.eqerror.zero.2.div}
\nabla \cdot ( \a^\ep \nabla w^\ep {-} \ahom \nabla u )
=
\nabla \cdot \sum_{k=1}^d 
\bigl( \phi_{e_k}^\ep \a^\ep  {-} \bfs_{e_k}^\ep \bigr) \nabla (\zeta_r \partial_{x_k} u) 
{+}
\nabla \cdot \bigl( (1{-}\zeta_r)(\a^\ep {-} \ahom) \nabla u  \bigr)
\,.
\end{equation}
Therefore, 
\begin{align} \notag
\lefteqn{
\left\| 
\nabla \cdot \a^\ep \nabla w^\ep - \nabla \cdot \ahom \nabla u 
\right\|_{H^{-1}(U)}
} \qquad &
\\ 
\notag &
\leq
C \sum_{k=1}^d \bigl( \| \phi_{e_k}^\ep \|_{L^2(U)} + \| \bfs_{e_k}^\ep \|_{L^2(U)} \bigr) 
\left\| \nabla (\zeta_r \partial_{x_k} u) \right\|_{L^\infty(U)} + C \left\| \nabla u \right\|_{L^2(U \setminus U_{2r})}
\\ \notag & 
\leq
C \mathcal{R}(m) \left\| \nabla (\zeta_r \partial_{x_k} u) \right\|_{L^\infty(U)} + C \left\| \nabla u \right\|_{L^2(U \setminus U_{2r})}
\,.
\end{align}
This is~\eqref{e.wep.eqerror}.
\smallskip

\emph{Step 2.} 
Since~$w^\ep - u^\ep \in H_0^1(U)$, we obtain by~\eqref{e.wep.eqerror} and the equation for~$u^\ep$ that
\begin{align}
\label{e.wepuepest.a}
\left\| \nabla w^\ep - \nabla u^\ep  \right\|_{L^2 (U)} 
&
\leq
C \left\| \nabla \cdot \a^\ep (\nabla w^\ep - \nabla u^\ep ) \right\|_{H^{-1}(U)} 
\notag \\ & 
=
C \left\| \nabla \cdot \a^\ep \nabla w^\ep - \nabla \cdot \ahom \nabla u \right\|_{H^{-1}(U)} 
\notag \\ & 
\leq 
C \left\| \nabla u \right\|_{L^2(U \setminus U_{2r})}
+ 
C \mathcal{R}(m)
\left\| \nabla (\zeta_r \partial_{x_k} u) \right\|_{L^\infty(U)}  
\,.
\end{align}
Note that, in view of~\eqref{e.zetar.cut.off},  
\begin{equation}
\left\| \nabla (\zeta_r \partial_{x_k} u) \right\|_{L^\infty(U)}  
\leq
Cr^{-1} \left\| \nabla u \right\|_{L^\infty(U_r)}
+
\left\| \nabla^2 u \right\|_{L^\infty(U_r)} 
\leq 
C r^{-1} \left\| \nabla u \right\|_{W^{1,\infty}(U_r)}
\,.
\end{equation}

\smallskip

\emph{Step 3.} To complete the proof of~\eqref{e.DP},
we need to show that 
\begin{equation*}
\left\| \nabla w^\ep {-} \nabla u  \right\|_{H^{-1} (U)}
+ \left\| \a^\ep \nabla w^\ep {-} \ahom \nabla u \right\|_{H^{-1}(U)} 
\leq 
C \left\| \nabla u \right\|_{L^{2}(U \setminus U_{2r})} 
+
C r^{- 1}  \mathcal{R}(m) 
\left\|  \nabla u   \right\|_{W^{1,\infty}(U_r)} 
\,.
\end{equation*}
For the gradients, we have, by~\eqref{e.wepdef},  
\begin{align*}
\left\| \nabla w^\ep - \nabla u  \right\|_{H^{-1} (U)}
=
\biggl\|  
\nabla \Bigl( \zeta_r \sum_{k=1}^d (\partial_{x_k}u) \phi_{e_k}^\ep \Bigr)
\biggr\|_{H^{-1} (U)}
& 
\leq
C
\biggl\|  \zeta_r \sum_{k=1}^d (\partial_{x_k}u) \phi_{e_k}^\ep 
\biggr\|_{L^2(U)}
\\ & 
\leq 
C \sum_{k=1}^d
\| \phi_{e_k}^\ep \|_{L^2(U)}
\| \zeta_r \nabla u \|_{L^\infty(U)}
\\ &
\leq 
C \mathcal{R}(m) \left\| \nabla u \right\|_{L^\infty (U_r)}
\,.
\end{align*}
For the fluxes, we start from~\eqref{e.wep.eqerror.zero.2} and apply the triangle inequality to get 
\begin{align*}
\left\| \a^\ep \nabla w^\ep {-} \ahom \nabla u \right\|_{H^{-1} (U)}
& 
\leq
C \sum_{k=1}^d \| \phi_{e_k}^\ep \|_{L^2(U)} \| \nabla (\zeta_r \partial_{x_k} u) \|_{L^\infty(U)}
+
C \| (1-\zeta_r) \nabla u \|_{L^2(U)} 
\notag \\  &
\qquad + \sum_{k=1}^d 
\| \bfs_{e_k}^\ep \|_{L^2(U)} \left( \|  \zeta_r \partial_{x_k} u \|_{L^\infty(U)}
{+}
\| \nabla( \zeta_r \partial_{x_k} u ) \|_{L^\infty(U)}
\right)
\\ & 
\leq 
C \mathcal{R}(m)\bigl(  \|  \zeta_r \partial_{x_k} u \|_{L^\infty(U)}
{+}  \left\| \nabla (\zeta_r \partial_{x_k} u) \right\|_{L^\infty(U)} \bigr)
+
C \left\| \nabla u \right\|_{L^2(U \setminus U_{2r})}
\\ &
\leq 
C \left\| \nabla u \right\|_{L^{2}(U \setminus U_{2r})} 
+
C r^{- 1}  \mathcal{R}(m) 
\left\|  \nabla u   \right\|_{W^{1,\infty}(U_r)} 
\,.
\end{align*}
The proof is now complete. 
\end{proof}

We next demonstrate that Theorem~\ref{t.qualitative.homogenization} is a simple consequence of Proposition~\ref{p.DP}. 

\begin{proof}[{Proof of Theorem~\ref{t.qualitative.homogenization}}]
By routine density arguments, it suffices to consider the case that the supremum in~\eqref{e.qualhomogP} is taken over smooth, bounded domains~$U\subseteq\Rd$,~$\f \in C^\infty(U)^d$ and~$g\in C^\infty(U)$.
Let~$u = u_{U,\f,g}$ and~$u^\ep = u^\ep_{U,\f,g}$, and observe that, by standard elliptic regularity,~$u \in C^\infty(\overline{U})$. In particular, if we allow our constants~$C$ to depend on~$(U,\f,g)$, then we have 
\begin{align*}
C \left\| \nabla u \right\|_{L^{2}(U \setminus U_{2r})} 
+
Cr^{-1} \mathcal{R}\bigl( \lceil |\log_3 \ep| \rceil \bigr) \left\|  \nabla u \right\|_{W^{1,\infty}(U_r)}
\leq 
Cr + Cr^{-1}  \mathcal{R}\bigl( \lceil |\log_3 \ep| \rceil \bigr).
\end{align*}
Applying~\eqref{e.DP} with~$r:= \mathcal{R}(\lceil |\log_3 \ep |\rceil)^{\nicefrac12}$, we obtain 
\begin{align*}
\bigl\| \nabla u^\ep  -\nabla u  \bigr\|_{H^{-1}(U)}
+
\bigl\| \a^\ep \nabla u^\ep  - \ahom \nabla u  \bigr\|_{H^{-1}(U)}
\leq 
C \mathcal{R}\bigl( \lceil |\log_3 \ep| \rceil \bigr)^{\nicefrac12}\,.
\end{align*}
By~\eqref{e.corrector.qualbound.L2} and~\eqref{e.Rm.def}, 
\begin{align*}
\limsup_{\ep\to 0} \Bigl( \bigl\| \nabla u^\ep  -\nabla u  \bigr\|_{H^{-1}(U)}
+
\bigl\| \a^\ep \nabla u^\ep  - \ahom \nabla u  \bigr\|_{H^{-1}(U)} \Bigr)
\leq
C \limsup_{\ep\to 0} \mathcal{R}\bigl( \lceil |\log_3 \ep| \rceil \bigr)^{\nicefrac12} = 0. 
\end{align*}
This completes the proof of the theorem. 
\end{proof}

\subsection{The variational approach to  homogenization}
\label{ss.variational}

De Giorgi and Spagnolo~\cite{DGS}\footnote{See~\cite{DG} for an English translation} were among the first to study homogenization in a mathematically rigorous way. They proposed a variational attack on the problem by formulating a scalar elliptic equation with periodic coefficients equivalently as an integral functional and then ``homogenized'' the latter by showing that it converges to a limiting integral functional (in a sense which later became known as~\emph{$\Gamma$-convergence}). 

\smallskip

Later, Dal Maso and Modica~\cite{DM1,DM2} observed that this variational approach leads to a very natural proof of qualitative homogenization in the more general stationary-ergodic (stochastic) setting. This is because, as observed first in~\cite{DM1}, the variational point of view leads to a natural family of subadditive quantities, namely~$\mu(U,p)$ defined in~\eqref{e.mu} below, which has an almost sure limit by the subadditive ergodic theorem (Proposition~\ref{p.subadditive.ergodic}). General homogenization theorems can then be read off by unpacking the information in this limit because it essentially gives first-order correctors that are strictly sublinear in the large-scale limit.

\smallskip

In this section, we present this approach pioneered in~\cite{DGS,DM1,DM2} and give a second proof of Theorem~\ref{t.qualitative.homogenization}, under the additional assumption that the coefficient field~$\a(\cdot)$ is symmetric: 
\begin{equation}
\label{e.symm.sec24}
\P \bigl[  \a = \a^t \bigr] = 1 \,.
\end{equation}
This restriction can be lifted, and the arguments here can be extended to the general case of nonsymmetric coefficients, as we will see in Chapter~\ref{s.nonsymm}. For simplicity, we restrict our attention here to the symmetric case. 

\smallskip

We begin by introducing the following subadditive quantity: 
for each bounded Lipschitz domain~$U\subseteq \Rd$ and~$p\in\Rd$, we define
\begin{equation}
\label{e.mu.sec2}
\mu(U,p) 
:= \inf_{u \in \ell_p+H^1_0(U)} 
\fint_U \frac12\nabla u \cdot \a\nabla u
\,,
\end{equation}
where~$\ell_p$ is the affine function~$x\mapsto p\cdot x$.
The variational problem on the right side of~\eqref{e.mu.sec2} has a unique minimizer, which we denote by~$v(\cdot,U,p)$. It is the unique weak solution of the Dirichlet problem 
\begin{equation}
\label{e.DP.sec24}
\left\{
\begin{aligned}
& -\nabla \cdot \a\nabla v(\cdot,U,p) = 0 & \mbox{in} & \ U,  \\
& v(\cdot,U,p) = \ell_p & \mbox{on} & \ \partial U.
\end{aligned}
\right.
\end{equation}
In particular,~$p\mapsto v(\cdot,U,p)$ is a linear map from~$\Rd\to H^1(U)$, and so the quantity~$\mu(U,p)$ is quadratic in~$p$. We therefore deduce the existence of a matrix~$\a(U)$ such that 
\begin{equation}
\label{e.quad.mu.sec2}
\mu(U,p) = \frac12 p\cdot \a(U) p. 
\end{equation}
We call~$\a(U)$ the \emph{coarse-grained matrix} with respect to~$U$. 

\smallskip

Observe that, for every~$p\in\Rd$, we have that 
\begin{equation}
\label{e.mu.bounds.sec2}
\frac12 \lambda |p|^2 \leq \mu(U,p) \leq p\cdot \frac12\biggl(\fint_U \a(x)\,dx\biggr) p \leq \frac12 {\Lambda}|p|^2.
\end{equation}
In other words, 
\begin{equation*}
\lambda  \Id \leq \a(U) \leq \Lambda  \Id. 
\end{equation*}
Since~$\mu(U,p)$ and hence~$\a(U)$ only depends on the restriction of~$\a(\cdot)$ to~$U$, it is clear that
\begin{equation}
\label{e.mu.local.sec2}
\a(U) \quad \mbox{is~$\F(U)$--measurable.}
\end{equation}
Finally, the subadditivity of~$\mu(U,p)$ is a simple consequence of the variational structure.
Indeed, if~$U, U_1,\ldots,U_N$ are bounded Lipschitz domains satisfying~\eqref{e.U.partition}, then we may take the minimizers~$v(\cdot,U_i,p)$ and glue them together by defining
\begin{equation*}
w(x) := v(x,U_i,p), \quad x\in U_i. 
\end{equation*}
This function~$w$ belongs to~$\ell_p + H^1_0(U)$, since~$v(x,U_i,p) - \ell_p\in H^1_0(U_i)$ for each~$i\in\{1,\ldots,N\}$. 
Testing the infimum in~\eqref{e.mu.sec2} with~$w$, we obtain subadditivity:
\begin{equation}
\label{e.subadd.sec2}
|U| \mu(U,p) 
\leq 
\int_{U} \frac12\nabla w \cdot \a\nabla w 
=
\sum_{i=1}^N 
\int_{U_i} \frac12\nabla v(\cdot,U_i,p)  \cdot \a\nabla v(\cdot,U_i,p) 
=
\sum_{i=1}^N |U_i| \mu(U_i,p)\,.
\end{equation}
This is~\eqref{e.nu.subadd} for~$\nu = \mu(\cdot,p)$. 
We can write this in matrix form as
\begin{equation*}
\a(U) \leq \sum_{i=1}^N \frac{|U_i|}{|U|} \a(U_i)\,.
\end{equation*}

\smallskip

We have checked each of the hypotheses of the subadditive ergodic theorem (Proposition~\ref{p.subadditive.ergodic}) for the quantity~$\nu = \mu(\cdot,p)$, for each~$p\in\Rd$. Applying the result, therefore, yields the existence of 
$\overline{\mu}(p) \in\ S^1_{\mathrm{inv}}(\P)$ such that, for every bounded open set~$U\subseteq \Rd$, 
\begin{equation}
\label{e.subadditive.ergodic.appl}
\limsup_{t\to\infty} \
\left| \mu(tU,p) - \overline{\mu}(p) \right| = 0, 
\qquad 
\mbox{$\P$--a.s.}
\end{equation}
Since~$\mu(U,p)$ is quadratic in~$p$ and bounded by~\eqref{e.mu.bounds.sec2}, so is~$p\mapsto \overline{\mu}(p)$, and the limit above is locally uniform in~$p$. In fact, we deduce the existence of a matrix~$\ahom$ with entries in~$S^\infty_{\mathrm{inv}}(\P)$ such that 
\begin{equation}
\label{e.conv.to.ahom}
\P \Bigl[ 
\limsup_{t\to\infty} \
\left| \a(tU) - \ahom \right| = 0
\Bigr] = 1.  
\end{equation}
By~\eqref{e.mu.bounds.sec2}, it is clear that~$\ahom$ satisfies~$\lambda  \Id \leq \ahom \leq \Lambda  \Id$,~$\P$--a.s.
Moreover, if~$\P$ is ergodic, then~$\ahom$ is deterministic. 

\smallskip

Let us see how we can extract a homogenization statement from the limit~\eqref{e.subadditive.ergodic.appl} of the subadditive quantities~$\mu(U,p)$. 
The plan is to show that the minimizers of~$\mu(\cu_n,p)$ become ``flat'' as~$n\to \infty$. As we will see, this flatness will tell us that certain ``finite-volume'' versions of the correctors satisfy bounds analogous to~\eqref{e.corrector.qualbound.Hm1}. 

For each~$n\in\N$, let~$v_n(\cdot,p) \in H^1_{\mathrm{loc}}(\Rd)$ be the global function obtained by gluing the minimizers of~$\mu(z+\cu_n,p)$ over the partition~$\{ z+\cu_n \,:\, z\in 3^n\Zd \}$ of~$\Rd$. That is, 
\begin{equation*}
v_n(x,p) := 
v(x,z+\cu_n,p), \quad x \in z+\cu_n, \ z\in3^n\Zd\,.
\end{equation*}
The idea for showing ``flatness'' is to argue that the difference between~$v_m$ and~$v_n$ is becoming small as~$m,n\to \infty$ since their energies are converging to the same constant (per unit volume) and the energy functional is uniformly convex. 
The smallness of~$v_m - v_n$ implies that if~$n$ represents a mesoscopic scale relative to~$m$, say~$n=\nicefrac m2$, then~$v_n$ is close to a plane from the perspective of the length scale~$3^m$ due to the Poincar\'e inequality and the fact that it is glued to an affine function on the boundary of the cubes~$z+\cu_n$ for~$z\in 3^n\Zd$. We will formalize this slightly more sophisticatedly by incorporating information from all scales (not just one mesoscale) using the multiscale Poincar\'e inequality. 

\smallskip

Let us turn now to the rigorous argument. We fix~$p\in\Rd$ and allow ourselves to drop the dependence of~$v_n$ and~$v_m$ on~$p$ from the expressions.
Due to the uniform convexity of the minimization problem~\eqref{e.mu.sec2}, we can measure the difference between~$v_m-v_n$ in~$H^1(\cu_m)$ for~$n<m$ in terms of the strictness in the subadditivity inequality~\eqref{e.subadd.sec2}. 
Since the integral functional is quadratic, this is particularly simple: using that the function~$v_m$ satisfies~\eqref{e.DP.sec24} with~$U=\cu_m$ and~$v_n-v_m\in H^1_0(\cu_m)$, we have
\begin{align}
\label{e.multiscale.snap}
\frac{\lambda}{2} \fint_{\cu_m} 
|\nabla v_m- \nabla v_n|^2
&
\leq 
\fint_{\cu_m} 
\frac12 \nabla (v_m-v_n) \cdot \a\nabla (v_m-v_n) 
\notag \\ &
=
\avsum_{z\in 3^n\Zd\cap \cu_m} 
\mu(z+\cu_n,p)
-
\mu(\cu_m,p) \,.
\end{align} 
In particular, since~$v_n \in \ell_p+ H^1_0(z+\cu_n)$ for every~$z\in 3^n\Zd$, we have that 
\begin{equation*}
\avsum_{z\in 3^n\Zd\cap \cu_m} 
\bigl| \bigl( \nabla v_m \bigr)_{z+\cu_n} - p \bigr|^2
\leq 
\frac{2}{\lambda} 
\avsum_{z\in 3^n\Zd\cap \cu_m} 
\mu(z+\cu_n,p)
-
\mu(\cu_m,p) \,.
\end{equation*}
By Proposition~\ref{p.MSP}, we deduce that 
\begin{align}
\left\| \nabla v_m - p \right\|_{\Hminusul(\cu_m)}
&
\leq
C
\left\| \nabla v_m - p \right\|_{ \underline{L}^2(\cu_m)} 
+
C 
\sum_{n=0}^{{m-1}} 3^{n} 
\biggl( \, \avsum_{z\in 3^n\Zd\cap \cu_{m}} 
\left| \left( \nabla v_m  \right)_{z+\cu_n} - p \right|^2 \biggr)^{\!\nicefrac12}
\notag \\ & 
\leq 
C|p|
+ 
C 
\sum_{n=0}^{{m-1}} 3^{n} 
\biggl( \, \avsum_{z\in 3^n\Zd\cap \cu_{m}} 
\mu(z+\cu_n,p)
-
\mu(\cu_m,p) 
\biggr)^{\!\nicefrac12}
\notag \\ & 
\leq 
C|p|
+ 
C|p|
\sum_{n=0}^{{m-1}} 3^{n} 
\biggl| 
 \a(\cu_m)-
\avsum_{z\in 3^n\Zd\cap \cu_{m}} 
\a(z+\cu_n)
\biggr|^{\nicefrac12}
\label{e.flatnessqual.sec24}
\\ & 
\leq 
C|p|
+ 
C |p|
\sum_{n=0}^{{m}} 3^{n} 
\biggl( \, \avsum_{z\in 3^n\Zd\cap \cu_{m}} 
| \a(z+\cu_n)
-
\ahom |
\biggr)^{\!\nicefrac12}
\label{e.flatnessqual.sec24.giveup}
\,.
\end{align}
It is easy to deduce from~\eqref{e.conv.to.ahom} that 
\begin{equation}
\label{e.multiscalequant.sec24}
\P \Biggl[
\limsup_{m\to\infty} 
\sum_{n=0}^{{m}} 3^{n-m} 
\avsum_{z\in 3^n\Zd\cap \cu_{m}} 
| \a(z+\cu_n)
-
\ahom |
= 0
\Biggr] = 1. 
\end{equation}
The previous two displays imply
\begin{equation}
\label{e.gradweak.sec24}
\P \biggl[
\limsup_{m\to\infty} \ 
3^{-m} \left\| \nabla v_m(\cdot,\cu_m,p) - p \right\|_{\Hminusul(\cu_m)} = 0
\biggr] = 1. 
\end{equation}
To obtain a similar bound for the fluxes, we use the identity 
\begin{equation}
\label{e.flux.identity.sec24}
\a(U) p = \fint_{U} \a\nabla v(\cdot,U,p),
\end{equation}
which is proved by an integration by parts using~\eqref{e.DP.sec24} (which is also given below in~\eqref{e.a.astar.formulas}), 
as well as~\eqref{e.multiscale.snap} and the triangle inequality, to obtain
\begin{align}
\lefteqn{
\left\| \a \nabla v_m - \a(\cu_m) p \right\|_{\Hminusul(\cu_m)}
} \qquad & 
\notag \\ &
\leq
C
\left\| \a \nabla v_m - \a(\cu_m) p \right\|_{ \underline{L}^2(\cu_m)} 
+
C 
\sum_{n=0}^{{m-1}} 3^{n} 
\biggl( \avsum_{z\in 3^n\Zd\cap \cu_{m}} 
\left| \left( \a \nabla v_m  \right)_{z+\cu_n} - \a(\cu_m) p \right|^2 \biggr)^{\!\nicefrac12}
\notag \\ & 
\leq C|p|
+ 
C|p|
\sum_{n=0}^{{m-1}} 3^{n} 
\biggl| 
 \a(\cu_m)-
\avsum_{z\in 3^n\Zd\cap \cu_{m}} 
\a(z+\cu_n)
\biggr|^{\nicefrac12}
\label{e.flatnessqual.flux.sec24} \\ & 
\leq 
C|p| 
+ 
C |p|
\sum_{n=0}^{{m}} 3^{n} 
\biggl( \avsum_{z\in 3^n\Zd\cap \cu_{m}} 
| \a(z+\cu_n)
-
\ahom |
\biggr)^{\!\nicefrac12}
\label{e.flatnessqual.flux.sec24.giveup}
\,.
\end{align}
With this in place of~\eqref{e.flatnessqual.sec24.giveup}, we obtain 
\begin{equation}
\label{e.fluxweak.sec24}
\P \biggl[
\limsup_{m\to\infty} \ 
3^{-m} \left\| \a \nabla v_m(\cdot,\cu_m,p) - \ahom p \right\|_{\Hminusul(\cu_m)} = 0
\biggr] = 1. 
\end{equation}
Using~\eqref{e.gradweak.sec24} and~\eqref{e.fluxweak.sec24}, we can build ``finite-volume correctors'' which are functionally the same as the correctors for the purposes of proving homogenization. For the correctors, we define, for each~$m\in\N$, 
\begin{equation}
\label{e.FVC.def}
\phi_{m,e} := v_m(\cdot,\cu_m,e) - \ell_e \in H^1_0(\cu_m)
\end{equation}
and for the flux correctors, we set
\begin{equation*}
\g_{m,e} := \a \nabla v_m(\cdot,\cu_m,e) - \a(\cu_m) e
\end{equation*}
and, analogously to~\eqref{e.fluxcorrect}, we let~$\bfs_{m,e}$ be matrix with entries~$\bfs_{m,e,ij} \in H^1(\cu_m)$ defined for each~$i,j\in\{1,\ldots,d\}$ as the solution of  
\begin{equation}
\label{e.fluxcorrect.sec24}
\left\{
\begin{aligned}
&
-\Delta \bfs_{m,e,ij} = \partial_{x_i} \g_{m,e,j} - \partial_{x_j} \g_{m,e,i} 
\quad \mbox{in}  \ \R^d, \\
& \bfs_{m,e,ij}  \ \mbox{is~$3^m\Z^d$--periodic and} \ ( \bfs_{m,e,ij} )_{\cu_m} = 0.
\end{aligned}
\right. 
\end{equation}
By a direct computation similar to the proof of~\eqref{e.divkg}, using also~\eqref{e.flux.identity.sec24}, we find that  
\begin{equation}
\label{e.fluxcorrect.FV.divg}
\nabla \cdot \bfs_{m,e} = \g_{m,e}
\end{equation}
We claim that 
\begin{equation}
\label{e.dualest.forflux}
\| \bfs_{m,e} \|_{\underline{L}^2(\cu_m)} 
\leq
C\| \g_{m,e} \|_{\Hminusul(\cu_m)} \,.
\end{equation}
To see this, fix~$i,j\in\{1,\ldots,N\}$ and solve the auxiliary problem 
\begin{equation*}
\left\{
\begin{aligned}
&
-\Delta w = \bfs_{m,e,ij} 
\quad \mbox{in}  \ \R^d \,,
\\
& w  \ \mbox{is~$3^m\Z^d$--periodic and} \ ( w )_{\cu_m} = 0
\,.
\end{aligned}
\right. 
\end{equation*}
Using the standard estimate~$\| w \|_{\underline{H}^2(\cu_m)} \leq C \| \bfs_{m,e,ij} \|_{\underline{L}^2(\cu_m)}$ (recall the underlined notation for norms defined in~\eqref{e.underlinednorms}) and the equations for~$\bfs_{m,e,ij}$ and~$w$, we compute
\begin{align*}
\left\| \bfs_{m,e,ij} \right\|_{\underline{L}^2(\cu_m)}^2 
=
\fint_{\cu_m} \nabla w \cdot \nabla \bfs_{m,e,ij} 
& 
= 
\fint_{\cu_m} \bigl( 
\partial_{x_j} w \g_{m,e,i} - \partial_{x_i} w \g_{m,e,j} \bigr)
\\ & 
\leq 
2 \left\| \nabla w \right\|_{\underline{H}^1(\cu_m)} \left\| \g_{m,e} \right\|_{\Hminusul(\cu_m)}
\\ &
\leq 
C \left\| \bfs_{m,e,ij} \right\|_{\underline{L}^2(\cu_m)} \left\| \g_{m,e} \right\|_{\Hminusul(\cu_m)}\,.
\end{align*}
Rearranging this, we get~\eqref{e.dualest.forflux}. 

\smallskip

It follows from~\eqref{e.gradweak.sec24},~\eqref{e.fluxweak.sec24} and~\eqref{e.dualest.forflux} that 
\begin{equation}
\label{e.corrector.qualbound.L2.sec24}
\P \biggl[ 
\limsup_{m\to \infty}
3^{-m} 
\bigl(
\left\|  \phi_{m,e}  \right\|_{\underline{L}^{2}(\cu_m)}
+
\left\|  \bfs_{m,e}  \right\|_{\underline{L}^{2}(\cu_m)}
\bigr) = 0 \biggr] = 1. 
\end{equation}
This can be compared to~\eqref{e.corrector.qualbound.L2}. Theorem~\ref{t.qualitative.homogenization}, under the additional symmetry assumption~\eqref{e.symm.sec24}, can now be proved by following the rest of the argument in Section~\ref{ss.random}, using these finite-volume correctors in place of~$\phi_e$ and~$\bfs_e$ and with no other changes in the argument. Indeed, the statement and proof of Proposition~\ref{p.DP} remain valid if we replace the correctors~$\phi_e - (\phi_e)_{\cu_m}$ and~$\bfs_e - (\bfs_e)_{\cu_m}$ with their finite volume analogues~$\phi_{m,e}$ and~$\bfs_{m,e}$ in the definition of~$\mathcal{R}(m)$ in~\eqref{e.Rm.def}.

\subsection{Regularity on large scales: the~\texorpdfstring{$C^{1,1-}$}{{C11-}} estimate}
\label{ss.reg}

Solutions of uniformly elliptic equations with rough coefficients do not generally possess very strong regularity. However, it turns out that solutions of elliptic equations with \emph{stationary} coefficients possess much better regularity properties---at least on sufficiently large length scales. 
This phenomenon of \emph{large-scale regularity} plays a central role in the quantitative theory of homogenization, particularly in the stochastic setting. 
In this section, we exhibit a qualitative statement demonstrating this property for very general setting stationary coefficient fields. Quantitative versions and extensions (which are of greater importance) will be presented later in Chapter~\ref{s.regularity}.

\smallskip

To get an idea of how we think about regularity in the context of homogenization, let us recall the ``textbook'' situation of a uniformly elliptic coefficient field~$\a(\cdot)$ satisfying~\eqref{e.ue} but with no further structure (such as periodicity or stationarity) and suppose that~$u\in H^1$ is a solution of~$-\nabla \cdot \a\nabla u = 0$ in~$B_1$. One of the primary concerns of classical elliptic regularity theory is to give sufficient conditions on~$\a(\cdot)$ so that~$u$ belongs to a certain regularity class. Some important results of this form are roughly summarized below:
\begin{equation}
\label{e.classical.regularity}
\left\{
\begin{aligned}
&
\mbox{$\a(\cdot)\in L^\infty(B_1)~$}  \ \implies \ \|\nabla u\|_{L^2(B_{\nicefrac12})} \leq C \bigl( \nicefrac{\Lambda}{\lambda} \bigr)^{\nicefrac12}\| u\|_{L^2(B_{1})} 
& \mbox{(Caccioppoli estimate)}
\\
& 
\mbox{$\a(\cdot)\in L^\infty(B_1)~$} 
\ \ \implies \ \
u \in C^{\delta}(B_{\nicefrac12}) 
\,, \quad 0< \delta \ll 1\,, 
& \mbox{(De Giorgi-Nash estimate)}
\\ & 
\mbox{$\a(\cdot)\in L^\infty(B_1)~$} 
\ \ \implies \ \
u \in W^{1,2+\delta} (B_{\nicefrac12}) 
\,, \quad 0<\delta \ll 1\,, 
& \mbox{(Meyers estimate)}
\\ & 
\a \in C(B_1) 
\ \ \implies \ \
u \in W^{1,p} (B_{\nicefrac12}) 
\,, \quad \forall p\in (1,\infty) \,, 
& \mbox{(Calder\'on-Zygmund estimate)}
\\ & 
\a \in C^{0,\gamma} (B_1)  
\ \ \implies \ \
u \in C^{1,\gamma} (B_{\nicefrac12}) 
\,, \quad \alpha\in (0,1) \,.
& \mbox{(Schauder estimate)}
\end{aligned}
\right.
\end{equation}
The first three estimates in the list, namely the Caccioppoli, De Giorgi-Nash and Meyers estimates, are 
\emph{scale-invariant} estimates. This is evident since there is no assumption on the coefficient field apart from uniform ellipticity and boundedness---scale-invariant assumptions. In particular, we can use these estimates on large scales, which is very useful in the context of homogenization. 
The regularity they give is rather weak: a slight improvement in the integrability of the gradient from~$L^2$ to~$L^{2+\delta}$ (Meyers estimate) or an oscillation decay estimate with a very slow rate of decay (De Giorgi-Nash estimate). 

\smallskip

Unfortunately, there can be no general improvement in these estimates. In fact, the simple~$2d$ checkerboard with four squares and piecewise constant scalar matrices serves as an elementary example. Fix~$\Lambda>1$ and let~$\a$ be the coefficient field in~$\R^2$ which is equal to~$\Lambda  \Id$ in the first and third quadrant (the ``white'' squares) and equal to~$\Id$ in the second and fourth quadrant (the ``black'' squares).  
By an elementary (if tedious) computation, 
one can show\footnote{Hint: use the symmetries of the coefficient field, work in polar coordinates, and demand continuity of the solutions and the normal component of its flux across the interfaces.} the existence of an~$\alpha$--homogeneous solution~$u \in H^1_{\mathrm{loc}}(\R^2)$ of~$\nabla \cdot \a\nabla u = 0$ in~$\R^2$ with exponent~$\alpha = O(\Lambda^{-1})$ as~$\Lambda \to \infty$ and~$\alpha < 1$ for every~$\Lambda > 1$. 
This example ensures that, at best, the exponent in both estimates satisfies~$\delta \leq O(\Lambda^{-1})$. Even for~$\Lambda$ very close to~$1$, the gradient~$|\nabla u|$ blows up near the origin. 
A similar example, with a radially symmetric coefficient field and solution, which can be computed explicitly, is due to Meyers (see~\cite[Chapter 3]{AKMBook} for this example and more discussion).

\smallskip

The second pair of results in~\eqref{e.classical.regularity} are not scale invariant: each is proved by a blow-up argument that sees only the smallest scales---more precisely, scales no larger than on which the coefficients oscillate. In the case of Calder\'on-Zygmund, the arguments work like this: fix a desired integrability exponent~$p\in(1,\infty)$ and zoom in to a small ball on which~$\osc \a \leq \delta$, for a tiny~$\delta$ which will depend on~$p$. The radius of this very small ball will of course depend on~$p$ and the modulus of continuity of~$\a(\cdot)$. For~$\delta$ very small, the equation in this small ball is then close to an equation with constant coefficients, and the estimate then can be proved perturbatively from the regularity of harmonic functions. Since~$\a(\cdot)$ is continuous, one can cover~$B_{\nicefrac12}$ with finitely many balls on which~$\osc \a \leq \delta$  to conclude. 

\smallskip

The Schauder estimates can be proved using an argument of a similar nature, perturbing off the regularity of harmonic functions on small scales, with the smallness of the scale given in terms of the assumed regularity of the coefficients. There one keeps track of the \emph{excess}, which is the term De Giorgi gave to the quantity 
\begin{equation}
\label{e.excess}
E(r) := \frac1r \min_{p \in\Rd, \, c\in\R} \| u - \ell_p - c \|_{\underline{L}^2(B_r)}
\,,
\end{equation}
which measures how much~$u$ deviates from an affine function at scale~$r>0$. To prove a~$C^{1,\gamma}$ estimate, it suffices to show that this quantity decays like~$O(r^\gamma)$ as~$r \downarrow 0$, and this is how the proof of the Schauder estimate goes. 

\smallskip

The Calder\'on-Zygmund and Schauder estimates have limited applicability in homogenization theory, because they say nothing about the behavior of solutions on large scales, only small scales. Indeed, they are useful only on scales smaller than the unit scale, in other words, below the scale of the periodic lattice~$\Zd$. This is unfortunately true regardless of how smooth the coefficients are assumed to be, indeed even if we have~$\a\in C^\infty$. This fact is somewhat hidden if one states these estimates in too soft of a form, as we were guilty of doing above in~\eqref{e.classical.regularity}, because it underemphasizes the importance of the dependence of the constants in the estimates on the smoothness of the coefficients. In other words, they are not quantitative statements. 

\smallskip

Let us present a more quantitative and illuminating way to state these estimates.
For the Calder\'on-Zygmund, the statement is: for each~$p\in(1,\infty)$, there exists~$\delta(p,d,\lambda,\Lambda)>0$ such that 
\begin{equation*}
\osc_{B_1} \a \leq \delta
\quad 
\implies
\| \nabla u \|_{L^p(B_{\nicefrac12})}
\leq C \| \nabla u \|_{L^2(B_{1})}\,.
\end{equation*}
For the Schauder estimate, the statement\footnote{Here~$[\cdot]_{C^{0,\alpha}(B_r)}$ denotes the H\"older seminorm and the prefactor constants~$C$ depend only on~$(d,\lambda,\Lambda)$ and, in particular, neither on~$p$ nor~$\alpha$.} is: for each~$\gamma\in (0,1)$, 
\begin{equation}
\label{e.Schauder}
r \leq r_0:= \bigl[ \a \bigr]_{C^{0,\gamma}(B_{1})}^{-\nicefrac1\gamma} \wedge \frac12
\quad 
\implies
\| \nabla u \|_{L^\infty(B_{r})}
+ r^\gamma \bigl[ \nabla u \bigl]_{C^{0,\gamma}(B_{r})}
\leq C \| \nabla u \|_{\underline{L}^2(B_{2r})}\,.
\end{equation}
In each case, the modulus of continuity of the coefficients~$\a(\cdot)$ implicitly determines a maximal length scale for the estimates to be valid.

\smallskip

Let us return back to the context of homogenization. Suppose that we have a coefficient field~$\a(\cdot)$ which is~$\Zd$--stationary (or even~$\Zd$--periodic) and consider a solution~$u\in H^1(B_R)$ of~$-\nabla \cdot \a\nabla u=0$ in a ball~$B_R$ with~$R\gg 1$. We may wish to control the energy density or the flux of~$u$, namely~$\nabla u \cdot \a\nabla u$ or~$\a\nabla u$, to show that these cannot concentrate in very small scales, or even in on \emph{mesoscales} (scales~$r$ with~$1 \ll r \ll R$). If, on the contrary, these quantities could concentrate into sets of small measure (relative to~$B_R$), then this would suggest that the solution~$u$ is very \emph{singular} with respect to the coefficients---that slight changes of~$\a(\cdot)$ can have a dramatic effect on~$u$. As it turns out, behavior like this is very bad for the quantitative homogenization theory. Therefore, controlling the size of~$|\nabla u|$ on scales~$r$ in the range~$1\ll r \ll R$ is of fundamental importance. 

\smallskip

While the estimates from classical elliptic regularity theory do not directly help us, the examples showing that the De Giorgi-Nash and Meyers estimates cannot be (in general) improved, like the four-square checkerboard described above actually provide some hope. After all, they are definitely not periodic or a realization of a stationary random field, as they have a macroscopic structure. 
Perhaps there is a way to show that a ``generic'' coefficient field has better regularity properties than the ``worst'' coefficient field. 

\smallskip

The proofs of the Calder\'on-Zygmund and Schauder estimates actually give us a clue of how to proceed. 
As in those estimates, we will perturb a constant-coefficient equation using a scale-by-scale iteration argument. What is different is that the closeness to the constant-coefficient equation will be obtained not by zooming in and using the continuity of the coefficients but instead by \emph{zooming out and using homogenization}. 
Indeed, homogenization says that, on large scales, the heterogeneous operator~$-\nabla \cdot \a\nabla$ should be ``close'' to the homogenized operator~$-\nabla \cdot \ahom\nabla$. 
The ability to approximate solutions by harmonic functions should have implications for their regularity.
It turns out that, by using this idea in the context of qualitative homogenization, it is possible to prove what we call ``large-scale~$C^{0,1}$ and~$C^{1,1-}$ estimates,'' as we will show below in Theorem~\ref{t.C11.sharp}. 

\smallskip

To explain this terminology, let us think about the form of the estimate we should try to prove. 
As usual in elliptic regularity, we want to pass information from one scale to the next in an iteration going down the scales geometrically. 
If we try to control, say, the~$L^\infty$ norm of the gradient, we may try to show something like 
\begin{equation*}
\| \nabla u \|_{\underline{L}^2(B_{\theta r})}
\leq
\bigl( 1 + \mathrm{Error}(r) \bigr)
\| \nabla u \|_{\underline{L}^2(B_r)}
\end{equation*}
for some~$\theta \in (0,\nicefrac12)$, where~$\mathrm{Error}(r)$ is the error made in the harmonic approximation at scale~$r$. 
If the error is small enough, then an iteration of this inequality down all the scales will give
\begin{equation*}
\sup_{s\in (0,r_0]} 
\| \nabla u \|_{\underline{L}^2(B_{s})}
\leq 
\biggl( \prod_{k=1}^\infty
\bigl( 1 + \mathrm{Error}(\theta^k r_0 ) \bigr)
\biggr) \| \nabla u \|_{\underline{L}^2(B_{r_0})}
\leq
C \| \nabla u \|_{\underline{L}^2(B_{r_0})}
\,.
\end{equation*}
If we repeat this iteration for balls centered at every possible point~$x$ and apply the Lebesgue differentiation theorem, we get an~$L^\infty$ bound on~$|\nabla u|$. This is roughly how the Schauder estimate~\eqref{e.Schauder} is proved (the actual proof is slightly more complicated, but this is the essential idea); in that case, the error function is bounded by~$\mathrm{Error}(r) \leq Cr^{\gamma}$ which leads to the product being finite.

\smallskip

In the case of homogenization, the idea is the same, but the error function gets \emph{worse}, rather than better, on small scales. We can start the iteration on an arbitrarily large scale because there the error is relatively small. Still, we must stop the iteration once the scale becomes too small for homogenization (which are scales of the order of the microscopic scale, in this case of order one). Therefore, it is natural to expect that solutions of a periodic or stationary equation satisfy an estimate of the form 
\begin{align}
\label{e.lipbound.pre}
\sup_{s \in [1,R]}
\| \nabla u \|_{\underline{L}^2(B_s)}
\leq 
C \| \nabla u \|_{\underline{L}^2(B_R)} \,.
\end{align}
We call this a \emph{large-scale Lipschitz (or~$C^{0,1}$) estimate} because it says that the function~$u$ behaves like a Lipschitz function on all scales in the range~$[1,R]$. Since it does not say anything about the very small scales in the range~$(0,1]$, we cannot apply Lebesgue differentiation and get control over the (literal)~$L^\infty$ norm of~$u$. However, this should still be thought of as an~$L^\infty$-type estimate. Indeed, from the point of view of homogenization theory, these small scales, which are smaller than the microscopic scale, do not matter; the regularity of solutions on these scales is determined by the regularity of the coefficients, not by homogenization and conflating these issues will actually lead to confusion. 
Anyway, if one is working with a coefficient field with H\"older continuous coefficients, then it is possible to apply~\eqref{e.Schauder} to bound the true~$L^\infty$ norm by the~$L^2$ norm in a microscopic ball, which can then be plugged into~\eqref{e.lipbound.pre}, in effect replacing the left side by~$|\nabla u(0)|$. Without such an assumption, a pointwise bound like this is not true---but this does not reduce the importance of the large-scale estimate. 

\smallskip

Avellaneda and Lin~\cite{AL1} proved the estimate~\eqref{e.lipbound.pre} for periodic coefficient fields by using a compactness argument, which is essentially a soft version of the ideas we sketched above. In fact, they proved a more precise \emph{large-scale~$C^{1,1-}$ estimate}. This asserts that if we modify the definition of the excess in~\eqref{e.excess} by replacing the affine functions with \emph{corrected} affine functions, then we obtain a decay of the excess analogous to a~$C^{1,\gamma}$ estimate, for any~$\gamma<1$. 
These ``corrected affine functions'' are the elements of the set~$\mathcal{A}_1$, defined by
\begin{align*}  
\mathcal{A}_1 = \bigl\{ \psi \in H^1_{\mathrm{loc}}(\Rd) \, : \, \nabla \psi = e + \nabla \phi_e \quad \mbox{for some} \ e \in \Rd \bigr\} \,.
\end{align*}
In other words,~$\mathcal{A}_1$ is the set of solutions which, for some~$e\in\Rd$, are the sum of the affine~$\ell_e$, the corrector~$\phi_e$ and any additive constant. The large-scale~$C^{1,1-}$ estimate is then the statement that, for every~$\gamma<1$ and solution~$u$ of~$-\nabla \cdot \a\nabla u = 0$ in~$B_R$, there exists~$\phi \in \mathcal{A}_1$ such that 
\begin{equation}
\label{e.AL.C11}
\left\| u - \phi \right\|_{\underline{L}^2(B_r)} 
\leq 
C \left( \frac r R \right)^{1+\gamma} 
\left\| u  \right\|_{\underline{L}^2(B_R)}.
\end{equation}
In short, a solution of~$-\nabla \cdot \a\nabla u=0$ can be approximated by an element of~$\mathcal{A}_1$ with almost the same precision as a harmonic function can be approximated by an affine function. 

\smallskip

Note that such an estimate says something about the \emph{local behavior} of solutions, 
namely that they exhibit a certain type of \emph{universality} on small scales.
Indeed, the estimate says that the periodic wiggles induced by the coefficient field are captured by the first-order correctors at leading order. The homogenization results proved above (for instance, the estimate of~$w^\ep - u^\ep$ in the proof of Proposition~\ref{p.DP}) already told us this \emph{on average} in an~$L^2$ sense, but the estimate~\eqref{e.AL.C11} asserts this in a uniform,~$L^\infty$-type sense. 

\smallskip

In the theorem below, we will prove a general version of the large-scale~$C^{0,1}$ and~$C^{1,1-}$ estimates, extending the result of Avellaneda and Lin to equations with merely stationary coefficient fields. 
We will follow the ideas we have sketched above of using an iteration down the scales in analogy to the Schauder estimates, a quantitative version of the argument of~\cite{AL1} which originated in~\cite{AS}.
There, they showed that there exists a \emph{random minimal scale}~$\X$, so that if~$r \leq R \leq \X$, then~\eqref{e.lipbound.pre} is valid with a deterministic constant~$C(d,\lambda,\Lambda)$. Thus, the only extra random object in the statement is this random scale~$\X$, which is the scale at which the error~$\mathrm{Error}(r)$ first becomes large as one iterates down the scales. In other words,~$\X$ is the smallest scale above which the homogenization error is sufficiently small. 

\smallskip

At the qualitative level, assuming only stationarity, it is not possible to improve~\eqref{e.X.qual.sec2}, which merely states that the minimal scale~$\X$ is almost surely finite. In Chapter~\ref{s.regularity}, we will bound~$\X$ much more strongly (and sharply) under quantitative assumptions on~$\P$: see Theorem~\ref{t.optimalstochasticintegrability}. Recall from~\eqref{e.def.A(U)} that~$\A(U)$ denotes the set of solutions of~$\nabla \cdot \a\nabla u=0$ in~$U$.

\begin{theorem}[Large-scale~$C^{0,1}$ and~$C^{1,1-}$ regularity]
\label{t.C11.sharp}
Let~$\gamma \in [\tfrac12,1)$ and~$\P$ be a~$\Zd$--stationary probability measure on~$(\Omega,\F)$. There exist~$C(\gamma,d,\lambda,\Lambda)<\infty$ and a random variable~$\X$, which we call the ``minimal scale for~$C^{1,\gamma}$ regularity'' and satisfies
\begin{equation}
\label{e.X.qual.sec2}
\P \left[ \X < \infty \right] = 1,
\end{equation}
such that, for every~$R\geq \X$ and~$u\in \A(B_R)$, we have 
\begin{equation}
\label{e.C01}
\sup_{r \in [\X,R]}
\left\| \nabla u \right\|_{\underline{L}^2(B_r)} 
\leq C \left\| \nabla u  \right\|_{\underline{L}^2(B_R)}
\,,
\end{equation}
and there exists~$\phi \in \A_1$ satisfying, for every~$r \in \left[ \X,  R \right]$, 
\begin{equation}
\label{e.intrinsicreg11-}
\left\| u - \phi \right\|_{\underline{L}^2(B_r)} \leq C \left( \frac r R \right)^{\!1+\gamma} 
\left\| u  \right\|_{\underline{L}^2(B_R)}.
\end{equation}
\end{theorem}
\begin{proof}
We let~$\delta\in (0,\tfrac12]$ be a small parameter which will be selected later and may depend only on~$(\gamma,d,\lambda,\Lambda)$. We define~$\X$ to be the random variable
\begin{equation}
\label{e.XC11}
\X:= 
\sup\left\{ 3^m \,:\, 
3^{-m} 
\sup_{e\in B_1}
\bigl(
\left\| \nabla \phi_e \right\|_{\Hminusul(\cu_m)}
+
\left\| \nabla \bfs_e \right\|_{\Hminusul(\cu_m)}
\bigr)
> \delta
\right\}
\,.
\end{equation}
According to~\eqref{e.corrector.qualbound.Hm1}, we have that~$\X$ is finite almost surely, that is, it satisfies~\eqref{e.X.qual.sec2}. 

\smallskip

We first record some simple estimates for~$\phi_e$. 
For every~$r\geq \X$ and~$m \in \N$ such that~$3^{m-1} < r \leq 3^m$, we have 
\begin{align}
\label{e.correbounds.pre}
\frac1r
\left\| \phi_e - (\phi_e)_{B_r} \right\|_{\underline{L}^{2}(B_r)} 
\leq  
C 
3^{-m} \left\| \nabla \phi_e \right\|_{\Hminusul(\cu_m)} 
\leq 
C\delta |e|
\,.
\end{align}
The first inequality follows from~\eqref{e.L2toHminus}. The Poincar\'e inequality and the Caccioppoli inequality, on the other hand, yield that
\begin{align*}  
\frac 1{2r} \left\| \ell_e + \phi_e - (\phi_e)_{B_r} \right\|_{\underline{L}^{2}(B_r)} 
\leq 
\left\| e + \nabla \phi_e \right\|_{\underline{L}^{2}(B_{r})}   
\leq
\frac{C}{r} \left\| \ell_e + \phi_e - (\phi_e)_{B_{2r}} \right\|_{\underline{L}^{2}(B_{2r})} 
\,.
\end{align*}
From the above two displays and the triangle inequality we obtain, assuming that~$\delta(d,\lambda,\Lambda)$ is sufficiently small, that, for every~$e\in B_1$ and~$r\geq \X$, 
\begin{equation}
\label{e.correbounds}
c |e|
\leq
\left\| e + \nabla \phi_e \right\|_{\underline{L}^{2}(B_{r})} 
\leq 
C |e| \,.
\end{equation}
In particular, for every~$r,s\geq \X$ and~$\psi\in \A_1$, 
\begin{equation}
\label{e.correbounds.abc}
\left\| \nabla \psi \right\|_{\underline{L}^{2}(B_{r})} 
\leq 
C \left\| \nabla \psi \right\|_{\underline{L}^{2}(B_{s})}
\,. 
\end{equation}

\smallskip

\emph{Step 1: Harmonic approximation.} 
Recall from~\eqref{e.def.Ahom(U)} that~$\Ahom(U)$ denotes the set of solutions of the homogenized equation in~\eqref{e.def.Ahom(U)}.
We claim that, for every~$r \geq \X$ and~$v \in \A(B_r)$, there exists~$w \in \Ahom (B_{r/2})$ such that 
\begin{equation}
\label{e.delta.harmonicapprox}
\left\| v - w \right\|_{\underline{L}^2(B_{r/2})}
\leq 
C\delta^{\frac{1}{d+6}} 
\left\| v \right\|_{\underline{L}^2(B_{r})}
\,.
\end{equation}
To prove this, if~$v\in \mathcal{A}(B_r)$, we select~$s\in [ \frac12 r, \frac34 r]$,  using the coarea formula, such that 
\begin{equation} \label{e.selectlayer}
\|  \nabla v \|_{\underline{L}^2(\partial B_s)}
\leq 
4 \|  \nabla v \|_{\underline{L}^2(B_{\nicefrac{3r}{4}})}
\leq 
Cr^{-1}\left\| v \right\|_{\underline{L}^2(B_{r})}
 \,,
\end{equation}
where the last inequality above was obtained by applying the Caccioppoli inequality. 
Let~$w \in H^1(B_s)$ be the solution of the Dirichlet problem
\begin{equation*}
\left\{
\begin{aligned}
& -\nabla \cdot \ahom \nabla w = 0 & \mbox{in} & \ B_{s}, \\
& w = v
& \mbox{on} & \ \partial B_{s}\,. \\
\end{aligned}
\right.
\end{equation*}
By Proposition~\ref{p.DP} and the assumption that~$r \geq \X$, we have, for every~$\kappa \in (0,\frac14)$,  that
\begin{equation}
\label{e.vtildekappa.to.w}
\frac1s
\| v - w \|_{\underline{L}^2(B_s)} 
\leq
C\kappa^{\nicefrac12} 
\| \nabla w \|_{\underline{L}^2 (B_{s} \setminus B_{(1-2\kappa)s})}
+
C \kappa^{-1} \delta 
\left\| \nabla w \right\|_{\underline{W}^{1,\infty}(B_{(1-\kappa)s})} \,.
\end{equation}
To estimate the terms on the right, we first apply the interior pointwise bounds for harmonic functions (note that~$\ahom$-harmonic functions are harmonic after an affine change of variables), we have, for every~$x \in B_{(1-\kappa)s}$,   
\begin{equation*}  
\kappa s | \nabla^2 w(x)| + | \nabla w(x)| 
\leq 
C  \| \nabla w \|_{\underline{L}^2(B_{\kappa s/2}(x))} 
\leq 
C \kappa^{-\nicefrac d2} \| \nabla w \|_{\underline{L}^2(B_s)}\,.
\end{equation*}
We deduce that 
\begin{equation*}
\left\| \nabla w \right\|_{\underline{W}^{1,\infty}(B_{(1-\kappa) s})}
\leq
C \kappa^{-1-\nicefrac d2} \| \nabla w \|_{\underline{L}^2(B_s(x))} \leq C \kappa^{-1-\nicefrac d2} r^{-1} \left\| v \right\|_{\underline{L}^2(B_{r})} \,.
\end{equation*}
Here we also used the fact that~$\| \nabla w \|_{\underline{L}^2(B_s(x))}  \leq C \| \nabla v \|_{\underline{L}^2(B_s(x))} \leq C r^{-1} \left\| v \right\|_{\underline{L}^2(B_{r})}$. The estimate for the ``boundary layer term''~$\| \nabla w \|_{L^2 (B_{s} \setminus B_{(1-\kappa)s})}$ is more tricky. There are various ways to estimate it, and a common choice is to use Meyers' higher integrability estimate---applied first to~$v$ and then to~$w$---and then the H\"older inequality (see~\cite{AKMBook} for details). Here, we present an alternative argument that relies on the representation formula for harmonic functions in terms of the Poisson kernel. The necessary argument is formalized in Lemma~\ref{l.bndrlayer}, below, giving, together with the selection of~$s$ in~\eqref{e.selectlayer}, that
\begin{align*}  
\| \nabla w \|_{\underline{L}^2 (B_{s} \setminus B_{s-\kappa s})} 
& 
\leq
C  |\log \kappa|   \| \nabla v \|_{\underline{L}^2(\partial B_{s})}  \leq C |\log \kappa| r^{-1} \left\| v \right\|_{\underline{L}^2(B_{r})}
\,.
\end{align*}
Combining the above displays yields that
\begin{align*}  
\| v - w \|_{L^2(B_s)}  
\leq 
C \Bigl( \kappa^{\nicefrac12} |\log \kappa |
+ \delta \kappa^{-2-\nicefrac d2} \Bigr)  \| v  \|_{L^2(B_r)}
\,.
\end{align*}
We obtain~\eqref{e.delta.harmonicapprox} by selecting~$\kappa := \delta^{\frac{2}{d+5}}$. 

\smallskip

\emph{Step 2: Excess decay iteration.}
Turning now to the proof of the theorem, fix~$R \geq 2\X$ and~$u\in \mathcal{A}(B_R)$. We show that, for every~$\gamma\in (0,1)$, there exists~$C(\gamma,d,\lambda,\Lambda)<\infty$ such that 
\begin{equation}
\label{e.excessdecaygiveth}
\sup_{r \in [\X,R]}
r^{-(1+\gamma)} \inf_{\psi \in \A_1}\left\| u - \psi \right\|_{\underline{L}^2(B_{r})}
\leq 
C R^{-(1+\gamma)} \left\| u \right\|_{\underline{L}^2(B_{R})}\,.
\end{equation}
For each~$s\in [2\X,R]$, let~$\psi_s\in \A_1$ be such that
\begin{equation*}
\left\| u - \psi_s  \right\|_{\underline{L}^2(B_s)} 
=
\inf_{\psi\in \A_1} \left\| u - \psi \right\|_{\underline{L}^2(B_s)} .
\end{equation*}
Let~$\overline{\delta}:= \delta^{\frac1{d+6}}$.  Applying the result of Step~1 above to~$v = u - \psi_s$, 
we can select a function~$\overline{w}_s \in \Ahom(B_{s/2})$ such that 
\begin{equation*}
\left\| u - \psi_s - \overline{w}_s \right\|_{\underline{L}^2(B_{s/2})}
\leq 
C\overline{\delta} \left\| u - \psi_s \right\|_{\underline{L}^2(B_{s})}.
\end{equation*}
By the interior~$C^{1,1}$ estimate for harmonic functions, we have, for every~$r\in \left(0,\tfrac s2\right)$,
\begin{equation*}
\left\| \overline{w}_s - \ell_{\nabla \overline{w}_s(0)} -  \overline{w}_s(0) \right\|_{L^\infty(B_r)} 
\leq 
C \left( \frac rs \right)^2 \left\| \overline{w}_s \right\|_{\underline{L}^1(B_{s/2})} 
\leq 
C \left( \frac rs \right)^2
\left\| u - \psi_s \right\|_{\underline{L}^2(B_{s})}.
\end{equation*}
Set
\begin{equation*}
\tilde{\psi}:= \psi_s + \left( \ell_{\nabla \overline{w}_s(0)} + \phi_{\nabla \overline{w}_s(0)} - \left( \phi_{\nabla \overline{w}_s(0)} \right)_{B_{s/2}} 
+ \overline{w}_s(0) \right) \in \A_1.
\end{equation*}
By the assumption~$s/2 \geq \X$,~\eqref{e.correbounds.pre} and the~$C^1$-estimate for harmonic functions, we have  
\begin{align*}
\bigl\| \phi_{\nabla \overline{w}_s(0)}  - ( \phi_{\nabla \overline{w}_s(0)} )_{B_{s/2}} 
\bigr\|_{\underline{L}^2(B_{s/2})}
\leq C\overline{\delta} s \left| \nabla \overline{w}_s(0) \right| 
&
\leq C\overline{\delta} \left\| \overline{w}_s \right\|_{\underline{L}^2(B_{s/2})} 
\leq C \overline{\delta} \left\| u - \psi_s \right\|_{\underline{L}^2(B_{s})}.
\end{align*}
Combining the above displays using the triangle inequality, we obtain, for every~$r\in \left(0,\tfrac s2 \right)$, 
\begin{equation*}
\left\| u - \psi_r \right\|_{\underline{L}^2(B_{r})}
\leq
\bigl\| u - \tilde{\psi} \bigr\|_{\underline{L}^2(B_{r})}
\leq 
C \Bigl( \left( \frac rs \right)^2 
+ \overline{\delta} \left( \frac sr \right)^{d/2}\Bigr) 
\left\| u - \psi_s \right\|_{\underline{L}^2(B_{s})}.
\end{equation*}
That is, 
\begin{equation*}
r^{-(1+\gamma)} \left\| u - \psi_r \right\|_{\underline{L}^2(B_{r})}
\leq
C \biggl( 
\left( \frac rs \right)^{1-\gamma} 
+ \overline{\delta} \left( \frac sr \right)^{d/2+1+\gamma}\biggr) 
s^{-(1+\gamma)} \left\| u - \psi_s \right\|_{\underline{L}^2(B_{s})}.
\end{equation*}
Taking~$r:=\theta s$ for~$\theta(\gamma,d,\lambda,\Lambda) \in (0,\frac1{100}]$ sufficiently small, we get
\begin{align*}
(\theta s)^{-(1+\gamma)} \left\| u - \psi_{\theta s} \right\|_{\underline{L}^2(B_{\theta s})}
&
\leq
\biggl( \frac 14 + C \overline{\delta} \theta ^{-\left(\frac d2+1+\gamma\right)}\biggr) 
s^{-(1+\gamma)} \left\| u - \psi_s \right\|_{\underline{L}^2(B_{s})}
\\ & 
\leq
\biggl( \frac 14 + C \overline{\delta} \biggr) 
s^{-(1+\gamma)} \left\| u - \psi_s \right\|_{\underline{L}^2(B_{s})}
\,.
\end{align*}
Finally, taking~$\delta(\gamma, d,\lambda,\Lambda)\in (0,10^{-4}]$ sufficiently small, so that~$\delta^{\frac1{d+6}} C \leq \frac14$, we obtain
\begin{equation*}
(\theta s)^{-(1+\gamma)} \left\| u - \psi_{\theta s} \right\|_{\underline{L}^2(B_{\theta s})}
\leq 
\frac12 s^{-(1+\gamma)} \left\| u - \psi_s \right\|_{\underline{L}^2(B_{s})}.
\end{equation*}
By iterating starting from~$s \in [\theta R,R]$, or taking supremum over~$s \in [2\X,R]$, yields that
\begin{equation}
\label{e.excessdecaygiveth.pre}
\sup_{r \in [2\theta \X,R]}
r^{-(1+\gamma)} \left\| u - \psi_{r} \right\|_{\underline{L}^2(B_{r})}
\leq 
\sup_{r \in [\theta R,R]} r^{-(1+\gamma)} \left\| u - \psi_r \right\|_{\underline{L}^2(B_{r})}\,.
\end{equation}
Now~\eqref{e.excessdecaygiveth} follows upon giving up a volume factor with the constant~$C = \theta^{-1-\gamma-\nicefrac d2}$. 

\smallskip

\emph{Step 3: The conclusion.}
By~\eqref{e.excessdecaygiveth} and the triangle inequality, if~$r\in [2\X, R]$ and~$\theta \in [\frac12, 1]$, then 
\begin{equation}
\label{e.catchme}
r^{-(1+\gamma)} \left\| \psi_{r} - \psi_{\theta r} \right\|_{\underline{L}^2(B_{r})}
\leq 
C R^{-(1+\gamma)} \left\| u \right\|_{\underline{L}^2(B_{R})}
\,.
\end{equation}
By~\eqref{e.correbounds.abc} and the above display, we deduce that 
\begin{align*}  
\sup_{s \in [\X , \infty)} \frac1s \left\| \psi_{r} - \psi_{\theta r} \right\|_{\underline{L}^2(B_{s})}
\leq 
C r^\gamma R^{-(1+\gamma)} \left\| u \right\|_{\underline{L}^2(B_{R})}
\,.
\end{align*}
Setting~$\psi = \psi_\X$ and summing over dyadic~$r$, we obtain, for every~$r\in[\X,\frac12R]$,  
\begin{equation*}
r^{-(1+\gamma)} \left\| \psi  - \psi_{r} \right\|_{\underline{L}^2(B_{r})}
\leq 
C R^{-(1+\gamma)} \left\| u \right\|_{\underline{L}^2(B_{R})}
\,.
\end{equation*}
Thus, by the triangle inequality, the above bound together with~\eqref{e.excessdecaygiveth} yields~\eqref{e.intrinsicreg11-}.

\smallskip

The estimate~\eqref{e.C01} follows from the Caccioppoli inequality, the Poincar\'e inequality, the triangle inequality,~\eqref{e.correbounds.abc} and~\eqref{e.intrinsicreg11-}. Indeed, supposing without loss of generality that~$(u)_{B_R} = 0$, we obtain, for every~$r\in[\X,\frac12R]$, 
\begin{align*}
\left\| \nabla u \right\|_{\underline{L}^2(B_r)} 
&
\leq 
\left\| \nabla (u-\psi) \right\|_{\underline{L}^2(B_r)}  + \left\| \nabla \psi \right\|_{\underline{L}^2(B_r)} 
\\ & 
\leq 
\frac{C}{r} 
\left\| u-\psi  \right\|_{\underline{L}^2(B_{2r})}  
+ 
C \left\| \nabla \psi \right\|_{\underline{L}^2(B_{R/2})} 
\\ & 
\leq 
Cr^\gamma R^{-(1+\gamma)} 
\left\| u \right\|_{\underline{L}^2(B_R)} 
+
\frac{C}{R} \left\| \psi \right\|_{\underline{L}^2(B_R)} 
\\ & 
\leq
\frac{C}{R} \left\| u \right\|_{\underline{L}^2(B_R)} 
\leq 
C \left\| \nabla u \right\|_{\underline{L}^2(B_R)} \,.
\end{align*}
This completes the proof of the theorem. 
\end{proof}

In the argument above, we used the following technical lemma to estimate the boundary layers of~$\ahom$-harmonic functions. Note that the need for this lemma can be circumvented, if desired, by using an argument based on Meyers' estimate: see~\cite{AKMBook}. 

\begin{lemma} \label{l.bndrlayer}
Suppose~$\kappa \in (0,\frac1{10}]$,~$s >0$, and~$g \in H^{1}(\partial B_{s})$. Let~$w \in H^1(B_s)$ solve~$-\nabla \cdot \ahom \nabla w = 0$ in~$B_s$ with the trace~$w = g$ on~$\partial B_s$. Then there exists a constant~$C(d,\lambda,\Lambda) < \infty$ such that 
\begin{align}  \label{e.homog.bndrlayer}
\| \nabla w \|_{\underline{L}^2(B_{s} \setminus B_{(1-\kappa)s})} 
& \leq 
C |\log \kappa|   \| \nabla_{\mathrm{tan}} g \|_{\underline{L}^2(\partial B_{s})}  
\,,
\end{align}
where~$\nabla_{\mathrm{tan}}$ denotes the tangential derivative. 
\end{lemma}
\begin{proof}
Throughout the proof, we denote~$x' = \frac{s}{|x|} x$ and~$x_n = x-x'$. Extend~$g$ to~$B_s\setminus B_{s/2}$ by~$\bar g(x) = g(x')$ with~$\nabla \bar g(x) = \frac{s}{|x|} \nabla_{\mathrm{tan}} g(x')$.  First, the Caccioppoli inequality provides us
\begin{align}  \label{e.bndrlayer.cacc}
\| \nabla w \|_{\underline{L}^2(B_{s} \setminus B_{(1-\kappa)s})}  
\leq
\frac{C}{\kappa s} \| w - \bar g \|_{\underline{L}^2(B_{s} \setminus B_{(1-2\kappa)s})}
+
C \| \nabla_{\mathrm{tan}} g \|_{L^2(\partial B_{s})}
\,.
\end{align}
Let then~$P(x,y)$ be the Poisson kernel for~$B_s$ with respect to the operator~$-\nabla \cdot \ahom \nabla$, satisfying the estimate
\begin{align*}  
|P(x,y)| \leq \frac{C\dist(x,\partial B_s)}{|x-y|^d}\,.
\end{align*}
We have the representation, with~$\mathcal{H}^{d-1}$ being the surface measure, 
\begin{align}  \label{e.bndrlayer.wrepresent}
w(x) = \bar g(x)  
+
 \int_{\partial B_s} P(x,y') (g(y') -  g(x')) \, d\mathcal{H}^{d-1}(y') 
 \,.
\end{align}
Defining the maximal function
\begin{align*}  
M f(x') = \sup_{r>0} r^{1-d} \int_{\partial B_s \cap B_r(x')} f(y') \, d\mathcal{H}^{d-1}(y') ,
\end{align*}
we obtain, by denoting~$z_{x,y}(t) := s \frac{x' + t(y-x)'}{|x' + t(y-x)'|}$ and~$A_{2^k|x_n|}(x') : = B_{2^{k+1} |x_n|}(x')\setminus B_{2^{k} |x_n|}(x')$, that
\begin{align}  
\lefteqn{
\int_{\partial B_s} |P(x,y')| |g(y') -  g(x')| \, d\mathcal{H}^{d-1}(y') 
} \qquad &
\notag \\ 
&
\leq 
C \int_{0}^1 \int_{\partial B_s } 
\frac{|x_n| |x' - y'| }{(x_n^2 + |x' - y'|^2)^{\nicefrac d2}} |\nabla_{\mathrm{tan}} g(z_{x,y}(t))| \, d\mathcal{H}^{d-1}(y') \, dt
\notag \\ 
& 
\leq
C |x_n|^{2-d}
\int_{0}^1 \int_{\partial B_s \cap B_{|x_n|}(x') }  |\nabla_{\mathrm{tan}} g(z_{x,y}(t))| \, d\mathcal{H}^{d-1}(y')  \, dt
\notag \\ 
&  \qquad
+
C |x_n| \sum_{k=0}^{ \lceil \log_2 (s/|x_n|) \rceil}
(2^{k} |x_n|)^{1-d}  
\int_{0}^1 \int_{\partial B_s  \cap A_{2^k|x_n|}(x')}  
 |\nabla_{\mathrm{tan}} g(z_{x,y}(t))| \, d\mathcal{H}^{d-1}(y') \, dt
 \notag \\ 
& 
\leq
C |x_n| \log \frac{s}{|x_n|} (M \nabla_{\mathrm{tan}} g)(x')
\,.
 \notag
\end{align}
It follows by~$L^2(\partial B_s) \to L^2(\partial B_s)$ boundedness of the maximal function that 
\begin{align*}  
\fint_{B_s \setminus B_{(1-\kappa)s}}  
\biggl( 
\int_{\partial B_s} |P(x,y')| |g(y') -  g(x')| \, d\mathcal{H}^{d-1}(y') 
\biggr)^2 \, dx 
\leq 
C (\kappa s)^2 \log^2 \kappa  \| \nabla_{\mathrm{tan}} g \|_{\underline{L}^2(B_s)}^2
\,.
\end{align*}
Merging this with the Caccioppoli estimate~\eqref{e.bndrlayer.cacc} and~\eqref{e.bndrlayer.wrepresent} completes the proof.  
\end{proof}

Removing the prefactor of~$|\log \kappa|$ from the previous lemma is possible by employing more careful analysis via singular integrals. However, the proof is slightly more involved, and the logarithmic prefactor is harmless for our purposes, so we omit the proof. 

\smallskip

We conclude with a corollary of Theorem~\ref{t.C11.sharp} giving a soft version of the~$C^{1,1-}$ estimate, namely a ``Liouville theorem'' which gives a classification of solutions which grow at most like~$O(|x|^{2-\ep})$ as~$|x|\to \infty$ for some~$\ep>0$: the only such solutions belong to~$\mathcal{A}_1$. 

\begin{corollary}[Liouville theorem] 
\label{cor.Liouville}
Let~$\P$ be a stationary probability measure. 
Suppose that~$\ep>0$ and~$u\in H^1_{\mathrm{loc}}(\Rd)$ satisfies~$-\nabla \cdot\a\nabla u = 0$ in~$\Rd$ and the growth condition
\begin{equation*}
\liminf_{r \to \infty} 
r^{-(2-\ep)} 
\| u \|_{\underline{L}^2(B_r)} 
< \infty. 
\end{equation*}
Suppose also that~$\a(\cdot)$ belongs to the event in which~$\X < \infty$ (an event of full probability by Theorem~\ref{t.C11.sharp}). 
Then~$u \in\mathcal{A}_1$.
\end{corollary}
\begin{proof}
Let~$\gamma \in ( 1-\ep, 1 )$. 
Select~$R_j \to \infty$ such that~$A:= \sup_{j\in\N} R_j^{-(2-\ep)} 
\| u \|_{\underline{L}^2(B_{R_j})} < \infty$. By Theorem~\ref{t.C11.sharp}, for every~$r> \X$ and~$j\in\N$ with~$R_j  > r$, we have that 
\begin{equation*}
\inf_{\phi\in \mathcal{A}_1} 
\left\| u - \phi \right\|_{\underline{L}^2(B_r)} 
\leq 
C \left( \frac r {R_j} \right)^{1+\gamma} 
\left\| u  \right\|_{\underline{L}^2(B_{R_j})}
\leq 
CA r^{1+\gamma} {R_j}^{1-\ep - \gamma}
\rightarrow 0 \quad \mbox{as} \ j\to \infty.
\end{equation*}
Therefore, in every ball~$B_r$ with~$r>\X$, the function~$u$ coincides with an element of~$\mathcal{A}_1$. Clearly this implies that~$u\in \mathcal{A}_1$. 
\end{proof}

\subsection*{Historical remarks and further reading}

Qualitative homogenization in the periodic setting was initiated in~\cite{DGS} and developed in the late 1970s and early 80s. Expositions of this classical theory can be found in the books~\cite{BLP,JKO,T}. For a more modern perspective, with a quantitative focus, see~\cite{ShenBook}. 

\smallskip

Stochastic homogenization in general stationary-ergodic environments was also developed during this time: by~\cite{PV,Koz} in the linear case, and~\cite{DM1,DM2} in the nonlinear case. 
Our presentation in the stochastic case follows~\cite{PV} in Section~\ref{ss.random} and~\cite{DM1,DM2} in Section~\ref{ss.variational}. 
The multiparameter ergodic theorems, Propositions~\ref{p.ergodic} and~\ref{p.subadditive.ergodic}, are generalizations of those appearing in~\cite{Becker} and~\cite{Akc}, respectively. Compared to those works, both our ergodic theorems and homogenization theorems require only~$\Zd$-stationarity rather than~$\Rd$-stationarity. This allows us to cover periodic homogenization results in the random setting as well as canonical examples like the random checkerboard. We also do not require ergodicity of the probability measure---the price to pay, of course, is that the homogenized coefficients are not deterministic but rather translation-invariant random variables.

\smallskip

Large-scale regularity estimates like the one presented in Section~\ref{ss.reg} first appeared in the periodic setting in the celebrated work of Avellaneda and Lin~\cite{AL1,AL2}, who proved them using a compactness method, which is a variant of the argument here. Such results were first obtained in the stochastic case in~\cite{AS}, under stronger ergodicity assumptions, and in the very general qualitative setting in~\cite{GNO3}.
The modest generalizations, both to the case of non-ergodic coefficient fields and for~$\Zd$-stationarity (as opposed to~$\Rd$-stationarity), appear to be new here.

\section{Homogenization in probabilistic language}
\label{s.probability}

Elliptic and parabolic equations in divergence form are the forward Kolmogorov equations of certain reversible diffusion processes. The relationship between the analytic and probabilistic objects---partial differential operators are the infinitesimal generators of the corresponding Markov processes---is so tight that they should be regarded as different formalizations for the same underlying phenomena. Therefore, a homogenization statement about the convergence of a sequence of partial differential operators can be stated equivalently in terms of the convergence of a sequence of diffusion processes. 

\smallskip

In this chapter, we will explore in more detail what homogenization means in the probabilistic language of diffusion processes. In particular, we will state an \emph{central limit theorem (or invariance principle) for diffusions in random environments}, which is an assertion about the convergence of the large-scale, blow-down limit of a symmetric diffusion process to a multiple of Brownian motion. We will present two proofs of the invariance principle, each of which relies on the sublinearity of the correctors stated in~\eqref{e.corrector.qualbound.L2}. The first proof, presented in Section~\ref{ss.invariance}, is more probabilistic in flavor and uses an abstract central limit theorem for general martingales. The second proof, presented in Section~\ref{ss.Green}, is more analytic: we prove a homogenization result for the parabolic Green function---the transition function for the process---and deduce the invariance principle as a corollary. We begin in Section~\ref{ss.markov} by recalling some basic definitions and properties of diffusion processes on~$\Rd$. 

\smallskip

This chapter is independent of the rest of the text and may be skipped on a first reading. An alternative treatment of most of the material presented here can be  found for instance in~\cite{Kallenberg}, including a very general statement asserting the equivalence between homogenization and central limit theorems for Feller processes~\cite[Theorem 19.25]{Kallenberg}.

\subsection{Symmetric diffusion processes and their generators}
\label{ss.markov}

This section is a crash course in Markov diffusion processes on~$\Rd$. We briefly summarize the definitions and basic properties of these objects, mostly without proof. 
The reader is invited to consult Revuz and Yor~\cite{RY} for a thorough presentation of the theory.

\smallskip

A \emph{stochastic process} on~$\Rd$ indexed by a set~$T$, taking values in a measurable space~$(E, \mathscr{E})$, is a family~$\{ X_t\,:\,t\in T\}$ of~$E$-valued random elements, defined with respect to an underlying probability space~$(\Upsilon, \mathcal{G}, \mathbf{P})$. 
We will exclusively consider the case in which the index set is~$T = \R_+ = [0,\infty)$, which is interpreted to be ``time,'' and the state space is~$(E,\mathscr{E}) = (\R^d,\mathscr{B})$, with~$\mathscr{B}$ the Borel~$\sigma$-algebra. 
We think of the trajectory~$t \mapsto X_t$ as a random curve in~$\Rd$. 
We say that~$\{ X_t\}$ \emph{has continuous sample paths}, or \emph{is a sample-continuous process} if the trajectories~$t\mapsto X_t$ are almost surely continuous.

\smallskip

Recall that a \emph{filtration}~$\{ \mathcal{G}_t \}$
on the probability space~$(\Upsilon, \mathcal{G}, \mathbf{P})$
is a family of~$\sigma$-algebras satisfying~$\mathcal{G}_s \subseteq \mathcal{G}_t \subseteq \mathcal{G}$ for every~$s<t$; in this case, we say that~$(\Upsilon, \mathcal{G}, \{ \mathcal{G}_t \},\mathbf{P})$ is a \emph{filtered probability space}. 
We denote by~$C_0(\Rd)$ the set of continuous functions on~$\Rd$ which vanish at infinity, with the topology induced by~$\| \cdot \|_{L^\infty}$. 

\smallskip

Roughly speaking, a Markov process is a stochastic process with the property that its future behavior depends only on the current state. 
To formalize this, we need to introduce a few more concepts. 
A \emph{transition probability kernel}~$N$ on~$(\Rd,\mathscr{B})$ is a function 
\begin{equation*}
N:\Rd \times \mathscr{B} \to \R \cup \{ +\infty\}
\end{equation*}
satisfying the following three properties:
\begin{itemize}
\item For every~$x	\in \Rd$, the mapping~$A \mapsto N(x,A)$ is a positive measure on~$\mathscr{B}$;

\item For every~$A \in \mathscr{B}$, the mapping~$x \mapsto N(x,A)$ is Borel measurable;

\item For every~$x	\in \Rd$,~$N(x,\Rd) =1$. 
\end{itemize}
A \emph{Markov transition function} is a family~$\{ P_{s,t}\,:\, 0 \leq s < t < \infty \}$ of transition probability kernels which satisfies the \emph{Chapman-Kolmogorov condition}
\begin{itemize} 

\item 
$\displaystyle\int_{\Rd} P_{s,t'}(x,dy) P_{t',t} (y,A) 
=
P_{s,t} (x,A),$
for every~$A \in \mathscr{B}$ and~$0\leq s < t' < t < \infty$.
\end{itemize}
We say that a Markov transition function~$\{ P_{s,t} \}$ is \emph{homogeneous} if~$P_{s,t}$ depends only on the difference~$t-s$; in this case we can write simply~$P_t:= P_{0,t}$ and the Chapman-Kolmogorov condition becomes the \emph{semigroup property}
\begin{equation*}
\int_{\Rd} P_{s}(x,dy) P_{t} (x,A) 
=
P_{t+s} (y,A),
\end{equation*}
A homogeneous Markov transition function~$\{ P_t \}$ is called a \emph{Feller transition function} if it satisfies the following two additional properties: 
\begin{itemize}
\item For every~$f \in C_0(\Rd)$, 
the function~$\displaystyle x\mapsto \int_{\Rd} f(y) P_t (x,dy)$ belongs to~$C_0(\Rd)$;
\item For every~$f\in C_0(\Rd)$, 
$\displaystyle\lim_{t \downarrow 0} 
\int_{\Rd} f(y) P_t (x,dy) = f(x)
$. 
\end{itemize}
If~$\{ P_t\}$ is Feller, it generates a semigroup of operators~$C_0(\Rd) \to C_0(\Rd)$ defined by
\begin{equation}
\label{e.TfromP}
f \mapsto \int_{\Rd}f(y) P_{t}(\cdot,dy)  
\,.
\end{equation}
Conversely, a semigroup of operators~$\{ T_t \}_{t\geq 0}$ on~$C_0(\Rd)$ is given by integration against a Feller transition function (as above) if and only if it satisfies, in addition to being a semigroup, the following:
\begin{itemize}

\item~$\| T_t f \|_{L^\infty(\Rd)} \leq \| f \|_{L^\infty(\Rd)}$ \
and \
$\displaystyle\lim_{t \downarrow 0} \| T_t f -f \|_{L^\infty(\Rd)} = 0$, \ for every~$f\in C_0(\Rd)$. 

\end{itemize}
It is evident that Feller semigroups and Feller transition functions are basically the same thing. It is, therefore, only a slight abuse of notation that we henceforth let~$P_t$ denote both the transition function~$P_t(x,A)$ as well as the semigroup of operators given in~\eqref{e.TfromP}. We also use the terms \emph{Feller transition function} and \emph{Feller semigroup} interchangeably. 

\smallskip

We next define the notion of a continuous-time Markov process on~$\Rd$. 
\begin{definition}[Markov processes]
Let~$(\Upsilon, \mathcal{G}, \{ \mathcal{G}_t \},\mathbf{P})$ be a filtered probability space,~$P_{s,t}$ a Markov transition function on~$(\Rd,\mathscr{B})$ and~$\{ X_t \}$ a stochastic process indexed by~$[0,\infty)$ and taking values in~$\Rd$. 
We say that~$\{ X_t \}$ is a \emph{Markov process with respect to~$\mathcal{G}_t$ with transition function~$P_{s,t}$} if:
\begin{itemize}
\item~$X_t$ is~$\mathcal{G}_t$--measurable for every~$t\in [0,\infty)$;
\item~$\displaystyle\mathbf{P} \bigl[ X_t \in A \, | \, \mathcal{G}_s \bigr] = P_{s,t}(X_s,A)$
for every~$A \in \mathscr{B}$ and~$0\leq s < t < \infty$.
\end{itemize}
A Markov process is called \emph{homogeneous} (respectively, \emph{Feller}) if its transition function is homogeneous (resp., Feller). 
\end{definition}

We let~$\mathbf{E}$ denote the expectation with respect to~$\mathbf{P}$. Observe that the condition~$\mathbf{P} \bigl[ X_t \in A \, | \, \mathcal{G}_s \bigr] = P_{s,t}(X_s,A)$ in the definition above may be written equivalently as
\begin{equation}
\label{e.Xt.to.Pt}
\mathbf{E} \bigl[ f(X_t) \, |\, \mathcal{G}_s \bigr] 
=
\int_{\Rd} 
f(y) P_{s,t} (X_s,dy) \,, \quad \forall f \in C_0(\Rd)\,.
\end{equation}

\begin{example}
The quintessential example of an~$\Rd$-valued Markov process is Brownian motion, indexed by~$[0,\infty)$, which we denote by~$\{ B_t \}_{t\geq 0}$. It is the Feller process starting at the origin,~$B_0=0$, with the transition function given by the Gaussian: that is, 
\begin{equation}
\label{e.transition.BM}
\mathbf{P}\bigl[ B_t \in A \, | \, \mathcal{G}_s \bigr]
=
\bigl( 2\pi (t-s) \bigr)^{-\nicefrac d2}
\int_A 
\exp \biggl( - \frac{|x-y|^2}{2(t-s)} \biggr)\, dy
\,.
\end{equation}
\end{example}

Two processes~$X$ and~$Y$ defined on~$(\Upsilon, \mathcal{G}, \{ \mathcal{G}_t \},\mathbf{P})$ are said to be \emph{modifications} of each other if, for every~$t \in [0,\infty)$, we have $\mathbf{P}[X_t = Y_t]=1$. A Markov process~$\{ X_t \}$ has \emph{cadlag paths} and is called a \emph{cadlag process} if, for~$\mathbf{P}$--almost every~$\omega \in \Upsilon$, the path~$t \mapsto X_t(\omega)$ is right-continuous and has left-hand limits; that is,~$\lim_{s\uparrow t} X_s(\omega)$ exists and~$X_t(\omega) = \lim_{s\downarrow t} X_s(\omega)$. 
As every Feller process has a cadlag modification (see \cite[Theorem III.2.7]{RY}), to avoid unnecessary measurability issues we henceforth always work with this modification. 

\smallskip

The following proposition states that every Markov transition function~$P_{s,t}$ gives rise to a Markov process~$X_t$ and that Markov processes are completely determined by their transition functions---in fact, by their finite-dimensional distributions. Of course, such a statement can be valid only if we identify Markov processes that have the same law. 
For this, it is convenient to work with a ``canonical'' probability space (much as we did in Section~\ref{ss.random} when we defined our random coefficient fields~$\a(x)$). 
We therefore take 
\begin{equation}
\label{e.Upsilon.canon}
\Upsilon = (\Rd)^{\R_+} = \bigl\{ X : \R_+ \to \Rd \bigr\}\,
\quad \mbox{and} \quad 
\mathcal{G} = \mathscr{B}^{\R_+}\,.
\end{equation}
We let~$X_t$ be the \emph{canonical} or \emph{coordinate process} on~$\Upsilon$,  given by~$X_t(Y) = Y_t$,~$Y \in \Upsilon$. The canonical filtration is defined by
\begin{equation}
\label{e.Gt.canon}
\mathcal{G}_t:= \mbox{the~$\sigma$--algebra generated by the random variables~$\big\{ X_s \,:\, s\in [0,t] \big\}$.}
\end{equation}

\begin{proposition}[Existence and uniqueness of Markov processes]
\label{p.Kolmogorov}
\hspace{0.1pt}
For every probability measure~$\nu$ on~$(\Rd,\mathscr{B})$ and Markov transition function~$P_{s,t}$,  there exists a unique probability measure~$\mathbf{P}^\nu$ on the canonical measurable space~$(\Upsilon,\mathcal{G})$, defined in~\eqref{e.Upsilon.canon}, such that the canonical process~$X_t$ is a Markov process with respect to the filtration~$\mathcal{G}_t$, defined in~\eqref{e.Gt.canon}, with transition function~$P_{s,t}$ and initial distribution~$\nu$; that is,~$\mathbf{P}^\nu \bigl[ X_0 \in A \bigr] = \nu(A)$. 
Moreover, for every~$k\in\N$,~$0 = t_0 < t_1 < \ldots < t_k$ and~$f_0,\ldots, f_k \in C_0(\Rd)$, 
\begin{equation}
\label{e.finitedimdist}
\mathbf{E}^\nu\Biggl[ \prod_{i=0}^k f_i(X_{t_i}) \Biggr] 
=
\int_{\Rd} f_0(x_0)\,d\nu(x_0)
\prod_{i=1}^k
\int_{\Rd} f_i(x_i) P_{t_{i-1},t_i} (x_{i-1},dx_i)\,,
\end{equation}
where~$\mathbf{E}^\nu$ is the expectation with respect to~$\mathbf{P}^\nu$.
\end{proposition}

See~\cite[III.1.4 \& III.1.5]{RY} for a proof of Proposition~\ref{p.Kolmogorov}, which follows essentially from the Kolmogorov extension theorem and the characterization of Markov processes by finite-dimensional distributions. 

\smallskip

Proposition~\ref{p.Kolmogorov} tells us that the Markov transition function contains all the information about the underlying Markov process and thereby opens the door for analysis methods in the study of diffusion processes. Indeed, we can study Markov transition functions using functional analysis techniques without using the language of probability theory.

\smallskip

As it turns out, many Markov transition functions we will be interested in are the parabolic Green functions for certain parabolic PDEs. To explain this connection, we must introduce the concept of an \emph{infinitesimal generator}. For simplicity, we will restrict our discussion to Feller processes.

\begin{definition}[Infinitesimal generator]
Consider a Feller process~~$\{ X_t \}$ with Feller semigroup~$\{ P_t\}$.
We define the \emph{infinitesimal generator of~$\{X_t\}$} by
\begin{equation}
\label{e.generator}
L f:= 
\lim_{t \downarrow 0} 
\frac1t 
\bigl( P_t f  - f \bigr)\,, \quad f \in \mathscr{D}_L\,, 
\end{equation}
where~$\mathscr{D}_L \subseteq C_0(\Rd)$, which we call the domain of~$L$, is the linear subspace of~$C_0(\Rd)$ consisting of the~$f$'s for which the above limit holds in the topology of~$C_0(\Rd)$. 
\end{definition}

The infinitesimal generator tells us how the Feller process evolves in an infinitesimal time step. The definition~\eqref{e.generator} suggests that, formally, we should have~$P_t = \exp(tL)$ and, in particular, be able to recover~$P_t$ from~$L$. The rigorous sense in which this is valid is the Kolmogorov backward equation, derived below in~\eqref{e.Kolmogorov.backward}. This tells us that Feller semigroups and their infinitesimal generators are basically the same thing. We can, therefore, reduce the study of any Feller process to the study of its infinitesimal generator.

\smallskip

In practice, many of the infinitesimal generators for processes we care about (and which arise naturally in physics) are partial differential operators---specifically, elliptic partial differential operators. This is how we can translate questions about Feller processes into questions about elliptic and parabolic PDEs and vice versa. 

\begin{example} 
The infinitesimal generator of Brownian motion on~$\Rd$ is~$\frac12\Delta$, and its domain contains~$C^2_0(\Rd)$. This is a simple consequence of It\^o's formula. 
Alternatively, we can perform the following direct computation. Recall by~\eqref{e.transition.BM} that the transition function of Brownian motion is the Gaussian
\begin{equation*}
\Phi(t,x,y) = (2\pi t)^{-\nicefrac d2} \exp\biggl( -\frac{|x-y|^2}{2t} \biggr)\,.
\end{equation*}
Using that~$\Phi(\cdot,y)$ satisfies~$\partial_t \Phi = \frac12 \Delta \Phi$, 
we compute, for any~$f\in C^2_0(\Rd)$, 
\begin{align*}
\lim_{t\downarrow 0} 
\frac1t \biggl( \int_{\Rd} \Phi(t,x,y) f(y)\,dy - f(x)  \biggr) 
&
= 
\lim_{t\downarrow 0} 
\lim_{s \downarrow 0} 
\int_{\Rd} \frac1t \bigl(\Phi(t+s,x,y) - \Phi(s,x,y)\bigr) f(y) \,dy
\\ &
=
\lim_{s\downarrow 0} 
\int_{\Rd} \partial_t \Phi(s,x,y)f(y) \,dy
\\ & 
=
\lim_{s\downarrow 0} 
\int_{\Rd} \frac12 \Delta \Phi(s,x,y) f(y) \,dy
\\ & 
=
\lim_{s\downarrow 0} 
\int_{\Rd} \Phi(s,x,y) \frac12 \Delta f(y) \,dy
\\ & 
=
\frac12 \Delta f(x) 
\,.
\end{align*}
Note that the presence of~$f$ allows us to interchange the limits in~$s$ and~$t$ in the above computation. 
\end{example}

In general, the domain of a Feller process is not straightforward to compute explicitly. We can, however, often show that certain nice subsets of~$C_0(\Rd)$ are subsets of~$\mathscr{D}_L$, such as~$C^k_0(\Rd)$, the subspace of~$C_0(\Rd)$ consisting of functions whose partial derivatives of order at most~$k$ belong to~$C_0(\Rd)$. Moreover,~$\mathscr{D}_L$ is always a dense set of~$C_0(\Rd)$ and~$L$ is a closed operator (see~\cite[VII.1.3]{RY}). 
Moreover,~$P_t$ maps~$\mathscr{D}_L$ into~$\mathscr{D}_L$ and commutes with~$L$ and,  consequently, using also the definition of~$L$ and the semigroup property, we deduce that if~$f\in \mathscr{D}_L$, then~$t \mapsto P_t f$ is differentiable and
\begin{equation}
\label{e.Kolmo}
\partial_t \bigl( P_t f \bigr) = L \bigl( P_t f \bigr) = P_t ( Lf) \,, \quad f\in \mathscr{D}_L\,.
\end{equation}
(See~\cite[VII.1.2]{RY}.) If we write this a bit more explicitly, it says that 
\begin{equation}
\label{e.Kolmog.forward}
\int_{\Rd} f(y) \partial_t P_t(\cdot,dy) = 
\int_{\Rd} Lf (y) P_t(\cdot,dy) 
\,, \quad \forall f \in \mathscr{D}_L\,.
\end{equation}
This says that, for every~$x\in\Rd$, the function~$P_t(x,\cdot)$ satisfies the equation
\begin{equation}
\label{e.Kolmogorov.forward.weak}
\partial_t P_t(x,\cdot) = L^* P_t(x,\cdot)
\end{equation}
in a weak sense (precisely, in the sense of~\eqref{e.Kolmog.forward}) where~$L^*$ denotes the formal adjoint of~$L$.
This equation is called the \emph{forward Kolmogorov equation}, or the \emph{Fokker-Planck equation}. 
The first equality in~\eqref{e.Kolmo} says that 
\begin{equation}
\label{e.Kolmog.backward}
\int_{\Rd} f(y) \partial_t  P_t(x,dy) = 
\biggr( L\int_{\Rd} f (y)  P_t(\cdot,dy) \, dy \biggr)(x) 
\,, \quad \forall x\in\Rd\,, \ \ f \in \mathscr{D}_L\,.
\end{equation}
Since~$L$ is linear, we should be able to put the~$L$ inside the integral and onto~$P_t(\cdot,dy)$, assuming, of course, that~$P_t(\cdot,dy) \in \mathscr{D}_L$. In this case, we would deduce that 
\begin{equation}
\label{e.Kolmogorov.backward}
\partial_t P_t(\cdot,A) = LP_t(\cdot,A) \,, \quad \forall A\in \mathscr{B}\,.
\end{equation}
which is the \emph{Kolmogorov backward equation}. Note that we may interpret~\eqref{e.Kolmog.backward} as a rigorous, weak version of~\eqref{e.Kolmogorov.backward}, in the event that we cannot justify~\eqref{e.Kolmogorov.backward} in a strong sense. 

\smallskip

Recall that a process~$\{ Z_t \}_{t\geq 0}$ on~$\Rd$ is a martingale with respect to the filtration~$\mathcal{G}_t$ if~$\E[ Z_t \vert \mathcal{G}_s ] = Z_s$ for every~$0\leq s < t<\infty$. 
For any~$f \in C_0(\Rd)$ which satisfies~$L f = 0$, we may apply~\eqref{e.Xt.to.Pt} to deduce that, for every~$s<t$, 
\begin{equation}
\mathbf{E}\bigl[ f(X_t) \,|\, \mathcal{G}_s \bigr] 
=
\int_{\Rd} 
f(y) P_{t-s} (X_s,dy) 
=
f(X_s)\,.
\end{equation}
Thus~$f(X_t)$ is a martingale with respect to~$\mathcal{G}_t$. 
More generally, we have the following result.

\begin{lemma}
Let~$\{X_t\}$ be a Feller process and~$L$ its infinitesimal generator. Then, for any~$f\in C_0(\Rd)$, the process
\begin{equation}
\label{e.L.gives.martingales}
M_t := f(X_t) - f(X_0) - \int_0^t (Lf)(X_s)\,ds
\quad \mbox{is a martingale.}
\end{equation}
\end{lemma}
\begin{proof}
Using~\eqref{e.Xt.to.Pt}, we compute
\begin{align*}
\mathbf{E} \bigl[ M_t \, \vert\, \mathcal{G}_s \bigr] 
&
=
\mathbf{E}\bigl[ f(X_t) \,|\, \mathcal{G}_s \bigr] 
- f(X_0) 
-
\mathbf{E}\biggl[ 
\int_0^t
(Lf)(X_r)\,dr
\,\Big\vert\, \mathcal{G}_s
\biggr]
\\ & 
= 
\! \int_{\Rd} \!
f(y) P_{t-s} (X_s,dy) 
- f(X_0)
-
\! \int_0^s (Lf)(X_r)\,dr
-
\int_s^t \!
\int_{\Rd} 
(Lf)(y) P_{t-r}(X_r,dy)
\,dr
\end{align*}
and, by~\eqref{e.Kolmog.forward}, 
\begin{align*}
\int_s^t 
\int_{\Rd} 
(Lf)(y) P_{t-r}(X_r,dy)
\,dr
&
=
\int_{\Rd} 
f(y) 
\int_s^t 
\partial_t P_{r-s}(X_s,dy)
\,dr
\\ & 
=
\int_{\Rd} 
f(y) P_{t-s}(X_s,dy)
-
f(X_s)\,.
\end{align*}
Combining these yields~$\mathbf{E} \bigl[ M_t \, \vert\, \mathcal{G}_s \bigr] = M_s$ and hence the claim~\eqref{e.L.gives.martingales}. 
\end{proof}

\begin{example}[Feller diffusion processes]
Let~$\b: \Rd \to \Rd$ be a bounded, Lipschitz vector field on~$\Rd$,~$\sigma:\Rd \to \R^{d\times d}$ a matrix-valued field which is also bounded and Lipschitz,~$\rho$ a measure on~$(\Rd,\mathscr{B})$, and~$\{B_t\}$ be a Brownian motion. 
According to the standard theory of stochastic differential equations, 
there exists a unique strong solution~$\{X_t\}$ of the equation
\begin{equation}
\label{e.SDE}
dX_t = \b(X_t) \, dt + \sigma(X_t) \, dB_t
\end{equation}
such that~$X_0$ has distribution~$\rho$. 
The equation~\eqref{e.SDE} means that the stochastic process~$\{ X_t\}$ is adapted to the same filtration as the Brownian motion and that the following integral equation is valid, where the second integral is interpreted in the sense of It\^o:
\begin{equation}
\label{e.SDE.Ito}
X_t = X_0 + \int_0^t \b(X_s) \, ds + \int_0^t \sigma(X_s) \, dB_s\,. 
\end{equation}
The solution~$\{ X_t\}$ of the SDE~\eqref{e.SDE} turns out to be a Feller process, and its infinitesimal generator has domain~$\mathscr{D}_L \supseteq C^2_0(\Rd)$ which, when restricted to~$C^2_0(\Rd)$. is given by 
\begin{equation}
\label{e.L.nondivform}
L = 
\sum_{i,j=1}^d \a_{ij} \partial_{x_i} \partial_{x_j} 
+
\sum_{j=1}^d \b_j \partial_{x_j} 
\,, \quad \mbox{where} \quad \a := \frac12\sigma\sigma^t\,.
\end{equation}
The Kolmogorov forward operator is the parabolic partial differential operator~$\partial_t - L$, and the Feller semigroup and transition function are, respectively, the flow of this parabolic equation and its parabolic Green function. (See~\cite[IX.2]{RY} for a proof of these facts.)
The Markov processes arising as strong solutions of the SDE in~\eqref{e.SDE} are called \emph{Feller diffusion processes}, or simply~\emph{diffusion processes}. 
\end{example}

The relationship between the diffusion process~\eqref{e.SDE} and the partial differential equation (its forward Kolmogorov equation)
\begin{equation}
\label{e.PDE.nondivform}
\partial_t u - \sum_{i,j=1}^d \a_{ij} \partial_{x_i} \partial_{x_j} u
-
\sum_{j=1}^d \b_j \partial_{x_j} u = 0 
\end{equation}
goes in both directions. On the one hand, we can write representation formulas for the solutions of the PDE in terms of the Markov process. The solution of~\eqref{e.PDE.nondivform} with initial data~$u(0,\cdot) = f \in C_0(\Rd)$ is given by the formula
\begin{equation}
\label{e.formula}
u(t,x) 
=
\mathbf{E}^x 
\bigl[
f( X_t) 
\bigr]
\,,
\end{equation}
where we denote~$\mathbf{E}^x=\mathbf{E}^{\delta_x}$ and~$\mathbf{P}^x=\mathbf{P}^{\delta_x}$ for the process starting from~$x$. 
The formula~\eqref{e.formula} is immediate from the definitions since the right side satisfies the PDE by~\eqref{e.Xt.to.Pt} and~\eqref{e.Kolmog.backward}. 
We can therefore gain information about the solutions of the PDE by carefully studying the trajectories of the diffusion process. 
Conversely, by varying the same formula over collections of well-chosen~$f$'s, we can use knowledge about solutions of the PDE to gain information about the statistical properties of the diffusion process. In principle, all information about the process is contained in the PDE and vice versa.\footnote{In this sense, the community of probabilists studying diffusion processes and the community of analysts studying properties of solutions of the parabolic partial differential equation~\eqref{e.PDE.nondivform} are just studying the same underlying objects using different languages.} 

\smallskip

Notice that the operator~$L$ is in \emph{nondivergence form}. 
We can, of course, write it in divergence form, assuming that~$\sigma$ and thus the diffusion matrix~$\a$ belong to~$C^1(\Rd)$:
\begin{equation*}
L f = 
\sum_{i,j=1}^d \partial_{x_i} \bigl(  \a_{ij} \partial_{x_j} f \bigr)
+
\sum_{j=1}^d \bigl( \b_j - \partial_{x_i} \a_{ij} \bigr) \partial_{x_j} f 
\,, \quad \forall f \in C^2_0(\Rd)\,.
\end{equation*}
In the particular case that~$\b_j = \partial_{x_i} \a_{ij}$, the drift term disappears and we get~$L = \nabla \cdot \a\nabla$, our familiar self-adjoint, divergence-form elliptic operator. The condition~$\b_j = \partial_{x_i} \a_{ij}$ is equivalent to
\begin{equation}
\label{e.drift.be.good}
\b_j = \frac12 \sum_{i,k=1}^d 
\partial_{x_i}( \sigma_{ik}\sigma_{kj}) 
\,,
\quad \mbox{that is,} \quad 
\b 
=
\frac12 \nabla \cdot (\sigma \sigma^t) \,,
\end{equation}
and so the SDE~\eqref{e.SDE} can be written in this case as
\begin{equation}
\label{e.SDE.reversible}
dX_t = \frac12 \nabla \cdot (\sigma \sigma^t)  \, dt + \sigma(X_t) \, dB_t
\,.
\end{equation}
Diffusion processes satisfying SDEs of the form~\eqref{e.SDE.reversible} comprise an important subclass of diffusion processes called \emph{reversible diffusions processes}.

\smallskip

Why is the word ``reversible'' used in this context? Intuitively, a time-reversible Markov process looks the same forwards and backwards in time; roughly speaking, this means that~$\mathbf{P}\bigl[ X_t \sim x \,|\, X_s \sim y \bigr] \approx \mathbf{P}\bigl[ X_s \sim x \,|\, X_t \sim y \bigr]$, for every~$0\leq s< t$. In the context of a Feller process, in terms of the transition function, this means that~$P_t(x,y) = P_t(y,x)$, where the measure~$P_t(x,\cdot)$ is given by a density which we denote by~$P_t(x,dy) = P_t(x,y)\,dy$.

\smallskip

Therefore, a \emph{reversible diffusion process} is one in which the Markov transition function is symmetric in its two variables. In view of the Kolmogorov forward and backward equations, this is equivalent to~$L=L^*$, the self-adjointness of the infinitesimal generator,\footnote{More generally, reversibility means self-adjointness of the infinitesimal generator with respect to the~$L^2$ inner product defined in terms of an equilibrium (or stationary) measure. In our context, this is just Lebesgue measure, but in general, it can be another measure.} which is evidently equivalent to~$L$ being a divergence-form elliptic operator, which is the same as the condition~\eqref{e.drift.be.good}. 

\smallskip

It may seem like the condition~\eqref{e.drift.be.good} is unnatural, enforced to shoehorn the generator~$L$ into one that fits our assumptions, giving us a stochastic process to play with. Nothing could be further from the truth! The reversibility condition corresponds to \emph{detailed balance conditions} in physical models, which are abundant. Thus, it is the diffusions with infinitesimal generators having a principal part in divergence-form, which are the important and physical ones, much more than the ones corresponding to operators with a principal part in nondivergence form. 

\smallskip

Up to this point, we have seen that a divergence-form operator of the form~$-\nabla \cdot \a\nabla$ has a close relationship to a diffusion process---namely, the solution of the SDE in~\eqref{e.SDE} with the particular choice of drift vector in~\eqref{e.drift.be.good}---provided that~$\a \in C^1$. A natural question is whether we can define such a process if~$\a\not\in C^1$. In this case, the elliptic operator~$\nabla \cdot \a\nabla$ does not map~$C^2_0(\Rd)$ to~$C_0(\Rd)$, so the parabolic equation~$\partial_t u - \nabla\cdot \a\nabla u = 0$ cannot be interpreted in the strong sense. But this is natural: we know that an elliptic operator~$\nabla \cdot \a\nabla$ naturally maps~$H^1$ to~$H^{-1}$, and that this gives us the proper notion of weak solution. 

\smallskip

It turns out that we should start from the differential operator~$\nabla \cdot\a\nabla$ and use analytic estimates for the PDE to build the corresponding diffusion process. Indeed, we do not need to restrict ourselves by defining diffusion processes only via strong solutions of~\eqref{e.SDE} in the It\^o sense and then computing their generators. We can also work in the other direction: corresponding to the elliptic operator~$\nabla \cdot \a\nabla$ is a parabolic Green function and, if it has the required properties, then we can appeal to Proposition~\ref{p.Kolmogorov} to build our Markov process for us. In order to invoke Proposition~\ref{p.Kolmogorov}, we need to have a Feller semigroup which maps~$C_0(\Rd)$ to~$C_0(\Rd)$. But, in the case of a uniformly elliptic coefficient field~$\a(\cdot)$, this is ensured by the De Giorgi-Nash H\"older estimate! 

\smallskip

We will summarize this observation and provide more details next. 

\begin{example}[Symmetric diffusion processes]
\label{ex.symmetric.diffusion}
Suppose that~$\a(\cdot)$ is a uniformly elliptic coefficient field belonging to~$\Omega(d,\lambda,\Lambda)$, defined in~\eqref{e.Omega.dLambda}. 
Consider the parabolic Green function~$P(t,x,y)$ associated to~$\nabla \cdot \a\nabla$, which is the solution of the initial-value problem
\begin{equation*}
\left\{
\begin{aligned}
& \partial_t P(\cdot,y) - \nabla \cdot \a\nabla P(\cdot,y) = 0 & \mbox{in} & \ (0,\infty) \times \Rd\,, 
\\ & 
P(0,\cdot,y) = \delta_y & \mbox{on} & \ \Rd\,.
\end{aligned}
\right.
\end{equation*}
The Nash-Aronson upper bound estimate  (see~\cite[Lemma E.8]{AKMBook}) asserts the existence of a constant~$C(d,\lambda,\Lambda)<\infty$ such that~$P$ satisfies, for every~$x,y \in\Rd$ and~$t>0$, 
\begin{equation}
\label{e.Nash-Aronson}
P(t,x,y) \leq 
Ct^{-\nicefrac d2} 
\exp\biggl( - \frac{ |x-y|^2 }{8\Lambda t} \biggr)
\,.
\end{equation}
By the De Giorgi-Nash H\"older estimate, we deduce that 
\begin{equation*}
P (t,\cdot,y) \in C^{0,\alpha} (\Rd) \cap C_0(\Rd)\,.
\end{equation*}
Since~$\int_{\Rd} P(t,x,y)\,dy = 1$, we have that 
\begin{equation*}
\| P_t f \|_{L^\infty(\Rd)} 
\leq 
\| f \|_{L^\infty(\Rd)}, \quad \forall f \in C_0(\Rd)
\end{equation*}
and thus the decay of the Green function in~\eqref{e.Nash-Aronson} is easily enough to imply that
\begin{equation*}
\int_{\Rd} f(y) P(t,\cdot,y)\,dy  \in C_0(\Rd), \quad \forall f \in C_0(\Rd)
\,.
\end{equation*}
By the symmetry of the Green function ($P(t,x,y) = P(t,y,x)$) we have that~$P$ is actually continuous jointly on~$(0,\infty) \times\Rd \times \Rd$.
It therefore follows that~$P_t (x,y) = P(t,x,y)\,dy$ is a Feller transition function. Thus, Proposition~\ref{e.finitedimdist} tells us that there must exist a unique (up to indistinguishability) Markov process~$\{ X_t \}$ with transition function~$P(t,\hspace{-1pt} x, \hspace{-1pt} y)dy$ and hence with infinitesimal generator~$\nabla \cdot\a\nabla$. 

\smallskip

The process~$\{X_t\}$ with infinitesimal generator~$\nabla\cdot\a\nabla$, for a coefficient field~$\a$ which may satisfy no regularity beyond belonging to~$L^\infty$, is called a \emph{symmetric diffusion process}.
Note that this process will not necessarily satisfy~\eqref{e.SDE} (with~$\b$ given in~\eqref{e.drift.be.good}) in the usual (classical) strong sense. 
The usual It\^o rules of stochastic calculus do not apply in the classical sense, nor will the process be, in general, a semimartingale. 
We can, however, think of it as some kind of ``weak'' solution of the SDE\footnote{As an alternative to our construction above, one could try to proceed using an approximation argument: regularizing the coefficient field slightly, one can obtain a diffusion process as a strong solution of~\eqref{e.SDE}, and then attempt to pass to a limit in the regularization parameter. This argument can be made to work---allowing one to prove the existence of a Markov process with infinitesimal generator~$\nabla \cdot\a\nabla$, which is a limit of solutions of regularized SDEs---and justifies the assertion that the diffusion process is a generalized solution of~\eqref{e.SDE}. However, the approximation argument also requires the De Giorgi-Nash estimate to pass to the limit! The unavoidability of the De Giorgi-Nash estimate is to be expected because a proof that there is a Markov process with infinitesimal generator~$\nabla \cdot \a\nabla$ implies that a large family of solutions of the parabolic equation are continuous.}  
and generalizations of the It\^o formula exist. In this direction, we refer the reader to the book~\cite{Fuku}, where one can find much more about symmetric diffusion processes.
\end{example}

We conclude this section by demonstrating how precise, quantitative information concerning the \emph{trajectories} of a symmetric diffusion process can be obtained from \emph{analytic} estimates on the solutions of the corresponding parabolic equation. We will show that the classical Nash-Aronson pointwise upper bound estimate for the parabolic Green function implies that, just like for Brownian motion, the trajectories of a symmetric diffusion process with arbitrary uniformly elliptic diffusion matrix~$\a(\cdot)$ are almost surely H\"older continuous, for any H\"older exponent~$\beta<\frac12$. 

\smallskip

For instance, the Nash-Aronson estimate can be found in~\cite[Lemma E.8]{AKMBook}. It says that the parabolic Green function of a general uniformly elliptic (or parabolic) operator is bounded from above and below by a multiple of the heat kernel, suitably dilated in time. 
Let us denote, for every~$\alpha>0$ and~$(t,x) \in (0,\infty) \times \Rd$,  
\begin{equation}
\label{e.heatkernel}
\Gamma_\alpha (t,x) := t^{-\nicefrac d2} \exp\biggl( -\frac{\alpha |x|^2}{t} \biggr)\,.
\end{equation}
The precise statement of the Nash-Aronson estimate is that there exist constants~$0<c(d,\lambda,\Lambda)\leq C(d,\lambda,\Lambda)<\infty$ such that, for every uniformly elliptic coefficient field~$\a \in \Omega$ (with~$\Omega$ defined as in~\eqref{e.Omega.dLambda}),
the parabolic Green function for~$\a$, denoted by~$P$, satisfies, for every~$x,y\in\Rd$ and~$t>0$, 
\begin{equation}
\label{e.Nash.Aronson}
c \Gamma_{C}(t,x-y) 
\leq
P(t,x,y) 
\leq 
C\Gamma_{c} (t,x-y)
\,.
\end{equation}
Here, we require only the upper bound estimate on~$P(t,x,y)$, which says that the probability of the process moving by a lot more than~$\sqrt{\Delta t}$ in a time interval of length~$\Delta t$ is becoming exponentially small as~$\Delta t \to 0$. This allows us to apply union bounds over all dyadic intervals to conclude by Borel-Cantelli that trajectories of the process must be~$C^{0,\beta}$ in time, for any~$\beta < \nicefrac12$. 

\smallskip

In the following proposition, we let~$(\Upsilon,\mathcal{G})$  and~$\{ \mathcal{G}_t \}$ be as in~\eqref{e.Upsilon.canon} and~\eqref{e.Gt.canon}.
Let~$\mathbf{P}$ be a probability measure on~$(\Upsilon,\mathcal{G})$ which makes the coordinate process~$\{ X_t\}$ a symmetric diffusion process generated by~$\nabla \cdot \a\nabla$ for a fixed symmetric, uniformly elliptic coefficient field~$\a\in\Omega$, where~$\Omega$ defined in~\eqref{e.Omega.dLambda}, with uniform ellipticity ratio~$\Lambda\geq1$.

\begin{proposition}[Trajectories are almost~$\nicefrac12$--H\"older]
\label{p.Holder.trajectories}
\hspace{0.1pt} For every exponent~$\beta \in (0,\nicefrac12)$, there exist constants~$C(d,\lambda,\Lambda,\beta)<\infty$ and~$c(d,\lambda,\Lambda,\beta) > 0$ such that, for every~$r\in (0,1]$ and~$\lambda\geq 1$,
\begin{equation}
\label{e.fluct.off}
\mathbf{P} 
\biggl[
\exists s,t \in [0,1]\,, 
\ |s-t|\leq r 
\ \mbox{and} \  
\ 
\bigl| X_{s} - X_{t} \bigr| 
\geq 
\lambda
r^\beta
\biggr]
\leq 
C \exp \bigl( - c\lambda^2 r^{-(1-2\beta)} \bigr)\,.
\end{equation}
Consequently,~$X_t$ has a modification with trajectories in~$C^{0,\beta}$,~$\mathbf{P}$--a.s.
\end{proposition}
\begin{proof}
For a fixed~$m \in \N \cap [n,\infty)$ and~$k \in \{0,\ldots,2^m-1\}$, we have, by~\eqref{e.Xt.to.Pt} and~\eqref{e.Nash.Aronson},
\begin{align*}
\mathbf{P} 
\Bigl[
\bigl| X_{(k+1)2^{-m}} - X_{k2^{-m}} \bigr| 
\geq \lambda 2^{-m\beta}  
\Bigr]
&
=
\mathbf{E}
\biggl[
\indc_{
\left\{
| X_{(k+1)2^{-m}} - X_{k2^{-m}} | 
\geq \lambda 2^{-m\beta}   \right\}
}
\biggr]
\\ & 
=
\mathbf{E}
\biggl[
\mathbf{E}
\biggl[
\indc_{ \left\{
|X_{(k+1)2^{-m}} - X_{k2^{-m}}  |
\geq \lambda 2^{-m\beta}  \right\}
}
\, \bigg\vert\, 
\mathcal{G}_{k2^{-m}} 
\biggr] 
\biggr]
\\ & 
=
\mathbf{E}
\biggl[
\int_{\Rd} 
\indc_{ \left \{ | y - X_{k2^{-m}} | 
\geq \lambda 2^{-m\beta} \right \}}
P(t,X_{k2^{-m}},y)\,dy
\biggr]\\ & 
\leq
C
\int_{\Rd} 
\indc_{ \left \{ |z|\geq \lambda 2^{-m\beta} \right\}}
2^{m\frac d2} 
\exp\bigl( -c 2^m |z|^2 \bigr) \,dz
\\ & 
\leq 
C 2^{md/2} \exp \bigl( - c\lambda^2 2^{m(1-2\beta)} \bigr)
\,.
\end{align*}
A union bound (there are~$2^m$ many elements in~$\{ 0,\ldots,2^m-1\}$) therefore yields
\begin{align*}
&
\mathbf{P} 
\biggl[
\exists m \in \N \cap [n,\infty)\,, \ \exists k \in \{0,\ldots,2^m-1\}\,, 
\  
\bigl| X_{(k+1)2^{-m}} - X_{k2^{-m}} \bigr| 
\geq \lambda2^{-m\beta}  
\biggr]
\notag \\ & \qquad
\leq 
\sum_{m=n}^\infty
C 2^{m( \frac{d+2}{2}) } \exp \bigl( - c\lambda^2 2^{m(1-2\beta)} \bigr)
\leq 
C \exp \bigl( - c \lambda^2 2^{n(1-2\beta)} \bigr)\,,
\end{align*}
which implies~\eqref{e.fluct.off} for all~$r,s \in \{ k2^{-m} \,:\, k,m\in\N, k\leq 2^m\}$.
It follows that the c\`adl\`ag version of the process must have H\"older continuous trajectories with exponent~$\beta$ and satisfy the bound~\eqref{e.fluct.off}.
\end{proof}

\subsection{The invariance principle for symmetric diffusions in random media}
\label{ss.invariance}

A natural interpretation of ``qualitative homogenization'' in the language of probability theory is given in the following theorem. It says that a symmetric diffusion process with generator~$\nabla \cdot\a\nabla$, where~$\a$ is a~$\Zd$--stationary random field, homogenizes with probability one to a diffusion process with generator~$\nabla\cdot \ahom \nabla$, which is a Brownian motion with covariance matrix~$2\ahom$. In other words, the correspondence between symmetric diffusion processes and elliptic partial differential operators ``commutes'' with homogenization.

\begin{theorem}[Invariance principle for symmetric diffusions]
\label{t.invariance}
Let~$\P$ be a~$\Zd$--stationary probability measure on~$(\Omega,\F)$, defined in Section~\ref{ss.random}. Let~$\ahom$ be as in Theorem~\ref{t.qualitative.homogenization} and~$\Omega_0 \in \F$ be the event of full probability on which the limits~\eqref{e.corrector.qualbound.L2} and~\eqref{e.corrector.qualbound.Linfty} are valid. 
Let~$\a\in \Omega_0$ and~$\{ X_t^{\a}\}$ denote the symmetric diffusion process with initial position~$X_0^{\a}=0$ and infinitesimal generator~$\nabla \cdot \a\nabla$, as constructed in Example~\ref{ex.symmetric.diffusion}.  
Define the rescaled process 
\begin{equation*}
X^{\a,\ep}_t := \ep X^{\a}_{ t/{\ep^2} }\,, \quad \ep > 0\,,
\end{equation*}
which has infinitesimal generator~$\nabla \cdot \a\bigl( \tfrac \cdot \ep \bigr) \nabla$. Then~$X^{\a,\ep}_t$ converges in law, as~$\ep \to 0$, to a process~$X^{\ahom}_t$, which is a Brownian motion on~$\Rd$ with covariance matrix~$2\ahom$; that is, the limiting process~$X^{\ahom}_t$ satisfies the stochastic differential equation
\begin{equation*}
dX^{\ahom}_t = \sqrt{2\ahom}\, dB_t\,.
\end{equation*}
\end{theorem}

We will give several proofs of~Theorem~\ref{t.invariance} in this chapter. 
The one we give first, in this section, is a mix of probabilistic and analytic arguments. 

\smallskip

The main probabilistic input we need is the following statement, which can be found in~\cite[Theorem 3.2]{McLeish}. 
It is a generalization of Donsker's invariance principle, which concerns sums of centered i.i.d.~random variables, to general martingales (which may have dependent differences). 
A statement like this is sometimes called a \emph{functional central limit theorem} or \emph{invariance principle} for martingales. 
The adjective \emph{functional} refers to the fact that it asserts convergence in law, not only of a single random variable to a normal random variable but of a stochastic process to a Brownian motion. The convergence is with respect to the Skorokhod topology on so-called \emph{c\`adl\`ag functions} (these are functions of an interval of~$\R$ which are right-continuous and have left limits):  see~\cite[Chapter VI]{JS} for details.

\begin{proposition}[Invariance principle for martingale difference arrays]
\label{p.MFCLT}
Assume that 
\begin{itemize}

\item~$(\Upsilon, \mathcal{G},\mathbf{P})$ is a probability space;

\item~$\{ N_m \}_{m\in\N}$ is a sequence of nonnegative, nondecreasing, integer-valued functions on the interval~$[0,1]$ satisfying~$N_m(t) \to \infty$ as~$m\to \infty$ for every~$t>0$;

\item~$\{ \mathcal{G}_{m,n} \,:\, m\in\N\,,  1 \leq n \leq N_m(1) \}$ is a triangular array of~$\sigma$-fields with~$\mathcal{G}_{m,i} \subseteq \mathcal{G}_{m,j} \subseteq \mathcal{G}$ for every~$i \leq j$;

\item~$\{ Y_{m,n} \,:\, m\in\N\,,  1 \leq n \leq N_m(1) \}$ is a martingale difference array adapted to~$\{ \mathcal{G}_{m,n} \}$; this means that, for every~$m\in\N$ and~$n\in\{1,\ldots,N_m(1) \}$, 
\begin{equation}
\label{e.MDA}
\mbox{$Y_{m,n}$ is~$\mathcal{G}_{m,n}$--measurable} 
\quad \mbox{and} \quad
\mbox{$\mathbf{E} \bigl[ Y_{m,n} \,|\, \mathcal{G}_{m,n-1} \bigr] = 0$}\,.
\end{equation}
\end{itemize}
Assume furthermore that
\begin{equation}
\label{e.invcond2}
\lim_{m\to \infty} 
\mathbf{E} \biggl[ \,
\max_{1\leq n \leq N_m(1)} 
Y_{m,n}^2
\biggr] = 0
\end{equation}
and that there exists~$\sigma^2>0$ such that, for every~$t \in [0,1]$ and~$\delta>0$,
\begin{equation}
\label{e.invcond1}
\lim_{m\to \infty} 
\mathbf{P} \Biggl[ \,
\biggl| \,
\sum_{n=1}^{N_m(t) }
Y_{m,n}^2 
- t\sigma^2
\biggr| > \delta
\Biggr] = 0\,.
\end{equation}
Then the stochastic process~$W_m$ defined by
\begin{equation*}
W_{m} (t) := \sum_{n=1}^{N_m(t)} Y_{m,n} 
\end{equation*}
converges in law, as~$m\to \infty$, to a Brownian motion on~$[0,1]$ with variance~$\sigma^2$\,.
\end{proposition}

Proposition~\ref{p.MFCLT} is a statement about a one-dimensional stochastic process converging to one-dimensional Brownian motion, but we immediately obtain an analogous result in arbitrary dimensions by applying the \emph{Cram\'er-Wold device}. This basic result in probability theory (see~\cite[Theorem 29.4]{Bill}) says that an~$\Rd$--valued stochastic process converges in law to zero if and only if each of its one-dimensional projections converges in law to zero.

\smallskip

As we will see once we have set up the argument, checking the assumptions of Proposition~\ref{p.MFCLT} in our application will boil down to the estimates~\eqref{e.corrector.qualbound.L2} and~\eqref{e.corrector.qualbound.Linfty} on the first-order correctors.  This is the be expected since the proof of our main homogenization result (Theorem~\ref{t.qualitative.homogenization}) also boiled down to the same estimates, and Theorem~\ref{t.invariance} is essentially the same result in a different language. 

\smallskip

The rest of this section will be focused on the details of the proof of Theorem~\ref{t.invariance}. Throughout, we freeze a coefficient field~$\a \in \Omega_0$, where~$\Omega_0 \in\F$ is the event of full~$\P$-probability described in the statement of the theorem. 
We let~$(\Upsilon,\mathcal{G})$  and~$\{ \mathcal{G}_t \}$ be as in~\eqref{e.Upsilon.canon} and~\eqref{e.Gt.canon}.
We consider the unique probability measure~$\mathbf{P}$ on~$(\Upsilon,\mathcal{G})$ which makes the coordinate process~$\{ X_t\}$ a symmetric diffusion process generated by~$\nabla \cdot \a\nabla$, with~$X_0=0$ (note that, compared with the statement of Theorem~\ref{t.invariance}, we are dropping~$\a$ from the notation for~$X_t$). The expectation with respect to~$\mathbf{P}$ is denoted by~$\mathbf{E}$.
We let~$P(t,x,y)$ denote the parabolic Green function for the operator~$\nabla \cdot \a\nabla$, which is also the transition function for the diffusion process~$\{ X_t\}$.
For each~$\ep>0$, we denote the process~$X_t^\ep$ by
\begin{equation*}
X_t^\ep : = \ep  X_{  t/{\ep^2} }\,.
\end{equation*}
Note that~$X_t^\ep$ has infinitesimal generator~$\nabla \cdot \a(\frac\cdot\ep)\nabla$ and transition function~$P^\ep$ given by
\begin{equation}
\label{e.Pep}
P^\ep(t,x,y) 
:=
P \Bigl ( \frac{t}{ \ep^2} ,  \frac{ x} { \ep} ,  \frac{y} {\ep} \Bigr)
\end{equation}
The process~$X_t^\ep$ is~$\mathcal{G}_{t}^\ep$-adapted, where~$\mathcal{G}_{t}^\ep: = \mathcal{G}_{t/\ep^2}$. 

\smallskip

\smallskip

In order to apply Proposition~\ref{p.MFCLT} to prove Theorem~\ref{t.invariance}, we need to identify an appropriate martingale or family of martingales.
Note that we cannot just use the symmetric diffusion process~$X_t$ itself because the presence of a nonzero drift~$\b$ in~\eqref{e.drift.be.good} ensures that it is not a martingale! 

\smallskip

However, as we have already observed in the previous section, the process~$t\mapsto u (X_t)$ is a martingale whenever~$u$ is a solution of~$-\nabla \cdot \a\nabla u = 0$. We will apply this fact to one of the corrected affine functions, that is, for fixed~$e \in\Rd$, 
\begin{equation*}
x\mapsto e \cdot x + \phi_{e}(x) - \phi_{e}(0)
\,.
\end{equation*}
Note that this expression is well-defined, since~$\phi_e$ is locally H\"older continuous and hence defined pointwise by the De Giorgi-Nash estimate.\footnote{We are anyway going to have to use many estimates in this chapter which hold only in the scalar case, so we might as well not be shy about it.} In fact, by adding a constant to~$\phi_e$, we may suppose without loss of generality that~$\phi_e(0)=0$ in the rest of this section. 

\smallskip

We also select, for each~$\ep \in (0,1]$, a mesoscopic time scale~$\tau_\ep >0$ to be chosen below (it will satisfy~$\ep^2 \ll \tau_\ep \ll 1$), 
a direction~$e \in\Rd$, and define
\begin{equation*}
M^\ep_n := e \cdot X^{\ep}_{\tau_\ep n} + \phi^\ep_{e}(X^{\ep}_{\tau_\ep n}) \,,
\end{equation*}
where we recall that~$\phi^\ep := \ep \phi(\frac \cdot \ep)$. 
Observe that we can write~$M^\ep_n$ as the sum 
\begin{equation*}
M^\ep_n = \sum_{j=1}^n Y^\ep_j, 
\end{equation*}
of martingale differences~$Y^\ep_n$, defined by
\begin{equation*}
Y^\ep_n := M^\ep_n - M^\ep_{n-1}
=
e\cdot \bigl( X^{\ep}_{{\tau_\ep}n} - X^{\ep}_{{\tau_\ep}(n-1)} \bigr)  
+ 
\phi^\ep_{e}(X^{\ep}_{{\tau_\ep}n}) - \phi^\ep_{e}(X^{\ep}_{{\tau_\ep}(n-1)})
\,.
\end{equation*}
Our strategy to prove Theorem~\ref{t.invariance} is to apply Proposition~\ref{p.MFCLT} to~$Y^{\ep}_n$ to obtain the convergence in law, as~$\ep\to0$, of the corresponding process~$W^\ep_t$ defined by
\begin{equation*}
W^\ep_t
:=
\sum_{n=1}^{\lceil t/\tau_\ep \rceil}
Y^\ep_n
\end{equation*}
to a one-dimensional Brownian motion with variance~$\sigma^2 = 2 e_k \cdot \ahom e_k$. 
Separately, we will check via a straightforward computation, using the sublinearity of the correctors in~\eqref{e.corrector.qualbound.Linfty}, that the process~$W^\ep_t$ is close, with very high probability, to~$e\cdot X^{\ep}_t$, the original process projected onto~$e_k$, for small~$\ep$. This will enable us to deduce the convergence in law of the projected process~$e \cdot X^{\ep}_t$ to a Brownian motion with variance~$2e\cdot \ahom e$ and hence, by the Cram\'er-Wold device, of the full diffusion process~$X^{\ep}_t$ to a~$d$-dimensional Brownian motion with covariance matrix~$2\ahom$. 

\smallskip

It remains to check the validity of the hypothesis of Proposition~\ref{p.MFCLT} and prove an estimate on the difference~$W^\ep_t - e\cdot X^{\ep}_t$. 
We proceed by reducing each of these to analytic estimates and the limits in~\eqref{e.corrector.qualbound.L2} and~\eqref{e.corrector.qualbound.Linfty}.

\smallskip

Luckily, some of our work is actually already done since Proposition~\ref{p.Holder.trajectories} implies~\eqref{e.invcond2}. Indeed, by~\eqref{e.fluct.off} we have that, for any~$\lambda\geq 1$ and~$\beta \in (0,\nicefrac12)$, 
\begin{align*}
\mathbf{P}
\biggl[ 
\max_{ 1 \leq n \leq \lceil t/\tau_\ep\rceil } 
|X^\ep_{n\tau_\ep} - X^\ep_{(n+1)\tau_\ep } |^2
\geq 
\lambda
\tau_\ep^{2\beta}
\biggr] 
&
\leq
C\exp \bigl( - c\lambda \tau_\ep^{1-2\beta} \bigr) 
\,.
\end{align*}
Recalling that~$X^\ep_0 = 0$, the bound~\eqref{e.fluct.off} with~$r=1$ also implies,  that, for every~$\lambda \geq 1$,
\begin{equation}
\label{e.u.stuck.bro}
\mathbf{P}\biggl[ \max_{ 0 \leq t \leq 1} |X^\ep_{t} | \geq \lambda \biggr]
\leq 
C \exp \bigl(-c\lambda^2\bigr)\,.
\end{equation}
It thus follows from the triangle inequality that, for every~$\mu,\lambda\geq 1$ and~$\beta \in(0,\nicefrac12)$,  
\begin{equation}
\label{e.max.Yepn.bonk}
\mathbf{P} 
\biggl[ 
\max_{ 1 \leq n \leq \lceil t/\tau_\ep\rceil } 
(Y^\ep_n)^2
\geq \mu \tau_\ep^{2\beta} + 8\| \phi^\ep_e\|_{L^\infty(B_\lambda)}^2
\biggr]
\leq 
C \exp \bigl( -c\mu \tau_\ep^{1-2\beta} \bigr) 
+
C \exp \bigl( -c\lambda^2  \bigr)
\,.
\end{equation}
In view of the limit
\begin{equation}
\label{e.phieep.limit.Linfty}
\lim_{\ep \to 0} \| \phi_e^\ep\|_{L^\infty(B_\lambda)}
= 0, 
\end{equation}
which holds for every fixed~$\lambda \geq 1$ by the assumed validity of the limit~\eqref{e.corrector.qualbound.Linfty} and the fact that~$\phi_e(0)=0$, 
we deduce that 
\begin{equation}
\label{e.invcond2.gotya}
\limsup_{\ep \to 0} 
\mathbf{E} 
\biggl[ 
\max_{ 1 \leq n \leq \lceil t/\tau_\ep\rceil } 
(Y^\ep_n)^2
\biggr]
= 0 
\,,
\end{equation}
under the (mere) constraint on the choice of the time scale~$\tau_\ep$ that
\begin{equation}
\label{e.tau.ep.constraint}
\limsup_{\ep \to 0} \tau_\ep = 0. 
\end{equation}
The inequality~\eqref{e.invcond2.gotya} is the assumption~\eqref{e.invcond2} in our context.  

\smallskip

The estimate~\eqref{e.fluct.off} also allows us to quickly estimate the difference~$W^\ep_t - e\cdot X^{\ep}_t$. 
Observe that, for every~$n\in \N\cap [1, \tau_\ep^{-1}]$ and~$t \in ( (n-1) \tau_\ep, n\tau_\ep]$, we have that 
\begin{equation*}
W^\ep_t - e\cdot X^{\ep}_t
=
e\cdot \bigl( X^\ep_{n\tau_\ep} - X^\ep_t \bigr) 
+
\phi_e^\ep(X^\ep_{n\tau_\ep })
\end{equation*}
and, therefore, using the same estimates as above, we find that 
\begin{equation*}
\mathbf{P} 
\biggl[ 
\sup_{ t \in [0,1] } 
\bigl| W^\ep_t - e\cdot X^{\ep}_t \bigr| 
\geq \tau_\ep^{\beta} + \| \phi^\ep_e\|_{L^\infty(B_\lambda)}
\biggr]
\leq 
C \exp \bigl( -c\tau_\ep^{1-2\beta} \bigr) 
+
C \exp \bigl( -c\lambda^2  \bigr)
\,.
\end{equation*}
Using once more the limits~\eqref{e.phieep.limit.Linfty} and~\eqref{e.tau.ep.constraint}, we obtain
\begin{equation}
\label{e.Wept.to.Xept}
\sup_{\delta>0} \limsup_{\ep \to 0}
\mathbf{P} 
\biggl[ 
\sup_{ t \in [0,1] } 
\bigl| W^\ep_t - e\cdot X^{\ep}_t \bigr| 
\geq \delta 
\biggr] = 0
\,.
\end{equation}
In other words,~$W^\ep_t - e\cdot X^\ep_t$ converges in law to the process which is identically zero, which, of course, is sufficient to ensure that any limiting stochastic process for~$W^\ep_t$, as~$\ep>0$, will also be a limit for~$X^\ep_t$. 

\smallskip

We have left to obtain the condition~\eqref{e.invcond1} for~$\sigma^2 = 2e \cdot \ahom e$. 
We apply the identity~\eqref{e.Xt.to.Pt}, which yields, for any~$\ep>0$ and~$0\leq s<t$, 
\begin{align*}
\lefteqn{ 
\mathbf{E} 
\biggl[ 
\Bigl(
e \cdot \bigl( X^{\ep}_{t} - X^{\ep}_{s} \bigr)  
+ 
\phi^\ep_{e}\bigl(X^{\ep}_{t}\bigr) - \phi^\ep_{e}\bigl(X^{\ep}_{s}\bigr)
\Bigr)^2
\, \Bigl|\, \mathcal{G}^\ep_s
\biggr]
} \qquad & 
\\ & 
=
\ep^2 
\mathbf{E} 
\biggl[ 
\Bigl(
e \cdot \bigl( X_{t/\ep^2} - X_{s/\ep^2} \bigr)  
+ 
\phi_{e}\bigl(X_{t/\ep^2}\bigr) - \phi_{e}\bigl(X_{s/\ep^2}\bigr)
\Bigr)^2
\, \Bigl|\, \mathcal{G}_{s/\ep^2}
\biggr]
\\ &  
=
\ep^2 \int_{\Rd} 
\Bigl( e\cdot \bigl(y- X_{s/\ep^2}\bigr) + \phi_e(y) - \phi_e\bigl(X_{s/\ep^2}\bigr) \Bigr)^2
P( (t-s)/\ep^2 , X_{s/\ep^2} ,y)\,dy 
\,.
\end{align*}
Therefore, 
\begin{align}
\label{e.condE.to.analysis.yes}
\lefteqn{
\mathbf{E} \Bigl[ (Y_{n}^\ep )^2 \, \big\vert \, \mathcal{G}_{(n-1)\tau_\ep/\ep^2} \Bigr] 
} \qquad & 
\notag \\ & 
=
\ep^2 
\! \int_{\Rd} \!
\Bigl( e\cdot (y- x) + \phi_e(y) - \phi_e(x) \Bigr)^2
P( \tau_\ep/\ep^2, x, y)\,dy \,
\biggr\vert_{x=X_{(n-1)\tau_\ep/\ep^2} }
\,.
\end{align}
The right side of~\eqref{e.condE.to.analysis.yes} can be studied by purely analytic means. In fact, in the following lemma we present an explicit identity for the quantity on the right side, which makes it clear that it is equal to~$\tau_\ep (2e\cdot \ahom e)$, plus some error terms which we can show are~$o(\tau_\ep)$. 
Since this is where the homogenized matrix~$\ahom$ makes its appearance, it is perhaps the most interesting step in the proof of Theorem~\ref{t.invariance}.

\begin{lemma}
\label{l.varianceX}
For every~$e, x \in\Rd$ and~$t >0$, 
\begin{align}
\label{e.varianceX}
\lefteqn{
\int_{\Rd} 
\bigl(e \cdot (y-x) {+}\phi_{e}(y) - \phi_e(x) \bigr)^2 P(t,x,y)\,dy 
-
(2 e\cdot \ahom e) t
} 
\ \ 
& 
\notag \\ &
=
2 \int_0^{t}   
\int_{\Rd}   
\Bigl( ( \phi_e(y) - \phi_e(x) )
\a(y)(e+\nabla \phi_{e}(y)) 
+ 
\bfs_e (y) e\Bigr)\cdot
\nabla P(s ,x,y)\,dy 
\,ds\,.
\end{align}
\end{lemma}
\begin{proof}
We proceed by computing the time derivative of the quantity on the left side of~\eqref{e.varianceX}, using the equations for both~$(t,y) \mapsto P(t,x,y)$ and the corrector~$\phi_e(\cdot)$:
\begin{align*}
\lefteqn{
\partial_t \int_{\Rd} 
(e \cdot (y{-}x)+\phi_{e}(y) {-} \phi_e(x))^2 
P(t,x,y)\,dx
}  & 
\\ &
=
\int_{\Rd} 
\bigl(e \cdot (y{-}x) {+}\phi_{e}(y) {-} \phi_e(x) \bigr)^2 
\partial_t P(t,x,y)\,dy
\\ & 
=
\int_{\Rd} 
2 (e \cdot (y{-}x)+\phi_{e}(y) {-} \phi_e(x) )
(e {+} \nabla \phi_e(y)) 
\cdot
\a(y)
\nabla P(t,x,y)\,dy
\\ & 
=
2e \cdot \!
\int_{\Rd} \!\!
\a (y)(e {+} \nabla \phi_e(y)) P(t,x,y)\,dy
-
2 \!\int_{\Rd} \!\!
(\phi_e(x) {-} \phi_e(y))
\a(y) (e {+} \nabla \phi_{e}(y)) 
\cdot 
\nabla P(t,x,y) \,dy\,.
\end{align*}
We introduce the flux correctors into the computation by way of the identity
\begin{align*}
\int_{\Rd} 
\bigl( 
\a (y) (e {+} \nabla \phi_{e}(y)) -\ahom e \bigr)
P(t,x,y)\,dy
=
-\int_{\Rd} 
\bfs_{e} (y) \nabla P(t,x,y)\,dy.
\end{align*}
Combining these, we obtain
\begin{align*}
\lefteqn{
\partial_t \int_{\Rd} 
\bigl(e \cdot (y{-} x) {+} \phi_{e}(y) {-} \phi_e(x) \bigr)^2 
P(t,x,y)\,dy
} \qquad & 
\\ & 
=
2e\cdot \ahom e 
+
2 \int_{\Rd} 
\bigl( ( \phi_e(y) {-} \phi_e(x) )
\a (y)(e {+} \nabla \phi_{e}(y)) 
+\bfs_e (y) e\bigr)
\cdot
\nabla P(t,x,y)\,dy
\,.
\end{align*}
After integration in~$t$, we obtain~\eqref{e.varianceX}.
\end{proof}

Let's estimate the term on the right side of~\eqref{e.varianceX} to see that it is really a small error. 
This requires a few deterministic estimates on the parabolic Green function in addition to the upper bound Nash-Aronson estimate in~\eqref{e.Nash.Aronson}.
The parabolic Caccioppoli inequality (see~\cite[Lemma 8.1]{AKMBook}) and~\eqref{e.Nash.Aronson} imply that, for every~$x,z\in\Rd$, 
\begin{align*}
\left\| \nabla P(t,x,\cdot) \right\|_{\underline{L}^2((\frac t2,t) \times B_{\sqrt{t}}(z))}
&
\leq 
Ct^{-\nicefrac12} 
\left\| P(t,x,\cdot) \right\|_{L^\infty((\frac t4,t)\times B_{\sqrt{2t}}(z))}
\leq
C t^{-\nicefrac12}  \Gamma_c (t,x-z) 
\,.
\end{align*}
Since the coefficients~$\a(\cdot)$ are time-independent, there is a free upgrade of the left side from~$L^2_tL^2_x$ to~$L^\infty_tL^2_x$ (see~\cite[Lemma 8.2]{AKMBook}), so that we obtain, for every~$x,z\in\Rd$ and~$t>0$, 
\begin{equation}
\label{e.quenchednablaP}
\left\| \nabla P(t,x,\cdot) \right\|_{\underline{L}^2(B_{\sqrt{t}}(z))}
\leq
C t^{-\nicefrac12}  \Gamma_c (t,x-z) 
\,.
\end{equation}
We can concisely summarize this as 
\begin{equation}
\label{e.concise.nablaGF}
\int_{\Rd} 
\frac{\bigl| \nabla P(t,x,y) \bigr|^2}
{ \Gamma_c(t,x-y) }
\, dy
\leq Ct^{-1} 
\,.
\end{equation}
Using De-Giorgi-Nash theory, the estimate~\eqref{e.quenchednablaP} can be further localized, and we obtain, for every~$r \leq \frac12 \sqrt{t}$, 
\begin{equation}
\label{e.quenchednablaP.improved}
\left\| \nabla P(t,x,\cdot) \right\|_{\underline{L}^2(B_{r}(z))}
\leq
\frac{C}{r} \osc_{(s,y)\in (t-r^2,t]\times B_{2r}(z)} P(s,x,y) 
\leq C r^{2\alpha -1} t^{-\alpha} \Gamma_c (t,x-z) 
\,.
\end{equation}
(The estimate~\eqref{e.quenchednablaP.improved} is not needed until the proof of Theorem~\ref{t.parabolic.GF} in the next section.)

\smallskip

Commencing with the estimate of the right side of~\eqref{e.varianceX}, we use the Cauchy-Schwarz inequality and~\eqref{e.concise.nablaGF} to get
\begin{align*}
\lefteqn{ 
\biggl| \int_{\Rd}   
\Bigl( ( \phi_e(y) - \phi_e(x) )
\a(y)(e+\nabla \phi_{e}(y)) 
+ 
\bfs_e (y) e\Bigr)\cdot
\nabla P(s ,x,y)\,dy
\biggr| 
}
\qquad & 
\\ & 
\leq
C\int_{\Rd}   
\Bigl( \bigl|  \phi_e(y)  {-}  \phi_e(x)  \bigr| 
\bigl( 1 {+} |\nabla \phi_{e}(y)| \bigr)  
+ 
\bigl| \bfs_e (y) {-} (\bfs_e) \bigr| \Bigr)
\bigl| \nabla P(s ,x,y) \bigr| \,dy
\\ & 
\leq 
Cs^{-\nicefrac12} 
\biggl( \int_{\Rd} \!  
\Bigl(  |  \phi_e(y) {-} \phi_e(x)   |^2 
\bigl( 1 {+} |\nabla \phi_{e}(y)|^2 \bigr)  
+ 
 | \bfs_e (y) {-} (\bfs_e)  |^2 \Bigr) \Gamma_c (s,x-y)\,dy 
\! \biggr)^{\!\!\nicefrac12} 
\notag \\ & 
\leq
C 
\sup_{r \geq {s}^{\nicefrac12}} 
\frac1{r} 
\Bigl( 
\left\| \phi_e \right\|_{L^\infty(B_{r}(x))}^2
\Bigl( 1 {+} 
\left\| \nabla \phi_e \right\|_{\underline{L}^2(B_{r}(x))}^2
\Bigr) 
+ \left\| \bfs_e {-} (\bfs_e)_{B_r(x)} \right\|_{\underline{L}^2(B_{r}(x))}^2
\Bigr)^{\nicefrac12} 
\notag \\ & 
\leq
C 
\sup_{r \geq {s}^{\nicefrac12}} 
\frac1{r} 
\Bigl( 
\left\| \phi_e \right\|_{L^\infty(B_{r}(x))}
\Bigl( 1 {+} 
\frac1{r} 
\left\| \phi_e \right\|_{\underline{L}^2(B_{r}(x))}
\Bigr) 
+ \left\| \bfs_e {-} (\bfs_e)_{B_r(x)} \right\|_{\underline{L}^2(B_{r}(x))}
\Bigr)
\,.
\end{align*}
In the second and third line above,~$(\bfs_e)$ can be any constant, for instance, the mean of~$\bfs_e$ with respect to~$\Gamma_c(t,x-y)\,dy$. 
The Caccioppoli inequality was used to get the last line.
Circling back to~\eqref{e.condE.to.analysis.yes} and~\eqref{e.varianceX}, we now obtain, for every~$x \in \Rd$,
\begin{align}
\label{e.varianceX.est}
\lefteqn{
\biggl| 
\int_{\Rd} 
\bigl(e \cdot (y-x) {+}\phi_{e}(y) - \phi_e(x) \bigr)^2 P(t,x,y)\,dy 
-
(2 e\cdot \ahom e) t
\biggr| 
} \quad & 
\notag \\ &
\leq
C\int_0^{t}   
\sup_{r \geq {s}^{\nicefrac12}} 
\frac1{r} 
\Bigl( 
\left\| \phi_e \right\|_{L^\infty(B_{r}(x))}
\Bigl( 1 {+} 
\frac1{r} 
\left\| \phi_e \right\|_{\underline{L}^2(B_{r}(x))}
\Bigr) 
+ \left\| \bfs_e {-} (\bfs_e)_{B_r(x)} \right\|_{\underline{L}^2(B_{r}(x))}
\Bigr)
\, ds 
\notag \\ & 
= : F(t,x)
\end{align}
and hence that
\begin{equation}
\label{e.condE.to.analysis.0}
\biggl| 
\mathbf{E} \Bigl[ (Y_{n}^\ep )^2 \, \big\vert \, \mathcal{G}_{(n-1)\tau_\ep/\ep^2} \Bigr] 
- 2 (e\cdot \a e) \tau_\ep
\biggr|
\leq
C\ep^2 F\Bigl( \frac{\tau_\ep} {\ep^2} , X_{(n-1)\tau_\ep/\ep^2}\Bigr)
\,.
\end{equation}
The assumed validity of~\eqref{e.corrector.qualbound.L2} and~\eqref{e.corrector.qualbound.Linfty} imply that~$t^{-1} F(t,x) \to 0$ as~$t\to \infty$. In fact, it is not difficult to obtain some uniformity in the convergence in~$x$ since, by a simple covering argument, we have that, for every~$\lambda\geq 1$, 
\begin{equation*}
\sup_{|x| \leq \lambda \sqrt{t} } F(t,x) 
\leq 
C\lambda^{\nicefrac d2} F\bigl( c\lambda^{-2}t,0 \bigr)
\,.
\end{equation*}
Consequently,
\begin{equation*}
\limsup_{ t \to \infty}
\frac 1t 
\sup_{|x| \leq \lambda \sqrt{t} }
F (t, x)
= 0
\,.
\end{equation*}
It follows that we can choose the sequence~$\tau_\ep$ so that it converges to zero sufficiently slowly as~$\ep \to 0$, to obtain that, for every~$\lambda\geq 1$, 
\begin{equation}
\label{e.F.uniform.trap.tauep}
\limsup_{ \ep \to 0}
\frac{\ep^2}{\tau_\ep}
\sup_{ |x| \leq \lambda/\ep}
F \Bigl( \frac {\tau_\ep}{\ep^2}, x \Bigr)
= 0
\,.
\end{equation}
To finish off the proof of Theorem~\ref{t.invariance}, we use~\eqref{e.condE.to.analysis.0} to obtain, for every~$t \in (0,1]$ and~$\lambda\geq 1$,
\begin{align}
\label{e.Yepn2.main}
\lefteqn{
\Biggl|
\,
\mathbf{E} 
\Biggl[ 
\sum_{n=1}^{\lceil t/\tau_\ep \rceil} 
(Y^\ep_n)^2\indc_{\{ \max_{ t\in[0,1]} |X^\ep_{t} | \leq \lambda\}}
\Biggr]
-
\tau_\ep \biggl\lceil \frac{t}{\tau_\ep} \biggr\rceil  \cdot 2 e\cdot \ahom e
\,
\Biggr|
} 
\qquad &
\notag \\ &
\leq
\sum_{n=1}^{\lceil t/\tau_\ep \rceil}
\mathbf{E} \biggl[
\Bigl| 
\mathbf{E} \Bigl[ (Y_{n}^\ep )^2 \, \big\vert \, \mathcal{G}_{(n-1)\tau_\ep/\ep^2} \Bigr] 
- 2 (e\cdot \ahom e) \tau_\ep
\Bigr|
\, \indc_{\{ \max_{  t\in[0,1]} |X^\ep_{t} | \leq \lambda \}} 
\biggr]
\notag \\ & \qquad 
+
C|e|^2 \mathbf{P} \Bigl[ \max_{ 0 \leq t \leq 1} |X^\ep_{t} | \geq \lambda \Bigr]
\notag \\ & 
\leq
\ep^2 \biggl\lceil \frac{t}{\tau_\ep} \biggr\rceil 
\sup_{ x \in B_{\lambda /\ep}}
F \Bigl( \frac {\tau_\ep}{\ep^2}, x \Bigr)
+
C|e|^2 \exp \bigl(-c\lambda^2\bigr)
\,.
\end{align}
In view of~\eqref{e.u.stuck.bro} and~\eqref{e.F.uniform.trap.tauep}, 
this implies 
\begin{equation*}
\sup_{\delta>0}
\limsup_{\ep \to 0} 
\mathbf{P} 
\Biggl[ 
\biggl| 
\sum_{n=1}^{\lceil t/\tau_\ep \rceil} 
(Y^\ep_n)^2 
-
t (2 e\cdot \ahom e) 
\,
\biggr|
\geq \delta 
\Biggr]
= 0\,.
\end{equation*}
This completes the verification of~\eqref{e.invcond1}. 
We may now apply Proposition~\ref{p.MFCLT} to obtain the convergence as~$\ep\to0$ of the process~$W^\ep_t$, and hence, by~\eqref{e.Wept.to.Xept}, also of~$e\cdot X^\ep_t$, to Brownian motion with covariance~$2e\cdot \ahom e$. Applying the Cram\'er-Wold device completes the proof of Theorem~\ref{t.invariance}.

\subsection{Homogenization of the parabolic Green function}
\label{ss.Green}

In this section, we give a second proof that qualitative homogenization implies the quenched invariance principle, which is independent of the previous section. This argument is more concrete than the one based on the martingale CLT. It is also more analytic: instead of using an ``off the shelf'' central limit theorem from probability theory (such as Proposition~\ref{p.MFCLT}), we demonstrate the convergence of the transition functions (the parabolic Green functions) to their homogenized limit, which happens to be a Gaussian (shifted by~$\ahom$). This implies the invariance principle for the process, namely the statement of Theorem~\ref{t.invariance}, by the uniqueness part of Proposition~\ref{p.Kolmogorov}, since there is only one possible subsequential limit.\footnote{The compactness here comes from the fact that the trajectories are almost surely H\"older continuous by Proposition~\ref{p.Holder.trajectories}, and so~$\P$ can be seen as a probability measure on the space of continuous sample paths.}

\smallskip

This argument in this section does not refer to the symmetric diffusion process at all. Of course, it comes as no surprise that the analytic estimates we must prove bear a striking resemblance to the one for the right side of~\eqref{e.varianceX}. However, because the argument does not pass through an abstract martingale CLT, it is less mysterious, easier to quantify, and stripped down to its essence. 

\smallskip

The parabolic Green function for the homogenized operator~$\nabla \cdot\ahom \nabla$ is denoted by 
\begin{equation*}
\overline{P} (t,x) 
:= 
( \det \ahom )^{-\nicefrac12} 
(4\pi t)^{-\nicefrac d2} 
\exp \biggl( 
-\frac{x\cdot \ahom^{-1} x }{4t} 
\biggr)
\,.
\end{equation*}
The main result of this section is the following theorem. 

\begin{theorem}[Homogenization of the parabolic Green function] \hspace{-2pt} 
\label{t.parabolic.GF}
Let~$(\Omega,\F)$ be as in Section~\ref{ss.random} and~$\P$ be a~$\Zd$--stationary probability measure on~$(\Omega,\F)$. Let~$\ahom$ be as in Theorem~\ref{t.qualitative.homogenization} and~$\Omega_0 \in \F$ be the event of full probability on which the limits~\eqref{e.corrector.qualbound.L2} and~\eqref{e.corrector.qualbound.Linfty} are valid. 
Let~$\a\in \Omega_0$ and~$P(t,x,y)$ be the parabolic Green function for the operator~$\nabla \cdot \a\nabla$.
Then there exist constants~$\alpha(d,\lambda,\Lambda) \in (0,\frac12]$ and~$C(d,\lambda,\Lambda) <\infty$ such that, for every~$s,t\geq 1$ with~$s \in (0,\tfrac12t]$ and~$x,y\in\Rd$, 
\begin{equation}
\label{e.P.to.Pbar}
\bigl| P(t,x,y) - \overline{P}(t,x-y) \bigr| 
\leq
C\biggl( 
\Bigl( \frac{s}{t} \Bigr)^{\! \alpha}
+
E_y\bigl(s^{\nicefrac12}\bigr)  \Bigl( \frac{t}{s} \Bigr)^{\!\nicefrac d4} \!\log\Bigl( \frac{t}{s} \Bigr) 
\biggr)
t^{-\nicefrac d2} \exp \biggl( -\frac{|x-y|^2}{Ct} \biggr)
\,,
\end{equation}
where~$E_y(s)$ is the random variable defined by
\begin{equation} \label{e.Es.def}
E_y(s): =
\sup_{ r \in [s,\infty)} 
\frac1r
\Bigl( 
\left\| \phi_e - (\phi_e)_{B_r(y)} \right\|_{\underline{L}^2(B_{r}(y))}
+
\left\| \bfs_e - (\bfs_e)_{B_r(y)} \right\|_{\underline{L}^2(B_{r}(y))}
\Bigr) 
\,.
\end{equation}
\end{theorem}

Before giving the proof, a few remarks are in order. The random variable~$E_0(s)$ is obviously a modulus quantifying the convergence of the limit~\eqref{e.corrector.qualbound.L2}. Thus~$E_y(s) \to 0$ as~$s \to \infty$, for every~$\a \in\Omega_0$. Since the parameter~$s$ is free, optimizing the left side guarantees that the prefactor in parentheses on the right side of~\eqref{e.P.to.Pbar} also vanishes as~$t \to \infty$:
\begin{equation*}
\limsup_{t\to \infty} 
\inf_{s \in [1, \frac12 t] } 
\biggl ( \!
\Bigl(  \frac{s}{t} \Bigr)^{\!\nicefrac12}
+
C  \log \frac{t}{s} 
E_y\bigl(s^{\nicefrac12}\bigr) \!
\biggr )
=
0.
\end{equation*}
Therefore~\eqref{e.P.to.Pbar} says that the difference between~$P(t,x,y)$ and~$\overline{P}(t,x-y)$ is bounded by a term which is~$o(1)$ as~$t \to\infty$ multiplied by~$t^{-\nicefrac d2} \exp( - |x-y|^2/Ct) = C\Gamma_c(t,x-y)$, which we know by the Nash-Aronson estimate~\eqref{e.Nash.Aronson} is roughly the same size as both of~$P(t,x,y)$ and~$\overline{P}(t,x-y)$.
Another way to state this result is in terms of the Green function~$P^\ep(t,x,y)$ for the rescaled operator~$\nabla\cdot \a(\tfrac\cdot\ep)\nabla$, which satisfies~$P^\ep (t,x,y) = P( t/\ep^2, x/\ep, y/\ep )$:
\begin{equation}
\label{e.Pep.to.P}
\bigl| P^\ep(t,x,y) - \overline{P}(t,x-y) \bigr| 
\leq
C
\inf_{s \in [\ep^2, \frac12t]} 
\biggl ( \!
\Bigl(  \frac{s}{t} \Bigr)^{\!\nicefrac12}
+
\log \frac{t}{s} 
E\bigl( \ep^{-1} s^{\nicefrac12}\bigr) \!
\biggr )
\Gamma_c(t,x-y)
\,.
\end{equation}
The right side converges to zero locally uniformly in~$(0,\infty)\times \Rd \times \Rd$ as~$\ep \to 0$, and therefore, we deduce the local uniform convergence of~$P^\ep$ to~$P$ as~$\ep \to 0$.

\begin{corollary}
\label{c.Pep.to.P}
Under the assumptions of Theorem~\ref{t.parabolic.GF}, if we let~$P^\ep$ denote the rescaled parabolic Green function given in~\eqref{e.Pep}, then, for any compact subset~$K \subseteq (0,\infty) \times \Rd \times \Rd$, 
\begin{equation}
\label{e.Pep.to.P.dis}
\limsup_{\ep \to 0}
\sup_{(t,x,y) \in K}  
\bigl| P^\ep (t,x,y) - \overline{P}(t,x-y) \bigr| = 0\,.
\end{equation}
\end{corollary}

Corollary~\ref{c.Pep.to.P} is actually equivalent to  Theorem~\ref{t.invariance}. Indeed,~\eqref{e.Pep.to.P.dis} says that the transition function for the rescaled process~$X^{\a,\ep}_t$ converges to the transition function for the limiting process~$X^{\ahom}_t$ which, by the uniqueness part of Proposition~\ref{p.Kolmogorov}, is the only possible weak subsequential limit of the law of~$X^{\a,\ep}_t$. 
Conversely, we can see that Theorem~\ref{t.invariance} implies~\eqref{e.Pep.to.P.dis} by using the formula~\eqref{e.formula} applied to suitable approximations of the parabolic Green function (using the Nash-Aronson estimate to control the approximation error). 

\smallskip

The reason we have stated the estimate~\eqref{e.P.to.Pbar} with an explicit modulus, instead of presenting Corollary~\ref{c.Pep.to.P} as the main result is to demonstrate that our argument here is \emph{quantitative}. Suppose we wish to obtain an estimate for the convergence rate of~$P^\ep$ to~$\overline{P}$, then~\eqref{e.Pep.to.P} makes it very clear what we should do: estimate the modulus~$E(s)$. Of course, this is precisely what we will do in later sections. 

\smallskip

We now present the proof of the theorem. 

\begin{proof}[{Proof of Theorem~\ref{t.parabolic.GF}}]
The reader should be aware that a similar derived estimate along very similar lines appears in~\cite[Theorem 8.20]{AKMBook}, 
but we have arranged the argument more efficiently here. 

\smallskip

Proceeding with the proof of~\eqref{e.P.to.Pbar}, we may take, without loss of generality,~$y=0$, and fix a large time~$t \geq 10$, which is the time at which we intend to estimate the difference~$P-\overline{P}$. We will write~$E=E_y$. We also pick a ``mesoscopic'' time~$t_0 \in [1,\frac12 t]$, which will be the time~$s$ in the statement of the proposition (the variable~$s$ will be repurposed here), and let~$\overline{Q}$ be the solution of the initial-value problem 
\begin{equation*}
\left\{
\begin{aligned}
& \bigl( \partial_t - \nabla \cdot\ahom\nabla \bigr) \overline{Q}
= 0 & \mbox{in} & \ (t_0,\infty)\,, \\
&  \overline{Q} = P(t_0,\cdot,0) & \mbox{on} & \ \{ t_0 \} \times \Rd\,.
\end{aligned}
\right.
\end{equation*}
Of course, we have the formula
\begin{equation}
\label{e.Qbar.convformula}
\overline{Q}(t,x) = \bigl( P(t_0,\cdot,0) \ast \overline{P}(t-t_0,\cdot) \bigr)(x)
\,.
\end{equation}
The difference~$\overline{P}(t,x,0) - \overline{Q}(t,x)$ is straightforward to estimate for~$t\gg t_0$, by a direct computation, using that the mass of~$P(t_0,\cdot,0) - \overline{P}(t_0,\cdot)$ vanishes. 
Recall that~$\Gamma_\alpha(t,x)$ is defined in~\eqref{e.heatkernel}. 
The claimed estimate is that there exists~$C(d,\lambda,\Lambda)<\infty$ such that, for every~$t \in [2t_0,\infty)$, 
\begin{equation}
\label{e.Pbar.to.Qbar}
\bigl| \overline{P}(t,x) - \overline{Q} (t,x)\bigr|
\leq 
C \Bigl( \frac t{t_0} \Bigr)^{\!\!-\nicefrac12}
\Gamma_c (t,x)
\,.
\end{equation}
By dilating time and changing the spatial coordinates by an affine transformation, one can reduce to the case~$t_0=1$ and~$\ahom =  \Id$. The estimate~\eqref{e.Pbar.to.Qbar} is then about the faster decay in time, compared to~$\Gamma_c$, of a solution to the heat equation with initial data having zero mean and which is localized near the origin. The proof of this is left as an exercise 
to the reader. 

\smallskip

Thus the remaining step is to compare~$\overline{Q}(t,x)$ to~$P(t,x,0)$ for large times~$t$. For this, we use a homogenization argument, which is almost the same as the one in Proposition~\ref{p.DP} and boils down to the corrector estimates in~\eqref{e.corrector.qualbound.L2} and~\eqref{e.corrector.qualbound.Linfty}. 
This works as follows: we select a time-cutoff function~$\zeta$ satisfying 
\begin{equation}
\label{e.zetacutends}
\indc_{ (3t_0, t-2t_0)} 
\leq 
\zeta 
\leq
\indc_{ (2t_0, t-t_0)}\,, 
\quad 
\| \zeta' \|_{L^\infty(\R)} \leq 2t_0^{-1}\,,
\end{equation}
denote 
\begin{equation*} 
\phi_{e_k}(t,x) :=  \phi_{e_k}(x) - (\phi_{e_k})_t \,, \qquad  (\phi_{e_k})_t  := (\phi_{e_k} \ast \overline{P}(t,\cdot))(0)  \,,
\end{equation*}
and define 
\begin{equation*}
H(t,x):= 
\overline{Q}(t,x) 
+
\zeta(t) \sum_{k=1}^d
\partial_{x_k} \overline{Q}(t,x) 
\phi_{e_k} (t,x) 
\,.
\end{equation*}
The correctors~$\phi_{e_k}$ we use in this expansion are the infinite-volume correctors\footnote{The finite-volume correctors defined in Section~\ref{ss.variational} can also be used, cut off at a suitable spatial scale (in the tails of the Green function terms).} defined in Section~\ref{ss.random}, with the additive constant chosen to depend on the time~$t$. 

\smallskip

Following the same computation as in the derivation of~\eqref{e.wep.eqerror.zero.2.div}, we find that 
\begin{align*}
\bigl( \partial_t - \nabla \cdot \a\nabla\bigr) (H-P) 
&
=
\bigl( \partial_t - \nabla \cdot \a\nabla\bigr) H 
-
\bigl( \partial_t - \nabla \cdot \ahom\nabla\bigr) \overline{Q}
\notag \\ & 
=
-\nabla \cdot \zeta \sum_{k=1}^d 
\bigl( (\phi_{e_k} - (\phi_{e_k})_t)  \a   - (\bfs_{e_k} - (\bfs_{e_k})_t ) \bigr) \nabla \partial_{x_k} \overline{Q} 
\notag  \\ & \qquad 
-
\nabla \cdot \bigl( (1-\zeta)(\a  - \ahom) \nabla \overline{Q}  \bigr)
+
\sum_{k=1}^d
  \partial_t \bigl( \phi_{e_k} \zeta \partial_{x_k}  \overline{Q}\bigr)
\,.
\end{align*}
Here~$(\bfs_{e_k})_t$ could be any function depending only on time and not space, valued in the skew-symmetric matrices. 
In view of the above display, we denote 
\begin{equation*}  
\mathbf{E}_k(t,x) :=  (\phi_{e_k}(x) - (\phi_{e_k})_t)  \a(x)   - (\bfs_{e_k}(x) - (\bfs_{e_k})_t )\,.
\end{equation*}
Since we have arranged that~$H - P = 0$ at the initial time~$t=t_0$, 
we can show that~$H-P$ is small if the right side of the equation above is small. In fact, by the Duhamel formula, we have
\begin{align}
\label{e.identity.HP}
(H-P)(t,x)
&
= 
\! \sum_{k=1}^d
\int_{t_0}^{t}
\zeta(s) \!
\int_{\Rd}  \!\!
\nabla_y P(t-s,x,z) 
\cdot
\mathbf{E}_k(s,z)\nabla \partial_{x_k} \overline{Q} (s,z)\,dz\,ds
\notag \\ & \qquad 
+
\int_{t_0}^{t}
(1-\zeta(s))
\int_{\Rd} 
\nabla_y P(t-s,x,z) 
\cdot (\a(z)-\ahom) \nabla\overline{Q}(s,z)\,dz\,ds
\notag \\ & \qquad 
+
\sum_{k=1}^d
\int_{t_0}^{t}
\int_{\Rd}
P(t-s,x,z) \partial_t \bigl( \phi_{e_k} \zeta \partial_{x_k}  \overline{Q}\bigr) (s,z)
\,dz\,ds
\,.
\end{align}

The expression on the right side of~\eqref{e.identity.HP} is not as complicated as it looks. Keep in mind that we want each term to be much smaller than~$P(t,x,y)$ itself, which by the Nash-Aronson upper bound estimate is of order 
\begin{equation}
\label{e.sameorderasP}
P(t,x,y) \simeq \overline{Q}(t,x) \simeq \overline{P}(t,x-y) 
\simeq \Gamma_\alpha(x-y,t)
\,.
\end{equation}
Here the symbol~$\simeq$ is not meant to be interpreted in the usual precise way; what we mean is that each term in~\eqref{e.sameorderasP} can be upper and lower bounded by the last term, up to a big or small prefactor constant~$C(d,\lambda,\Lambda)$ and changing the value of~$\alpha$ in the exponential by any large or small constant that depends only on~$(d,\lambda,\Lambda)$. 
One can now eyeball each term on the right side and see that it will be small, using the following heuristics:~$x$~derivatives on~$\overline{Q}$ gain us a factor of~$s^{-\nicefrac12}$; time derivatives on~$\overline{Q}$ or~$\zeta$ or~$( \phi_{e_k})_s$ gain us a factor of~$t_0^{-1}$ in the worst case; meanwhile the~$L^2$ oscillation of~$\phi_{e_k}$ and~$\bfs_{e_k}$ in balls of radius~$s^{\nicefrac12}$ are of size~$o(s^{\nicefrac12})$. The integrand in the second term in~\eqref{e.identity.HP} lives in a thin initial layer of size~$t_0$, outside of which~$1-\zeta$ vanishes; we therefore pick up a factor of~$o(1)$ from the ratio~$t_0/t$. In this way, we can glance at each of the three terms on the right side of~\eqref{e.identity.HP} and see that they can be made small. 

\smallskip

To be more precise, we need deterministic bounds on the parabolic Green function~$P$. In addition to the upper bound Nash-Aronson estimate in~\eqref{e.Nash.Aronson}, we also need the following gradient estimate, weighted by the heat kernel:
\begin{equation}
\label{e.nablaGF.weighted}
\int_{t_0}^t  
\int_{\Rd} 
|\nabla_y P(t-s,x,y) |^2 \Gamma_c(s,y) \,dy\,ds
\leq C \Gamma_c(t,x)\,. 
\end{equation}
Recall that in the previous section, we also used a gradient estimate for~$P$, namely~\eqref{e.concise.nablaGF}, and the estimate~\eqref{e.nablaGF.weighted} is somewhat analogous. It can be proved in the same way as~\eqref{e.concise.nablaGF}---and in our setting, we do not even need the time-integrated version---but we give an elementary and self-contained proof here below, which does not use the fact that the coefficients are independent of the time variable. We will also use estimates on the correctors in the following form: for every~$s>0$ and~$z \in \Rd$,  
\begin{equation} \label{e.Green.E.est}
\biggl( 
\int_{\Rd} 
\Bigl( \bigl|\phi_{e_k}(s,z) \bigr|^2  + s^2 \bigl| \partial_t \phi_{e_k}(s,z) \bigr|^2 
+
\bigl|\mathbf{E}_k(s,z) \bigr|^2 \Bigr)
\Gamma_c(s,z)
\,dz
\biggr)^{\!\!\nicefrac12 } 
\leq C s^{\nicefrac12} E(s^{\nicefrac12}) \,,
\end{equation}
where we recall that~$E$ is defined above in~\eqref{e.Es.def}. We also need the following pointwise bounds on~$\overline{Q}$, which are immediate from~\eqref{e.Qbar.convformula} and bounds on~$\overline{P}$ (which is the heat kernel, up to an affine change of variables): for every~$s\geq 2t_0$ and~$z \in \Rd$, 
\begin{equation} \label{e.Qbar.bound}
s^{\nicefrac 32} |\nabla \partial_t \overline{Q} (s,z)| + 
\sum_{k=0}^2 s^{\nicefrac k2}|\nabla^k \overline{Q} (s,z)| \leq C \Gamma_c(s,z)\,.
\end{equation}
We postpone the demonstration of~\eqref{e.nablaGF.weighted} and present the estimates for the three terms on the right side of~\eqref{e.identity.HP}. 

\smallskip

To estimate the first term, by the Cauchy-Schwarz inequality,~\eqref{e.nablaGF.weighted},~\eqref{e.Green.E.est} and~\eqref{e.Qbar.bound}, we get
\begin{align*}
\lefteqn{
\int_{t_0}^{t}
\zeta(s) 
\int_{\Rd} 
\bigl| \nabla_y P(t-s,x,z) \bigr|
\bigl|\mathbf{E}_k(s,z) \bigr| \bigl|  \nabla \partial_{x_k} \overline{Q} (s,z) \bigr| 
\,dz\,ds
} \quad & 
\notag \\ & 
\leq
C
\int_{2t_0}^{t-t_0}\!\!
s^{-1}
\biggl( 
\int_{\Rd} 
\frac{ | \nabla_y P(t{-}s,x,z) |^2}{\Gamma_c(t{-}s,x{-}z)}
dz
\biggr)^{\!\!\nicefrac12} 
\biggl( 
\int_{\Rd} 
\Gamma_c(t{-}s,x-z)
\bigl|\mathbf{E}_k(s,z) \bigr|^2 
\Gamma_c (s,z)^2
dz
\biggr)^{\!\!\nicefrac12} 
\!\!ds
\notag \\ & 
\leq
C t^{\nicefrac d2} \Gamma_c (t,x) 
\int_{t_0}^{t-t_0} (t-s)^{-\nicefrac d4-\nicefrac12}  s^{-\nicefrac d4-1}
\biggl( 
\int_{\Rd} 
\bigl|\mathbf{E}_k(s,z) \bigr|^2 
\Gamma_c(s,z)
\,dz 
\biggr)^{\!\!\nicefrac12 }
 \, ds
\notag \\ &
\leq
C E(t_0^{\nicefrac12}) t^{\nicefrac d2} \Gamma_c (t,x) 
\int_{t_0}^{t-t_0} (t-s)^{-\nicefrac d4-\nicefrac12}  s^{-\nicefrac d4-\nicefrac12}
\,ds 
\notag \\ & 
\leq
C \Bigl( \frac{t}{t_0} \Bigr)^{\!\nicefrac d4 -\nicefrac12} E\bigl(t_0^{\nicefrac12}\bigr) 
\Gamma_c(t,x)  
\,.
\end{align*}
For the second term, we apply~\eqref{e.quenchednablaP.improved} and~\eqref{e.Qbar.bound} to deduce on the interval~$(t_0,4t_0)$ that
\begin{align*} 
\lefteqn{
\int_{t_0}^{4t_0} (1- \zeta(s))
\int_{\Rd} 
\bigl| \nabla_y P(t-s,x,z) \bigr| \bigl| \nabla\overline{Q}(s,z) \bigr| \,dz\,ds 
} \qquad &
\notag \\ & 
\leq 
C \int_{2 t_0}^{4t_0} s^{-1} \Bigl( \frac{s}{t-s} \Bigr)^{\! \alpha}  
\int_{\Rd}  \Gamma_c(t-s,x-z)  \Gamma_c(s,z) \,dz\,ds 
\leq 
C \Bigl( \frac{t_0}{t} \Bigr)^{\! \alpha} \Gamma_c(t,x) 
\,.
\end{align*}
For the terminal time interval, we use the following: for every~$z' \in B_{s^{\nicefrac12}}(z)$,
\begin{align*} 
 \frac{ \Gamma_{c}(s,z) }{\Gamma_{c/2}(s,z')} 
= 
C \exp\biggl(- \frac{c |z|^2}{s} + \frac{c |z'|^2}{2 s} \biggr)
\leq 
C
\,.
\end{align*}
Using this together with~\eqref{e.quenchednablaP} and~\eqref{e.Qbar.bound} we obtain 
\begin{align*} 
\lefteqn{
\int_{t-2t_0}^{t - t_0} (1- \zeta(s))
\int_{\Rd} 
\bigl| \nabla_y P(t-s,x,z) \bigr| \bigl| \nabla\overline{Q}(s,z) \bigr| \,dz\,ds
}
\qquad & 
\notag \\ 
& 
\leq
C \int_{t-2t_0}^{t - t_0}  s^{-\nicefrac12} \int_{\Rd} 
\bigl| \nabla_y P(t-s,x,z) \bigr| \Gamma_c(s,z)  \,dz\,ds
\notag \\ 
&  
= 
\sum_{z' \in t_0^{\nicefrac12}\Z^d} 
\int_{t-2t_0}^{t - t_0}  s^{-\nicefrac12} 
\Gamma_{c/2}(s,z') \int_{z' + t_0^{\nicefrac12} \cu_0} \bigl| \nabla_y P(t-s,x,z) \bigr|
\notag \\ 
&  
\leq 
\sum_{z' \in t_0^{\nicefrac12}\Z^d} 
\int_{t-2t_0}^{t - t_0}  s^{-\nicefrac12} t_0^{-\nicefrac12}  \Gamma_{c/2}(s,z') 
\Gamma_{c/2}(t-s,x-z')
\notag \\ 
&  
\leq 
C \Bigl( \frac{t_0}{t} \Bigr)^{\nicefrac12} \int_{\R^d} \Gamma_{c/4}(s,z) \Gamma_{c/4}(t-s,x-z)\, dz 
\notag \\ 
&  
= 
C \Bigl( \frac{t_0}{t} \Bigr)^{\nicefrac12} \Gamma_c(t,x) \,,
\end{align*}
after shrinking~$c$. Similarly, for the third term, we obtain, using~\eqref{e.Green.E.est} together with~\eqref{e.sameorderasP} and H\"older's inequality, 
\begin{align*}
\lefteqn{
\sum_{k=1}^d
\int_{t_0}^{t}
\int_{\Rd}
P(t-s,x,z) \bigl| \partial_t \bigl( \phi_{e_k} \zeta \partial_{x_k}  \overline{Q}\bigr)(s,z)\bigr| 
\,dz\,ds
} \quad & 
\\ & 
\leq
C
\int_{t_0}^{t - t_0}
s^{-\nicefrac 32}
\Bigl( \frac{t}{t-s} \Bigr)^{\nicefrac d4} 
\biggl( 
\int_{\Rd} 
\Bigl| s^3 \partial_t \bigl( \phi_{e_k} \zeta \partial_{x_k}  \overline{Q}\bigr)(s,z)  \Bigr|
\Gamma_c(s,z)
\,dz
\biggr)^{\!\!\nicefrac12 }   
\, ds \,
\Gamma_c(t,x)
\notag \\ & 
\leq 
C  \biggl(  \int_{t_0}^{t_1-t_0} s^{-1}
\Bigl( \frac{t}{t-s} \Bigr)^{\nicefrac d4} E\bigl(t_0^{\nicefrac12}\bigr)   \, ds  \biggr)  \Gamma_c(t,x) 
\notag \\ & 
\leq 
C \Bigl( \frac{t}{t_0} \Bigr)^{\!\nicefrac d4} \log \frac{t}{t_0} 
E\bigl(t_0^{\nicefrac12}\bigr) 
\Gamma_c(t,x) 
\,.
\end{align*}
Inserting these estimates into~\eqref{e.identity.HP} yields, for every~$t \in [t_0,t_1]$,
\begin{equation*}
\bigl | (H-P)(t,x) \bigr |
\leq
C\biggl( \!
\Bigl( \frac{t_0}{t} \Bigr)^{\! \alpha}
+
\Bigl( \frac{t}{t_0} \Bigr)^{\nicefrac d4} \log \frac{t}{t_0}  
E\bigl(t_0^{\nicefrac12}\bigr) \!
\biggr)
t^{-\nicefrac d2} \exp \biggl( -\frac{|x-y|^2}{Ct} \biggr)
\,.
\end{equation*}
The combination of the previous display,~\eqref{e.Pbar.to.Qbar} and the triangle inequality yields~\eqref{e.P.to.Pbar} at time~$t=t_1$. This completes the proof.
\end{proof}

\subsection*{Historical remarks and further reading}

The proof of the invariance principle in Section~\ref{ss.invariance} can be seen as a ``quenched" version of the typical proof used for the invariance principle in the probability literature, which appeared first in the work of Kipnis and Varadhan~\cite{KV} and centers on the so-called \emph{environmental process} (``the environment from the point of view of the particle''). 
While the general idea is the same, usually these kinds of arguments are presented in the discrete setting, in which~$\Rd$ is replaced by~$\Zd$ and the analogous setup is called \emph{the random conductance model}: see~\cite{Biskup,EGMN}. 
A nice exposition (and much more in this direction) can be found in the book~\cite{KLO}. 

Our analytic, quenched variation of this argument, in particular the use of Lemma~\ref{l.varianceX}, is new to our knowledge. The second proof, including the qualitative homogenization of the parabolic Green function under qualitative assumptions, also appears to be new. However, a very closely related result was previously proved in~\cite[Section 8.6]{AKMBook}. More detailed, quantitative information on the parabolic Green function appears in~\cite[Chapter 9]{AKMBook}.


\section{Quantifying ergodicity}
\label{s.CFS} 

If we review each of the two proofs of qualitative homogenization given in Chapter~\ref{s.qualitativetheory}, we find that every step can be quantified except one: the appeal to an ergodic theorem (either Proposition~\ref{p.ergodic}, Corollary~\ref{c.stat.Hminusone}, or Proposition~\ref{p.subadditive.ergodic}). 
This is, however a very important step and, 
unfortunately, these types of abstract ergodic theorems do not provide a rate of convergence. This could not be reasonably expected, of course, since the ergodic theorem applies to processes that may converge arbitrarily slowly. A reasonable hope is a quantitative version of the ergodic theorem under a stronger (quantitative) assumption on~$\P$, and perhaps for a certain restricted subclass of stationary random fields. 

\smallskip

This raises the question, which is the focus of this chapter: What is the best, most convenient way to quantify ergodicity for proving quantitative convergence rates in stochastic homogenization? 
This is a nontrivial question, because a quantitative ergodic assumption (or ``mixing condition'')  needs to be adapted to the particular problem, and elliptic homogenization is a particular problem with a lot of structure. 
There have been many different proposals over the last several decades. 
This chapter introduces a family of quantitative mixing conditions that are optimally designed for use in elliptic homogenization.

\smallskip

Selecting a mixing condition to work under is a delicate and unpleasant business,\footnote{Mixing conditions are something of a sideshow to the rest of the quantitative theory of homogenization. They play an important but very limited role in the theory. The reader is strongly encouraged to skip this chapter on first reading, using it as a reference while proceeding directly to the more central business conducted in Chapter~\ref{s.subadd}.} and not a decision to be taken lightly. For starters, there is an enormous proliferation of mixing conditions, and entire monographs are devoted to cataloging them~\cite{Brad,B1,Torq}.
Choosing a particular assumption or class of assumptions can seem  undesirably ad hoc. Some readers may incorrectly infer that our arguments cannot handle \emph{some other} mixing assumption, which they may find particularly important.  Meanwhile, from an expository point of view, there is nothing more distracting  than engaging in a game of whack-a-mole by trying to handle different mixing conditions on a case-by-case basis. This adds very little of value. On the other hand, attempting to find a \emph{most general} assumption to work under is difficult and not without risk; if it is not well-chosen, or too complicated, then your arguments may be rendered unreadable. 
If it is not tailored to the problem, then it may not effectively capture all of the ergodicity in the problem, and then we run the risk of obtaining sub-optimal quantitative estimates---obviously something to be avoided! 
We must tread very cautiously.

\smallskip

The mixing condition introduced below in Definition~\ref{d.CFS} is relatively simple to state and well-motivated. It is flexible enough to cover all practical examples and is relatively easy to check, as we demonstrate below by checking that it is satisfied on many examples. In fact, we believe that it covers, as particular cases, all previous mixing conditions used in the literature. It is well-adapted to the setting of elliptic homogenization because it does not change the proof structure. It, therefore, does not make the exposition much more complicated than it would be under the finite range of dependence assumption ($\FRD$, defined in Section~\ref{ss.FRD} below). Finally, the quantitative estimates we obtain are sharp when our results are specialized to any of the canonical examples. 

\smallskip

Some of these results we can prove are actually new. For instance, in Chapter~\ref{s.renormalization}, we obtain optimal estimates---simultaneously in terms of scaling of the error and stochastic integrability---on the first-order correctors under our general mixing condition; when this result is applied to the particular case of coefficients fields which are local functions of Gaussian random fields with algebraic decay of correlations (see Section~\ref{ss.GRF} below), we obtain optimal estimates which have been conjectured for years,\footnote{See for instance~\cite{GNO3} and paragraph (iii) on page 5 of the third arXiv version of~\cite{GNO2}, available at https://arxiv.org/pdf/1409.2678v3.pdf (August 2015).} but until now left unproven. 

\smallskip

We call our mixing condition~\emph{concentration for sums} ($\CFS$). As this name suggests, it is a \emph{linear} concentration inequality, for sums of random variables with local dependence on the coefficient field~$\a(\cdot)$. It is evidently a very general mixing condition, because any assumption of ``quantitative ergodicity'' worthy of the name should imply some kind of concentration for finite sums of bounded random variables which depend only locally on the coefficient field. Also, the~$\CFS$ condition only requires that we check the concentration inequality on a subclass of random variables which are particularly nicely behaved in the sense of having bounded ``Malliavin derivatives.'' It turns out that all of the random variables we encounter in elliptic homogenization theory to which we need to apply the mixing condition have this property. 

\subsection{Concentration for sums:~\texorpdfstring{$\CFS$}{{(CFS)}}}
\label{ss.CFS.intro}

In this section, we introduce the mixing condition~$\CFS$ we use in most of this manuscript. As a \emph{linear} concentration inequality, it is quite easy to check in practice. It contains as special cases essentially all of the conditions previously used in the quantitative stochastic homogenization literature---including nonlinear concentration inequalities like~$\LSI$ (see Section~\ref{ss.LSI}) as well as strong mixing conditions like~$\FRD$ (Section~\ref{ss.FRD}). On the other hand,~$\CFS$ is strong enough to allow for a theory of quantitative homogenization with estimates that are optimal both in the size of the error and in stochastic integrability, and it is sufficiently flexible to allow for perturbations in the random coefficient field, a property that will be of practical use.

\subsubsection{The probability space~$\Omega$}
\label{ss.probspace} 

Fix a dimension~$d\geq 2$. Fix also ellipticity constants~$\lambda,\Lambda \in (0,\infty)$ with~$\lambda\leq \Lambda$. We let~$\Omega$ be the collection of (Lebesgue) measurable maps~$\a(\cdot)$ from~$\Rd$ into the set of~$d\times d$ matrices with real coefficients satisfying
the uniform ellipticity condition
\begin{equation} 
\label{e.ellipticity}
e \cdot \a(x) e \geq \lambda |e|^2 
\qquad \mbox{and} \qquad 
e \cdot \a^{-1}(x) e \geq \Lambda^{-1} |e|^2 \,,
\qquad \forall x,e\in\Rd
\,. 
\end{equation}
We sometimes call elements of~$\Omega$ the \emph{coefficient fields}. 
While~$\Omega$ depends on the parameters~$(d,\lambda,\Lambda)$, we typically suppress this dependence from the notation. 
We let~$\a$ denote ``the canonical coefficient field.'' This means we identify the symbol~$\a$ with the random field~$\Omega \ni \a \mapsto \a$. This is similar to the declaration that~``$X_t$ is the coordinate process'' below~\eqref{e.Upsilon.canon}. 

\smallskip

We denote the entries of~$\a$ by~$(\a_{ij})_{i,j\in\{1,\ldots,d\}}$. Given a Borel subset~$U\subseteq \Rd$, we let~$\F(U)$ denote the~$\sigma$--algebra on~$\Omega$ generated by the random variables 
\begin{equation} 
\label{e.FU}
\a \mapsto \int_U \a_{ij} (x) \varphi(x)\,dx, 
\quad 
i,j \in\{1,\ldots,d\}, 
\ 
\varphi\in C^\infty_{\mathrm{c}}(\Rd). 
\end{equation}
We think of~$\F(U)$ as containing precisely the information we can infer about the random field~$\a(\cdot)$ by looking only at its behavior in the set~$U$. We also define~$\F:=\F(\Rd)$. A \emph{random variable on}~$\Omega$ is an~$\F$--measurable function from~$\Omega$ to~$\R$.

\subsubsection{Malliavin derivatives}

For each random variable~$X$ on~$\Omega$ and Borel~$U\subseteq\Rd$, we define a random variable called the \emph{Malliavin derivative}\footnote{This is a misleading name for the object we are defining, but the term ``Malliavin derivative" has been used so frequently in this way in the stochastic homogenization literature that we may now be stuck with it. In practice,~$\left| \partial_{\a(U)} X \right|$ is the local Lipschitz constant for the random variable~$X$ with respect to variations in the coefficient field restricted to~$U$. It is clear then that~$\left| \partial_{\a(U)} X \right|$ bounds the Fr\'echet-type derivative of~$X$ with respect to the coefficient field restricted to~$U$. Note that the Fr\'echet derivative has a well-understood connection to the classical Malliavin derivative, for example in the case of isonormal Gaussian processes.}  
$|\partial_{\a(U)} X|$ of~$X$ with respect to the environment restricted to~$U$ by 
\begin{align} 
\label{e.Mall}
\lefteqn{
\left| \partial_{\a(U)} X\right| (\a) 
} \qquad & 
\notag \\ & 
:=
\limsup_{t\to 0} 
\frac1{2t} 
\sup \,
\bigl\{ X(\a_1) - X(\a_2) 
\,:\,
\a_1,\a_2 \in\Omega, \ \left| \a - \a_i \right| \leq t \indc_{U} \ \mbox{for} \ i\in\{1,2\}
\bigr\}.
\end{align}
The random variable~$|\partial_{\a(U)}X|$ quantifies the dependence of~$X$ on the restriction of the coefficient field to~$U$. Observe that~$X$ is~$\F(U)$--measurable if and only if~$|\partial_{\a(\Rd\setminus U)} X| = 0$. 

\smallskip

One useful property of the Malliavin derivative is its additivity: for every~$\a\in\Omega$ and disjoint Borel subsets~$U,V\subseteq \Rd$, we have that 
\begin{equation}
\label{e.malliavin.additivity}
\left| \partial_{\a(U\cup V)} X \right| 
= 
\left| \partial_{\a(U)} X \right| 
+
\left| \partial_{\a(V)} X  \right|.
\end{equation}

\subsubsection{Notation for Orlicz space norms}
Given a random variable~$X$ on~$\Omega$, a constant~$A>0$ and an increasing function~$\Psi:[1,\infty) \to [0,\infty)$ 

we use the notation 
\begin{equation}
X \leq \O_\Psi \left( A \right)
\end{equation}
to represent the statement that
\begin{equation}
\label{e.weakcondOPsi}
\P[ X > tA  ] \leq \frac1{\Psi(t)}\,, 
\qquad \forall t \in [1,\infty)
\, . 
\end{equation}
We also write~$X = \O_\Psi(A)$ to mean that~$|X| \leq \O_\Psi(A)$. The expression
\begin{equation*}
X = \O_{\Psi_1}(A_1) + \O_{\Psi_2}(A_2), 
\end{equation*}
means that~$X$ can be split into a sum of two random variables,~$X=Y+Z$, such that~$Y=\O_{\Psi_1}(A_1)$ and~$Z = \O_{\Psi_2}(A_2)$. 

\smallskip

Why do we use this~$\O_\Psi$ notation? Instead of writing~$X \leq \O_\Psi(A)$, we could express this inequality equivalently in terms of a norm or seminorm of~$X$. The reason is that, in our context, the random variables~$X$ we study can be quite long to express, and the~$\O_\Psi(\cdot)$ notation allows us to string together inequalities efficiently. It is also quite convenient to reserve the norm symbol~$\| \cdot \|$ for \emph{spatial} norms in the physical space, as opposed to norms of random variables in the probability space. 

\smallskip

For each~$s\in (0,\infty]$, we define~$\Gamma_s$ by 
\begin{equation}
\label{e.Gamma.s.geq1}
\Gamma_s(t) := 
\left\{
\begin{aligned}
& \exp\biggl( \frac1s t ^s \biggr) - \exp\biggl( \frac1s \biggr) & \ \mbox{if} \ s \in (0,1)\,, \\
& \exp\biggl( \frac1s |t|^s \biggr) - 1 & \ \mbox{if} \ s \in [1,\infty)
\end{aligned}
\right.
\end{equation}
and, for~$s=\infty$, 
\begin{equation}
\label{e.Gamma.infty}
\Gamma_\infty(t) = 
+\infty 
\end{equation}
Note that the inequality~$X \leq \O_{\Gamma_\infty}(A)$ is equivalent to the statement that~$X \leq A$,~$\P$--a.s. 
It will be convenient also to denote, for~$\gamma>0$, 
\begin{equation*}
X = \O_\Psi^\gamma(A) 
\quad \iff \quad
X^\gamma = \O_\Psi (A^\gamma). 
\end{equation*}

\smallskip

Throughout, we will assume that our functions~$\Psi$ satisfy the minimal growth condition
\begin{align}
\label{e.Young.growth}
\int_{1}^\infty 
\frac{t}{ \Psi(t)} \,dt < +\infty \,,
\end{align}
which ensures that
\begin{equation*}
X = \O_{\Psi}(A) 
\quad \implies \quad 
\var[X] \leq C A^2
\end{equation*}
for a constant~$C(\Psi)<\infty$. 
We also assume that a generalized triangle inequality for the~$\O_\Psi(\cdot)$ notation. This means that there exists a constant~$C_{\Psi} \in [1,\infty)$ such that, if~$\{ X_j \}_{j\in\N}$ is a sequence of nonnegative random variables satisfying~$X_j = \O_{\Psi}(1)$ for every~$j \in \N$ and~$\{ a_j\}_{j \in \N}>0$ a sequence of nonnegative numbers satisfying~$\sum_{j=1}^\infty a_j =1$, then
\begin{equation} 
\label{e.weak.triangle.ass}
\sum_{j =1}^\infty a_j X_j = \O_{\Psi}(C_{\Psi}) \,.
\end{equation}  
The following lemma presents a sufficient condition for~$\Psi$ to satisfy~\eqref{e.weak.triangle.ass}, namely the doubling condition~\eqref{e.Psi.pgrowth.supplementary1} below. Note that this is satisfied by most relevant examples, such as~$\Psi = \Gamma_s(\cdot)$ or~$\Psi(t) = t^p$ for~$p>2$. 

\begin{lemma} 
\label{e.supp.you.up}
Let~$p>1$ and~$N \in [1,\infty)$. Suppose that~$\Psi:[1,\infty) \to [1,\infty)$ is an increasing function that satisfies
\begin{equation}  \label{e.Psi.pgrowth.supplementary1}
s^p \leq N \frac{\Psi(t s)}{\Psi(t)}\,
\quad \forall s, t \in [1,\infty) \,.
\end{equation}
Let~$\{ X_j \}_{j\in\N}$ be a sequence of nonnegative random variables satisfying~$X_j = \O_{\Psi}(1)$ for every~$j \in \N$. Let~$\{ a_j\}_{j \in \N}>0$,~$a_j \geq 0$, be such that~$\sum_j a_j =1$. Then there exists a constant~$C(p,N) < \infty$ such that~$\sum_{j} a_j X_j = \O_\Psi(C)$.  
\end{lemma}
\begin{proof}
Fix~$t \geq 1$, and let~$K\geq 2$ be a parameter to be fixed by means of~$(p,N)$. Decompose~$X_j = m_j + v_j$ with~$m_j := \indc_{\{X_j>t \} } X_j$ and~$v_j := \indc_{\{X_j \leq t \} } X_j$. By Markov's inequality, 
\begin{equation*}  
\P\Bigl[ \sum_j a_j X_j > K t  \Bigr] \leq \P\Bigl[ \sum_j a_j m_j > (K-1)t  \Bigr] 
\leq
\frac{1}{t(K-1)}\sum_{j} a_j \E[m_j]
\,.
\end{equation*}
Writing~$m_j = \indc_{\{X_j > t \} } (X_j -t) + \indc_{\{X_j > t \} } t$, we compute
\begin{align*}  
\E[m_j] 
& 
= 
\int_{t}^{\infty} \P[X_j  > s ] \, ds + t \P [ X_j > t]
\\ & 
\leq  
t \int_{1}^{\infty} \frac{1}{\Psi(s t)} \, ds +  \frac{t}{\Psi(t)} 
\leq  
\frac{t}{\Psi(t)} \int_{1}^{\infty} N s^{-p} \, ds +   \frac{t}{\Psi(t)}
\leq  
\Bigl( \frac{N}{p-1} +1 \Bigr) \frac{t}{\Psi(t)} 
\,.
\end{align*}
The result now follows by taking~$K := 2 + \frac{N}{p-1}$.
\end{proof}

\subsubsection{Concentration For Sums}
We now introduce the general quantitative ergodic hypotheses used throughout the article, which we name \emph{concentration for sums}, denoted by~$\CFS(\beta,\Psi,\Psi',\gamma)$ for short. 
Here~$\beta \in \left[0,1\right)$,~$\gamma\in [0,\infty)$ and~$\Psi$ and~$\Psi'$ are a pair of increasing functions~$[1,\infty) \to [1,\infty)$ satisfying~\eqref{e.Young.growth} and the generalized triangle inequality.
The~$\CFS(\beta,\Psi,\Psi',\gamma)$ condition says roughly that, for every large cube~$\cu_m$, if we have a collection of random variables~$\{ X_z\}$ which are measurable with respect to the mesocopic subcubes~$z+\cu_n \subseteq \cu_m$ with~$\beta m < n < m$, have Malliavin derivative bounded by one, and have fluctuations bounded by~$\O_{\Psi'}(1)$. The sample mean of these random variables exhibits some concentration with respect to~$\O_\Psi$.

\smallskip

Here and in what follows, we will put a slash through finite sums to denote average: if~$S$ is a finite set, then~$|S|$ denotes the cardinality of~$S$, and we write
\begin{equation*}
\avsum_{i \in S} a_i := \frac{1}{|S|} \sum_{i\in S} a_i.
\end{equation*}

\begin{definition}[Concentration for sums~$\CFS(\beta,\Psi)$]
\label{d.CFS}
Let~$\beta \in \left[0,1\right)$ and~$\Psi:[1,\infty) \to [0,\infty)$ be an increasing function satisfying~\eqref{e.Young.growth} and the generalized triangle inequality~\eqref{e.weak.triangle.ass} for a some~$C_\Psi<\infty$.
A probability measure~$\P$ on~$(\Omega,\F)$ \emph{satisfies~$\CFS(\beta,\Psi)$} provided that: 
\begin{itemize}
\item
For every~$m,n\in\N$ with~$\beta m < n<m$, and  family~$\{ X_z \,:\, z\in 3^n\Zd\cap \cu_m\}$ of random variables satisfying, for every~$z\in\Zd$,
\begin{equation} 
\label{e.CFS.ass}
\left\{
\begin{aligned}
& \E\left[ X_z \right]=0 \,, 
\\ & 
\left| X_z \right| \leq 1\,,
\\ &   
\left| \partial_{\a(z+\cu_n)} X_z\right| \leq 1 \,,
\\
& X_z \quad \mbox{is~$\F(z+\cu_n)$--measurable} \,, \\
\end{aligned}
\right.
\end{equation}
the random variable~$\overline{X}$ defined by
\begin{equation} 
\label{e.Z.norm}
\overline{X}  :=  \avsum_{z\in 3^n\Zd\cap \cu_m}  X_z
\end{equation}
satisfies the bound
\begin{equation}
\label{e.CFS}
\overline{X}
\leq
\O_{\Psi}
\bigl( 3^{-\frac d2(m-n)} \bigr).
\end{equation}
\end{itemize}
\end{definition}

The condition~$\CFS(\beta,\Psi)$ is the most general and commonly used mixing condition in this manuscript and the only mixing condition used before Chapter~\ref{s.renormalization}. 
Sometimes, we require the following stronger variant of the~$\CFS(\beta,\Psi)$ condition. 
\begin{definition}[Stronger concentration for sums,~$\CFS(\beta,\Psi,\Psi^\prime,\gamma)$] \hspace{-2pt} 
\label{d.CFS.strong}
Let~$\beta \in \left[0,1\right)$,~$\gamma\in [0,\infty)$ and~$\Psi,\Psi':[1,\infty) \to [0,\infty)$ be increasing functions satisfying~\eqref{e.Young.growth} and the generalized triangle inequality~\eqref{e.weak.triangle.ass}.
We say that a probability measure~$\P$ on~$(\Omega,\F)$ \emph{satisfies~$\CFS(\beta,\Psi,\Psi',\gamma)$} provided that: 
\begin{itemize}
\item
For every~$a\in (0,1]$,~$m,n\in\N$ with~$\beta m < n<m$, and  family~$\{ X_z \,:\, z\in 3^n\Zd\cap \cu_m\}$ of random variables satisfying
\begin{equation} 
\label{e.CFS.strong.ass}
\left\{
\begin{aligned}
& \E\left[ X_z \right]=0\,, 
\\ & 
\left| X_z \right| \leq \O_{\Psi'}(a)\,,
\\ &
\left| \partial_{\a(z+\cu_n)} X_z\right| \leq 1
\\ & 
X_z \quad \mbox{is~$\F(z+\cu_n)$--measurable,}\\
\end{aligned}
\right.
\end{equation}
the random variable 
\begin{equation} 
\label{e.Z.strong.norm}
\overline{X}  :=  \avsum_{z\in 3^n\Zd\cap \cu_m}  X_z
\end{equation}
satisfies the bound
\begin{equation}
\label{e.strong.CFS}
\overline{X}
\leq
\O_{\Psi}\left( a 3^{-\frac d2(m-n)} 
+ 3^{-\frac d2(1-\beta)m - n\gamma} \right).
\end{equation}
\end{itemize}
We also write~$\CFS(\beta,\Psi,\Psi') := \CFS(\beta,\Psi,\Psi',0)$. 
\end{definition}

The condition~$\CFS(\beta,\Psi)$ is the weakest possible assumption since it only requires that the concentration inequality be valid for sums of bounded random variables. Indeed, for every~$\Psi'$ with~$\Psi'(1) \leq 1$ and~$\gamma\geq 0$,
\begin{equation}
\label{e.yes.weaker}
\CFS(\beta,\Psi,\Psi',\gamma) 
\implies
\CFS(\beta,\Psi,\Psi')
\implies
\CFS(\beta,\Psi)
\,.
\end{equation}

\subsubsection{Which~$\CFS$ conditions are used and where they are used}

\begin{itemize}
\item The condition~$\CFS(\beta,\Psi)$ is the only mixing condition used in Sections~\ref{s.subadd},~\ref{s.nonsymm} and~\ref{s.regularity}. 
This is where we give sharp estimates for the minimal length scale on which the  homogenization error becomes small, and on which the large-scale regularity estimates are valid. Note~$\CFS(\beta,\Psi)$ covers essentially all known examples which exhibit at least a power-like decay of correlations.  

\item The renormalization theory in Chapter~\ref{s.renormalization} requires a stronger condition in order to obtain optimal stochastic moments at the optimal scale. There we assume essentially that~$\P$ satisfies~$\CFS(\beta,\Psi,\Psi,\gamma)$, with~$\gamma>0$ in the case~$\beta=0$, and we apply the assumption in~Section~\ref{s.fluctboot} (see Step~4 of the proof of Proposition~\ref{p.psipsi.coarse}). This assumption is satisfied by most reasonable examples (see below), but it is not implied by abstract nonlinear concentration inequality assumptions~$\LSI$ or~$\SG$. However, the estimates for the latter can be easily recovered by applying the concentration inequalities to the first-order correctors directly, as explained in Section~\ref{s.LSI}. 
\end{itemize}

Throughout the rest of the manuscript, we collect all the parameters into a list and denote them as
\begin{equation} 
\label{e.data.def}
\data := (\beta,d,\lambda,\Lambda,\CFS,\Psi) \,.
\end{equation}
We define~$C(\data)$ as a constant that depends on the parameters listed above. The dependence on~$\Psi$ refers to~$C_{\Psi}$ in equation~\eqref{e.weak.triangle.ass}, growth condition~\eqref{e.Young.growth}, and the value of~$\Psi(1)$. We have included~$\CFS$ in the list to emphasize that the~$\CFS$ condition was applied in the context where~$\data$ is mentioned.

\subsection{Examples of coefficient fields satisfying~\texorpdfstring{$\CFS$}{{CFS}}}

\label{ss.examples}

We will spend most of the rest of the chapter verifying that the~$\CFS$ condition is satisfied in various important examples.

\subsubsection{Finite range of  dependence}
\label{ss.FRD}
The quintessential quantitative ergodic condition for stationary random fields in Euclidean space is \emph{finite range of dependence}, which we denote by~$\FRD(L)$ for a parameter~$L>0$ which has units of length and is called the~\emph{range of dependence}. It asserts that, for every pair of Borel subsets~$U,V\subseteq\Rd$,
\begin{equation} 
\label{e.unitrange}
\dist(U,V) \geq L
\quad\mbox{implies that} \quad 
\mbox{$\F(U)$ and~$\F(V)$ are~$\P$--independent.}
\end{equation}
By a change of scale, it suffices to consider only the case~$L=1$. We will show that 
\begin{equation}
\label{e.URDyes}
\FRD(1) \implies \CFS(0,\Gamma_2(c\cdot),\Gamma_2,\infty),
\end{equation}
for a universal constant~$c(d)>0$. In particular, the unit range of dependence assumption implies a normal concentration of sums down to the microscopic scale.

\smallskip

We pick~$m,n\in\N$ with~$n<m$ and family~$\{ X_z \,:\, z\in 3^n\Zd\cap \cu_m\}$ of random variables satisfying~\eqref{e.CFS.strong.ass} with~$\Psi'=\Gamma_2$. We do not need to use the assumed estimate of the Malliavin derivative of the~$X_z$'s, just their~$\O_{\Gamma_2}$-boundedness by~$a$, which implies
\begin{equation}
\label{e.Xz.bound.indy}
\P \bigl[ X_z > as \bigr] \leq 3 \exp\biggl( - \frac12 s^2 \biggr)\,, \qquad \forall s\geq 1\,.
\end{equation}
Let~$Z$ be defined by 
\begin{equation}
\label{e.Z}
Z:= \sum_{z\in 3^n\Zd \cap\cu_m} 
X_z. 
\end{equation}
Note that this~$Z$ is~$3^{d(m-n)} \overline{X}$, where~$\overline{X}$ is defined in~\eqref{e.Z.strong.norm}. 
We proceed using the standard exponential moment method. 
For every~$t\in\R$, we use the independence assumption to obtain
\begin{equation*}
\log \E \left[ \exp\left( t Z\right) \right]
= 
\log \E \biggl[ \exp\biggl( t\sum_{z} X_z \biggr) \biggr]
= 
\sum_{z} \log \E \bigl[ \exp ( t X_z ) \bigr].
\end{equation*}
Actually, we have just cheated, but only slightly. While the random variables~$X_z$ are~$\F(z+\cu_n)$--measurable, these~$\sigma$--algebras are not quite independent because any two adjacent subcubes are touching and thus not separated by a unit distance. To fix this, we need to color all the subcubes with~$3^d$ different colors so that no two cubes with the same color are touching. For instance, we can color each subcube depending on its relative positions in its predecessor cube, that is, the larger cube congruent to~$\cu_{n+1}$, which contains it.  Any collection of~$\sigma$-algebras corresponding to subcubes with the same color is then independent. Repeating the above calculation for each color group and using H\"older's inequality to relate this to the original quantity, we get
\begin{equation}
\label{e.indy.expmom}
\log \E \left[ \exp\left( t  Z\right) \right]
\leq 
3^{-d} \sum_{z} \log \E \left[ \exp \left( 3^dt  X_z  \right) \right].
\end{equation}
Using~\eqref{e.Xz.bound.indy}, we compute, for each~$z\in 3^n\Zd\cap \cu_m$ and~$t>0$, 
\begin{align*}
\E \left[ \exp \left( 3^dt  X_z  \right) \right]
& 
=
at3^d 
\int_{-\infty}^\infty 
\exp(at 3^d s) \P \bigl[ X_z > a s \bigr] \,ds
\\ & 
\leq
at3^d 
\int_{-\infty}^1 
\exp(at 3^d s) \,ds
+
at3^d 
\int_{1}^\infty 
\exp(at 3^d s) \P \bigl[ X_z > a s \bigr] \,ds
\\
&
\leq 
\exp( at3^d) 
+
3at3^d 
\int_{1}^\infty 
\exp\biggl(at 3^d s -\frac12s^2 \biggr) \,ds
\\ & 
\leq
\exp( at3^d) 
+
3\sqrt{2\pi} 
at3^d \exp\biggl( \frac123^{2d} a^2 t^2 \biggr)
\\ & 
\leq 
C \exp\bigl( 3^{2d} a^2 t^2 \bigr)
\,.
\end{align*}
Taking the logarithm of this expression and substituting the result into~\eqref{e.indy.expmom}, we obtain 
\begin{equation*}
\log \E \left[ \exp\left( tZ\right) \right]
\leq 
3^{d(1+m-n)} a^2 t^2 + C 3^{d(m-n-1)} 
\,.
\end{equation*}
By Chebyshev's inequality, for every~$t,\lambda>0$, 
\begin{align*}
\P \bigl[ Z > \lambda \bigr]
\leq 
\exp\left( -t\lambda \right)
\E \left[  \exp \left( t Z \right) \right]
\leq 
\exp\left( -t\lambda + 3^{d(1+m-n)} a^2 t^2 + C3^{d(m-n-1)} \right)\,.
\end{align*}
Taking~$t:= \frac12 \lambda a^{-2} 3^{-d(1+m-n)}$, we get
\begin{equation} 
\label{e.unitrange.result}
\P \bigl[ Z > \lambda \bigr]
\leq
\exp\left( -\frac14 3^{-d(1+m-n)} a^{-2} \lambda^2 + C3^{d(m-n-1)} \right)\,.
\end{equation}
Using that~$a\leq 1$, we deduce that
\begin{equation*}
\overline{X} = 
3^{-d(m-n)} Z 
\leq 
\O_{\Gamma_2}\bigl(Ca 3^{-\frac d2(m-n)} \bigr) \,,
\end{equation*}
which implies~\eqref{e.strong.CFS} for~$\Psi = \Gamma_2(c\cdot)$. This completes the proof of~\eqref{e.URDyes}. 

\smallskip

The implication~\eqref{e.URDyes} can be generalized to~$\Gamma_s$ for~$s \in (1,2]$. Indeed, according to~\cite[Lemma A.9 \& A.12]{AKMBook}, for each~$s\in (1,2]$, there exists~$M(s,d)>0$ such that 
\begin{equation}
\label{e.URDyes.s}
\FRD(1) \implies \CFS(0,\Gamma_s,\Gamma_s(M \cdot),\infty) \,.
\end{equation}
This can be further extended to all~$s> 0$. but we leave this proof to the reader. 

\subsubsection{Approximate finite range dependence}
\label{ss.AFRD}
It is often convenient to check that a given random coefficient field can be well-approximated by a sequence of fields with a finite range of dependence. Typically, the size of the error in the approximation is random, but it becomes smaller as the range of dependence of the approximating field becomes larger.  

\smallskip

To formalize this notion, 
we say that~$\P$ is \emph{approximately finite range dependent with parameters~$s \in (0,\infty)$ and~$K \in [1,\infty)$}, denoted by~$\AFRD(s,K)$ for short, if there exists a family~$\{ \a_m\}_{m\in\N}$ of~$\F$-measurable random fields such that 
\begin{align}
\label{e.AFRD1}
\text{the law of~$\a_m$ is~$3^m\Zd$--stationary and  satisfies~$\FRD(3^m)$} 
\end{align}
and
\begin{align}
\label{e.AFRD2}
\left\| \a - \a_m \right\|_{L^\infty(\cu_m)} 
\leq 
\O_{\Gamma_2}(K3^{-ms})
\,.
\end{align}
We claim there exists~$c(d)>0$ such that, for every~$K \in [1,\infty)$ and~$s \in (0,\infty) \setminus \{ \tfrac d2 \}$, 
\begin{equation}
\label{e.AFRD.okay}
\AFRD\left(s,K \right) \implies
\CFS \Bigl(  
\bigl( 1 - \tfrac {2s}{d} \bigr)_+ , \,
\Gamma_2 \bigl( c \bigl( 1 + K (1+s^{-1}) \bigl| \tfrac d2-s\bigr|^{-1} \bigr)   \cdot \bigr), \,
\Gamma_2, \, 
\bigl( s - \tfrac d2 \bigr)_+
\Bigr)\,.
\end{equation}
To prove~\eqref{e.AFRD.okay}, 
we fix a parameter~$\beta \in [0,1)$ to be chosen below and fix~$m,n\in\N$ such that~$\beta m < n < m$ and a family~$\{ X_z \,:\, z\in 3^n\Zd\cap \cu_m\}$ of random variables satisfying~\eqref{e.CFS.strong.ass} with~$\Psi'=\Gamma_2$.
We suppose that~$\P$ satisfies~$\AFRD(s,K)$ for~$s,K\in(0,\infty)$. By the assumptions~\eqref{e.AFRD2} and~\eqref{e.CFS.strong.ass} we deduce that, for~$k \in \N$, 
\begin{equation} 
\label{e.Xz.ak.diff}
\left| X_{z} (\a_k) - X_{z}(\a_{k-1}) \right| 
\leq
\left| \partial_{\a(z+\cu_n)} X_{z} \right| \left\| {\a}_k - {\a}_{k-1}  \right\|_{L^\infty(z+\cu_n)} 
\leq 
\O_{\Gamma_2} 
\left(CK3^{-k s} \right)
 \,.
\end{equation}
Therefore, by telescope summation, 
\begin{equation} 
\label{e.Xz.ak}
\left| X_{z} (\a_k) \right|  \leq \O_{\Gamma_2} 
\left(C(Ks^{-1} 3^{-k s} \vee a) \right) \,.
\end{equation}
Let~$Z=\Z(\a)$ be given by~\eqref{e.Z} and, 
for~$j \in \N$,~$k \in\N \cap (n,m]$ and~$z\in 3^k\Zd\cap\cu_m$, define
\begin{equation*}
Z(\a_k) = \sum_{z \in 3^n \Zd \cap \cu_m} \bigl( X_{z}(\a_k) - \E\bigl[ X_{z}(\a_k) \bigr]  \bigr) 
\,.
\end{equation*}
We next split~$Z(\a)$ by writing
\begin{equation*}
Z(\a) 
=
Z(\a) - Z({\a}_{m-2}) + Z({\a}_n) + \sum_{k=n}^{m-3} \left( Z({\a}_{k+1}) - Z({\a}_k) \right)\,,
\end{equation*}
We use~\eqref{e.URDyes},~\eqref{e.Xz.ak} and Lemma~\ref{e.supp.you.up} to bound 
\begin{equation*}
\left| Z({\a}_n) \right| 
\leq 
\O_{\Gamma_2} \biggl( C \biggl( \frac{|\cu_m|}{|\cu_n|} \biggr)^{\!\!\nicefrac12}   ( K s^{-1} 3^{-ns} \vee a)     \, \biggr) .
\end{equation*}
Next, we write 
\begin{equation*}
Z({\a}_{k+1}) - Z({\a}_k)
=
\sum_{z \in 3^n \Zd \cap \cu_m} \bigl( X_{z}(\a_{k+1}) -X_{z}(\a_k) - \E\bigl[ X_{z}(\a_{k+1}) - X_{z}(\a_k) \bigr]  \bigr)
\end{equation*}
and observe that the sum can be split into~$3^{d(k-n+2)}$ many sums, each containing~$3^{d(m-k-2)}$ mutually independent members having zero mean and bounded by~\eqref{e.Xz.ak.diff}. Therefore, using \eqref{e.URDyes} and Lemma~\ref{e.supp.you.up},
\begin{equation*}
\left| Z({\a}_{k+1}) - Z({\a}_k) \right|
\leq 
\O_{\Gamma_2} \biggl( CK3^{-ks}\frac{|\cu_k|}{|\cu_n|}\left( \frac{|\cu_m|}{|\cu_k|} \right)^{\!\!\nicefrac12} \, \biggr) \,.
\end{equation*}
The first term is bounded by~\eqref{e.AFRD2} and~\eqref{e.CFS.strong.ass} as
\begin{equation*}
\left| Z(\a) - Z({\a}_{m-2}) \right| 
\leq 
\O_{\Gamma_2}
\left(CK3^{-ms} \frac{|\cu_m|}{|\cu_n|}\right),
\end{equation*}
Putting these together and using
\begin{align*}  
\frac{|\cu_n|}{|\cu_m|}  \sum_{k=n}^m 3^{-ks}\frac{|\cu_k|}{|\cu_n|}\left( \frac{|\cu_m|}{|\cu_k|} \right)^{\!\!\nicefrac12}
& =
3^{-m s} 
\sum_{k=0}^{m-n}  3^{ k(s - \frac d2) } 
\\ &
\leq
C
\begin{cases}
\frac{1}{d-2s} 3^{-m s}   &   \mbox{ if } s < \frac{d}{2},
\\
(m-n) 3^{-m \frac d2 }   & \mbox{ if } s = \frac{d}{2} ,
\\
\frac{1}{2s-d} 3^{-m \frac d2 } 3^{-n (s - \frac d2)} & \mbox{ if } s > \frac{d}{2} ,
\end{cases}
\end{align*}
 we obtain by Lemma~\ref{e.supp.you.up} that
\begin{equation*}
\left| Z(\a) \right| = \O_{\Gamma_2} \left( 
C 3^{\frac d2 (m-n)} a 
+ 
C K \frac{|\cu_m|}{|\cu_n|}  \begin{cases}
\frac{1}{s(d-2s)} 3^{-m s}   &   \mbox{ if } s < \frac{d}{2}
\\
(m - n) 3^{-m \frac d2 }   & \mbox{ if } s = \frac{d}{2} 
\\
\frac{1}{2s-d} 3^{-m \frac d2 }   3^{-n (s - \frac d2)}  & \mbox{ if } s > \frac{d}{2} 
\end{cases} \right) 
\,.
\end{equation*}
This completes the proof of~\eqref{e.AFRD.okay}.

\smallskip

We can generalize~\eqref{e.AFRD.okay} by using~\eqref{e.URDyes.s} in place of~\eqref{e.URDyes}, to obtain, for every~$\sigma \in (1,2]$, 
\begin{equation}
\label{e.AFRD.okay.sigma}
\AFRD\left(s,K \right) \implies
\CFS \Bigl(  
\bigl( 1 - \tfrac {2s}{d} \bigr)_+ , \,
\Gamma_\sigma \bigl( c\bigl( 1 + K (1+s^{-1}) \bigl| \tfrac d2-s\bigr|^{-1} \bigr)   \cdot \bigr), \,
\Gamma_\sigma, \,
\bigl( s - \tfrac d2 \bigr)_+
\Bigr)\,,
\end{equation}
where the constant~$c$ now depends on~$\sigma$ in addition to~$d$.

\subsubsection{Local functions of Gaussian fields}
\label{ss.GRF}
Suppose~$\P$ is defined implicitly by
\begin{equation}
\label{e.Gaussa0}
\a (x) : = \a_0(F(x)) \,,
\end{equation}
where~$\a_0 : \R^k \to \R^{d\times d}$,~$k\in\N$, is a given deterministic function satisfying, for~$K>0$,  
\begin{equation}
\label{e.Gaussa0.lip}
\left| \a_0 (X) - \a_0(Y) \right| \leq K |X-Y| \quad \forall X,Y \in \R^k, \quad 
 \Id \leq \a_0 \leq \Lambda  \Id
\,,
\end{equation}
and~$F$ is an~$\R^k$-valued, centered,~$\Rd$--stationary Gaussian field with covariance function 
\begin{equation}
\mathrm{c}(x) : = \E [ F(x) \otimes F(0) ]\,.
\end{equation}
We assume that the latter can be written as~$c = f \ast f$ where the kernel~$f:\Rd \to \R^k$ satisfies, for some~$s \in (0,\infty) \setminus \{ \tfrac d2 \}$ and~$K\in [1,\infty)$, 
\begin{equation}
\label{e.kerneldecay.ass}
\left| \nabla^n f(x) \right|
\leq 
K\bigl(1+|x| \bigr)^{-(\frac d2+s)-n},
\quad 
\forall n\in\{ 0,\ldots, \left\lfloor \tfrac d2 \right\rfloor +1\}
\end{equation}
so that the field~$F$ can be realized as the convolution of~$f$ against a scalar white noise~$W$:
\begin{equation}
F(x) = \int_{\Rd} f(x-y)W(y)\,dy.
\end{equation}
Notice that this assumption implies that the covariances of the Gaussian field have power-like decay with exponent~$2s$; more precisely, there exists a constant~$C(d)<\infty$ such that
\begin{equation*}
\cov\bigl[F(x),F(y) \bigr]
\leq
CK \bigl(1+|x| \bigr)^{-2s}
\,.
\end{equation*}
We will show, in this particular case, that  
\begin{equation}
\label{e.Gauss.to.AFRD}
\P \ \mbox{ satisfies } \AFRD(s,CK).
\end{equation}
In view of~\eqref{e.AFRD.okay.sigma},
we deduce also that, for every~$\sigma \in (1,2]$, 
\begin{align}
\label{e.Gauss.to.CFS}
\P \ \mbox{ satisfies }  \quad
\CFS \Bigl(  
\bigl( 1 - \tfrac {2s}{d} \bigr)_+ , \,
\Gamma_\sigma \bigl( cK^{-1} \bigl| \tfrac d2-s\bigr| \cdot \bigr), \,
\Gamma_\sigma, \,
\bigl( s - \tfrac d2 \bigr)_+
\Bigr)\,.
\end{align}
To prove~\eqref{e.Gauss.to.AFRD}, we select a smooth partition of unity 
$\{ \xi_{m} \}_{m\in\N}$ such that, for every~$m\in\N$, 
\begin{align*}
\indc_{\cu_{m-1}}
\leq
\sum_{k=0}^m \xi_k
\leq
\indc_{\cu_{m}}
\quad \text{and} \quad
\big\| \nabla^k \xi_m \big\|_{L^\infty(\Rd)}
\leq 
C3^{-mk},
\quad 
\forall k\in\{ 0,\ldots, \left\lfloor \tfrac d2 \right\rfloor +1\}
\,.
\end{align*}
It is clear from the above that~$\xi_0$ is supported in~$\cu_0$ and, for every~$k\geq1$, the function~$\xi_{k}$ is supported in~$\cu_k\setminus\cu_{k-2}$. We then define the Gaussian random field 
\begin{equation*}
F_m(x):= \sum_{n=0}^m 
\int_{\Rd} \xi_n(x-y) f(x-y) W(y)\,dy.
\end{equation*}
It is clear that~$F_m$ has a range of dependence at most~$2\sqrt{d} 3^m \leq C3^m$. Moreover, a routine computation, using the properties\footnote{Rescaled distribution~$W_R = R^{\nicefrac d2} W(R \, \cdot )$ is still white noise, and to get size bounds for the white noise,  one can apply, for example,~\cite[Proposition 5.14]{AKMBook}.}  of white noise and the assumption~\eqref{e.kerneldecay.ass}, yields by Lemma~\ref{e.supp.you.up} that
\begin{equation*}
\left\| F - F_m \right\|_{L^\infty(\cu_0)} 
\leq 
\O_{\Gamma_2}\biggl(
C \sum_{n=m+1}^\infty 2^{ \nicefrac{ n d}2  } \big\| (\xi_n f)(2^{n+1} \cdot ) \big\|_{H^{\left\lfloor \nicefrac d2 \right\rfloor +1}(B_{1})} 
\biggr)
\leq 
\O_{\Gamma_2}\bigl( CK 3^{-m s}\bigr)
\,.
\end{equation*}
It is now evident from the previous display and~\eqref{e.Gaussa0}--\eqref{e.Gaussa0.lip} that the law of the Gaussian field~$\a(\cdot)$ satisfies~$\AFRD\left(s,CK\right)$.

\subsubsection{Strong~$\alpha$-mixing with algebraic decay of correlations}
We next consider the case that~$\P$ satisfies, for every pair of Borel sets~$U,V\subseteq\Rd$, 
\begin{equation}
\label{e.alphamixing}
\left| \P\left[ A\cap B\right] - \P\left[ A \right]\P\left[ B \right] \right|
\leq
\phi\left( \dist(U,V) \right),
\quad 
\mbox{for every}\ A\in \F(U),\, B\in \F(V),
\end{equation}
where~$\phi:[0,\infty) \to [0,1]$ is a monotone nonincreasing function satisfying~$\lim_{t\to \infty} \phi(t) = 0$. This is the so-called ``$\alpha$--mixing'' condition with decorrelation rate given by~$\phi$, and we denote it by~$\alpha\mathsf{Mix}(\phi)$. Note that the condition~\eqref{e.alphamixing} for~$\phi = \indc_{[0,1]}$ is equivalent to the unit range of dependence condition~\eqref{e.unitrange}. 

\smallskip

We will show that, for each~$A,\alpha \in (0,\infty)$ and~$k\in\N$, there exists~$C(d,k)<\infty$ such that 
\begin{equation} 
\label{e.alphamix.to.CFS}
\alpha\mathsf{Mix}\left(t\mapsto At^{-\alpha} \right)
\implies \CFS\left( \frac{kd}{\alpha+kd} ,t\mapsto C(1+A)t^{-2k}\right)
\,.
\end{equation}
Let~$\beta \in (0,1]$ and~$m,n\in\N$ with~$\beta m<n<m$, a family~$\{ X_z \,:\, z\in 3^n\Zd\cap \cu_m\}$ of random variables satisfying~\eqref{e.CFS.ass} and let~$Z$ be as in~\eqref{e.Z}. We intend to use the finite moment method.
That is, we will estimate~$\E\left[ Z^{2k} \right]$ for~$k\in\N$, and apply Chebyshev to get a bound on~$\P\left[ Z>t \right]$. This works essentially the same way as the exponential moment method; dealing directly with finite moments is just somewhat messier. 

\smallskip

As in the case of unit range dependence, we split the sum on the right side of~\eqref{e.Z}:
\begin{equation}
Z = \sum_{z'\in 3^n\Zd \cap \cu_{n+1}} \sum_{z\in (z'+3^{n+1}\Zd)\cap \cu_m} X_z.
\end{equation}
There are~$3^d$ elements of~$3^n\Zd \cap \cu_{n+1}$ and the distinct elements of~$z\in (z'+3^{n+1}\Zd)\cap \cu_m$ are separated by a distance of at least~$3^n$. We have that, for~$C(k)<\infty$, 
\begin{equation}
\E\bigl[ Z^{2k} \bigr] 
\leq 
C\sum_{z'\in 3^n\Zd \cap \cu_{n+1}}
\E \biggl[ 
\biggl( \sum_{z\in (z'+3^{n+1}\Zd)\cap \cu_m} X_z \biggr)^{\!\! 2k} 
\biggr]
\,.
\end{equation}
Since the estimate of each of the~$3^d$ terms in the sum above is the same, we will focus on estimating the case~$z'=0$, that is, we try to bound~$\E \left[ \hat{Z}^{2k} \right]$, where 
\begin{equation}
\hat{Z}:= \sum_{z\in  3^{n+1}\Zd \cap \cu_m} X_z
\,.  
\end{equation}
We have 
\begin{equation*} \label{}
\E \bigl[ \hat{Z}^{2k} \bigr]
=
\sum_{z_1,\ldots,z_{2k} \in 3^{n+1}\Zd\cap \cu_m} \E \biggl[ \prod_{i=1}^{2k} X_{z_i} \biggr]
\,.  
\end{equation*}
Let~$H_j$ be the collection of~$2k$-tuples~$z = (z_1,\ldots,z_{2k}) \in \left( 3^{n+1}\Zd\cap \cu_m \right)^{2k}$ with exactly~$j$ distinct entries. For~$z \in H_j$, let~$z_{\mathrm{dist}} \in \left( 3^{n+1}\Zd\cap \cu_m \right)^{j}$ be a~$j$-tuple of distinct members of~$z$ and, for~$z \in H_j$ and~$y \in z_{\mathrm{dist}}$, let~$q_z(y)$ denote the multiplicity of~$y$ in~$z$. The~$\alpha$-mixing condition~\eqref{e.alphamixing} implies, for each~$z \in H_j$,
\begin{equation} 
\label{e.howtomixalphaly}
\biggl| 
\E \biggl[ \prod_{y\in  z}  X_{y} \biggr]
-
\prod_{y \in z_{\mathrm{dist}}} \E \biggr[ X_{y}^{q_z(y)} \biggr]
\biggr|
\leq 
4 j  \phi(3^n) \E \left[  \prod_{y\in z} |  X_{y} | \right]
\leq 
4 j  \phi(3^n)
\,.
\end{equation}
Indeed,~\eqref{e.howtomixalphaly} is a consequence of induction on the following simple covariance estimate, valid for any pair~$Y_1,Y_2$ of random variables with~$|Y_i| \leq 1$ which are measurable with respect to~$U_1,U_2\subseteq\Rd$, respectively, with~$\dist(U_1,U_2)> D$:
\begin{align*}
\cov\left[Y_1,Y_2 \right]
&
=
\int_{-1}^1\int_{-1}^1
\left( \P \left[ Y_1>t_1 \ \mbox{and} \ Y_2>t_2 \right] 
-
\P \left[ Y_1>t_1\right]\P \left[ Y_2>t_2\right]
\right)\,dt_1\,dt_2
\\ & 
\leq 
\int_{-1}^1\int_{-1}^1 \phi(D) \,dt_1\,dt_2 = 4\phi(D)
\,.
\end{align*}
For details, see~\cite[Appendix A]{AM}.
Inserting~\eqref{e.howtomixalphaly} into the display preceding it, we get 
\begin{equation*} \label{}
\E \bigl[ \hat{Z}^{2k} \bigr]
\leq 
\sum_{j=1}^{2k} \sum_{z \in H_j} \biggl( \prod_{y \in z_{\mathrm{dist}}} \bigl | \E \bigl[ X_{y}^{q_z(y)} \bigr] \bigr| + 4 j\phi(3^n) \biggr) \, . 
\end{equation*}
If~$z \in H_j$ for~$j>k$, then at least one entry of~$z$ must have multiplicity one, that is~$q_z(y) = 1$ for some~$y \in z_{\mathrm{dist}}$, and thus~$\prod_{y \in z_{\mathrm{dist}}}  \E \bigl[  X_{y}^{q_z(y)} \bigr]   =0$ since~$\E[X_y] = 0$ for every~$y \in z$. If, on the other hand,~$z \in H_j$ for some~$j\leq k$, we use the fact that~$|X_{z_i}| \leq 1$ to crudely bound  the product as~$\prod_{y \in z_{\mathrm{dist}}} \big| \E \bigl[  X_{y}^{q_z(y)} \bigr] \big|  \leq 1$.
We deduce that 
\begin{equation*} \label{}
\E \bigl[ \hat{Z}^{2k} \bigr]
\leq 
\sum_{j=1}^{k} \left| H_j \right| + 4 \phi(3^n) \sum_{j=1}^{2k} j \left| H_j \right|
\,.
\end{equation*}
Estimating~$|H_j|$ is a combinatorics exercise. Using Stirling's approximation~$j!\geq j^j \exp(-j)$ and setting~$N:=\left| 3^{n+1}\Zd\cap \cu_m\right| = 3^{d(m-n-1)}$, we find that
\begin{equation*} \label{}
\left| H_j \right| 
\leq \frac1{j!} N^jj^{2k}
\leq \exp(j) N^j j^{2k-j}
\,.
\end{equation*}
Assuming that~$N>2k$, we get that 
\begin{equation*} \label{}
\begin{aligned}
&
\sum_{j=1}^k |H_j| 
\leq 
(e N k)^k \sum_{j=1}^k \Bigl(  \frac{e N }{j}   \Bigr)^{j-k} 
\leq (4 Nk)^k,
\quad\mbox{and} 
\\ & 
\sum_{j=1}^{2k} j |H_j| 
\leq 
(e N)^{2k} \sum_{j=1}^{2k} \Bigl(  \frac{e N }{j}   \Bigr)^{j-2k}
\leq (4 N)^{2k}. 
\end{aligned} 
\end{equation*}
Combining the above, we arrive at
\begin{equation*} \label{}
\E \bigl[ \hat{Z}^{2k} \bigr]
\leq 
(4Nk)^k 
+ 
4 (4N)^{2k} \phi(3^n)
\,.
\end{equation*}
Now we apply Chebyshev's inequality to obtain, for every~$t >0$, 
\begin{equation} 
\label{e.cheby.up}
\P \bigl[  \hat{Z} > t \bigr]
\leq 
t^{-2k} \E \bigl[ \hat{Z}^{2k} \bigr] 
\leq 
\Bigl(\frac{4Nk}{t^2}\Bigr)^k 
+ 
4 \Bigl( \frac{4N}{t} \Bigr)^{2k} \phi(3^n)
\,.
\end{equation}
We next optimize the right side of~\eqref{e.cheby.up} for~$cN^{\nicefrac12} \leq t \leq N$, subject also to the constraint~$N>2k$. 
Observe that, in the particular case of a unit range of dependence in which~$\phi=\indc_{[0,1]}$, the second term on the right side of~\eqref{e.cheby.up} vanishes, and we may therefore take~$k:= \frac{t^2}{30N}$ (which we note is smaller than~$\frac12N$) to recover~\eqref{e.unitrange.result},  albeit with a much smaller~$c$ and after including a prefactor constant~$C$ in front of the exponential. 

\smallskip

In the case~$\phi(s) = As^{-\alpha}$, we think of~$k$ as being fixed and impose the restrictions~$n \geq \beta m$ and~$\beta \geq dk( \alpha+dk)^{-1}$, or equivalently~$\alpha \beta \geq d(1-\beta) k$, to get
\begin{equation*} \label{}
\phi(3^n ) 
\leq 
A 3^{-\beta m\alpha} 
\leq 
A 3^{-d(1-\beta)  mk} 
\leq 
3^{dk} A N^{-k}. 
\end{equation*}
The second term  in~\eqref{e.cheby.up} is then estimated, for~$C(d)<\infty$, by
\begin{align*}
4 \Bigl( \frac{4N}{t} \Bigr)^{2k} \phi(3^n)
\leq
A \Bigl( \frac{C N}{t^2} \Bigr)^{k}. 
\end{align*}
We deduce that, for a constant~$C(d)<\infty$,
\begin{equation*} \label{}
\P \bigl[ \hat{Z} > t \bigr] 
\leq 
(1+A) \Bigl( \frac{CkN}{t^2} \Bigr)^{k}, \quad \forall t>0. 
\end{equation*}
This completes the proof of the implication~\eqref{e.alphamix.to.CFS}.

\smallskip

The analysis in this example can be extended relatively straightforwardly in various situations. 
For instance, if the modulus~$\phi$ has a faster-than-polynomial decay, say~$\phi(t) = \exp( 1 - t^{\gamma} )$ for~$\gamma\in(0,\infty)$, then one can obtain a substantial strengthening of the~$\CFS$ condition with corresponding stretched exponential moment bounds on the fluctuations. We leave this to the reader.

\subsubsection{Logarithmic Sobolev inequality~$(\LSI)$}
\label{ss.LSI}

We next consider the case that the probability measure~$\P$ satisfies a certain form of the logarithmic Sobolev inequality. Given parameters~$\beta\in [0,1)$ and~$\rho\in (0,\infty)$, we say that~$\P$ satisfies~$\LSI(\beta,\rho)$ if, for every~$m,n\in\N$ with~$\beta m  < n < m$ and~$\F(\cu_m)$--measurable random variable~$X$, we have 
\begin{equation} 
\label{e.LSI.ass}
\Ent\left[  X^2 \right]
\leq 
\frac1\rho\,
\E \Biggl[
\sum_{z\in 3^n\Zd \cap \cu
_m}
\bigl| \partial_{\a(z+\cu_n)} X
\bigr|^2
\Biggr],
\end{equation}
where~$\Ent$ is the entropy functional defined for nonnegative random variables~$Y$ by 
\begin{equation*} \label{}
\Ent\left[  Y \right]:=
\E \left[  Y\log Y \right]
-\E \left[  Y \right] \log \E \left[  Y \right].
\end{equation*}
Note that this assumption is essentially the same\footnote{See~{\tt https://arxiv.org/pdf/1409.2678v3.pdf}. The condition stated there allowed for a partition~$\{ D \}$  of~$\Rd$ into more general sets satisfying
\begin{equation*} \label{}
\diam(D) \simeq \dist(D,0)^{\beta},
\end{equation*}
whereas our sets are triadic cubes of the same size. The fact that the subcubes are the same size and not growing coarser far from the origin is allowed by the fact that we consider random variables that are measurable with respect to a fixed macroscopic cube (which fixes the largest macroscopic scale). Therefore, our partition is actually coarser, and thus, our inequality is weaker than the assumption in~\cite{GNO2}, if you restrict to partitions of triadic cubes. On the other hand, we use triadic cubes rather than general sets merely for convenience and simplicity (not because our arguments depend on it) and, regardless, this choice does not restrict the applicability of the assumption in any important way.} as the one considered in the third arXiv version of~\cite{GNO2}. 

\smallskip

We claim that 
\begin{equation} 
\label{e.LSI.okay}
\LSI(\beta,\rho) \implies \CFS\bigl(\beta, \Gamma_2\bigl(c \rho^{\nicefrac12}  \cdot  \bigr)  \bigr).
\end{equation}
The proof of~\eqref{e.LSI.okay} is a routine application of the \emph{Herbst argument}, which is the name given to the standard method of combining entropy inequalities with the exponential moment method to obtain (Gaussian-type) concentration inequalities. 

\smallskip

To prove~\eqref{e.LSI.okay}, select~$m,n\in\N$ with~$\beta m<n<m$ and a family~$\{ X_z \,:\, z\in 3^n\Zd\cap \cu_m\}$ of random variables satisfying~\eqref{e.CFS.ass}. By~\eqref{e.CFS.ass} and~\eqref{e.LSI.ass}, for~$t\in\R$ and with~$Z$ defined by~\eqref{e.Z},
we compute
\begin{align*}
\Ent\left[  \exp (tZ)  \right]
& 
\leq 
\frac1\rho\,
\E \Biggl[
\sum_{z\in 3^n\Zd \cap \cu_m}
\bigl| \partial_{\a(z+\cu_n)} \exp\left( \tfrac12tZ \right)
\bigr|^2
\Biggr]
\\ & 
=
\frac1\rho\,
\E \Biggl[  \exp\left( tZ \right) 
\sum_{z\in 3^n\Zd \cap \cu_m}
\frac1{4}t^2 
\bigl| \partial_{\a(z+\cu_n)} X_z \bigr|^2
\Biggr]
\leq 
\frac1{4\rho}t^2 \frac{|\cu_m|}{|\cu_n|}
\E \bigl[  \exp\left( tZ \right) \bigr].
\end{align*}
We can express this inequality in terms of~$f(t):= \log \E \left[  \exp\left( t Z  \right) \right]$ as
\begin{equation*} \label{}
\frac{d}{dt} \left(\frac{f(t)}{t} \right) 
= 
\frac{f'(t)}{t} - \frac{f(t)}{t^2} 
\leq 
\frac1{4\rho} \frac{|\cu_m|}{|\cu_n|}.
\end{equation*}
Using~$f(0) = 0$ and~$f'(0) = 0$, we find by integration that~$f(t) \leq \frac1{4\rho}t^2 \frac{|\cu_m|}{|\cu_n|}$. That is, 
\begin{equation*}
\E \left[  \exp \left( t Z \right) \right]
\leq 
\exp\left( \frac1{4\rho} t^2 \frac{|\cu_m|}{|\cu_n|} \right) 
\quad 
\mbox{ for every~$t \in \R$}\,.
\end{equation*}
Thus, by Chebyshev's inequality, 
\begin{equation}
\P \left[ |Z| > \lambda \right] 
\leq 
\inf_{t>0} 
2 \exp\left( -\lambda t + \frac1{4 \rho} t^2 \frac{|\cu_m|}{|\cu_n|} \right)
=
2\exp \left( - \rho \lambda^2 \frac{|\cu_n|}{|\cu_m|} \right)
\,.
\end{equation}
This completes the proof of~\eqref{e.LSI.okay}.

\smallskip

Notice that the argument above does not require the full power of the assumed LSI inequality~\eqref{e.LSI.ass}. First, we may replace the second moment of~$|\partial_{\a(z+\cu_n)} X|$ appearing on the right side by its~$L^\infty(\Omega,\P)$ norm squared. This is a considerable weakening of the inequality~\eqref{e.LSI.ass}, and yet it still implies~$\CFS(\beta,\Gamma_2(c\rho^{\nicefrac12}\cdot))$. The second and more important point is that we do not need that~\eqref{e.LSI.ass} holds for general nonlinear random variables, rather just for sums. Thus~$\CFS(\beta,\Psi)$ is indeed \emph{significantly more general} than the ``nonlinear'' LSI-type concentration inequality~\eqref{e.LSI.ass}.

\smallskip 

We next modify the Herbst argument presented above to show that~$\LSI$ implies exponential concentration for random variables with Malliavin derivatives with exponential moments.

\begin{lemma}[Exponential moments under~$\LSI(\beta,\rho)$]
\label{l.LSI.moments}
Let~$\P$ satisfy~$\LSI(\beta,\rho)$ and~$m,n\in\N$ with~$\beta m < n < m$  and~$X$ be an~$\F(\cu_m)$--measurable random variable. Define
\begin{equation}
\label{e.Zderv.again}
Z := \sum_{z\in3^n\Zd\cap \cu_m} |  \partial_{\a(z+\cu_n)} X|^2.
\end{equation}
Then, for every~$q\in [1,\infty)$, 
\begin{align}
\label{e.LSI.momentbound}
\E \left[ \big| X - \E[X] \big|^{2q} \right]^{\frac1{q}}
\leq 
\frac{q}{\rho} \E \bigl[  Z^{q} \bigr]^{\frac1{q}}.
\end{align}
Moreover, for every~$s\in (0,\infty)$ and~$\theta>0$, 
\begin{align}
\label{e.LSI.expmoments}
\E \left[ \exp\left( \frac{1+s}{8s}\left( \frac{ \left| X - \E[X] \right|}{\theta} \right)^{\!\!\frac{2s}{1+s}} \right) \right]
\leq 
4
\E \left[ \exp\left( \frac1{s} \left( \frac{Z}{\rho\theta^2}  \right)^{\!\!s}\, \right) \right].
\end{align}
\end{lemma}
\begin{proof}
Suppose without loss of generality that~$\E[X]=0$ and~$\E[X^2] = 1$. 
Define
\begin{equation*}
f(t):= \E \bigl[ X^{2t} \bigr]
\end{equation*}
and observe that
\begin{equation*}
f'(t) =  \E \bigl[  X^{2t} \log X^2\bigr] = 
\frac1t \Ent\bigl[  X^{2t} \bigr]
+\frac1t f(t) \log f(t) \,.
\end{equation*}
Applying the~$\LSI(\beta,\rho)$ assumption and the H\"older inequality, we therefore obtain, for~$t \geq 1$, 
\begin{align}
\label{e.getmomentsLSI}
f'(t)
\leq
\frac{t}\rho \E\bigl[  X^{2t-2} Z \bigr]
+\frac1t f(t) \log f(t) 
& 
\leq 
\frac{t}\rho \E \bigl[  X^{2t} \bigr]^{\frac{t-1}{t}} \E \bigl[  Z^t \bigr]^{\nicefrac1t}  +\frac1t f(t) \log f(t) 
\notag \\ & 
= 
\frac{t}\rho f(t)^{\frac{t-1}{t}} \E \bigl[  Z^t \bigr]^{\nicefrac1t} +\frac1t f(t) \log f(t) \,.
\end{align}
Rearranging this yields
\begin{equation*}
\frac{d}{dt} \Bigl( \E \bigl[ X^{2t} \bigr]^{\nicefrac1t} \Bigr) 
=
\frac{d}{dt} \bigl( f(t)^{\frac1t} \bigr) 
=
\frac1t f(t)^{\frac1t-1} f'(t)
- \frac 1{t^2} f(t)^{\frac1t} \log f(t) 
\leq \frac1\rho\E \bigl[  Z^t \bigr]^{\nicefrac1t}.
\end{equation*}
Integrating this inequality over~$t\in [1,q]$ for~$q \geq 1$, we get
\begin{equation*}
\E \bigl[  X^{2q} \bigr]^{\nicefrac1q} 
\leq 
\E \bigl[  X^{2} \bigr] 
+ 
\frac{q-1}{\rho} \E \bigl[  Z^{q} \bigr]^{\nicefrac1q}
\leq 
\frac1\rho\E \bigl[  Z \bigr]
+
\frac{q-1}{\rho} \E \bigl[  Z^{q} \bigr]^{\nicefrac1q}
\leq 
\frac{q}{\rho} \E \bigl[  Z^{q} \bigr]^{\nicefrac1q},
\end{equation*}
where we used~$\LSI(\beta,\rho)$ once more to get~$\E[X^2]\leq \rho^{-1}\E[Z]$. To see this, apply~$\LSI(\beta,\rho)$ for~$Y=1 + \ep X$, assuming momentarily that~$X$ is bounded and mean zero, and then pass~$\ep \to 0$ to get the desired bound by appealing to density. This completes the proof of~\eqref{e.LSI.momentbound}. 

\smallskip

To see that~\eqref{e.LSI.momentbound} implies~\eqref{e.LSI.expmoments}, use Chebyshev's inequality to compute, for~$t>0$, 
\begin{align*}
\P \bigl[ |X| > t \bigr]
\leq
t^{-2q} \E \bigl[ |X|^{2q} \bigr]
&
\leq
\Bigl( \frac{t}{\theta} \Bigr)^{\!\!-2q} 
q^{q} \,
\E \biggl[ \Bigl( \frac{Z}{\rho\theta^2} \Bigr)^{q} \biggr]\,.
\end{align*}
Using the elementary inequality 
\begin{equation*}
x^q \leq q^{q/s} \exp\Bigl( \frac1s x^s \Bigr)\,, \quad \forall x \geq0\,, q \geq s\,,
\end{equation*}
we obtain
\begin{equation*}
\P \bigl[ |X| \geq t \bigr]
\leq
\Bigl( \frac{t}{\theta} \Bigr)^{\!\!-2q} 
q^{\frac{q(1+s)}{s}}
\E \left[ \exp\Bigl( \frac1{s} \Bigl( \frac{Z}{\rho\theta^2}  \Bigr)^{s} \Bigr) \right] 
\,.
\end{equation*}
Optimizing over~$q$ leads to the choice~$q = \frac12 (t / \theta)^{2s/(1+s)}$, which is a valid choice (remember the constraint~$q\geq s$) as long as~$(t / \theta)^{2s/(1+s)} \geq 2s$, which is valid if~$t / \theta \geq s$. We obtain
\begin{equation*}
\P \bigl[ |X| > t \bigr]
\leq
\exp\biggl( 
-\frac{\log 2}{2}
\frac{1+s}{s} \Bigl( \frac{t}{\theta} \Bigr)^{\frac{2s}{1+s}}
\biggr) \,
\E \left[ \exp\Bigl( \frac1{s} \Bigl( \frac{Z}{\rho\theta^2}  \Bigr)^{s} \Bigr) \right]
\,.
\end{equation*}
Integration in~$t$ now yields the result. 
\end{proof}

\subsubsection{Spectral gap inequality~$(\mathsf{SG})$}
\label{ss.SG}

We next consider the case that the probability measure~$\P$ satisfies a so-called spectral gap inequality, also called an Efron-Stein inequality. This is a weaker variant of the LSI assumption from the previous example. 

\smallskip

Given~$\beta\in (0,1]$ and~$\rho\in (0,\infty)$, we say that~$\P$ satisfies~$\SG(\beta,\rho)$ if, for every~$m,n\in\N$ with~$\beta m  < n < m$ and~$\F(\cu_m)$--measurable random variable~$X$, we have 
\begin{equation} 
\label{e.SG.ass}
\var\left[  X \right]
\leq 
\frac1\rho\,
\E \Biggl[
\sum_{z\in 3^n\Zd \cap \cu_m}
\left| \partial_{\a(z+\cu_n)} X
\right|^2
\Biggr]. 
\end{equation}
By considering~$Y = 1 + \ep X$ and assuming that~$X$ is bounded, and passing then~$\ep \to 0$ and appealing to density, we see that
\begin{equation*} \label{}
\LSI(\beta,\rho) \implies \SG(\beta,\rho).
\end{equation*}
It is well-known (see, for instance,~\cite{L}) that a spectral gap assumption like~\eqref{e.SG.ass} should imply an exponential concentration inequality. In contrast, we saw above that the stronger LSI assumption gave us a stronger Gaussian-type normal concentration inequality. Therefore we should expect---and indeed we will show---that, for a universal~$c<\infty$, 
\begin{equation} 
\label{e.SG.okay}
\SG(\beta,\rho) \implies \CFS\bigl(\beta,\Gamma_1\bigl(c\rho^{\nicefrac12} \cdot \bigr)  \bigr).
\end{equation}
To prove~\eqref{e.SG.okay}, as in~\cite{AiS} and~\cite[Section 2]{L}, the idea is to follow the Herbst argument computation in the previous section and to see what we get if we use~$\SG(\beta,\rho)$ instead of~$\LSI(\beta,\rho)$. 

\smallskip

As before, we pick~$m,n\in\N$ with~$\beta m<n<m$ and a family~$\{ X_z \,:\, z\in 3^n\Zd\cap \cu_m\}$ of random variables satisfying~\eqref{e.CFS.ass}. We combine~\eqref{e.CFS.ass} and~\eqref{e.SG.ass} to get, for~$t\in\R$,
\begin{align*}
\E \left[ \exp(tZ) \right] - \E\left[ \exp\left(\tfrac12tZ\right) \right]^2 
&
= 
\var\left[  \exp\left(\tfrac12tZ\right)   \right]
\\ & 
\leq 
\frac1\rho\,
\E \Biggl[
\sum_{z\in 3^n\Zd \cap \cu_m}
\left| \partial_{\a(z+\cu_n)} \exp\left( \tfrac12tZ \right)
\right|^2
\Biggr]
\\ & 
=
\frac1\rho\,
\E \Biggl[  \exp\left( tZ \right) 
\sum_{z\in 3^n\Zd \cap \cu_m}
\frac1{4\rho}t^2 
\left| \partial_{\a(z+\cu_n)} X_z \right|^2
\Biggr]
\\ & 
\leq 
\frac1{4\rho}t^2 \frac{|\cu_m|}{|\cu_n|}
\E \left[  \exp\left( tZ \right) \right].
\end{align*}
Putting~$f(t):= \E \left[  \exp\left( t Z \right) \right]$ as before, we may write this as 
\begin{equation*} \label{}
f(t) - f\left(\tfrac t2\right)^2
\leq 
\frac1{4\rho}t^2 \frac{|\cu_m|}{|\cu_n|} f(t).
\end{equation*}
Therefore, 
\begin{equation*} \label{}
f(t) \leq \biggl( 1 - \frac1{4\rho}t^2 \frac{|\cu_m|}{|\cu_n|} \biggr)^{\!\!-1} f \left(\tfrac t2\right)^2.
\end{equation*}
Iterating this inequality starting from~$t_0 := \bigl( \rho \frac{|\cu_n|}{|\cu_m|}  \bigr)^{\nicefrac12}$, we obtain
\begin{equation*} \label{}
f\left( t_0 \right) 
\leq
f\left( 2^{-N}t_0 \right)^{2^N}
\prod_{n=0}^{N-1} \biggl( 1 - \frac1{4^{n+1}} \biggr)^{\!\!-2^n} .
\end{equation*}
Since~$Z$ is centered, it is clear that~$f\left(2^{-N} t_0 \right)^{2^N} \to 1$ as~$N\to \infty$. Moreover, the product stays bounded uniformly in~$N$ since
\begin{align*} \label{}
\log \prod_{n=0}^{N-1} \left( 1 - \frac1{4^{n+1}} \right)^{\!\!-2^n}
&
\leq
\sum_{n=0}^{N-1} 2^n
\left| \log \left( 1 - \frac1{4^{n+1}} \right) \right|
\leq
C \sum_{n=0}^{N-1} 2^n \cdot \frac1{4^{n+1}}
\leq C \sum_{n=0}^\infty 2^{-n} < \infty. 
\end{align*}
We deduce that~$f\left(t_0\right) \leq C$, that is, 
\begin{equation*} \label{}
\E \biggl[ \exp\biggl( \biggl( \rho \frac{|\cu_n|}{|\cu_m|}  \biggr)^{\!\!\nicefrac12} Z \biggr) \biggr] \leq C. 
\end{equation*}
This implies that~$Z = \O_{\Psi}\Big (\bigl( \frac{|\cu_m|}{|\cu_n|}\bigr)^{\nicefrac12}   \Bigr)$ for~$\Psi(t)=\Gamma_2\bigl(c\sqrt{\rho}t \bigr)$ and completes the proof of~\eqref{e.SG.okay}.

\begin{lemma}[Exponential moments under~$\SG(\beta,\rho)$] 
\label{l.SG.moments}
Assume that~$\P$ satisfies~$\SG(\beta,\rho)$ with~$\beta \in (0,1)$ and~$\rho>0$. Let~$m,n\in\N$ with~$\beta m < n < m$  and~$X$ be an~$\F(\cu_m)$--measurable random variable. Define
\begin{equation}
\label{e.Zderv.again.2}
Z := \sum_{z\in3^n\Zd\cap \cu_m} |  \partial_{\a(z+\cu_n)} X|^2.
\end{equation}
Then, for every~$q\in [1,\infty)$, 
\begin{align}
\label{e.SG.momentbound}
\E \left[ \big| X - \E[X] \big|^{2q} \right]^{\frac1{q}}
\leq 
\frac{q^2}{\rho} \E \bigl[  Z^{q} \bigr]^{\frac1{q}}.
\end{align}
Moreover, for every~$s\in (0,\infty)$ and~$\theta>0$, 
\begin{align}
\label{e.SG.expmoments}
\E \left[ \exp\Biggl( \frac{1+2s}{16s}\biggl( \frac{ \left| X - \E[X] \right|}{\theta} \biggr)^{\!\!\frac{2s}{1+2s}} \Biggr) \right]
\leq 
4 \E \left[ \exp\biggl( \frac1{s} \left( \frac{Z}{\rho\theta^2}  \biggr)^{\!\!s} \right) \right].
\end{align}
\end{lemma}
\begin{proof}
We may suppose that~$\E[X]=0$. 
Fix~$q\in (1,\infty)$ and compute, using~$\SG(\beta,\rho)$, 
\begin{align*}
\E \bigl[  X^{2q} \bigr] - \E \bigl[  X^{q} \bigr]^2 
=
\var \bigl[  X^q \bigr]
& 
\leq
\frac1\rho 
\E \Biggl[
\sum_{z\in 3^n\Zd \cap \cu_m}
\left| \partial_{\a(z+\cu_n)} \bigl( X^q \bigr)
\right|^2
\Biggr]
=
\frac1\rho \E \left[ q^2 X^{2q-2} Z \right].
\end{align*}
By H\"older's inequality, we obtain
\begin{equation*}
\E \bigl[  X^{2q} \bigr] - \E \bigl[  X^{q} \bigr]^2 
\leq
\frac{q^2} \rho 
\E \bigl[  X^{2q} \bigr]^{\frac{q-1}{q}} \E \bigl[  Z^{q} \bigr]^{\frac1q}.
\end{equation*}
This is essentially a finite difference version of~\eqref{e.getmomentsLSI}, except with a factor of~$q^2$ instead of~$q$ on the right side. A simple induction argument now yields~\eqref{e.SG.momentbound}. 
The proof that~\eqref{e.SG.momentbound} implies~\eqref{e.SG.expmoments} is almost the same as the proof that~\eqref{e.LSI.momentbound} implies~\eqref{e.LSI.expmoments}, so we omit it.
\end{proof}

\subsection{Generalized concentration for sums}
\label{s.closure}

In this section, we state an extension of the~$\CFS$ conditions to more general random fields and with more general norms.

\smallskip

\subsubsection{A generalized~$\CFS$ condition}
\label{ss.CFS.gen}

Let~$K \subseteq \R^{N}$ be a nonempty subset of Euclidean space in dimension~$N\geq 1$. Let~$\Upsilon$ denote the collection of (Lebesgue) measurable maps~$\b:\Rd \to K$. For each Borel set~$U\subseteq \Rd$ Let~$\mathcal{G}(U)$ be the~$\sigma$--algebra on~$\Upsilon$ generated by the random variables \begin{equation*}
\b \mapsto \int_{U} \b_i(x) \varphi (x) \,dx \,, \quad i \in \{1,\ldots,N\}, \ \varphi\in C^\infty_{\mathrm{c}}(U)\,.
\end{equation*}
We set~$\mathcal{G}:=\mathcal{G}(\Rd)$.
The translation group~$\{ T_y \}_{y\in\Rd}$ acts on~$\Upsilon$ and~$\mathcal{G}$ in the obvious way, namely~$T_y\b := \b(\cdot+y)$. 
The canonical element of~$\Upsilon$ is denoted by~$\b$. 
We say that a probability measure~$\P$ on~$(\Upsilon,\mathcal{G})$ is~$\Zd$--stationary if~$\P\circ T_z = \P$ for every~$z\in\Zd$. 

\smallskip

We extend the definition of the Malliavin derivative in the following way. 
Let~$U\subseteq\Rd$ be a bounded Borel subset and~$\vertiii \cdot \vertiii_U$ be a norm on the set of Lebesgue measurable maps from~$U \to \R^N$. 
For each random variable~$X$ on~$\Upsilon$, we define the Malliavin derivative~$|\partial_{\b(U)} X|$ with respect to~$\b\vert_U$, and with variations taken with respect to~$\vertiii \cdot \vertiii_U$, by
\begin{multline} 
\label{e.Mall.Maul.yes}
\left| \partial_{\b(U)} X\right| (\b) 
:=
\limsup_{t\to 0} 
\frac1{2t} 
\sup
\Bigl\{ X(\b_1) - X(\b_2) 
\,:\,
\b_1,\b_2 \in\Upsilon, \ \b = \b_1 = \b_2 \ \mbox{in} \ \Rd\setminus U, 
\\
\vertiii (\b - \b_i) \vertiii_U \leq t  \ \mbox{for} \ i\in\{1,2\}
\Bigr\}.
\end{multline}
Note that~$|\partial_{\b(U)} X|$ depends on the norm~$\vertiii \cdot \vertiii_U$, even though we do not explicitly display this in the notation. 

We define the generalized concentration for sums conditions, which we also denote by~$\CFS(\beta,\Psi)$ and~$\CFS(\beta,\Psi,\Psi^\prime,\gamma)$, in a completely analogous way to Definitions~\ref{d.CFS} and~\ref{d.CFS.strong}. 
For completeness, we state the first definition anyway, though it is nearly a copy-paste of Definition~\ref{d.CFS}. We extend Definition~\ref{d.CFS.strong} to this setting as well, but we do not restate it. 

\begin{definition}[Concentration for sums~$\CFS(\beta,\Psi)$, generalized]
\label{d.CFS.general}
Let~$\beta \in \left[0,1\right)$ and~$\Psi:[1,\infty) \to [0,\infty)$ be an increasing function satisfying~\eqref{e.Young.growth} and~\eqref{e.weak.triangle.ass}. Suppose that~$\vertiii\cdot \vertiii_n$ is a norm on the set of Lebesgue measurable maps from~$\cu_n \to \R^N$, which we extend to be a norm on~$z+\cu_n$, for every~$z\in\Zd$, by translation. 
A probability measure~$\P$ on~$(\Upsilon,\mathcal{G})$ \emph{satisfies~$\CFS(\beta,\Psi)$} provided that: 
\begin{itemize}
\item
For every~$m,n\in\N$ with~$\beta m < n<m$, and  family~$\{ X_z \,:\, z\in 3^n\Zd\cap \cu_m\}$ of random variables satisfying, for every~$z\in\Zd$,
\begin{equation} 
\label{e.CFS.ass.gen}
\left\{
\begin{aligned}
& \E\left[ X_z \right]=0 \,, 
\\ & 
\left| X_z \right| \leq 1\,,
\\ &   
\left| \partial_{\b(z+\cu_n)} X_z\right| \leq 1 \,,
\\
& X_z \quad \mbox{is~$\mathcal{G}(z+\cu_n)$--measurable} \,, \\
\end{aligned}
\right.
\end{equation}
where the Malliavin derivative is defined with respect to~$\vertiii \cdot \vertiii_{n}$ as in~\eqref{e.Mall.Maul.yes} above,
the random variable~$\overline{X}$ defined by
\begin{equation} 
\label{e.Z.norm.gen}
\overline{X}  :=  \avsum_{z\in 3^n\Zd\cap \cu_m}  X_z
\end{equation}
satisfies the bound
\begin{equation}
\label{e.CFS.gen}
\overline{X}
\leq
\O_{\Psi}
\bigl( 3^{-\frac d2(m-n)} \bigr).
\end{equation}
\end{itemize}
\end{definition}
We need the flexibility of this more general~$\CFS$ condition in Section~\ref{ss.rhs} when considering equations with a divergence-form right-hand side. There we work with a pair~$\b = (\a,\f)$ taking values in~$\R^{d\times d} \times \Rd$, and we use the norm
\begin{equation}
\label{e.triple.norm.sec6}
\vertiii (\a,\f) \vertiii_n:= \| \a \|_{L^\infty(\cu_n)} + \| \f \|_{\underline{L}^2(\cu_n)}
\end{equation}
and take~$\CFS(\beta,\Psi)$ as our assumption. Thus, we are more lenient with respect to variation in the forcing vector field. 
Later, in Section~\ref{ss.rhs.optimal}, we will consider the renormalization theory for this equation, and we will assume strong conditions of the form~$\CFS(\beta,\Psi,\Psi^\prime,\gamma)$ in which the norm is slightly strengthened to an~$\ell^\infty$-$L^2$ type norm of the form:
\begin{equation}
\label{e.triple.your.bar.meso}
\vertiii (\a,\f) \vertiii_n:= \| \a \|_{L^\infty(\cu_n)} 
+
\sup_{z\in 3^{k_n} \Zd \cap \cu_n}
\| \f \|_{\underline{L}^2(z+\cu_{k_n})}
\,,
\end{equation}
where~$k_n < n$ represents an appropriate mesoscopic scale, and in our case, we will actually take~$k_n =  \lceil n - (\log n)^{\nicefrac12} \rceil$.

\smallskip

Clearly, the examples given in the previous section can be carried over to this more general~$\CFS$ condition as well. The only catch is that the approximating norms in the case of Section~\ref{ss.AFRD} must be with respect to this new norm~$\vertiii \cdot \vertiii_n$; that is, the left side of an inequality~\eqref{e.AFRD2} in the approximate finite range assumption needs to be with respect to~$\vertiii \cdot \vertiii_m$, not~$\| \cdot\|_{L^\infty(\cu_m)}$. Similarly, the~$\LSI$ and~$\SG$ conditions must be modified to accommodate the slightly altered notion of the Malliavin derivative above. 

\subsection*{Historical remarks and further reading}

The monographs~\cite{Brad,B1,Torq} contain much more information about mixing conditions and quantitative ergodicity for random fields. For more on concentration inequalities like those appearing in Lemmas~\ref{l.LSI.moments} and~\ref{l.SG.moments}, see~\cite{BLM}.

\section{Coarse-graining the coefficient field and iterating up the scales}
\label{s.subadd}

In this chapter, we will provide our first quantitative estimates of the rate of homogenization. Throughout this chapter, we assume that~$\P$ is a~$\Zd$--stationary probability measure on~$(\Omega,\F)$, with the probability space defined as in Section~\ref{ss.probspace}. We will focus here on the special case that~$\a(\cdot)$ is symmetric:
\begin{equation}
\label{e.symm}
\P \bigl[  \a = \a^t \bigr] = 1 \,.
\end{equation}
The theory we develop in this chapter is extended to the general case of possibly nonsymmetric coefficients in Chapter~\ref{s.nonsymm}, below.

\smallskip

As we have seen previously, for instance, in~Proposition~\ref{p.DP} or Theorem~\ref{t.parabolic.GF}, this boils down to quantifying the rate of the limits in~\eqref{e.corrector.qualbound.Hm1} or~\eqref{e.corrector.qualbound.L2} for the first-order correctors (the two sets of limits are equivalent). 
In other words, we need a convergence rate in the ergodic theorem for the stationary random field~$\nabla \phi_e$ and the flux~$\a(e+\nabla \phi_e)$. This requires, of course, some way to quantify the ergodicity of the random field~$\nabla \phi_e$. 

\smallskip

Even if we make a strong quantitative assumption of ergodicity on the coefficient field~$\a(\cdot)$, such as a finite range of dependence, it is by no means obvious how we are to apply this to the gradient corrector field~$\nabla \phi_e$.
If we view a solution such as~$x\mapsto e\cdot x + \phi_e(x)$ as a function of~$\a(\cdot)$, then this function evidently has a complicated structure. It is both highly nonlinear and highly nonlocal and potentially singular. We will have to fight against the possibility that a small perturbation in~$\a(\cdot)$ in a relatively small set would have a relatively huge change in the solution on a large scale.
Therefore, the main task of any theory of quantitative homogenization is to transform, as efficiently as possible, quantitative information on the behavior of the coefficient field into quantitative information on the solutions of the PDE, particularly the correctors.

\smallskip

Our approach in this chapter is to quantify the variational proof in Section~\ref{ss.variational}. 
Rather than pursuing estimates on the infinite-volume correctors directly, we plan to deduce bounds on them as a consequence of estimates on the auxiliary quantities we defined in~\eqref{e.mu.sec2} and~\eqref{e.quad.mu.sec2}, which we call the \emph{coarse-grained coefficients}, denoted by~$\a(U)$. 
We will focus on the limit~\eqref{e.conv.to.ahom}, obtained from applying the subadditive ergodic theorem. 
As we have seen in Section~\ref{ss.variational}, quantifying the speed of convergence of~$\a(\cu_n)$ to the limit~$\ahom$ as~$n\to \infty$ leads to an explicit rate of convergence of the limits~\eqref{e.corrector.qualbound.L2} for the finite-volume correctors.\footnote{This is good enough: we will pass from bounds on the finite-volume correctors to homogenization estimates and, eventually, to bounds on the infinite-volume correctors (see Theorem~\ref{t.optimalstochasticintegrability}).}

\smallskip

From this point of view, the homogenization process is seen as a ``flow'' from small scales to large scales. As the scale increases, the ellipticity contrast in the coarse-grained field decreases until, finally, in the large-scale limit, it converges to~$\ahom$. 
As the coarse-grained field gets close to its limit, one can essentially \emph{linearize the renormalization map}. As we will see, this linearization turns out to be a simple average---the larger scale~$\a(\cu_m)$ is, up to small errors, the average of~$\a(z+\cu_n)$ over its subcubes. At this point, ergodicity can be sharply transferred since we are dealing with a sum of i.i.d.~random variables. 

\smallskip

This perspective centers the coefficients of the equation as the primary protagonist in our story, with the space of solutions of the PDE---and especially any particular linear subspace of the solutions like the correctors---as supporting characters. One of the reasons this is a clever idea is that the \emph{the coarse-grained coefficients have better ergodicity than the correctors}, a theme we will also encounter in Chapter~\ref{s.renormalization}. 

\smallskip

In the first section, we will re-introduce the coarse-grained coefficient field~$\a(U)$, the dual coarse-grained field~$\a_*(U)$, and explore their basic properties. We will perform the key renormalization argument, where we iterate up the scales to obtain a rate of convergence for~\eqref{e.conv.to.ahom}, in Section~\ref{ss.subadd.conv}. This is really the heart of the theory. 
In Section~\ref{ss.det}, we deduce estimates for the rate of the limits in~\eqref{e.corrector.qualbound.L2} for the finite-volume correctors and consequently obtain homogenization error estimates which quantify the qualitative statement of Theorem~\ref{t.qualitative.homogenization}. These will be used in Chapter~\ref{s.regularity} to quantify the minimal scale~$\X$ appearing in the large-scale regularity estimate given in Theorem~\ref{t.C11.sharp}. 

\subsection{Coarse-grained coefficients in the symmetric case}
\label{ss.subadd}

In Section~\ref{ss.variational} we encountered the subadditive quantity~$\mu(U,p)$, defined for each bounded Lipschitz domain~$U\subseteq \Rd$ and~$p\in\Rd$ by
\begin{equation}
\label{e.mu}
\mu(U,p) 
:= \inf_{u \in \ell_p+H^1_0(U)} 
\fint_U \frac12\nabla u \cdot \a\nabla u 
\,.
\end{equation}
In other words,~$\mu(U,p)$ is the volume-normalized energy of the solution of the Dirichlet problem with affine boundary data~$\ell_p$.
As noted in~\eqref{e.quad.mu.sec2}, the quantity~$\mu(U,p)$ is quadratic in~$p$ and  
can, therefore, be written as
\begin{equation}
\label{e.quad.mu}
\mu(U,p) = \frac12 p\cdot \a(U) p
\end{equation}
for a matrix~$\a(U)$, which we call the \emph{coarse-grained coefficient field with respect to~$U$}.

\smallskip

As~$U$ shrinks to a point~$x\in\Rd$, the matrix~$\a(U)$ converges to~$\a(x)$ almost everywhere, by the Lebesgue differentiation theorem. In the limit, as~$U$ becomes a large domain (converging to~$\Rd$), the matrix~$\a(U)$ converges to the deterministic, homogenized matrix~$\ahom$, as we will show. The discrepancy between~$\a(U)$ and~$\ahom$ represents the homogenization error, and we are therefore interested in obtaining a convergence rate for the limit~$\lim_{m\to \infty} \a(\cu_m) = \ahom$.

\smallskip

However, the random matrix~$\a(U)$ is only half of the picture. There is another dual random matrix~$\a_*(U)$ which rightfully competes with~$\a(U)$ for the title of the coarse-grained coefficient field. We introduce, for each bounded Lipschitz domain~$U\subseteq\Rd$ and~$q\in\Rd$, the ``dual'' quantity 
\begin{equation}
\label{e.mustar}
\mu_*(U,q)
:= \sup_{w\in H^1(U)} 
\fint_U \left( -\frac12\nabla w \cdot \a\nabla w + q\cdot \nabla w \right). 
\end{equation}
As~$q\mapsto \mu_*(U,q)$ is also quadratic, we may find a matrix~$\a_*^{-1}(U)$ such that
\begin{equation}
\label{e.quad.mu.star}
\mu_*(U,q) = \frac12 q\cdot \a_*^{-1}(U) q.
\end{equation}
We will see below in Lemma~\ref{l.J.basicprops} that~$\a_*^{-1}(U)$ is indeed invertible, so the notation is not misleading.
The notation is also consistent in the sense that, if~$\a(x)=\ahom$ is constant, then~$\ahom = \a_*(U) = \a(U)$ for every domain~$U$. 
As shown below, we have that~$\a_*(U) \leq \a(U)$ but equality does not generally hold. 
The fact that we have two competing versions of the ``coarse-grained coefficient field'' that are not necessarily equal is because the heterogeneous operator is \emph{not} behaving like a constant-coefficient operator on a finite scale. We think of~$\a(U)$ and~$\a_*(U)$ as upper and lower bounds for what a reasonable notion of the ``coarse-grained coefficient'' could mean, and the difference~$\a(U)-\a_*(U)$ (or perhaps it is better to use their ratio) as a measure of our uncertainty.

\smallskip

It may seem strange that the form of~$\mu(U,p)$ and~$\mu_*(U,q)$ appear to be different. For instance, the former is a minimization problem with fixed boundary data, and the latter is a maximization problem with no boundary data constraint. However, each of these quantities can be written in a dual variational form, with solenoidal vector fields in place of potentials, which then reveals a deeper symmetry between them: we have that 
\begin{equation}
\label{e.mu.dual}
\mu(U,p) 
:= \sup_{\g \in L^2_{\sol}(U)} 
\fint_U \biggl(-\frac12\g \cdot \a^{-1} \g + p\cdot \g \biggr)
\end{equation}
and
\begin{equation}
\label{e.mustar.dual}
\mu_*(U,q) 
:= \inf_{\g \in q + L^2_{\sol,0}(U)} 
\fint_U \frac12\g \cdot \a^{-1} \g \,.
\end{equation}
We leave the proofs of these identities as an exercise to the reader (note that the first identity~\eqref{e.mu.dual} is implicit in the formula~\eqref{e.variational.J} below).

\smallskip 

As the quantities~$\mu$ and~$\mu_*$ are both subadditive, this means that we should expect the random matrices~$\a(U)$ and~$\a_*(U)$ to be decreasing and increasing, respectively, as~$U$ becomes larger and we move up the scales. Thus, roughly speaking,~$\ahom$ lies in between~$\a(U)$ and~$\a_*(U)$, and we can study the convergence of each quantity to~$\ahom$ by controlling their \emph{difference}. It turns out to be much easier to control the difference between these two quantities than to estimate their difference from~$\ahom$ directly.
This is the main idea of the renormalization approach to quantitative homogenization introduced in~\cite{AS}, which introduced the quantity~$\mu_*$ with precisely this purpose in mind. 
We will explore this in the next section, where we estimate the difference between the matrices~$\a(\cu_m)$ and~$\a_*(\cu_m)$ for large~$m$ by using an iteration scheme up the scales.

\smallskip

In the rest of this section, we prepare for this analysis by exploring some basic properties of the subadditive quantities and their matrices~$\a(\cu_m)$ and~$\a_*(\cu_m)$.

\smallskip

The existence of unique (up to additive constants, in the case of~\eqref{e.mustar}) extremal functions for both~\eqref{e.mu} and~\eqref{e.mustar} is a simple consequence of the uniform convexity of the integral functionals. Computing the first variations, we find that the optimizers are both weak solutions of the equation~$-\nabla \cdot \a \nabla u=0$ in~$U$. Indeed, the minimizer of~\eqref{e.mu} is the solution of the Dirichlet problem with boundary data~$\ell_p$. In contrast, the maximizer of~\eqref{e.mustar} is the solution~$w$ of the Neumann problem with the normal flux~$\mathbf{n} \cdot \a\nabla w$ equal to~$\mathbf{n} \cdot q$ on the boundary of~$U$, where~$\mathbf{n}$ is the normal vector on~$\partial U$.   

\smallskip

By testing the definition of~$\mu_*(U,q)$ with the minimizer of~$\mu(U,p)$, we obtain
\begin{equation}
\label{e.fenchel}
\mu(U,p) + \mu_*(U,q) \geq p\cdot q. 
\end{equation}
In the matrix notation, this is equivalent to 
\begin{equation}
\label{e.matrixordering}
\a_*(U) \leq \a(U) .
\end{equation}
We recognize~\eqref{e.fenchel} as resembling the Fenchel-Young inequality in convex analysis. 
However, we emphasize that while we call~$\mu_*$ the ``dual'' quantity to~$\mu$, the function~$q\mapsto \mu_*(U,q)$ is \emph{not} the convex conjugate function (Legendre transform)  of~$p \mapsto \mu(U,p)$. Indeed, they are convex conjugates if and only if~$\a(U) = \a_*(U)$, and as mentioned already, we certainly do not have equality in~\eqref{e.matrixordering} in general! As we will show, what is actually true is that these two functions are becoming convex dual functions in the limit as~$U$ becomes very large, equivalent to the difference~$\a(U) - \a_*(U)$ becoming small. 

\smallskip

We combine~$\mu(U,p)$ and~$\mu_*(U,q)$ into a single quantity~$J(U,p,q)$, defined by
\begin{align}
\label{e.Jaas}
J(U,p,q) 
:=
\mu(U,p) + \mu_*(U,q) - p\cdot q
=
\frac12p\cdot \a(U) p + \frac12 q\cdot \a_*^{-1}(U)q - p\cdot q. 
\end{align}
Evidently,~$J$ is nonnegative by~\eqref{e.fenchel}. Note that~$J(U,p,0) = \mu(U,p)$ and~$J(U,0,q) = \mu_*(U,q)$, and so we regard the quantity~$J$ as the ``master quantity'' that contains the information of both~$\a(U)$ and~$\a_*(U)$ together. 
Note that~$J(U,p,q)$ may also be written as
\begin{equation}
\label{e.magic}
J(U,p,q) 
= 
\frac12 p \cdot (\a(U) - \a_*(U)) p  +   \frac12 (q  - \a_*(U) p) \cdot \a_*^{-1}(U) (q  - \a_*(U)  p)
\,.
\end{equation}
From this form we see that the size of~$J(U,p,\ahom p)$ controls the difference between the matrices~$\a (U)$,~$\a_*(U)$ and~$\ahom$. 

\smallskip

In the following lemma, we list many important properties of~$J$ and the matrices~$\a(U)$ and~$\a_*(U)$ which are used many times below. These properties are deterministic and are valid for every fixed, symmetric coefficient field~$\a\in\Omega$. While a long lemma with about a dozen different statements might be unpleasant at first glance, we ask the reader to bear with us. Each of these properties is essential to our analysis, and listing them together in one place will make them easier to find when we need to refer to them later. 

\smallskip

We denote the set of solutions~$u$ of~$-\nabla \cdot \a\nabla u=0$ in an open set~$U\subseteq\Rd$ by 
\begin{equation}
\label{e.def.AU}
\mathcal{A}(U) := 
\left\{ 
u\in H^1_{\mathrm{loc}}(U)
\,:\,
-\nabla \cdot \a\nabla u = 0 \quad \mbox{in} \ U \right\}
\,.
\end{equation}

\begin{lemma}[Properties of the coarse-grained coefficients]
\label{l.J.basicprops}
Assume~$\a\in\Omega$ is a coefficient field satisfying~$\a=\a^t$. 
Then the following assertions are valid for every bounded Lipschitz domain~$U \subseteq\Rd$:

\begin{itemize} 
\item~$\a(U)$ and~$\a_*(U)$ are ordered and satisfy the same ellipticity bounds as~$\a(x)$:
\begin{equation}
\label{e.a.bounds}
\lambda  \Id 
\leq
\biggl( \fint_{U} \a^{-1} (x)\,dx \biggr)^{\!\!-1} 
\leq 
\a_*(U) 
\leq 
\a(U) 
\leq 
\fint_U \a(x)\,dx 
\leq 
\Lambda  \Id.
\end{equation}

\item The quantity~$J(U,p,q)$ and the matrices~$\a(U)$ and~$\a_*(U)$ are~$\F(U)$--measurable. 

\item Variational representation of~$J$: for every~$p,q\in\Rd$, 
\begin{equation}
\label{e.variational.J}
J(U,p,q) = \max_{v\in \mathcal{A}(U)} 
\fint_U \left( -\frac12 \nabla v\cdot \a\nabla v -p\cdot \a\nabla v + q\cdot \nabla v   \right).
\end{equation}

\item The maximizer of~\eqref{e.variational.J} is unique up to additive constants and denoted by~$v(\cdot,U,p,q)$. It is the difference of the maximizer of~$\mu_*(U,q)$ in~\eqref{e.mustar} and the minimizer of~$\mu(U,p)$ in~\eqref{e.mu}. In particular, 
\begin{equation}
\label{e.vUpq.linear}
(p,q) \mapsto v(\cdot,U,p,q) \quad \mbox{is linear.}
\end{equation}

\item Characterization in terms of the energy of the maximizer: for every~$p,q\in\Rd$, 
\begin{equation}
\label{e.Jenergyv}
J(U,p,q) = \fint_U \frac12 \nabla v(\cdot,U,p,q) \cdot \a\nabla v(\cdot,U,p,q).
\end{equation}

\item Characterization in terms of spatial averages of gradients \& fluxes: for every~$p,q\in\Rd$, 
\begin{equation}
\label{e.a.astar.formulas}
\fint_U \nabla v(\cdot,U,p,q) 
= \a_*^{-1}(U) q - p
\qquad \mbox{and} \qquad
\fint_U \a \nabla v(\cdot,U,p,q) 
=
q - \a(U) p 
\,.
\end{equation}

\item First variation: for every~$w\in \A(U)$ and~$p,q\in\Rd$, 
\begin{equation}
\label{e.firstvar}
q\cdot \fint_U \nabla w - p \cdot \fint_U \a \nabla w 
=
\fint_U \nabla w \cdot \a \nabla v(\cdot,U,p,q) . 
\end{equation}

\item Second variation and quadratic response: for every~$w\in \A(U)$ and~$p,q\in\Rd$, 
\begin{align}
\label{e.quadresp}
& J(U,p,q) - \fint_U \left( -\frac12 \nabla w \cdot \a\nabla w -p\cdot \a\nabla w+ q\cdot \nabla w   \right)
\notag \\ & \qquad \qquad\qquad 
=
\fint_U \frac12 \left( \nabla v(\cdot,U,p,q) - \nabla w \right)\cdot \a\left( \nabla v(\cdot,U,p,q) - \nabla w \right).
\end{align}

\item Subadditivity: for every~$m,n\in\N$ with~$n<m$ and~$p,q\in\Rd$, 
\begin{align}
\label{e.subaddJ}
J(\cu_m,p,q) 
\leq 
\avsum_{z\in 3^n\Zd \cap \cu_m} 
J(z+\cu_n,p,q)
\end{align}
and, equivalently, 
\begin{align}
\label{e.subadda}
\a(\cu_m) 
\leq 
\avsum_{z\in 3^n\Zd \cap \cu_m} \!\!\!
\a(z+\cu_n) 
\quad \mbox{and} \quad 
\a_*(\cu_m) 
\geq 
\biggl(
\avsum_{z\in 3^n\Zd \cap \cu_m} \!\!\!
\a_*^{-1}(z+\cu_n) 
\biggr)^{\!\!-1}\!\!.
\end{align}

\item Control of~$\a-\a_*$ by the convex duality defect: there exists a constant~$C(d,\lambda,\Lambda)<\infty$ such that, for every matrix~$\b\in \R^{d\times d}_{\mathrm{sym}}$,
\begin{equation} 
\label{e.diagonalset}
\left| \a(U) - \a_*(U) \right| 
+
\left| \a(U) - \b \right|^2 
+
\left| \a_*(U) - \b \right|^2 
\leq C \sum_{i=1}^d J(U,e_i,\b e_i)
\,. 
\end{equation}

\item Control of additivity defect by convex duality defect: for every matrix~$\b\in \R^{d\times d}_{\mathrm{sym}}$ and~$m,n\in\N$ with~$n<m$, 
\begin{align}
\label{e.additivity.by.J}
\!\!\!\avsum_{z\in 3^n\Zd\cap \cu_m} \!\!\!\!
\a(z+\cu_n) 
- 
\biggl( \avsum_{z\in 3^n\Zd\cap \cu_m}
\!\!\!\!
\a_*^{-1} (z+\cu_n) \biggr)^{\!\! -1} \!\!
\leq
2 \! \!\!\! \! \avsum_{z\in 3^n\Zd\cap \cu_m} 
\sum_{i=1}^d 
J(z+\cu_n,e_i,\b e_i) \Id 
\,. 
\end{align}

\item Coarse-graining flux inequality: for every~$u\in \A(U)$ and~$e\in\Rd$, 
\begin{equation}
\label{e.fluxmaps}
\left| e\cdot 
\fint_U \bigl( \a_*(U) - \a \bigr) \nabla u
\right| 
\leq 
\bigl(e\cdot \bigl(\a(U) - \a_*(U)\bigr) e\bigr)^{\nicefrac12} 
\biggl( \fint_U\nabla w \cdot\a\nabla w \biggr)^{\!\!\nicefrac12} 
\,.
\end{equation}

\item Coarse-graining energy inequality: for every~$u\in H^1(U)$, 
\begin{align}
\label{e.energymaps}
\frac12\left( \fint_U \nabla u \right) \cdot \a_*(U) \left( \fint_U \nabla u \right)
\leq
\fint_U \frac12 \nabla u \cdot \a\nabla u 
\,.
\end{align}
Likewise, for every~$\g \in L^2_{\sol}(U)$,
\begin{align}
\label{e.energymaps.dual}
\frac12\left( \fint_U \g \right) \cdot \a^{-1}(U) \left( \fint_U \g \right)
\leq
\fint_U \frac12 \g \cdot \a^{-1}\g 
\,.
\end{align}
\end{itemize}
\end{lemma}
\begin{proof}
We begin with the proof of~\eqref{e.a.bounds}. 
The uniform ellipticity assumption~\eqref{e.ellipticity} clearly implies~$\fint_U \a(x)\,dx \leq \Lambda  \Id$ and~$\fint_U \a^{-1}(x) \,dx \leq \lambda^{-1}  \Id$. 
Testing the definition of~$\mu$ in~\eqref{e.mu} with the affine function~$\ell_p$ and varying~$p$ yields~$\a(U) \leq \fint_U \a(x)\,dx$. 
We next observe that, for every~$w \in H^1(U)$, 
\begin{align*}
\fint_U  \left( q \cdot \nabla w - \frac12\nabla w\cdot \a\nabla w \right) 
\leq 
\fint_U
\sup_{p\in\Rd} 
\left( q \cdot p - \frac12 p \cdot \a p \right) 
=
\frac12q\cdot 
\left(
\fint_U
\a^{-1}(x)\,dx \right)q.
\end{align*}
Taking the supremum over~$w\in H^1(U)$ yields~$\a_*^{-1}(U) \leq \fint_U \a^{-1}(x)\,dx$, which is the desired lower bound for~$\a_*(U)$. Finally, the inequality~$\a_*(U) \leq \a(U)$ was obtained above in~\eqref{e.matrixordering}. 

\smallskip

The~$\F(U)$--measurability of~$\mu(U,p)$ and~$\mu_*(U,q)$ and hence of~$J(U,p,q)$,~$\a(U)$ and~$\a_*(U)$ is immediate from the definitions.

\smallskip

We next give the proof of~\eqref{e.variational.J} and of the assertion below it. 
Let~$\hat{v}$ denote the minimizer of~$\mu(U,p)$ in~\eqref{e.mu}. Note that, by~$\hat{v} - \ell_p \in H^1_0(U)$ and~$\hat{v}\in\A(U)$, we have that~$p=\fint_U \nabla \hat{v}$ as well as~$p\cdot \fint_U\a\nabla w = \fint_U \nabla \hat{v}\cdot \a\nabla w$ for every~$w\in\A(U)$. 
We next fix~$w\in \A(U)$ and compute
\begin{align}
\label{e.Jvarproof}
\lefteqn{
\mu(U,p) + 
\fint_U \left( -\frac12 \nabla w \cdot \a\nabla w + q\cdot\nabla w \right) \,dx
- p\cdot q
} \qquad  & 
\notag \\ &
=
\fint_U \left( \frac12\nabla \hat{v}\cdot \a\nabla \hat{v} -\frac12 \nabla w \cdot \a\nabla w + q\cdot( \nabla w - \nabla \hat{v}) \right) \,dx
\notag \\ &
=
\fint_U \left( - \frac12\left( \nabla w - \nabla \hat{v}\right) \cdot \a\left( \nabla w - \nabla \hat{v}\right) - \nabla \hat{v} \cdot \a\left( \nabla w - \nabla \hat{v} \right) + q\cdot( \nabla w - \nabla \hat{v}) \right) \,dx
\notag \\ &
=
\fint_U \left( - \frac12\left( \nabla w - \nabla \hat{v}\right) \cdot \a\left( \nabla w - \nabla \hat{v}\right) - p \cdot \a\left( \nabla w - \nabla \hat{v} \right) + q\cdot( \nabla w - \nabla \hat{v}) \right) \,dx.
\end{align}
Taking the supremum of the left and right sides of~\eqref{e.Jvarproof} over~$w\in \A(U)$ yields~\eqref{e.variational.J}. Note that the maximizer of~$\mu_*(U,q)$ in~\eqref{e.mustar} belongs to~$\A(U)$, so the supremum can be restricted to~$\A(U)$. The assertion below~\eqref{e.variational.J} is also immediate from~\eqref{e.Jvarproof}. The linearity~\eqref{e.vUpq.linear} of~$(p,q) \mapsto v(\cdot,U,p,q)$ is then immediate. 

\smallskip

We turn next to the proof of the first and second variations,~\eqref{e.firstvar} and~\eqref{e.quadresp}.
Denote the maximizer of the variational problem in~\eqref{e.variational.J}
by~$v:= v(\cdot,U,p,q)$. Select a test function~$w\in \A(U)$ and compute
\begin{align*}
\lefteqn{
J(U,p,q) -
\fint_U \left(
-\frac12\nabla (v+w) \cdot \a\nabla (v+w) - p\cdot \a\nabla (v+w) + q\cdot \nabla (v+w) 
\right)
} \qquad\qquad\qquad & 
\\ &
=
\fint_U \left(
-\frac12\nabla w \cdot \a\nabla w - \nabla v\cdot \a\nabla w - p\cdot \a \nabla w + q\cdot \nabla w 
\right)
\,.
\end{align*}
As this expression is nonnegative for every~$w\in\A(U)$, we may replace~$w$ by~$tw$ and~$-tw$ and then send~$t\to 0$ after dividing by~$t$ to obtain
\begin{equation*}
0 = \fint_U \left(
\nabla v\cdot \a\nabla w + p\cdot\a \nabla w - q\cdot \nabla w 
\right).
\end{equation*}
This is~\eqref{e.firstvar}. 
To get~\eqref{e.quadresp}, replace~$w$ with~$w-v$ and substitute it into the previous display. 

\smallskip

The identity~\eqref{e.Jenergyv} is immediate from~\eqref{e.quadresp}: we just take~$w=0$ in the latter.  

\smallskip

The identity~\eqref{e.a.astar.formulas} is obtained by~\eqref{e.firstvar}, the polarization identity followed by~\eqref{e.Jenergyv} and~\eqref{e.Jaas}. We get
\begin{align}
\label{e.fluxmaps.computation}
\lefteqn{ 
q' \cdot \fint_U \nabla v(\cdot,U,p,q) - p' \cdot \fint_U \a \nabla v(\cdot,U,p,q) 
} \qquad & 
\notag \\ & 
=
\fint_U \nabla v(\cdot,U,p,q) \cdot \a \nabla v(\cdot,U,p',q')
\notag \\ & 
=
\fint_U \frac14 \nabla v(\cdot,U,p{+}p',q{+}q') \cdot \a \nabla v(\cdot,U,p{+}p',q{+}q')
\notag \\ & \qquad 
-
\fint_U \frac14 \nabla v(\cdot,U,p{-}p',q{-}q') \cdot \a \nabla v(\cdot,U,p{-}p',q{-}q')
\notag \\ & 
=
\frac12 J(U,p{+}p',q{+}q') - \frac12 J(U,p{-}p',q{-}q')
\notag \\ & 
= 
p'\cdot \a(U) p + q'\cdot \a_*^{-1}(U) q - p\cdot q' - q \cdot p'
\notag \\ & 
= 
q' \cdot ( \a_*^{-1}(U) q - p) -
p'\cdot ( q - \a(U) p  ) 
\,. 
\end{align}
This yields~\eqref{e.a.astar.formulas}.  

\smallskip

We next show subadditivity. 
Let~$\{ U_i \}_{i=1,\ldots,N}$ be a partition of~$U$ (up to a null set). By~\eqref{e.variational.J}, 
\begin{align*}
J(U,p,q) & 
= 
\sup_{v\in \mathcal{A}(U)} 
\fint_U \left( -\frac12 \nabla v\cdot \a\nabla v -p\cdot \a\nabla v + q\cdot \nabla v   \right)
\\ & 
=
\sup_{v\in \mathcal{A}(U)} 
\sum_{i=1}^N
\frac{|U_i|}{|U|}
\fint_{U_i} \left( -\frac12 \nabla v\cdot \a\nabla v -p\cdot \a\nabla v + q\cdot \nabla v   \right)
\\ & 
\leq
\sum_{i=1}^N
\frac{|U_i|}{|U|}
\sup_{v\in \mathcal{A}(U_i)} 
\fint_{U_i} \left( -\frac12 \nabla v\cdot \a\nabla v -p\cdot \a\nabla v + q\cdot \nabla v   \right)
=
\sum_{i=1}^N
\frac{|U_i|}{|U|}
J(U_i,p,q). 
\end{align*}
Specializing to the case~$U = \cu_m$ with partition~$\{ z+\cu_n \,:\, z\in 3^n\Zd \cap \cu_m \}$, we obtain~\eqref{e.subaddJ}. The inequalities~\eqref{e.subadda} are immediate from~\eqref{e.Jaas} and~\eqref{e.subaddJ}. 

\smallskip

We turn to the proof of~\eqref{e.diagonalset}.
According to~\eqref{e.magic}, for any symmetric matrix~$\b$, 
\begin{equation}
\label{e.magic.b.plug}
J(U,e,\b e) 
=
\frac12 e \cdot \bigl( \a(U) - \a_*(U) \bigr) e
+
\frac12 e \cdot \bigl(\b - \a_*(U)\bigr) \a_*^{-1}(U) \bigl(\b - \a_*(U)\bigr) e
\,.
\end{equation}
the fact that~$\a (U)- \a_*(U) \geq 0$ and that~$\a^{-1}_* \geq \Lambda^{-1}  \Id \geq c \Id$ by~\eqref{e.a.bounds}, we deduce that 
\begin{equation*}
\left| \a(U) - \a_*(U) \right|
+ \left| \a_*(U) - \b \right|^2
\leq
C\sup_{e \in B_1} 
J(U,e,\b e) 
\leq 
C \sum_{i=1}^d 
J(U,e_i,\b e_i). 
\end{equation*}
This yields~\eqref{e.diagonalset}.
Moreover, we note that 
\begin{equation}
\label{e.JU.aastar.gap}
J(U,e,\a_*(U) e) 
=
\frac12 e \cdot \bigl( \a(U) - \a_*(U) \bigr) e
\,.
\end{equation}

\smallskip

We next prove~\eqref{e.additivity.by.J}. This is a consequence of~\eqref{e.diagonalset} and the fact that the mean and harmonic mean of a finite sequence of positive real numbers differ by, at most, a multiple of the sample variance. In terms of matrices, the estimate says that, for every sequence~$\b_1,\ldots,\b_N$ of symmetric positive definite matrices with~$\mathbf{m} := \avsum_{i=1}^N \b_i$ and~$\mathbf{h}:= \bigl(\avsum_{i=1}^N \b_i^{-1} \bigr)^{\!\!-1}$, we find the following elementary identity, for every~$\b \in \R^{d\times d}_{\sym}$,  
\begin{equation}
\label{e.mean.harm.mean.iden}
\mathbf{m} 
= 
\mathbf{h} +  \avsum_{i=1}^N  (\b - \b_i) \b_i^{-1} (\b - \b_i)  - (\mathbf{h} - \b) \mathbf{h}^{-1} (\mathbf{h} - \b)\,.
\end{equation}
The last term is negative so that we deduce that
\begin{equation} 
\label{e.harmonic.mean.mean.quad.response}
\mathbf{m} 
\leq
\mathbf{h} + \avsum_{i=1}^N  (\b - \b_i) \b_i^{-1} (\b - \b_i) \,.
\end{equation}
Applying~\eqref{e.harmonic.mean.mean.quad.response} to~$\{\a_*(z+\cu_n)\,:\, z\in 3^n\Zd\cap\cu_m \}$, we deduce that
\begin{align*} 
\avsum_{z\in 3^n\Zd\cap\cu_m} \!\!\!
\a_*(z+\cu_n) 
& 
\leq 
\biggl( \avsum_{z\in 3^n\Zd\cap\cu_m} \a_*^{-1}(z+\cu_n)  \biggr)^{\!-1} 
\notag \\ 
& \quad
+  \!\!\!
\avsum_{z\in 3^n\Zd\cap\cu_m} \!\!\! (\a_*(z+\cu_n) - \b)\a_*^{-1}(z+\cu_n) (\a_*(z+\cu_n) - \b)\,. 
\end{align*}
Therefore,~\eqref{e.magic.b.plug} yields that, for every~$e\in\Rd$, 
\begin{equation*} 
\!\!\!  \avsum_{z\in 3^n\Zd\cap\cu_m} \!\!\!
e \cdot \a(z+\cu_n) e
\leq 
e \cdot \biggl( \avsum_{z\in 3^n\Zd\cap\cu_m} \a_*^{-1}(z+\cu_n)  \biggr)^{\! -1} e
+ 
2 \!\!\! \avsum_{z\in 3^n\Zd\cap\cu_m} \!\!\! 
J(z+\cu_n,e,\b e)\,.
\end{equation*}
This gives~\eqref{e.additivity.by.J}.

\smallskip

Next, we prove~\eqref{e.fluxmaps}. We use~\eqref{e.firstvar}, the Cauchy-Schwarz inequality,~\eqref{e.Jenergyv} and~\eqref{e.JU.aastar.gap} to obtain, for every~$w\in\A(U)$ and~$e\in\Rd$, 
\begin{align*}
e \cdot 
\fint_U 
( \a_*(U) - \a ) \nabla w
&
=
\fint_U \nabla v(\cdot,U,e,\a_*(U)e) 
\cdot 
\a \nabla w
\\ & 
\leq
\left(\fint_U \nabla v(\cdot,U,e,\a_*(U)e) \cdot\a\nabla v(\cdot,U,e,\a_*(U)e)\right)^{\!\!\nicefrac12} \!
\left( \fint_U\nabla w \cdot\a\nabla w \right)^{\!\!\nicefrac12} 
\\ & 
=
( 2J(U,e,\a_*(U)e))^{\nicefrac12}
\biggl( \fint_U\nabla w \cdot\a\nabla w \biggr)^{\!\!\nicefrac12} 
\\ & 
= 
\bigl(e\cdot \bigl(\a(U) - \a_*(U)\bigr) e\bigr)^{\nicefrac12} 
\biggl( \fint_U\nabla w \cdot\a\nabla w \biggr)^{\!\!\nicefrac12} 
\,.
\end{align*}
This concludes the proof of~\eqref{e.fluxmaps}.

\smallskip

Finally, we give the proof of~\eqref{e.energymaps} and~\eqref{e.energymaps.dual}.
By the definition of~$\mu_*$, we have that, for every~$u\in H^1(U)$, if we set~$p:= \fint_U \nabla u$ then we have 
\begin{align*}
\frac12p \cdot \a_*(U) p
&
=
\sup_{q\in \Rd}
\biggl( p\cdot q - \frac12 q \cdot \a_*^{-1}(U) q \biggr)
\notag \\ & 
=
\sup_{q\in \Rd}
\inf_{v\in H^1(U)} 
\fint_U \biggl( q\cdot (\nabla u - \nabla v) + \frac12\nabla v \cdot \a\nabla v \biggr)
\leq 
\fint_U \frac 12\nabla u \cdot \a\nabla u
\,.
\end{align*}
This yields~\eqref{e.energymaps}.
To obtain~\eqref{e.energymaps.dual}, we fix
$\g \in L^2_{\sol} (U)$, set~$q:= -\fint_U \g$ and similarly compute, using the dual variational representation~\eqref{e.mu.dual},
\begin{align*}
\frac12q \cdot \a^{-1}(U) q
&
=
\sup_{p\in \Rd}
\biggl( p\cdot q - \frac12 p \cdot \a(U) p \biggr)
\notag \\ & 
=
\sup_{p\in \Rd}
\inf_{\mathbf{h} \in L^2_{\sol} (U)} 
\fint_U \biggl( p\cdot (\mathbf{h} - \g) + \frac12\mathbf{h} \cdot \a^{-1} \mathbf{h} \biggr)
\leq 
\fint_U \frac 12\g \cdot \a^{-1} \g
\,.
\end{align*}
This completes the proof of the lemma. 
\end{proof}

\subsection{Further properties}

Of the many properties listed in Lemma~\ref{l.J.basicprops}, one that perhaps deserves to be highlighted is~\eqref{e.diagonalset}.\footnote{Note that~\eqref{e.diagonalset} appears strange at first glance because the first term on the left has no square, but the other two terms are squared. This is not a typo but a subtle and important point! The two coarse-grained matrices are indeed typically much closer to each other than they are to the homogenized matrix.} On a qualitative level, it says that if we can find a matrix~$\b$ such that~$J(U,p,q)$ vanishes for every~$(p,q) = (e,\b e)$, then we know everything about~$J(U,\cdot)$ since~$\a(U) = \a_*(U) = \b$. We recognize this from the convex analysis: if~$J(U,e,\b e)$ vanishes for all~$e\in\Rd$ then~$\mu(U,\cdot)$ and~$\mu_*(U,\cdot)$ are indeed convex dual functions; and equality in the Fenchel-Young inequality uniquely specifies the gradients of the convex dual pair of functions. What the estimate~\eqref{e.diagonalset} says is that, at least for uniformly quadratic functions, this relationship has some stability: if we can find a matrix~$\b$ such that~$J(U,e,\b e)$ is small, then~$\a(U)$ and~$\a_*(U)$ are nearly equal and both are nearly equal to~$\b$. For this reason, we often call~$\inf_{q \in\Rd} J(\cu_n,p,q)$ the \emph{convex duality defect}.

\smallskip

On a technical level, the inequality~\eqref{e.diagonalset} says that if we can find a \emph{deterministic} matrix~$\b$ such that~$\E\left[ J(U,e,\b e) \right]$ is small for every~$|e|=1$, then we have control over the expected difference~$\E \left[ \a(U) - \a_*(U) \right]$ as well as control on the variance of both~$\a(U)$ and~$\a_*(U)$. 
Therefore, we can control both subadditive quantities and obtain a quantitative rate of homogenization if we can find a deterministic matrix~$\b$ and a convergence rate for this limit: 
\begin{equation} 
\label{e.Jtozero}
\lim_{n\to\infty} \E \left[ J(\cu_n,p,\b p) \right] = 0. 
\end{equation}
This is precisely the strategy we pursue in Section~\ref{ss.subadd.conv}. Of course, the deterministic matrix~$\b$ will be~$\ahom$, the homogenized matrix defined below in~\eqref{e.ahom.def}. 

\smallskip

Observe that the combination of~\eqref{e.matrixordering},~\eqref{e.subadda} and~\eqref{e.additivity.by.J} asserts that, for any~$\b\in \R^{d\times d}_{\mathrm{sym}}$,
\begin{align}
\label{e.wrap.around}
\a_*(\cu_m) 
& 
\leq 
\a(\cu_m) 
\notag \\ & 
\leq 
\avsum_{z\in 3^n\Zd \cap \cu_m} 
\a(z+\cu_n) 
\notag \\ & 
\leq 
\biggl(
\avsum_{z\in 3^n\Zd \cap \cu_m} 
\a_*^{-1}(z+\cu_n) 
\biggr)^{\!\!-1}
+ 2 \biggl (\avsum_{z\in 3^n\Zd\cap \cu_m}
\sum_{i=1}^d 
J(z+\cu_n,e_i,\b e_i) \biggr )\Id
\notag \\ & 
\leq 
\a_*(\cu_m) + 2\biggl ( \avsum_{z\in 3^n\Zd\cap \cu_m}
\sum_{i=1}^d 
J(z+\cu_n,e_i,\b e_i)\biggr )\Id
\,.
\end{align}
As~\eqref{e.wrap.around} has circled, each of the inequalities in the display is strict by at most the last term on the right side. In particular, we obtain bounds on the \emph{additivity defects}:
\begin{equation}
\label{e.add.defect.a}
\avsum_{z\in 3^n\Zd \cap \cu_m} 
\a(z+\cu_n) 
\leq 
\a(\cu_m) + 2 \avsum_{z\in 3^n\Zd\cap \cu_m}
\sum_{i=1}^d 
J(z+\cu_n,e_i,\b e_i) \Id
\end{equation}
and, using also lower bounds for~$\a_*$ in~\eqref{e.a.bounds},
\begin{equation}
\label{e.add.defect.astar}
\avsum_{z\in 3^n\Zd \cap \cu_m} 
\a_*^{-1}(z+\cu_n) 
\leq 
\a_*^{-1} (\cu_m) + 2 \lambda^{-2} \avsum_{z\in 3^n\Zd\cap \cu_m}
\sum_{i=1}^d 
J(z+\cu_n,e_i,\b e_i) \Id
\,.
\end{equation}
In other words, each coarse-grained matrices fails to be additive by at most a multiple of the convex duality defect, averaged over the smaller scale cubes. This is essential to our strategy because this ``additivity'' allows us to replace~$\a(\cu_m)$ with~$\avsumtext_{z\in 3^n\Zd \cap \cu_m} \a(z+\cu_n)$, at least up to a suitably small error, and thus to take advantage of the stochastic cancellations. 

\smallskip

Indeed, this is the only way we will use the mixing condition~\eqref{e.CFS} in this chapter: it will be applied exclusively to sums of~$J(\cdot,p,q)$ indexed over families~$\{ z+\cu_n\,:\, z\in 3^n\Zd\cap \cu_m \}$ of triadic cubes, for integers~$n<m$. We have already shown in Lemma~\ref{l.J.basicprops} that, up to a normalizing factor and after subtracting its expectation, the corresponding random variables~$J(z+\cu_n,p,q)$ satisfy the conditions in~\eqref{e.CFS.ass}---with the exception of the bound on its Malliavin derivative, which is checked in the following lemma. For future reference, we note that the statement and proof of this lemma do not use the symmetry assumption~\eqref{e.symm} on the coefficient field, so it is also valid for the more general quantity~$J$ defined in Chapter~\ref{s.nonsymm}.

\begin{lemma}
\label{l.malliavin}
There exists~$C(d,\lambda,\Lambda)<\infty$ such that, for every bounded domain~$U\subseteq\Rd$ and~$p,q\in \R^d$, 
\begin{equation}
\label{e.malliavin.estimate}
\left| 
\partial_{\a(U)} J(U,p,q) \right|
\leq
C (|p| + |q|) J(U,p,q)^{\nicefrac12}.  
\end{equation}
\end{lemma}
\begin{proof}
Fix~$p,q \in \R^d$,~$t \in (0,1)$, and pick~$\a,\tilde\a\in\Omega$ such that~$| \a-\tilde{\a} | \leq t\indc_{U}$. Denote by~$J(U,p,q)$ and~$\tilde{J}(U,p,q)$ the quantities corresponding to~$\a$ and~$\tilde{\a}$, respectively, and the maximizers in their variational formulation~\eqref{e.variational.J}
by~$v$ and~$\tilde{v}$. Let~$w \in v + H_0^1(U)$ solve~$-\nabla \cdot \tilde \a \nabla w  = 0$ in~$U$. Observe that 
\begin{equation*}  
-\nabla \cdot \a \nabla (w -v) = \nabla \cdot (\tilde\a - \a) \nabla w \quad \mbox{in } U
\,.
\end{equation*}
Testing this equation with~$w-v$ yields
\begin{align*}
\fint_U \nabla (w-v) \cdot \a\nabla (w-v) 
&
=
\fint_U \nabla (w-v) \cdot (\a - \tilde{\a}) \nabla w 
\\ & 
=
\fint_U \nabla (w-v) \cdot (\a - \tilde{\a}) \nabla v 
+
\fint_U \nabla (w-v) \cdot (\a - \tilde{\a}) \nabla (w-v)
\,.
\end{align*}
Thus
\begin{equation*}
(\lambda-t) 
\fint_U \bigl| \nabla (w-v) \bigr|^2
\leq
t \fint_U \bigl| \nabla (w-v)  \bigr| \bigl|  \nabla v \bigr|
\end{equation*}
This implies by Cauchy-Schwarz that, if~$t\in (0,\frac12\lambda)$, then
\begin{align*}  
\fint_U \bigl| \nabla (w-v) \bigr|^2
\leq 
C t^2 \fint_{U} \bigl| \nabla v \bigr|^2
.
\end{align*}
It follows from the triangle inequality that
\begin{align*} \notag  
\tilde{J}(U,p,q) 
&
\geq
\fint_{U} \left( - \frac12\nabla w \cdot \tilde{\a}\nabla w -p\cdot\tilde{\a} \nabla w+q\cdot \nabla w \right)
\\ \notag
& 
\geq 
\fint_{U} \left( 
- \frac12\nabla v\cdot \a \nabla v
- p\cdot\a\nabla v
+q\cdot \nabla v
\right) 
- C t \fint_U \bigl ( | \nabla v |^2 + ( |p| +|q|) | \nabla v | \bigr)
\\ & 
\geq
 J(U,p,q) - Ct \bigl( J(U,p,q) + ( |p| + |q| ) J(U,p,q)^{\nicefrac12} \bigr).
\end{align*}
Since~$J(U,p,q) \leq C(|p| + |q|)^2$, we deduce that 
\begin{equation*} \label{}
J(U,p,q) 
-
\tilde{J}(U,p,q)
\leq  Ct (|p| + |q|) J(U,p,q)^{\nicefrac12}. 
\end{equation*}
By symmetry, we also have 
\begin{equation*} \label{}
\tilde{J}(U,p,q) 
-
J(U,p,q)
\leq  Ct (|p| + |q|) J(U,p,q)^{\nicefrac12}. 
\end{equation*}
Combining these and sending~$t\to 0$ yields~\eqref{e.malliavin.estimate}. 
\end{proof}

An important advantage of working with a subadditive quantity---particularly one that is nonnegative and converging to zero---is that we obtain very strong control of the fluctuations, essentially for free, once we obtain control over the expectation. This is one way of seeing why it is a much better idea to try to control the difference~$\a(U) - \a_*(U)$ than to try to control either~$\a(U) - \ahom$ or~$\ahom - \a_*(U)$ directly.

\smallskip

We formalize this idea in the following lemma, which does not use the particular structure of the quantity~$J$, just that it is nonnegative and subadditive. In particular, the symmetry assumption~\eqref{e.symm} is not used here, a remark we will need in Chapter~\ref{s.nonsymm}. 
Note that the left side of~\eqref{e.whatCFSgives} is an upper bound for~$J(\cu_m,p,q)$.

\begin{lemma}
\label{l.J.upperfluct}
Assume that~$\P$ is~$\Zd$--stationary and satisfies~$\CFS(\beta,\Psi)$ and~\eqref{e.symm}.
There exists a constant\footnote{Recall that~$\data$ refers to the list of parameters as defined in~$\eqref{e.data.def}$.}~$C(\dataref)<\infty$ such that, for every~$p,q\in B_1$ 
and~$m,n,l\in\N$ with~$n < m$ and~$\beta m < l < m$,
\begin{align}
\label{e.whatCFSgives}
\avsum_{z\in 3^n\Zd\cap \cu_m} 
J(z+\cu_n,p,q)
\leq 
\E \left[ J( \cu_{n\wedge l},p,q) \right]
+ 
\O_\Psi\bigl( C3^{-\frac d2(m-l)} \bigr)
\,.
\end{align}
\end{lemma}
\begin{proof}
If~$n\geq l$, then by subadditivity, we have that 
\begin{align*}
\avsum_{z\in 3^n\Zd\cap \cu_m} 
J(z+\cu_n,p,q)
\leq 
\avsum_{z\in 3^l\Zd\cap \cu_m} 
J(z+\cu_l,p,q).
\end{align*}
We apply the assumption~$\CFS(\beta,\Psi)$, noticing that~$X_z = c \left( J(z+\cu_l,p,q) - \E \left[ J(\cu_l,p,q) \right]\right)$ satisfies~\eqref{e.CFS.ass} for a suitably small constant~$c(\Lambda)>0$ by Lemmas~\ref{l.J.basicprops} and~\ref{l.malliavin}, 
to get
\begin{equation*}
\avsum_{z\in 3^l\Zd\cap \cu_m} 
J(z+\cu_l,p,q) 
\leq 
\E \left[ J(\cu_l,p,q) \right] 
+ 
\O_\Psi\bigl( C3^{-\frac d2(m-l)}\bigr)
\,,
\end{equation*}
which gives us~\eqref{e.whatCFSgives} in the case~$n\geq l$. 
In the case~$n<l$, we instead write 
\begin{equation*}
\avsum_{z\in 3^n\Zd\cap \cu_m} 
J(z+\cu_n,p,q)
=
\avsum_{z'\in 3^l\Zd\cap \cu_m}
\avsum_{z\in 3^n\Zd\cap (z'+\cu_l)} 
J(z+\cu_n,p,q)
\end{equation*}
and apply the~$\CFS(\beta,\Psi)$ condition to the inner sum on the right, which, after subtracting its expectation, once again satisfies the condition~\eqref{e.CFS.ass} by Lemmas~\ref{l.J.basicprops} and~\ref{l.malliavin} and the additivity of the Malliavin derivative. We obtain:
\begin{align*}
\avsum_{z'\in 3^l\Zd\cap \cu_m}
\avsum_{z\in 3^n\Zd\cap (z'+\cu_l)} 
J(z+\cu_n,p,q)
\leq 
\E \left[ J(\cu_n,p,q) \right] 
+
\O_\Psi\bigl( C3^{-\frac d2(m-l)} \bigr)
\,. 
\end{align*}
This is~\eqref{e.whatCFSgives} in the case~$n< l$ and completes the proof of the lemma. 
\end{proof}

\subsection{Algebraic rate of convergence by iterating up the scales}
\label{ss.subadd.conv}

In this section, we follow the ideas of~\cite{AS,AKMBook} to get a rate of convergence for the limit~\eqref{e.Jtozero} under the assumption that~$\P$ satisfies~$\CFS(\beta,\Psi)$. 

\begin{theorem}[{Algebraic rate of decay}]
\label{t.subadd.converge}
Assume~$\P$ is a~$\Zd$--stationary measure on~$(\Omega,\F)$ and satisfies~$\CFS(\beta,\Psi)$ and~\eqref{e.symm}. 
Then there exists an exponent~$\alpha(\beta,d,\lambda,\Lambda) \in (0,1]$ 
and a constant~$C(\beta,\dataref)<\infty$
such that, for every~$m,n \in\N$ with~$\beta m < n < m$, 
\begin{equation}
\label{e.aastar.big.smash}
|\a(\cu_m) - \a_*(\cu_m)|
+
| \a(\cu_m) - \ahom |^2 
+
|\a_*(\cu_m) - \ahom |^2
\leq 
C3^{-n\alpha}
+ 
\O_\Psi \bigl( C3^{-\frac d2(m-n)} \bigr) 
\,.
\end{equation}
\end{theorem}

See also Corollary~\ref{c.subadd.converge}, below, which contains a statement similar to Theorem~\ref{t.subadd.converge} that is somewhat more useful. 
The estimate~\eqref{e.aastar.big.smash} is \emph{optimal in stochastic integrability}, meaning that the random part of the error on the right side of~\eqref{e.aastar.big.smash} is bounded as sharply as possible: see Remark~\ref{r.X.optimal}. 

\smallskip

\subsubsection{An algebraic rate for the means}

The main step in the proof of Theorem~\ref{t.subadd.converge} is to get a rate of convergence for the means of~$\a(\cu_m)$ and~$\a_*(\cu_m)$. The stochastic fluctuations can then be crushed by having a dual pair of subadditive quantities, as we will see. 

\smallskip

We introduce the following averaged versions of the coarse-grained coefficients:
\begin{equation}
\label{e.ahomahoms}
\ahom(U):= \E\left[ \a(U)\right] 
\quad \mbox{and} \quad 
\ahom_*(U):= \E\left[ \a_*^{-1}(U)\right]^{-1}.
\end{equation}
These are monotone: by the subadditivity and stationarity of~$J$, and the fact that~$\P$ is assumed to be~$\Zd$--stationary, we have, for every~$m\in \N$,
\begin{equation}
\label{e.matrix.subadd}
\ahom(\cu_{m+1}) \leq  \ahom(\cu_m) 
\quad \mbox{and} \quad 
\ahom_*(\cu_{m}) \leq  \ahom_*(\cu_{m+1}).
\end{equation}
We define the homogenized matrix~$\ahom$ by 
\begin{equation}
\label{e.ahom.def}
\ahom := \lim_{m\to \infty} \ahom_*(\cu_m). 
\end{equation}
We will learn, of course, as a consequence of the analysis, that~$\ahom$ is also the limit of~$\ahom(\cu_m)$ and so~\eqref{e.ahom.def} does not conflict with the definition of the homogenized matrix in Chapter~\ref{s.qualitativetheory}. 

\smallskip

The following proposition is the focus of most of this section. 

\begin{proposition}
\label{p.algebraicrate.E}
Assume~$\P$ is~$\Zd$--stationary and satisfies~$\CFS(\beta,\Psi)$ and~\eqref{e.symm}.
There exists~$\alpha(\beta,d,\lambda,\Lambda) >0$ and~$C(\dataref)<\infty$
such that, for every~$m\in\N$,  
\begin{align}
\label{e.EJtozero.rate}
\bigl| \ahom(\cu_m) - \ahom \bigr|
+
\bigl| \ahom_*(\cu_m) - \ahom \bigr|
\leq 
C3^{-m\alpha}. 
\end{align}
\end{proposition}

Observe that taking the expectation of~\eqref{e.Jaas} gives
\begin{equation}
\label{e.Jaas.W}
\E \bigl[ J(U,p,q) \bigr] 
=
\frac12p\cdot \ahom(U) p + \frac12 q\cdot \ahom_*^{-1}(U)q - p\cdot q. 
\end{equation}
Applying this~$U=\cu_m$ and~$(p,q) = (e, \ahom_*(\cu_m)e)$ and rearranging, we get 
\begin{equation}
\label{e.magic.b.plug.E}
\E \left[ J(\cu_m,e,\ahom_*(\cu_m) e) \right]
=\frac12 e \cdot \bigl( \ahom(\cu_m) - \ahom_*(\cu_m) \bigr) e
\,. 
\end{equation}
Similarly, the left side of~\eqref{e.EJtozero.rate} is equivalent to~$\sup_{e\in B_1}
\E \left[ J(\cu_m,e,\ahom e) \right]$ in the sense that 
\begin{equation}
\label{e.magic.b.plug.E.ahom}
\bigl| \ahom(\cu_m) - \ahom \bigr|
+
\bigl| \ahom_*(\cu_m) - \ahom \bigr|
\leq
\sup_{e\in B_1}
\E \left[ J(\cu_m,e,\ahom e) \right]
\leq 
C\bigl( \bigl| \ahom(\cu_m) - \ahom \bigr|
+
\bigl| \ahom_*(\cu_m) - \ahom \bigr|\bigr) 
\,.
\end{equation}
To derive the first inequality of~\eqref{e.magic.b.plug.E.ahom}, we have
\begin{align}
\label{e.magic.b.plug.E.ahom.pre}
\bigl| \ahom(\cu_m) - \ahom \bigr|
+
\bigl| \ahom_*(\cu_m) - \ahom \bigr|
\leq 
2 \bigl| \ahom(\cu_m) - \ahom_*(\cu_m) \bigr|
&
\leq
4 \E \left[ J(\cu_m,e,\ahom_*(\cu_m) e) \right]
\notag \\ & 
\leq
4 \E \left[ J(\cu_m,e,\ahom e) \right]
\,.
\end{align}
The first inequality in the display above is due to the fact that~$\ahom_*(\cu_m)\leq \ahom \leq \ahom(\cu_m)$; the second is immediate from~\eqref{e.magic.b.plug.E}; and the third inequality is a consequence of the observation that~$q \mapsto \E [ J(U,p,q)]$ attains its minimum at~$q = \ahom_*(U)p$, which we see from differentiating~\eqref{e.Jaas.W}.  
To prove the second inequality of~\eqref{e.magic.b.plug.E.ahom}, we take the expectation of~\eqref{e.magic.b.plug} and use that the matrices are upper bounded by~$C$ to find that 
\begin{align*}
\lefteqn{
\sup_{e\in B_1}
\E \left[ J(\cu_m,e,\ahom e) \right]
} \qquad 
\notag \\ & 
\leq
\sup_{e\in B_1}\Bigl( 
\frac12 e \cdot \bigl( \ahom(\cu_m) - \ahom_*(\cu_m) \bigr) e
+
\frac12 (\ahom-\ahom_*(\cu_m))\ahom_*^{-1} (\cu_m)(\ahom-\ahom_*(\cu_m))
\Bigr) 
\notag \\ & 
\leq 
C|\ahom(\cu_m) - \ahom_*(\cu_m) |
+
C| \ahom_*(\cu_m) - \ahom |
\notag \\ & 
\leq 
C\bigl( \bigl| \ahom(\cu_m) - \ahom \bigr|
+
\bigl| \ahom_*(\cu_m) - \ahom \bigr| \bigr)\,.
\end{align*}

The basic idea of the proof of Proposition~\ref{p.algebraicrate.E} is to show that~$| \ahom(\cu_m) - \ahom_*(\cu_m)|$ contracts by a multiplicative factor strictly less than one as we increment~$m$, 
allowing us to obtain the algebraic rate of convergence by \emph{iteration up the scales}.

\smallskip

We define the \emph{expected additivity defect at scale~$3^m$} by
\begin{equation}
\label{e.taun}
\tau_m := \left|\ahom (\cu_{m}) - \ahom(\cu_{m-1} )\right| + \left|  \ahom_*(\cu_{m}) - \ahom_*(\cu_{m-1} )\right|.
\end{equation}
This measures how much the expectations of the subadditive quantities change between successive triadic scales. 
Observe that, by~\eqref{e.magic.b.plug.E} and the monotonicity of the matrices,
\begin{align}
\label{e.subaddcontrol}
\tau_m
& 
\leq 
C \sum_{i=1}^d  e_i \cdot  
\bigl(\ahom(\cu_{m-1}) - \ahom(\cu_{m}) + \ahom_*(\cu_{m}) - \ahom_*(\cu_{m-1})   \bigr)  e_i
\notag \\ & 
\leq
C \sum_{i=1}^d \left( \E \left[ J(\cu_{m-1},e_i,\ahom_*(\cu_{m-1}) e_i)\right] - \E \left[ J(\cu_{m},e_i,\ahom_*(\cu_{m}) e_i)\right] \right)
\leq
C \tau_m
\,.
\end{align}
In view of~\eqref{e.magic.b.plug.E} and~\eqref{e.subaddcontrol}, it is natural to try to prove Proposition~\ref{p.algebraicrate.E} by bounding the quantity~$K_m:= \sum_{i=1}^d \E \left[ J(\cu_{m},e_i,\ahom_*(\cu_m) e_i)\right]$ by some expression involving~$\tau_m$. For instance, if we could prove the estimate 
\begin{equation}
\label{e.Km.naivewish}
K_m \leq C\tau_m, 
\end{equation}
then we would obtain from~\eqref{e.subaddcontrol} that, for some constant~$C$,
\begin{equation*}
K_{m+1} \leq \frac{C}{C+1} K_m
\,.
\end{equation*}
This could then be iterated to obtain 
$K_m \leq \theta^m K_0 \leq C\theta^m$ for~$\theta : = C/(C+1)<1$, which yields~\eqref{e.EJtozero.rate}. We will not prove~\eqref{e.Km.naivewish}, exactly, although this is the rough idea. The actual estimate, which is found in~\eqref{e.combo} below, is slightly more complicated: it has more error terms and has a sum of~$\tau_k$'s across many scales but with geometric discount factors for the smaller scales. 
This estimate can still be iterated to obtain an estimate like~$K_m \leq C \theta^m$. 

\smallskip

There are two main lemmas in the proof of Proposition~\ref{p.algebraicrate.E}. The first is an estimate showing that the expected convex duality defect can be controlled by~$\tau_n$ and the variance of~$\a_*^{-1}(\cu_n)$ on all smaller scales---with discount factors for the smaller scales. 

\begin{lemma}
[Control of the convex duality defect]
\label{l.flatness.rules}
Assume~$\a\in\Omega$ is a coefficient field satisfying~$\a=\a^t$. 
There exists~$C(d,\lambda,\Lambda)<\infty$ such that, for every~$m\in\N$ and~$p,q\in B_1$, 
\begin{align}
\label{e.flatness.quenched}
J(\cu_{m-1},p,q) 
&
\leq 
C
\sum_{n=0}^{{m}} 3^{n-m} 
\avsum_{y\in 3^{n}\Zd\cap \cu_m} 
( J(y+\cu_{n},p,q) - J (\cu_m,p,q) )
\notag \\ & \qquad 
+
C\sum_{n=0}^{{m}} 3^{n-m} 
\avsum_{y\in 3^n\Zd\cap \cu_{m}} 
\left| \a_*^{-1}(y+\cu_n) q - p\right|^2
+C3^{-2m}
\,.
\end{align}
In particular, if~$\P$ is~$\Zd$--stationary and satisfies~$\CFS(\beta,\Psi)$ and~\eqref{e.symm}, then, for every~$m\in\N$ and~$e\in  B_1$, 
\begin{align}
\label{e.flatness.rules}
\E \left[ J(\cu_m,e,\ahom_*(\cu_m)e) \right]
\leq 
C 
\sum_{n=0}^{m} 3^{n-m} 
\left( \tau_n + \var\left[ \a_*^{-1}(\cu_n) \right] \right)
+C 3^{-2m}\,. 
\end{align}
\end{lemma}
\begin{proof}
To simplify the notation, for each~$n\in\N$, we let~$v_{n,z} \hspace{-1pt}  :=  \hspace{-1pt}  v(\cdot,z+\cu_n,p,q)$.
We also denote~$v_{n}:= v_{n,0}$.

\smallskip

\emph{Step 1.} We claim that, for every~$m\in\N$ and~$p,q\in B_1$, 
\begin{align}
\label{e.makeroom}
J(\cu_{m-1},p,q) 
&
\leq 
C
3^{-2m}
\min_{k\in\R}
\left\| v(\cdot,\cu_m,p,q) - k \right\|_{\underline{L}^2(\cu_m)}^2
\notag \\ & \qquad 
+
\sum_{z\in 3^{m-1}\Zd\cap \cu_m} 
2( J(z+\cu_{m-1},p,q) - J (\cu_m,p,q) )
\,.
\end{align}
Compute, using~\eqref{e.Jenergyv} and~\eqref{e.quadresp}, 
\begin{align*}
\lefteqn{
J(\cu_{m-1},p,q) 
} \quad 
\notag \\ &  
=
\fint_{\cu_{m-1}} 
\frac12
\nabla v_{m-1} 
\cdot \a\nabla v_{m-1}
\\ & 
\leq 
\fint_{\cu_{m-1}} 
\nabla v_{m} 
\cdot \a\nabla v_{m}
+
\fint_{\cu_{m-1}} 
\left( \nabla v_{m}
-
\nabla v_{m-1} \right) 
\cdot 
\a\left( \nabla v_{m}
-
\nabla v_{m-1}
\right) 
\\ & 
\leq 
\fint_{\cu_{m-1}} 
\nabla v_{m} 
\cdot \a\nabla v_{m}
+
\sum_{z\in 3^{m-1}\Zd\cap \cu_m} 
\fint_{z+\cu_{m-1}} 
\left( \nabla v_{m}
-
\nabla v_{m-1,z} \right) 
\cdot 
\a\left( \nabla v_{m}
-
\nabla v_{m-1,z}
\right) 
\\ & 
=
\fint_{\cu_{m-1}} 
\nabla v_{m} 
\cdot \a\nabla v_{m}
+
\sum_{z\in 3^{m-1}\Zd\cap \cu_m} 
2( J(z+\cu_{m-1},p,q) - J (\cu_m,p,q) )
\,.
\end{align*}
By the Caccioppoli inequality, for every~$k\in \R$,
\begin{equation*}
\fint_{\cu_{m-1}}
\nabla v_m \cdot\a\nabla v_m
\leq
C
3^{-2m} \min_{k\in\R}
\left\| v_m -k \right\|_{\underline{L}^2(\cu_m)}^2
\,.
\end{equation*}
Combining this with the display above yields~\eqref{e.makeroom}. 

\smallskip

\emph{Step 2.} We show that,
for every~$m\in\N$ and~$p,q\in B_1$, 
\begin{align}
\label{e.multiscope}
\lefteqn{
3^{-m} 
\inf_{k\in\R} 
\left\| v(\cdot,\cu_m,p,q) - k \right\|_{\underline{L}^2(\cu_m)} 
} \qquad & 
\notag \\ &
\leq
C3^{-m} 
+
C
\sum_{n=0}^{{m}} 3^{n-m} 
\biggl(
\avsum_{y\in 3^{n}\Zd\cap \cu_m} 
( J(y+\cu_{n},p,q) - J (\cu_m,p,q) )
\biggr)^{\!\!\nicefrac12}
\notag \\ & \qquad 
+
C\sum_{n=0}^{{m}} 3^{n-m} 
\biggl(\avsum_{y\in 3^n\Zd\cap \cu_{m}} 
\left| \a_*^{-1}(y+\cu_n) q - p\right|^2
\biggr)^{\!\!\nicefrac12}
\,.
\end{align}
By the multiscale Poincar\'e inequality (Proposition~\ref{p.MSP}), we have
\begin{align*}
3^{-m} 
\inf_{k\in\R} 
\left\| v_m {-} k \right\|_{\underline{L}^2(\cu_m)} 
&
\leq 
C3^{-m} \left\| \nabla v_m \right\|_{\underline{L}^2(\cu_m)} 
+
C \! 
\sum_{n=0}^{m} 3^{n-m} 
\biggl( \avsum_{y\in 3^n\Zd\cap \cu_{m}} \!\!
\big| \left( \nabla v_m \right)_{y+\cu_n} \big|^2 \biggr)^{\!\nicefrac12}
.
\end{align*}
Using the triangle inequality and~\eqref{e.a.astar.formulas}, we observe that 
\begin{align*}
\lefteqn{
\biggl(
\avsum_{y\in 3^n\Zd\cap \cu_{m}} 
\big| \left( \nabla v_m \right)_{y+\cu_n} \big|^2\biggr)^{\!\!\nicefrac12}
} \quad & 
\\ &
\leq 
\biggl(\avsum_{y\in 3^n\Zd\cap \cu_{m}} \!\!
\big|  \left( \nabla v_{m} \right)_{y+\cu_n} -  \left( \nabla v_{n,y} \right)_{y+\cu_n} \big|^2\biggr)^{\!\!\nicefrac12}
+
\biggl(\avsum_{y\in 3^n\Zd\cap \cu_{m}} \!\!
\left| \a_*^{-1}(y+\cu_n) q - p\right|^2
\biggr)^{\!\!\nicefrac12}
\,.
\end{align*}
By~\eqref{e.Jenergyv} and~\eqref{e.quadresp}, we estimate
\begin{align*}
\avsum_{y\in 3^n\Zd\cap \cu_{m}} \!\!\!
\big|  \left( \nabla v_{m} \right)_{y+\cu_n}  - \left( \nabla v_{n,y} \right)_{y+\cu_n} \big|^2
\leq 
\frac2\lambda \!\!
\avsum_{y\in 3^{n}\Zd\cap \cu_m} \!\!
( J(y+\cu_{n},p,q) - J (\cu_m,p,q) )
.
\end{align*}
Combining the previous displays, using also that~$\left\| \nabla v_m \right\|_{\underline{L}^2(\cu_m)}^2  \hspace{-1pt}  \leq \hspace{-1pt} CJ(\cu_m,p,q) \hspace{-1pt} \leq \hspace{-1pt} C$ by~\eqref{e.Jenergyv}, we obtain~\eqref{e.multiscope}.

\smallskip

\emph{Step 3.} The conclusion. The estimate~\eqref{e.flatness.quenched}
is an immediate consequence of~\eqref{e.makeroom} and~\eqref{e.multiscope}.
The inequality~\eqref{e.flatness.rules} is obtained by taking the expectation of~\eqref{e.flatness.quenched}, using 
$\E \left[ J(\cu_m,e,\ahom_*(\cu_m)e) \right] \leq \E \left[ J(\cu_{m-1},e,\ahom_*(\cu_m)e) \right]$ and the following estimates:
\begin{align*}
\E\biggl [
\sum_{n=0}^{{m}} 3^{n-m} \!\!\!
\avsum_{y\in 3^{n}\Zd\cap \cu_m} \!\!\!
\bigl ( J(y+\cu_{n},e,\ahom_*(\cu_m)e) - J (\cu_m,e,\ahom_*(\cu_m)e) \bigr ) \biggr ]
&
\leq
\sum_{n=0}^{{m}} 3^{n-m} \!\!
\sum_{k=n+1}^{{m}} \tau_k 
\\ &
\leq
C\sum_{n=0}^{{m}} 3^{n-m} \tau_n\,,
\end{align*}
as well as
\begin{align*}
\lefteqn{
\E \Bigg[
\sum_{n=0}^{{m}} 3^{n-m} 
\avsum_{y\in 3^n\Zd\cap \cu_{m}} 
\left| \a_*^{-1}(y+\cu_n) \ahom_*(\cu_m)e - e\right|^2
\Bigg]
} \quad & 
\notag \\ & 
=
\E \Bigg[
\sum_{n=0}^{{m}} 3^{n-m} 
\avsum_{y\in 3^n\Zd\cap \cu_{m}} 
\left| \bigl(
\a_*^{-1}(y+\cu_n) - \ahom_*^{-1} (\cu_m) \bigr) \ahom_*(\cu_m) e
\right|^2
\Bigg]
\\ & 
\leq 
C
\sum_{n=0}^{{m}} 3^{n-m} 
\var \bigl[  \a_*^{-1}(\cu_n) \bigr]
+
C
\sum_{n=0}^{{m}} 3^{n-m} 
\left| \bigl(
\ahom_*^{-1}(\cu_n) - \ahom_*^{-1} (\cu_m) \bigr) 
\right|^2
\end{align*}
and, finally, 
\begin{equation*}
\sum_{n=0}^{{m}} 3^{n-m} 
\left| \bigl(
\ahom_*^{-1}(\cu_n) - \ahom_*^{-1} (\cu_m) \bigr) 
\right|^2
\leq 
C\sum_{n=0}^{{m}} 
3^{n-m} \sum_{k=n+1}^m \tau_k
\leq 
C\sum_{n=0}^{{m}} 
3^{n-m} \tau_n.
\end{equation*}
The proof of the lemma is now complete.
\end{proof}

In the following lemma, we control the variances of the coarse-grained matrices~$\a(\cu_m)$ and~$\a_*(\cu_m)$ in terms of the \emph{square} of the expectation of~$J(\cu_n,e,\b e)$.
It is here that we use the mixing condition~$\CFS(\beta,\Psi)$, in the form of Lemma~\ref{l.J.upperfluct}, and it is the only time we use it in the proof of Proposition~\ref{p.algebraicrate.E}. 

\begin{lemma}[Decay of the variance]
\label{l.flatness}
Assume that~$\P$ satisfies~$\CFS(\beta,\Psi)$ and~\eqref{e.symm}, and is~$\Zd$--stationary.
There exists a constant~$C(\dataref)<\infty$ such that, for every~$\b\in \R^{d\times d}_{\sym}$ and~$m,n \in\N$ with~$\beta m  < n < m$,  
\begin{align}
\label{e.variance.J2}
\lefteqn{
\var \bigl[  \a(\cu_m) \bigr] 
+
\var \bigl[  \a_*^{-1}(\cu_m) \bigr] 
} \qquad & 
\notag \\ & 
\leq
C \min \biggl\{ 
\sup_{|e|=1}
\E \left[ J(\cu_n,e,\b e) \right]^2
\,,
\sum_{k=n+1}^m \tau_k
\biggr\}
+
C(1+|\b|^4) 3^{-d(m-n)} 
\,.
\end{align}
\end{lemma}
\begin{proof}
Fix~$m,n$ as in the statement. By the triangle inequality, 
\begin{equation}
\label{e.variance.presplit}
\var\bigl [ \a(\cu_m) \bigr] 
\leq 
2\var \Biggl[ \avsum_{z\in 3^n\Zd\cap\cu_m}
\a(z+\cu_n) \Biggr] 
+
2 \E \Biggl[
\biggl| \a(\cu_m) - \!\!\!\! \avsum_{z\in 3^n\Zd\cap\cu_m}
\a(z+\cu_n) \biggr|^2 \Biggr]
\,.
\end{equation}
For the first term on the right side, we apply Lemma~\ref{l.J.upperfluct} to obtain, for every~$\beta m < n < m$, 
\begin{align*}
\biggl| \avsum_{z\in 3^{n} \Zd\cap\cu_{m}} 
\a(z+\cu_{n}) 
- 
\ahom(\cu_n) \biggr|
\leq
\O_\Psi \left( C3^{-\frac d2(m-n)} \right)
\,.
\end{align*}
In view of~\eqref{e.Young.growth}, this implies that, for every~$m,n\in\N$ with~$\beta m<n<m$, 
\begin{equation}
\label{e.varbounda}
\var \Biggl[ \avsum_{z\in 3^n \Zd\cap\cu_{m}} 
\a(z+\cu_{n}) \Biggr] 
\leq 
C 3^{-d (m-n)}
\,.
\end{equation}
The second term on the right side of~\eqref{e.variance.presplit} is estimated in two different ways. 
First, we may use~\eqref{e.wrap.around}, to obtain that, 
for any matrix~$\b \in \R^{d\times d}_{\sym}$,
\begin{equation}
\label{e.wrap.app}
\biggl| \a(\cu_m) -  \avsum_{z\in 3^n\Zd\cap\cu_m}
\a(z+\cu_n) \biggr|
\leq 
C \avsum_{z\in 3^n\Zd\cap\cu_m}
\sum_{i=1}^d
J(z+\cu_n,e_i,\b e_i) 
\,.
\end{equation}
Using Lemma~\ref{l.J.upperfluct} with~$l=n$ and~\eqref{e.Young.growth}, we obtain 
\begin{equation}
\label{e.variance.quad.smack}
\E \Biggl[
\biggl| \a(\cu_m) - \avsum_{z\in 3^n\Zd\cap\cu_m}
\a(z+\cu_n) \biggr|^2
\Biggr]
\leq
C \sum_{i=1}^d\E \bigl[ J(\cu_n,e_i,\b e_i) \bigr]^2 + C(1+|\b|)^4 3^{d(n-m)}
\,.
\end{equation}
Alternatively, we may use the nonnegativity of
\begin{equation*}
\avsum_{z\in 3^n\Zd\cap\cu_m}
\a(z+\cu_n)
-
\a(\cu_m)\,,
\end{equation*}
which is due to subadditivity, to find that 
\begin{align}
\label{e.alternatively.tauk}
\E \Biggl[
\biggl| \a(\cu_m) - \avsum_{z\in 3^n\Zd\cap\cu_m} 
\a(z+\cu_n) \biggr|^2
\Biggr]
&
\leq
C\E \Biggl[
\biggl| \a(\cu_m) - \avsum_{z\in 3^n\Zd\cap\cu_m} 
\a(z+\cu_n) \biggr|
\Biggr]
\notag \\ & 
=
C\biggl|  \E \biggl[
\a(\cu_m) - \avsum_{z\in 3^n\Zd\cap\cu_m} 
\a(z+\cu_n) 
\biggr]\biggr|
\notag \\ & 
=
C \bigl| \ahom(\cu_m) - \ahom(\cu_n) \bigr|
=
C \sum_{k=n+1}^m \tau_k
\,.
\end{align}
This completes the proof of the estimate for~$\var[ \a(\cu_m)]$. The bound for~$\var[ \a_*^{-1}(\cu_m)]$ is proved similarly, and so we omit the argument. 
\end{proof}

Combining the result of the previous two lemmas, we find that, for every~$k,m,n\in\N$ with~$m\beta < k < n< m$,
\begin{align}
\label{e.combo}
\lefteqn{
\sup_{e\in B_1}
\E \left[ J(\cu_m,e,\ahom_*(\cu_m)e) \right]
} \ \ & 
\notag \\ &
\leq 
C 
\sum_{j=n}^{m} 3^{j-m} 
\left( \tau_j + \var\left[ \a_*^{-1}(\cu_j) \right] \right)
+C 3^{-(m-n)}
\notag \\ & 
\leq 
C
\Bigl( \sum_{j=k}^m \tau_j  \Bigr)\wedge 
\sup_{e\in B_1 }
\E \left[ J(\cu_k,e,\ahom_*(\cu_k) e) \right]^2
+
C 
\bigl( 3^{-(m-n)} + 3^{-d(n-k)} \bigr)
+
C\sum_{j=n}^m 3^{j-m} \tau_j
\,.
\end{align}
There is an expectation of~$J$ on the left and the~\emph{square} of an expectation of~$J$ on the right. This is very good! The rest are garbage error terms and some~$\tau_j$'s, which we anticipated in the heuristic discussion around~\eqref{e.Km.naivewish}. 
Thus we should expect that if we can make the quantity~$\sup_{e\in B_1 }
\E \left[ J(\cu_k,e,\ahom_*(\cu_k) e) \right]$ small enough, then an iteration argument should yield an algebraic rate of decay for it. 

\smallskip

This intuition is formalized in the following elementary but technical lemma. 

\begin{lemma} 
\label{l.iteration}
Let~$\{ F_j \}_{j=0}^\infty\subseteq\R_+$  be a nonincreasing sequence. 
Suppose that~$A \in [1,\infty)$ and~$\alpha,\delta,\kappa \in (0,\tfrac12]$ are such that~$F_m$ satisfies
\begin{equation}
\label{e.Fiter.ass}
\left\{
\begin{aligned}
& F_m \leq A \sum_{k=1}^m 3^{(k-m)\alpha} (F_{k-1} - F_k)  
+  
A F_{\lfloor (1-\kappa) m \rfloor}^2   
+   
\delta 3^{-\alpha m} 
\,, \quad \forall m \in\N\,, 
\\ & 
F_0 \leq \delta 
\,.
\end{aligned}
\right.
\end{equation}
Define~$\gamma := \frac{\alpha}{20A}$. 
If~$\delta \leq 10^{-6} A^{-2} \alpha^2$, then we have
\begin{equation}
\label{e.Fiter.res}
F_n \leq  1200A \delta \alpha^{-1}  3^{-n\gamma}
\,, \quad \forall n\in\N
\,.
\end{equation}
\end{lemma}
\begin{proof} 
Fix~$\gamma \in (0,\frac12 \alpha]$ to be selected. It is convenient to work with the composite quantity 
\begin{equation*}
G_m := \sum_{k=0}^m 3^{(k-m)\gamma} F_k
\,.
\end{equation*}
We will first transform the assumption~\eqref{e.Fiter.ass} into an inequality for~$G_m$. 
The claim is that 
\begin{equation}
\label{e.Gn.preiter}
G_n \leq 
\frac{5A}{\alpha} (G_{n-1} - G_{n} +  3^{-n\gamma} F_{0})
+
2A G_{\lfloor (1-\kappa) n \rfloor}^2 
+
5\delta \alpha^{-1} 3^{-n\gamma}\,.
\end{equation}
To derive~\eqref{e.Gn.preiter}, we multiply~\eqref{e.Fiter.ass} by~$3^{(m-n)\gamma}$ and sum over~$m \in\{ 0,\ldots,n\}$. The left side becomes~$G_n$. The first term on the right side is estimated by interchanging the order of the sums as follows:
\begin{align*} 
A \sum_{m=1}^n 3^{(m-n)\gamma} 
\sum_{k=1}^m 3^{(k-m)\alpha} (F_{k-1} - F_k) 
& =
A \sum_{k=1}^n 3^{(k-n)\gamma}  (F_{k-1} - F_k)   
\sum_{m=0}^{n-k} 3^{-(\alpha - \gamma)m}  
\\ & 
\leq 
\frac{A}{1-3^{-\nicefrac \alpha2}} (G_{n-1} - G_{n} +  3^{-n\gamma} F_{0})
\\ & 
\leq 
\frac{5A}{\alpha} (G_{n-1} - G_{n} +  3^{-n\gamma} F_{0})
\,.
\end{align*}
For the second term on the right side, we have 
\begin{align*}
A \sum_{m=0}^n 
3^{(m-n)\gamma} F_{\lfloor (1-\kappa) m \rfloor}^2 
& \leq 
2A 
\sum_{k=1}^{\lfloor (1-\kappa) n \rfloor}
3^{( k -\lfloor (1-\kappa) n \rfloor )(1-\kappa)^{-1} \gamma} 
F_{k}^2 
\\ & 
= 
2A 
\sum_{k=1}^{\lfloor (1-\kappa) n \rfloor}
3^{- ( k -\lfloor (1-\kappa) n \rfloor ) ( 2 - \frac{1}{1-\kappa} ) \gamma}
\bigl( 
3^{( k -\lfloor (1-\kappa) n \rfloor ) \gamma}
F_{k} \bigr)^2 
\\ & 
\leq 
2A G_{\lfloor (1-\kappa) n \rfloor}^2
\,.
\end{align*}
For the third term on the right side, we compute
\begin{equation*}
\sum_{ m=0}^n 3^{ (m-n)\gamma} \delta 3^{-\alpha m} 
=
\delta 3^{-n\gamma} \sum_{ m=0}^n 3^{-(\alpha-\gamma)m} 
\leq 
\frac{\delta  3^{-n\gamma}}{1-3^{-\nicefrac \alpha2}}
\leq 
5\delta \alpha^{-1}  3^{-n\gamma}  \,.
\end{equation*}
Combining these, we obtain~\eqref{e.Gn.preiter}.

\smallskip

We rearrange~\eqref{e.Gn.preiter} and use~$F_0 \leq \delta$ to obtain, for~$\theta := 5A (5A+\alpha)^{-1} < 1$,
\begin{equation}
\label{e.recurse.Gn}
G_n \leq 
\theta  G_{n-1} 
+
G_{\lfloor (1-\kappa) n \rfloor}^2 
+
10 \delta 3^{-n\gamma}\,.
\end{equation}
We now try to prove by induction that~$G_n$ satisfies the bound
\begin{equation*}
G_n \leq K \delta 3^{-n\gamma}\,, \qquad \forall n\in\N\,,
\end{equation*}	
where~$K$ is a parameter to be selected. Obviously this is valid for~$n =0$ and~$K=1$, by assumption, since~$G_0 = F_0 \leq \delta$. 
Assuming it holds for~$n \leq m$, we plug it into~\eqref{e.recurse.Gn} to get
\begin{align*}
G_{m+1} 
&
\leq 
\theta K \delta 3^{-m\gamma}
+
K^2 \delta^2 3^{-2 \lfloor (1-\kappa) m \rfloor \gamma}
+
10 \delta 3^{-m\gamma}
\leq
K\delta 3^{-(m+1) \gamma}
\Bigl( 
\frac{\theta K + 10}{K} 3^\gamma
+
K\delta 3^{2\gamma}
\Bigr)
\,.
\end{align*}
We discover therefore that we succeed, and may close the induction, provided that~$\gamma$ is chosen small enough that~$3^\gamma \theta \leq \frac{1+\theta}{2}$ (it suffices to take~$\gamma = \alpha(20A)^{-1}$), then the constant~$K$ is chosen so that~${10}{K} 3^\gamma \leq \frac14(1-\theta)$ (it suffices to take~$K=1200A/\alpha$), and finally~$\delta$ is chosen so small that~$K\delta 3^{2\gamma} \leq \frac14(1-\theta)$ (it suffices to take~$\delta = \frac{\alpha}{400KA}\geq \frac{\alpha^2}{500000 A^2}$). We obtain 
\begin{equation*}
F_n \leq G_n \leq \frac{120\delta}{1-\theta} 3^{-n\gamma}
\leq 
\frac{1200A\delta}{\alpha} 3^{-n\gamma}
\,, \qquad \forall n\in\N\,,
\end{equation*}
which completes the proof. 
\end{proof}

Before we can apply Lemma~\ref{l.iteration} to the quantity estimated in~\eqref{e.combo}, we need some initial smallness. 
We accomplish this in the following lemma, by a simple pigeonhole argument. In fact, it is rather easy to get a crude, logarithmic-type convergence rate for the quantity~$\sup_{e\in B_1 }\E \left[ J(\cu_k,e,\ahom_*(\cu_k) e) \right]$ by using monotonicity to argue that it cannot change too much on \emph{every} scale. It would otherwise run out of room! Note that since the length scale of~$\cu_n$ is~$3^n$ and~$n = c\log 3^n$, the estimate~\eqref{e.lograte} below is indeed a logarithmic rate of convergence (with a doubly logarithmic correction). 

\begin{lemma}[Pigeonhole lemma]
\label{l.pigeonhole}
Assume that~$\P$ is~$\Zd$--stationary and satisfies~$\CFS(\beta,\Psi)$ and~\eqref{e.symm}.
There exists a constant~$C(\dataref)<\infty$ such that, for every~$n\in\N$ with~$n \geq C  (1-\beta)^{-1}$, 
\begin{align}
\label{e.lograte}
\sup_{e\in B_1 }
\E \left[ J(\cu_{n},e,\ahom_*(\cu_{n}) e) \right]
\leq
\frac{C \log n}{n} + \exp \Bigl ( - \frac{(1-\beta)n}{C} \Bigr) 
\,.
\end{align}
\end{lemma}
\begin{proof}
By~\eqref{e.subaddcontrol}, we have that 
\begin{equation*}
\sum_{k=1}^\infty \tau_k 
\leq 
C\Bigl(  
\bigl| \E \left[ \a_*^{-1}(\cu_0) \right] - \ahom^{-1} \bigr| 
+ 
\bigl| \E \left[ \a(\cu_0) \right] - \ahom  \bigr| 
\Bigr)
\leq C\,.
\end{equation*}
The pigeonhole principle implies the existence of a scale~$k_0 \in\N$ satisfying 
\begin{align}
\label{e.kdelta}
\frac{N}{1-\beta}
\leq
k_0 
\leq 
\frac{N}{1-\beta} 
+ 
\frac{C N}{\delta}
\qquad \text{and} \qquad
\sum_{k=k_0}^{k_0+N}
\tau_k 
\leq 
\delta
\,. 
\end{align}
Applying~\eqref{e.combo}, we find that
\begin{equation*}
\sup_{e\in B_1}
\E \left[ J(\cu_{k_0+N},e,\ahom_*(\cu_{k_0+N})e) \right]
\leq 
C\delta + C3^{-N}.
\end{equation*}
Setting~$N := C | \log \delta |$ and~$m:=k_0+N$,
we have shown the existence of a constant~$C(d)<\infty$ and, for every~$\delta>0$, an integer~$m \in\N$ satisfying
\begin{equation*}
m
\leq 
\frac{2 N}{1-\beta} 
+ 
\frac{C N}{\delta}
\leq 
C 
\Bigl( \frac{1}{1-\beta} + \frac{C}{\delta} \Bigr) 
\left| \log \delta \right|
\end{equation*}
and
\begin{equation*}
\sup_{e\in B_1 }
\E \left[ J(\cu_{m},e,\ahom_*(\cu_{m}) e) \right]
\leq C\delta\,. 
\end{equation*}
Due to the monotonicity of~$n\mapsto \E \left[ J(\cu_{n},p,q) \right]$ and the fact that~$q\mapsto  \E \left[ J(\cu_{n},p,q) \right]$ has its minimum at~$q=\ahom_*(\cu_n)p$ (see~\eqref{e.Jaas} and~\eqref{e.ahomahoms}), this estimate implies~\eqref{e.lograte}. 
\end{proof}

\begin{proof}[{Proof of Proposition~\ref{p.algebraicrate.E}}]
We verify the assumption of Lemma~\ref{l.iteration} with the quantity
\begin{align}
\label{e.Fm.def}
F_m
& 
:= 
\sum_{i=1}^d \E \left[ J(\cu_{m+m_0},e_i,\ahom_*(\cu_{m+m_0})e_i) \right]
\,.
\end{align}
Recall that, by~\eqref{e.subaddcontrol}, we have
\begin{equation}
\label{e.Finc.bytau}
F_m - F_{m+1}
\geq 
c
\tau_{m+ m_0 + 1}
.
\end{equation}
We now fix parameters~$\kappa = \frac12 \wedge (1-\beta)$,~$k = \lceil (1-\kappa) (m+m_0) \rceil$,~$n = \lfloor \frac{k + m + m_0}{2} \rfloor$  and~$\theta = \frac14 \kappa$. 
Combining~\eqref{e.combo} and~\eqref{e.Finc.bytau} yields~\eqref{e.Fiter.ass} with these choices of parameters. The condition on~$\delta$ is satisfied by Lemma~\ref{l.pigeonhole} provided that we take~$m_0$ to be sufficiently large, and this~$m_0$ depends only on~$\data := (\beta,d,\lambda,\Lambda,\CFS,\Psi)$. The result of Lemma~\ref{l.pigeonhole} then yields~\eqref{e.Fiter.res} for~$F_m$ defined in~\eqref{e.Fm.def}.
That is, we have shown that, for constants~$\alpha(d,\beta,\lambda,\Lambda)>0$ and~$C(\dataref)<\infty$,
\begin{equation}
\label{e.EJ.scalemat}
\sum_{i=1}^d \E \bigl[ J(\cu_{m},e_i,\ahom_*(\cu_{m})e_i) \bigr]
\leq
C 3^{-\alpha m} \,.
\end{equation}
To conclude, we need to replace the matrix~$\ahom_*(\cu_{m})$ in~\eqref{e.EJ.scalemat} with~$\ahom$. 
For this we use~\eqref{e.magic.b.plug.E} to find that 
\begin{equation*}
\E \bigl[ J(\cu_{m},e,\ahom_*(\cu_{m})e) \bigr]
=
\frac12 e \cdot \bigl( \ahom(\cu_m) - \ahom_*(\cu_m) \bigr) e
\geq 
\frac 12 e \cdot \bigl( \ahom - \ahom_*(\cu_m) \bigr) e
\,.
\end{equation*}
We deduce from this and~\eqref{e.EJ.scalemat} that 
\begin{equation*}
| \ahom - \ahom_*(\cu_m) |
\leq C 3^{-\alpha m} \,.
\end{equation*}
Combining this estimate with~\eqref{e.EJ.scalemat} yields~\eqref{e.EJtozero.rate} and completes the proof. 
\end{proof}

\subsubsection{Fluctuation bounds}

Proposition~\ref{p.algebraicrate.E} gives an algebraic rate of convergence for our subadditive quantities in~$L^1$ in the probability space. Since the subadditive quantities are uniformly bounded, this, of course, immediately implies convergence of all finite moments with an explicit algebraic rate. That is, by~\eqref{e.a.bounds}~\eqref{e.diagonalset}, and~\eqref{e.EJtozero.rate}, we have, for every~$p< \infty$ and~$m\in\N$, 
\begin{align}
\label{e.giveup}
\E \bigl[ \left| \a(\cu_m) - \ahom \right|^p \bigr]^{\nicefrac1p}
+
\E \bigl[ \left| \a_*(\cu_m) - \ahom \right|^p \bigr]^{\nicefrac1p}
\leq 
C 3^{-\nicefrac{m\alpha}{p}}
\,.
\end{align}
This may seem already quite satisfactory, yielding nice quantitative homogenization statements. 
However, without exerting ourselves too much further, we can obtain \emph{optimal} estimates in terms of stochastic integrability---namely the statement of Theorem~\ref{t.subadd.converge}. This is very important to the theory, and it is a perk of working with the dual pair of subadditive quantities. 

\begin{proof}[{Proof of Theorem~\ref{t.subadd.converge}}]
\label{proof.t.subadd.converge}
According to Lemma~\ref{l.J.upperfluct} 
and Proposition~\ref{p.algebraicrate.E},
for every~$e \in B_1$,
\begin{equation*}
J(\cu_m,e,\ahom e) 
\leq 
\E \left[ J( \cu_{n},e,\ahom e) \right]
+ 
\O_\Psi\bigl( C3^{-\frac d2(m-n)} \bigr)
\leq 
C3^{-n\alpha} + 
\O_\Psi\bigl( C3^{-\frac d2(m-n)} \bigr)
\,. 
\end{equation*}
The estimate~\eqref{e.aastar.big.smash} follows from this and~\eqref{e.diagonalset}. 
\end{proof}

\begin{remark}
The exponent~$\alpha$ in the statement of Theorem~\ref{t.subadd.converge} can be taken to be equal to the one in Proposition~\ref{p.algebraicrate.E}.  
\end{remark}

We next state a slightly souped-up version of Theorem~\ref{t.subadd.converge}. 
We have seen throughout this section and as well as previously (see for instance~\eqref{e.flatnessqual.sec24} and~\eqref{e.flatnessqual.flux.sec24}) that it is convenient to work with multiscale quantities such as the following one, defined for each~$m\in\N$ by 
\begin{equation} 
\label{e.mcE.0}
\mathcal{E}(m) 
:=
\sum_{n=0}^{m}
3^{n-m} 
\biggl( 
\avsum_{z\in 3^n\Zd\cap \cu_m}
\!\!
\Bigl( \bigl| \a(z+\cu_n) - \a_*(z+\cu_n)  \bigr|
+
\bigl| \a_*(z+\cu_n) - \ahom \bigr|^2 \Bigr) 
\biggr)^{\!\nicefrac12} 
\,.
\end{equation}
In view of~\eqref{e.magic.b.plug}, there exists a constant~$C(d,\lambda,\Lambda)<\infty$ such that 
\begin{equation}
\label{e.J.by.mcE}
C^{-1} \mathcal{E} (m) 
\leq
\sum_{n=0}^{m}
3^{n-m} 
\biggl( 
\sum_{i=1}^d
\avsum_{z\in 3^n\Zd\cap \cu_m}
\!\!
J(z+\cu_n,e_i,\ahom e_i) 
\biggr)^{\!\nicefrac12} 
\leq C \mathcal{E} (m) 
\,.
\end{equation}
The quantity~$\mathcal{E}(m)$ keeps track of how close the coarse-grained matrices are to the homogenized matrix in subcubes of the form~$z+\cu_n \subseteq \cu_m$, discounted by the factor of~$3^{n-m}$.
The motivation for defining~$\mathcal{E}(m)$ in terms of the coarse-grained matrices across many scales is because it naturally appears when one estimates the weak norms of the gradient and flux of the finite-volume correctors (see below in the proof of Theorem~\ref{t.quant.DP}) which, as we have already seen, is what we need to control to obtain estimates of the homogenization error for the Dirichlet problem. The discount factor~$3^{n-m}$ comes from the Poincar\'e inequality or, to be more precise, Proposition~\ref{p.MSP}.  
Also, rather than control the size of~$\mathcal{E}(m)$ directly, we seek to estimate a \emph{minimal scale} at which it becomes small---and continues to decay algebraically.
The proof is an extrapolation of Theorem~\ref{t.subadd.converge}.

\begin{corollary}
\label{c.subadd.converge}
Assume that~$\P$ is~$\Zd$--stationary and satisfies~$\CFS(\beta,\Psi)$ and~\eqref{e.symm} and let~$\alpha(\beta,d,\lambda,\Lambda) \in (0,1]$ be the exponent in the statement of Proposition~\ref{p.algebraicrate.E}. 
Then, for every~$\rho \in [\beta , 1)$, there exists~$C(\rho,\alpha,\dataref)<\infty$ such that 
\begin{equation}
\label{e.Em.summedout}
\mathcal{E}(m)^2
\leq 
C3^{-\rho \alpha m } + \O_\Psi\bigl( C 3^{-\frac d2(1-\rho)  m} \bigr) \,.
\end{equation}
Moreover, for every~$\delta\in (0,1]$ and~$\theta < \frac{\alpha d}{d+2\alpha} \wedge \frac d2(1-\beta)$, there exists~$C(\delta,\gamma,\dataref)<\infty$ and a random variable~$\X$ satisfying 
\begin{equation}
\label{e.mmmbound0}
\X^{\frac d2 (1-\beta)}
= \O_\Psi(C)
\end{equation}
such that, for every~$m,n\in\N$ with~$n<m$ and~$3^n\geq \X$, 
\begin{equation} 
\label{e.aastar.concentrate.DD}
\mathcal{E}(m)^2
\leq \delta^2 3^{-\theta(m-n)}
\,.
\end{equation}
\end{corollary}

The corollary roughly says that, on every scale larger than the random \emph{minimal scale}~$\X$, the coarse-grained coefficients have homogenized, up to error~$\delta$, and that on scales larger than~$\X$ the homogenization error begins to decrease algebraically. Notice that the right side of~\eqref{e.aastar.concentrate.DD} is \emph{deterministic}---the only random object in the statement of Corollary~\ref{c.subadd.converge}, other than the random variables~$\mathcal{E}(m)$ we are estimating, is the minimal scale~$\X$. 
The size of this minimal scale is quantified by~\eqref{e.mmmbound0}, which says that~$\X$ is typically of order one and has tails that behave like 
\begin{equation*}
\P \left[ \X > t \right] 
\leq 
\Psi\bigl( ct^{\frac d2(1-\beta)} \bigr)^{-1}. 
\end{equation*}
Many estimates in stochastic homogenization have this feature: rather than bounding scale-dependent random variables with a random right-hand side, we bound them with deterministic right-hand sides, but only on scales larger than a random minimal scale. This choice is often equivalent but nearly always more general. 

\smallskip

\begin{remark}[{Optimality of~\eqref{e.mmmbound0}}]
\label{r.X.optimal}
We emphasize that this estimate for the stochastic integrability of~$\X$ is sharp. 
For instance, in the case that~$\P$ satisfies either~$\FRD(1)$ or~$\LSI(0,\rho)$, then, by~\eqref{e.URDyes} and~\eqref{e.LSI.okay}, we have~\eqref{e.mmmbound0} \& \eqref{e.aastar.concentrate.DD} valid for~$\beta=0$ and~$\Psi(t):= C\exp\left( -ct^2 \right)$. 
In particular, by~\eqref{e.diagonalset} and the fact that~$\mathcal{E}(m) \geq \sup_{|e|=1} J(\cu_m,e,\ahom e)$, we obtain that, for every~$\delta>0$, there exists~$c(\delta,\rho,d,\lambda,\Lambda)>0$ such that 
\begin{equation} 
\label{e.oppt}
\P \bigl[  
\left| \a(\cu_m) - \ahom \right|
+ \left| \a_*(\cu_m) - \ahom \right|
> \delta 
\bigr]
\leq \exp \left( -c\left| \cu_m \right| \right). 
\end{equation}
To see that this is optimal, consider a random checkerboard with two colors (black and white) and each square determined by a fair coin flip. The probability that the restriction of the checkerboard to~$\cu_m$ is colored completely white (or completely black) has probability~$2^{-|\cu_m|} = \exp\left( -c\left| \cu_m\right| \right)$. If the white squares correspond to~$\a(x) =  \Id$ and the black squares correspond to~$\a(x) =2 \Id$, then an all-white checkerboard in~$\cu_m$ would imply~$\a(\cu_m) = \a_*(\cu_m) =  \Id$ and an all-black checkerboard would imply~$\a(\cu_m) = \a_*(\cu_m) =2  \Id$. Therefore the probability of a deviation from~$\ahom$ of size at least~$\nicefrac12$ must be no smaller than~$2^{-|\cu_m|} = \exp\left( -c\left| \cu_m\right| \right)$. This matches the upper bound in~\eqref{e.oppt} up to the constant~$c$ in the exponential. 
\end{remark}

Let us turn to the proof of Corollary~\ref{c.subadd.converge}. 
Since we will use the argument in the proof of the corollary several times, we will state a more general lemma that formalizes it in a somewhat flexible way. Note that the argument uses only the nonnegativity and the subadditivity of the quantity~$J$ (and in particular does not require the assumption that~$\a$ be symmetric). 

\begin{lemma}
\label{l.mathcalE.minscale}
Assume that~$\P$ is~$\Zd$--stationary and satisfies~$\CFS(\beta,\Psi)$. 
Suppose that there exist~$\alpha \in (0,1)$ and~$K\in[1,\infty)$ such that 
\begin{align}
\label{e.EJtozero.rate.ass}
\sup_{e\in B_1}
\E \left[ J(\cu_m,e,\ahom e) \right] 
\leq 
K 3^{-m\alpha}. 
\end{align}
Then, for each~$\rho \in [\beta , 1)$, there exists~$C(\rho,\alpha,K,\dataref)<\infty$ such that, for every~$m\in\N$, 
\begin{equation}
\label{e.Em.summedout.pre}
\mathcal{E}(m)^2
\leq 
C3^{-\rho \alpha m } + \O_\Psi\bigl( C 3^{-\frac d2(1-\rho)m} \bigr) \,.
\end{equation}
Moreover, for every~$\delta>0$ and exponent~$\theta$ satisfying 
\begin{equation}
\label{e.theta.restriction}
0 < \theta < 
\min \Bigl\{ 
\frac{\alpha d}{d+2\alpha} 
\,,\, 
\frac d2(1-\beta) 
\Bigr\}\,,
\end{equation}
there exist a constant~$C(\theta,\delta,K,\dataref)<\infty$ and a random variable~$\X$ satisfying 
\begin{equation}
\label{e.minscaleint}
\X^{\frac d2 (1-\beta)}
=
\O_{\Psi}(C)\,,
\end{equation}
such that, for every~$m \in \N$ with~$3^m\geq \X$,
\begin{equation}
\label{e.convssmaxsmax}
\mathcal{E}(m) ^2
\leq
\delta^2 
\Bigl( \frac{3^m}{\X} \Bigr)^{\!-\theta}.
\end{equation}
\end{lemma}
\begin{proof}
By~\eqref{e.J.by.mcE}, Lemma~\ref{l.J.upperfluct} and the assumption~\eqref{e.EJtozero.rate.ass}, we obtain, for every~$m,l \in \N$ with~$\beta m < l < m$,
\begin{align}
\label{e.Emestimate}
\mathcal{E}(m)^2
&
\leq
C\sum_{n=0}^{m}
3^{n-m} \sum_{i=1}^d \avsum_{z \in 3^n\Zd\cap \cu_m} 
 J(z+\cu_n,e_i,\ahom e_i)
\notag \\ & 
\leq 
\sum_{n=0}^{m}
3^{n-m}
\Bigl( 
C3^{-(n\wedge l)\alpha} 
+
\O_\Psi \bigl( C3^{-\frac d2(m-l)} \bigr)
\Bigr)
\leq
C3^{-l\alpha} + \O_\Psi \bigl( C3^{-\frac d2(m-l)} \bigr)
\,.
\end{align}
Note that in the last inequality, we used the assumption that~$\alpha<1$. Taking~$l:= \lceil \rho m\rceil +1$ for~$\rho \in [\beta,1)$, we obtain
\begin{equation}
\label{e.mathEsqr.bound}
\mathcal{E}(m)^2
\leq 
C3^{-\rho \alpha m } + \O_\Psi\bigl( C 3^{-\frac d2(1-\rho)m} \bigr) \,.
\end{equation}
This is~\eqref{e.Em.summedout.pre}. 

\smallskip 

We turn to the proof of the last statement of the lemma. 
Fix~$\delta\in (0,1]$ and~$\theta$ satisfying 
\begin{equation}
\label{e.theta.box}
0 < \theta < 
\min \Bigl\{ 
\frac{\alpha d}{d+2\alpha} 
\,,\, 
\frac d2(1-\beta) 
\Bigr\}
\,. 
\end{equation}
Define 
\begin{equation}
\label{e.def.of.X}
\X := \sup\biggl\{ 3^N \,:\, 
N \in \N \ \mbox{and} \ 
\exists m\in\N \cap [N,\infty)
\,,\ \mathcal{E}(m)^2 >  \delta^2 3^{-\theta(m-N)}
\biggr\} \,.
\end{equation}
It is clear from this definition that~\eqref{e.convssmaxsmax} holds for every~$m\in \N$ with~$3^m\geq \X$. 
To prove the estimate~\eqref{e.minscaleint}, we 
fix~$N\in \N$ with~$N\geq N_0$, where~$N_0$ is defined to be the smallest positive integer satisfying 
\begin{equation}
\label{e.N0.box}
N_0 \geq \theta^{-1} \log_3 \bigl( 2 \delta^{-2} C_{\eqref{e.mathEsqr.bound}}\bigr)
\,.
\end{equation}
We consider, for each~$m\in\N$ with~$m\geq N$,
\begin{equation}
\label{e.rho.m.box}
\rho(m):= 
\max \biggl\{ 
\beta 
\, , \,
\frac{\theta(m-N)+\log_3 ( 2 \delta^{-2} C_{\eqref{e.mathEsqr.bound}})}{\alpha m}
\biggr\} \,.
\end{equation}
It is clear from~\eqref{e.theta.box},~\eqref{e.N0.box} and~\eqref{e.rho.m.box} that~$\rho(m) \in  [\beta, \frac{d}{d+2\alpha})\subseteq [\beta,1)$ and
\begin{equation*}
\max \Bigl\{ 
C_{\eqref{e.mathEsqr.bound}} 3^{-\rho(m) \alpha m }
\,,\, 
C_{\eqref{e.mathEsqr.bound}} 3^{-\frac d2(1-\rho(m)) m }
\Bigr\} 
\leq \frac12 \delta^2 3^{-\theta(m-N)} \,, \quad \forall m\in\N\cap [N,\infty)\,.
\end{equation*}
Applying~\eqref{e.mathEsqr.bound} and using the previous display, we find that 
\begin{equation}
\label{e.mathEsqr.bound.app}
\mathcal{E}(m)^2
\leq 
\frac12 \delta^2 3^{-\theta(m-N)}  + \O_\Psi\bigl( C 3^{-\frac d2(1-\rho(m))m} \bigr) 
\,.
\end{equation}
and hence 
\begin{equation*}
\P \Bigl[ \mathcal{E}(m)^2 >  \delta 3^{-\theta(m-N)} \Bigr] 
\leq 
\Bigl( \Psi \bigl( \tfrac 12 C_{\eqref{e.mathEsqr.bound.app}}^{-1} \delta^2 3^{-\theta (m-N)+\frac d2(1-\rho(m))m}\bigr) \Bigr)^{-1} 
\,.
\end{equation*}
We can write this as 
\begin{equation*}
\indc_{\{ \mathcal{E}(m)^2 >  \delta^2 3^{-\theta(m-N)} \}} 
\leq 
\O_{\Psi} \bigl( C3^{\theta (m-N)-\frac d2(1-\rho(m))m} \bigr) 
\end{equation*}
Summing this over~$m\geq N$ yields
\begin{equation*}
\indc_{\{ \X \geq 3^N \}} 
\leq
\sum_{m=N}^\infty
\indc_{\{ \mathcal{E}(m)^2 >  \delta 3^{-\theta(m-N)} \}} 
\leq 
\O_{\Psi} \biggl( C\sum_{m=N}^\infty3^{\theta (m-N)-\frac d2(1-\rho(m))m} \biggr) \,.
\end{equation*}
To conclude the proof, we need to demonstrate that 
\begin{equation}
\label{e.in.hell}
\sum_{m=N}^\infty3^{\theta (m-N)-\frac d2(1-\rho(m))m}
\leq 
C 3^{-\frac d2(1-\beta)N}\,.
\end{equation}
We break the sum on the left side into two parts: first the~$m$'s for which~$\rho(m) = \beta$ and second the~$m$'s satisfying~$\rho(m) = \frac{\theta(m-N)+C}{\alpha m}$. The first part is easy to estimate, as~$\theta < \frac d2(1-\beta)$: 
\begin{equation*}
\sum_{m=N}^\infty
3^{\theta (m-N)-\frac d2(1-\beta)m}
=
3^{-\theta N} \sum_{m=N}^\infty
3^{(\theta -\frac d2(1-\beta))m}
\leq 
C 3^{-\frac d2(1-\beta)N}\,.
\end{equation*}
For the second part, we observe that~$\rho(m) = \frac{\theta(m-N)+C}{\alpha m}$ only if~$(\frac \theta \alpha -\beta)m > \frac \theta \alpha N - C$. 
In this case the exponent of~$3$ on the left side of~\eqref{e.in.hell} is
\begin{equation*}
\theta (m-N)-\frac d2(1-\rho(m))m + C
= m \biggl( \theta -  \frac d2 \Bigl( 1 - \frac \theta \alpha \Bigr) \biggr) - N \biggl( \theta + \frac d2 \frac \theta\alpha\biggr) + C \,.
\end{equation*}
The factor of~$m$ is negative, since~$\theta < \frac {d\alpha}{d+2\alpha}$. Therefore we obtain
\begin{align*}
\lefteqn{ 
\sum_{m=N}^\infty
3^{\theta (m-N)-\frac d2(1-\rho(m))m} 
\indc_{\{ \rho(m) = \frac{\theta(m-N)+C}{\alpha m} \}} 
} \qquad & 
\notag \\ & 
\leq
\sum_{m \geq \frac \theta \alpha (\frac \theta \alpha -\beta)^{-1} N-C}
C3^{m ( \theta -  \frac d2 ( 1 - \frac \theta \alpha ) ) - N ( \theta + \frac d2 \frac \theta\alpha)} 
=
C3^{( \theta -  \frac d2 ( 1 - \frac \theta \alpha ) )\frac \theta \alpha (\frac \theta \alpha -\beta)^{-1} N - ( \theta + \frac d2 \frac \theta\alpha)N } 
\,.
\end{align*}
The exponent of~$3$ on the right side above satisfies 
\begin{align*}
\underbrace{
\Bigl( \theta -  \frac d2 \Bigl( 1 - \frac \theta \alpha \Bigr) \Bigr)
}_{< 0} 
\underbrace{ 
\frac \theta \alpha \Bigl(\frac \theta \alpha -\beta\Bigr)^{-1}
}_{\geq 1 }
N - \Bigl( \theta + \frac d2 \frac \theta\alpha\Bigr)N 
&
\leq 
\Bigl( \theta -  \frac d2 \Bigl( 1 - \frac \theta \alpha \Bigr) \Bigr)  N - \Bigl( \theta + \frac d2 \frac \theta\alpha\Bigr)N 
\notag \\ & 
= -\frac d2 N \,.
\end{align*}
The proof is now complete. 
\end{proof}

For later use, we include here the remark that the minimal scale~$\X$ in Lemma~\ref{l.mathcalE.minscale} has nice locality properties---it can be replaced by a family of~$\cu_k$--measurable minimal scales, and its Malliavin derivative is bounded. 

\begin{lemma} 
\label{l.localX}
Fix~$k \in \N$. Under the assumptions of Lemma~\ref{l.mathcalE.minscale}, satisfied for every~$m \leq k$, there exists a constant~$C(\dataref)<\infty$ and a minimal scale~$\X_k$ such that~$\X_k$ is~$\F(\cu_k)$-measurable and satisfies 
\begin{equation}
\label{e.minscaleint.local}
\X_k^{\frac d2 (1-\beta)}
= 
\O_{\Psi}(C) \qand
\X_k \leq 3^{k+1}\,,
\end{equation}
and if~$\X_k < 3^k$, then, for every~$m \in \N$ with~$j \leq m \leq k$ and~$3^j \geq \X_k$,
\begin{align}
\label{e.convssmaxsmax.local}
\mathcal{E}(m)^2
\leq
\delta^2 3^{-\theta(m-j)}.
\end{align}
Moreover, the Malliavin derivative of~$\X_k$ is bounded by~$\X_k$ itself:
\begin{align} \label{e.mallliavin.local}
\bigl| \partial_{\a(\cu_k)} \X_k  \bigr| \leq C (1+\X_k) .
\end{align}
\end{lemma}
\begin{proof}
The definition of~$\X$ in~\eqref{e.def.of.X} is clearly sensitive to small changes in the coefficient field, so we need to modify it. 
We select a smooth cutoff function~$\zeta:[0,\infty) \to [0,\infty)$ satisfying 
\begin{equation*}
\zeta = 0 \ \text{on} \ [0,\nicefrac12] 
\,, \quad
\zeta = 1 \ \text{on} \ [1,\infty)  
\,, \quad
| \zeta' | \leq 3
\,,
\end{equation*}
and define, for every~$k\in\N$,
\begin{equation*}
\X_k 
:= 
\max_{0 \leq N \leq k} 
3^N
\max_{N \leq m \leq k} 
\zeta \Bigl( \frac{\mathcal{E}(m)^2}{\delta^2 3^{-\theta(m-N)} } \Bigr) 
\,.
\end{equation*}
It is clear from comparing its definition to~\eqref{e.def.of.X} and using Lemma~\ref{l.mathcalE.minscale} that~$\X_k$ is~$\F(\cu_k)$--measurable and satisfies the assertions~\eqref{e.minscaleint.local} and~\eqref{e.convssmaxsmax.local}. 
The second part follows easily using the chain rule, the boundedness of~$J$ and~\eqref{e.malliavin.estimate}. \end{proof}

\subsection{Homogenization error estimates}
\label{ss.det}

In this section, we complete our first quantitative homogenization result, which we formalize in terms of the Dirichlet problem. 
As for the other results in this chapter, this statement focuses on an optimal estimate for the random length scale at which the homogenization error becomes small (and then decays algebraically from that point onwards).
Improved bounds on the rate exponent will be addressed later in Chapter~\ref{s.renormalization}. 
We also seek estimates that are \emph{uniform in the boundary data}; that is, we want to estimate the homogenization error for \emph{all} solutions at once. The argument we present is once again completely deterministic and reduces the homogenization error for a general boundary-value problem to the convergence of the coarse-grained coefficients in triadic subcubes.

\begin{theorem}
\label{t.quant.DP}
Assume~$\P$ satisfies~$\CFS(\beta,\Psi)$ and~\eqref{e.symm}. 
Let~$\theta >0$ be as in the statement of Lemma~\ref{l.mathcalE.minscale} and let~$\delta>0$. Then there exists a constant~$C_1(\delta,\theta,\dataref)<\infty$ and a random variable~$\X$ satisfying 
\begin{equation}
\label{e.mmmbound1}
\X^{\frac d2 (1-\beta)}
= \O_\Psi(C_1)
\end{equation}
and, for every Lipschitz domain~$U\subseteq \cu_0$, a constant~$C(U,d,\lambda,\Lambda)<\infty$ such that, for every~$\ep \in (0,\X^{-1} ]$,~$r \in (0,1]$ and~$u\in H^1(U) \cap W^{2,\infty}(U_r)$, where we denote
\begin{equation*}
U_r := \{x \in U \, : \, \dist(x,\partial U)> r\}\,,
\end{equation*}
if~$u^\ep\in H^1(U)$ is the solution of the Dirichlet problem
\begin{equation*}
\left\{
\begin{aligned}
& -\nabla \cdot \left( \a\left(\tfrac \cdot\ep\right) \nabla u^\ep \right) = -\nabla \cdot \ahom \nabla u &  \mbox{in} & \ U, \\
& u^\ep = u & \mbox{on} & \ \partial U,
\end{aligned}
\right.
\end{equation*}
then we have the estimate
\begin{multline}
\label{e.quant.DP}
\| \nabla u^\ep - \nabla u \|_{H^{-1}(U)} 
+ 
\| \a(\tfrac \cdot \ep) \nabla u^\ep - \ahom \nabla u \|_{H^{-1}(U)} 
\\ 
\leq
C \| \nabla u \|_{{L}^2(U \setminus U_r)} 
+ 
\delta r^{-1} 
(\X \ep)^{\frac{\theta}2} 
\| \nabla u \|_{L^\infty(U_r)} 
\,.
\end{multline}
\end{theorem}

The free parameter~$r$ allows for flexibility when applying Theorem~\ref{t.quant.DP}. The rate at which~$\| \nabla u \|_{{L}^2(U \setminus U_r)}$ tends to zero depends, of course, on the regularity of the boundary trace of~$u$ on~$\partial U$. By standard results in elliptic regularity (see~\cite[Proposition B.16]{AKMBook}) we know that, if~$\nabla u \in W^{1,\infty}(U)$  for instance, then~$\| \nabla u \|_{{L}^2(U \setminus U_r)} \lesssim r^{\nicefrac12}\| \nabla u \|_{{L}^2(U)}$. In this case, the choice~$r = \delta^{\nicefrac12} (\X \ep)^{\nicefrac \theta2}$ yields 
\begin{equation*}
\| \nabla u^\ep - \nabla u \|_{H^{-1}(U)} 
+ 
\| \a(\tfrac \cdot \ep) \nabla u^\ep - \ahom \nabla u \|_{H^{-1}(U)} 
\leq
\delta^{\nicefrac12} 
(\X \ep)^{\nicefrac \theta 2}
\| \nabla u \|_{L^\infty(U)} ,
\end{equation*}
and so after renaming~$\delta \leftarrow \delta^{\nicefrac12}$ and~$\theta \leftarrow \frac12\theta$, we have made the first term on the right side of~\eqref{e.quant.DP} disappear at the cost of replacing~$U_r$ by~$U$ in the norm. 
We can perform a similar trick whenever~$u \in W^{1,p}(U)$ for~$p>2$, using the H\"older inequality to show that the~$L^2$ norm of~$\nabla u$ in the boundary layer vanishes at least like a power of~$r$. Thanks to the global Meyers estimate (see~\cite[Appendix C]{AKMBook}) this is very often the case! However, not every~$u$ will have such regularity near the boundary, and the point of stating the estimate as we have done is to have the freedom to allow~$r$ to depend on~$u$. 
Note that the assumed~$W^{2,\infty}$ interior regularity of~$u$ can be relaxed---see~\cite[Section 6.3]{AKMBook} for sharp results. 

\smallskip

The proof of Theorem~\ref{t.quant.DP} is easy because we have already done all the work for it: we just need to glue together the statements of~Proposition~\ref{p.DP} and Corollary~\ref{c.subadd.converge}, with the help of some estimates already proved in Section~\ref{ss.variational}, namely~\eqref{e.flatnessqual.sec24} and~\eqref{e.flatnessqual.flux.sec24}.

\begin{proof}[{Proof of Theorem~\ref{t.quant.DP}}]
Recall that in Section~\ref{ss.variational} we defined finite-volume correctors~$\phi_{m,e}$ and flux correctors~$\bfs_{m,e}$ in the cube~$\cu_m$ by~\eqref{e.FVC.def} and~\eqref{e.fluxcorrect.sec24}.
In the notation of the present section, we take~$\phi_{m,e}$ to be the difference of the maximizer~$v(\cu_m,-e,0)$ of~$J(\cu_m,-e,0)$ and the affine function~$\ell_e$, and let~$\bfs_{m,e}$ be an anti-symmetric matrix potential for the flux~$\a\nabla v(\cdot,\cu_m, -e,0) - \a(\cu_m)e$. Note that the latter has zero mean in~$\cu_m$ by~\eqref{e.a.astar.formulas}.
According to~\eqref{e.flatnessqual.sec24} and~\eqref{e.flatnessqual.flux.sec24},
we have that, for every~$|e|=1$,
\begin{align*}
\lefteqn{ 
\| \nabla \phi_{m,e} \|_{\Hminusul(\cu_m)} 
+
\| \a (e +\nabla \phi_{m,e}) - \ahom e \|_{\Hminusul(\cu_m)} 
} \qquad & 
\notag \\ & 
\leq
C
+ 
C
\sum_{n=0}^{{m-1}} 3^{n} 
\biggl| 
 \a(\cu_m)-
\avsum_{z\in 3^n\Zd\cap \cu_{m}} 
\a(z+\cu_n)
\biggr|^{\nicefrac12}
\,.
\end{align*}
Note that we use the penultimate lines~\eqref{e.flatnessqual.sec24} and~\eqref{e.flatnessqual.flux.sec24}, but not the last line of these displays, which are too brutal and result in a loss of the exponent. 
We have a better way to estimate these now, using both coarse-grained quantities together instead of a crude application of the triangle inequality to compare scales. Here, we use~\eqref{e.add.defect.a} and~\eqref{e.magic.b.plug}, which give 
\begin{align*}
\biggl| 
\a(\cu_m)- \!
\avsum_{z\in 3^n\Zd\cap \cu_{m}} 
\a(z{+}\cu_n)
\biggr|
&
\leq 
 2 \avsum_{z\in 3^n\Zd\cap \cu_m}
\sum_{i=1}^d 
J(z{+}\cu_n,e_i,\ahom e_i) \Id
\notag \\ & 
\leq
C \!\!\avsum_{z\in 3^n\Zd\cap \cu_m} \!\!
\bigl( \bigl| \a(z{+}\cu_n) - \a_*(z{+}\cu_n)  \bigr|
+
\bigl| \a(z{+}\cu_n) - \ahom \bigr|^2 \bigr) 
\end{align*}
which, combined with the previous display and the definition of~$\mathcal{E}(m)$ in~\eqref{e.mcE.0}, yields
\begin{equation}
\label{e.FV.sublinearity.0}
3^{-m} \| \nabla \phi_{m,e} \|_{\Hminusul(\cu_m)} 
+
3^{-m} \| \a (e +\nabla \phi_{m,e}) - \ahom e \|_{\Hminusul(\cu_m)} 
\leq 
C3^{-m}  + C\mathcal{E}(m)\,.
\end{equation}
In view of Corollary~\ref{c.subadd.converge}, we deduce the existence, for each~$\delta>0$, of a random variable~$\X$ satisfying~\eqref{e.mmmbound0} such that~$3^{m} \geq \X$ implies that, for every~$|e|=1$, 
\begin{equation}
\label{e.FV.sublinearity}
3^{-m}\| \nabla \phi_{m,e} \|_{\Hminusul(\cu_m)} 
+
3^{-m} \| \a (e +\nabla \phi_{m,e}) - \ahom e \|_{\Hminusul(\cu_m)} 
\leq
C3^{-m}  
+ 
C \delta  \bigl( \X 3^{-m} \bigr)^{\frac{\theta}2}
\,.
\end{equation}
Here~$\theta>0$ is as in the statement of the corollary.
By~\eqref{e.L2toHminus} and~\eqref{e.dualest.forflux}, we deduce that, for every~$|e|=1$, the bounds
\begin{equation}
\label{e.fvc.bounds.symm}
3^{-m} \| \phi_{m,e} \|_{\underline{L}^2(\cu_m)} 
+
3^{-m} \| \bfs_{m,e} \|_{\underline{L}^2(\cu_m)} 
\leq
C \delta \bigl( \X 3^{-m} \bigr)^{\frac{\theta}2}
\,.
\end{equation}
Restating the result in terms of the rescaled correctors~$\phi_{m,e}^\ep:= \ep \phi_{m,e}(\tfrac \cdot \ep)$ and~$\bfs_{m,e}^\ep:= \ep \bfs_{m,e}(\tfrac \cdot \ep)$, we obtain, for every~$\ep \in (0,\X^{-1})$, 
\begin{equation*}
\| \phi_{m,e}^\ep \|_{{L}^2(\cu_0)} 
+
\| \bfs_{m,e}^\ep \|_{{L}^2(\cu_0)} 
\leq 
C \delta \bigl( \X \ep \bigr)^{\frac{\theta}2}
\,.
\end{equation*}
With~$u$ and~$u^\ep$ as in Theorem~\ref{t.quant.DP}, we may apply Proposition~\ref{p.DP} to obtain, for every~$\ep>0$ satisfying~$\ep < \X^{-1}$, the estimate
\begin{align*}
\left\| \nabla u^\ep {-} \nabla  u  \right\|_{H^{-1}(U)}  
+ \left\| \a\left(\tfrac\cdot\ep\right) \nabla u^\ep {-} \ahom \nabla  u  \right\|_{H^{-1}(U)}
\leq 
C \left\| \nabla u \right\|_{L^{2}(U \setminus U_{2r})} 
{+}
\frac{C\delta}{r}
\bigl( \X \ep \bigr)^{\frac{\theta}2}
\left\|  \nabla u \right\|_{W^{1,\infty}(U_r)}
\,.
\end{align*}
Note that, as mentioned in Section~\ref{ss.variational}, the statement of the proposition is written in terms of infinite-volume correctors, but its proof repeated verbatim for the finite-volume correctors yields the same estimate. This completes the proof of Theorem~\ref{t.quant.DP}. 
\end{proof}

In Chapter~\ref{s.regularity}, we will need the following variant of~Theorem~\ref{t.quant.DP}. It is a straightforward corollary of the theorem, so we leave the proof for the reader. 

\begin{corollary}
\label{c.quant.DP.interior}
Assume~$\P$ satisfies~$\CFS(\beta,\Psi)$ and~\eqref{e.symm}. 
Let~$\theta >0$ be as in the statement of Lemma~\ref{l.mathcalE.minscale}. Then for every~$\delta>0$, there exists~$C(\delta,\dataref)<\infty$ and a random variable~$\X$ satisfying 
\begin{equation}
\label{e.mmmbound2}
\X^{\frac d2 (1-\beta)}
= \O_\Psi(C)
\end{equation}
such that, for every~$r \geq \X$ and~$u \in H^1(B_{2r})$ satisfying the equation~$-\nabla \cdot\a\nabla u = 0$ in~$B_{2r}$, for every~$x \in \cu_0$ there exists a solution~$\bar{u} \in H^1(B_r(x))$ of~$\nabla \cdot \ahom\nabla \bar{u} = 0$ in~$B_r(x)$ satisfying the estimate
\begin{equation*}
\| u - \bar u \|_{\underline{L}^2(B_r(x))} 
+  
\|\a \nabla u - \ahom \nabla \bar u \|_{\underline{H}^{-1}(B_r(x))} 
\leq 
\delta \Bigl( \frac{\X}{r} \Bigr)^{\!\frac \theta2}
\| u \|_{\underline{L}^2(B_{2r})} 
\,.
\end{equation*}
\end{corollary}

\subsection*{Historical remarks and further reading}

The quantity~$\mu(U,p)$ was first introduced in 1973 by De Giorgi and Spagnolo~\cite{DGS} in one of the first mathematical papers\footnote{See~\cite[pp. 361--379]{DG} for an English translation of~\cite{DGS}.}
on homogenization. They proved the qualitative homogenization for periodic elliptic operators (stated in the language of so-called \emph{$G$--convergence}), which was based on, in our notation,  
the convergence of~$\mu(\cu_m,p)$ to~$\overline{\mu}(p)$ as~$m\to \infty$. 
This result was generalized to the case of general stationary-ergodic environments in~1986 by Dal~Maso and Modica~\cite{DM1,DM2}, again using~$\mu(\cu_m,p)$ and a version of the subadditive ergodic theorem to obtain the convergence of~$\mu(\cu_m,p)$ to a deterministic limit and consequently obtaining the qualitative homogenization of the Dirichlet problem. 

\smallskip

Each of these homogenization results was qualitative. However, the desire to prove quantitative results was already present in~\cite{DGS}. Indeed, a conjecture was stated in~\cite{DGS} with quantitative results in mind (see just before their Theorem~3.4). In our notation, it asserted an estimate for the difference of~$\a(\cu_m)$ and~$\b(\cu_m)$ for two coefficient fields~$\a$ and~$\b$ in terms of the differences of the quantity for~$\a(z+\cu_n)$ and~$\b(z+\cu_m)$ over all family of triadic subcubes~$\{ z+ \cu_n\,:\, z\in 3^n\Zd, \ -\infty < n \leq n_0\}$, for some fixed~$n_0 < m$. 
While we believe this conjecture to be false,\footnote{One needs both~$\a(U)$ as well as its fraternal twin~$\a_*(U)$, for instance, for a statement like this.} it demonstrates the surprising fact that a multiscale, iterative argument for obtaining quantitative convergence of~$\mu(\cu_m,p)$ to~$\ahom$ was already anticipated nearly fifty years ago. 

\smallskip

Most of the results of this chapter, including the introduction of the dual subadditive quantity and the iteration up the scales to get the algebraic convergence rate, originated in~\cite{AS}. That paper presented the result for general nonlinear equations with variational structure (the coefficients given by the gradient of convex functions). The method has subsequently been generalized and used to obtain quantitative results in many different contexts, including: 
\begin{itemize}
\item General nonlinear equations~\cite{AM};

\item Uniformly parabolic equations with coefficients varying in space and time~\cite{ABM}; 

\item Supercritical (bond) percolation clusters~\cite{AD1}; 

\item Interacting particle systems~\cite{GGM,GGMN,FGW};

\item Linear elasticity with random, anisotropic materials~\cite{SZ1};

\item Gradient interface models in statistical physics~\cite{D1,AW,DW,AD2};

\item Local laws for Coulomb gases~\cite{ArmSer};

\item Homogenization of differential forms~\cite{D3};

\item Effective viscosity in sedimentation~\cite{DG0};

\item Homogenization of high contrast equations~\cite{AK.HC,ABK.SD}.

\end{itemize}

This chapter is the first presentation of the coarse-graining ideas for such general mixing assumptions, although the fact that the renormalization ideas were flexible, and that such a result was possible to prove, was already well-known (see~\cite[page xxiv]{AKMBook}); results under~$\alpha$-mixing conditions were also proved previously in~\cite{AM}.

We will see below in Section~\ref{ss.almostone} that the small exponent~$\alpha>0$ in the statements of Theorem~\ref{t.subadd.converge}, Proposition~\ref{p.algebraicrate.E} and Corollary~\ref{c.subadd.converge} can be improved, using a renormalization argument, to any~$\alpha$ in the set~$(0, \frac d2(1-\beta)] \cap (0,1)$. This result, which was first proved in~\cite{AKM1} in the case of finite range of dependence, is sharp for the convergence of~$\a(\cu_m)$ and~$\a_*(\cu)$ to~$\ahom$, due to the presence of boundary layers.

\section{Coarse-graining for general linear elliptic equations}
\label{s.nonsymm}

In this chapter, we generalize the approach developed in the previous section. We will consider the case that the coefficient field~$\a(\cdot)$ is not necessarily symmetric; that is, we are going to remove the extra assumption~\eqref{e.symm} present in the previous section. In Section~\ref{ss.rhs}, we will also include the case of equations with a right-side in divergence form. 

\subsection{Coarse-grained matrices for equations with nonsymmetric coefficients}
\label{ss.nonsymm}

We henceforth introduce the symmetric (and positive)~$\s$ and anti-symmetric~$\k$ parts of the coefficient field~$\a$, which are defined by 
\begin{equation}
\label{e.a.symmetric.parts}
\s := \frac12(\a+\a^t)
\quad \mbox{and} \quad 
\k := \frac12(\a-\a^t)\,. 
\end{equation}
Notice that, since~$\s>0$, the uniform ellipticity condition~\eqref{e.ue} can be written as 
\begin{equation}
\label{e.ellipticity.nonsymm}
\s^{-1} \leq \lambda^{-1}  \Id 
\quad \mbox{and} \quad
\s + \k\s^{-1}\k^t \leq \Lambda  \Id\,.
\end{equation}
Indeed, the second condition in~\eqref{e.ue} is equivalent to~$\a^t \a \leq \Lambda \s$, which is equivalent to
\begin{equation*}
\Lambda  \Id \geq \a^t \s^{-1} \a  =  
 (\s-\k) \s^{-1} (\s+\k) = \s + \k^t \s^{-1} \k \,.
\end{equation*}

As mentioned above, the divergence-form equation~$-\nabla \cdot \a \nabla u = 0$ in the symmetric case can be interpreted physically as an electrostatics problem, with~$\a(x)$ representing the conductivity of a solid material at the point~$x$. The reader may then wonder: what is the physical motivation for considering nonsymmetric matrices in divergence-form equations?

\smallskip 

Tartar\footnote{Luc \textsc{Tartar}, French-American mathematician, born in 1946 in Paris. Tartar played a leading role in developing the theory of periodic homogenization in the 1980s before moving on to make fundamental contributions to other topics, such as hyperbolic conservation laws. His beautiful and idiosyncratic book~\cite{T} on homogenization, containing his reflections on the historical development of the topic, is highly recommended reading (especially for the footnotes).} 
poses this question on page 99 of his book~\cite{T} on homogenization theory next to his admission that, at the time he and Murat were developing the theory of periodic homogenization, they had no physical motivation in mind for the extension to the case of nonsymmetric coefficients. He then points to one possible physical motivation concerning the Hall effect in classical electrodynamics, which appears in the work on Milton. 

\smallskip

A less obscure answer, which Tartar does not mention, is that the skew-symmetric part~$\k$ of the coefficient field~$\a$ encodes an \emph{incompressible drift}. If~$\k$ is smooth enough, we may define~$\b := \nabla \cdot \k$ and compute 
\begin{equation*}
\nabla\cdot ( \k \nabla u) 
=
\sum_{i,j=1}^d 
\partial_{x_i} (  \k_{ij}\partial_{x_j} u) 
=
\sum_{i,j=1}^d 
(\partial_{x_i} \k_{ij})\partial_{x_j} u + \underbrace{\sum_{i,j=1}^d 
\k_{ij}\partial_{x_i} \partial_{x_j} u}_{=0}
=
\b \cdot \nabla u
\,.
\end{equation*}
In other words, the skew-symmetric part~$\k$ is the stream matrix for a divergence-free drift. Therefore the class of equations of the form~$-\nabla \cdot \a\nabla u=0$ for all elliptic (and  nonsymmetric) coefficient fields~$\a$ includes the \emph{advection-diffusion equation}
\begin{equation}
\label{e.passivescalar}
\partial_t u
-\kappa \Delta u + \b \cdot \nabla u = 0
\,,
\end{equation}
which arises in many physical models. In the context of fluid mechanics, it is often called \emph{the passive scalar equation}. 
The homogenization problem corresponding to~\eqref{e.passivescalar} is related to the study of \emph{advection-enhanced dissipation}. The drift term results in a homogenized matrix with a symmetric part that is larger---possibly much larger---than~$\kappa  \Id$. One is interested in the dependence of~$\ahom$ on the structure of the vector field~$\b$ and on~$\kappa$, particularly as~$\kappa$ becomes small, which is related to the development of turbulence in passive scalar transport. We refer the reader to the survey~\cite{MK} for much more on this topic. 

\smallskip

We would like to generalize the approach of Chapter~\ref{s.subadd}
to include equations like~\eqref{e.passivescalar} into our theory of quantitative homogenization. This may initially seem far from straightforward because these arguments are fundamentally \emph{variational}, the coarse-grained coefficients having been defined in terms of variational problems. Unfortunately, operators with nonsymmetric coefficients are not self-adjoint, and equations such as~\eqref{e.passivescalar} are called ``non-variational'' for a reason. 
Therefore, it is hopeless to generalize the ideas of Chapter~\ref{s.subadd}.

\smallskip

\emph{Fortunately, this is completely wrong!} All divergence form elliptic equations, even those corresponding to operators that are not self-adjoint, are variational and can be studied by means of the calculus of variations. There is no such thing as a ``non-variational'' divergence-form equation. The fact that this is not completely standard is an unfortunate historical quirk in the development of the calculus of variations. 
A thorough discussion of the variational interpretation of general divergence-form equations can be found in~\cite[Section 10.1]{AKMBook} or~\cite{AM}. 

\smallskip

We will not delve deeply into this background here but rather proceed with the generalization of the previous section. The presentation here is still self-contained, but some of the quantities we consider might be unexpected to readers unfamiliar with this less standard variational theory. 

\subsubsection{The coarse-grained coefficients in the nonsymmetric case}
\label{ss.subadd.doubletrouble.new}

If~$\k(\cdot)\neq0$, then the quantities~$\mu$ and~$\mu_*$ defined in~\eqref{e.mu} and~\eqref{e.mustar} are no longer natural because they depend only on the symmetric part of the coefficient field~$\a(\cdot)$; in particular, the extremizing functions are not solutions of the equation, but rather the equation with~$\s(\cdot)$ in place of~$\a(\cdot)$. 
It turns out that the right formula to generalize is~\eqref{e.variational.J}, which we now take as a definition: with~$\A(U)$ as in~\eqref{e.def.AU}, 
we define, for every~$p,q\in\Rd$, 
\begin{equation}
\label{e.variational.J.nonsymm}
J(U,p,q) 
:= 
\max_{v\in \mathcal{A}(U)} 
\fint_U \left( -\frac12 \nabla v\cdot \s\nabla v -p\cdot \a\nabla v + q\cdot \nabla v   \right).
\end{equation}
As before,~$v(\cdot,U,p,q)$ denotes the unique maximizer. 
Some but not all the properties of~$J$ listed in Lemma~\ref{l.J.basicprops} successfully generalize to the nonsymmetric setting. For instance,~$J$ is nonnegative and quadratic in~$(p,q)$, the first and second variation identities~\eqref{e.firstvar} and~\eqref{e.quadresp} are valid.
The maximizer of~\eqref{e.variational.J.nonsymm} is unique up to additive constants and 
\begin{equation}
\label{e.vUpq.linear.nosymm}
(p,q) \mapsto \nabla v(\cdot,U,p,q) \quad \mbox{is linear.}
\end{equation}
Of course,~$J$ continues to be subadditive as in~\eqref{e.subaddJ}. We will see later that, as in the symmetric case,~$J$ encodes sufficient information to define and control finite-volume correctors, which can then be used to obtain homogenization estimates.

\smallskip

We next introduce coarse-grained matrices~$\s(U)$,~$\s_*(U)$ and~$\k(U)$, which will be defined so that they satisfy 
\begin{equation}
\label{e.Jaas.nosymm}
J(U,p,q) =
\frac 12p \cdot \s(U)p 
+ \frac 12 (q+\k(U) p) \cdot \s_*^{-1}(U) (q+\k(U) p) 
- p \cdot q \,.
\end{equation}
The matrices~$\s(U)$ and~$\s_*(U)$ are symmetric, but~$\k(U)$ is not skew-symmetric, in general.
These are defined as follows: we first take~$\s_*(U)$ to be the symmetric matrix satisfying 
\begin{equation}
\label{e.sastU.def}
\frac 12 q\cdot \s_*^{-1}(U)  q = J(U,0,q)\,.
\end{equation}
We next define~$\k(U)$ by
\begin{equation}
\label{e.mU.def}
q \cdot \s_*^{-1}(U) \k(U) p =
J(U,p,q) - J(U,0,q) - J(U,p,0) +p \cdot q \,,
\end{equation}
and, finally, we take~$\s(U)$ to be the symmetric matrix satisfying
\begin{equation}
\label{e.sU.def}
\frac 12 p \cdot \s(U) p = 
 J(U,p,0) - \frac12p \cdot \bigl ( \k^t(U) \s_*^{-1}(U) \k(U)  \bigr ) p 	
 \,.
\end{equation}
From these definitions and a little algebra, using~\eqref{e.vUpq.linear.nosymm}, one obtains~\eqref{e.Jaas.nosymm}.
We also define
\begin{equation}
\label{e.def.a.nosymm}
\a(U) := \s(U) - \k^t(U) 
\quad \mbox{and} \quad
\a_*(U) :=  \s_*(U)  - \k^t(U)\,.
\end{equation}
Observe that these definitions agree with the ones in the symmetric case since~$\k(U)$ vanishes in that case. Finally, define
\begin{equation*}
\b(U) := \bigl ( \s + \k^t\s_*^{-1}\k\bigr )(U) 
\,.
\end{equation*}
We let~$\shom(U)$,~$\shom_*(U)$,~$\khom(U)$,~$\bhom(U)$,~$\ahom(U)$ and~$\ahom_*(U)$ denote deterministic matrices defined by:
\begin{equation}
\label{e.meet.the.homs}
\left\{
\begin{aligned}
& \shom_*(U) := \E \bigl[ \s_*^{-1}(U) \bigr]^{-1} \,, 
\\ & 
\khom(U) 
:= \shom_*(U)  \E \bigl[ \s_*^{-1}(U) \k(U) \bigr] \,, 
\\ & 
\bhom(U) := \shom(U) + \khom^t(U) \shom_*^{-1}(U) \khom(U)
=
\E \bigl[ \s(U) + \k^t(U) \s_*^{-1}(U) \k(U) \bigr]\,,
\\ & 
\ahom (U) :=  \shom(U) - \khom^t(U)\,,
\\ & 
\ahom_* (U) :=  \shom_*(U) -  \khom^t(U)\,.
\end{aligned}
\right.
\end{equation}
We see immediately that the first line of~\eqref{e.meet.the.homs} defines~$\shom_*(U)$, the second line defines~$\khom(U)$, and the third line defines~$\shom(U)$ and~$\bhom(U)$. 
We have defined these matrices so that taking the expectations of any of the formulas above amounts to just putting bars over the matrices; for instance, from~\eqref{e.Jaas.nosymm} we get
\begin{equation}
\label{e.Jaas.nosymm.E}
\E \bigl[ J(U,p,q) \bigr] =
\frac 12p \cdot \shom(U)p 
+ \frac 12 (q+\khom(U) p) \cdot \shom_*^{-1}(U) (q+\khom(U) p) 
- p \cdot q \,.
\end{equation}
We next collect some properties of the coarse-grained coefficients we have just defined, generalizing Lemma~\ref{l.J.basicprops}.

\begin{lemma}[Properties of the coarse-grained coefficients]
\label{l.J.basicprops.nosymm}
The following assertions are valid for every bounded Lipschitz domain~$U \subseteq\Rd$:

\begin{itemize} 
\item~$\s(U)$ and~$\s_*(U)$ are symmetric and~$\a(U)$ and~$\a_*(U)$ satisfy the same ellipticity bounds\footnote{The matrices~$\s(U)$ and~$\s_*(U)$ also satisfy the ordering~$\s_*(U)\leq \s(U)$, but we have postponed this to~\eqref{e.nonobvious.ordering} below since the proof requires variational principles, which are not as immediate as the other statements in this lemma.} as~$\a(x)$:
\begin{equation}
\label{e.a.bounds.nosymm}
\left\{
\begin{aligned}
& \lambda  \Id 
\leq
\biggl( \fint_{U} \s^{-1} (x)\,dx \biggr)^{\!\!-1} 
\leq 
\s_*(U) 
\,, \\
&
\b(U)
\leq 
\fint_U \bigl ( \s + \k^t \s_*^{-1} \k \bigr )(x)\, dx
\leq \Lambda \Id \,.
\end{aligned}
\right.
\end{equation}

\item The quantity~$J(U,p,q)$ and the matrices~$\s(U)$,~$\s_*(U)$ and~$\k(U)$ are~$\F(U)$--measurable. 

\item First variation: for every~$w\in \A(U)$,
\begin{equation}
\label{e.firstvar.nosymm}
q\cdot \fint_U \nabla w - p \cdot \fint_U \a \nabla w 
=
\fint_U \nabla w \cdot \s \nabla v(\cdot,U,p,q)\, . 
\end{equation}

\item Second variation and quadratic response: for every~$w\in \A(U)$, 
\begin{multline}
\label{e.quadresp.nosymm}
J(U,p,q) - \fint_U \Bigl  ( -\frac12 \nabla w \cdot \s\nabla w -p\cdot \a\nabla w+ q\cdot \nabla w   \Bigr  )
\\
=
\fint_U \frac12 \bigl ( \nabla v(\cdot,U,p,q) - \nabla w \bigr )\cdot \s\bigl ( \nabla v(\cdot,U,p,q) - \nabla w \bigr )\,.
\end{multline}

\item Characterization in terms of the energy of the maximizer: for every~$p,q\in\Rd$, 
\begin{equation}
\label{e.Jenergyv.nosymm}
J(U,p,q) = \fint_U \frac12 \nabla v(\cdot,U,p,q) \cdot \s \nabla v(\cdot,U,p,q).
\end{equation}

\item Characterization in terms of spatial averages of gradients \& fluxes: for every~$p,q\in\Rd$, 
\begin{equation}
\label{e.a.formulas.nosymm}
\left\{
\begin{aligned}
& \fint_U  \nabla v(\cdot,U,p,q)
=
-p +
\s_*^{-1} (U)
\bigl ( q + \k(U)p \bigr )
\,,\\
&
\fint_U  \a \nabla v(\cdot,U,p,q)
= 
\bigl (\Id -\k^t  \s_*^{-1}\bigr ) (U)  q
-
\bigl (\s +\k^t  \s_*^{-1} \k   \bigr )(U) p
\,.
\end{aligned}
\right.
\end{equation}

\item Subadditivity: for every~$m,n\in\N$ with~$n<m$ and~$p,q\in\Rd$, 
\begin{equation}
\label{e.subaddJ.nosymm}
J(\cu_m,p,q) 
\leq 
\avsum_{z\in 3^n\Zd \cap \cu_m} 
J(z+\cu_n,p,q)
\,.
\end{equation}

\item Coarse-graining flux inequality: for every~$w\in \A(U)$ and~$e\in\Rd$, 
\begin{equation}
\label{e.fluxmaps.nosymm}
\biggl| e \cdot
\fint_U \bigl( \a_*(U) - \a \bigr) \nabla w
\biggr|
\leq 
\left( \fint_U \nabla w \cdot \s \nabla w \right)^{\nicefrac12} 
\bigl( 2J(U, e , \a_*^t(U) e )\bigr )^{\nicefrac12} 
\,.
\end{equation}

\item Coarse-graining energy inequalities: for every~$u\in H^1(U)$, 
\begin{equation}
\label{e.energymaps.nonsymm}
\frac12\left( \fint_U \nabla u \right) \cdot \s_*(U) \left( \fint_U \nabla u \right)
\leq
\fint_U \frac12 \nabla u \cdot \s\nabla u 
\end{equation}
and
\begin{equation}
\label{e.energymaps.nonsymm.dual}
\frac12\left( \fint_U \a\nabla u \right) \cdot \b^{-1}(U) \left( \fint_U \a\nabla u \right)
\leq
\fint_U \frac12 \nabla u \cdot \s\nabla u 
\,.
\end{equation}
\end{itemize}
\end{lemma}

\begin{proof}
The bounds in~\eqref{e.a.bounds.nosymm} are obtained similarly to the proof of~\eqref{e.a.bounds} in the symmetric case.
We compute
\begin{equation*}
\frac 12q \cdot \s_*^{-1}(U) q 
= 
\max_{v\in \mathcal{A}(U)} 
\fint_U \left( -\frac12 \nabla v\cdot \s\nabla v + q\cdot \nabla v   \right)
\leq 
\fint_U \frac12 q\cdot \s^{-1} q \,.
\end{equation*}
This implies the lower bound on~$\s_*(U)$ in~\eqref{e.a.bounds.nosymm}. 
Similarly,
\begin{align*}
\frac 12 p \cdot \bigl ( \s + \k^t \s_*^{-1}\k \bigr )(U) p
&
= 
\max_{v\in \mathcal{A}(U)} 
\fint_U \left( -\frac12 \nabla v\cdot \s\nabla v -p\cdot \a\nabla v  \right)
\\ &
\leq 
\fint_U
\frac 12
\a^t p \cdot \s^{-1}\a^t p
=
\frac 12p \cdot 
\fint_U
\bigl ( \s + \k^t \s^{-1} \k \bigr ) p
\,.
\end{align*}
This gives the second line of~\eqref{e.a.bounds.nosymm}.

\smallskip

The proof of the first and second variations,~\eqref{e.firstvar.nosymm} and~\eqref{e.quadresp.nosymm}, also closely follows the symmetric case. The main difference is that the cross term in the computation below produces the symmetric part~$\s(\cdot)$ of~$\a(\cdot)$ instead of~$\a(\cdot)$ itself for the energy term.
Let~$v:= v(\cdot,U,p,q)$ be the maximizer in~\eqref{e.variational.J.nonsymm} and compute, for~$w\in \A(U)$,
\begin{align*}
\lefteqn{
J(U,p,q) -
\fint_U \left(
-\frac12\nabla (v+w) \cdot \s\nabla (v+w) - p\cdot \a\nabla (v+w) + q\cdot \nabla (v+w) 
\right)
} \qquad\qquad\qquad & 
\\ &
=
- \fint_U \left(
-\frac12\nabla w \cdot \s\nabla w - \nabla v\cdot \s\nabla w - p\cdot \a \nabla w + q\cdot \nabla w 
\right)
\,.
\end{align*}
This expression is nonnegative for every~$w\in\A(U)$, and so we may replace~$w$ by~$tw$ and~$-tw$ and then send~$t\to 0$ after dividing by~$t$ to obtain
\begin{equation*}
0 = \fint_U \left(
\nabla v\cdot \s\nabla w + p\cdot\a \nabla w - q\cdot \nabla w 
\right).
\end{equation*}
This is~\eqref{e.firstvar.nosymm}. Substituting it into the previous identity with~$w-v$ in place of~$w$ yields~\eqref{e.quadresp.nosymm}. 

\smallskip 

The identity~\eqref{e.Jenergyv.nosymm} is obtained from the combination of the definition~\eqref{e.variational.J.nonsymm} and the first variation~\eqref{e.firstvar.nosymm} applied with~$w=v(\cdot,U,p,q)$.

\smallskip

We next show that~\eqref{e.firstvar.nosymm} implies~\eqref{e.a.formulas.nosymm}, by polarization.
Using also~\eqref{e.Jenergyv.nosymm}  and~\eqref{e.Jaas.nosymm}, we find that 
\begin{align*}
\lefteqn{ 
q'\cdot \fint_U  \nabla v(\cdot,U,p,q)  - p' \cdot \fint_U \a \nabla  v(\cdot,U,p,q)
} \qquad & 
\notag \\ & 
= 
\fint_U \nabla v(\cdot , U, p',q') \cdot \s\nabla v(\cdot , U, p,q)
\notag \\ & 
=
\fint_U -\frac12  \nabla v(\cdot , U, p'-p,q'-q) \cdot \s\nabla v(\cdot , U, p'-p,q'-q)
\notag \\ & \qquad
+
\fint_U \frac12  \nabla v(\cdot , U, p,q) \cdot \s \nabla v(\cdot , U, p,q)
+
\fint_U 
\frac12 \nabla v(\cdot , U, p',q') \cdot \s\nabla v(\cdot , U, p',q') 
\\ & 
=J(U,p,q) + J(U,p',q') - J(U,p-p',q-q')
\\ & 
=
p'\cdot \s(U) p 
+
(q'+\k(U) p' )\cdot \s_{*}^{-1} (U) (q+\k(U)p) 
- p'\cdot q - q' \cdot p
\\ & 
=
q' \cdot \bigl(  \s_{*}^{-1} (U) (q+\k(U)p) - p \bigr) 
-p'\cdot \bigl( -  \s(U) p - \k^t(U) \s_{*}^{-1} (U) (q+\k(U)p)  + q  \bigr) 
\,.
\end{align*}
This implies~\eqref{e.a.formulas.nosymm}. 

\smallskip

Subadditivity~\eqref{e.subaddJ.nosymm} is immediate from the definition~\eqref{e.variational.J.nonsymm}.

\smallskip 

We next prove~\eqref{e.fluxmaps.nosymm}. 
Using~\eqref{e.firstvar.nosymm} and~\eqref{e.Jenergyv.nosymm}, we compute, for~$e\in \Rd$ and~$w \in \A(U)$,
\begin{align*}
\lefteqn{
\biggl |
e\cdot
\fint_{U} 
\bigl (
\a \nabla w -
(\s_*-\k^t)(U) \nabla w
\bigr )
\biggr |
} \qquad &
\\ &
=
\biggl |
\fint_U
\nabla w 
\cdot 
\s \nabla v\bigl (\cdot, U, e , (\s_*-\k)(U) e \bigr ) 
\biggr |
\\ &
\leq 
\biggl( \fint_U \nabla w \cdot \s \nabla w \biggr )^{\!\!\nicefrac12} \!
\biggl( \fint_U \nabla v\bigl (\cdot, U, e , (\s_*-\k)(U) e \bigr )  \cdot \s \nabla v\bigl (\cdot, U, e , (\s_*-\k)(U) e \bigr )  \biggr )^{\nicefrac12}
\\ &
\leq
\bigl(2J (U, e , (\s_*-\k)(U) e) \bigr ) ^{\nicefrac12}
\biggl( \fint_U \nabla w \cdot \s \nabla w \biggr )^{\!\!\nicefrac12}
\,.
\end{align*}
This is~\eqref{e.fluxmaps.nosymm}.

\smallskip 

We next give the proof of~\eqref{e.energymaps.nonsymm}, which is basically the same as the one of~\eqref{e.energymaps} in the symmetric case. We need to restrict the inequality to solutions instead of all~$H^1$ functions. 
By the definition of~$\s_*(U)$, we have that, for every~$u\in \A(U)$, if we set~$p:= \fint_U \nabla u$ then we have 
\begin{align*}
\frac12p \cdot \s_*(U) p
&
=
\sup_{q\in \Rd}
\Bigl( p\cdot q - \frac12 q \cdot \s_*^{-1}(U) q \Bigr)
\notag \\ & 
=
\sup_{q\in \Rd}
\Bigl( p\cdot q - J(U,0,q) \Bigr)
\notag \\ & 
=
\sup_{q\in \Rd}
\inf_{v\in \A(U)} 
\fint_U \Bigl( q\cdot (\nabla u - \nabla v) + \frac12\nabla v \cdot \s\nabla v \Bigr)
\leq 
\fint_U \frac 12\nabla u \cdot \s\nabla u
\,.
\end{align*}
This yields~\eqref{e.energymaps.nonsymm}.
For~\eqref{e.energymaps.nonsymm.dual} we set~$q:= -\,\fint_U \a \nabla u$ and compute 

\smallskip

To prove~\eqref{e.energymaps.nonsymm.dual}, we select~$u\in \A(U)$, set~$q:= -\,\fint_U \a \nabla u$ and compute 
\begin{align*}
\frac12q \cdot \b^{-1}(U) q
&
=
\sup_{p\in \Rd}
\Bigl( p\cdot q - \frac12 p \cdot \b(U) p \Bigr)
\notag \\ & 
=
\sup_{p\in \Rd}
\Bigl( p\cdot q - J(U,p,0) \Bigr)
\notag \\ & 
=
\sup_{p\in \Rd}
\inf_{v\in \A(U)} 
\fint_U \Bigl( p\cdot (\a\nabla v - \a\nabla u) + \frac12\nabla v \cdot \s\nabla v \Bigr)
\leq 
\fint_U \frac 12\nabla u \cdot \s\nabla u
\,.
\end{align*}
This completes the proof of the lemma. 
\end{proof}

\subsection{Variational structure and further properties}
\label{ss.doublevar}

Some key properties essential to the iteration argument in Section~\ref{ss.subadd.conv} may appear to be missing for the general~$J$ we defined in~\eqref{e.variational.J.nonsymm}, in comparison to the case of symmetric coefficients. In particular, as is evident from~\eqref{e.Jaas.nosymm}, in general, the dependence of~$J$ on~$p$ and~$q$ does not split into a form as nicely as we found in~\eqref{e.Jaas}. At first glance, there are no ``subquantities'' like~$\mu$ and~$\mu_*$, corresponding to the Dirichlet and Neumann problems.

\smallskip

These ``missing'' properties are variational in nature, and we will discover their generalizations by examining the variational structure of non-self adjoint equations. So, let us back up for a moment and derive the variational formulation for general (not necessarily self-adjoint) divergence-form elliptic equations, extending the Dirichlet principle in the self-adjoint case. We begin with the identity
\begin{align*}
\frac12 (\a p - q) \cdot \s^{-1}  (\a p - q)
=
A(p,q,\cdot) - p\cdot q
\,,
\end{align*}
where~$A$ is the quadratic form in~$2d$ variables defined by
\begin{align}
\label{e.Apq.def}
A(p,q,\cdot) 
:=
\frac12 p\cdot \s p + \frac 12 (q-\k p) \cdot \s^{-1}(q-\k p)
\,.
\end{align}
In particular, 
\begin{equation}
\label{e.Apq.equality.iff}
A(p,q,\cdot) \geq p\cdot q\,,  \quad \forall p,q\in\Rd\,, 
\quad \mbox{and} \quad 
A(p,q,\cdot) = p\cdot q
\iff
q = \a(x) p \,.
\end{equation}
It follows that 
\begin{align}
\label{e.yesvar}
u\in \A(U) & \iff
\a\nabla u \in \Lsol(U)
\iff 
\inf_{\g \in \Lsol(U)}
\int_U 
\bigl( A(\nabla u,\g,\cdot) - \nabla u\cdot \g \bigr)
= 0
\,.
\end{align}
Since the integrand on the right side of~\eqref{e.yesvar} is nonnegative by~\eqref{e.Apq.equality.iff}, we see that 
\begin{equation}
\label{e.yesyesvar.local}
u\in \A(U) 
\iff
\text{$u$ is a minimizer of }
w\mapsto \inf_{\g \in \Lsol(U)}
\int_U 
\bigl( A(\nabla w,\g,\cdot) - \nabla w\cdot \g \bigr)
\,.
\end{equation}
This is a variational principle! The reader can check that, in the case that~$\a(\cdot)$ is symmetric, the equivalence~\eqref{e.yesyesvar} reduces to the usual Dirichlet principle (and gives the dual Dirichlet principle for free). 

\smallskip

Since the term~$\int_U \nabla w\cdot \g$ depends only on~$\g$ and the trace of~$w$ on~$\partial U$, we may discard it, to obtain the equivalence 
\begin{equation}
\label{e.yesyesvar}
u\in \A(U) 
\iff
\text{$u$ is a local minimizer of }
w\mapsto \inf_{\g \in \Lsol(U)}
\int_U 
A(\nabla w,\g,\cdot)
\,.
\end{equation}
The right side of~\eqref{e.yesyesvar} is a uniformly convex minimization problem since the map~$(p,q) \mapsto A(p,q,\cdot)$ is uniformly convex. Indeed,~$A$  is a quadratic function of~$(p,q)$ which can be written in block matrix form as
\begin{equation}
\label{e.bfA.form}
A(p,q,\cdot) 
=
\frac12 \begin{pmatrix} p \\ q \end{pmatrix}
\cdot
\bfA(\cdot)
\begin{pmatrix} p \\ q \end{pmatrix}
\qquad \mbox{where} \qquad
\bfA := 
\begin{pmatrix} \s + \k^t\s^{-1}\k & -\k^t\s^{-1} \\ - \s^{-1}\k & \s^{-1} \end{pmatrix}
\,.
\end{equation}
Note also that~$\det \bfA = 1$ and its inverse is given by
\begin{equation}
\label{e.bfA.form.inv}
\bfA^{-1} = 
\begin{pmatrix} 
 \s^{-1}& - \s^{-1}\k  \\ - \k^t\s^{-1}&\s + \k^t\s^{-1}\k  \end{pmatrix}
 \,.
\end{equation}
Notice that the right side of~\eqref{e.Apq.def} bears a resemblance to~\eqref{e.Jaas.nosymm}, which is no coincidence.

\smallskip

This variational interpretation of general elliptic equations with nonsymmetric coefficients suggests an alternative way to generalize the coarse-grained quantities defined in Chapter~\ref{s.subadd}: directly coarse-grain the~$2d$-by-$2d$ matrix~$\bfA$ itself. Following this idea, we define, for each bounded domain~$U\subseteq\Rd$ and~$P ,Q \in \R^{2d}$, the quantity 
\begin{equation}
\label{e.bfJ.var}
\bfJ(U, P,Q )
:=
\sup_{X  \in \S(U)}
\fint_U
\biggl( -\frac12  X  \cdot \bfA X -  P \cdot \bfA X +Q\cdot X \biggr)
\,,
\end{equation}
where~$\S(U)$ is the linear subspace of~$\Lpot(U) \times \Lsol(U)$
defined by
\begin{equation}
\label{e.bfS}
\S(U) :=
\biggl\{
X \in \Lpot(U) \times \Lsol(U) \,:\, 
\int_U Y \cdot \bfA X=0
\; \;  \forall Y \in \Lpoto(U) \times \Lsolo(U)
\biggr\}
\,.
\end{equation}
We denote the unique maximizer in~$\S(U)$ of~$\bfJ(U,P,Q)$ by~$S(\cdot,U,P,Q)$. 

\smallskip

We may consider~$\bfJ$ as a ``double-variable'' version of~$J$ given in~\eqref{e.variational.J} and may proceed by straightforwardly extending the analysis of Chapter~\ref{s.subadd}.
This is a perfectly valid point of view which, somewhat surprisingly, turns out to be equivalent to the definitions in the previous section starting with~$J$ in~\eqref{e.variational.J.nonsymm}. 
The variational approach is therefore not an alternative to---but rather enriches our understanding of---the coarse-grained quantities introduced above in Section~\ref{ss.subadd.doubletrouble.new}. In the next lemma, we complete the list of properties presented in~Lemma~\ref{l.J.basicprops.nosymm} by adding the ones provided by this variational perspective. In unraveling this connection, a natural role is played by the adjoint operator~$-\nabla \cdot \a^t\nabla$, where~$\a^t$ is the transpose of~$\a(\cdot)$. We therefore introduce the analogous quantity for the adjoint operator by defining
\begin{equation}
\label{e.def.AsU}
\mathcal{A}^*(U) := 
\left\{ 
u\in H^1_{\mathrm{loc}}(U)
\,:\,
-\nabla \cdot \a^t\nabla u = 0 \quad \mbox{in} \ U \right\}
\end{equation}
and 
\begin{equation}
\label{e.variational.J.nonsymm.dual}
J^*(U,p,q) 
:= 
\max_{v^*\in \mathcal{A}^*(U)} 
\fint_U \left( -\frac12 \nabla v^*\cdot \s\nabla v^* -p\cdot \a^t\nabla v^* + q\cdot \nabla v^*   \right).
\end{equation}
The unique maximizer is denoted by~$v^*(\cdot,U,p,q)$.

\smallskip

\begin{lemma}[Further properties of the coarse-grained coefficients]
\label{l.J.basicprops.nosymm.moar}
The following assertions are valid for every bounded Lipschitz domain~$U \subseteq\Rd$:

\begin{itemize} 

\item There exists~$2d$-by-$2d$ real symmetric matrices~$\bfA(U)$ and~$\bfA_*(U)$ such that  
\begin{align}
\label{e.Jsplitting.nonsymm}
\bfJ(U,P,Q )
=
\frac12 P \cdot \bfA(U) P 
+ 
\frac12 Q \cdot \bfA_*^{-1}(U) Q
- P \cdot Q, \qquad \forall P,Q  \in \R^{2d}
\,.
\end{align}

\item \emph{Subadditivity.}
For every~$m,n\in\N$ with~$n<m$, 
\begin{equation}
\label{e.subadda.nosymm}
\bfA(\cu_m) 
\leq 
\avsum_{z\in 3^n\Zd \cap \cu_m} \!\!\!
\bfA(z+\cu_n) 
\quad \mbox{and} \quad 
\bfA_*^{-1} (\cu_m) 
\leq 
\avsum_{z\in 3^n\Zd \cap \cu_m} \!\!\!
\bfA_*^{-1}(z+\cu_n) 
\,.
\end{equation}

\item \emph{Formulas for~$\bfA(U)$ and~$\bfA_*(U)$ in terms of~$\s(U)$,~$\s_*(U)$ and~$\k(U)$.} We have
\begin{equation}
\label{e.matrices.formulas}
\left\{
\begin{aligned}
& 
\bfA(U)
= \begin{pmatrix} (\s + \k^t\s_*^{-1}\k)(U) & -(\k^t\s_*^{-1})(U) \\ -(\s_*^{-1}\k)(U) & \s_*^{-1}(U) \end{pmatrix}\,,
\\ & 
\bfA_*^{-1}(U)
= \begin{pmatrix} \s_*^{-1}(U) & -(\s_*^{-1}\k)(U) \\ -(\k^t\s_*^{-1})(U) & (\s + \k^t\s_*^{-1}\k)(U) 
\end{pmatrix}\,.
\end{aligned}
\right.
\end{equation}

\item \emph{Formula for~$\bfJ$ in terms of~$J$ and~$J^*$.} For every~$p,p^*,q,q^*\in\Rd$, 
\begin{equation}
\label{e.bfJ.byJJstar}
\bfJ
\biggl(U, \begin{pmatrix} p  \\ q \end{pmatrix}, \begin{pmatrix} q^* \\ p^* \end{pmatrix} \biggr)
=
\frac12 J\bigl(U,p-p^*,q^*-q\bigr)
+ 
\frac12 J^*\bigl(U,p^*+p,q^*+q\bigr)
\,.
\end{equation}

\item \emph{Formula for~$\S$ in terms of~$v$ and~$v^*$}. For every~$p,p^*,q,q^*\in\Rd$, 
\begin{equation}
\label{e.maximizers.J.to.bfJ}
S
\biggl(\cdot,U, \begin{pmatrix} p  \\ q \end{pmatrix}, \begin{pmatrix} q^* \\ p^* \end{pmatrix} \biggr)
=
\begin{pmatrix} \nabla v (\cdot,U,p{-}p^*,q^*{-}q)
+ 
\nabla v^*\bigl(\cdot,U,p^*{+}p,q^*{+}q\bigr) \\ 
\a \nabla v (\cdot,U,p{-}p^*,q^*{-}q)
- 
\a^t \nabla v^*\bigl(\cdot,U,p^*{+}p,q^*{+}q\bigr)
\end{pmatrix}
\,.
\end{equation}


\item \emph{Redundancy of the adjoint quantity.}
For every~$p,q\in \Rd$,
\begin{equation}
\label{e.Jaas.nosymm.star}
J^*(U,p,q) = 
\frac 12p \cdot \s(U)p 
+ \frac 12 (q -\k(U) p) \cdot \s_*^{-1}(U) (q -\k(U) p) 
- p \cdot q 
\end{equation}

\item \emph{Formula for the sum of~$J$ and~$J^*$.} For every~$p,q,h\in\Rd$,
\begin{align}
\label{e.JJstar1}
J(U,p,q-h) + J^*(U,p,q+h)
&
=
p \cdot (\s -\s_*)(U)p
\notag \\ & \qquad
+
\bigl (q - \s_*(U)p\bigr )\cdot \s_*^{-1}(U) \bigl (q - \s_*(U)p\bigr )
\notag \\ & \qquad
+
\bigl ( h-\k(U)p \bigr )  \cdot \s_*^{-1}(U)\bigl ( h-\k(U) p\bigr ) 
\,.
\end{align}

\item \emph{Ordering of~$\s(U)$ and~$\s_*(U)$.} We have
\begin{equation}
\label{e.nonobvious.ordering}
\s_*(U) \leq \s(U)
\,.
\end{equation}

\item \emph{Control of the symmetric part of~$\k(U)$.} We have
\begin{equation}
\label{e.ksym.by.sss}
(\k+\k^t)(U) \leq (\s-\s_*)(U)
\quad \mbox{and} \quad
-(\k+\k^t)(U) \leq (\s-\s_*)(U)\,.
\end{equation}

\item Coarse-graining inequality: coarse-graining estimate: for every~$w\in \A(U)$ and~$e\in\Rd$, 
\begin{equation}
\label{e.coarse.graining.nosymm}
\biggl| e\cdot \fint_U ( \a_*(U) - \a ) \nabla w \biggr| 
\leq 
2^{\nicefrac12} 
\bigl( e\cdot( \s(U) - \s_*(U))e \bigr)^{\nicefrac12} \biggl( \fint_U \nabla w \cdot \s \nabla w \biggr)^{\!\nicefrac12} 
\,.
\end{equation}

\item
Estimate of the Malliavin derivatives: there exists a constant~$C(d,\lambda,\Lambda)<\infty$ such that
\begin{equation} 
\label{e.maul.Mall}
\bigl| \partial_{\a(U)} \bfA(U) \bigr| + \bigl| \partial_{\a(U)} \bfA_*^{-1}(U) \bigr| \leq C 
 \,.
\end{equation}
\end{itemize}

\end{lemma}
\begin{proof}
We begin by characterizing the space~$\S(U)$ defined above in~\eqref{e.bfS} in terms of solutions of the equation and its adjoint. The claim is that 
\begin{equation}
\label{e.findS}
\S(U)
=
\left\{ (\nabla v+\nabla v^*, \a\nabla v - \a^t\nabla v^*) \,:\, v\in \A(U), \ v^*\in \A^*(U)
\right\}
\,.
\end{equation}
Observe that~$S=(\nabla u,\mathbf{h}) \in \Lpot(U)\times\Lsol(U)$ belongs to~$\S(U)$ if and only if 
\begin{align*}
& 
\int_U \bigl(
\nabla w \cdot \s \nabla u + (\mathbf{h}'-\k\nabla w) \cdot \s^{-1} (\mathbf{h} - \k\nabla u )
\bigr) = 0, \quad \forall (\nabla w,\mathbf{h}') \in \Lpoto(U) \times \Lsolo(U),
\end{align*}
and taking~$w=0$ in this condition yields that~$\s^{-1}(\mathbf{h}-\k\nabla u)$ is orthogonal to~$\Lsolo(U)$, which implies that it is a potential: there exists~$\nabla \tilde u\in \Lpot(U)$ such that~$\nabla \tilde u = \s^{-1}(\mathbf{h}-\k\nabla u)$. The condition above is, therefore, equivalent to 
\begin{align*}
& 
\exists \nabla u \in \Lpot(U), \quad 
\mathbf{h} = \s \nabla \tilde u +\k \nabla u, 
\quad \mbox{and} \quad 
\int_U \nabla w \cdot (\s \nabla \tilde u +\k \nabla u) = 0  \quad \forall w \in H^1_0(U)
\\ & \qquad 
\iff
\exists \nabla \tilde u \in \Lpot(U), \quad 
\mathbf{h} = \s \nabla \tilde u+\k \nabla u 
\quad \mbox{and}\quad
\s \nabla \tilde u +\k \nabla u \in \Lsol(U). 
\end{align*}
But if the latter holds, then~$v = \frac12(u + \tilde u)\in\A(U)$,~$v^* = \frac12(u-\tilde u)\in\A^*(U)$ and~$\mathbf{h} = \a\nabla v - \a^t \nabla v^*$. This is~\eqref{e.findS}. 

\smallskip

We next prove~\eqref{e.bfJ.byJJstar} and~\eqref{e.maximizers.J.to.bfJ}, using~\eqref{e.findS}. 
Observe that, for every~$p,p^*\in\Rd$, 
\begin{equation}
\label{e.iden.AP}
\bfA \begin{pmatrix} p   \\ \a p \end{pmatrix}
=
\begin{pmatrix} \a p   \\  p    \end{pmatrix}
\qand
\bfA \begin{pmatrix} p^*  \\  - \a^t p^* \end{pmatrix}
=
\begin{pmatrix} \a^t p^*  \\  - p^* \end{pmatrix}
\end{equation}
and
\begin{equation}
\label{e.iden.eqadjeq}
\frac12  \begin{pmatrix} p + p^* \\ \a p  - \a^t p^* \end{pmatrix} \cdot \bfA \begin{pmatrix} p + p^*  \\ \a p - \a^t p^* \end{pmatrix}
=
p \cdot \a p  +p^* \cdot \a^t p^*
\,.
\end{equation}
Hence, for every~$v,v^*\in H^1(U)$ and~$p,p^*,q,q^*\in\Rd$, 
\begin{align}
\label{e.iden.PAP}
& 
{-}\frac12  \begin{pmatrix} \nabla v{+}\nabla v^*  
\\ \a\nabla v {-} \a^t\nabla v^*  \end{pmatrix} 
\!\cdot\!
\bfA 
\begin{pmatrix} \nabla v{+}\nabla v^*  \\ 
\a\nabla v {-} \a^t\nabla v^*  \end{pmatrix} 
\! - \!
\begin{pmatrix} p \\ q \end{pmatrix} 
\!\cdot\!
\bfA
\begin{pmatrix} \nabla v{+}\nabla v^*  \\ 
\a\nabla v {-} \a^t\nabla v^* \end{pmatrix}
\! + \!
\begin{pmatrix} q^* \\ p^* \end{pmatrix}
\!\cdot\!
\begin{pmatrix} \nabla v{+}\nabla v^* \\ 
\a\nabla v {-} \a^t\nabla v^* \end{pmatrix}
\notag \\ & \qquad
=
- \frac12 \nabla v \cdot \a \nabla v - \frac12 \nabla v^* \cdot \a^t \nabla v^*
\! - \!
\begin{pmatrix} p \\ q \end{pmatrix} 
\!\cdot\!
\begin{pmatrix} \a\nabla v {+} \a^t\nabla v^* \\ \nabla v - \nabla v^*  \end{pmatrix}
\! + \!
\begin{pmatrix} q^* \\ p^* \end{pmatrix}
\!\cdot\!
\begin{pmatrix} \nabla v{+}\nabla v^* \\ 
\a\nabla v {-} \a^t\nabla v^* \end{pmatrix}
\notag \\ & \qquad
=
- \frac12 \nabla v \cdot \a\nabla v 
- (p-p^*)\cdot \a\nabla v
+(q^*-q)\cdot \nabla v
\notag \\ & \qquad \qquad
- \frac12 \nabla v^* \cdot \a^t \nabla v^*
-(p^*+p)\cdot\a^t\nabla v^* 
+(q^*+q)\cdot \nabla v^*
\,.
\end{align}
Taking the supremum of the expression on the right side over all pairs~$(v,v^*) \in \A(U) \times \A^*(U)$ yields the right side of~\eqref{e.bfJ.byJJstar}. If we take the same supremum of the expression on the left side, we obtain, in view of~\eqref{e.findS}, the right side of~\eqref{e.bfJ.var}. 
This yields~\eqref{e.bfJ.byJJstar} as well as~\eqref{e.maximizers.J.to.bfJ}.

\smallskip

Next, it is easy to check that 
\begin{equation}
\label{e.Slinear}
(P,Q) \mapsto S(\cdot,U,P,Q) \quad \mbox{is linear}
\end{equation}
and the first variation for~$\bfJ$ reads as follows:
for every~$T\in \S(U)$ and~$P,Q \in\R^{2d}$, 
\begin{equation}
\label{e.firstvar.nonsymm.dbl}
\fint_U (Q \cdot T - P \cdot \bfA T )
=
\fint_U T \cdot \bfA S(\cdot,U,P,Q ). 
\end{equation} 

\smallskip

We then prove the following variational characterization of~$\bfJ(U,P,0)$: for every~$P\in\R^{2d}$,  
\begin{equation}
\label{e.J.P0.Dir}
\bfJ(U,P,0)
= \inf \biggl\{ 
\fint_{U} \frac12   
(X + P) \cdot \bfA (X + P)
\, : \, X \in  \Lpoto(U) \times  \Lsolo(U)  \biggr\}
\,,
\end{equation}
and that if~$X(\cdot,U,P)\in  \Lpoto(U) \times  \Lsolo(U)$ denotes the minimizer, then
\begin{equation}
\label{e.J.P0.Dir.SX}
S(\cdot,U,P,0) = - \bigl (P + X(\cdot,U,P) \bigr  )\,.
\end{equation}
This is analogous to the proof of~\eqref{e.variational.J}, which boiled down to showing that~$\mu(U,p)$ was equal to the right side of~\eqref{e.variational.J} with~$q=0$, and that the maximizer for the right side of~\eqref{e.variational.J} with~$q=0$ is the minimizer of~$\mu(U,-p)$. Similarly, using the first variation~\eqref{e.firstvar.nonsymm.dbl}, we get 
\begin{equation}
\label{e.J.0Q.Neu}
\bfJ(U,0,Q)
= \sup \biggl\{ 
\fint_{U} \biggl( - \frac12 X \cdot \bfA X + X \cdot Q \biggr)
\, : \, X \in  \Lpot(U) \times  \Lsol(U)  \biggr\}
\,,
\end{equation}
and that if~$X^*(\cdot,U,Q)\in  \Lpot(U) \times  \Lsol(U)$ denotes the maximizer, then
\begin{equation}
\label{e.J.Q0.Dir.SX}
S(\cdot,U,0,Q) = X^*(\cdot,U,Q)\,.
\end{equation}

\smallskip

To prove~\eqref{e.J.P0.Dir} and~\eqref{e.J.P0.Dir.SX}, 
we first observe that~$P+X(\cdot,U,P)$, where~$X(\cdot,U,P)$ is the minimizer of the right side of~\eqref{e.J.P0.Dir}, belongs to~$\mathcal{S}(U)$, by the first variation for each of the quantities in~\eqref{e.J.P0.Dir} and the definition of~$\mathcal{S}(U)$ in~\eqref{e.bfS}. Moreover, we have that
\begin{equation*}
\fint_U 
(P+X(\cdot,U,P) )\cdot \bfA S(\cdot,U,P,0)
= 0 =
\fint_U 
X(\cdot,U,P) \cdot \bfA (P+X(\cdot,U,P,0))
\,.
\end{equation*}
We continue by writing~$X=X(\cdot,U,P)$ and~$S=S(\cdot,U,P,0)$, for short. It follows from the above and~\eqref{e.firstvar.nonsymm.dbl} that  
\begin{align*}
\fint_U
\frac 12 
\bigl( P{+}X{+}S \bigr) \cdot \bfA \bigl( P{+}X{+}S \bigr)
&
=
\fint_U \!\frac 12 \bigl( P{+}X \bigr) \cdot \bfA \bigl( P{+}X \bigr)
-
\bfJ(U, -P, 0 )
\notag \\ &
=
\fint_U \!
\Bigl( -\frac12 (P{+}X) \cdot \bfA (P{+}X) + P \cdot \bfA (P{+}X) \Bigr)
-
\bfJ(U, -P, 0 )
\notag \\ & 
\leq 0
\,.
\end{align*}
Thus~$P+X + S = 0$, which is~\eqref{e.J.P0.Dir.SX}, and we obtain 
\begin{equation*}
\bfJ(U, P, 0 )
=
\bfJ(U, -P, 0 )
=
\fint_U \frac 12 \bigl( P+X \bigr) \cdot \bfA \bigl( P+X \bigr)\,.
\end{equation*}
This completes the proof of~\eqref{e.J.P0.Dir}. The proof for~\eqref{e.J.0Q.Neu} and~\eqref{e.J.Q0.Dir.SX} is similar, and we omit it.

\smallskip

We next turn to the demonstration of~\eqref{e.Jsplitting.nonsymm}. We define~$\bfA(U)$ and~$\bfA_*^{-1}(U)$ to be the matrices satisfying
\begin{equation*}
\bfJ(U,P,0)
=
\frac12 P \cdot \bfA(U) P 
\quad \mbox{and} \quad 
\bfJ(U,0,Q)
=
\frac12 Q \cdot \bfA_*^{-1}(U) Q
\,,\quad \forall P,Q  \in \R^{2d}\,,
\end{equation*}
and we obtain~\eqref{e.Jsplitting.nonsymm} using the first variation. 
From~\eqref{e.Slinear} and~\eqref{e.firstvar.nonsymm.dbl}, we obtain
\begin{align}
\label{e.bfJ.formula.via.bfA}
\bfJ(U,P,Q) &
= \fint_U 
\frac12 S(\cdot,U,P,Q) \cdot \bfA S(\cdot,U,P,Q) 
\notag \\ & 
=
\fint_U 
\frac12 S(\cdot,U,P,0) \cdot \bfA S(\cdot,U,P,0) 
+
\fint_U 
\frac12 S(\cdot,U,0,Q) \cdot \bfA S(\cdot,U,0,Q) 
\notag \\ & \qquad 
+
\fint_U 
S(\cdot,U,P,0) \cdot \bfA S(\cdot,U,0,Q) 
\notag \\ & 
= 
\frac12 P \cdot \bfA(U) P
+
\frac12 Q \cdot \bfA_*^{-1}(U) Q
+
Q\cdot \fint_U 
S(\cdot,U,P,0) 
\,.
\end{align}
Meanwhile, by~\eqref{e.J.P0.Dir.SX}, we have that 
\begin{equation*}
\fint_U 
S(\cdot,U,P,0) 
=
-P\,.
\end{equation*}
This completes the proof of~\eqref{e.Jsplitting.nonsymm}. 

\smallskip 

The subadditivity property~\eqref{e.subadda.nosymm} follows from the variational representation~\eqref{e.bfJ.var} of~$\bfJ$ and the formula~\eqref{e.Jsplitting.nonsymm}.

\smallskip

We turn to the proof of the formula~\eqref{e.Jaas.nosymm.star}, which says that~$J^*$ contains the same information as~$J$. 
By~\eqref{e.bfJ.byJJstar} and~\eqref{e.Jsplitting.nonsymm}, we have that 
\begin{equation*}
J(U,p,0) = 
\frac12 \bfJ
\biggl(U, \begin{pmatrix} p  \\ 0 \end{pmatrix}, \begin{pmatrix} 0 \\ -p \end{pmatrix} \biggr)
=
\frac12 \bfJ
\biggl(U, \begin{pmatrix} p  \\ 0 \end{pmatrix}, \begin{pmatrix} 0 \\ p \end{pmatrix} \biggr)
=
J^*(U,p,0) 
\end{equation*}
and
\begin{equation*}
J(U,0,q) = 
\frac12 \bfJ
\biggl(U, \begin{pmatrix} 0  \\ -q \end{pmatrix}, \begin{pmatrix} q \\ 0 \end{pmatrix} \biggr)
=
\frac12 \bfJ
\biggl(U, \begin{pmatrix} 0  \\ q \end{pmatrix}, \begin{pmatrix} q \\ 0 \end{pmatrix} \biggr)
=
J^*(U,0,q) 
\,.
\end{equation*}
For the cross-term, we similarly compute
\begin{align*}
J\bigl(U,p,q\bigr)
+ 
J^*\bigl(U, p ,q\bigr)
&
=
2\bfJ
\biggl(U, \begin{pmatrix} p  \\ 0 \end{pmatrix}, \begin{pmatrix} q \\ 0 \end{pmatrix} \biggr)
\\ & 
=
\begin{pmatrix} p  \\ 0 \end{pmatrix} \cdot \bfA(U) \begin{pmatrix} p  \\ 0 \end{pmatrix}
+ 
\begin{pmatrix} q  \\ 0 \end{pmatrix} \cdot \bfA_*^{-1}(U) \begin{pmatrix} q  \\ 0 \end{pmatrix}
-2p\cdot q
\,.
\end{align*}
Polarizing the latter identity yields 
\begin{equation*}
J(U,p,q) - J(U,p,0) - J(U,0,q) 
+
\bigl( J^*(U,p,q) - J^*(U,p,0) - J^*(U,0,q) \bigr) 
=
-2 p\cdot q
\,.
\end{equation*}
That is, 
\begin{equation*}
J(U,p,q) - J(U,p,0) - J(U,0,q) 
+ p\cdot q 
=
- \bigl( J^*(U,p,q) - J^*(U,p,0) - J^*(U,0,q)
+ p\cdot q
\bigr)
\,.
\end{equation*}
Given the definitions~\eqref{e.sastU.def},~\eqref{e.mU.def} and~\eqref{e.sU.def}, we obtain that the coarse-grained matrices for the adjoint with coefficients~$\a^t(\cdot)$ are the same as for the operator with coefficients~$\a(\cdot)$, except that~$\k$ is replaced by~$-\k$. The identity~\eqref{e.Jaas.nosymm.star} now follows from~\eqref{e.Jaas.nosymm}.

\smallskip

We next prove the formulas for~$\bfA(U)$ and~$\bfA_*^{-1}(U)$ in~\eqref{e.matrices.formulas}. 
Using~\eqref{e.bfJ.byJJstar},~\eqref{e.Jsplitting.nonsymm} and~\eqref{e.Jaas.nosymm.star}, we obtain 
\begin{align*}
\frac12 
\begin{pmatrix} p   \\ q\end{pmatrix}\cdot \bfA(U) \begin{pmatrix} p   \\ q\end{pmatrix}
= \bfJ\biggl( \! U, \begin{pmatrix} p   \\ q\end{pmatrix}, 0 \!\biggr) \!
& 
=
\frac12 J(U,p,-q) + \frac12 J^*(U,p,q) 
\\ & 
=
\frac12 p\cdot \s(U) p + \frac12 (q-\k(U)p) \cdot \s_*^{-1}(U)  (q-\k(U)p)
\\ & 
=
\frac12 \begin{pmatrix} p   \\ q\end{pmatrix}\cdot
\begin{pmatrix} (\s + \k^t\s_*^{-1}\k)(U) & -(\k^t\s_*^{-1})(U) \\ -(\s_*^{-1}\k)(U) & \s_*^{-1}(U) \end{pmatrix}
\begin{pmatrix} p  \\q\end{pmatrix}\,.
\end{align*}
This gives the formula for~$\bfA(U)$. The formula for~$\bfA_*^{-1}(U)$ is obtained similarly.

\smallskip

We next give the proof of~\eqref{e.JJstar1}. 
Observe that the identity~\eqref{e.Jaas.nosymm} can be written as 
\begin{align}
\label{e.J.magic.nosymm}
J(U,p,q) 
& 
=
\frac 12p \cdot \bigl (\s(U)- \s_* (U) \bigr )p 
+
\frac12 p \cdot \bigl( \k(U) + \k^t(U) \bigr) p
\notag \\ & 
\qquad 
+ \frac 12 \bigl (q -(\s_*-\k)(U) p\bigr ) \cdot \s_*^{-1} (U) \bigl (q-(\s_*-\k)(U) p\bigr ) 
\,.
\end{align}
Similarly, we can rewrite~\eqref{e.Jaas.nosymm.star} as 
\begin{align}
\label{e.J.magic.nosymm.star}
J^*(U,p,q) 
& 
=
\frac 12p \cdot \bigl (\s(U)- \s_* (U) \bigr )p 
-
\frac12 p \cdot \bigl( \k(U) + \k^t(U) \bigr) p
\notag \\ & 
\qquad 
+ \frac 12 \bigl (q -(\s_*+\k)(U) p\bigr ) \cdot \s_*^{-1} (U) \bigl (q-(\s_*+\k)(U) p\bigr ) 
\,.
\end{align}
By summing~\eqref{e.J.magic.nosymm} and~\eqref{e.J.magic.nosymm.star}, we obtain~\eqref{e.JJstar1}. 

\smallskip

Substituting~$(p,q,h)=(e,\s_*(U)e,\k(U)e)$ in~\eqref{e.JJstar1} yields
\begin{equation}
\label{e.sss.magic.iden.nosymm}
J(U,e,(\s_*- \k)(U) e) 
+ J^*(U,e,(\s_*+\k)(U) e)
=
e \cdot \bigl( \s(U) - \s_*(U) \bigr) e
\,.
\end{equation}	
We deduce the ordering~\eqref{e.nonobvious.ordering}  from~\eqref{e.sss.magic.iden.nosymm} and the nonnegativity of~$J$ and~$J^*$.

\smallskip

We next prove~\eqref{e.ksym.by.sss}. 
We observe that, according to~\eqref{e.J.magic.nosymm} and~\eqref{e.sss.magic.iden.nosymm}, for every~$e \in \Rd$, 
\begin{align*}
\frac12 e\cdot (\k+\k^t)(U) e 
&
=
J(U,e,(\s_*- \k)(U) e) 
-
\frac12 e\cdot (\s-\s_*)(U) e
\notag \\ & 
=
\frac12 e\cdot (\s-\s_*)(U) e
-
J^*(U,e,(\s_*+\k)(U) e)
\leq 
\frac12 e\cdot (\s-\s_*)(U) e
\,.
\end{align*}
Thus~$(\k+\k^t)(U) \leq (\s-\s_*)(U)$. By symmetry (or applying the same argument with~$J^*$ in place of~$J$), we obtain also that~$-(\k+\k^t)(U) \leq (\s-\s_*)(U)$.

\smallskip

The combination of~\eqref{e.fluxmaps.nosymm} and~\eqref{e.sss.magic.iden.nosymm} imply~\eqref{e.coarse.graining.nosymm}. 

\smallskip

Finally, we observe that the identities~\eqref{e.Jsplitting.nonsymm} and~\eqref{e.bfJ.byJJstar}
reduce the bound~\eqref{e.maul.Mall} on the Malliavin derivatives of~$\bfA$ and~$\bfA_*$ to bounds on the Malliavin derivatives of~$J$ and~$J^*$. The latter have already been proved in Lemma~\ref{l.malliavin}---indeed, as noted in the comments preceding the statement of that lemma, its proof does not use the symmetry assumption. 
\end{proof}

We should regard the ``double-variable'' quantity~$\bfJ(U,P,Q)$ and its associated coarse-grained matrices~$\bfA(U)$ and~$\bfA_*(U)$ as an alternative way of encoding the information contained in the more familiar coarse-grained matrices~$\s(U)$,~$\s_*(U)$ and~$\k(U)$. Since~\eqref{e.Jsplitting.nonsymm} and~\eqref{e.matrices.formulas} make it clear that these quantities are essentially the same; the reader may wonder why we introduce so many different definitions. The reason is that the coarse-grained matrices possess a lot of structure, and having several different but equivalent points of view can be quite helpful.
For example, it is quite challenging to prove the simple but important fact that~$\s_*(U) \leq \s(U)$, or the formula~\eqref{e.Jaas.nosymm.star} for the adjoint without introducing the double variable quantities.

\smallskip


As we have seen in Lemma~\ref{l.J.basicprops}, in the case that~$\a(\cdot)$ is symmetric, the quantity~$J(U,p,q)$ has the interpretation as the energy of the difference of the Dirichlet and Neumann problems with affine data given by~$p$ and~$q$, respectively. In the general nonsymmetric case, 
there is no longer a simple identity relating~$J(U,p,q)$ to the difference of the solutions to the Dirichlet and Neumann problems. However, we show in the next lemma that it is still equivalent (up to a multiplicative constant) to the energy of this difference for the adjoint equation.  

\begin{lemma}
\label{l.Dirichlet.to.Neumann}
Let~$\a_0 = \s_{\bfzero} +\k_0$ with~$\s_{\bfzero} \in \R^{d\times d}_{\mathrm{sym}}$ and~$\lambda \Id \leq \s_{\bfzero} \leq \Lambda \Id$ and~$\k_0\in\R^{d\times d}_{\mathrm{skew}}$. 
Let~$U$ be a bounded Lipschitz domain and let~$w^D_p, w^N_{q} \in H^1(U)$ respectively denote, for every~$p,q\in\Rd$, the solutions of
\begin{equation}
\label{e.DirNeu.P}
\left\{
\begin{aligned}
& -\nabla\cdot \a \nabla w^D_p = 0 & \mbox{in} & \ U\,, \\ 
& w^D_p = \ell_p & \mbox{on} & \ \partial U\,,
\end{aligned}
\right.
\quad \mbox{and} \quad 
\left\{
\begin{aligned}
& -\nabla\cdot \a \nabla w^N_q = 0 & \mbox{in} & \ U\,, \\ 
& \mathbf{n} \cdot (\a \nabla w^N_{q} - q) = 0 & \mbox{on} & \ \partial U\,.
\end{aligned}
\right.
\end{equation}
Then there exists~$C(d,\lambda,\Lambda)<\infty$ such that we have the following estimates:
\begin{equation}
\label{e.DirNeu.P.est}
\left\{ 
\begin{aligned}
&
\| \nabla w^D_p - \tfrac12\nabla v(\cdot,U,-p, \a_0 p) \|_{\underline{L}^2(U)}^2 
\leq 
C J^*(U,p, \a_0 p )
\\ &
\| \nabla w^N_q - \tfrac12\nabla v(\cdot,U,-\a_0^{-1}q, q) \|_{\underline{L}^2(U)} ^2
\leq 
C J^*(U,\a_0^{-1}q, q)\,.
\\ & 
J^*(U,p, q )
\leq
C \| \nabla w^D_p - \nabla w^N_{q} \|_{\underline{L}^2(U)}^2
\,.
\end{aligned}
\right. 
\end{equation}
\end{lemma}
\begin{proof}
We use~\eqref{e.maximizers.J.to.bfJ} and the characterization~\eqref{e.J.P0.Dir}--\eqref{e.J.P0.Dir.SX}. With~$p\in \Rd$ fixed, we denote 
\begin{equation*}
z_p:= \frac12 ( v (\cdot,U,p,-(\s_{\bfzero} +\k_0)p )
+ 
v^*(\cdot,U,p,(\s_{\bfzero} +\k_0)p) \bigr) \,,
\end{equation*}
and we observe that this function belongs to~$\ell_p + H^1_0 (U)$. 
We therefore have, using Cauchy-Schwarz and~\eqref{e.Jenergyv.nosymm},  
\begin{align*}
\fint_U \nabla ( w_p^D - z_p) \cdot \a \nabla ( w_p^D - z_p)
&
=
\frac12 \fint_U \nabla ( w_p^D - z_p) \cdot \a \nabla v^*(\cdot,U,p,(\s_{\bfzero} +\k_0)p) 
\notag \\ &
\leq
C \| \nabla ( w_p^D - z_p) \|_{\underline{L}^2(U)} 
\| \nabla v^*(\cdot,U,p,(\s_{\bfzero}+\k_0)p) \|_{\underline{L}^2(U)}
\notag \\ &
\leq
C \| \nabla ( w_p^D - z_p) \|_{\underline{L}^2(U)} 
( J^*(U,p,(\s_{\bfzero}+\k_0)p))^{\nicefrac12} \,.
\end{align*}
Therefore, 
\begin{align*}
\lefteqn{
\| \nabla w^D_p - \tfrac12\nabla v(\cdot,U,-p, (\s_{\bfzero}+\k_0) p) \|_{\underline{L}^2(U)} 
} \qquad & 
\notag \\ & 
\leq
\| \nabla (w^D_p - z_p) \|_{\underline{L}^2(U)} 
+
\|\tfrac12 \nabla v^*(\cdot,U,p,(\s_{\bfzero}+\k_0)p) \|_{\underline{L}^2(U)}
\leq 
C( J^*(U,p,(\s_{\bfzero}+\k_0)p))^{\nicefrac12} \,,
\end{align*}
which implies the result. 
The estimate for the Neumann solution is obtained very similarly by comparing~$w_q^N$ to~$z_{\a_0^{-1}q}$. We omit the details. 

\smallskip

In order to obtain the last line of~\eqref{e.DirNeu.P.est}, we use the fact that~$\nabla w_p^D - p \in L^2_{\pot,0}(U)$, and~$\a^t\nabla w_q^N-q \in L^2_{\sol,0}(U)$ to obtain that 
\begin{align*}
J^*(U,p,q) 
& 
=
\sup_{u \in \A^*(U)} 
\fint_U 
\biggl( 
-\frac12 \nabla u\cdot \s \nabla u 
- p\cdot \a^t \nabla u
+ q \cdot \nabla u 
\biggr) 
\notag \\ & 
=
\sup_{u \in \A^*(U)} 
\fint_U 
\biggl( 
-\frac12 \nabla u\cdot \s \nabla u 
- \nabla w_p^D \cdot \a^t \nabla u
+ \a \nabla w_{q}^N \cdot \nabla u 
\biggr) 
\,.
\end{align*}
We next use the pointwise bound 
\begin{align*}
-\frac12 \nabla u\cdot \s \nabla u 
- \nabla w_p^D \cdot \a^t \nabla u
+ \a \nabla w_{q}^N \cdot \nabla u 
&
\leq 
\max_{P \in\Rd} 
\Bigl( 
-\frac12 P\cdot \s P 
- \nabla w_p^D \cdot \a^t P
+ \a \nabla w_{q}^N \cdot P 
\Bigr) 
\notag \\ & 
=
\max_{P \in\Rd} 
\Bigl( 
-\frac12 P\cdot \s P 
+ P \cdot \bigl( \a \nabla w_{q}^N  - \a \nabla w_p^D  \bigr)
\Bigr) 
\notag \\ & 
=
\bigl( \a \nabla w_{q}^N  - \a \nabla w_p^D  \bigr) \cdot \s^{-1}  \bigl( \a \nabla w_{q}^N  - \a \nabla w_p^D  \bigr)
 \,.
\end{align*}
We deduce, therefore, that 
\begin{equation*}
J^*(U,p,q) 
\leq 
\fint_U 
\bigl( \nabla w_{q}^N  - \nabla w_p^D  \bigr) \cdot \a^t\s^{-1}\a  \bigl( \nabla w_{q}^N  - \nabla w_p^D  \bigr)
\leq 
\Lambda
\bigl\| \nabla w_{q}^N  - \nabla w_p^D \bigr\|_{\underline{L}^2(U)}^2\,.
\end{equation*}
This completes the proof. 
\end{proof}

We conclude this section by exploring more properties of the coarse-grained matrices in a sequence of remarks and exercises. 

\begin{exercise}
In analogy to~\eqref{e.energymaps} \&~\eqref{e.energymaps.dual}, prove the following coarse-graining inequalities: 
\begin{equation}
\label{e.energymaps.nonsymm.bfA}
(X)_U \cdot \bfA_*(U) (X)_U \leq 
\fint_{U} X\cdot \bfA X \,, \quad \forall X \in L^2_{\pot}(U) \times L^2_{\sol}(U)
\end{equation}
and
\begin{equation}
\label{e.energymaps.dual.nonsymm.bfA}
(Y)_U \cdot \bfA^{-1} (U) (Y)_U \leq 
\fint_{U} Y \cdot \bfA^{-1} Y \,, \quad 
\forall Y \in L^2_{\sol}(U) \times L^2_{\pot}(U)\,.
\end{equation}
\end{exercise}

\begin{exercise}
Show that~$\bfA^{-1}(U)$ and~$\bfA_*(U)$ are given by the formulas
\begin{equation}
\label{e.bigA.formulas.inv}
\left\{
\begin{aligned}
& 
\bfA^{-1}(U)
= 
\begin{pmatrix} 
\s^{-1}(U) 
& (\s^{-1}\k^t )(U) 
\\ (\k \s^{-1})(U) 
& (\s_* + \k \s^{-1}\k^t )(U) 
\end{pmatrix}\,,
\\ & 
\bfA_*(U)
= 
\begin{pmatrix} 
 (\s_* + \k \s^{-1}\k^t )(U) 
& (\k \s^{-1})(U) 
\\ (\s^{-1}\k^t )(U) 
& \s^{-1}(U)
\end{pmatrix}
\,.
\end{aligned}
\right.
\end{equation}
\end{exercise}

\begin{remark}
\label{r.commutes.with.skew}
Coarse-graining commutes with the addition or subtraction of a constant anti-symmetric matrix. To see this, let~$\k_0\in \R^{d\times d}_{\skew}$ be a fixed anti-symmetric matrix and let~$J_{\mathrm{\k_0}}$ denote the quantity~$J$ for the coefficient field~$\a(x)+\k_0$. Since~$\nabla \cdot \k_0 \nabla \phi = 0$ for any~$\phi \in H^1(U)$, adding a constant anti-symmetric matrix to~$\a$ does not alter the set of solutions of the equation. We therefore have that
\begin{equation}
\label{e.J.k.naught}
J_{\k_0}(U,p,q) 
=
\sup_{u \in  \A(U) }
\fint_{U} \biggl( - \frac12 \nabla u \cdot \s \nabla u - p \cdot ( \a + \k_0 ) \nabla u + q \cdot \nabla u \biggr) 
=
J(U,p,q + \k_0 p) 
\,.
\end{equation} 
It follows that 
\begin{equation}
\label{e.k0.commutes.with.J}
\left\{
\begin{aligned}
& \s(U;\a+\k_0) = \s(U;\a) \,, \\ 
& \s_*(U;\a+\k_0) = \s_*(U;\a)  \,, \\
& \k(U;\a+\k_0) = \k(U;\a) + \k_0\,.
\end{aligned}
\right.
\end{equation}
\end{remark}

\begin{exercise}
Check the details of Remark~\ref{r.commutes.with.skew}. 
\end{exercise}

\begin{exercise}
\label{exer.ordering.homs}
By repeating the proof of~\eqref{e.sss.magic.iden.nosymm} for the annealed quantities, show that 
\begin{equation}
\label{e.sss.magic.iden.nosymm.homs}
\E \bigl[ J(U,e,(\shom_*- \khom)(U) e) 
+ J^*(U,e,(\shom_*+\khom)(U) e) \bigr] 
=
e \cdot \bigl( \shom(U) - \shom_*(U) \bigr) e
\,.
\end{equation}
Deduce in particular that~$\shom(U) \geq \shom_*(U)$ and 
\begin{equation*}
|\ahom(U)-\ahom_*(U)| \leq C|\shom(U) - \shom_*(U)|
\,.
\end{equation*}
\end{exercise}

\begin{exercise}[{Bounds on~$\bfA(U)$ and~$\bfA_*(U)$}]
Show that~$\bfA(U)$ and~$\bfA_*^{-1}(U)$ satisfy the
\begin{equation}
\label{e.bfA.bfA.star.inv.bounds}
\biggl( \fint_{U} \bfA^{-1} (x)\,dx \biggr)^{\!-1} \leq \bfA_{*}(U) \leq \bfA(U) \leq
\fint_{U} \bfA  (x)\,dx
\,.
\end{equation}
Deduce from~\eqref{e.bfA.bfA.star.inv.bounds} and the uniform ellipticity bound~\eqref{e.ellipticity.nonsymm} that, for every~$\eta \in (0,\infty)$, 
\begin{equation}
\label{e.bfA.bounds}
\begin{pmatrix} (1{+}\eta)^{-1}  \lambda \Id  & 0 \\ 0 & (1{+}\eta^{-1})^{-1}  \Lambda^{-1}  \Id \end{pmatrix}
\leq \bfA_*(U) 
\leq \bfA(U) 
\leq 
\begin{pmatrix} (1{+}\eta^{-1}) \Lambda  \Id & 0 \\ 0 & (1{+}\eta) \lambda^{-1}  \Id \end{pmatrix}
\,.
\end{equation}
\end{exercise}

\begin{exercise}
By a variation of the proof of~\eqref{e.fluxmaps.nosymm}, show that, for any symmetric matrix~$\s_0$, 
\begin{equation}
\label{e.fluxmaps.nosymm.nostar}
\biggl| e \cdot
\fint_U \bigl( \s_0 - \k^t(U) - \a \bigr) \nabla w
\biggr|
\leq 
\left( \fint_U \nabla w \cdot \s \nabla w \right)^{\nicefrac12} 
\bigl( 2J(U, e , (\s_0 - \k(U) ) e )\bigr )^{\nicefrac12} 
\,.
\end{equation}
Deduce the following coarse-graining inequality for~$\a(U)$ in place of~$\a_*(U)$:
\begin{equation}
\label{e.coarse.graining.nosymm.nostar}
\biggl| \s^{-1} (U) e\cdot \fint_U ( \a(U) - \a ) \nabla w \biggr| 
\leq 
2^{\nicefrac12} 
\bigl( e\cdot( \s_*^{-1} (U) - \s^{-1} (U))e \bigr)^{\nicefrac12} \biggl( \fint_U \nabla w \cdot \s \nabla w \biggr)^{\!\nicefrac12} 
\,.
\end{equation}
\end{exercise}

%
%
%
%
%

\begin{remark}
Given an invertible, symmetric matrix~$\tilde{\mathbf{s}} \in\R^{d\times d}_{\mathrm{sym}}$ and another matrix~$\tilde{\mathbf{k}}\in \R^{d\times d}$ (not necessarily anti-symmetric), consider the block matrix
\begin{equation}
\label{e.tildes.relate}
\tilde{\mathbf{A}}
:=
\begin{pmatrix} \tilde{\mathbf{s}} - \tilde{\mathbf{k}}\tilde{\mathbf{s}}^{-1}\tilde{\mathbf{k}} & \tilde{\mathbf{k}}\tilde{\mathbf{s}}^{-1} \\ - \tilde{\mathbf{s}}^{-1}\tilde{\mathbf{k}} & \tilde{\mathbf{s}}^{-1} \end{pmatrix}
\,.
\end{equation}
Then by a direct computation, we may check that, for every~$p,p^*\in\Rd$, 
\begin{equation*}
\tilde{\mathbf{A}}  
\begin{pmatrix} p \\ \tilde{\mathbf{s}} p^* + \tilde{\mathbf{k}} p  \end{pmatrix}
=
\begin{pmatrix} \tilde{\mathbf{s}} p + \tilde{\mathbf{k}} p^* \\ p^*  \end{pmatrix} 
\,.
\end{equation*}
We therefore deduce from~\eqref{e.bfJ.byJJstar} that 
\begin{align}
\label{e.bfJ.to.JJstar2}
\lefteqn{
\bfJ
\biggl(U, \begin{pmatrix} p \\ \tilde{\mathbf{s}} p^* + \tilde{\mathbf{k}} p  \end{pmatrix} \,, 
\tilde{\mathbf{A}} \begin{pmatrix} p \\ \tilde{\mathbf{s}} p^* + \tilde{\mathbf{k}} p  \end{pmatrix}
\biggr)
} \qquad & 
\notag \\ & 
=
\bfJ
\biggl(U, \begin{pmatrix} p \\ \tilde{\mathbf{s}} p^* + \tilde{\mathbf{k}} p  \end{pmatrix} \,, 
\begin{pmatrix} \tilde{\mathbf{s}} p + \tilde{\mathbf{k}} p^* \\ p^*  \end{pmatrix}
\biggr)
\notag \\ &
=
\frac12 \bigl ( J\bigl(U,p-p^*,
(\tilde{\mathbf{s}}-\tilde{\mathbf{k}}) (p-p^*)  \bigr) 
+ 
J^*\bigl(U,p+p^*,
(\tilde{\mathbf{s}}+\tilde{\mathbf{k}}) (p+p^*)  \bigr) 
\bigr )
\,.
\end{align}
We deduce that, for a constant~$C(d)<\infty$, 
\begin{equation}
\label{e.bfJ.to.JJstar.minset}
\sup_{P\in \R^{2d}, \, |P|=1}
\bfJ(\cu_n,P, \tilde{\mathbf{A}} P) 
\leq 
C
\sum_{i=1}^d
\bigl ( 
 J\bigl (\cu_n,e_i,(\tilde{\mathbf{s}}-\tilde{\mathbf{k}}) e_i\bigr ) +  J^*\bigl (\cu_n,e_i,(\tilde{\mathbf{s}}+\tilde{\mathbf{k}}) e_i\bigr )
\bigr )
\end{equation}
and, conversely,
\begin{equation}
\label{e.bfJ.to.JJstar.minset.2}
\sup_{e\in\Rd\,, |e|=1} 
\bigl ( 
 J\bigl (\cu_n,e,(\tilde{\mathbf{s}}-\tilde{\mathbf{k}}) e \bigr ) +  J^*\bigl (\cu_n,e,(\tilde{\mathbf{s}}+\tilde{\mathbf{k}}) e\bigr )
\bigr )
\leq 
C
\sum_{i=1}^{2d} 
\bfJ(\cu_n,e_i, \tilde{\mathbf{A}} e_i) 
\,.
\end{equation}
\end{remark}

\begin{remark}
Using the above identities and estimates, we obtain the following exact analogy to~\eqref{e.diagonalset} by repeating the same argument nearly verbatim: there exists a constant~$C(d,\lambda,\Lambda)<\infty$ such that, for every matrix~$\mathbf{B} \in \R^{2d\times2d}$, 
\begin{equation} 
\label{e.diagonalset.bigA}
\left| \bfA(U) - \bfA_*(U) \right| 
+
\left| \bfA(U) - \mathbf{B} \right|^2 
+
\left| \bfA_*(U) - \mathbf{B} \right|^2 
\leq C \sum_{i=1}^{2d} \bfJ(U,e_i,\mathbf{B} e_i)
\,. 
\end{equation}
Using this inequality, we obtain analogs of~\eqref{e.additivity.by.J},~\eqref{e.add.defect.a} and~\eqref{e.add.defect.astar}, arguing exactly as in the proof of those estimates in the symmetric case. In particular, we obtain a constant~$C(d,\lambda,\Lambda)$ such that, for any matrix~$\mathbf{B}\in\R^{2d\times2d}$, 
\begin{equation}
\label{e.add.defect.a.nosymm}
\avsum_{z\in 3^n\Zd \cap \cu_m} 
\bfA(z+\cu_n) 
\leq 
\bfA(\cu_m) + C \avsum_{z\in 3^n\Zd\cap \cu_m}
\sum_{i=1}^{2d} 
\bfJ(z+\cu_n,e_i,\mathbf{B} e_i)
\Itwod
\end{equation}
and
\begin{equation}
\label{e.add.defect.astar.nosymm}
\avsum_{z\in 3^n\Zd \cap \cu_m} 
\bfA_*^{-1}(z+\cu_n) 
\leq 
\bfA_*^{-1} (\cu_m) + C \avsum_{z\in 3^n\Zd\cap \cu_m}
\sum_{i=1}^{2d} 
\bfJ(z+\cu_n,e_i,\mathbf{B} e_i) \Itwod \,.
\end{equation}
These bounds will be needed in the next section to obtain variance bounds that generalize those of Lemma~\ref{l.flatness}. 
\end{remark}

In the next exercise, the reader is asked to generalize~\eqref{e.a.formulas.nosymm} using a similar computation as in~\eqref{e.fluxmaps.computation}. 
\begin{exercise}
Generalize the identities~\eqref{e.a.formulas.nosymm} in by showing that 
\begin{equation}
\label{e.formulas.nosymm.byS}
\fint_U \bfA S(\cdot,U,P,Q) = Q-\bfA(U) P 
\qand
\fint_U S(\cdot,U,P,Q) = \bfA_*^{-1}(U) Q - P 
\,.
\end{equation}
Explain why~\eqref{e.formulas.nosymm.byS} is a generalization of~\eqref{e.a.formulas.nosymm}. 
Explain why the formula for~$\fint_U  S(\cdot,U,P,Q)$ contains no new information compared to the first identity in~\eqref{e.formulas.nosymm.byS}. 
\end{exercise}

\begin{exercise}
Show the formulas 
\begin{equation} 
\label{e.J.by.means.of.bfA}
J(U,p,q) = \frac12\begin{pmatrix} -p \\ q \end{pmatrix} \cdot \bfA(U) \begin{pmatrix} -p \\ q \end{pmatrix} - p\cdot q
\,, \quad
J^*(U,p,q) = \frac12\begin{pmatrix} p \\ q \end{pmatrix} \cdot \bfA(U) \begin{pmatrix} p \\ q \end{pmatrix} - p\cdot q\,.
\end{equation}
\end{exercise}

\begin{exercise}
Show that the limit
\begin{equation}
\label{e.ahom.def.nonsymm}
\ahom:= \lim_{n\to \infty} \ahom_*(\cu_n)
\end{equation}
exists and that~$\ahom$ satisfies the same uniform ellipticity bounds as in~\eqref{e.ellipticity.nonsymm}. Hint: use the monotonicity of~$\bfAhom(\cu_n)$. 
\end{exercise}

\begin{remark}[Enhancement of diffusion by advection]
The definitions of~$\a_*(U)$ in~\eqref{e.mustar}-\eqref{e.quad.mu.star} and~$\s_*(U)$ in~\eqref{e.variational.J.nonsymm}-\eqref{e.sastU.def} are very similar, and at first glance may appear to be equal in the sense that~$\s_*(U)$ has no dependence on the anti-symmetric part of the field~$\a(\cdot)$ at all. However, this is not the case---the restriction in the variational formula~\eqref{e.variational.J.nonsymm} that~$u$ belong to~$\A(U)$, a strict subset of~$H^1(U)$, means that in general we will only have the ordering~$\s_*^{-1}(U;\a) \leq \a_*^{-1}(U;\s)$. In other words, the matrix~$\s_*(U;\a)$ defined in~\eqref{e.variational.J.nonsymm}-\eqref{e.sastU.def} with respect to~$\a$ is larger than the matrix~$\a_*(U;\s)$, defined in~\eqref{e.mustar}-\eqref{e.quad.mu.star} with respect to the symmetric field~$\s(\cdot)$. This inequality is typically strict. Indeed, it is a coarse-grained version of the \emph{enhancement of diffusion due to advection}, which roughly states that the presence of a divergence-form vector field (the divergence of the matrix~$\k(\cdot)$) will make the homogenized diffusion matrix larger. 
\end{remark} 

\subsection{Quantitative homogenization estimates in the nonsymmetric case}
\label{ss.subadd.conv.nosymm}

In this section, we will explain how to generalize the main results of Chapter~\ref{s.subadd} to the case of nonsymmetric coefficients. 
We are interested in formulating an extension of Corollary~\ref{c.subadd.converge}, the stronger version of Theorem~\ref{t.subadd.converge}, which is what is mainly used in practice. 
To that end, we generalize the random variable defined in~\eqref{e.mcE.0} to the general nonsymmetric setting by defining, for each~$m\in\N$,
\begin{align} 
\label{e.mcE.0.nonsymm}
\mathcal{E}(m) 
:= 
\sum_{n=0}^{m}
3^{n-m} 
\biggl( 
\avsum_{z\in 3^n\Zd\cap \cu_m}
\!\!
\Bigl( \bigl| (\s- \s_*)(z+\cu_n)  \bigr|
+
\bigl| \a_*(z+\cu_n) - \ahom \bigr|^2
\Bigr) 
\biggr)^{\!\nicefrac12} 
\,,
\end{align}
where~$\ahom$ is defined in~\eqref{e.ahom.def} above. 
The random variable~$\mathcal{E}(m)$ is the same as the one defined in~\eqref{e.mcE.0} in the symmetric case. 
We remark that in view of~\eqref{e.JJstar1}, the quantity 
\begin{equation}
\label{e.oneway}
\bigl| (\s- \s_*)(U)  \bigr|
+
\bigl| \a_*(U) - \ahom \bigr|^2
\end{equation}
is equivalent to (bounded on both sides by positive, deterministic constants multiplied by) 
\begin{equation}
\label{e.anotherway}
\sum_{i=1}^d 
\bigl( 
J(U,e_i,\ahom^t e_i ) + J^*(U,e_i,\ahom e_i )
\bigr) \,.
\end{equation}
We could, therefore, have written~$\mathcal{E}(m)$ in terms of the quantity on the right side of~\eqref{e.anotherway}. 

\begin{theorem}
\label{t.subadd.converge.nosymm}
Assume~$\P$ is a~$\Zd$--stationary measure on~$(\Omega,\F)$ and satisfies~$\CFS(\beta,\Psi)$. Then, for every~$\delta>0$, there exist constants~$\alpha(\beta,d,\lambda,\Lambda) \in (0,\tfrac12]$ and~$C(\delta,\dataref)<\infty$ and a random variable~$\X$ satisfying 
\begin{equation}
\label{e.mmmbound.nosymm}
\X^{\frac d2 (1-\beta)}
= \O_\Psi(C)
\end{equation}
such that, for every~$m\in\N$ with~$3^m\geq \X$, 
\begin{equation} 
\label{e.aastar.concentrate.DD.nosymm}
\mathcal{E}(m)^2
\leq \delta^2 ( \X 3^{-m})^{\alpha} 
\,.
\end{equation}
\end{theorem}

The proof of Lemma~\ref{l.mathcalE.minscale} applies nearly verbatim with~$J$ replaced by~$\bfJ$. We therefore obtain Theorem~\ref{t.subadd.converge.nosymm} as a consequence of the following statement, which is a generalization of Proposition~\ref{p.algebraicrate.E} and the main focus of this section.

\begin{proposition}[Algebraic rate of decay for~$J$ \&~$J^*$]
\label{p.algebraic.nosymm}
\hspace{-1pt} 
Assume that~$\P$ is a~$\Zd$--stationary measure on~$(\Omega,\F)$ and satisfies~$\CFS(\beta,\Psi)$. 
There exist~$\alpha(\beta,d,\lambda,\Lambda) >0$ and~$C(\dataref)<\infty$
such that, for every~$m\in\N$,  
\begin{equation}
\label{e.EJtozero.rate.nosymm}
\bigl| \ahom(\cu_m) - \ahom \bigr|
+
\bigl| \ahom_*(\cu_m) - \ahom \bigr|
\leq 
C3^{-m\alpha}. 
\end{equation}
\end{proposition}

The proof of Proposition~\ref{p.algebraic.nosymm} follows the one of Proposition~\ref{p.algebraicrate.E} very closely. 
We define the expected subadditivity defect, generalizing~\eqref{e.taun}, by
\begin{equation}
\label{e.taun.nosymm}
\tau_m := \left|\shom (\cu_{m}) - \shom(\cu_{m-1} )\right| + \left|  \shom_*(\cu_{m}) - \shom_*(\cu_{m-1} )\right|.
\end{equation}
We need to generalize~\eqref{e.subaddcontrol}. We first use subadditivity to show that the coarse-grained matrices~$\shom_*(\cu_m)$ and~$\shom(\cu_m)$ are, respectively, nondecreasing and nonincreasing in~$m$. For~$\shom_*$, we use~\eqref{e.subaddJ.nosymm} with~$(p,q)=(0,e)$ to find that 
\begin{equation*}
\s_*^{-1}(\cu_m) \leq 
\avsum_{z\in 3^n\Zd \cap \cu_m} 
\s_*^{-1}(z+\cu_n)\,,
\end{equation*}
and it follows after taking expectations and using stationarity that 
\begin{equation}
\label{e.shom.star.monotone}
\shom_*^{-1} (\cu_m) \leq \shom_*^{-1}(\cu_n) \,, \quad \forall n,m\in\N\,, \ n \leq m. 
\end{equation}
For the quantity~$\s$, we will prove that
\begin{equation}
\label{e.shom.monotone}
\shom (\cu_m) \leq \shom (\cu_n) \,, \quad \forall n,m\in\N\,, \ n \leq m. 
\end{equation}
The monotonicity of~$\shom$ is more subtle than that of~$\shom_*$, since~$\s(U)$ is not a subadditive quantity, in general, and~$\shom(U)$ is not the expectation of~$\s(U)$. To prove~\eqref{e.shom.monotone}, we apply~\eqref{e.JJstar1} with the choice~$(p,q,h) = (e,0,\khom(\cu_{m-1})e)$ to get
\begin{align}
\label{e.JJstar.noq}
\lefteqn{ 
\bigl ( J(U,e,\khom(\cu_{m-1})e) + J^*(U,e,\khom(\cu_{m-1})e) \bigr )
} \qquad & 
\notag \\ & 
=
e \cdot \s(U) e
+
\bigl ( \khom(\cu_{m-1}) -\k(U) \bigr )e  \cdot \s_*^{-1}(U)\bigl ( \khom(\cu_{m-1}) -\k(U) \bigr ) e
\,.
\end{align}
Taking expectations, we get 
\begin{align}
\label{e.JJstar.noq.annealed}
\lefteqn{ 
\E \bigl[ J(U,e,\khom(\cu_m)e) \bigr] + \E \bigl[ J^*(U,e,\khom(\cu_{m-1})e) \bigr ]
} \qquad & 
\notag \\ & 
=
e \cdot \shom(U) e
+
\bigl ( \khom(\cu_{m-1}) -\khom(U) \bigr )e  \cdot \shom_*^{-1}(U)\bigl ( \khom(\cu_{m-1}) -\khom(U) \bigr ) e
\,.
\end{align}
By using subadditivity~\eqref{e.subaddJ.nosymm} for both~$J$ and~$J^*$, we have that 
\begin{equation*}
\left\{
\begin{aligned}
& \E \bigl[ J(\cu_m,e,\khom(\cu_{m-1})e) \bigr] \leq 
\E \bigl[ J(\cu_{m-1},e,\khom(\cu_{m-1})e) \bigr]
\,, \quad \mbox{and}  \\ & 
\E \bigl[ J^*(\cu_m,e,\khom(\cu_{m-1})e) \bigr ]
\leq 
\E \bigl[ J^*(\cu_{m-1},e,\khom(\cu_{m-1})e) \bigr ]
\,.
\end{aligned}
\right.
\end{equation*}
Combining the previous two displays, we obtain
\begin{equation*}
\shom(\cu_m) 
\leq 
\shom(\cu_{m-1}) 
-
\bigl ( \khom(\cu_{m-1}) -\khom(\cu_m) \bigr )e  \cdot \shom_*^{-1}(\cu_m)\bigl ( \khom(\cu_{m-1}) -\khom(\cu_m) \bigr )
\leq \shom(\cu_{m-1})\,.
\end{equation*}
This completes the proof of~\eqref{e.shom.monotone}.

\smallskip

In view of~Exercise~\ref{exer.ordering.homs},~\eqref{e.shom.star.monotone} and~\eqref{e.shom.monotone}, we have that 
\begin{align}
\label{e.subaddcontrol.nosymm}
\tau_m 
&
\leq 
2
\bigl|
(\shom (\cu_{m-1}) - \shom_*(\cu_{m-1}))
-
(\shom (\cu_{m}) - \shom_*(\cu_m))
\bigr|
\notag \\& 
=
2\sup_{|e|=1} 
\bigl\{ e \cdot 
(\shom - \shom_*)(\cu_{m-1})
e 
-
e\cdot 
(\shom - \shom_*)(\cu_m)
e 
\bigr\}
\notag \\& 
=
2\sup_{|e|=1} 
\Bigl\{
\E \bigl[ J(\cu_m,e,(\shom_*- \khom )(\cu_m) e) \bigr]  
-
\E \bigl[ J(\cu_{m-1},e,(\shom_*- \khom )(\cu_{m-1}) e) \bigr] 
\notag \\ & \qquad \qquad 
+ \E\bigl[ J^*(\cu_m,e,(\shom_*+\khom)(\cu_m) e) \bigr]
-
\E\bigl[ J^*(\cu_{m-1},e,(\shom_*+\khom)(\cu_{m-1}) e) \bigr]
\Bigr\}
\notag \\ & 
\leq C \tau_m
\,.
\end{align}
This is the needed generalization of~\eqref{e.subaddcontrol}. 

\smallskip

We continue with the generalization of Lemma~\ref{l.flatness.rules} to the general nonsymmetric setting. It is nearly a copy-paste of Lemma~\ref{l.flatness.rules}, but we include the details for the convenience of the reader. 

\begin{lemma}
[Control of the convex duality defect]
\label{l.flatness.rules.nosymm}
There exists~$C(d,\lambda,\Lambda)<\infty$ such that, for every~$m\in\N$ and~$p,q\in B_1$, 
\begin{align}
\label{e.flatness.quenched.nosymm}
J(\cu_{m-1},p,q) 
&
\leq 
C\sum_{n=0}^{{m}} 3^{n-m} 
\!\! \avsum_{y\in 3^n\Zd\cap \cu_{m}} \!\!
\left| p -\s_*^{-1}(y+\cu_n) (q +\k(y+\cu_n)p )  \right|^2
\notag \\ & \qquad 
+ C
\sum_{n=0}^{{m}} 3^{n-m} 
\!\! \avsum_{y\in 3^{n}\Zd\cap \cu_m}  \!\!
( J(y+\cu_{n},p,q) - J (\cu_m,p,q) )
+C3^{-2m}
\,.
\end{align}
In particular, for every~$m\in\N$ and~$e\in  B_1$, 
\begin{align}
\label{e.flatness.rules.nosymm}
\E \left[ J(\cu_m,e,(\shom_*-\khom)(\cu_m)e) \right]
\leq 
C 
\sum_{n=0}^{m} 3^{n-m} 
\left( \tau_n + \var\left[ \a_*(\cu_n) \right] \right)
+C 3^{-2m}\,. 
\end{align}
\end{lemma}

\begin{proof}
\emph{Step 1.} We claim that, for every~$m\in\N$ and~$p,q\in B_1$, 
\begin{align}
\label{e.makeroom.nosymm}
J(\cu_{m-1},p,q) 
&
\leq 
C
3^{-2m}
\inf_{k\in\R}
\left\| v(\cdot,\cu_m,p,q) - k \right\|_{\underline{L}^2(\cu_m)}^2
\notag \\ & \qquad 
+
\sum_{z\in 3^{m-1}\Zd\cap \cu_m} 
2( J(z+\cu_{m-1},p,q) - J (\cu_m,p,q) )
\,.
\end{align}
For simplicity, denote~$v_n \hspace{-1pt}  :=  \hspace{-1pt}  v(\cdot,\cu_n,p,q)$ for each~$n\in\N$ and compute, using~\eqref{e.quadresp.nosymm} and~\eqref{e.Jenergyv.nosymm}, 
\begin{align*}
J(\cu_{m-1},p,q) 
& 
=
\fint_{\cu_{m-1}} 
\frac12
\nabla v_{m-1} 
\cdot \a\nabla v_{m-1}
\\ & 
\leq 
\fint_{\cu_{m-1}} 
\nabla v_{m} 
\cdot \a\nabla v_{m}
+
\fint_{\cu_{m-1}} 
\left( \nabla v_{m}
-
\nabla v_{m-1} \right) 
\cdot 
\a\left( \nabla v_{m}
-
\nabla v_{m-1}
\right) 
\\ & 
\leq
\fint_{\cu_{m-1}} 
\nabla v_{m} 
\cdot \a\nabla v_{m}
+
3^d 
\avsum_{z\in 3^{m-1}\Zd\cap \cu_m} 
2( J(z+\cu_{m-1},p,q) - J (\cu_m,p,q) )
\,.
\end{align*}
By the Caccioppoli inequality,
\begin{equation*}
\fint_{\cu_{m-1}}
\nabla v_m \cdot\a\nabla v_m
\leq
C
3^{-2m}
\left\| v_m -k \right\|_{\underline{L}^2(\cu_m)}^2
\,.
\end{equation*}
Combining this with the display above yields~\eqref{e.makeroom.nosymm}. 

\smallskip

\emph{Step 2.} We show that,
for every~$m\in\N$ and~$p,q\in B_1$, 
\begin{align}
\label{e.multiscope.nosymm}
\lefteqn{
3^{-m} 
\inf_{k\in\R} 
\left\| v(\cdot,\cu_m,p,q) - k \right\|_{\underline{L}^2(\cu_m)} 
} \qquad & 
\notag \\ &
\leq
C3^{-m} 
+
C
\sum_{n=0}^{{m}} 3^{n-m} 
\biggl(
\avsum_{y\in 3^{n}\Zd\cap \cu_m} 
( J(y+\cu_{n},p,q) - J (\cu_m,p,q) )
\biggr)^{\!\!\nicefrac12}
\notag \\ & \qquad 
+
C\sum_{n=0}^{{m}} 3^{n-m} 
\biggl(\avsum_{y\in 3^n\Zd\cap \cu_{m}} 
\left| p -\s_*^{-1}(y+\cu_n) (q +\k(y+\cu_n)p ) \right|^2
\biggr)^{\!\!\nicefrac12}
\,.
\end{align}
By the multiscale Poincar\'e inequality (Proposition~\ref{p.MSP}), we have
\begin{align*}
3^{-m} 
\inf_{k\in\R} 
\left\| v_m {-} k \right\|_{\underline{L}^2(\cu_m)} 
&
\leq 
C3^{-m} \left\| \nabla v_m \right\|_{\underline{L}^2(\cu_m)} 
+
C \! 
\sum_{n=0}^{m} 3^{n-m} 
\biggl( \avsum_{y\in 3^n\Zd\cap \cu_{m}} \!\!
\big| \left( \nabla v_m \right)_{y+\cu_n} \big|^2 \biggr)^{\!\nicefrac12}
.
\end{align*}
Using the triangle inequality and the identities~\eqref{e.a.formulas.nosymm}, we observe that 
\begin{align*}
\biggl(
\avsum_{y\in 3^n\Zd\cap \cu_{m}} 
\big| \left( \nabla v_m \right)_{y+\cu_n} \big|^2\biggr)^{\!\!\nicefrac12}
&
\leq 
\biggl(\avsum_{y\in 3^n\Zd\cap \cu_{m}} \!\!
\big|  \left( \nabla v_{m} \right)_{y+\cu_n} -  \left( \nabla v_{n} \right)_{y+\cu_n} \big|^2\biggr)^{\!\!\nicefrac12}
\notag \\ & \qquad 
+
\biggl(\avsum_{y\in 3^n\Zd\cap \cu_{m}} \!\!
\left| p -\s_*^{-1}(y+\cu_n) (q +\k(y+\cu_n)p ) \right|^2
\biggr)^{\!\!\nicefrac12}
\,.
\end{align*}
By~\eqref{e.Jenergyv.nosymm} and~\eqref{e.quadresp.nosymm}, we estimate
\begin{align*}
\avsum_{y\in 3^n\Zd\cap \cu_{m}} \!\!\!
\big|  \left( \nabla v_{m} \right)_{y+\cu_n}  - \left( \nabla v_{n} \right)_{y+\cu_n} \big|^2
\leq 
2 \!\!
\avsum_{y\in 3^{n}\Zd\cap \cu_m} \!\!
( J(y+\cu_{n},p,q) - J (\cu_m,p,q) )
.
\end{align*}
Combining the previous displays, using also that~$\left\| \nabla v_m \right\|_{\underline{L}^2(\cu_m)}^2  \hspace{-1pt}  \leq \hspace{-1pt} CJ(\cu_m,p,q) \hspace{-1pt} \leq \hspace{-1pt} C$ by~\eqref{e.Jenergyv.nosymm}, we obtain~\eqref{e.multiscope.nosymm}.

\smallskip

\emph{Step 3.} The conclusion. The estimate~\eqref{e.flatness.quenched.nosymm}
is an immediate consequence of~\eqref{e.makeroom.nosymm} and~\eqref{e.multiscope.nosymm}.
The inequality~\eqref{e.flatness.rules.nosymm} is obtained by taking the expectation of~\eqref{e.flatness.quenched.nosymm}, using 
$\E \left[ J(\cu_m,e,(\shom_*-\khom)(\cu_m)e) \right] \leq \E \left[ J(\cu_{m-1},e,(\shom_*-\khom)(\cu_m)e) \right]$,~\eqref{e.subaddcontrol.nosymm} and the following:
\begin{align*}
\E\Biggl [
\sum_{n=0}^{{m}} 3^{n-m} \!\!\!
\avsum_{y\in 3^{n}\Zd\cap \cu_m} \!\!\!
( J(y+\cu_{n},p,q) - J (\cu_m,p,q) ) \Biggr ]
&
\leq
C\!\sum_{n=0}^{{m}} \! 3^{n-m} \!\!
\sum_{k=n+1}^{{m}} \tau_k 
\leq
C\! \sum_{n=0}^{{m}} \! 3^{n-m} \tau_n\,,
\end{align*}
as well as 
\begin{align*}
\lefteqn{
\E \Bigg[
\sum_{n=0}^{{m}} 3^{n-m} 
\avsum_{y\in 3^n\Zd\cap \cu_{m}} \!\!
\left| e -\s_*^{-1}(y+\cu_n) ( \shom_*(\cu_m) e +\k(y+\cu_n)e - \khom(\cu_m)e )  \right|^2
\Bigg]
} \qquad & 
\notag \\ & 
=
\E \Bigg[
\sum_{n=0}^{{m}} 3^{n-m} 
\avsum_{y\in 3^n\Zd\cap \cu_{m}} 
\left| 
\s_*^{-1}(y+\cu_n) 
\bigl(
(\s_*-\k)(y+\cu_n) - (\shom_*-\khom)(\cu_m)
\bigr)
\right|^2
\Bigg]
\notag \\ & 
\leq 
C
\sum_{n=0}^{{m}} 3^{n-m} 
\var \bigl[  \a_*(\cu_n) \bigr]
+
C
\sum_{n=0}^{{m}} 3^{n-m} 
\left| \bigl(
\ahom_*(\cu_n) - \ahom_* (\cu_m) \bigr) 
\right|^2
\end{align*}
and, finally, 
\begin{align*}
\sum_{n=0}^{{m}} 3^{n-m} 
\left| \bigl(
\ahom_*(\cu_n) - \ahom_* (\cu_m) \bigr) 
\right|^2
&\leq 
C\sum_{n=0}^{{m}} 3^{n-m} 
\left| \bigl(
\ahom_*(\cu_n) - \ahom_* (\cu_m) \bigr) 
\right|
\notag \\ & 
\leq 
C\sum_{n=0}^{{m}} 
3^{n-m} \sum_{k=n+1}^m \tau_k
\leq 
C\sum_{n=0}^{{m}} 
3^{n-m} \tau_n.
\end{align*}
The proof of the lemma is now complete.
\end{proof}

We turn next to the generalization of Lemma~\ref{l.flatness}. It is convenient to formalize this with double-variable algebra. 

\begin{lemma}[Decay of the variance]
\label{l.flatness.nosymm}
There exists a constant~$C(\dataref)<\infty$ such that, for every invertible~$\tilde{\mathbf{s}} \in \R^{d\times d}_{\mathrm{sym}}$,~$\tilde{\k}\in \R^{d\times d}$ and~$m,n\in\N$ with~$\beta m  < n < m$,  
\begin{align}
\label{e.variance.J2.nosymm.sms}
\var \bigl[  \bfA(\cu_m) \bigr] 
&
\leq
C
\min \biggl \{
\sup_{|e|=1}
\E \left[ J\bigl (\cu_n,e,(\tilde{\mathbf{s}}-\tilde{\k}) e\bigr ) +  J^*\bigl (\cu_n,e,(\tilde{\mathbf{s}}+\tilde{\k}) e\bigr )\right]^2
\,,\, \sum_{k=n+1}^m \tau_k
\biggr\} 
\notag \\ & \qquad 
+
C\bigl(1+|\tilde{\mathbf{s}}^{-1}| + |\tilde{\mathbf{s}}| + |\tilde{\k}\tilde{\mathbf{s}}^{-1}\tilde{\k}| \bigr)^4
3^{-d (m-n)}
\,.
\end{align}
\end{lemma}
\begin{proof}
In view of~\eqref{e.bfJ.to.JJstar.minset}, it suffices to show that, for every~$\mathbf{B}\in \R^{2d\times 2d}_{\sym}$ and~$m,n\in\N$ with~$\beta m  < n < m$,  
\begin{align}
\label{e.variance.J2.nosymm}
\var \bigl[  \bfA(\cu_m) \bigr] 
\leq
C \min \biggl\{ 
\sup_{|e|=1}
\E \bigl[ \bfJ(\cu_n,e,\mathbf{B} e) \bigr]^2
\,,
\sum_{k=n+1}^m \tau_k
\biggr\}
+
C\bigl(1+|\mathbf{B}|\bigr)^4
3^{-d (m-n)}
\,.
\end{align}
Note that the form of this estimate is nearly identical to~\eqref{e.variance.J2} in Lemma~\ref{l.flatness}. The proof of~\eqref{e.variance.J2.nosymm} is also almost verbatim that of the latter, but we present it here anyway. 

\smallskip

Like in the first step of the proof of Lemma~\ref{l.flatness}, we use the triangle inequality and Lemma~\ref{l.J.upperfluct} (note that the symmetry assumption is not used in that lemma, as mentioned before its statement) to obtain that 
\begin{align}
\label{e.variance.presplit.nosymm}
\var\bigl [ \bfA(\cu_m) \bigr] 
&
\leq 
2\var \Biggl[ \avsum_{z\in 3^n\Zd\cap\cu_m}
\bfA(z+\cu_n) \Biggr] 
+
2 \E \Biggl[
\biggl| \bfA(\cu_m) - \!\!\!\! \avsum_{z\in 3^n\Zd\cap\cu_m}
\bfA(z+\cu_n) \biggr|^2 \Biggr]
\notag \\ &
\leq
C 3^{-d (m-n)}
+
2 \E \Biggl[
\biggl| \bfA(\cu_m) - \!\!\!\! \avsum_{z\in 3^n\Zd\cap\cu_m}
\bfA(z+\cu_n) \biggr|^2 \Biggr]
\,.
\end{align}
Using~\eqref{e.subadda.nosymm},~\eqref{e.diagonalset.bigA} and~\eqref{e.add.defect.astar.nosymm} and the uniform bounds on~$\bfA$ and~$\bfA_*$ in~\eqref{e.bfA.bounds}, we can make a string of inequalities like the one in~\eqref{e.wrap.around}, which then yields, for every~$\mathbf{B} \in \R^{2d\times 2d}_{\sym}$, 
\begin{equation}
\label{e.wrap.app.nosymm}
\Biggl| \bfA(\cu_m) - \!\!\!\! \avsum_{z\in 3^n\Zd\cap\cu_m}
\bfA(z+\cu_n) \Biggr| 
\leq 
C \!\! \avsum_{z\in 3^n\Zd\cap\cu_m}
\sum_{i=1}^{2d}
\bfJ(z+\cu_n,e_i,\mathbf{B} e_i) 
\,.
\end{equation}
This perfectly matches the bound~\eqref{e.wrap.app} in the symmetric case. 
Squaring it, taking expectations and applying Lemma~\ref{l.J.upperfluct}, we obtain the following analog of~\eqref{e.variance.quad.smack}:
\begin{align}
\label{e.splitJsquared}
\E \Biggl[
\biggl| \bfA(\cu_m) - \avsum_{z\in 3^n\Zd\cap\cu_m}
\bfA(z+\cu_n) \biggr|^2
\Biggr]
\leq
C 
\sum_{i=1}^d
\E \bigl[ \bfJ(\cu_n,e_i,\mathbf{B} e_i) \bigr]^2
+ C(1+|\mathbf{B}|)^4 3^{d(n-m)}
\,.
\end{align}
Alternatively, as in the proof of~\eqref{e.alternatively.tauk}, we can use the nonnegativity of
\begin{equation*}
\avsum_{z\in 3^n\Zd\cap\cu_m}
\bfA(z+\cu_n)
-
\bfA(\cu_m)\,,
\end{equation*}
to get  
\begin{equation}
\label{e.alternatively.tauk.nosymm}
\E \Biggl[
\biggl| \bfA(\cu_m) - \avsum_{z\in 3^n\Zd\cap\cu_m} 
\bfA(z+\cu_n) \biggr|^2
\Biggr]
\leq
C \sum_{k=n+1}^m \tau_k
\,.
\end{equation}
Inserting~\eqref{e.splitJsquared} and~\eqref{e.alternatively.tauk.nosymm} into~\eqref{e.variance.presplit.nosymm} yields~\eqref{e.variance.J2.nosymm}. 
\end{proof}

The combination of the previous two lemmas now yields the following generalization of~\eqref{e.combo}: for every~$k,m,n\in\N$ with~$m\beta < k < n< m$,
\begin{align}
\label{e.combo.nosymm}
\lefteqn{
\sup_{e\in B_1}
\Bigl (
\E \left[ J(\cu_m,e,(\shom_*-\khom)(\cu_m)e) \right]
+\E \left[ J^*(\cu_m,e,(\shom_*+\khom)(\cu_m)e) \right]
\Bigr )
} \quad & 
\notag \\ &
\leq 
C 
\sum_{j=n}^{m} 3^{j-m} 
\bigl( \tau_j + \var[ \a_*(\cu_j) ]  \bigr)
+C 3^{-(m-n)}
\notag \\ & 
\leq 
C\min \biggl \{
\sup_{|e|=1}
\E \bigl[ J\bigl (\cu_k,e,(\shom_*-\khom)(\cu_k) e\bigr ) +  J^*\bigl (\cu_k,e,(\shom_*+\khom)(\cu_k) e\bigr )\bigr]^2
\,,\, \sum_{j=k+1}^m \tau_j
\biggr\} 
\notag \\ & \qquad
+
C 
\bigl( 3^{-(m-n)} + 3^{-d(n-k)} \bigr)
+
C\sum_{j=n}^m 3^{j-m} \tau_j
\,.
\end{align}
At this point, the rest of the argument in the proof of Proposition~\ref{p.algebraic.nosymm} follows very closely to the proof of Proposition~\ref{p.algebraicrate.E}, starting from~\eqref{e.combo}. There are essentially no important differences. For the iteration, we can use Lemma~\ref{l.iteration} as stated. The initial smallness can be obtained by Lemma~\ref{l.pigeonhole}, which is valid in the nonsymmetric case if we replace~$\ahom_*(\cu_n)$ with~$(\shom_*-\khom)(\cu_n)$ in~\eqref{e.lograte}. To avoid excessive repetition, we omit these details. 

\smallskip

We likewise omit the derivation of Theorem~\ref{t.subadd.converge.nosymm} from Proposition~\ref{p.algebraic.nosymm} since the argument is exactly the same as the proofs of Theorem~\ref{t.subadd.converge} and Corollary~\ref{c.subadd.converge} in Chapter~\ref{s.subadd} (see page~\pageref{proof.t.subadd.converge}). Note that Lemmas~\ref{l.J.upperfluct} and~\ref{l.mathcalE.minscale} are valid as stated in the general nonsymmetric case since they only use subadditivity and stationarity, as stated in the paragraphs above the statements of those lemmas. 

\smallskip

This completes our presentation of Proposition~\ref{p.algebraic.nosymm} and Theorem~\ref{t.subadd.converge.nosymm}. We conclude this section with a remark, which is needed in the following sections, concerning localization for the minimal scale~$\X$ in Theorem~\ref{t.subadd.converge.nosymm}.

\begin{remark} 
\label{r.localX.nonsymm}
The statement of Lemma~\ref{l.localX}, in which the random variable~$\X$ is replaced by one which has local dependence, holds verbatim for the random variable~$\X$ appearing in the statement of Theorem~\ref{t.subadd.converge.nosymm}. Compared to the symmetric case, 
we only need to make superficial notation changes to the proof. 
\end{remark}

Theorem~\ref{t.quant.DP} generalizes verbatim to the general nonsymmetric case. 

\begin{theorem}
\label{t.quant.DP.nosymm}Theorem~\ref{t.quant.DP} is valid as stated without the symmetry assumption~\eqref{e.symm}. 
\end{theorem}

The proof of Theorem~\ref{t.quant.DP} does not use the assumption of~\eqref{e.symm} per se; it just uses the estimate~\eqref{e.fvc.bounds.symm}, which was proved under the symmetry assumption. We restate it here: 
\begin{equation}
\label{e.FV.sublinearity.nosymm}
3^{-m} \| \phi_{m,e} \|_{\underline{L}^2(\cu_m)} 
+
3^{-m} \| \bfs_{m,e} \|_{\underline{L}^2(\cu_m)} 
\leq 
C \delta \bigl( \X 3^{-m} \bigr)^{\frac\theta 2}
\,.
\end{equation}
We proved this inequality in the symmetric case by combining~\eqref{e.flatnessqual.sec24} and~\eqref{e.flatnessqual.flux.sec24}, which were proved in Section~\ref{ss.variational}, with the conclusion of 
Corollary~\ref{c.subadd.converge}, in view of the definition of~$\mathcal{E}(m)$ in~\eqref{e.mcE.0}. 
To prove Theorem~\ref{t.quant.DP.nosymm}, we need to define finite-volume correctors and show that the estimate for the coarse-grained matrices stated in Theorem~\ref{t.subadd.converge.nosymm} allows us to control them sufficiently well that we can obtain the same estimate.

\smallskip

The finite-volume correctors are defined in the general case of nonsymmetric coefficients as follows:
\begin{equation}
\label{e.FVCtomu}
\phi_{m,e}: = v(\cdot,\cu_m,0,\shom_*(\cu_m) e) - \ell_e
\,, \quad e\in \Rd\,.
\end{equation}
That is, we define~$\phi_{m,e}$ as the maximizer for~$J(\cu_m,0,\shom_*(\cu_m) e)$ minus the affine function~$\ell_e$ with slope~$e$. Observe that, by~\eqref{e.a.formulas.nosymm}, the spatial averages of~$\nabla \phi_{m,e}$ should be close to zero. 
To obtain the bounds~\eqref{e.FV.sublinearity.nosymm}, we follow the proof in the symmetric case. We first apply the multiscale Poincar\'e inequality (Proposition~\ref{p.MSP}) to get
\begin{align}
\label{e.phime.grad.decomp}
\| \nabla \phi_{m,e} \|_{\Hminusul(\cu_m)} 
&
\leq
C
+
C\!\!
\sum_{n=0}^{{m-1}} 3^{n} 
\biggl( \, \avsum_{z\in 3^n\Zd\cap \cu_{m}} 
\left| \left( \nabla v(\cdot,\cu_m,0,\shom_*(\cu_m) e)  \right)_{z+\cu_n} - e \right|^2 \biggr)^{\!\!\nicefrac12}
\,.
\end{align}
We estimate the second term on the right side using the triangle inequality,~\eqref{e.a.formulas.nosymm} and~\eqref{e.quadresp.nosymm}, to get 
\begin{align}
\lefteqn{
\avsum_{z\in 3^n\Zd\cap \cu_{m}} 
\left| \left( \nabla v(\cdot,\cu_m,0,\shom_*(\cu_m) e)  \right)_{z+\cu_n} - e \right|^2
} 
\qquad & 
\notag \\ & 
\leq
2 \avsum_{z\in 3^n\Zd\cap \cu_{m}} 
\bigl| 
\s_*^{-1} (z{+}\cu_n) \shom_*(\cu_m) - \Id
\bigr|^2
\notag \\ & \qquad \quad 
+
2\avsum_{z\in 3^n\Zd\cap \cu_{m}} 
\bigl\| 
\nabla v(\cdot,\cu_m,0,\shom_*(\cu_m) e) - \nabla v(\cdot,z+\cu_n,0,\shom_*(\cu_m) e)
\bigr\|_{\underline{L}^2(z{+}\cu_m)}^2
\notag \\ & 
\leq
C\avsum_{z\in 3^n\Zd\cap \cu_{m}} 
\left| 
\s_*(z+\cu_n) - \shom_*(\cu_m) \right|^2
\notag \\ & \qquad \quad 
+
C\avsum_{z\in 3^n\Zd\cap \cu_{m}} 
\bigl( 
J(z+\cu_n,0,\shom_*(\cu_m) e) - J(\cu_m,0,\shom_*(\cu_m) e) \bigr)
\label{}
\notag \\ & 
\leq 
C\avsum_{z\in 3^n\Zd\cap \cu_{m}} 
\Bigl( 
\bigl| 
\s_*(z+\cu_n) - \shom_*(\cu_m) \bigr|^2
+
\bigl| \s_*(z+\cu_n) - \s_*(\cu_m) \bigr| 
\Bigr)
\label{e.triangledown.nosymm.betterer}
\\ & 
\leq 
C\mathcal{E}(m)^2 \,,
\label{e.triangledown.nosymm}
\end{align}
where we used~\eqref{e.JJstar1} in the penultimate line, and the definition of~$\mathcal{E}(m)$ in~\eqref{e.mcE.0.nonsymm} in the last line; see the remark in the sentence containing~\eqref{e.oneway} and~\eqref{e.anotherway}. 
Plugging this back into~\eqref{e.phime.grad.decomp} yields
\begin{equation}
\label{e.FCV.nonsymm.grad}
3^{-m} \| \phi_{m,e} \|_{\underline{L}^2(\cu_m)} 
\leq C 
3^{-m} \| \nabla \phi_{m,e} \|_{\Hminusul(\cu_m)} 
\leq
C3^{-m} 
+ 
C\mathcal{E}(m)
\,.
\end{equation}
The bound for the fluxes is obtained similarly, using the second line of~\eqref{e.a.formulas.nosymm} rather than the first. We omit the computation. 
The result is
\begin{equation}
\label{e.FCV.nonsymm.flux}
3^{-m} \| \a \nabla\phi_{m,e} - \ahom e \|_{\Hminusul(\cu_m)} 
\leq
C3^{-m} 
+ 
C\mathcal{E}(m)
\,.
\end{equation}
Together with Theorem~\ref{t.subadd.converge.nosymm} we then deduce that, for every~$\delta \in (0,1)$, there exist a constants~$\alpha(\beta,d,\lambda,\Lambda) \in (0,\tfrac12]$ and~$C(\delta,\dataref)<\infty$, and a random minimal scale~$\X$ such that~$3^m \geq \X$ implies 
\begin{equation} 
\label{e.FVC}
3^{-m} \| \phi_{m,e} \|_{\underline{L}^2(\cu_m)} + 3^{-m} \| \a \nabla\phi_{m,e} - \ahom e \|_{\Hminusul(\cu_m)}  
\leq 
\delta ( \X 3^{-m})^{\alpha}
\,.
\end{equation}
We then define finite-volume flux correctors as in~\eqref{e.fluxcorrect.sec24}. The bound~\eqref{e.dualest.forflux}, together with the previous display, then implies the desired bound on the finite-volume flux correctors. 

\smallskip

This gives us exact analogs of~\eqref{e.flatnessqual.sec24} and~\eqref{e.flatnessqual.flux.sec24}, allowing us to follow the rest of the proof of the estimate~\eqref{e.fvc.bounds.symm} to obtain~\eqref{e.FV.sublinearity.nosymm}, with the aid of Theorem~\ref{t.subadd.converge.nosymm}. With this estimate in hand, as mentioned above, we no longer need to distinguish between the symmetric and nonsymmetric cases. The rest of the proof of Theorem~\ref{t.quant.DP} now applies verbatim to the nonsymmetric case.
This concludes the proof of Theorem~\ref{t.quant.DP.nosymm}.

\subsection{Coarse-grained Poincar\'e and Caccioppoli inequalities}
\label{ss.coarse.graining.inequalities}

In this section, we introduce ``coarse-grained'' Poincar\'e and Caccioppoli inequalities that refine the classical estimates, particularly for coefficient fields that exhibit highly oscillatory behavior on small scales but are more regular on larger scales. While the results presented here are not essential for the main text,\footnote{In particular, Propositions~\ref{p.coarse.grained.Poincare} and~\ref{p.coarse.grained.Caccioppoli} are not used elsewhere in the text and this section can be skipped on a first reading.} we include them to help build intuition about the coarse-grained coefficients because they hold intrinsic and independent interest and because they serve as an introduction to the ideas behind some recent advances in the field of quantitative homogenization~\cite{AK.HC}.

\smallskip

We first recall the statements of the classical Poincar\'e and Caccioppoli inequalities relative to a given uniformly elliptic coefficient field~$\a(x)$ with symmetric and anti-symmetric parts denoted by~$\s(x)$ and~$\k(x)$, respectively, and satisfying the uniform ellipticity condition~\eqref{e.ellipticity.nonsymm} with ellipticity constants~$\lambda$ and~$\Lambda$. 

\begin{itemize}

\item The Poincar\'e inequality asserts that there exists a universal~$C<\infty$ such that, for every~$m\in\N$ and~$u\in H^1(\cu_m)$, 
\begin{equation}
\label{e.classical.Poincare}
3^{-2m} \fint_{\cu_m} \bigl| u - (u)_{\cu_m} \bigr|^2 
\leq
C\lambda^{-1}\fint_{\cu_m} 
\nabla u \cdot \a\nabla u \,.
\end{equation}
\item The Caccioppoli inequality asserts that there exists a universal~$C<\infty$ such that, for every~$m\in\N$ and solution~$u \in \A(\cu_m)$, we have 
\begin{equation}
\label{e.classical.Caccioppoli}
\fint_{\cu_{m-1}} 
\nabla u \cdot \a\nabla u
\leq 
C \Lambda  
3^{-2m} \fint_{\cu_m}  | u-(u)_{\cu_m}|^2 
\,.
\end{equation}
\end{itemize}
These statements are close to being converses of each other, although the cube on the left side of~\eqref{e.classical.Caccioppoli} must be smaller than the one on the right side. Each of these statements is \emph{scale-invariant} in the sense that the statement for general~$m$ follows from the statement for~$m=0$ and a change of coordinates. 

\smallskip

The Poincar\'e inequality in~\eqref{e.classical.Poincare} is stated for the quadratic form~$\nabla u \cdot \a \nabla u$ on the right side instead of~$| \nabla u|^2$; the former however follows immediately from the latter by the first inequality of ~\eqref{e.ellipticity.nonsymm}, which implies that~$\lambda^{-1}\nabla u \cdot \a \nabla u \geq |\nabla u|^2$. The first nontrivial Neumann eigenvalue of Laplacian in the cube~$\cu_m$ is $\pi^2$, so~\eqref{e.classical.Poincare} is valid with~$C= \pi^2 < 10$. The reader will note that the statement of the Poincar\'e is really a statement about the symmetric part~$\s(x)$ of~$\a(x)$---the anti-symmetric part~$\k(x)$ plays no role. 

\smallskip

The proof of~\eqref{e.classical.Caccioppoli} is quite straightforward, so we briefly recall it here. We assume that~$(u)_{\cu_m}=0$ and test the equation for~$u$ with $\varphi^2u$ for some~$\varphi\in C^\infty_c(\cu_m)$ to get 
\begin{align}
\label{e.cacc.cauch}
\int_{\cu_m} \varphi^2 \nabla u \cdot \a\nabla u
&
=
-\int_{\cu_m} 2 \varphi u \nabla \varphi \cdot \a\nabla u
\leq
\frac12 \int_{\cu_m} \varphi^2 \nabla u \cdot \a\nabla u
+
2\Lambda \int_{\cu_m} |\nabla \varphi|^2 u^2 
\,.
\end{align}
Here we used Young's inequality and the fact that~$| \a e | \leq \Lambda^{\nicefrac12} (e\cdot \a (x) e)^{\sfrac12}$ for every~$e\in\Rd$, a consequence of the second inequality of~\eqref{e.ellipticity.nonsymm}. Reabsorbing the first term on the right side of~\eqref{e.cacc.cauch} and then selecting~$\varphi$ so that~$\indc_{\cu_{m-1}} \leq \varphi \leq \indc_{\cu_{m}}$ and~$\| \nabla \varphi\|_{L^\infty(\cu_m)} \leq 2\cdot 3^{-m}$, we obtain~\eqref{e.classical.Caccioppoli} with the constant~$C = 16$. 

\smallskip

The ``coarse-grained'' Poincar\'e and Caccioppoli inequalities presented in the statements below improve upon the classical statements by replacing the ellipticity constants~$\lambda$ and~$\Lambda$ appearing on the right sides of~\eqref{e.classical.Poincare} and~\eqref{e.classical.Caccioppoli} by the \emph{coarse-grained ellipticity constants} defined for each~$m\in\N$ and~$s\in [0,2)$ by 
\begin{equation}
\label{e.coarse.grained.ellipticity}
\left\{
\begin{aligned}
& \overline{\Lambda}_{s}(\cu_m)
:= 
\sup_{k \in \Z, \, k\leq m} 
3^{s(k-m)} 
\max_{z\in 3^k\Zd \cap \cu_m} 
\bigl| \b(z+\cu_k) \bigr|
\,, \\  &
\overline{\lambda}_{s}(\cu_m) 
:=
\biggl(
\sup_{k \in \Z, \, k\leq m} 
3^{s(k-m)} 
\max_{z\in 3^k\Zd \cap \cu_m} 
\bigl| \s_*^{-1}(z+\cu_k) \bigr|\biggl)^{\!-1}
\,.
\end{aligned}
\right.
\end{equation}
It is immediate from~\eqref{e.a.bounds.nosymm} that we have~$\lambda \leq \overline{\lambda}_{s}(\cu_m) \leq \overline{\Lambda}_{s}(\cu_m)\leq \Lambda$, in general. The coarse-grained ellipticity ratio~$\overline{\Lambda}_{s}(\cu_m)/\overline{\lambda}_{s}(\cu_m) $ may be smaller (often much smaller!) than the ratio of uniform ellipticity~$\Lambda/\lambda$. The constants~$\overline{\Lambda}_{s}(\cu_m)$ and~$\overline{\lambda}_{s}(\cu_m)$ are scale-dependent, with the exponent~$s\in [0,2)$ determining the extent to which small scales are discounted. Note that if we take~$s=0$ in~\eqref{e.coarse.grained.ellipticity}, then we recover the usual uniform ellipticity constants (by the Lebesgue differentiation theorem). As~$s$ increases,  the scale discounting becomes more extreme.

\smallskip

The coarse-grained Poincar\'e and Caccioppoli inequalities assert that~\eqref{e.classical.Poincare} and~\eqref{e.classical.Caccioppoli} remain valid if we replace~$\lambda$ and~$\Lambda$ with the coarse-grained ellipticity constants. The price to pay is that the constants~$C$ depend on~$d$ and the scaling discount exponents~$s$, the Poincar\'e inequality is only valid for solutions of the equation (unless~$\a(x)$ is symmetric), and the Caccioppoli inequality has an additional factor, namely a power of the coarse-grained ellipticity ratio. 
\begin{proposition}[Coarse-grained Poincar\'e inequality]
\label{p.coarse.grained.Poincare}
There exists~$C(d)<\infty$ such that, for every~$s\in [0,2)$ and~$u\in \A(\cu_m)$, 
\begin{equation}
\label{e.coarse.grained.Poincare}
3^{-2m} \fint_{\cu_m} \bigl| u - (u)_{\cu_m} \bigr|^2 
\leq
C(2-s)^{-2} \overline{\lambda}^{-1}_{s} (\cu_m) \fint_{\cu_m} 
\nabla u \cdot \a\nabla u
\,.
\end{equation}
Moreover, in the symmetric case~$\a=\a^t$, the inequality holds for every~$u\in H^1(\cu_m)$. 
\end{proposition}

\begin{proposition}[Coarse-grained Caccioppoli inequality]
\label{p.coarse.grained.Caccioppoli}
Let~$s_1,s_2\in [0,2)$ be such that~$s_1+s_2 < 2$.
There exists~$C(s_1,s_2,d)<\infty$ such that, for every~$u\in\A(\cu_m)$,  
\begin{equation}
\label{e.coarse.grained.Caccioppoli}
\fint_{\cu_{m-1}} 
\nabla u \cdot \a\nabla u
\leq 
C \overline{\Lambda}_{s_1}(\cu_m) 
\biggl( \frac{\overline{\Lambda}_{s_1}(\cu_m)}{\overline{\lambda}_{s_2}(\cu_m)} \biggr)^{\!\frac{s_1}{2-s_1-s_2}}
3^{-2m}  \fint_{\cu_m}  | u-(u)_{\cu_m}|^2 
\,.
\end{equation}
\end{proposition}

We remark that the classical inequalities are recovered from the above statements by taking~$s=s_1=s_2=0$.

\smallskip 

The proofs of Propositions~\ref{p.coarse.grained.Poincare} and~\ref{p.coarse.grained.Caccioppoli} are based on the coarse-graining estimates stated in~\eqref{e.energymaps.nonsymm} and~\eqref{e.energymaps.nonsymm.dual} as well as the functional inequalities for weak norms given in the statement and proof the multiscale Poincar\'e inequality (Proposition~\ref{p.MSP}). 

\smallskip

We begin with the proof of the coarse-grained Poincar\'e inequality, which is much simpler. 

\begin{proof}[{Proof of Proposition~\ref{p.coarse.grained.Poincare}}]
Select~$u\in \A(\cu_m)$ and apply~\eqref{e.MSPw} from Proposition~\ref{p.MSP} and then~\eqref{e.energymaps.nonsymm} to get 
\begin{align*}
\bigl\| u - (u)_{\cu_m} \bigr\|_{\underline{L}^2(\cu_m)} 
&
\leq 
C\! \sum_{n=-\infty}^{{m}} \!
3^{n} 
\biggl( 
\avsum_{z\in 3^n\Zd\cap \cu_{m}} 
\bigl|  ( \nabla u )_{z+\cu_n} \bigr|^2
\biggr)^{\!\!\nicefrac12}
\notag \\ & 
\leq 
C\! \sum_{n=-\infty}^{{m}} \!
3^{n} 
\biggl( 
\avsum_{z \in 3^n\Zd\cap \cu_m} \! \!
\bigl| \s_*^{-1}(z+\cu_n) \bigr|  
( \nabla u   )_{z+\cu_n} 
\cdot
\s_*(z+\cu_n) 
( \nabla u )_{z+\cu_n} \!
\biggr)^{\!\!\nicefrac12}
\notag \\ & 
\leq
C \sum_{n=-\infty}^{{m}} 3^{n} 
\max_{z\in 3^n\Zd\cap \cu_m}
\bigl| \s_*^{-1}(z+\cu_n) \bigr|^{\nicefrac12}   
\biggl( 
\avsum_{z\in 3^n\Zd\cap \cu_m} 
\fint_{z+\cu_n}  \nabla u \cdot \a\nabla u
\biggr)^{\!\nicefrac12}
\notag \\ & 
= 
C \biggl(  \fint_{\cu_m} 
\nabla u \cdot \a\nabla u\biggr)^{\nicefrac12}
\sum_{n=-\infty}^{{m}} 3^{n} 
\max_{z\in 3^n\Zd\cap \cu_m }
\bigl| \s_*^{-1}(z+\cu_n) \bigr|^{\nicefrac12}  
\notag \\ & 
\leq 
C \biggl(  \fint_{\cu_m} 
\nabla u \cdot \a\nabla u\biggr)^{\nicefrac12}
\overline{\lambda}^{-\nicefrac12}_{s} (\cu_m)
\sum_{n=-\infty}^{{m}} 3^{n-\frac12s(n-m)} 
\notag \\ & 
=
C3^m(2-s)^{-1} \overline{\lambda}^{-\nicefrac12}_{s} (\cu_m) 
\biggl(  \fint_{\cu_m} 
\nabla u \cdot \a\nabla u\biggr)^{\nicefrac12}
\,.
\end{align*}
In the symmetric case, we may apply~\eqref{e.energymaps} instead of~\eqref{e.energymaps.nonsymm}, which we note is valid for all~$u\in H^1(\cu_m)$, not only solutions. 
The proof is now complete. 
\end{proof}

Before proving the coarse-grained Caccioppoli inequality, we present a preliminary estimate that closely resembles~\eqref{e.coarse.grained.Caccioppoli}, but includes an additional term on the right-hand side and employs slightly stronger versions of the coarse-grained ellipticity constants.

\smallskip

The main step in the proof of the classical Caccioppoli inequality lies in using ellipticity to bound the middle expression in~\eqref{e.cacc.cauch}: the integral~$\fint_{\cu_m} u\varphi \nabla \varphi \cdot \a\nabla u$. Our goal is to estimate this integral in terms of coarse-grained ellipticity constants rather than  the uniform ellipticity upper bound~$\Lambda$. Our estimate is found below in~\eqref{e.the.real.Caccioppoli}, which should be compared with the inequality in~\eqref{e.cacc.cauch}. The final term on the right-hand side of~\eqref{e.the.real.Caccioppoli} has no counterpart in~\eqref{e.cacc.cauch}; this extra term represents a coarse-graining error, which can typically be made very small by a separation of scales (selecting the free parameter~$n$ so that~$3^n$ is small compared to the length scale~$\min\{ \| \nabla  \phi\|_{L^\infty}^{-1}, \| \nabla^2  \phi\|_{L^\infty}^{-\nicefrac12} \}$ of the cutoff function~$\varphi$).

\begin{lemma}
\label{l.coarse.grained.Caccioppoli.prelim}
Define variants of~\eqref{e.coarse.grained.ellipticity} for each~$m,n\in\N$ with~$n\leq m$ and~$s\in (0,1]$ by 
\begin{equation}
\label{e.CG.lambdas.ALT}
\left\{
\begin{aligned}
& \overline{\Lambda}_{s,n}(\cu_m)
:=
\biggl(  
\sum_{k=-\infty}^{n} 
3^{s(k-n)} 
\max_{z\in 3^k\Zd \cap \cu_m} 
\bigl| \b(z+\cu_k) \bigr|^{\nicefrac12} \biggr)^{\!2}
\,, \\  &
\overline{\lambda}_{s,n}(\cu_m) 
:=
\biggl( 
\sum_{k=-\infty}^{n} 
3^{s(k-n)} 
\max_{z\in 3^k\Zd \cap \cu_m} 
\bigl| \s_*^{-1}(z+\cu_k) \bigr|^{\nicefrac12} 
\biggr)^{\!-2}
\,.
\end{aligned}
\right.
\end{equation}
There exists~$C(d)<\infty$ such that, for every~$\ep \in (0,1]$,~$m,n\in\N$ with~$n<m$,~$\varphi\in C^2 (\cu_m)$ and solution~$u\in \A(\cu_m)$,
we have
\begin{align}
\label{e.the.real.Caccioppoli}
\lefteqn{
\biggl| \fint_{\cu_m} u\varphi \nabla \varphi \cdot \a\nabla u 
\biggr| 
} \qquad &
\notag \\ &
\leq
\ep \fint_{\cu_m} \varphi^2 \nabla u\cdot \a\nabla u 
+
\frac C\ep\overline{ \Lambda}_{s,n}(\cu_m) \bigl( \bigl\|  \nabla \varphi  \bigr\|_{L^\infty}^2 
+
 3^{2n} \bigl\|  \nabla^2 \varphi  \bigr\|_{L^\infty}^2 \bigr)   \fint_{\cu_m} u^2
\notag \\ & \qquad 
+
C\Bigl( \frac1{\ep s^2 }\frac{\overline{\Lambda}_{s,n}(\cu_m)}{\overline{\lambda}_{1-s,n}(\cu_m)} 3^{2n} \|  \nabla \varphi\|_{L^\infty}^2 + \ep 3^{4n} \| \nabla^2 \varphi\|_{L^\infty}^2 \Bigr)   
\fint_{\cu_m} \nabla u\cdot \a\nabla u
\end{align}
and, consequently, if~$\varphi^2 u \in H^1_0(\cu_m)$, then 
\begin{align}
\label{e.coarse.grain.your.Caccioppoli.withextraterm}
\lefteqn{ 
\fint_{\cu_m} 
\varphi^2 \nabla u \cdot \a\nabla u
} \qquad & 
\notag \\ & 
\leq 
C
\overline{\Lambda}_{s,n}(\cu_m) 
\bigl( \bigl\|  \nabla \varphi  \bigr\|_{L^\infty}^2 
+
 3^{2n} \bigl\|  \nabla^2 \varphi  \bigr\|_{L^\infty}^2 \bigr)
\fint_{\cu_m}  u^2 
\notag \\ & \qquad
+
C
\Bigl( \frac1{s^2 \ep}\frac{\overline{\Lambda}_{s,n}(\cu_m)}{\overline{\lambda}_{1-s,n}(\cu_m)} 3^{2n} \| \nabla \varphi\|_{L^\infty}^2 + \ep 3^{4n} \| \nabla^2 \varphi\|_{L^\infty}^2   \Bigr)  
\fint_{\cu_m} \nabla u\cdot \a\nabla u
\,.
\end{align}
\end{lemma}
\begin{proof}
The second inequality~\eqref{e.coarse.grain.your.Caccioppoli.withextraterm} follows from the first one~\eqref{e.the.real.Caccioppoli}. To see this, we test the equation for~$u\in \A(\cu_m)$ with~$\varphi^2 u$, as in~\eqref{e.cacc.cauch}, to obtain 
\begin{align}
\fint_{\cu_m} 
\varphi^2 \nabla u \cdot \a\nabla u 
=
- \fint_{\cu_m} 2 u\varphi \nabla \varphi \cdot \a\nabla u
\,.
\end{align}
Applying~\eqref{e.the.real.Caccioppoli} with~$\ep = \nicefrac14$ to estimate the right side, we then reabsorb the integral of~$\varphi^2 \nabla u \cdot \a\nabla u$ on the left side. The result is~\eqref{e.coarse.grain.your.Caccioppoli.withextraterm}.

\smallskip

We turn to the proof of~\eqref{e.the.real.Caccioppoli}. Fix~$m,n\in\N$ with~$n<m$ and split the integral on the left side of~\eqref{e.the.real.Caccioppoli} as 
\begin{align}
\label{e.Caccioppoli.testing}
\fint_{\cu_m} u\varphi \nabla \varphi \cdot \a\nabla u
=
\avsum_{z\in 3^n\Zd \cap \cu_m} 
\fint_{z+\cu_n} u\varphi \nabla \varphi \cdot \a\nabla u
\,.
\end{align}
We intend to apply the functional inequality~\eqref{e.general.Besov.pairing} to the integrals on the right side of~\eqref{e.Caccioppoli.testing}. 
Recall that this inequality asserts that, for every~$s\in(0,1]$ and~$f,g\in L^2(z+\cu_n)$ with~$(g)_{z+\cu_n} = 0$, 
\begin{align}
\label{e.general.Besov.pairing.again}
\biggl| \fint_{z+\cu_n} f g  \biggr| 
&
\leq 
3^{d+s}
\sup_{k\leq n} 
3^{-s(k-n)} 
\biggl( 
\avsum_{z'\in 3^{k} \Zd\cap (z+\cu_n)} 
\fint_{z'+\cu_k} 
\bigl| g -
(g)_{z'+\cu_{k}} \bigr|^2 
\biggr)^{\!\nicefrac12}
\notag \\ & \qquad\qquad\qquad \times
\sum_{k=-\infty}^{n-1} 
3^{s(k-n)}
\biggl( 
\avsum_{z'\in 3^{k} \Zd\cap (z+\cu_n)} 
\bigl| (f)_{z'+\cu_{k}} \bigr|^2 
\biggr)^{\!\nicefrac12} 
\,.
\end{align}
We will apply~\eqref{e.general.Besov.pairing.again} with~$f = \a\nabla u$ and~$g = u\varphi \nabla \varphi - (u\varphi \nabla \varphi)_{z+\cu_n}$. 
We first estimate each of the factors on the right side. 

\smallskip

For the second factor, we use the coarse-graining inequality~\eqref{e.energymaps.nonsymm.dual} to find  
\begin{align}
\label{e.Cacc.first.factor}
\lefteqn{ 
\sum_{k=-\infty}^{n-1} 
3^{s(k-n)}
\biggl( 
\avsum_{z'\in 3^{k} \Zd\cap (z+\cu_n)} 
\bigl| ( \a\nabla u)_{z'+\cu_{k}} \bigr|^2 
\biggr)^{\!\nicefrac12} 
} \quad & 
\notag \\ &
\leq 
\sum_{k=-\infty}^{n-1} 
3^{s(k-n)}
\biggl( 
\avsum_{z'\in 3^{k} \Zd\cap (z+\cu_n)} 
\bigl| \b(z'+\cu_k) \bigr| 
( \a\nabla u)_{z'+\cu_{k}} \cdot \b^{-1}(z'+\cu_k)( \a\nabla u)_{z'+\cu_{k}}
\biggr)^{\!\nicefrac12} 
\notag \\ & 
\leq 
\sum_{k=-\infty}^{n-1} 
3^{s(k-n)} 
\max_{z'\in 3^k\Zd \cap (z+\cu_n)} 
\bigl| \b(z'+\cu_k) \bigr| ^{\nicefrac12} 
\biggl( 
\avsum_{z'\in 3^{k} \Zd\cap (z+\cu_n)} 
\fint_{z'+\cu_k} \nabla u \cdot \a\nabla u
\biggr)^{\!\nicefrac12} 
\notag \\ & 
=
\sum_{k=-\infty}^{n-1} 
3^{s(k-n)} 
\max_{z'\in 3^k\Zd \cap (z+\cu_n)} 
\bigl| \b(z'+\cu_k) \bigr| ^{\nicefrac12} 
\biggl( 
\fint_{z+\cu_n} \nabla u \cdot \a\nabla u
\biggr)^{\!\nicefrac12} 
\notag \\ & 
\leq 
\overline{\Lambda}_{s,n}^{\nicefrac12} 
(\cu_m)
\biggl( 
\fint_{z+\cu_n} \nabla u \cdot \a\nabla u
\biggr)^{\!\nicefrac12}
\,.
\end{align}
The estimate for the first factor is more complicated, because we need to make~$\nabla u$ appear. We begin by observing that
\begin{align*}
\lefteqn{ 
\fint_{z'+\cu_k} 
\bigl| u\varphi \nabla \varphi  -
(u\varphi \nabla \varphi )_{z'+\cu_{k}} \bigr|^2
} \ & 
\notag \\ & 
\leq
6\fint_{z'+\cu_k} 
\varphi^2 \big| u - (u)_{z'+\cu_{k}}  \bigr|^2  
\bigl| \nabla \varphi\bigr|^2
+
3\bigl| 
(u)_{z'+\cu_{k}} \bigr|^2 
\fint_{z'+\cu_k} \bigl|
 \varphi \nabla \varphi
- ( \varphi \nabla \varphi)_{z'+\cu_{k}} \bigr|^2
\notag \\ & 
\leq 
C \| \varphi \nabla \varphi\|_{L^\infty(z'+\cu_k)} ^2
\fint_{z'+\cu_k} 
 \big| u - (u)_{z'+\cu_{k}}  \bigr|^2  
+
C 3^{2k}\bigl\| | \nabla^2 \varphi|  |\varphi| + | \nabla \varphi|^2 \bigr\|_{L^\infty(z'+\cu_{k})}^2  
\fint_{z'+\cu_k}
u^2 
\,.
\end{align*}
The first term in the previous display is estimated using Proposition~\ref{p.MSP}, namely~\eqref{e.MSPw}. We get 
\begin{align*}
\fint_{z'+\cu_k} 
\bigl| u  -
(u )_{z'+\cu_{k}} \bigr|^2 
&
\leq 
C\biggl(  \sum_{l=-\infty}^{{k}} 3^{l} 
\biggl( 
\avsum_{z''\in 3^l\Zd\cap (z'+\cu_{k})} 
\bigl|  ( \nabla u   )_{z''+\cu_l} \bigr|^2
\biggr)^{\!\!\nicefrac12}\biggr)^2
\,.
\end{align*}
Summing this over~$z'\in 3^k\Zd \cap (z+\cu_n)$ and using H\"older's inequality, we get, for every parameter~$s \in (0,1)$, 
\begin{align*}
\lefteqn{ 
\avsum_{z'\in 3^k\Zd\cap (z+\cu_n)}
\fint_{z'+\cu_k}
\bigl| u   -
(u  )_{z'+\cu_{k}} \bigr|^2 
} \qquad \qquad & 
\notag \\ & 
\leq 
C \avsum_{z'\in 3^k\Zd\cap (z+\cu_n)}
\biggl(  \sum_{l=-\infty}^{{k}} 3^{l} 
\biggl( 
\avsum_{z''\in 3^l\Zd\cap (z'+\cu_{k})} 
\bigl|  ( \nabla u )_{z''+\cu_l} \bigr|^2
\biggr)^{\!\!\nicefrac12}\biggr)^2
\notag \\ & 
\leq 
C
\underbrace{
\sum_{l=-\infty}^k
3^{2sl}
}_{
\leq Cs^{-1}3^{2sk} 
}
\sum_{l=-\infty}^k 
3^{2(1-s)l} 
\avsum_{z'\in 3^l\Zd\cap (z+\cu_{n})} 
\bigl|  ( \nabla u )_{z'+\cu_l} \bigr|^2
\,.
\end{align*}
Next, we apply the coarse-graining inequality~\eqref{e.energymaps.nonsymm} in a similar manner as in the proof of Proposition~\ref{p.coarse.grained.Poincare} to get 
\begin{align*}
\avsum_{z'\in 3^l\Zd\cap (z+\cu_{n})} 
\bigl|  ( \nabla u   )_{z'+\cu_l} \bigr|^2
&
\leq 
\avsum_{z'\in 3^l\Zd\cap (z+\cu_{n})} 
\bigl| \s_*^{-1}(z'+\cu_l) \bigr|  
( \nabla u   )_{z'+\cu_l} 
\cdot
\s_*(z'+\cu_l) 
( \nabla u   )_{z'+\cu_l} 
\notag \\ & 
\leq 
\max_{z'\in 3^l\Zd\cap (z+\cu_{n})}
\bigl| \s_*^{-1}(z'+\cu_l) \bigr|  
\avsum_{z'\in 3^l\Zd\cap (z+\cu_{n})} 
\fint_{z'+\cu_l}  \nabla u \cdot \a\nabla u
\notag \\ & 
\leq 
\max_{z'\in 3^l\Zd\cap (z+\cu_{n})}
\bigl| \s_*^{-1}(z'+\cu_l) \bigr| 
\fint_{z+\cu_n} 
\nabla u \cdot \a\nabla u
\,.
\end{align*}
Thus, by taking supremum, we obtain 
\begin{align*} 
\lefteqn{
\sup_{k\leq n} 
3^{-2sk} 
\avsum_{z'\in 3^k\Zd\cap (z+\cu_n)}
\fint_{z'+\cu_k}
\bigl| u   -
(u  )_{z'+\cu_{k}} \bigr|^2 
} \qquad & 
\notag \\ & 
\leq
C s^{-1}  \fint_{z+\cu_n}  \nabla u \cdot \a\nabla u
\sum_{l=-\infty}^n 
3^{2(1-s)l}  \max_{z'\in 3^l\Zd\cap (z+\cu_{n})}
\bigl| \s_*^{-1}(z'+\cu_l) \bigr|
\,.
\end{align*}
Furthermore, the last sum can be estimated by means of~$\overline{\lambda}_{1-s,n}(\cu_m)$ as follows:
\begin{align*} 
\lefteqn{
\sum_{l=-\infty}^n 
3^{2(1-s)l}  \max_{z'\in 3^l\Zd\cap (z+\cu_{n})}
\bigl| \s_*^{-1}(z'+\cu_l) \bigr| 
} \qquad &  
\notag \\ & 
\leq
\biggl( \sum_{l=-\infty}^n 
3^{(1-s)l}  \max_{z'\in 3^l\Zd\cap \cu_{m}}
\bigl| \s_*^{-1}(z'+\cu_l) \bigr|^{\nicefrac12} 
\biggr)^{2}
=
3^{2(1-s)n} \overline{\lambda}_{1-s,n}^{-1}(\cu_m) 
\,.
\end{align*}
Combining the above displays yields
\begin{align}
\label{e.ready.to.rock}
\lefteqn{ 
\sup_{k\leq n} 
3^{-s(k-n)} 
\biggl( 
\avsum_{z'\in 3^{k} \Zd\cap (z+\cu_n)} 
\fint_{z'+\cu_k} 
\bigl| u\varphi \nabla \varphi  -
(u\varphi \nabla \varphi )_{z'+\cu_{k}} \bigr|^2 
\biggr)^{\!\nicefrac12}
} \qquad &
\notag \\ & 
\leq
Cs^{-1} 3^{n}  \overline{\lambda}_{1-s,n}^{-\nicefrac12}(\cu_m)  \| \varphi \nabla \varphi\|_{L^\infty(z+\cu_n)}
\biggl( \fint_{z+\cu_n}  \nabla u \cdot \a\nabla u \biggr)^{\!\nicefrac12} 
\notag \\ & \qquad 
+
C3^{n} \bigl\| | \nabla^2 \varphi|  |\varphi| + | \nabla \varphi|^2 \bigr\|_{L^\infty(z+\cu_{n})}
\biggl( \fint_{z+\cu_n} u^2 \biggr)^{\!\nicefrac12}
\,.
\end{align}
Using
\begin{equation}
\label{e.idiot}
\| \varphi\|_{L^\infty(z+\cu_n)}^2 
\fint_{z+\cu_n} \nabla u \cdot \a\nabla u
\leq 
\fint_{z+\cu_n} \bigl(\varphi^2 + C 3^{2n} \| \nabla \varphi \|_{L^\infty(z+\cu_{n})}^2 \bigr) \nabla u \cdot \a\nabla u\,,
\end{equation}
we also get
\begin{align}
\label{e.rocked}
\lefteqn{ 
\sup_{k\leq n}  3^{-s(k-n)} 
\biggl( 
\avsum_{z'\in 3^{k} \Zd\cap (z+\cu_n)} 
\fint_{z'+\cu_k} 
\bigl| u\varphi \nabla \varphi  -
(u\varphi \nabla \varphi )_{z'+\cu_{k}} \bigr|^2 
\biggr)^{\!\nicefrac12}
} \qquad &
\notag \\ & 
\leq 
Cs^{-1} 3^n \| \nabla \varphi\|_{L^\infty(z+\cu_{n})}
\overline{\lambda}^{-\nicefrac12}_{1-s,n}(\cu_m)
\biggl( \fint_{z+\cu_n} \bigl(\varphi^2 + 3^{2n} \| \nabla \varphi \|_{L^\infty(z+\cu_{n})}^2 \bigr) \nabla u \cdot \a\nabla u
\biggr)^{\!\nicefrac12}
\notag \\ & \qquad 
+
C3^n \bigl\| | \nabla^2 \varphi|  |\varphi| + | \nabla \varphi|^2 \bigr\|_{L^\infty(z+\cu_{n})} 
\biggl( \fint_{z+\cu_n} u^2 \biggr)^{\!\nicefrac12}
\,.
\end{align}
We next insert~\eqref{e.Cacc.first.factor} and~\eqref{e.rocked} into~\eqref{e.general.Besov.pairing.again}, use Young's inequality and rearrange some terms, to obtain, for every~$K>0$, 
\begin{align*}
\lefteqn{
\biggl| \fint_{z+\cu_n} u\varphi \nabla \varphi \cdot \a\nabla u 
-
\fint_{z+\cu_n} u\varphi \nabla \varphi \cdot \fint_{z+\cu_n}\a\nabla u
\biggr| 
} \qquad & 
\notag \\ & 
\leq 
Cs^{-1}3^{n}  \| \nabla \varphi\|_{L^\infty} 
\Bigl( \frac{\overline{\Lambda}_{s,n}(\cu_m)}{\overline{\lambda}_{1-s,n}(\cu_m)} \Bigr)^{\nicefrac12} 
\frac1K \fint_{z+\cu_n} \varphi^2 \nabla u \cdot \a\nabla u
\notag \\ &  \qquad 
+
Cs^{-1} 3^{n}\| \nabla \varphi\|_{L^\infty}  
\Bigl( \frac{\overline{\Lambda}_{s,n}(\cu_m)}{\overline{\lambda}_{1-s,n}(\cu_m)} \Bigr)^{\nicefrac12}
\bigl( K + 3^{n}\| \nabla \varphi\|_{L^\infty} \bigr)
\fint_{z+\cu_n} \nabla u \cdot \a\nabla u
\notag \\ &  \qquad 
+ 
C3^{n}\bigl\| | \nabla^2 \varphi|  |\varphi| + | \nabla \varphi|^2 \bigr\|_{L^\infty}   \overline{\Lambda}_{s,n}^{\nicefrac12} (\cu_m)
\biggl( \fint_{z+\cu_n} u^2 \biggr)^{\!\nicefrac12}\!
\biggl( 
\fint_{z+\cu_n} \nabla u \cdot \a\nabla u
\biggr)^{\!\nicefrac12}
\,.
\end{align*}
We use Young's inequality, the selection 
\begin{equation*}
K:= \frac{C}{s \ep} 3^{n}\| \nabla \varphi\|_{L^\infty}\Bigl( \frac{\overline{\Lambda}_{s,n}(\cu_m)}{\overline{\lambda}_{1-s,n}(\cu_m)} \Bigr)^{\nicefrac12} 
\,,
\end{equation*}
for sufficiently large constant~$C(d)<\infty$ and
\begin{align*} 
\lefteqn{
3^n \bigl\| | \nabla^2 \varphi|  |\varphi|  \bigr\|_{L^\infty} 
\overline{\Lambda}_{s,n}^{\nicefrac12} (\cu_m)
\biggl( \fint_{z+\cu_n} u^2 \biggr)^{\!\nicefrac12}\!
\biggl( 
\fint_{z+\cu_n} \nabla u \cdot \a\nabla u
\biggr)^{\!\nicefrac12}
} \qquad &
\notag \\ &
\leq
C3^n \bigl\|  \nabla^2 \varphi  \bigr\|_{L^\infty} 
 \overline{\Lambda}_{s,n}^{\nicefrac12} (\cu_m)
\biggl( \fint_{z+\cu_n} u^2 \biggr)^{\!\nicefrac12}\!
\biggl( 
\fint_{z+\cu_n} \bigl(\varphi^2 + 3^{2n} |\nabla \varphi|^2 \bigr)\nabla u \cdot \a\nabla u
\biggr)^{\!\nicefrac12}
\notag \\ &
\leq
\frac{C}{\ep}  \bigl(   \bigl\|  \nabla \varphi  \bigr\|_{L^\infty}^2 + 3^{2n} \bigl\|  \nabla^2 \varphi  \bigr\|_{L^\infty}^2 \bigr)
 \overline{\Lambda}_{s,n}(\cu_m)
\fint_{z+\cu_n} u^2 
\notag \\ & \qquad 
+ \frac{\ep}{4} \fint_{z+\cu_n} \varphi^2 \nabla u \cdot \a\nabla u 
+ \frac{\ep}{4} 3^{4n} \bigl\|  \nabla^2 \varphi  \bigr\|_{L^\infty}^2  \fint_{z+\cu_n} \nabla u \cdot \a\nabla u 
\end{align*}
to obtain
\begin{align} 
\label{e.dont.stop.believing}
\lefteqn{
\biggl| \fint_{z+\cu_n} u\varphi \nabla \varphi \cdot \a\nabla u 
-
\fint_{z+\cu_n} u\varphi \nabla \varphi \cdot \fint_{z+\cu_n}\a\nabla u
\biggr| 
} \qquad &
\notag \\ &
\leq
\frac{\ep}{2} \fint_{z+\cu_n} \varphi^2 \nabla u \cdot \a\nabla u 
+
\frac{C}{\ep}  \bigl( \bigl\|  \nabla \varphi  \bigr\|_{L^\infty}^2 + 3^{2n} \bigl\|  \nabla^2 \varphi  \bigr\|_{L^\infty}^2  \bigr)
 \overline{\Lambda}_{s,n}(\cu_m)
\fint_{z+\cu_n} u^2 
\notag \\ & \qquad 
+
C
\Bigl( \frac1{s^2 \ep}\frac{\overline{\Lambda}_{s,n}(\cu_m)}{\overline{\lambda}_{1-s,n}(\cu_m)} 3^{2n} \| \nabla \varphi\|_{L^\infty}^2 + \ep 3^{4n} \| \nabla^2 \varphi\|_{L^\infty}^2   \Bigr)  
\fint_{\cu_m} \nabla u\cdot \a\nabla u
\,.
\end{align}
We next estimate, using the coarse-graining inequality~\eqref{e.energymaps.nonsymm.dual} and also~\eqref{e.idiot} again, 
\begin{align*}
\lefteqn{ 
\biggl| 
\fint_{z+\cu_n} u\varphi \nabla \varphi \cdot \fint_{z+\cu_n}\a\nabla u
\biggr|^2
} \quad & 
\notag \\ & 
\leq 
C\| \varphi \nabla \varphi\|_{L^\infty}^2 |\b(z+\cu_n)|
\Bigl( 
(\a\nabla u)_{z+\cu_n} \cdot \b^{-1}(z+\cu_n)(\a\nabla u)_{z+\cu_n}\Bigr)
\fint_{z+\cu_n} u^2
\notag \\ & 
\leq 
C\| \varphi \nabla \varphi\|_{L^\infty}^2 |\b(z+\cu_n)|
\biggl( 
\fint_{z+\cu_n} 
\nabla u \cdot \a\nabla u
\biggr)
\fint_{z+\cu_n} u^2
\notag \\ & 
\leq 
C\| \nabla \varphi\|_{L^\infty}^2 
\overline{\Lambda}_{s,n}(\cu_m) 
\biggl( 
\fint_{z+\cu_n} 
(\varphi^2 +C3^{n}\| \nabla \varphi\|_{L^\infty}) \nabla u \cdot \a\nabla u
\biggr)
\fint_{z+\cu_n} u^2
\,.
\end{align*}
Therefore, by Young's inequality,
\begin{align*}
\lefteqn{ 
\biggl| 
\fint_{z+\cu_n} u\varphi \nabla \varphi \cdot \fint_{z+\cu_n}\a\nabla u
\biggr|
} \quad & 
\notag \\ & 
\leq 
C\| \nabla \varphi\|_{L^\infty}
\overline{\Lambda}_{s,n}^{\nicefrac12} (\cu_m) 
\biggl( 
\fint_{z+\cu_n} 
(\varphi^2 +C3^{n}\| \nabla \varphi\|_{L^\infty}) \nabla u \cdot \a\nabla u
\biggr)^{\!\nicefrac12} 
\biggl( \fint_{z+\cu_n} u^2 \biggr)^{\!\nicefrac12}
\notag \\ & 
\leq 
\frac{C}{\ep}  \| \nabla \varphi\|_{L^\infty}^2 \overline{\Lambda}_{s,n}(\cu_m) \fint_{z+\cu_n} u^2  
+ \frac{\ep}{2} \fint_{z+\cu_n} \varphi^2 \nabla u \cdot \a\nabla u 
+ \frac{\ep}{2} 3^{2n} \bigl\|  \nabla \varphi  \bigr\|_{L^\infty}^2  \fint_{z+\cu_n} \nabla u \cdot \a\nabla u 
\,.
\end{align*}
Combining~\eqref{e.dont.stop.believing} and the previous display, summing over~$z\in 3^n\Zd \cap\cu_m$, we arrive at~\eqref{e.the.real.Caccioppoli}, completing the proof. 
\end{proof}

To complete the proof of Proposition~\ref{p.coarse.grained.Caccioppoli}, we will show by a covering argument that the second term on the right side of~\eqref{e.coarse.grain.your.Caccioppoli.withextraterm} can be removed at the cost of increasing the discounting parameters in the coarse-grained ellipticity constants and increasing the first term by a factor which is a power of the coarse-grained ellipticity ratio.

\begin{proof}[Proof of Proposition~\ref{p.coarse.grained.Caccioppoli}] 
Let~$u \in \mathcal{A}(\cu_m)$, and fix~$s_1,s_2 \in [0,2)$ with~$s_1+s_2 < 2$. Define
\begin{equation*} 
s := \frac12 s_1 + \frac14 ( 2- s_1 - s_2) \in (\nicefrac{s_1}2,1) \cap (0,1-\nicefrac {s_2}2) \,.
\end{equation*}
Let~$l,h \in \Z$ with~$h<l <m$ and let~$z \in 3^l \Zd$. 
In view of the definitions in~\eqref{e.CG.lambdas.ALT}, we have  
\begin{align*} 
\overline{\Lambda}_{s,h}(z+\cu_l) 
&
=
\biggl(  
\sum_{k=-\infty}^{h} 
3^{s(k-h)} 
\max_{z'\in 3^k\Zd \cap (z+\cu_l)} 
\bigl| \b(z'+\cu_k) \bigr|^{\nicefrac12} \biggr)^{\!2}
\notag 
\\  & \leq
3^{s_1(m-h)}
\overline{\Lambda}_{s_1}(\cu_m) 
\biggl(  
\sum_{k=-\infty}^{h} 
3^{(s - \frac12 s_1) (k-h)} 
\biggr)^{\!2}
\leq \frac{C3^{s_1(m-h)}}{(2- s_1 - s_2)^{2}}
\overline{\Lambda}_{s_1}(\cu_m) 
\end{align*}
and, similarly,
\begin{equation*} 
\overline{\lambda}_{1-s,h}^{-1}(z+\cu_l)
\leq
 \frac{C3^{s_2(m-h)}}{(2 -s_1 -s_2)^{2}}
\overline{\lambda}_{s_2}^{-1}(\cu_m)
\,.
\end{equation*}
We apply Lemma~\ref{l.coarse.grained.Caccioppoli.prelim} where~$\varphi\in C^\infty_c (\cu_m)$ is chosen to satisfy~$\indc_{\cu_{m-1}} \leq \varphi \leq \indc_{\cu_m}$ and 
\begin{equation}
\label{e.cutoff.function}
3^{jm} \|\nabla^j \varphi\|_{L^\infty(\cu_m)} \leq C\,, \quad \forall j \in \{0,1,2\}\,,
\end{equation}
and then we use the previous two displays to obtain
\begin{align}
\label{e.coarse.grain.your.Caccioppoli.withextraterm.again}
\fint_{z+\cu_{l-1}} 
\nabla u \cdot \a\nabla u
& 
\leq 
\frac{C3^{s_1(m-h)}}{(2 -s_1-s_2)^{2}}
\overline{\Lambda}_{s_1}(\cu_m)
3^{-2l}\fint_{z+\cu_l} \!\! u^2 
\notag \\ & \qquad 
+\frac{C3^{2(h-l) + (s_1+s_2)(m-h)}}{(2- s_1 - s_2)^{6}}
\frac{\Lambda_{s_1}(\cu_m)}{\lambda_{s_2}(\cu_m)}
\fint_{z+\cu_l} \!\!
\nabla u \cdot \a\nabla u 
\,.
\end{align}
Next, we take~$r ,R \in [\nicefrac12,1]$ such that~$r<R$, and denote 
\begin{equation*} 
r \cu_m := \bigl( - r \tfrac{1}{2} 3^{m}, r \tfrac{1}{2} 3^{m} \bigr)^d\,.
\end{equation*} 
We find~$l \in\Z$ such that~$3^l \leq c(R-r) 3^m$ and
\begin{equation*} 
3^{s_1(m-h)} 3^{-2l}
\leq \frac{C3^{s_1(l-h)}}{(R-r)^{s_1+2}} 3^{-2m}
\qand
3^{2(h-l) + (s_1+s_2)(m-h)}
\leq \frac{C3^{-(2-s_1-s_2)(l-h)}}{(R-r)^{s_1+s_2}}  \,.
\end{equation*}
By covering~$r \cu_m$ with cubes of the type~$z+\cu_{l-1}$ with~$z \in 3^{l -1}\Zd$ and~$z+\cu_l \subseteq R\cu_m$, we obtain, using~\eqref{e.coarse.grain.your.Caccioppoli.withextraterm.again}, 
\begin{align*}
\fint_{r \cu_m} 
\nabla u \cdot \a\nabla u
& 
\leq 
\frac{C3^{s_1(l-h)}}{(2 - s_1 - s_2)^{2} (R-r)^{s_1+2}}
\overline{\Lambda}_{s_1}(\cu_m)
3^{-2m}\fint_{\cu_m} \!\! u^2 
\notag \\ & \qquad 
+\frac{C3^{-(2-s_1-s_2)(l-h)}}{(2- s_1 - s_2)^{6} (R-r)^{s_1+s_2}}
\frac{\Lambda_{s_1}(\cu_m)}{\lambda_{s_2}(\cu_m)}
\fint_{R \cu_m} \!\!
\nabla u \cdot \a\nabla u 
\,.
\end{align*}
We now select~$h \in \Z$ so small that
\begin{equation*} 
\frac{C3^{-(2-s_1-s_2)(l-h)}}{(2 -s_1 -s_2)^{6} (R-r)^{s_1+s_2}}
\frac{\Lambda_{s_1}(\cu_m)}{\lambda_{s_2}(\cu_m)} \leq \frac12\,.
\end{equation*}
From the previous two displays we deduce the existence of constants~$C(s_1,s_2,d)$ and~$\xi(s_1,s_2)$ such that
\begin{align*} 
\fint_{r \cu_m} 
\nabla u \cdot \a\nabla u
&
\leq
\frac12
\fint_{R \cu_m} 
\nabla u \cdot \a\nabla u
+ 
\frac{C\overline{\Lambda}_{s_1}(\cu_m)}{(R-r)^{\xi}} \Bigl( \frac{\Lambda_{s_1}(\cu_m)}{\lambda_{s_2}(\cu_m)} \Bigr)^{\frac{s_1}{2-s_1-s_2}}  3^{-2m}\fint_{\cu_m} \!\! u^2 
\,.
\end{align*}
This inequality implies the desired conclusion~\eqref{e.coarse.grained.Caccioppoli}, using a standard iteration argument (found, for example, in~\cite[Lemma C.6]{AKMBook}). 
\end{proof}

\begin{exercise}[Coarse-grained Caccioppoli with right-hand side]
In the case~$\a = \a^t$ is a symmetric coefficient field, generalize Proposition~\ref{p.coarse.grained.Caccioppoli} by allowing for solutions with a right-hand side by proving the following statement. For every~$s_1,s_2\in [0,2)$ such that~$s_1+s_2 < 2$, there exists~$C(s_1,s_2,d)<\infty$ such that, for every~$f \in L^2(\cu_m)$ and solution~$u\in H^1(\cu_m)$ of the equation
\begin{equation}
\label{e.with.rhs.f}
-\nabla \cdot \a\nabla u = f 
\quad \mbox{in} \ 
\cu_m
\,,
\end{equation}
we have the estimate
\begin{align}
\label{e.coarse.grained.Caccioppoli.with.f}
\fint_{\cu_{m-1}} 
\nabla u \cdot \a\nabla u
&
\leq 
C \overline{\Lambda}_{s_1}(\cu_m) 
\biggl( \frac{\overline{\Lambda}_{s_1}(\cu_m)}{\overline{\lambda}_{s_2}(\cu_m)} \biggr)^{\frac{s_1}{2-s_1-s_2}}
3^{-2m}  \fint_{\cu_m}  | u-(u)_{\cu_m}|^2 
\notag \\ & \qquad  \
+
C
\overline{\lambda}^{-1}_{s_2} (\cu_m)
\biggl( \frac{\overline{\Lambda}_{s_1}(\cu_m)}{\overline{\lambda}_{s_2}(\cu_m)} \biggr)^{\frac{2-s_2}{2-s_1-s_2}}
3^{2m}  \fint_{\cu_m} f^2 
\,.
\end{align}
What follows is a sketch of the proof of~\eqref{e.coarse.grained.Caccioppoli.with.f}. 
Split~$u=\tilde{u} + v + w$ where~$v\in H^1_0(\cu_m)$ is the solution of the Dirichlet problem 
\begin{equation*}
\left\{
\begin{aligned}
& -\nabla \cdot \a\nabla v = (f)_{\cu_m}  & \mbox{in} & \ \cu_m\,, 
\\ 
& 
v = 0 & \mbox{on} & \ \partial \cu_m\,,
\end{aligned}
\right.
\end{equation*}
and~$w\in H^1(\cu_m)$ is the solution of the Neumann problem 
\begin{equation*}
\left\{
\begin{aligned}
& -\nabla \cdot \a\nabla w = f - (f)_{\cu_m}  & \mbox{in} & \ \cu_m\,, 
\\ 
& 
\mathbf{n} \cdot \a\nabla w = 0 & \mbox{on} & \ \partial \cu_m\,.
\end{aligned}
\right.
\end{equation*}
It follows that~$\tilde{u}:= u -v - w$ is a solution of the equation with zero right-hand side, that is~$\tilde{u} \in \A(\cu_m)$. We can therefore apply the statement of Proposition~\ref{p.coarse.grained.Caccioppoli} to~$\tilde{u}$. This gives us that 
\begin{equation*}
\fint_{\cu_{m-1}} \nabla \tilde{u} \cdot \a\nabla\tilde{u} 
\leq 
C \overline{\Lambda}_{s_1}(\cu_m) 
\biggl( \frac{\overline{\Lambda}_{s_1}(\cu_m)}{\overline{\lambda}_{s_2}(\cu_m)} \biggr)^{\frac{s_1}{2-s_1-s_2}}
3^{-2m} \bigl\| \tilde{u}-(\tilde{u} )_{\cu_m} \bigr\|_{\underline{L}^2(\cu_m)}^2
\,.
\end{equation*}
Estimates on~$v$ and~$w$ can be obtained straightforwardly by testing their equations with~$v$ and~$w$, respectively, and then applying the Poincar\'e inequality in Proposition~\ref{p.coarse.grained.Poincare}---note that this is where we use the symmetry assumption. We obtain
\begin{equation*}
\fint_{\cu_m} \nabla v \cdot \a\nabla v \leq 
3^{2m} 
\overline{\lambda}^{-1}_{s} (\cu_m) 
\bigl| (f)_{\cu_m} \bigr|^2
\quad \mbox{and} \quad
\| v \|_{\underline{L}^2(\cu_m)}^2
\leq 
3^{4m} 
\overline{\lambda}^{-2}_{s} (\cu_m) 
\bigl| (f)_{\cu_m} \bigr|^2
\end{equation*}
and, for~$w$, 
\begin{equation*}
\fint_{\cu_m} \nabla w \cdot \a\nabla w \leq 
3^{2m} 
\overline{\lambda}^{-1}_{s} (\cu_m) 
\bigl\| f \bigr\|_{\underline{L}^2(\cu_m)}^2
\quad \mbox{and} \quad
\| w \|_{\underline{L}^2(\cu_m)}^2
\leq 
3^{4m} 
\overline{\lambda}^{-2}_{s} (\cu_m) 
\bigl\| f \bigr\|_{\underline{L}^2(\cu_m)}^2
\,.
\end{equation*}
Combining the previous three displays and using the triangle inequality in the form
\begin{align*}
\bigl\| \tilde{u}-(\tilde{u} )_{\cu_m} \bigr\|_{\underline{L}^2(\cu_m)}
&
\leq 
\bigl\| u-(u )_{\cu_m} \bigr\|_{\underline{L}^2(\cu_m)}
+
\| v \|_{\underline{L}^2(\cu_m)} + \| w \|_{\underline{L}^2(\cu_m)}
\notag \\ & 
\leq
\bigl\| u-(u )_{\cu_m} \bigr\|_{\underline{L}^2(\cu_m)}
+
3^{2m} \overline{\lambda}^{-1}_{s} (\cu_m) 
\bigl\| f \bigr\|_{\underline{L}^2(\cu_m)}\,,
\end{align*}
we obtain the desired inequality~\eqref{e.coarse.grained.Caccioppoli.with.f}. 
\end{exercise}

\subsection{Equations with a divergence-form forcing term}
\label{ss.rhs} 

In this section, we explain how to extend the theory described in this chapter to cover equations having the form 
\begin{equation}
\label{e.withRHS}
-\nabla \cdot \a \nabla u = \nabla\cdot \f,
\end{equation}
where, in addition to~$\a$, the vector-valued function~$\f$ is another stationary random field oscillating on the same length scale as~$\a$. 
This kind of equation arises very often in applications. For instance, the equations for the \emph{higher-order} correctors have the form of~\eqref{e.withRHS}, as do certain linearized equations one encounters when studying homogenization of nonlinear equations. 
The generalization to such equations does not increase the difficulty or level of complexity of the analysis; we just need some minor variations of what we have already presented. We include these details here for completeness. The results in this section are only used in Section~\ref{ss.rhs.optimal}, and may be skipped on a first reading. 

\subsubsection{Modifications to the formalism}
\label{sss.formalisms.for.RHS}

We begin by briefly discussing how to modify our assumptions, including the mixing assumptions, to allow for a random forcing vector field~$\f$. Very often in applications,~$\f$ will be a function of the coefficient field~$\a(\cdot)$, in other words, an element of~$V^2(\P)$ where~$(\Omega,\F)$ is the specific measurable space described in Section~\ref{ss.probspace}. 
However, in order to consider a more general situation (in which~$\f$ and~$\a$ are possibly uncorrelated, for example) we enlarge the probability space~$\Omega$ by redefining it to consist of pairs~$(\a,\f)$ with~$\a$ satisfying~\eqref{e.ellipticity} and~$\f\in L^2_{\mathrm{loc}}(\Rd)$. The family of~$\sigma$--algebras are also enlarged, so that~$\F(U)$ is generated by the family of random variables 
\begin{equation*} \label{}
(\a,\f) \mapsto \int_U ( \a_{ij} (x) + \f_k(x))  \varphi(x)\,dx, 
\quad 
i,j,k \in\{1,\ldots,d\}, 
\quad 
\varphi\in C^\infty_{\mathrm{c}}(\Rd). 
\end{equation*}
As expected, the translation semigroup~$\{ T_y \}$ defined in~\eqref{e.rvF} now acts on the pair~$(\a,\f)$, and the probability measure~$\P$ we consider will be assumed to be~$\Zd$--stationary with respect to this translation group. 

\smallskip

As for the necessary modifications to the~$\CFS$ condition, we alter the definition of Malliavin derivative, as in Section~\ref{ss.CFS.gen}, with the norm given by~\eqref{e.triple.norm.sec6}. 
In particular, we are slightly more lenient with variations in the vector field~$\f$ by allowing for~$L^2$-type variations instead of~$L^\infty$ as in the coefficient field~$\a$. When we say that~$\P$ satisfies~$\CFS(\beta,\Psi)$ in this section, we mean that it satisfies Definition~\ref{d.CFS.general} with the Malliavin derivative~$|\partial_{(\a,\f)(U)}|$ defined in~\eqref{e.Mall.Maul.yes} in place of~$|\partial_{\a(U)}|$ and~$\vertiii\cdot\vertiii$ given in~\eqref{e.triple.norm.sec6}.

\smallskip

We need another assumption on~$\P$, in addition to~$\Zd$--stationary and~$\CFS(\beta,\Psi)$. We have so far assumed no \emph{boundedness} of~$\f$: just that it belong to~$L^2_{\mathrm{loc}}(\Rd)$. 
A very natural condition to take would be that~$\| \f \|_{L^2(\cu_0)}$ is uniformly bounded by a deterministic constant, which would be the analogous assumption to the very strict uniform ellipticity condition we require the diffusion matrix~$\a(\cdot)$ to satisfy. 
With applications in mind, and because it does not cost anything in terms of the proof complexity (right-hand sides being quite a bit easier to handle than diffusion matrices), 
we will assume a weaker condition that quantifies the local~$L^2$ integrability of~$\f$, namely that there exists~$\mathsf{K}\in [0,\infty)$ such that, for every~$m\in\N$, 
\begin{equation}
\label{e.f.dumbbound}
\| \f \|_{\underline{L}^2(\cu_m)}^2 
\leq 
1 +  \O_\Psi\bigl(\mathsf{K}  3^{-\frac d2 (1-\beta) m} \bigr)
\,.
\end{equation}
If~$\f$ is a Gaussian random field, then~\eqref{e.f.dumbbound} is satisfied with~$\Psi=\Gamma_2$ with~$\E [ \| \f \|_{\underline{L}^2(\cu_m)}^2 ] \leq 1$, where the exponent~$\beta$ and constant~$\mathsf{K}$ are related to the decorrelation of the Gaussian field (as in Section~\ref{ss.GRF}).

\smallskip

Throughout this section, we assume that
\begin{equation} 
\label{e.f.sec.ass}
\P \mbox{ is~$\Zd$-stationary and satisfies~$\CFS(\beta,\Psi)$ and~\eqref{e.f.dumbbound}}
\end{equation}
and we add the parameter~$\mathsf{K}$ to the list of~$\data$ in~\eqref{e.data.def}, now denoting
\begin{equation} 
\label{e.data.def.f}
\data_{\f} := (\mathsf{K},\beta,d,\lambda,\Lambda,\CFS,\Psi)
\,.
\end{equation}

We turn next to the definitions of the counterparts of the quantity~$J$ defined in~\eqref{e.variational.J.nonsymm}. If~$U\subseteq\Rd$ is a bounded Lipschitz domain and~$\f \in L^2(U)$, we define, for every~$p,q\in\Rd$, 
\begin{equation}
\label{e.variational.J.nonsymm.withf}
J_{\f} (U,p,q) 
:= 
\max_{v\in \mathcal{A}_{\bfzero} (U)} 
\fint_U \Bigl( -\frac12 \nabla v\cdot \s\nabla v -p\cdot \a\nabla v   + (q-\f)\cdot \nabla v \Bigr).
\end{equation}
Note very carefully that the maximum is over the set~$\A_{\bfzero}(U) = \A(U)$, which is the set of solutions of the equation with \emph{zero} right-side, as defined in~\eqref{e.def.A(U)}. 
As such, the zero function belongs to the admissible set and so~$J_{\f}(U,p,q) \geq 0$. We let~$v_{\f}(\cdot,U,p,q)$ denote the maximizer of the optimization problem on the right side of~\eqref{e.variational.J.nonsymm.withf}. We also define the adjoint quantity 
\begin{equation}
\label{e.variational.J.nonsymm.dual.withf}
J^*_{\f}(U,p,q) 
:= 
\max_{v^*\in \mathcal{A}_{\bfzero}^*(U)} 
\fint_U \Bigl( -\frac12 \nabla v^*\cdot \s\nabla v^* -p\cdot \a^t\nabla v^* + (q-\f)\cdot \nabla v^*   \Bigr) 
\end{equation}
with the maximizer~$v_{\f}^*(\cdot,U,p,q) \in \mathcal{A}_{\bfzero}^*(U)$,
and set, for~$p,q,p^*,q^* \in \Rd$, 
\begin{equation}
\label{e.bfJ.withf}
\bfJ_{\f}
\biggl(U, \begin{pmatrix} p  \\ q \end{pmatrix}, \begin{pmatrix} q^* \\ p^* \end{pmatrix} \biggr)
:=
J_{\f}\bigl(U,p-p^*,q^*-q\bigr)
+ 
J^*_{\f}\bigl(U,p^*+p,q^*+q\bigr)
\,.
\end{equation}
Next, by artificially lifting the vector field~$\f$ to a~$\R^{2d}$-valued random field by defining
\begin{equation*}
\bfF:=
\begin{pmatrix} \f  \\  0 \end{pmatrix}
\,,
\end{equation*}
we compute, similarly to~\eqref{e.iden.PAP}, for every~$v,v^*\in H^1(U)$ and~$p,q,p^*,q^*\in\Rd$, 
\begin{align}
\label{e.iden.PAP.again}
& 
{-}\frac12  X 
\!\cdot\!
\bfA 
X
\! - \!
\begin{pmatrix} p \\ q \end{pmatrix} 
\!\cdot\!
\bfA
X
\! + \!
\biggl( \begin{pmatrix} q^* \\ p^* \end{pmatrix} - \mathbf{F} \biggr)
\!\cdot\!
X
\notag \\ & \qquad
=
- \frac12 \nabla v \cdot \a\nabla v 
- (p-p^*)\cdot \a\nabla v
+(q^*-q - \f )\cdot \nabla v
\notag \\ & \qquad \qquad
- \frac12 \nabla v^* \cdot \a^t \nabla v^*
-(p^*+p)\cdot\a^t\nabla v^* 
+(q^*+q- \f)\cdot \nabla v^*
\end{align}
with
\begin{equation*} 
X:= \begin{pmatrix} \nabla v{+}\nabla v^*  
\\ \a\nabla v {-} \a^t\nabla v^*  \end{pmatrix}  \,.
\end{equation*}
Therefore,  taking supremum over~$(v,v^*) \in \mathcal{A}(U) \times \mathcal{A}^*(U)$, 
we obtain, for every~$P,Q \in \R^{2d}$, 
\begin{align}
\label{e.bfJ.var.f}
\bfJ_{\f}(U, P,Q )
& 
=
\sup_{X \in \S_{\bfzero}(U)}
\fint_U
\Bigl( -\frac12  X  \! \cdot \! \bfA X - P  \cdot \!\bfA X  + (Q - \bfF)  \cdot  X\Bigr)
\,.
\end{align}

\smallskip

Note that the maximization in~\eqref{e.bfJ.var.f} is with respect to solutions~$\S_{\bfzero}(U) = \S(U)$, 
where~$\S(U)$ refers to the solutions of the homogeneous problem as in~\eqref{e.findS}. We denote by~$X_{\f}(\cdot,U,P,Q)$  the maximizer in~$\S_{\bfzero}(U)$ of the variational problem on the right side of~\eqref{e.bfJ.var.f}. Analogously to~\eqref{e.maximizers.J.to.bfJ}, we have, for every~$p,q,p^*,q^*\in\Rd$, 
\begin{equation}
\label{e.maximizers.J.to.bfJ.f}
X_{\f}
\biggl(\cdot,U, \begin{pmatrix} p  \\ q \end{pmatrix}, \begin{pmatrix} q^* \\ p^* \end{pmatrix} \biggr)
=
\begin{pmatrix} \nabla v_{\f} (\cdot,U,p{-}p^*,q^*{-}q)
+ 
\nabla v_{\f}^*\bigl(\cdot,U,p^*{+}p,q^*{+}q\bigr) \\ 
\a \nabla v_{\f} (\cdot,U,p{-}p^*,q^*{-}q)
- 
\a^t \nabla v_{\f}^*\bigl(\cdot,U,p^*{+}p,q^*{+}q\bigr)
\end{pmatrix}
\,.
\end{equation}

\smallskip

If~$\bfF = 0$, then maximizing problem in~\eqref{e.bfJ.var.f} becomes~\eqref{e.bfJ.var} and
\begin{align*}  
\bfJ_{\bfzero}(U, P,Q ) = \bfJ(U, P,Q )
\qand
X_0(\cdot,U,P,Q) = S(\cdot,U,P,Q) 
\,.
\end{align*}
In particular,~$\bfJ_{\bfzero}(U, P,Q )$ and~$X_0(\cdot,U,P,Q)$ have all the properties from Lemma~\ref{l.J.basicprops.nosymm.moar}.

\smallskip

Like~$\bfJ=\bfJ_{\bfzero}$, the quantity~$\bfJ_{\f}$ is subadditive in the same sense as~$J$ in~\eqref{e.subaddJ.nosymm}. 
Since the zero function belongs to~$\S_{\bfzero}(U)$, we have that, for every~$P,Q\in \R^{2d}$,  
\begin{equation}
\label{e.bfJ.nonneg}
\bfJ_{\f}(U,P,Q) \geq 0\,.
\end{equation}
Computing the first variation of the optimization problem in~\eqref{e.bfJ.var.f} yields the following characterization of~$X_{\f}$:
\begin{equation}
\label{e.firstvar.Jf}
\fint_U 
\Bigl( 
- \bfA X_{\f}(\cdot,U,P,Q) 
-\bfA P
+Q-\bfF
\Bigr)\cdot Y = 0\,, 
\qquad \forall \,  Y \in \S_{\bfzero}(U)\,, \, \forall \, P,Q\in \R^{2d}\,.
\end{equation}
Using the above first variation, we obtain the identity 
\begin{equation} 
\label{e.split.X.f}
X_{\f}(\cdot,U,P,Q)  
= 
X_{\bfzero}(\cdot,U,P,Q) + X_{\f}(\cdot,U,0,0) 
\,.
\end{equation}
We define the coarse-grained vectors~$\bfF(U)$ and~$\bfF_*(U)$ by 
\begin{equation} \label{e.coarseF.def}
\left\{
\begin{aligned}
& \bfF(U) 
:=
-\,
\fint_{U} \bfA X_{\f}(\cdot,U,0,0) \,,
\\ & 
\bfF_*(U) 
:=
-\bfA_*(U) \fint_{U} X_{\f}(\cdot,U,0,0) \,,
\end{aligned}
\right.
\end{equation}
and then see by~\eqref{e.formulas.nosymm.byS} that
\begin{equation} 
\label{e.average.X.f}
\left\{
\begin{aligned}
&\fint_{U} \bfA X_{\f}(\cdot,U,P,Q) 
 =
-\bfF(U) + Q - \bfA(U) P\,, \\
& \fint_{U} X_{\f}(\cdot,U,P,Q) = \bfA_*^{-1}(U) (Q - \bfF_*(U)) - P \,.  
\end{aligned}
\right.
\end{equation}
Using again the first variation~\eqref{e.firstvar.Jf} together with~\eqref{e.split.X.f}, we can rewrite~$\bfJ_{\f}$ as
\begin{equation} 
\label{e.Jf.repre0}
\bfJ_{\f}(U, P,Q ) 
=  
\bfJ_{\bfzero}(U, P,Q )  
+ P\cdot \bfF(U) 
- Q \cdot \bfA_*^{-1}(U) \bfF_*(U)  
+ \bfJ_{\f}(U, 0,0) 
\,.
\end{equation}
We see that the quadratic part of~$\bfJ_{\f}$ in~$(P,Q)$ is the same as~$\bfJ_{\bfzero}$, but the formula has extra linear parts determined by~$\bfF$ and~$\bfF_*$, and a constant part given by~$\bfJ_{\f}(U,0,0)$. We can use this to our advantage since we have already analyzed~$\bfJ_{\bfzero}$ in the previous section.

\subsubsection{The main result}

We present a quantitative estimates for \emph{finite-volume, zero-slope correctors}. 
We remark that this implies a statement like the one of Theorem~\ref{t.quant.DP} for the equation with a right-hand side as a corollary. We do not have to repeat the two-scale expansion again to get this corollary: we subtract the finite-volume zero-slope corrector from the solution (since the zero-slope corrector is small, this is almost like subtracting zero), which removes the right-hand side. We can then quote Theorem~\ref{t.quant.DP} (or, more generally, Theorem~\ref{t.quant.DP.nosymm}) to handle the difference.

\begin{theorem}[Finite-volume zero-slope corrector]
\label{t.zeroslope}
\hspace{-4pt}
Assume that~$\P$ satisfies~\eqref{e.f.sec.ass}. 
\hspace{-6pt}
There exists a constant vector~$\overline{\f}\in\R^d$, a constant~$\overline{\mu} \in [0,\infty)$, an exponent~$\alpha(\beta,d,\lambda,\Lambda) \in (0,\tfrac12]$ and, for every~$\delta>0$, a constant~$C(\delta,\datareff)<\infty$ and a random variable~$\X$ satisfying 
\begin{equation}
\label{e.mmmbound}
\X^{\frac d2 (1-\beta)}
= \O_\Psi(C)
\end{equation}
such that, for every~$m\in\N$ with~$3^m \geq \X$, there exists~$\psi_{\f,m} \in H^1(\cu_m)$ satisfying the equation
\begin{equation}
-\nabla \cdot \a\nabla \psi_{\f,m} = \nabla \cdot \f \quad \mbox{in} \ \cu_m\,,
\end{equation}
together with the estimates
\begin{equation}
\label{e.zeroslope.sublin}
3^{-m}\| \nabla \psi_{\f,m} \|_{\Hminusul(\cu_m)} 
+
3^{-m}\| \a \nabla \psi_{\f,m} +\f - \overline{\f} \|_{\Hminusul(\cu_m)} 
\leq
\delta ( \X 3^{-m} )^{\alpha}
\end{equation}
and
\begin{equation} 
\label{e.zeroslope.energy}
\biggl| \fint_{\cu_m} \s \nabla\psi_{\f,m} \cdot \nabla \psi_{\f,m}  - \overline{\mu} \biggr| 
+
\biggl| \fint_{\cu_m} \f \cdot \nabla \psi_{\f,m}  + \overline{\mu} \biggr| 
\leq
\delta ( \X 3^{-m} )^{\alpha} 
 \,.
\end{equation}
\end{theorem}

The proof of Theorem~\ref{t.zeroslope} is the focus of the rest of this section.

\subsubsection{The coarse-grained vector field}

We proceed by generalizing the coarse-grained quantities defined in Section~\ref{ss.doublevar}. We define the inhomogeneous variant~$\S_{\f} (U)$ of~$\S(U)$ (defined in~\eqref{e.bfS}) by
\begin{align}
\label{e.Sf.this}
\S_{\f} (U) 
:=
\Bigl\{
X \in \Lpot(U) \times \Lsol(U) 
 : 
\int_U 
Y \cdot 
\bigl( \bfA X + 2 \bfF \bigr)
=0, \, 
\forall Y \in
\Lpoto(U) \times \Lsolo(U)
\Bigr\}
\,.
\end{align}
Note that this space is affine but not linear. 

In the following lemma, we generalize the representation of~$\S(U)$ in~\eqref{e.findS} to the case of~$\S_{\f}(U)$ with nonzero~$\f$.
To that end, we introduce the following notation for solutions of the equation and its adjoint with right-side~$\nabla\cdot \f$, which can be compared with~\eqref{e.def.AU} and~\eqref{e.def.AsU}:
\begin{equation*}
\mathcal{A}_{\f} (U) := 
\bigl\{ u \in H^1(U)\,:\, -\nabla\cdot \a\nabla u = \nabla \cdot \f \bigr\}\,, 
 \ 
\mathcal{A}_{\f}^* (U) := 
\bigl\{ u \in H^1(U)\,:\, -\nabla\cdot \a^t\nabla u = \nabla \cdot \f \bigr\}\,. 
\end{equation*}
This is a straightforward adaption of the proof of~\eqref{e.findS}, but the algebra is a little tricky, so for the convenience of the reader we present all the details.

\begin{lemma}[Variational representation of~$\mathcal{S}_{\f}(U)$]
\label{l.f.S.findme}
For every bounded domain~$U\subseteq\Rd$,
\begin{equation}
\label{e.f.S.findme}
\S_{\f}(U)
=
\left\{ (\nabla v+\nabla v^*, \a\nabla v - \a^t\nabla v^*) \,:\, v\in \A_{\f}(U), \ v^*\in \A_{\f}^*(U)
\right\}
\,.
\end{equation}
\end{lemma}
\begin{proof}
An element~$(\nabla u,\mathbf{h}) \in \Lpot(U)\times\Lsol(U)$ belongs to~$\S_{\f}(U)$ if and only if 
\begin{multline}
\label{e.f.thecondition}
\int_U \bigl(
\nabla w \cdot \s \nabla u + (\mathbf{g} -\k\nabla w) 
\cdot \s^{-1} (\mathbf{h} - \k \nabla u )
+  2 \nabla w \cdot \f )
\bigr) = 0\,, 
\\
\forall (\nabla w,\mathbf{g}) \in \Lpoto(U) \times \Lsolo(U)\,.
\end{multline}
Taking~$w=0$ above yields that~$\s^{-1}(\mathbf{h}-\k \nabla u)$ is orthogonal to~$\Lsolo(U)$, which further implies that there exists~$\nabla \tilde u \in \Lpot(U)$ such that~$\mathbf{h} = \s \nabla \tilde u + \k \nabla u$. Inserting this into~\eqref{e.f.thecondition} and taking~$\mathbf{g}= 0$ there yields, for every~$\nabla w\in \Lpoto(U)$,  
\begin{equation*}
0 
=
\int_U 
\nabla w \cdot ( \s \nabla u +  \k\nabla \tilde u + 2\f )\,.
\end{equation*}
In conclusion,~\eqref{e.f.thecondition} implies that there exists~$\nabla \tilde u \in \Lpot(U)$ such that
\begin{align*}  
\mathbf{h}  = \k\nabla u + \s \nabla \tilde u \in \Lsol(U)
\qand 
\s \nabla u +  \k\nabla \tilde u  + 2\f \in \Lsol(U)
\,.
\end{align*}
Conversely, it is easy to check that if~$(\nabla u, \mathbf{h})$ is defined by the above two identities for some~$\nabla \tilde u \in \Lpot(U)$, then also~\eqref{e.f.thecondition} is valid. Writing~$u = v + v^*$ and~$\tilde u =  v - v^*$, 
we get
\begin{equation*}
\mathbf{h} = \a\nabla v - \a^t \nabla v^*  \in \Lsol(U) 
\qand
\s \nabla u +  \k\nabla \tilde u  + 2\f 
= 
\left\{ 
\begin{aligned}
2\a \nabla v + 2\f - \mathbf{h} \in \Lsol(U) \, , \\
2\a^t \nabla v^* + 2\f + \mathbf{h} \in \Lsol(U)  \,,
\end{aligned}
\right.
\end{equation*}
and thus~$v  \in\A_{\f}(U)$ and~$v^* \in\A_{\f}^*(U)$. This gives us the representation~\eqref{e.f.S.findme}. 
\end{proof}

Notice that we have the following representations for~$v$ and~$v^*$ in terms of the element~$Z$ of~$\S_{\f}(U)$ in~\eqref{e.f.S.findme} they represent: 
\begin{equation} \label{e.f.splittingformula.for.Z}
Z = 
\begin{pmatrix} \nabla v + \nabla v^* \\ \a \nabla v - \a^t \nabla v^* \end{pmatrix}  \in \S_{\f}(U) 
\implies 
\left\{
\begin{aligned}
& \begin{pmatrix} \nabla v \\ \a \nabla v + \f \end{pmatrix} = \frac12 \bigl(Z + \rota ( \bfA Z + 2 \bfF) \bigr) 
 \,, 
\\ &   
\begin{pmatrix} \nabla v^* \\ - \a^t \nabla v^* - \f \end{pmatrix} = \frac12 \bigl( Z - \rota ( \bfA Z + 2 \bfF) \bigr) 
\,.
\end{aligned}
\right.
\end{equation}
Here, and in what follows, we denote
\begin{equation} 
\label{e.rota}
\rota:= \begin{pmatrix} 0 & \Id \\ \Id & 0 \end{pmatrix} \,.
\end{equation}

\smallskip

It will be useful to write a more general version of~\eqref{e.Jf.repre0}, with the ``constant term'' allowed to be~$\bfJ_{\f}(U,P_0,Q_0)$ for any~$(P_0,Q_0)\in\R^{4d}$. By subtracting~\eqref{e.Jf.repre0} and~\eqref{e.Jf.repre0} with~$(P,Q)=(P_0,Q_0)$, we get
\begin{align}
\label{e.Jf.repre.gen}
\bfJ_{\f}(U,P,Q) 
&
=
\bfJ_{\bfzero} (U,P,Q) 
+ (P - P_0) \cdot \bfF(U)
- (Q - Q_0) \cdot \bfA_*^{-1}(U) \bfF_*(U) 
\notag \\ & \qquad 
+ \bfJ_{\f}(U,P_0,Q_0) - \bfJ_{\bfzero} (U,P_0,Q_0)  \,.
\end{align}
Expanding around~$(P_0,Q_0)=(0,\bfF_*(U))$, we have, after some algebraic manipulations,  
\begin{align} 
\label{e.Jf.repre0zoop}
\bfJ_{\f}(U, P,Q ) 
&
=
\frac 12 P \cdot \bfA(U) P 
+
\frac12 (Q-\bfF_*(U)  ) \cdot \bfA_*^{-1}(U) (Q-\bfF_*(U)  )
-
P \cdot (Q-\bfF_*(U))
\notag \\ & \qquad 
+ 
\bfJ_{\f}(U,0,\bfF_*(U) )
+ P \cdot (\bfF(U)-\bfF_*(U))
\\ & \label{e.Jf.repre0zoopie}
=
\bfJ_{\bfzero}(U,P,Q-\bfF_*(U)) + 
\bfJ_{\f}(U,0,\bfF_*(U) )
+ P \cdot (\bfF(U)-\bfF_*(U))
\,.
\end{align}

\subsubsection{Splitting formula for~$\bfJ_{\f}$}

In order to proceed, we need a version of the identity~\eqref{e.Jsplitting.nonsymm} and another, dual version of it. This requires us to define two additional subadditive quantities:  
for every~$P ,Q \in \R^{2d}$,
\begin{equation}
\label{e.dual.smitz}
\left\{
\begin{aligned}
& \mu_{\f}(U,P)
: = \inf \biggl\{ 
\fint_{U} \Bigl( \frac12   
Z\cdot \bfA Z + 2 \bfF\cdot Z \Bigr)
\, : \, Z \in P+ \bigl( \Lpoto(U) \times  \Lsolo(U) \bigr) \biggr\}
\,,
\\ & 
\mu_{\f}^*(U,Q)
: = \sup \biggl\{ 
\fint_{U} \Bigl( - \frac12 Z \cdot \bfA Z + (Q-2\bfF) \cdot Z   \Bigr)
\, : \, Z \in \Lpot(U)\times \Lsol(U) \biggr\} 
\,.
\end{aligned}
\right.
\end{equation}
These quantities are subadditive. 
We denote their optimizers by~$Z_{\f}(\cdot,U,P)$ and~$Z_{\f}^*(\cdot,U,Q)$, respectively. Comparing these to the variational problems in~\eqref{e.J.P0.Dir} and~\eqref{e.J.0Q.Neu}, we have 
\begin{equation} 
\label{e.Z.zero.split}
Z_{\mathbf{0}}(\cdot,U,P) = - S(\cdot,U,P,0)
\qand
Z_{\mathbf{0}}^*(\cdot,U,Q) = S(\cdot,U,0,Q)
 \,.
\end{equation}
The optimizers are further characterized by the first variation identities:
\begin{equation}
\label{e.var.dual.smitz}
\fint_U Y \cdot \bigl( \bfA Z_{\f}(\cdot,U,P) +2 \bfF  \bigr) = 0 
\,,
\qquad \forall Y \in \Lpoto(U) \times  \Lsolo(U)  
\end{equation}
and
\begin{equation}
\label{e.var.dual.smitzenn}
\fint_U Y \cdot \bigl( \bfA Z^*_{\f}(\cdot,U,Q) -Q+2\bfF  \bigr) = 0
\,,
\qquad \forall Y \in \Lpot(U) \times  \Lsol(U)  
\,.
\end{equation}
From the first variation, we deduce the energy estimates
\begin{equation} 
\label{e.energy.estimate.Z}
\bigl\|\bfA^{\nicefrac12}Z_{\f}(\cdot,U,P) 
\bigr\|_{\underline{L}^2(U)} \leq 10 \bigl\|\s^{-\nicefrac12}\f 
\bigr\|_{\underline{L}^2(U)}  + 10 \bigl\|\bfA^{\nicefrac12}P
\bigr\|_{\underline{L}^2(U)}  
\end{equation}
and
\begin{equation} 
\label{e.energy.estimate.Zstar}
\bigl\|\bfA^{\nicefrac12}Z_{\f}^*(\cdot,U,Q) 
\bigr\|_{\underline{L}^2(U)} \leq 10 \bigl\|\s^{-\nicefrac12}\f 
\bigr\|_{\underline{L}^2(U)}  + 10 \bigl\|\bfA^{-\nicefrac12}Q
\bigr\|_{\underline{L}^2(U)}  
\,.
\end{equation}
We also have the identities
\begin{equation}
\label{e.trivial.blitzen}
P = \fint_U Z_{\f}(\cdot,U,P)
\quad \mbox{and} \quad 
Q = \fint_U \bigl( \bfA Z^*_{\f}(\cdot,U,Q)+2\bfF\bigr)
\,.
\end{equation}
To prove the first of these, we use that~$Z_{\f}(\cdot,U,P) \in  P+ \bigl( \Lpoto(U) \times  \Lsolo(U) \bigr)$ and use Stokes' theorem. For the second, we test~\eqref{e.var.dual.smitzenn} with~$Y$ a constant vector in~$\R^{2d}$.

\smallskip

We next prove of the splitting formula~\eqref{e.Jsplitting.nonsymm} for~$\bfJ_{\f}$, with~$\mu_{\f}(U,P)$ and~$\mu^*_{\f}(U,Q)$ in place of the quadratic terms in~\eqref{e.Jsplitting.nonsymm}.

\begin{lemma} \label{l.f.varyou.varme}
For every~$P ,Q \in \R^{2d}$, 
\begin{equation}  
\label{e.Jf.repre1}
\bfJ_{\f}(U, P,Q ) = \mu_{\f}(U,P) + \mu_{\f}^*(U,Q) - P \cdot Q \,,
\end{equation}
and the maximizer of~$\bfJ_{\f}(U,P,Q)$ satisfies
\begin{align}  \label{e.Jf.repre2}
X_{\f}(\cdot,U,P,Q) = X_{\bfzero}(\cdot,U,P,Q) + X_{\f}(\cdot,U,0,0) = Z_{\f}^*(\cdot,U,Q) -  Z_{\f}(\cdot,U,P) 
\,.
\end{align}
\end{lemma}

\begin{proof}
By the definition~\eqref{e.Sf.this} of~$\S_{\f}(U)$ together with~\eqref{e.var.dual.smitz} and~\eqref{e.var.dual.smitzenn}, we deduce that  both~$Z_{\f}(\cdot,U,P)$ and~$Z_{\f}^*(\cdot,U,Q)$ belong to~$\S_{\f}(U)$. Thus, by Lemma~\ref{l.f.S.findme},~\eqref{e.var.dual.smitz} and~\eqref{e.var.dual.smitzenn},
we see that
\begin{align}
\label{e.decomp.splitten}
Z_{\f}(\cdot,U,P) = Z_{\bfzero}(\cdot,U,P) {+} Z_{\f}(\cdot,U,0)
\,, \quad
Z_{\f}^*(\cdot,U,Q) = Z_{\bfzero}^*(\cdot,U,Q) {+} Z_{\f}^*(\cdot,U,0)
\,.
\end{align}
Using~\eqref{e.decomp.splitten} we obtain
\begin{equation}
\label{e.crossterm.smudge}
\fint_U
(\bfA Z_{\f}(\cdot,U,0) + 2\bfF) \cdot (Z_{\bfzero}(\cdot,U,P) - P ) = 0
\qand
\fint_U
Z_{\f}(\cdot,U,0) \cdot \bfA Z_{\bfzero}(\cdot,U,P) = 0\,,
\end{equation} 
and therefore, by applying~\eqref{e.decomp.splitten} and~\eqref{e.var.dual.smitz} once more, 
\begin{align*}  
\mu_{\f}(U,P) 
&
= 
\fint_U \Bigl( \frac12 Z_{\f}(\cdot,U,P) \cdot\bfA Z_{\f}(\cdot,U,P) 
+ 2 \bfF \cdot Z_{\f}(\cdot,U,P)  \Bigr)
\\ & 
=
\fint_U 
\frac 12  Z_{\bfzero}(\cdot,U,P) \cdot \bfA  Z_{\bfzero}(\cdot,U,P)
+
\fint_U \Bigl( \frac12 Z_{\f}(\cdot,U,0) \cdot \bfA Z_{\f}(\cdot,U,0)  + 2\bfF \cdot Z_{\f}(\cdot,U,0) \Bigr)
\\ & \qquad 
+ \fint_U 2\bfF \cdot Z_{\bfzero}(\cdot,U,P) 
\\ & 
=
\bfJ_{\bfzero}(U,P,0) + \mu_{\f}(U,0)
+ P \cdot   \fint_U\bigl(\bfA  Z_{\f}(\cdot,U,0) +2\bfF\bigr)
\,.
\end{align*}
Similarly, applying now~\eqref{e.decomp.splitten} and~\eqref{e.var.dual.smitzenn}, we get
\begin{equation*}  
\mu_{\f}^*(U,Q) 
= 
\bfJ_{\bfzero}(U,0,Q) + \mu_{\f}^*(U,0) + Q \cdot \fint_{U} Z_{\f}^*(\cdot,U,0)  
\,.
\end{equation*}
For every~$Z^* \in \S_{\f}(U)$, we observe that~$X :=  Z^* - Z_{\f}(\cdot,U,P)$ belongs to~$\S_{\bfzero}(U)$ and satisfies 
\begin{align*} 
&
\mu_{\f}(U,P) + \fint_{U} \Bigl( - \frac12 Z^* {\cdot}\, \bfA Z^*  + (Q-2\bfF) {\cdot}\, Z^* \Bigr) - P\cdot Q
\notag \\ & \qquad\quad
=
\fint_{U} \Bigl( - \frac12 X {\cdot}\, \bfA X - P{\cdot}\, \bfA X + (Q-2 \bfF){\cdot}\, X\Bigr)\,.
\end{align*}
Consequently, both~\eqref{e.Jf.repre1} and~\eqref{e.Jf.repre2} follow since~$Z^* = Z_{\f}^*(\cdot,U,Q)\in \S_{\f}(U)$ maximizes the left-hand side and~$X_{\f}(\cdot,U,P,Q)$ maximizes the right-hand side in~$\S_{\bfzero}(U)$, in which also the difference~$Z_{\f}^*(\cdot,U,Q) - Z_{\f}(\cdot,U,P)$ belongs to. The proof is complete. 
\end{proof}

In view of~\eqref{e.coarseF.def},~\eqref{e.trivial.blitzen} and~\eqref{e.Jf.repre2}, we have the following representations for~$\bfF(U)$ and~$\bfF_*(U)$ in terms of the optimizers for~$\mu_{\f}(U,0)$ and~$\mu_{\f}^*(U,0)$: 
\begin{equation}
\label{e.bfF.by.Zf}
\bfF(U) 
=
-\,
\fint_{U}  \bfA \bigl( Z_{\f}^*(\cdot,U,0) - Z_{\f}(\cdot,U,0) \bigr) 
= 
\fint_{U} \bigl( \bfA  Z_{\f}(\cdot,U,0) +  2\bfF\bigr) 
\end{equation}
and
\begin{equation}
\label{e.bfF.by.Zf.star}
\bfF_*(U) 
=
-\,
\bfA_{*}(U) \fint_{U} ( Z_{\f}^*(\cdot,U,0) - Z_{\f}(\cdot,U,0)) = - \bfA_{*}(U)  \fint_{U} Z_{\f}^*(\cdot,U,0) 
\,.
\end{equation}
As a byproduct of the previous proof, we may rewrite the energies as follows:
\begin{equation} 
\label{e.mu.f.mustar.f.formula}
\left\{
\begin{aligned}
& \mu_{\f}(U,P) = 
\frac12 P \cdot \bfA(U) \cdot P
+ P \cdot   \bfF(U)  
 + \mu_{\f}(U,0)
\,,  
\\
& \mu_{\f}^*(U,Q) 
= 
\frac12 Q \cdot \bfA_*^{-1} (U) \cdot Q - Q \cdot \bfA_*^{-1}(U) \bfF_*(U) 
+ \mu_{\f}^*(U,0) 
\,.
\end{aligned}
\right.
\end{equation}

\smallskip

The representation of~$\bfJ_{\f}$ in~\eqref{e.Jf.repre1} and of the coarse-grained vector field in~\eqref{e.bfF.by.Zf} and~\eqref{e.bfF.by.Zf.star} will allow us to estimate the Malliavin derivatives of these quantities. By the representation~\eqref{e.Jf.repre0}, it is enough to obtain estimates for~$\bfF(U)$,~$\bfF_*(U)$,~$\mu_{\f}(U,0)$ and~$\mu_{\f}^*(U,0)$. 

\begin{lemma}
\label{l.Malliavin.f.base}
There exists~$C(d,\lambda,\Lambda)<\infty$ such that 
\begin{equation}
\label{e.Malliavin.bfF}
\bigl| \partial_{(\a,\f)(U)} \bfF(U) \bigr| 
+
\bigl| \partial_{(\a,\f)(U)} \bfF_*(U) \bigr| 
\leq 
C\bigl (1+ \| \f \|_{\underline{L}^2(U)} \bigr) 
\end{equation}
and 
\begin{equation}
\label{e.Malliavin.F.muf}
\bigl| \partial_{(\a,\f)(U)} \mu_{\f}(U,0) \bigr| 
+
\bigl| \partial_{(\a,\f)(U)} \mu_{\f}^*(U,0) \bigr| 
\leq 
C \| \f \|_{\underline{L}^2(U)} 
\bigl (1+ \| \f \|_{\underline{L}^2(U)} \bigr) \,.
\end{equation}
\end{lemma}
\begin{proof}
We fix a small parameter~$t\in (0,1]$, a bounded Lipschitz domain~$U\subseteq \Rd$ and~$(\a,\f)\in\Omega$ and we consider a variation~$(\tilde{\a},\tilde{\f})$ satisfying
\begin{equation*}
\| \a - \tilde{\a} \|_{L^\infty(U)} \leq t \quad \mbox{and} \quad 
\| \f - \tilde{\f} \|_{\underline{L}^2(U)} \leq t\,.
\end{equation*}
We denote all of the quantities associated to~$(\tilde{\a},\tilde{\f})$ with tildes.
Using~\eqref{e.var.dual.smitz} with~$P=0$ and~$Y = Z_{\f}(\cdot,U,0) - \tilde{Z}_{\f}(\cdot,U,0)$, we get 
\begin{equation*}
\left\{
\begin{aligned}
&  \fint_U \bigl( Z_{\f}(\cdot,U,0) - \tilde{Z}_{\f}(\cdot,U,0) \bigr) \cdot \bigl( \bfA Z_{\f}(\cdot,U,0) +2\bfF  \bigr) = 0 \,, 
\quad \mbox{and} \\ &
 \fint_U \bigl( Z_{\f}(\cdot,U,0) - \tilde{Z}_{\f}(\cdot,U,0) \bigr) \cdot \bigl( \tilde\bfA \tilde{Z}_{\f}(\cdot,U,0) + 2\tilde\bfF  \bigr) = 0 \,.
\end{aligned}
\right.
\end{equation*}
Subtracting these, we get 
\begin{align*}
& \fint_U 
\bigl( Z_{\f}(\cdot,U,0) - \tilde{Z}_{\f}(\cdot,U,0) \bigr) \cdot \tilde{\bfA}
\bigl( Z_{\f}(\cdot,U,0) - \tilde{Z}_{\f}(\cdot,U,0) \bigr)
\\ & \qquad 
=
\fint_U \Bigl( \bigl( Z_{\f}(\cdot,U,0) - \tilde{Z}_{\f}(\cdot,U,0) \bigr) \cdot \bigl( (\tilde{\bfA} - \bfA ) Z_{\f}(\cdot,U,0) 
+
2(\tilde{\bfF} - \bfF) \bigr) \Big) \,.
\end{align*}
We deduce that 
\begin{align*}
&
\| Z_{\f}(\cdot,U,0) - \tilde{Z}_{\f}(\cdot,U,0)  \|_{\underline{L}^2(U)}^2
\\ & \qquad 
\leq 
C\| Z_{\f}(\cdot,U,0) - \tilde{Z}_{\f}(\cdot,U,0)  \|_{\underline{L}^2(U)}
\bigl( 
\| \tilde{\bfA} - \bfA \|_{L^\infty(U)} \|  Z_{\f}(\cdot,U,0) \|_{\underline{L}^2(U)}
+
\| \tilde{\bfF} - \bfF \|_{\underline{L}^2(U)} 
\bigr)
\end{align*}
and then 
\begin{align*}
\lefteqn{
\| Z_{\f}(\cdot,U,0) - \tilde{Z}_{\f}(\cdot,U,0)  \|_{\underline{L}^2(U)} 
}\qquad &
\notag \\ &
\leq 
C \Bigl( \| \f\|_{\underline{L}^2(U)} \| \a - \tilde{\a} \|_{L^\infty(U)}
+
\| \f - \tilde{\f} \|_{\underline{L}^2(U)} 
\Bigr)
\leq
C \bigl(1 + \| \f\|_{\underline{L}^2(U)} \| \bigr) t
\,.
\end{align*}
Using the formula~\eqref{e.bfF.by.Zf}, we deduce that 
\begin{align*}
\bigl| \bfF(U) - \tilde{\bfF}(U) \bigr| 
&
=
\biggl| \fint_{U} 
\Bigl( 
\tilde{\bfA}
\bigl( Z_{\f}(\cdot,U,0) - \tilde{Z}_{\f}(\cdot,U,0) \bigr) 
+  
\bigl( \bfA - \tilde{\bfA} \bigr)  Z_{\f}(\cdot,U,0) 
+
2(\bfF - \tilde{\bfF} )
\Bigr) 
\biggr| 
\\ & 
\leq 
C \Bigl( \| \f\|_{\underline{L}^2(U)} \| \a - \tilde{\a} \|_{L^\infty(U)}
+
\| \f - \tilde{\f} \|_{\underline{L}^2(U)} 
\Bigr)
\leq C \bigl(1 + \| \f\|_{\underline{L}^2(U)} \| \bigr) t\,.
\end{align*}
We also get 
\begin{align*}
\mu_{\f}(U,0) - \tilde{\mu}_{\f}(U,0)  
&
\leq 
\fint_{U} 
\Bigl( 
\frac12 Z_{\f}(\cdot,U,0)  \cdot \bfA Z_{\f}(\cdot,U,0) 
- \frac12 \tilde{Z}_{\f}(\cdot,U,0)  \cdot \tilde{\bfA}\tilde{Z}_{\f}(\cdot,U,0) 
\Big) 
\\ & 
=
\fint_{U} 
\Bigl( 
\frac12 \bigl( Z_{\f}(\cdot,U,0) - \tilde{Z}_{\f}(\cdot,U,0) \bigr)  \cdot \tilde\bfA \bigl( Z_{\f}(\cdot,U,0) + \tilde{Z}_{\f}(\cdot,U,0) \bigr)  
\\ & \qquad \quad
+
\frac12 Z_{\f}(\cdot,U,0)  \cdot \bigl( \bfA - \tilde{\bfA} \bigr) Z_{\f}(\cdot,U,0) 
\Big)  
\\ & 
\leq
C \bigl(  \| \f \|_{\underline{L}^2(U)} + \| \tilde{\f} \|_{\underline{L}^2(U)} \bigr)
\bigl(1 + \| \f\|_{\underline{L}^2(U)} \bigr) t
\\ & 
\leq 
C \bigl(  \| \f \|_{\underline{L}^2(U)} + t \bigr)
\bigl(1 + \| \f\|_{\underline{L}^2(U)} \bigr) t
\,.
\end{align*}
Sending~$t\to 0$ yields the bound for~$| \partial_{(\a,\f)(U)} \bfF(U)|$ in~\eqref{e.Malliavin.bfF} and~$| \partial_{(\a,\f)(U)}\mu_{\f}(U,0) |$ in~\eqref{e.Malliavin.F.muf}.  The bounds for~$| \partial_{(\a,\f)(U)} \bfF_*(U)|$ and~$| \partial_{(\a,\f)(U)}\mu_{\f}^*(U,0) |$ are proved analogously and we omit the computation. 
\end{proof}

\subsubsection{Iterating up the scales to get the algebraic rate}

We can control the convergence of coarse-grained vectors~$\bfF$ and~$\bfF_*$ without an iteration in the scales. The reason is that the fluctuations of the coarse-grained matrices~$\bfA$ and~$\bfA_*$ control the means of these quantities and their difference. Once the means are controlled, we can apply the mixing condition to control the fluctuations. 
This is a coarse-graining manifestation of the following idea: if we were to homogenize the equation
\begin{equation*}
-\Delta u = \nabla \cdot \f \,,
\end{equation*}
then, because the left side is Laplacian, we can just integrate~$\f$ against the elliptic Green function to obtain the solution. This would reveal that the coarse-grained~$\f$ should be a simple average of~$\f$; in other words, the coarse-graining is additive. 

\begin{lemma} 
\label{l.bF.converge}
There exist constants~$\alpha(\beta,d,\lambda,\Lambda) \in (0,1)$ and~$C(\datareff) < \infty$, and a constant vector~$\overline{\bfF} \in \R^{2d}$,  such that, for every~$m \in \N$, 
\begin{equation}
\label{e.bfF.convergeE}
\E \bigl[  \bigl| \bfF(\cu_m) - \overline{\bfF} \bigr|^2 \Bigr]
+
\E \bigl[  \bigl| \bfF_*(\cu_m) - \overline{\bfF} \bigr|^2 \Bigr]
\leq 
C3^{-m\alpha}\,.
\end{equation}
\end{lemma}

\begin{proof}

\emph{Step 1.} We first show that, for any invertible symmetric matrix~$\mathbf{B} \in \R^{2d \times 2d}$, we have for any Lipschitz domain~$U$ that
\begin{equation} 
\label{e.F.vs.Fstar.withB}
\bigl| \mathbf{B}^{-\nicefrac12}\bigl(\bfF(U) - \bfF_*(U)   \bigr) \bigr|^2
\leq
\bfJ_{\f}(U,0,0) \bigl| \mathbf{B}^{-\nicefrac12} \bigl( \bfA(U) - \bfA_*(U) \bigr)\mathbf{B}^{-\nicefrac12} \bigr|
\end{equation}
and, for every~$m,n \in \N$ with~$m\geq n$, 
\begin{align} 
\label{e.f.coarseF.additive}
\lefteqn{
\bigg| \avsum_{z \in 3^n \Z^d \cap \cu_m} \mathbf{B}^{-\nicefrac12}  \bigl( \bfF(\cu_m) -   \bfF(z+\cu_n)   \bigr) \bigg|^2 
} \quad &
\notag \\ &
\leq
\biggl(\avsum_{z \in 3^n \Z^d \cap \cu_m} \!\! \bigl( \bfJ_{\f}(z+\cu_n, 0,0)  - \bfJ_{\f}(\cu_m, 0,0) \bigr)\biggr)
\notag \\ & \qquad \times
 \biggl| \avsum_{z \in 3^n \Z^d \cap \cu_m} \mathbf{B}^{-\nicefrac12}  \bigl( \bfA(\cu_m) -  \bfA(z+\cu_n) \bigr)\mathbf{B}^{-\nicefrac12}   \biggr|
\,.
\end{align}
Hence, the two notions of coarse-grained vectors of~$\bfF$,~$\bfF(U)$ and~$\bfF_*(U)$,  are essentially the same once the coarse-grained matrices have converged, and~$U \mapsto \bfF(U)$ becomes additive, provided we have that additivity of the coarse-grained matrices~$\bfA(U)$ with a cost of a square root. To show~\eqref{e.F.vs.Fstar.withB}, observe that, by~\eqref{e.coarseF.def},~\eqref{e.firstvar.nonsymm.dbl},~\eqref{e.bfJ.formula.via.bfA} and Cauchy-Schwarz inequality, we obtain
\begin{align} \label{e.F.vs.Fstar}
P\cdot \bigl(\bfF(U) - \bfF_*(U)   \bigr)
& =
P\cdot \fint_{U} \bigl( \bfA - \bfA_*(U)  \bigr) X_{\f}(\cdot,U,0,0) 
\notag \\ & 
=   \fint_{U} X_{\f}(\cdot,U,0,0) \cdot \bfA S (\cdot,U,P,\bfA_*(U) P) 
\notag \\ & 
\leq 
\frac12 \left( \fint_{U} \bigl| \bfA^{\nicefrac12}X_{\f}(\cdot,U,0,0) \bigr|^2 \right)^{\!\!\nicefrac12} 
\left| P\cdot ( \bfA(U) - \bfA_*(U)) P \right|^{\nicefrac12} 
\notag \\ & 
\leq 
\bigl( \bfJ_{\f}(U,0,0) \bigr)^{\nicefrac12} \bigl| P\cdot ( \bfA(U) - \bfA_*(U)) P \bigr|^{\nicefrac12}
\,.
\end{align}

\smallskip

To show~\eqref{e.f.coarseF.additive}, using~\eqref{e.Jf.repre0} and the subadditivity of~$U \mapsto \bfJ_{\f}(U, P,0)$, we have 
\begin{align}
\label{e.weird.F.subadd.def}
\lefteqn{
 P\cdot \biggl(  \bfF(\cu_m) - \!\!\!\! \avsum_{z \in 3^n \Z^d \cap \cu_m} \!\!\!\! \bfF(z+\cu_n)  \biggr) 
} & \qquad 
\notag \\  & = \avsum_{z \in 3^n \Z^d \cap \cu_m} \!\!\!\! \bfJ_{\f}(z+\cu_n, 0,0)  - \bfJ_{\f}(\cu_m, 0,0) 
+  \bfJ_{\f}(\cu_m, P,0)  - \!\!\!\! \avsum_{z \in 3^n \Z^d \cap \cu_m} \!\!\!\! \bfJ_{\f}(z+\cu_n, P,0) 
\notag \\ & \qquad 
+ \avsum_{z \in 3^n \Z^d \cap \cu_m} \!\!\!\! \bfJ_{\bfzero}(z+\cu_n, P,0)  - \bfJ_{\bfzero}(\cu_m, P,0)  
\notag \\ &
\leq
\avsum_{z \in 3^n \Z^d \cap \cu_m} \!\!\!\! \Bigl( \bfJ_{\f}(z+\cu_n, 0,0)  - \bfJ_{\f}(\cu_m, 0,0)   +  \frac12 P\cdot \bigl( \bfA(z+\cu_n) - \bfA(\cu_m) \bigr)P\Bigr)  \,.
\end{align}
Fix~$\delta \in (0,1)$ and an invertible symmetric matrix~$\mathbf{B} \in \R^{2d \times 2d}$, and set
\begin{align*} 
\Upsilon_\delta & := \Biggl(\delta + \! \! \! \! \! \avsum_{z \in 3^n \Z^d \cap \cu_m} \!\!\!\! \! \!\bfJ_{\f}(z+\cu_n, 0,0)  - \bfJ_{\f}(\cu_m, 0,0)\Biggr) 
\notag \\ & \qquad  \times \Biggl(\delta + \biggl| \avsum_{z \in 3^n \Z^d \cap \cu_m} \!\!\!\!  \mathbf{B}^{-\nicefrac12} \bigl( \bfA(\cu_m) - \bfA(z+\cu_n) \bigr) \mathbf{B}^{-\nicefrac12}  \biggr| \Biggr)^{\! -1} 
\end{align*} 
and
\begin{align*} 
\mathbf{e}_\delta & :=  \mathbf{B}^{-\nicefrac12}\biggl( \bfF(\cu_m) - \!\!\!\! \avsum_{z \in 3^n \Z^d \cap \cu_m} \!\!\!\! \bfF(z+\cu_n) \biggr) 
\notag \\ & \qquad  \times  
\biggl( \delta + \biggl| \avsum_{z \in 3^n \Z^d \cap \cu_m} \!\!\!\!  \mathbf{B}^{-\nicefrac12} \bigl( \bfF(\cu_m) -  \bfF(z+\cu_n) \bigr) \biggr| \biggr)^{\! -1}
\,.
\end{align*} 
We insert the vector~$P :=  \Upsilon_\delta^{\nicefrac12}  \mathbf{e}_\delta$ in~\eqref{e.weird.F.subadd.def} and obtain, after dividing the resulting estimate from both sides by~$\Upsilon_\delta^{\nicefrac12}$ and sending~$\delta \to 0$, the estimate~\eqref{e.f.coarseF.additive}.  

\smallskip

\emph{Step 2.} We show~\eqref{e.bfF.convergeE}. Taking the expectation of the square roots of~\eqref{e.F.vs.Fstar.withB} and~\eqref{e.f.coarseF.additive} (with~$\mathbf{B} = \Itwod$), an application H\"older's inequality gives us
\begin{multline*} 
\bigl| \E[\bfF(\cu_m)] -   \E[\bfF(\cu_n)] \bigr|  
+ 
\bigl| \E[\bfF(\cu_n)] -   \E[\bfF_*(\cu_n)] \bigr|  
\\ 
\leq 
C \E\bigl[ \bfJ_{\f}(\cu_n,0,0) \bigr]^{\nicefrac12} \bigl| \bfAhom(\cu_n) -   \bfAhom(\cu_m)\bigr|^{\nicefrac12}
 \,.
\end{multline*}
This, together with Proposition~\ref{p.algebraic.nosymm}, 
yields that there exist constants~$\alpha(\beta,d,\lambda,\Lambda) \in (0,1)$ and~$C(\datareff) < \infty$ such that, for every~$m,n \in \N$ with~$m \geq n$, 
\begin{align}  \label{e.bF.diffexp}
\bigl| \E[\bfF(\cu_m)] -   \E[\bfF(\cu_n)] \bigr|  
+ 
\bigl| \E[\bfF(\cu_n)] -   \E[\bfF_*(\cu_n)] \bigr|  
\leq 
C 3^{-n \alpha}  \E\bigl[ \| \f \|_{\underline{L}^2(\cu_0)}^2 \bigr]^{\nicefrac12} 
\,.
\end{align}
Thus the sequences~$\{ \E [ \bfF(\cu_m)]\}_{m\in\N}$ and~$\{ \E [ \bfF_*(\cu_m)]\}_{m\in\N}$ are Cauchy. In view of~\eqref{e.F.vs.Fstar} the limits must be the same; so there exists~$\overline{\bfF} \in \R^{2d}$ such that, for every~$m\in\N$,
\begin{equation} \label{e.f.coarseF.expectation}
| \E[\bfF(\cu_m)] -  \overline{\bfF}|
+
|  \E[ \bfF_*(\cu_m)] - \overline{\bfF} |   
\leq 
C 3^{-m \alpha}  \E\Bigl[ \| \f \|_{\underline{L}^2(\cu_0)}^2 \Bigr]^{\nicefrac12} 
\,.
\end{equation}

\smallskip

We next estimate the variance of~$\bfF(\cu_m)$. Using~\eqref{e.f.coarseF.additive} we have, for every~$m,n\in\N$ with~$\beta m < n < m$,  
\begin{align} 
\label{e.varF.cu.m.splitting}  
\lefteqn{
\var \bigl[ \bfF(\cu_m) \indc_{\{ \| \f \|_{\underline{L}^2(\cu_m)} \leq 2 \}} \bigr] 
} \qquad &
\\ 
 &
\leq
2\var \Biggl[ \avsum_{z \in 3^n \Z^d \cap \cu_m}  \bfF(z+\cu_n)  \Biggr]
+
C
\E \Biggl[ \biggl|  \bfA(\cu_m) -\avsum_{z \in 3^n \Z^d \cap \cu_m} \bfA(z+\cu_n)  \biggr| 
\Biggr]
\notag \\ & 
\leq 
2\var \Biggl[ \avsum_{z \in 3^n \Z^d \cap \cu_m} \!\!  \bfF(z+\cu_n)  \Biggr]
+
C 3^{-n\alpha} 
\,.
\end{align}
To estimate the first term on the right side, we first write
\begin{align*}
\avsum_{z \in 3^n \Z^d \cap \cu_m} \!\!  \bfF(z+\cu_n) 
& 
=
\avsum_{z \in 3^n \Z^d \cap \cu_m} \!\!  \bfF(z+\cu_n)  \rho\bigl( \| \f \|_{\underline{L}^2(z+\cu_n)} \bigr)
\\ & \quad 
+ \!\!
\avsum_{z \in 3^n \Z^d \cap \cu_m} \!\!  \bfF(z+\cu_n) \bigl( 1 -  \rho\bigl( \| \f \|_{\underline{L}^2(z+\cu_n)} \bigr) \bigr) 
\end{align*}
where~$\rho : [0,\infty) \to [0,1]$ is a smooth function satisfying~$\indc_{[0, 2]} \leq \rho \leq \indc_{[0,3]}$ and~$| \rho' | \leq 2$. 
Using the assumption~\eqref{e.f.dumbbound} and the triangle inequality~\eqref{e.weak.triangle.ass}, we have 
\begin{align}
\biggl| \avsum_{z \in 3^n \Z^d \cap \cu_m} \!\!  \bfF(z+\cu_n) \bigl( 1 -  \rho\bigl( \| \f \|_{\underline{L}^2(z+\cu_n)} \bigr) \bigr) \biggl|^2 
&
\leq 
\avsum_{z \in 3^n \Z^d \cap \cu_m} \!\!  \bigl|  \bfF(z+\cu_n)\bigr|^2 \indc_{\| \f \|_{\{ \underline{L}^2(z+\cu_n)}>2\}} 
\notag \\ & 
\leq 
C \avsum_{z \in 3^n \Z^d \cap \cu_m} \indc_{\| \f \|_{\{ \underline{L}^2(z+\cu_n)}>2\}} 
\notag \\ & 
\leq 
\O_\Psi \bigl( C \mathsf{K} 3^{-\frac d2(1-\beta)n} \bigr) \,.
\label{e.remove.bad.stuff.for.f}
\end{align}
On the other hand, using the~$\CFS$ assumption and the fact that, by~\eqref{e.Malliavin.bfF}, 
\begin{equation*}
\bigl| \partial_{(\a,\f)(U)} \bigl[ \bfF(U) \rho\bigl( \| \f \|_{\underline{L}^2(U)} \bigr)  \bigr| 
\leq 
C\,,
\end{equation*}
we obtain 
\begin{equation}
\label{e.apply.CFS.var.for.f}
\avsum_{z \in 3^n \Z^d \cap \cu_m} \!\!  \bfF(z+\cu_n)  \rho\bigl( \| \f \|_{\underline{L}^2(z+\cu_n)} \bigr)
-
\E \bigl[ \bfF(\cu_n)  \rho\bigl( \| \f \|_{\underline{L}^2(\cu_n)} \bigr) \bigr] 
=
\O_\Psi (C3^{-\frac d2(m-n)}) \,.
\end{equation}
Combining~\eqref{e.remove.bad.stuff.for.f} and~\eqref{e.apply.CFS.var.for.f} yields
\begin{equation*}
\var \Biggl[ \avsum_{z \in 3^n \Z^d \cap \cu_m} \bfF(z+\cu_n)  \Biggr]
\leq 
C( 3^{-d(m-n)} + 3^{-d(1-\beta)n}) \,.
\end{equation*}
Inserting this estimate into~\eqref{e.varF.cu.m.splitting} and taking~$n:= \lceil (\beta \vee \tfrac12) m \rceil+1$, we obtain
\begin{equation*}
\var \bigl[ \bfF(\cu_m) \indc_{\{ \| \f \|_{\cu_m} \leq 2 \}} \bigr] 
\leq 
C3^{-d(m-n)}
+
C3^{-n\alpha} 
\leq 
C3^{-m\alpha}\,,
\end{equation*}
where we shrank the value of~$\alpha$ in the last term. 
The indicator function can be removed using~\eqref{e.f.dumbbound}. 
The proof is complete by~\eqref{e.f.coarseF.expectation}. 
\end{proof}

With an eye toward~\eqref{e.Jf.repre0zoopie}, we next seek to control~$\bfJ_{\f}(\cu_{m},0,\overline{\bfF})$. This is where we need an iteration in the scales. However, the iteration in this case is slightly simpler than the previous ones we have seen.

\begin{lemma} \label{l.f.J.minimal.expectation}
There exist constants~$\alpha(\beta,d,\lambda,\Lambda) \in (0,1)$ and~$C(\datareff) < \infty$ such that, 
for every~$m\in\N$, we have the estimate
\begin{equation}  \label{e.f.J.minimal.expectation}
\E[\bfJ_{\f}(\cu_{m},0,\overline{\bfF})] \leq C 3^{-m\alpha} \,.
\end{equation}
\end{lemma}
\begin{proof}
Let~$(q^*,p^*) :=  \overline{\bfF}$ and~$(p,q) := 0$.  
By~\eqref{e.bfJ.withf} we have 
\begin{equation*} 
\bfJ_{\f}\bigl(U, 0, \overline{\bfF} \bigr)
=
\frac12 J_{\f}\bigl(U,-p^*,q^*\bigr)
+ 
\frac12 J^*_{\f}\bigl(U,p^*,q^*\bigr) 
\end{equation*}
and, by~\eqref{e.maximizers.J.to.bfJ.f},
\begin{equation*} 
X_{\f}(\cdot,U,0,\overline{\bfF}) 
= 
\begin{pmatrix} \nabla v_{\f} (\cdot,U,-p^*,q^*)
+ 
\nabla v_{\f}^*\bigl(\cdot,U,p^*,q^*\bigr) \\ 
\a \nabla v_{\f} (\cdot,U,-p^*,q^*)
- 
\a^t \nabla v_{\f}^*\bigl(\cdot,U,p^*,q^*\bigr)
\end{pmatrix}
\end{equation*}
with~$v_{\f} := v_{\f} (\cdot,U,-p^*,q^*)$ and~$v_{\f}^* := v_{\f}^*\bigl(\cdot,U,p^*,q^*\bigr)$ belonging to~$\A(U)$ and~$\A^*(U)$, respectively. Similarly to~\eqref{e.f.splittingformula.for.Z}, 
we then get, with~$\rota$ defined by~\eqref{e.rota}, 
\begin{equation*} 
\begin{pmatrix} \nabla v_{\f} \\ \a \nabla v_{\f} \end{pmatrix} = \frac12 \bigl(X_{\f} + \rota \bfA X_{\f} \bigr) 
\qand
\begin{pmatrix} \nabla v_{\f}^* \\ - \a^t \nabla v_{\f}^* \end{pmatrix} = \frac12 \bigl(X_{\f} - \rota \bfA X_{\f} \bigr) 
 \,.
\end{equation*}
Thus, the weak norms of~$(\nabla v_{\f},\a \nabla v_{\f})$ and~$(\nabla v_{\f}^*,- \a^t \nabla v_{\f}^*)$ can be controlled by the weak norms of~$X_{\f}$ and~$\bfA X_{\f}$, and these can in turn be estimated using~\eqref{e.average.X.f}. We have, for every~$P \in \R^{2d}$, 
\begin{align*} 
\fint_{z+\cu_n} X_{\f}(\cdot,\cu_{m},0,P) 
& =
\fint_{z+\cu_n}  \bigl( X_{\f}(\cdot,\cu_{m},0,P) - X_{\f}(\cdot,z+\cu_{n},0,P) \bigr) 
\notag \\ & \qquad 
+ \bfA_*^{-1}(z+\cu_n) (P- \bfF_*(z+\cu_n))
\end{align*}
and
\begin{align*} 
\fint_{z+\cu_n} \bfA X_{\f}(\cdot,\cu_{m},0,P) 
& =
\fint_{z+\cu_n}  \bigl( \bfA X_{\f}(\cdot,\cu_{m},0,P) - \bfA X_{\f}(\cdot,z+\cu_{n},0,P) \bigr) 
\notag \\ & \qquad 
+ P- \bfF(z+\cu_n)
\end{align*}
Therefore, by the last three displays, 
\begin{align*} 
\lefteqn{
\avsum_{z \in 3^n \Zd \cap \cu_m} \Bigl( \bigl| \bigl(\nabla v_{\f}(\cdot,\cu_m,-p^*,q^*)\bigr)_{z+\cu_n} \bigr|^2 + \bigl| \bigl(\nabla v_{\f}^*(\cdot,\cu_m,p^*,q^*)\bigr)_{z+\cu_n} \bigr|^2 \Bigr)  \,.
} \qquad &
\notag \\ &
\leq C \avsum_{z \in 3^n \Zd \cap \cu_m}  \Bigl(
\bfJ_{\f}\bigl(z+\cu_n, 0, \overline{\bfF} \bigr) - \bfJ_{\f}\bigl(\cu_m, 0, \overline{\bfF} \bigr)   \Bigr)
\notag \\ & \qquad
+ C \avsum_{z \in 3^n \Zd \cap \cu_m}
\Bigl( \bigl| \overline{\bfF} - \bfF_*(z+\cu_n)\bigr|^2 + \bigl| \overline{\bfF} - \bfF(z+\cu_n) \bigr|^2 \Bigr)
\,.
\end{align*}
We may now repeat Steps 1-2 of the proof of Lemma~\ref{l.flatness.rules.nosymm} and deduce that
\begin{align*} 
\lefteqn{
J_{\f}(\cu_{m-1},-p^*,q^*) 
} \qquad & 
\notag \\ &
\leq  
C 3^{-2m} \| \nabla v_{\f} \|_{\underline{L}^2(\cu_m)}^2 
+ 
C
\sum_{n=0}^{{m}} 3^{n-m} 
\Bigl(
\bfJ_{\f}\bigl(z+\cu_n, 0, \overline{\bfF} \bigr) - \bfJ_{\f}\bigl(\cu_m, 0, \overline{\bfF} \bigr)   \Bigr)
\notag \\ & \qquad
+ C \sum_{n=0}^{{m}} 3^{n-m}  \avsum_{z \in 3^n \Zd \cap \cu_m}
\Bigl( \bigl| \overline{\bfF} - \bfF_*(z+\cu_n)\bigr|^2 + \bigl| \overline{\bfF} - \bfF(z+\cu_n) \bigr|^2 \Bigr)
\,,
\end{align*}
and similarly for~$J_{\f}^*(\cu_{m-1},p^*,q^*)$. By taking expectations of the resulting estimates and using~$\Zd$--stationarity, subadditivity, Lemma~\ref{l.bF.converge} and the bound~$\E\bigl[\bfJ_{\f}(\cu_{m},0,\overline{\bfF})\bigr] \leq C$, we deduce that 
\begin{equation*}  
\E\bigl[ \bfJ_{\f}(\cu_{m},0,\overline{\bfF}))\bigr]
\leq 
C \sum_{n=1}^m 3^{n-m}  \bigl( \E\bigl[ \bfJ_{\f}(\cu_n,0,\overline{\bfF})\bigr] - \E\bigl[  \bfJ_{\f}(\cu_m,0,\overline{\bfF}) \bigr] \bigr)
+
C 3^{-m\alpha} 
\,.
\end{equation*}
A simple iteration argument, together with Lemma~\ref{l.bF.converge}, yields
\begin{equation*}
\E[\bfJ_{\f}(\cu_{m},0,\overline{\bfF})] \leq C 3^{-m\alpha}\,,
\end{equation*}
which is~\eqref{e.f.J.minimal.expectation}. The proof is complete. 
\end{proof}

We have now finished most of the heavy lifting and are ready to explore the consequences of what we've shown. 
Analogously to the definition of~$\mathcal{E}(m)$ in~\eqref{e.mcE.0.nonsymm}, we define
\begin{equation} 
\label{e.mathcalE.f}
\mathcal{E}_{\f}(m)
:= 
\sum_{n=0}^{m} 3^{n-m}   \avsum_{z \in 3^n \Z^d \cap \cu_m}  \bfJ_{\f}\bigl(z{+}\cu_n,0, \overline{\bfF} \bigr)  
\end{equation}
and
\begin{equation} 
\label{e.mathcalE.again.f}
\mathcal{E}_{\bfzero}(m)
:=
\sum_{i=1}^{2d} 
\sum_{n=0}^{m} 3^{n-m}  \bfJ_{\bfzero}\bigl(z{+}\cu_n,e_i, \bfAhom e_i \bigr) 
\leq C \mathcal{E}(m)
 \,,
\end{equation}
where~$\mathcal{E}(m)$ has been defined by~\eqref{e.mcE.0.nonsymm}. 
Similar to~\eqref{e.Jf.repre0zoopie}, we have the identity
\begin{align}
\label{e.expand.around.bfFbar}
\lefteqn{
\bfJ_{\f}(U, P,Q ) 
} \quad  & \notag \\ & 
=
\bfJ_{\bfzero}(U,P,Q{-}\overline\bfF ) {+} 
\bfJ_{\f}(U,0,\overline{\bfF} )
{+} P \cdot (\bfF(U) {-}\overline{\bfF})
{+} (Q{-}\overline{\bfF} ) \cdot \bfA_*^{-1}(U) (\overline{\bfF} {-} \bfF_*(U))\,.
\end{align}
Observe that, by~\eqref{e.expand.around.bfFbar},  
\begin{align*}
0
\leq \bfJ_{\f}(U,P,\overline{\bfF}) 
& 
=
\bfJ_{\bfzero}(U,P,0) 
+
\bfJ_{\f}(U,0,\overline{\bfF} )
+ P \cdot (\bfF(U) {-}\overline{\bfF})
\\ & 
=
\frac12 P \cdot \bfA(U) P 
+
\bfJ_{\f}(U,0,\overline{\bfF} )
+ P \cdot (\bfF(U) {-}\overline{\bfF})
\,.
\end{align*}
Rewriting this as 
\begin{equation*}
P \cdot ( \overline{\bfF}-\bfF(U)) 
\leq 
C |P|^2 + \bfJ_{\f}(U,0,\overline{\bfF} )\,,
\end{equation*}
and then taking~$P := (2C)^{-1} ( \overline{\bfF}-\bfF(U))$, we deduce after reabsorption that 
\begin{equation*}
\big| \bfF(U) - \overline{\bfF}\bigr|^2
\leq 
C\bfJ_{\f}(U,0,\overline{\bfF} )\,.
\end{equation*}
By a similar computation, 
\begin{equation*}
\big|\bfF_*(U) -  \overline{\bfF}\bigr|^2
\leq 
C\bfJ_{\f}(U,0,\overline{\bfF} )\,.
\end{equation*}
Consequently, we obtain
\begin{equation}
\label{e.F.and.F.minset}
\sum_{n=0}^{m} 3^{n-m} \avsum_{z \in 3^n \Z^d \cap \cu_m} \!\!\! 
\Bigl(
\big|\bfF(z+\cu_n) -  \overline{\bfF}\bigr|^2
+
\big|\bfF_*(z+\cu_n) -  \overline{\bfF}\bigr|^2
\Bigr)
\leq 
C\mathcal{E}_{\f}(m)\,.
\end{equation}

\smallskip

We next demonstrate how~$\mathcal{E}_{\f}(m)$ can be used to control the weak norms of the minimizer of~$\mu_{\f}(U,P)$ defined by~\eqref{e.dual.smitz}.

\begin{lemma} 
\label{l.f.Hminusone.zero}
There exists a constant~$C(d,\lambda,\Lambda)<\infty$ such that, for every~$m \in \N$, 
\begin{multline}
\label{e.f.Hminusone.zero}  
\bigl\| Z_{\f}(\cdot,\cu_m,P) -P \bigr\|_{\Hminusul(\cu_m)}
+
\bigl\| \bfA Z_{\f}(\cdot,\cu_m,P) + 2\bfF - \overline{\bfF} - \bfAhom P \bigr\|_{\Hminusul(\cu_m)}
\\
\leq 
C( \| \f \|_{\underline{L}^2(\cu_m)} + |P|) + C3^m \bigl(  |P| \mathcal{E}_{\bfzero}^{\nicefrac12}(m) + \mathcal{E}_{\f}^{\nicefrac12}(m) \bigr) 
\end{multline}
and
\begin{multline}
\label{e.f.Hminusone.zero.star}  
\bigl\| Z_{\f}^*(\cdot,\cu_m,P) - \bfAhom^{-1}(Q - \overline{\bfF} ) \bigr\|_{\Hminusul(\cu_m)}
+
\bigl\| \bfA Z_{\f}(\cdot,\cu_m,Q) + 2\bfF - Q \bigr\|_{\Hminusul(\cu_m)}
\\
\leq 
C( \| \f \|_{\underline{L}^2(\cu_m)} + |Q|) + C3^m \bigl(  |Q| \mathcal{E}_{\bfzero}^{\nicefrac12}(m) + \mathcal{E}_{\f}^{\nicefrac12}(m) \bigr) 
\,.
\end{multline}

\end{lemma}

\begin{proof}
By~\eqref{e.trivial.blitzen} we have
\begin{equation*} 
0 = \fint_U Z_{\f}(\cdot,U,0)
\quad \mbox{and} \quad 
0 = \fint_U \bigl( \bfA Z^*_{\f}(\cdot,U,0)+2\bfF\bigr)
\,,
\end{equation*}
and, by~\eqref{e.bfF.by.Zf} and~\eqref{e.bfF.by.Zf.star}, 
\begin{equation*} 
\bfF(U) 
= 
\fint_{U} \bigl( \bfA  Z_{\f}(\cdot,U,0) +  2\bfF\bigr)
\qand
\bfF_*(U) 
= - \bfA_{*}(U)  \fint_{U} Z_{\f}^*(\cdot,U,0) 
  \,.
\end{equation*}
Moreover, by~\eqref{e.formulas.nosymm.byS} and~$X_{\mathbf{0}}(\cdot,U,P,Q) = S(\cdot,U,P,Q)$, we get   
\begin{equation*} 
\fint_U \bfA X_{\mathbf{0}}(\cdot,U,P,Q) = Q-\bfA(U) P 
\qand
\fint_U X_{\mathbf{0}}(\cdot,U,P,Q) = \bfA_*^{-1}(U) Q - P \,.
\end{equation*}
By~\eqref{e.Z.zero.split}, we have~$Z_{\mathbf{0}}(\cdot,U,P) = -X_{\mathbf{0}}(\cdot,U,P,0)$ and~$Z_{\mathbf{0}}^*(\cdot,U,Q) = X_{\mathbf{0}}(\cdot,U,0,Q)$. Thus, by~\eqref{e.decomp.splitten}, we get~$Z_{\mathbf{f}}(\cdot,U,P) = -X_{\mathbf{0}}(\cdot,U,P,0) + Z_{\mathbf{f}}(\cdot,U,0)$ and~$Z_{\mathbf{f}}^*(\cdot,U,Q) = X_{\mathbf{0}}(\cdot,U,0,Q) + Z_{\mathbf{f}}^*(\cdot,U,0)$. We now deduce formulas for all the averages:
\begin{equation} 
\label{e.f.Hminusone.zero.means}
\fint_U Z_{\f}(\cdot,U,P) = P 
\qand
\fint_{U} \bigl( AZ_{\f}(\cdot,U,P) + 2\bfF \bigr)
=  \bfF(U) + \bfA(U) P
\,,
\end{equation}
and
\begin{equation} 
\label{e.f.Hminusone.zero.means.star}
\fint_U Z_{\f}^*(\cdot,U,Q) = \bfA_*^{-1}(U)\bigl(Q - \bfF_*(U) \bigr)  
\qand
\fint_{U} \bigl( AZ_{\f}^*(\cdot,U,Q) + 2\bfF \bigr) = Q
 \,.
\end{equation}

\smallskip

We now proceed to prove~\eqref{e.f.Hminusone.zero}. 
By~\eqref{e.f.Hminusone.zero.means} and the multiscale Poincar\'e inequality (Proposition~\ref{p.MSP}), we get  
\begin{align}  
\label{e.mean.Zf.split}
\lefteqn{
\bigl\| Z_{\f}(\cdot,\cu_m,P) - P \bigr\|_{\Hminusul(\cu_m)}
} \qquad &
\notag \\ &
\leq
C\bigl\| Z_{\f}(\cdot,\cu_m,P) - P \bigr\|_{\underline{L}^{2}(\cu_m)}
\notag \\ & \qquad 
+ 
C\sum_{n=0}^{m} 3^{n} \Biggl( \avsum_{z \in 3^n \Z^d \cap \cu_m}    \biggl| \fint_{z + \cu_n} \bigl( X_{\bfzero}(\cdot,\cu_m,P,0) - X_{\bfzero}(\cdot,z+\cu_n,P,0) \bigr)   \biggr|^2 \Biggr)^{\! \nicefrac12}
\notag \\ & \qquad   
+ 
C\sum_{n=0}^{m} 3^{n} \Biggl( \avsum_{z \in 3^n \Z^d \cap \cu_m}    \biggl| \fint_{z + \cu_n} \bigl( Z_{\f}(\cdot,\cu_m,0) - Z_{\f}(\cdot,z+\cu_n,0) \bigr)   \biggr|^2 \Biggr)^{\! \nicefrac12}
\,.
\end{align}
The first term can be trivially estimated as
\begin{equation} 
\label{e.mean.Zf.split.first}
\bigl\| Z_{\f}(\cdot,\cu_m,P) - P \bigr\|_{\underline{L}^{2}(\cu_m)} 
\leq C\bigl(\|\f\|_{\underline{L}^2(\cu_m)}+|P|\bigr) \,.
\end{equation}
The second term on the right can be estimated by H\"older's inequality as
\begin{align} 
\label{e.mean.Zf.split.second}
\lefteqn{
\avsum_{z \in 3^n \Z^d \cap \cu_m}    \biggl| \fint_{z + \cu_n} \bigl( X_{\bfzero}(\cdot,\cu_m,P,0) - X_{\bfzero}(\cdot,z+\cu_n,P,0) \bigr)   \biggr|^2 
} \quad &
\notag \\ &
\leq
\avsum_{z \in 3^n \Z^d \cap \cu_m}  \bigl\|\bfA^{-1} \bigr\|_{\underline{L}^1(z+\cu_n)}  \fint_{z + \cu_n} \Bigl| \bfA^{\nicefrac12} \bigl( X_{\bfzero}(\cdot,\cu_m,P,0) - X_{\bfzero}(\cdot,z+\cu_n,P,0) \bigr)   \Bigr|^2 
\,.
\end{align}
Now, by~\eqref{e.wrap.app.nosymm} and~\eqref{e.firstvar.Jf}, we get 
\begin{align} 
\label{e.X.naught.additivity}
\lefteqn{
\sum_{n=0}^{m} 3^{n-m}
\avsum_{z \in 3^n \Z^d \cap \cu_m}   \fint_{z + \cu_n} \Bigl| \bfA^{\nicefrac12} \bigl( X_{\bfzero}(\cdot,\cu_m,P,0) - X_{\bfzero}(\cdot,z+\cu_n,P,0) \bigr)   \Bigr|^2 
} \qquad &
\notag \\ &
\leq 
C \sum_{n=0}^{m} 3^{n-m} \avsum_{z \in 3^n \Z^d \cap \cu_m} \bigl( \bfJ_{\bfzero}(z+\cu_n,P,0) -   \bfJ_{\bfzero}(\cu_m,P,0)  \bigr)
 \leq C |P|^2 \mathcal{E}_{\bfzero}(m)
\,.
\end{align}
To estimate the third term in~\eqref{e.mean.Zf.split}, we first observe that, by~\eqref{e.var.dual.smitz}, we have 
\begin{equation*} 
\bfA Z_{\f}(\cdot,z+\cu_n,0) - \bfA Z_{\f}(\cdot,\cu_m,0) \in \Lsol(z+\cu_n) \times  \Lpot(z+\cu_n) \,,
\end{equation*}
which gives us
\begin{equation*} 
\fint_{z+\cu_n} \bfA Z_{\f}(\cdot,\cu_m,0) \cdot Z_{\f}(\cdot,z+\cu_n,0) 
=
\fint_{z+\cu_n}  \bfA Z_{\f}(\cdot,z+\cu_n,0)  \cdot  Z_{\f}(\cdot,z+\cu_n,0) 
\,.
\end{equation*}
Thus, we get, using also H\"older's inequality and the first variation~\eqref{e.var.dual.smitz}, 
\begin{align} 
\label{e.mean.Zf.split.third} 
\lefteqn{\avsum_{z \in 3^n \Z^d \cap \cu_m}  \biggl| \fint_{z {+} \cu_n} \bigl( Z_{\f}(\cdot,\cu_m,0) - Z_{\f}(\cdot,z{+}\cu_n,0) \bigr)   \biggr|^2 } 
\qquad & 
\notag\\ &
\leq
 \avsum_{z \in 3^n \Z^d \cap \cu_m}  \bigl\| \bfA^{-1} \bigr\|_{\underline{L}^1(z+\cu_n)}
\fint_{z + \cu_n} \Bigl| \bfA^{\nicefrac12} \bigl(  Z_{\f}(\cdot,\cu_m,0) - Z_{\f}(\cdot,z+\cu_n,0)  \bigr) \Bigr|^2
\notag\\ &
=
C  \avsum_{z \in 3^n \Z^d \cap \cu_m}   \!\!\!  (\mu_{\f}(z{+}\cu_n,0) - \mu_{\f}(\cu_m,0) )
\,.
\end{align}
To estimate the last term on the right, by subadditivity and~\eqref{e.Jf.repre1}, we have 
\begin{align}  
\label{e.mu.f.subadd}
\mu_{\f}(\cu_m,0) 
& \leq 
\!\!\! \avsum_{z \in 3^n \Z^d \cap \cu_m} \!\!\!   \mu_{\f}(z{+}\cu_n,0) 
\notag \\ & 
=  - 
\!\!\! \avsum_{z \in 3^n \Z^d \cap \cu_m} \!\!\!   \mu_{\f}^*(z{+}\cu_n,\overline{\bfF})  
+ \!\!\! \avsum_{z \in 3^n \Z^d \cap \cu_m} \!\!\!   \bfJ_{\f}(z{+}\cu_{n},0,\overline{\bfF})
\notag \\ & 
\leq 
- \mu_{\f}^*(\cu_m,\overline{\bfF})   +  \!\!\! \avsum_{z \in 3^n \Z^d \cap \cu_m} \!\!\!   \bfJ_{\f}(z{+}\cu_{n},0,\overline{\bfF})  
\notag \\ & 
=
\mu_{\f}(\cu_m,0)
+
\!\!\! \avsum_{z \in 3^n \Z^d \cap \cu_m} \!\!\!   ( \bfJ_{\f}(z{+}\cu_{n},0,\overline{\bfF}) - \bfJ_{\f}(\cu_{m},0,\overline{\bfF}))\,.
\end{align}
Combining the above displays leads to
\begin{equation} 
\label{e.f.Hminus.average}
\sum_{n=0}^m 3^{n-m} \avsum_{z \in 3^n \Z^d \cap \cu_m} 
\fint_{z + \cu_n} \Bigl| \bfA^{\nicefrac12} \bigl(  Z_{\f}(\cdot,\cu_m,0) - Z_{\f}(\cdot,z+\cu_n,0)  \bigr) \Bigr|^2
\leq
C \mathcal{E}_{\f}(m)
\,.
\end{equation}
Combining~\eqref{e.mean.Zf.split} with~\eqref{e.mean.Zf.split.first},~\eqref{e.mean.Zf.split.second},~\eqref{e.X.naught.additivity},~\eqref{e.mean.Zf.split.third} and~\eqref{e.f.Hminus.average} leads to the desired estimate for the first term on the left in~\eqref{e.f.Hminusone.zero}.  

\smallskip 

We will then estimate the second term on the left in~\eqref{e.f.Hminusone.zero}. By~\eqref{e.f.Hminusone.zero.means}, the  multiscale Poincar\'e inequality implies
\begin{align*} 
\lefteqn{
\bigl\| \bfA Z_{\f}(\cdot,\cu_m,P) + 2 \bfF - \overline{\bfF} - \bfAhom P\bigr\|_{\underline{H}^{-1}(\cu_m)}^2
} \qquad &
\notag \\ &
\leq 
C\bigl\| \bfA Z_{\f}(\cdot,\cu_m,P) + 2 \bfF - \overline{\bfF} - \bfAhom P  \bigr\|_{\underline{L}^{2}(\cu_m)}^2
\notag \\ & \qquad 
+ 
C 3^{2m}\sum_{n=0}^{m} 3^{n-m} \avsum_{z \in 3^n \Z^d \cap \cu_m}   \biggl| \fint_{z + \cu_n} \bfA \bigl(  X_{\bfzero}(\cdot,\cu_m,P,0) -  X_{\bfzero}(\cdot,z+\cu_n,P,0)  \bigr) \biggr|^2
\notag \\ & \qquad 
+ 
C 3^{2m} \sum_{n=0}^{m} 3^{n-m} \avsum_{z \in 3^n \Z^d \cap \cu_m}   \biggl| \fint_{z + \cu_n} \bfA \bigl(  Z_{\f}(\cdot,\cu_m,0) - Z_{\f}(\cdot,z+\cu_n,0)  \bigr) \biggr|^2
\notag \\ & \qquad 
+ 
C 3^{2m} \sum_{n=0}^{m} 3^{n-m} \avsum_{z \in 3^n \Z^d \cap \cu_m}   \bigl| \bfF(z+\cu_n) - \overline{\bfF} + (\bfA(z+\cu_n) - \bfAhom)P   \bigr|^2
\,.
\end{align*}
The first term is again trivially bounded 
\begin{equation*} 
\bigl\| \bfA Z_{\f}(\cdot,\cu_m,P) + 2 \bfF - \overline{\bfF} - \bfAhom P  \bigr\|_{\underline{L}^{2}(\cu_m)} 
\leq C(\| \f \|_{\underline{L}^2(\cu_m)}+|P|) \,.
\end{equation*}
The second and the third terms can be bounded similarly as before (see~\eqref{e.X.naught.additivity} and~\eqref{e.f.Hminus.average}), and we omit them. The last term can be estimated using~\eqref{e.F.and.F.minset} and~\eqref{e.diagonalset.bigA} as
\begin{equation*} 
\sum_{n=0}^{m} 3^{n-m} \avsum_{z \in 3^n \Z^d \cap \cu_m}   \bigl| \bfF(z+\cu_n) - \overline{\bfF} + (\bfA(z+\cu_n) - \bfAhom)P   \bigr|^2 
\leq
C\mathcal{E}_{\f}(m) + C |P|^2 \mathcal{E}_{\bfzero}(m)
\,.
\end{equation*}
Combining the above displays concludes the proof of~\eqref{e.f.Hminusone.zero}. The proof of~\eqref{e.f.Hminusone.zero} is completely analogous using~\eqref{e.f.Hminusone.zero.means.star} instead of~\eqref{e.f.Hminusone.zero.means}, and we omit it. 
\end{proof}

We next obtain a bound on the size of~$\mathcal{E}_{\f}$, which can be compared to the statement of Theorem~\ref{t.subadd.converge.nosymm}.

\begin{corollary}
\label{c.Jf.minsetE}
Let~$\delta \in (0,1]$. 
There exist constants~$\alpha(\beta,d,\lambda,\Lambda) \in (0,1)$ and~$C(\delta,\datareff) < \infty$, 
and a random variable~$\X$ satisfying~$\X^{\frac d2 (1-\beta)} = \O_\Psi(C)$ such that, 
for every~$m\in\N$ with~$3^m\geq \X$, 
\begin{equation}
\label{e.Jf.minset.smash.again}
\mathcal{E}_{\f}(m) \leq  \delta (\X3^{-m})^{ \alpha} \,.
\end{equation}
Morevover, there is another random variable,~$\X_{k}$, bounded by~$3^{k+1}$, such that~$\X_{k}$ is~$\F(\cu_k)$-measurable,~$\X_k^{\frac d2 (1-\beta)} = \O_{\Psi}(C)$, the Malliavin derivative of~$\X_{k}$ is bounded by itself as
\begin{align} \label{e.f.mallliavin.local.again}
\bigl| \partial_{(\a,\f)(\cu_k)} \X_{k} \bigr| \leq C (1+\X_{k}) 
\,,
\end{align}
and if~$m,k \in \N$ are such that~$\X_{k} \leq 3^m < 3^k$, then
\begin{align}
\label{e.convssmaxsmax.local.again}
\mathcal{E}_{\f}(m) 
\leq
\delta  (\X_{k} 3^{-m})^{ \alpha}
\,,
\end{align}
\end{corollary}
\begin{proof}
It follows from~\eqref{e.expand.around.bfFbar} and the previous two lemmas that, for every~$P\in\R^{2d}$, 
\begin{equation}
\label{e.Jf.minsetE}
\E \Bigl[\bfJ_{\f} \bigl(U, P, \bfAhom P+ \overline{\bfF} \bigr) \Bigr]
\leq 
C\bigl( 1+ |P|^2 \bigr) 3^{-m\alpha} 
\,.
\end{equation}
We may now argue as in the proof of Lemmas~\ref{l.mathcalE.minscale} and~\ref{l.localX}.
The only difference is that, as in the proof of Lemma~\ref{l.bF.converge},  we need to use the cutoff function~$\rho$ to achieve the boundedness of each of the terms in the mesoscopic sum before applying the~$\CFS$ assumption. The error can be estimated separately, using~\eqref{e.f.dumbbound}. 
The required estimates on the Malliavin derivative of~$\bfJ_{\f}$ are found by using the formula~\eqref{e.Jf.repre0}, Lemma~\ref{l.Malliavin.f.base} and~\eqref{e.maul.Mall}. 
\end{proof}

We conclude this section with the proof of Theorem~\ref{t.zeroslope}. 

\begin{proof}[{Proof of Theorem~\ref{t.zeroslope}}]
Let~$P := (p,q)\in \R^{2d}$ to be fixed.  In view of~\eqref{e.f.splittingformula.for.Z}, we define
\begin{align*}  
\left\{
\begin{aligned}
\begin{pmatrix} \nabla v_{\f}(\cdot,U,P)  \\ \a \nabla v_{\f}(\cdot,U,P) + \f \end{pmatrix}  
:=
\frac12 \Bigl( Z_{\f}\bigl(\cdot,U,P \bigr) + \rota \bigl( \bfA Z_{\f}\bigl (\cdot,U,P \bigr) + 2 \bfF\bigr) \Bigr) \,,
\\
\begin{pmatrix} \nabla v_{\f}^*(\cdot,U,P)  \\ - \a \nabla v_{\f}^*(\cdot,U,P) - \f \end{pmatrix}  
:=
\frac12 \Bigl( Z_{\f}\bigl(\cdot,U,P \bigr) - \rota \bigl( \bfA Z_{\f}\bigl (\cdot,U,P \bigr) + 2 \bfF\bigr) \Bigr) \,,
\end{aligned}
\right.
\end{align*}

\smallskip

\emph{Step 1.} We construct~$\psi_{\f}(\cdot,U)$ and show~\eqref{e.zeroslope.sublin}. By Lemma~\ref{l.f.S.findme},~$v_{\f}(\cdot,U,P)$ belongs to~$\mathcal{A}_{\f}(U)$. 
Now, by Lemma~\ref{l.f.Hminusone.zero}, we obtain that 
\begin{multline*}  
\biggl\| \begin{pmatrix} \nabla v_{\f}(\cdot,\cu_m,P)  \\ \a \nabla v_{\f}(\cdot,\cu_m,P) + \f \end{pmatrix}    - \frac12 \bigr( P+ \rota (\overline{\bfF} - \bfAhom P)  \bigl)  \biggr\|_{\Hminusul(\cu_m)} 
\\
\leq
C( \| \f \|_{\underline{L}^2(\cu_m)} + |P|) + C3^m \bigl(  |P| \mathcal{E}_{\bfzero}^{\nicefrac12}(m) + \mathcal{E}_{\f}^{\nicefrac12}(m) \bigr) 
\end{multline*}
We then select the slope~$P$. Denoting~$\overline{\bfF}  = (\overline{\bfF}_1, \overline{\bfF}_2)$, we obtain by a straightforward computation that for~$\hat{P} := - ( \overline{\bfF}_2 , \khom \overline{\bfF}_2)$  
we have 
\begin{align*} 
\hat{P}+ \rota (\overline{\bfF} - \bfAhom \hat{P})
&
= 
- \begin{pmatrix}  \overline{\bfF}_2 \\ \khom \overline{\bfF}_2 \end{pmatrix}  
+\begin{pmatrix}  \overline{\bfF}_2 \\  \overline{\bfF}_1\end{pmatrix}  
- \begin{pmatrix} -\shom^{-1} \khom  & \shom^{-1}  \\ \shom + \khom^t \shom^{-1} \khom & - \khom^t \shom^{-1} \end{pmatrix} \begin{pmatrix} \overline{\bfF}_2 \\ \khom \overline{\bfF}_2 \end{pmatrix} 
\notag \\ &
=
\begin{pmatrix} 0  \\  
 \overline{\bfF}_1 - \ahom \overline{\bfF}_2  \end{pmatrix}
\,.
\end{align*}
We thus set
\begin{equation*}  
\begin{pmatrix} \nabla \psi_{\f}(\cdot,U)  \\ \a \nabla \psi_{\f}(\cdot,U) \end{pmatrix} 
:=
\begin{pmatrix} \nabla v_{\f}(\cdot,U,\hat{P})  \\ \a \nabla v_{\f}(\cdot,U,\hat{P} ) \end{pmatrix} 
\qand
\overline{\f}  := \frac12 ( \overline{\bfF}_1 - \ahom\overline{\bfF}_2)
\,, 
\end{equation*}
and obtain, in view of Theorem~\ref{t.subadd.converge.nosymm} and
Corollary~\ref{c.Jf.minsetE}, the estimate in~\eqref{e.zeroslope.sublin}.

\smallskip

The rest of the proof is devoted to showing the estimate for the energy of~$\psi_{\f}$. 

\smallskip

\emph{Step 2.} We show the identity
\begin{equation} 
\label{e.energy.psi.f.pre}
\bigl\| \s^{\nicefrac 12} \nabla \psi_{\f}(\cdot,U) \bigr\|_{\underline{L}^2(U)}^2
= 
\begin{pmatrix} -\overline{\bfF}_2 \\ \khom \overline{\bfF}_2\end{pmatrix} \cdot \bfA(U) \begin{pmatrix} - \overline{\bfF}_2 \\ \khom \overline{\bfF}_2\end{pmatrix}
-  \mu_{\f}(U,0) 
\,.
\end{equation}
Using~\eqref{e.decomp.splitten},~\eqref{e.Z.zero.split} and~\eqref{e.maximizers.J.to.bfJ}, we write~$\psi_{\f}(\cdot,U) = \hat{\psi}_0 + \tilde{\psi}_{\f}$ with
\begin{equation*} 
\hat{\psi}_0:= v(\cdot,U,-\overline{\bfF}_2,-\khom \overline{\bfF}_2 )\,,  \quad
\tilde{\psi}_{\f} := v_{\f}(\cdot,U,0) \qand 
\tilde{\psi}_{\f}^* := v_{\f}^*(\cdot,U,0)
\,.
\end{equation*}
Notice that~$\hat{\psi}_0 \in \mathcal{A}(U)$ and
\begin{equation*} 
Z_{\f}(\cdot,U,0) 
: =
\begin{pmatrix} 
\nabla \tilde{\psi}_{\f}  + \nabla \tilde{\psi}_{\f}^* 
\\ 
 \a \nabla \tilde{\psi}_{\f}
- \a^t \nabla \tilde{\psi}_{\f}^*
 \end{pmatrix} 
 \in
\Lpoto(U) \times  \Lsolo(U)\,.
\end{equation*}
We first claim that the energy of~$\psi_{\f}(\cdot,U)$ splits as follows:
\begin{equation} 
\label{e.split.the.energy.with.f}
\fint_{U} \s \nabla \psi_{\f}(\cdot,U) \cdot \nabla \psi_{\f}(\cdot,U)
=
\fint_{U} \s \nabla \hat{\psi}_0 \cdot \nabla \hat{\psi}_0 
+
\fint_{U} \s \nabla \tilde{\psi}_{\f} \cdot \nabla \tilde{\psi}_{\f}
\,.
\end{equation}
To see this, by~$\hat{\psi}_0 \in \mathcal{A}(U)$,~$\nabla \tilde{\psi}_{\f} + \nabla \tilde{\psi}_{\f}^* \in \Lpoto(U)$ and~$\a \nabla \tilde{\psi}_{\f} - \a^t \nabla \tilde{\psi}_{\f}^* \in \Lsolo(U)$, we have   
\begin{equation*} 
\fint_{U}\a \nabla \hat{\psi}_0 \cdot \nabla \tilde{\psi}_{\f} 
= 
- \fint_{U}\a \nabla \hat{\psi}_0 \cdot \nabla \tilde{\psi}_{\f}^* 
=
- \fint_{U} \nabla \hat{\psi}_0 \cdot \a^t \nabla \tilde{\psi}_{\f}^*
=
- \fint_{U} \nabla \hat{\psi}_0 \cdot \a \nabla \tilde{\psi}_{\f}
\,.
\end{equation*}
This implies that
\begin{equation*} 
\fint_{U} \s \nabla \hat{\psi}_0 \cdot \nabla \tilde{\psi}_{\f} = \frac12 \fint_{U} (\a + \a^t) \nabla \hat{\psi}_0 \cdot \nabla \tilde{\psi}_{\f}
=  0
\,,
\end{equation*}
and thus~\eqref{e.split.the.energy.with.f} follows. Furthermore, by~$\a\nabla \tilde{\psi}_{\f}- \a^t \nabla \tilde{\psi}_{\f}^*  \in \Lsolo(U)$, we get
\begin{equation} 
\label{e.v.vstar.crossterm}
\fint_{U} \a \nabla \tilde{\psi}_{\f} \cdot \nabla \tilde{\psi}_{\f}
= \fint_{U} \a^t \nabla \tilde{\psi}_{\f}^* \cdot \nabla \tilde{\psi}_{\f}
= \fint_{U} \a \nabla \tilde{\psi}_{\f} \cdot \nabla \tilde{\psi}_{\f}^*
= \fint_{U} \a^t \nabla \tilde{\psi}_{\f}^* \cdot \nabla \tilde{\psi}_{\f}^* \,.
\end{equation}
Consequently, 
\begin{equation} 
\label{e.energy.tilde.psi.f}
- \mu_{\f}(U,0) 
= \frac12 \fint_{U} Z_{\f}(\cdot,U,0) \cdot \bfA Z_{\f}(\cdot,U,0)
=
\fint_{U} \a \nabla \tilde{\psi}_{\f} \cdot \nabla \tilde{\psi}_{\f}
=
 \fint_{U} \a^t \nabla \tilde{\psi}_{\f}^* \cdot \nabla \tilde{\psi}_{\f}^* 
\,.
\end{equation}
By~\eqref{e.J.by.means.of.bfA},~\eqref{e.Jenergyv.nosymm} and the antisymmetry of~$\khom$, we have
\begin{equation} 
\label{e.energy.psi.zero}
\fint_{U} \frac12 \s \nabla \hat{\psi}_0 \cdot \nabla \hat{\psi}_0 
= J(U,-\overline{\bfF}_2, - \khom \overline{\bfF}_2)
=
\frac12 \begin{pmatrix} -\overline{\bfF}_2 \\ \khom \overline{\bfF}_2\end{pmatrix} \cdot \bfA(U) \begin{pmatrix} -\overline{\bfF}_2 \\ \khom \overline{\bfF}_2\end{pmatrix}
\,.
\end{equation}
Combining this with~\eqref{e.split.the.energy.with.f} and~\eqref{e.v.vstar.crossterm} yields~\eqref{e.energy.psi.f.pre}.

\smallskip

\emph{Step 3.} We next show that
\begin{align} 
\label{e.additivity.psi.f.E}
\lefteqn{
\avsum_{z \in 3^n \Zd \cap \cu_m} \E \Bigl[
\bigl\| \s^{\nicefrac12} \bigl(
\nabla \psi_{\f} (\cdot,\cu_m)-
\nabla \psi_{\f} (\cdot,z+\cu_n) \bigr) \bigr\|_{\underline{L}^2(z+\cu_n)}^2 
 \Bigr]
} \qquad \quad &
\notag \\ &
\leq  2 \E\bigl[ \mu_{\f}(\cu_n,0) - \mu_{\f}(\cu_m,0)  \bigr]
+  \hat{P} \cdot \bigl( \bfAhom(\cu_n) - 
\bfAhom(\cu_m)  \bigr)\hat{P} 
\,.
\end{align}
For~$P \in \R^{2d}$, set~$Z := Z_{\f}(\cdot,\cu_m,P)$ and~$Z_{z,n}:= Z_{\f}(\cdot,z+\cu_n,P)$, and let~$\tilde Z$ be the extension such that~$\tilde Z:= Z_{z,n}$ in~$z+\cu_n$ for~$z \in3^n \Zd$. Since~$Z_{z,n}  \in P + \Lpoto(z+\cu_n) \times \Lsolo(z+\cu_n)$, the extension~$\tilde Z$  belongs to~$P  + \Lpoto(\cu_m) \times \Lsolo(\cu_m)$ as well. We now have
\begin{align*} 
\fint_{\cu_m} \bfA Z \cdot \tilde Z 
& 
=
\avsum_{z \in 3^n \Zd \cap \cu_m} \fint_{z+\cu_n} \bfA (Z - Z_{z,n})  \cdot Z_{z,n} + \avsum_{z \in 3^n \Zd \cap \cu_m}\fint_{z+\cu_n}   \bfA Z_{z,n}  \cdot Z_{z,n} 
\notag \\ &
=
\avsum_{z \in 3^n \Zd \cap \cu_m} P\cdot \fint_{z+\cu_n} \bfA (Z - Z_{z,n})  
- 2\avsum_{z \in 3^n \Zd \cap \cu_m}  \mu_{\f}(z+\cu_n,P)
\notag \\ &
=
P\cdot \biggl(\bfF(\cu_m) - \avsum_{z \in 3^n \Zd \cap \cu_m} \bfF(z+\cu_n) \biggr)
- 2\avsum_{z \in 3^n \Zd \cap \cu_m}  \mu_{\f}(z+\cu_n,P)
\,.
\end{align*}
By~\eqref{e.mu.f.mustar.f.formula}, 
\begin{equation*} 
 \mu_{\f}(U,P) = 
\frac12 P \cdot \bfA(U) \cdot P
+ P \cdot   \bfF(U)  
 + \mu_{\f}(U,0)\,.
\end{equation*}
Therefore,
\begin{multline} 
\label{e.quadratic.response.Z.f}
\avsum_{z \in 3^n \Zd \cap \cu_m} \fint_{z+\cu_n} \frac12  \bfA \bigl(Z_{\f}(\cdot,\cu_m,P) - Z_{\f}(\cdot,z+\cu_n,P)\bigr) \cdot  \bigl(Z_{\f}(\cdot,\cu_m,P) - Z_{\f}(\cdot,z+\cu_n,P)\bigr)
\\
= 
\avsum_{z \in 3^n \Zd \cap \cu_m}  \biggl( \mu_{\f}(z+\cu_n,0) - \mu_{\f}(\cu_m,0) +  \frac12 P\cdot\bigl(\bfA(z+\cu_n) - \bfA(\cu_m)\bigr) P \biggr)
\,,
\end{multline}
from which~\eqref{e.additivity.psi.f.E} follows after taking expectation and using~$\Zd$-stationarity. 

\smallskip

\emph{Step 4.} 
We claim that there exists~$C(d)<\infty$ such that, for every~$\ep \in (0,1]$,  
\begin{align} 
\label{e.divcurl.for.psi.f}
\lefteqn{
\biggl| \E\biggl[ \fint_{\cu_m} \bigl( \a \nabla \psi_{\f} (\cdot,\cu_m) + \f \bigr) \cdot \nabla \psi_{\f}(\cdot,\cu_m) \biggr]\biggr|
} \quad &
\notag \\ & 
= 
C 3^{-m} \E\Bigl[ \bigl[ \nabla \psi_{\f}(\cdot,\cu_m)  \bigr]_{\Hminusul(\cu_m)}^2\Bigr]^{\nicefrac12} \E\Bigl[\bigl\| 
\a \nabla \psi_{\f} (\cdot,\cu_m) + \f \bigr\|_{\underline{L}^2(\cu_m)}^2\Bigr]^{\nicefrac12}
\notag \\ & \qquad
+ (1+\ep^{-1}) 
\avsum_{z \in 3^n \Zd\cap \cu_m} \E \Bigl[
\bigl\| \s^{\nicefrac12} \bigl(
\nabla \psi_{\f} (\cdot,\cu_m)-
\nabla \psi_{\f} (\cdot,z+\cu_n) \bigr) \bigr\|_{\underline{L}^2(z+\cu_n)}^2 
 \Bigr]
\notag \\ & \qquad
+ C (\ep + 3^{-(m-n)}) \E \Bigl[
\bigl\| \s^{\nicefrac12} \nabla \psi_{\f}(\cdot,\cu_m) \bigr\|_{\underline{L}^2(\cu_m)}^2 
+
\bigl\| \s^{-\nicefrac12} \f \bigr\|_{\underline{L}^2(\cu_m)}^2 
 \Bigr]
\,.
\end{align}
Let~$\varphi \in C_{\mathrm{c}}^\infty(\cu_m)$ be such that~$0 \leq \varphi \leq 2$,~$(\varphi)_{\cu_m} = 1$ and~$\| \nabla \varphi\|_{L^\infty(\cu_m)} \leq 3^{2-m}$. We decompose as
\begin{align*} 
\lefteqn{
\fint_{\cu_m} \bigl( \a \nabla \psi_{\f} (\cdot,\cu_m) + \f \bigr) \cdot \nabla \psi_{\f}(\cdot,\cu_m)
} \quad &
\notag \\ & 
= 
\fint_{\cu_m} \bigl( \a \nabla \psi_{\f} (\cdot,\cu_m) + \f \bigr) \cdot \nabla \psi_{\f}(\cdot,\cu_m) \varphi
+
\fint_{\cu_m} \bigl( \s \nabla \psi_{\f} (\cdot,\cu_m) + \f \bigr) \cdot \nabla \psi_{\f}(\cdot,\cu_m) (1-\varphi)
\,.
\end{align*}
For the first term on the right, we use the fact that~$\a \nabla \psi_{\f} (\cdot,\cu_m) + \f$ is divergence-free, and thus get by integration by parts that
\begin{align*} 
\lefteqn{
\biggl| \fint_{\cu_m} \bigl( \a \nabla \psi_{\f} (\cdot,\cu_m) + \f \bigr) \cdot \nabla \psi_{\f}(\cdot,\cu_m) \varphi  \biggr|
} \qquad & 
\notag \\  &
=
\biggl| \fint_{\cu_m} \bigl( \psi_{\f}(\cdot,\cu_m) -  ( \psi_{\f}(\cdot,\cu_m))_{\cu_m} \bigr)   \bigl( \a \nabla \psi_{\f} (\cdot,\cu_m) + \f \bigr) \cdot \nabla \varphi 
\biggr|
\notag \\  & 
\leq 
C 3^{-m} \bigl[ \nabla \psi_{\f}(\cdot,\cu_m)  \bigr]_{\Hminusul(\cu_m)} \bigl\| 
\a \nabla \psi_{\f} (\cdot,\cu_m) + \f \bigr\|_{\underline{L}^2(\cu_m)} 
\,.
\end{align*}
We next replace the cut-off function~$\varphi$ by its mean in the small cube~$z+\cu_n$. The error is very small:
\begin{align*} 
\lefteqn{
\biggl| 
\fint_{\cu_m} \bigl( \s \nabla \psi_{\f} (\cdot,\cu_m) + \f \bigr) \cdot \nabla \psi_{\f}(\cdot,\cu_m) \bigl( \varphi - (\varphi)_{z+\cu_n} \bigr)
\biggr|
} \qquad &
\notag \\ &
\leq C 3^{-(m-n)} 
\bigl( \bigl\| 
\s^{\nicefrac12} \nabla \psi_{\f} (\cdot,\cu_m) \bigr\|_{\underline{L}^2(z+\cu_n)}  
+
\bigl\| 
\s^{-\nicefrac12} \f \bigr\|_{\underline{L}^2(z+\cu_n)} 
\bigr)
\bigl\| \s^{\nicefrac12} \nabla \psi_{\f} (\cdot,\cu_m) \bigr\|_{\underline{L}^2(\cu_m)} 
\,.
\end{align*}
In each of the small cubes, we localize the integral using 
\begin{align} 
\label{e.divcurl.intermediate.pre}
\lefteqn{
\biggl| \fint_{z+\cu_n} \bigl( \s \nabla \psi_{\f} (\cdot,\cu_m) + \f \bigr) \cdot \nabla \psi_{\f}(\cdot,\cu_m) - 
\fint_{z+\cu_n}  \bigl( \s \nabla \psi_{\f} (\cdot,z+\cu_n) + \f \bigr) \cdot \nabla \psi_{\f}(\cdot,z+\cu_n) \biggr|
} \qquad &
\notag \\ &
\leq 
\bigl( \bigl\| 
\s^{\nicefrac12} \nabla \psi_{\f} (\cdot,\cu_m) \bigr\|_{\underline{L}^2(z+\cu_n)}   + \bigl\| \s^{\nicefrac12}
\nabla \psi_{\f} (\cdot,z+\cu_n) \bigr\|_{\underline{L}^2(z+\cu_n)}  
+
\bigl\| 
\s^{-\nicefrac12} \f \bigr\|_{\underline{L}^2(z+\cu_n)} 
 \bigr) 
\notag \\ & \qquad \cdot
\bigl\| \s^{\nicefrac12} \bigl(
\nabla \psi_{\f} (\cdot,\cu_m)-
\nabla \psi_{\f} (\cdot,z+\cu_n) \bigr) \bigr\|_{\underline{L}^2(z+\cu_n)} 
\,.
\end{align}
The contribution of the second term on the left vanishes in expectation by~$\Zd$ -stationarity and~$(\varphi)_{\cu_m} = 1$: 
\begin{align*} 
\lefteqn{
\E\biggl[ \avsum_{z \in 3^n \Zd \cap \cu_m} (1-(\varphi)_{z + \cu_n}) 
\fint_{z+\cu_n}  \bigl( \s \nabla \psi_{\f} (\cdot,z+\cu_n) + \f \bigr) \cdot \nabla \psi_{\f}(\cdot,z+\cu_n) \biggr]
} \qquad &
\notag \\ &
= 
\biggl( 
\avsum_{z \in 3^n \Zd \cap \cu_m} (1-(\varphi)_{z + \cu_n})  \biggr)
\E\biggl[ \fint_{\cu_n}  \bigl( \s \nabla \psi_{\f} (\cdot,\cu_n) + \f \bigr) \cdot \nabla \psi_{\f}(\cdot,\cu_n) \biggr]
= 0
\,.
\end{align*}
It follows by H\"older's inequality and the triangle inequality that
\begin{align*} 
\lefteqn{
\E\biggl[ \avsum_{z \in 3^n \Zd \cap \cu_m} (1-(\varphi)_{z + \cu_n}) 
\fint_{z+\cu_n} \bigl( \s \nabla \psi_{\f} (\cdot,\cu_m) + \f \bigr) \cdot \nabla \psi_{\f}(\cdot,\cu_m) \biggr]
} \qquad &
\notag \\ &
\leq  \avsum_{z \in 3^n \Zd\cap \cu_m} \E \Bigl[
\bigl\| \s^{\nicefrac12} \bigl(
\nabla \psi_{\f} (\cdot,\cu_m)-
\nabla \psi_{\f} (\cdot,z+\cu_n) \bigr) \bigr\|_{\underline{L}^2(z+\cu_n)} 
 \Bigr]^2 
 \notag \\ & \qquad 
 + 
4 \Bigl( \E \Bigl[
\bigl\| \s^{\nicefrac12} \nabla \psi_{\f}(\cdot,\cu_m) \bigr\|_{\underline{L}^2(\cu_m)}^2 
+
\bigl\| \s^{-\nicefrac12} \f \bigr\|_{\underline{L}^2(\cu_m)}^2 
\Bigr]^2   \Bigr)^{\nicefrac12} 
\notag \\ & \qquad \qquad
\times
\biggl( \avsum_{z \in 3^n \Zd\cap \cu_m} \E \Bigl[
\bigl\| \s^{\nicefrac12} \bigl(
\nabla \psi_{\f} (\cdot,\cu_m)-
\nabla \psi_{\f} (\cdot,z+\cu_n) \bigr) \bigr\|_{\underline{L}^2(z+\cu_n)}^2 
 \Bigr] \biggr)^{\!\nicefrac12}
\,.
\end{align*}
Collecting the above estimates, we obtain~\eqref{e.divcurl.for.psi.f} by Young's inequality. Before we proceed to the conclusion, let us make a note of a quenched implication of~\eqref{e.divcurl.intermediate.pre} and~\eqref{e.quadratic.response.Z.f}. Indeed, by~\eqref{e.divcurl.intermediate.pre},~\eqref{e.quadratic.response.Z.f} and~\eqref{e.energy.estimate.Z}, there exists~$C(d,\lambda,\Lambda)<\infty$ such that
\begin{align} 
\label{e.divcurl.intermediate}
\lefteqn{
\biggl| \fint_{\cu_m} \bigl( \s \nabla \psi_{\f} (\cdot,\cu_m) + \f \bigr) \cdot \nabla \psi_{\f}(\cdot,\cu_m) 
} \quad &
\notag \\ &
- 
\avsum_{z \in 3^n \Zd \cap \cu_m} \fint_{z+\cu_n}  \bigl( \s \nabla \psi_{\f} (\cdot,z+\cu_n) + \f \bigr) \cdot \nabla \psi_{\f}(\cdot,z+\cu_n) \biggr|
\notag \\ & \quad 
\leq 
C \Bigl( \bigl\| \f \|_{\underline{L}^2(\cu_m)} + |\hat P| \Bigr)
\biggl( \avsum_{z \in 3^n \Zd \cap \cu_m}  \bigl( \mu_{\f}(z+\cu_n,0) - \mu_{\f}(\cu_m,0)  \bigr) \biggr)^{\!\nicefrac12}
\notag \\ & \quad \qquad 
+
C \Bigl( \bigl\| \f \|_{\underline{L}^2(\cu_m)} + |\hat P| \Bigr)
\biggl( \avsum_{z \in 3^n \Zd \cap \cu_m}  \hat P \cdot \bigl( \bfA(z+\cu_n,0) - \bfA(\cu_m,0) \bigr)\hat P \biggr)^{\!\nicefrac12}
\,.
\end{align}

\smallskip

\emph{Step 5.} Conclusion. We combine~\eqref{e.divcurl.for.psi.f} and~\eqref{e.additivity.psi.f.E} to obtain 
\begin{align*} 
\lefteqn{
\biggl| \E\biggl[ \fint_{\cu_m} \bigl( \a \nabla \psi_{\f} (\cdot,\cu_m) + \f \bigr) \cdot \nabla \psi_{\f}(\cdot,\cu_m) \biggr]\biggr|
} \qquad &
\notag \\ &
\leq
C 
3^{-m} \E\Bigl[ \bigl[ \nabla \psi_{\f}(\cdot,\cu_m)  \bigr]_{\Hminusul(\cu_m)}^2\Bigr]^{\nicefrac12} \E \bigl[ \| \f \|_{L^2(\cu_0)}^2 \bigr]^{\nicefrac12}
+
C (\ep + 3^{-(m-n)}) \E \bigl[ \| \f \|_{L^2(\cu_0)}^2 \bigr]
 \Bigr]
 \notag \\ & \qquad 
+
 C\ep^{-1} \bigl( \E\bigl[ \mu_{\f}(\cu_n,0) - \mu_{\f}(\cu_m,0)  \bigr]
+  \hat{P} \cdot \bigl( \bfAhom(\cu_n) - 
\bfAhom(\cu_m)  \bigr)\hat{P} \bigr)
\,.
\end{align*}
By~\eqref{e.mu.f.subadd} we deduce that
\begin{align*} 
0 
& 
\leq 
\!\!\! \avsum_{z \in 3^n \Z^d \cap \cu_m}  \mu_{\f}(z{+}\cu_n,0)  
- \mu_{\f}(\cu_m,0) 
\leq 
\!\!\! \avsum_{z \in 3^n \Z^d \cap \cu_m}   
\bfJ_{\f}\bigl(z{+}\cu_n,0, \overline{\bfF} \bigr)
-
 \bfJ_{\f}\bigl(\cu_m,0, \overline{\bfF} \bigr)
 \,.
\end{align*}
Therefore, by taking expectation and using Lemma~\ref{l.f.J.minimal.expectation}, we obtain, for every~$m,n\in\N$ with~$m \geq n$, 
\begin{equation*} 
0\leq \E\bigl[ \mu_{\f}(\cu_n,0) \bigr] - \E\bigl[ \mu_{\f}(\cu_m,0) \bigr] \leq C 3^{-n \alpha}
 \,.
\end{equation*}
Thus~$\{\mu_{\f}(\cu_n,0) \bigr]\}_{n \in \N}$ is a Cauchy sequence and there exists~$\overline{\mu}_{\f} \in \R$ such that
\begin{equation*} 
\bigl| \E\bigl[ \mu_{\f}(\cu_m,0) \bigr]  - \overline{\mu}_{\f} \bigr| \leq   C 3^{-m \alpha}
 \,.
\end{equation*} 
We then deduce by Proposition~\ref{p.algebraic.nosymm}, Lemma~\ref{l.f.J.minimal.expectation} and~\eqref{e.zeroslope.sublin} that 
\begin{equation*} 
\biggl| \E\biggl[ \fint_{\cu_m} \bigl( \a \nabla \psi_{\f} (\cdot,\cu_m) + \f \bigr) \cdot \nabla \psi_{\f}(\cdot,\cu_m) \biggr]\biggr|
\leq   C 3^{-m \alpha}
\,.
\end{equation*}
Furthermore, in view of~\eqref{e.energy.psi.f.pre}, we set
\begin{equation*} 
\overline{\mu} := 
\begin{pmatrix} -\overline{\bfF}_2 \\ \khom \overline{\bfF}_2\end{pmatrix} \cdot \bfAhom \begin{pmatrix} - \overline{\bfF}_2 \\ \khom \overline{\bfF}_2\end{pmatrix}
-  \overline{\mu}_{\f} 
 \,.
\end{equation*}
Using then~\eqref{e.divcurl.intermediate} and the assumption that~$\P$ satisfies~\eqref{e.f.sec.ass}, we may proceed as in the proof of Corollary~\ref{c.Jf.minsetE} to find a minimal scale so that the quenched estimate~\eqref{e.zeroslope.energy} is valid. This completes the proof. 
\end{proof}

\begin{remark} \label{r.zeroslope}
Similar to~Lemma~\ref{l.localX}, 
Theorem~\ref{t.zeroslope} has a localized version in which the random variable~$\X$ has local dependence. For this, we may use Corollary~\ref{c.Jf.minsetE} and Lemma~\ref{l.f.Hminusone.zero}.
\end{remark}

\subsection*{Historical remarks and further reading}

The non-standard variational principle has an interesting history. It has been independently discovered many separate times, and in this process of writing this we have found several more examples. We refer to~\cite{AM} and\cite[Chapter 10]{AKMBook} for this background. In the context of homogenization, it was used in~\cite{FP1,GMS,ADMZ} prior to be being used as the basis for the quantitative renormalization argument in~\cite{AM}. 

The coarse-graining estimates Section~\ref{ss.coarse.graining.inequalities} first appeared, in a form close to the one stated here, in~\cite{AK.HC}. 

The results in Section~\ref{ss.rhs}, for  equations with a divergence-form right-hand side, was implicitly treated in~\cite{AM} as that paper covered much more general nonlinear operators.


\section{Large-scale regularity theory}
\label{s.regularity}

In this chapter, we revisit the large-scale regularity theory developed in Section~\ref{ss.reg}. There we found a minimal scale~$\X$ which satisfies merely~$\P[\X<\infty] = 1$, such that~$C^{0,1}$ and~$C^{1,1-}$ regularity holds (with respect to balls centered at the origin) on scales larger than~$\X$: see the statement of Theorem~\ref{t.C11.sharp}. 
We will improve on this result in two main ways: first, by estimating the stochastic moments of the minimal scale~$\X$ in Theorem~\ref{t.C11.sharp}, and second, by bootstrapping the regularity from~$C^{1,1-}$ to~$C^\infty$-type regularity. 

\smallskip

Recall that the proof of the large-scale~$C^{1,1-}$ is based on harmonic approximation of the solutions on large scales by homogenization. The minimal scale~$\X$ is defined in the proof (see~\eqref{e.XC11}) as the largest scale at which the first-order correctors do not satisfy a sublinearity estimate. In Chapter~\ref{s.qualitativetheory}, with only qualitative homogenization at our disposal, we could not give a quantitative bound on this random variable. However, the estimates in the previous two chapters have put us in a position to estimate~$\X$ sharply.

\begin{theorem}
\label{t.optimalstochasticintegrability}
Suppose that~$\P$ satisfies~$\CFS(\beta,\Psi)$ and let~$\delta\in(0,1)$. Then there exist constants~$C(\delta,\dataref)<\infty$,~$\alpha(\beta,d,\lambda,\Lambda) \in (0,1)$  and a random variable satisfying   
\begin{equation} \label{e.X.integrability}
\X^{\frac d2(1-\beta)} = \O_\Psi(C)\,,
\end{equation}
such that, for every~$e \in B_1$, the global corrector solving the equation
\begin{equation*}  
- \nabla \cdot \a (e + \nabla \phi_e) = 0 \quad \mbox{in } \R^d
\end{equation*}
satisfies, for every~$r \geq \X$, the estimate
\begin{equation} \label{e.optimalstochasticintegrability}
\frac1r \|\phi_e - (\phi_e)_{B_r} \|_{\underline{L}^2(B_r)}
+
\frac1r \|\nabla \phi_e \|_{\Hminusul(B_r)} 
+ 
\frac1r \|\a(e + \nabla \phi_e) - \ahom e \|_{\Hminusul(B_r)} 
\leq 
\delta \Bigl(\frac{r}{\X} \Bigr)^{\!-\alpha} \,.
\end{equation}
Moreover, the random variable in Theorem~\ref{t.C11.sharp} also satisfies~\eqref{e.X.integrability}, with the constant~$C$ in~\eqref{e.X.integrability} depending additionally on~$\gamma \in [\tfrac12,1)$. 
\end{theorem}

The exponent~$\alpha(\beta,d,\lambda,\Lambda)>0$ is equal to~$\nicefrac \theta 2$, where~$\theta$ is the exponent appearing in the statement of Lemma~\ref{l.mathcalE.minscale}. If this exponent~$\theta$ can be improved (as it will be in the next section), then the exponent~$\alpha$ in Theorem~\ref{t.optimalstochasticintegrability} also improves. 

\smallskip

The second improvement to the regularity we make is a higher-order extension of  Theorem~\ref{t.C11.sharp}. We identify a finite-dimensional subspace of ``heterogeneous polynomials" and show that solutions can be approximated by elements of this subspace up to an error consistent with a~$C^{k,1}$ approximation. The proof consists of two parts. The first is a Liouville-type result for heterogeneous solutions. Let us motivate this by considering harmonic functions. It is a classical result that a harmonic function growing polynomially at infinity is necessarily a harmonic polynomial. In particular, the dimension of such harmonic functions is finite and can be explicitly computed. Having this in mind, we define, for every~$k \in \N$, 
\begin{equation}  
\label{e.Ak.def}
\mathcal{A}_k  = \big\{ u \in H^1_{\mathrm{loc}}(\R^d) \, : \, -\nabla \cdot \a \nabla u = 0, \; \lim_{r\to \infty} r^{-k-1} \left\| u \right\|_{\underline{L}^2 \left( B_{r} \right)} = 0  \big\} \,.
\end{equation}
In Theorem~\ref{t.Ck1} below, we show that such a space is isomorphic with the space of~$\ahom$-harmonic polynomials and, in particular, has the same dimension. Having defined~$\mathcal{A}_k$, we have the second analogy to the theory of harmonic functions. It is easy to show that the Taylor series of a harmonic function consists of homogeneous harmonic polynomials, and therefore, the space of harmonic polynomials tracks down the regularity of any harmonic function. It turns out that the same is true for heterogeneous solutions as well, but now the harmonic polynomials are switched to the members of~$\mathcal{A}_k$. The statement is as follows. For every~$k \in \N$ there exists a constant~$C(k,\dataref)<\infty$ such that, for every~$R\geq \X$ and~$u\in H^1(B_R)$ satisfying~$-\nabla \cdot \a\nabla u = 0$ in~$B_R$, we have the existence of~$\psi \in \mathcal{A}_k$ such that, for every~$r \in [\X,R]$, 
\begin{equation*}  
\| u - \psi \|_{\underline{L}^2(B_r)} \leq C \Bigl(\frac rR \Bigr)^{k+1} \| u \|_{\underline{L}^2(B_r)}  \,.
\end{equation*}

The arguments here essentially follow the ones in~\cite[Chapter 3]{AKMBook}. We emphasize that the arguments of the latter are \emph{purely deterministic}; in other words, no mixing assumption is ever used in that chapter. The only stochastic ingredient is the result on the convergence of the subadditive quantities, which we have already generalized to our mixing condition in the previous section.

\subsection{Large-scale regularity with optimal stochastic integrability}
\label{ss.C01}

In this section, we give the proof of Theorem~\ref{t.optimalstochasticintegrability}. At first glance, the statement appears to be an immediate consequence of similar-looking estimates~\eqref{e.FCV.nonsymm.grad} and~\eqref{e.FCV.nonsymm.flux} for the finite-volume correctors, which we recall here: for every~$|e|=1$ and~$m\in\N$, 
\begin{equation}
\label{e.thm61.almost}
3^{-m} \| \phi_{m,e} \|_{\underline{L}^2(\cu_m)} 
+
3^{-m} \| \a \nabla\phi_{m,e} - \ahom e \|_{\Hminusul(\cu_m)} 
\leq
C3^{-m} 
+ 
C\mathcal{E}(m)
\,.
\end{equation}
This estimate, combined with Theorem~\ref{t.subadd.converge.nosymm}, would imply Theorem~\ref{t.optimalstochasticintegrability} if we could replace the~$\phi_{m,e}$ on the left side by the infinite volume-corrector~$\phi_{e}$. 

\smallskip

We will indeed obtain Theorem~\ref{t.optimalstochasticintegrability} from~\eqref{e.thm61.almost}, but the implication is not so immediate, because the domain~$\cu_m$ on the left side of~\eqref{e.thm61.almost} is also changing in the parameter~$m$. What is missing is a way to ``localize'' the estimate by obtaining some spatial uniformity. Precisely, we need to estimate the difference~$\phi_{m,e} - \phi_{m-1,e}$ on a cube~$\cu_n$ for~$n$ possibly much smaller than~$n$, and hope that the error is summable in~$m$, uniformly in~$n$. 

\smallskip

This is a consequence of the more general \emph{large-scale~$C^{0,1}$ estimate}, already proved in a qualitative form in~\eqref{e.C01}. We will first re-prove this result in a quantitative form, using the bound we have for the finite-volume correctors.

\begin{theorem}[{Large-scale~$C^{0,1}$-regularity}]
\label{t.C01}
Assume that~$\P$ is a~$\Zd$--stationary measure on~$(\Omega,\F)$ and satisfies~$\CFS(\beta,\Psi)$. 
There exist constants~$C(\dataref)<\infty$ and~$C'(d,\lambda,\Lambda)$ and a random variable~$\X$ (called the ``minimal scale'') which satisfies
\begin{equation}
\label{e.XLipschitz}
\X^{\frac d2(1-\beta)} 
= 
\O_{\Psi}(C) 
\end{equation}
such that, for every~$R\geq \X$,~$f\in H^{-1}(B_R)$ and weak solution~$u\in H^1(B_{R})$ of
\begin{equation} 
\label{e.wksrt2}
-\nabla \cdot \left( \a\nabla u \right) = f \quad \mbox{in} \ B_{R},
\end{equation}
we have, for every~$r \in [\X, \frac12 R]$, 
\begin{equation}
\label{e.Lipschitz}
\sup_{ t \in \left[r , \frac12 R\right]} \bigl\| \nabla u \bigr\|_{\underline{L}^2(B_{t})}
\leq
\frac{C'}{R} \bigl\| u - ( u )_{B_{R}} \bigr\|_{\underline{L}^2(B_{R})} 
+ C_d \int_{r}^R \| f \|_{\underline{H}^{-1}(B_t)} \, \frac{dt}{t}
\,.
\end{equation}
\end{theorem}

As explained in Section~\ref{ss.reg}, the estimate~\eqref{e.Lipschitz} is a generalization of the pointwise gradient bound for harmonic functions, which states that a harmonic function~$u$ in~$B_R$ satisfies
\begin{equation*}
\bigl| \nabla u (0) \bigr| 
\leq 
\frac{C}{R} 
\left\| u - ( u )_{B_{R}} \right\|_{\underline{L}^2(B_{R})} 
\,.
\end{equation*}
Let us comment on the second term on the right side of~\eqref{e.Lipschitz}. 
Recall that a pointwise gradient estimate for a solution of the Poisson equation~$-\Delta u = f$ is valid for~$f\in L^p$ with~$p > d$, and can be written as
\begin{equation}
\label{e.naive.Poisson}
\bigl| \nabla u (0) \bigr| 
\leq 
\frac{C}{R} 
\left\| u - ( u )_{B_{R}} \right\|_{\underline{L}^2(B_{R})} 
+
C R \| f \|_{\underline{L}^p(B_R)} \,,
\end{equation}
with the prefactor constant~$C$ in the second term on the right also depending on~$p$. That the exponent~$d$ should be the critical one for a pointwise gradient estimate is immediate from the fact that the Sobolev embedding~$W^{2,p} \hookrightarrow
W^{1,\infty}$ is valid if and only if~$p>d$. We next observe, using the scaling properties of our volume-normalized norms~\eqref{e.Sobolev.scaling} and H\"older's inequality, that for any~$p\in (d,\infty]$, 
\begin{equation*}
\| f \|_{\underline{H}^{-1}(B_t)}
\leq
t \| f \|_{\underline{L}^{2}(B_t)}
\leq 
t \Bigl( \frac{R}{t} \Bigr)^{\!\nicefrac dp}
\| f \|_{\underline{L}^{p}(B_R)}
\end{equation*}
and therefore
\begin{equation*}
\int_{r}^R \| f \|_{\underline{H}^{-1}(B_t)} \, \frac{dt}{t}
\leq
\| f \|_{\underline{L}^{p}(B_R)} 
\int_{0}^R  \Bigl( \frac{R}{t} \Bigr)^{\!\nicefrac dp} \, dt
=
R \| f \|_{\underline{L}^{p}(B_R)} 
\int_{0}^1 t^{-\nicefrac dp} \, dt
\,.
\end{equation*}
The last integral on the right side is of course finite if and only if~$p>d$, in which case it is equal to~$(1-\nicefrac dp)^{-1}$. Therefore we see that the second term on the right side of~\eqref{e.Lipschitz} can be replaced, in desired, by~$C(1-\nicefrac dp)^{-1} R \| f \|_{\underline{L}^{p}(B_R)}$ in the case~$f\in L^p(B_R)$ for~$p>d$. The form~\eqref{e.Lipschitz} is of course more general, but its scaling is consistent with that of~\eqref{e.naive.Poisson}.

\smallskip

The proof of Theorem~\ref{t.C01} is accomplished by quantitative harmonic approximation following the method introduced in~\cite{AS}. The latter is a quantitative version of an idea introduced earlier by Avellaneda and Lin~\cite{AL1}, who proved a version of Theorem~\ref{t.C01} for periodic coefficient fields by a compactness argument.

\smallskip

The main idea is roughly analogous to the classical proof of the Schauder estimates: on a sequence of balls with geometrically decreasing radii, one replaces the solution~$u$ by the solution~$\overline{u}$ of the homogenized equation with the same boundary values. One then applies the better regularity satisfied by the latter and then passes this back to the original function by the triangle inequality. An error is made on every scale, namely the size~$\left\| u - \overline{u} \right\|$ of the difference between~$u$ and~$\overline{u}$, but one hopes that this is summable across the scales. For the Schauder estimates, one uses the regularity of the coefficients to see that the approximation becomes better on small scales. Here, we use \emph{quantitative homogenization}, in the form of Proposition~\ref{p.DP}, to see that the approximation becomes better on large scales.\footnote{This can be formalized in different ways, but the underlying idea is always the same. For instance, in~\cite{GNO2}, quantitative homogenization is obtained not by control over subadditive quantities but by sublinear estimates on the first-order correctors. These methods are very closely related. Indeed, the optimizing functions for the subadditive quantities are finite-volume versions of the first-order correctors, and we proved Proposition~\ref{p.DP} precisely by using the control over subadditive quantities to gain estimates on their sublinear growth. 
}
We formalize the main idea in the following statement, which is a variant of the one in~\cite[Lemma 5.1]{AS} (see~\cite[Lemma 3.4]{AKMBook} for a nicer presentation).

\begin{lemma}
\label{l.harm.approx}
Fix~$X,R >0$ such that~$R \geq 2X$. Let a constant coefficient field~$\bhom$ satisfy assumptions in~\eqref{e.ue}. There exists constants~$C(d,\lambda,\Lambda)<\infty$ and~$c(d,\lambda,\Lambda) \in (0,1]$ such that if~$\omega,\gamma:[1,\infty) \to (0,\infty)$ are two functions such that 
\begin{equation} 
\label{e.Dini}
\int_{X}^R \omega\Bigl( \frac{r}X \Bigr) \, \frac{dr}{r} \leq c
\quad \mbox{and} \quad \int_{X}^R  \gamma(r)  \, \frac{dr}{r} < \infty,
\end{equation}
and there is a function~$u\in L^2(B_{R})$ having the property that, for every~$r\in \left[ X, \frac12R \right]$, there exists a weak solution~$w_r \in H^1(B_{r})$ of 
\begin{equation*} \label{}
-\nabla \cdot \left( \bhom \nabla w_r \right) = 0 \quad \mbox{in} \ B_{r},
\end{equation*}
such that 
\begin{equation} 
\label{e.wharmapproxit}
\left\| u - w_r \right\|_{\underline{L}^2(B_{r/2})} \leq    
\omega\Bigl( \frac{r}X \Bigr)
\left\| u - \left( u \right)_{B_{r}} \right\|_{\underline{L}^2(B_{r})}  + r \gamma(r) 
\end{equation}
then~$u$ satisfies, for every~$r \geq X$, the estimate
\begin{equation}
\label{e.elipsc}
\sup_{t \in \left[r, R\right]} \,
\frac1{t} \left\| u - ( u )_{B_{t}} \right\|_{\underline{L}^2(B_{t})}
\leq
\frac{C}{R} \left\| u - ( u )_{B_{R}} \right\|_{\underline{L}^2(B_{R})} 
+ C \int_{r}^R \gamma(t) \, \frac{dt}{t}
.
\end{equation}
\end{lemma}

The main difference between the lemma stated above and~\cite[Lemma 5.1]{AS} is that the latter was stated assuming~$\omega(r)= Cr^{-\alpha}$, while here we formalize it with a more general Dini-type modulus~$\omega(\frac \cdot X)$, measured in units of~$X$, to be compatible with sub-algebraic rates of homogenization. Indeed,
the only property of the algebraic modulus~$\omega(r)= Cr^{-\alpha}$ that is used in the proof of~\cite[Lemma 5.1]{AS} is that it satisfies~\eqref{e.Dini}, as already observed in~\cite[Section 3]{ASh} and subsequently in~\cite{FO}. Nevertheless, we will give complete proof here for the reader's convenience, although we postpone it to the end of this section. We show first how to obtain Theorem~\ref{t.C01} from the lemma. 

\smallskip

If we admit Lemma~\ref{l.harm.approx}, the main step remaining in the proof of Theorem~\ref{t.C01} is to check the hypotheses of the former. This is accomplished by combining Proposition~\ref{p.DP} with Lemma~\ref{l.localX} and Remark~\ref{r.localX.nonsymm}, and writing the resulting statement in a slightly different form. The statement is \emph{localized} in the sense that we restrict ourselves to the ball~$B_{3dt}$ and obtain a sensitivity estimate for the minimal scale as in~\eqref{e.XSI.again}. 

\begin{proposition}[Harmonic approximation]
\label{p.harm.we.have}
Assume that~$\P$ is a~$\Zd$--stationary measure on~$(\Omega,\F)$ and satisfies~$\CFS(\beta,\Psi)$. Let~$t \in [1,\infty]$ and~$\delta,\theta \in (0,\frac12)$. Then there exist constants~$C(\delta,\theta,\dataref)<\infty$,~$\alpha(\beta,d,\lambda,\Lambda) \in (0,1]$ and~$\F(B_{3dt})$-measurable random variable 
$\X_t:\Omega \to [1,Ct]$ satisfying
\begin{equation}
\label{e.XSI.again}
\X_t^{\frac d2(1-\beta)} 
= 
\O_{\Psi}(C) 
\qand
\bigl| \partial_{\a(B_{3dt})} \X_t \bigr| \leq C(1 + \X_t) 
\end{equation}
such that if~$\X_t \leq t$, then, for every~$R \in [\X_t,t]$,~$x\in \cu_1$,~$g \in H^{-1}(B_R(x))$ and~$u \in H^1(B_R(x))$ solving~$-\nabla \cdot \a \nabla u = g$ in~$B_R(x)$, there exists~$\bar u  \in \Ahom(B_{(1-\theta)R}(x))$  
satisfying the estimate
\begin{multline}
\label{e.harm.we.have1}
\left\| u - \overline{u}  \right\|_{\underline{L}^2(B_{(1-\theta)R}(x))} 
+
\left\| \nabla u - \nabla \overline{u}  \right\|_{\underline{H}^{-1}(B_{(1-\theta)R}(x))} 
+
\left\| \a \nabla u - \ahom \nabla \overline{u}  \right\|_{\underline{H}^{-1}(B_{(1-\theta)R}(x))} 
\\
\leq 
\delta\Bigl( \frac{R}{\X_t} \Bigr)^{-\alpha} 
\| u - (u)_{B_R(x)}\|_{\underline{L}^{2} \left( B_{R}(x) \right)}
+
CR \left\| g \right\|_{\underline{H}^{-1}(B_R(x))}
\,.
\end{multline}
Moreover,~$t \mapsto \X_t$ is increasing and~$\X := \limsup_{t \to \infty} \X_t = \O_{\Psi}(C)$. 
\end{proposition}

In the above statement, we allow~$x \in \cu_0$ using a minimal scale anchored to the origin. 

\begin{proof}
Fix~$t \in [1,\infty]$ and~$\delta,\theta \in (0,\frac12)$.
Without loss of generality, we may take~$x=0$. 
First, we let~$v \in u + H_0^1(B_R)$ be the solution of~$-\nabla \cdot \left( \a\nabla v \right) = 0$ with the Dirichlet boundary data~$u$. Then, by the Poincar\'e inequality and the basic energy estimate, we get that
\begin{align}  \label{e.v.basicenergyest}
R^{-1} \left\| u - v \right\|_{\underline{L}^{2} \left( B_{R} \right)} 
+
  \left\| \nabla u - \nabla v \right\|_{\underline{L}^{2} \left( B_{R} \right)}
\leq 
C  \left\| g \right\|_{\underline{H}^{-1}(B_R)}
.
\end{align}
We find~$\rho \in ((1-\frac12\theta)R,(1-\frac13\theta)R)$  such that~$\| \nabla v \|_{\underline{L}^2(\partial B_{\rho})} \leq C \theta^{-1} \| \nabla v \|_{\underline{L}^2(B_{(1-\theta)R})}$. Let~$\bar u \in v + H_0^1(B_\rho)$ solve~$\nabla \cdot \ahom \nabla \bar u = 0$ in~$B_\rho$. We can repeat the computation in the proof of Proposition~\ref{p.DP}, using now the localized correctors instead of the global correctors. Letting~$m \in \N$ be such that~$3^{m-1} \leq R < 3^m \leq 3t$, the rescaled version of~\eqref{e.DP} with finite-volume correctors then reads as 
\begin{multline*} 
\frac1\rho\left\| \nabla v - \nabla \overline{u}  \right\|_{\underline{H}^{-1}(B_{\rho})} 
+
\frac1\rho \left\| \a \nabla v - \ahom \nabla \overline{u}  \right\|_{\underline{H}^{-1}(B_{\rho})} 
\\ 
\leq 
C \| \nabla \overline{u} \|_{\underline{L}^2(B_{\rho} \setminus B_{(1-\eta)\rho} ) } 
+ C \eta^{-2} \sup_{|e| \leq 1}\bigl( \| \phi_{m,e} \|_{\underline{L}^2(\cu_m)} 
+
\| \bfs_{m,e} \|_{\underline{L}^2(\cu_m)}  \bigr) 
 \| \nabla \overline{u} \|_{W^{1,\infty}(B_{(1-\eta)\rho} ) }
\,.
\end{multline*}
By Lemma~\ref{l.localX} and Remark~\ref{r.localX.nonsymm}, we find~$\F(B_{3d t})$-measurable random variable~$\X_t$ satisfying~\eqref{e.XSI.again} so that~\eqref{e.FV.sublinearity.nosymm} is valid for~$m\in \N$ with~$\frac13 \X_t \leq 3^m \leq 3 t$ with~$\X_t$ instead of~$\X$. The construction in Lemma~\ref{l.localX} and Remark~\ref{r.localX.nonsymm} also guarantees that the last line of the statement is valid. By the deterministic Lemma~\ref{l.bndrlayer} on the boundary layers and by the Caccioppoli inequality, we obtain, for every~$\eta \in (0,\frac \theta{10})$, 
\begin{equation*}
\| \nabla \bar u \|_{L^2(B_{\rho} \setminus B_{(1-\eta) \rho})} 
\leq 
C \eta | \log \eta| \| \nabla v \|_{\underline{L}^2(\partial B_{\rho})} 
\leq
C \theta^{-\frac d2-2} \eta | \log \eta| R^{-1} \| v - (v)_{B_R} \|_{L^2(B_{R})}\,.
\end{equation*}
On the other hand, the Caccioppoli estimate for~$\overline{u}$ yields
\begin{equation*} 
 \| \nabla \overline{u} \|_{W^{1,\infty}(B_{(1-\eta)\rho} ) }
 \leq 
C  \eta^{- \frac{d}{2} - 2}  \rho^{-2}
 \|  \overline{u} - (\overline{u})_{B_{\rho}}\|_{\underline{L}^{2}(B_{\rho} ) }  
 \,.
\end{equation*}
Combining the above three displays with~\eqref{e.FV.sublinearity.nosymm} and optimizing in~$\eta$ gives us
\begin{align*}
\lefteqn{  
\left\| \nabla v - \nabla \overline{u}  \right\|_{\underline{H}^{-1}(B_{\rho})} 
+
\left\| \a \nabla v - \ahom \nabla \overline{u}  \right\|_{\underline{H}^{-1}(B_{\rho})}} \qquad  
\notag \\ & 
\leq C \theta^{-d^2-16}  \delta^{\frac{1}{d+8}} \Bigl( \frac{R}{\X_t} \Bigr)^{-\frac{\alpha}{d+8}}   \Bigl( \| v - (v)_{B_R} \|_{L^2(B_{R})} + \| v - \overline{u}\|_{\underline{L}^{2}(B_{\rho} ) } \Bigr) \,.
\end{align*}
Shrinking~$\delta$ (which then increases the constant~$C$ in ) and relabeling~$\alpha$, we can reabsorb the last term by~$\| f \|_{\underline{L}^{2}(B_{\rho} ) } \leq C\left\| \nabla f  \right\|_{\underline{H}^{-1}(B_{\rho})}$,  provided that~$R \geq \X_t$. Thus, we obtain~\eqref{e.harm.we.have1} by the above display,~\eqref{e.v.basicenergyest} and the triangle inequality. 
\end{proof}

\begin{proof}[Proof of Theorem~\ref{t.C01}]
The statement follows immediately from Lemma~\ref{l.harm.approx} and Proposition~\ref{p.harm.we.have}. Define
\begin{align*}  
\omega(r) := \delta r ^{\alpha}  
\quad 
\mbox{and}
\quad
\gamma(r) := \left\|  f  \right\|_{\underline{H}^{-1}(B_r)},
\quad r\in (0,\infty),
\end{align*}
where the parameter~$\delta \leq \alpha c$ in Lemma~\ref{p.harm.we.have} is chosen so small that 
\begin{align*}  
\int_{1}^\infty \omega (r) \, \frac{dr}{r}  = \delta \int_{1}^\infty r^{-\alpha} \, \frac{dr}{r} = \frac{\delta}{\alpha} \leq c \,.
\end{align*}
Therefore, with~$\delta\in (0,1)$ fixed, we find~$C<\infty$ and a random variable~$\X$ satisfying~\eqref{e.XLipschitz}, such that  Lemma~\ref{l.harm.approx} is applicable with~$X = \X$ by Proposition~\ref{p.harm.we.have}.  The final ingredient in the proof is the Caccioppoli inequality giving, for every~$r \in (0,R]$, 
\begin{equation*} \label{}
\left\| \nabla u \right\|_{\underline{L}^2(B_{r/2})}
\leq 
\frac C{r} \left\| u - \left( u \right)_{B_{r}} \right\|_{\underline{L}^2(B_{r})} + C  \left\| f \right\|_{\underline{H}^{-1}(B_{r})},
\end{equation*}
which, when combined with~\eqref{e.elipsc}, completes the proof.
\end{proof}

We next present the proof of Lemma~\ref{l.harm.approx}.

\begin{proof}[{Proof of Lemma~\ref{l.harm.approx}}]
We repeat the proof of~\cite[Lemma 3.4]{AKMBook} nearly verbatim. We show that there exists~$C(\omega,d,\lambda,\Lambda)<\infty$ such that
\begin{equation} \label{e.harm.approx.goal}
\mathsf{M}_r 
:= 
\sup_{t \in [r,R]}  \frac 1t \left\| u- (u)_{B_t} \right\|_{\underline{L}^2 \left( B_{t} \right)} 
\leq 
\frac CR \left\| u- (u)_{B_R} \right\|_{\underline{L}^2 \left( B_{R} \right)}  
+
C \int_{ r}^{R}  \gamma(t) \, \frac{dt}{t} 
\,.
\end{equation}
To begin the proof, we let~$\sigma \in (0,\nicefrac12]$ be a constant to be determined below. For every~$r \in (0,\sigma R]$, denote by~$\ell_{r}$ the best affine approximation of~$u$, that is, 
\begin{equation*} 
\inf_{\ell  \; \mathrm{affine}} \left\| u - \ell \right\|_{\underline{L}^2(B_{r})}  = \left\| u - \ell_r \right\|_{\underline{L}^2(B_{r})}
\end{equation*}
and, for~$r \in [\sigma R,R]$, choose it to be the constant function~$\ell_r = (u)_{B_R}$. By the assumption~\eqref{e.wharmapproxit}, we find, for every~$r  \in [X, \tfrac12 R]$, a function~$w_r \in H^1(B_{r})$ solving~$-\nabla \cdot \left( \bhom \nabla w_r \right) = 0$ in~$B_{r/2}$ such that 
\begin{equation} 
\label{e.wharmapproxit.applied}
\frac1r \left\| u - w_r \right\|_{\underline{L}^2(B_{r/2})} \leq    
\omega\Bigl( \frac{r}X \Bigr) \mathsf{M}_r  + \gamma(r)
\,. 
\end{equation}
Since~$w_r$ satisfies the equation with constant coefficients, we obtain, by differentiating the equation~$d$ times and then using the Caccioppoli inequality and Morrey's inequality, that
\begin{equation*}  
\inf_{\ell  \; \mathrm{affine}}  \left\| w_r - \ell \right\|_{\underline{L}^2 \left( B_{ \sigma r} \right)} 
\leq 
C(\sigma r)^2 \left\| \nabla^2 w_r \right\|_{L^{\infty}(B_{r/4})}
\leq
C\sigma^2 \left\| w_r - \ell_{r} \right\|_{\underline{L}^2 \left( B_{r/2} \right)} 
\,.
\end{equation*}
Thus, the triangle inequality and the previous two displays yield 
\begin{equation}  
\label{e.readyforiterz}
\frac1{\sigma r}\left\| u - \ell_{\sigma r} \right\|_{\underline{L}^2 \left( B_{ \sigma r} \right)} 
\leq
C\sigma \frac 1r \left\| u- \ell_r \right\|_{\underline{L}^2 \left( B_{r} \right)} 
+
C\sigma^{-\frac d2 - 1} \left( \omega\Bigl( \frac{r}X \Bigr)  \mathsf{M}_r  + \gamma(r)\right)
\,.
\end{equation}
Taking~$\sigma(d,\lambda,\Lambda): = (2C_{\eqref{e.readyforiterz}})^{-1}$, we obtain by integration, changing the variables and applying the previous display that
\begin{align*} 
\int_{\sigma r}^{\sigma R} 
\frac1{t} \left\| u - \ell_{t} \right\|_{\underline{L}^2 \left( B_{ t} \right)} \, \frac{dt}{t}
& 
=
\int_{r}^{R} 
\frac1{\sigma t} \left\| u - \ell_{\sigma t} \right\|_{\underline{L}^2 \left( B_{\sigma t} \right)} \, \frac{dt}{t}
\notag \\ &
\leq 
\frac12
\int_{r}^{R} 
\frac1{t} \left\| u - \ell_{t} \right\|_{\underline{L}^2 \left( B_{ t} \right)} \, \frac{dt}{t}
+
C
\int_{r}^{R} 
\left(  \omega \bigl( \tfrac{t}X \bigr)  \mathsf{M}_t + \gamma(t)\right) \, \frac{dt}{t}\,.
\end{align*}
Therefore, after reabsorption and recalling that~$\ell_r = (u)_{B_R}$ for~$r \in [\sigma R,R]$, we get 
\begin{equation*}  
\int_{r}^{R} 
\frac1{t} \left\| u - \ell_{t} \right\|_{\underline{L}^2 \left( B_{ t} \right)} \, \frac{dt}{t}
\leq 
\frac CR \left\| u- (u)_{B_R} \right\|_{\underline{L}^2 \left( B_{R} \right)}   
+ C \int_{ r}^{R}  \Bigl( \omega\bigl( \tfrac{t}X \bigr)  \mathsf{M}_t  + \gamma(t) \Bigr) \, \frac{dt}{t} 
\,.
\end{equation*}
Furthermore, by the triangle inequality and the minimality of~$\ell_t$'s,  we deduce that 
\begin{equation*} 
\sup_{s,s' \in (t/2,t)} 
| \nabla \ell_s -  \nabla \ell_{s'}| 
\leq 
\frac{C}{t} \sup_{s,s' \in (t/2,t)} 
 \left\| \ell_s - \ell_{s'} \right\|_{\underline{L}^2 \left( B_{t/2} \right)}
 \leq 
 \frac{C}{t}  \left\| u - \ell_{t} \right\|_{\underline{L}^2 \left( B_{t} \right)}
 \,.
\end{equation*}
Therefore,
\begin{equation*} 
\int_{r}^{R} 
\sup_{s,s' \in (t/2,t)} 
| \nabla \ell_s -  \nabla \ell_{s'}| \, \frac{dt}{t}  
\leq \frac CR \left\| u- (u)_{B_R} \right\|_{\underline{L}^2 \left( B_{R} \right)}   
+ C \int_{ r}^{R}  \Bigl( \omega\bigl( \tfrac{t}X \bigr)  \mathsf{M}_t  + \gamma(t) \Bigr) \, \frac{dt}{t} 
\,.
\end{equation*}
This implies, by telescoping summation and the fact that~$\nabla \ell_r = 0$ for~$r \in [\sigma R,R]$, that
\begin{equation*}  
\sup_{t \in [r,R]} |\nabla \ell_{t}| 
\leq 
\frac CR \left\| u- (u)_{B_R} \right\|_{\underline{L}^2 \left( B_{R} \right)}   
+ C \int_{ r}^{R}  \Bigl(  \omega\bigl( \tfrac{t}X \bigr) \mathsf{M}_t  + \gamma(t) \Bigr)  \, \frac{dt}{t}
\,.
\end{equation*}
Using the elementary inequality
\begin{equation*}  
\sup_{t \in [r/2,r] } \frac1{t} \left\| u - (u)_{B_t} \right\|_{\underline{L}^2 \left( B_{ t} \right)} 
\leq C \int_{r}^{2r} \frac1{t} \left\| u - \ell_{t} \right\|_{\underline{L}^2 \left( B_{ t} \right)} \, \frac{dt}{t} 
+ C \sup_{t \in [r,2r]} |\nabla \ell_{t}|  
\,,
\end{equation*}
we then obtain, for every~$r \in [X,\tfrac12 R]$, that 
\begin{equation}  
\label{e.mathsfM.estimate}
\mathsf{M}_r  \leq \frac CR \left\| u- (u)_{B_R} \right\|_{\underline{L}^2 \left( B_{R} \right)}   
+ C \mathsf{M}_r  \int_{ r}^{R}  \omega\bigl( \tfrac{t}X \bigr)  \, \frac{dt}{t}  + C \int_{ r}^{R}  \gamma(t)  \, \frac{dt}{t} 
\,.
\end{equation}
Finally, by~\eqref{e.Dini}, we take~$c(d,\lambda,\Lambda) := (2C_{\eqref{e.mathsfM.estimate}})$ so that
\begin{equation*}  
\int_{X}^{R}  \omega\bigl( \tfrac{t}X \bigr)  \, \frac{dt}{t} \leq \frac1{2C_{\eqref{e.mathsfM.estimate}}},
\end{equation*}
and hence~\eqref{e.harm.approx.goal} follows after reabsorption. The proof is complete. 
\end{proof}

We are now ready to prove Theorem~\ref{t.optimalstochasticintegrability}.

\begin{proof}[Proof of Theorem~\ref{t.optimalstochasticintegrability}]
Let~$\delta>0$ and~$\X$ be the maximum of the random variables~$\X$ appearing in Theorem~\ref{t.C01} and the inequality~\eqref{e.FVC} (we can use~\eqref{e.FV.sublinearity} instead in the symmetric case). 
Recall that the latter says that, for every~$m\in\N$ with~$3^m\geq \X$ and~$|e|=1$, we have  
\begin{equation}
\label{e.FV.sublinearity.again}
3^{-m} \| \nabla \phi_{m,e} \|_{\Hminusul(\cu_m)} 
+
3^{-m} \| \a (e+ \nabla \phi_{m,e}) -\ahom e  \|_{\Hminusul(\cu_m)} 
\leq 
\delta \bigl( \X 3^{-m} \bigr)^{\frac\theta2}
\,.
\end{equation}
Here the exponent~$\theta$ is as in Lemma~\ref{l.mathcalE.minscale}: we can take it to be any positive exponent smaller than~$
\frac{\alpha d}{d+2\alpha} \wedge   \frac d2(1-\beta)$ where~$\alpha>0$ is the exponent in~\eqref{e.EJtozero.rate.ass} (for both~$J$ and~$J^*$ in the general nonsymmetric case). 

\smallskip

Since the difference~$v_{m,e} := \phi_{m+1,e}-\phi_{m,e}$ belongs to~$\A(\cu_m)$, we obtain by Theorem~\ref{t.C01} and summing over scales that, for every~$r \in \left[\X , 3^{m-1} \right]$, 
\begin{equation}
\label{e.sumovescalesphi}
\sum_{n = m}^\infty \left\| \nabla v_{n,e} \right\|_{\underline{L}^2(B_{r})} 
\leq
C \delta \bigl( \X 3^{-m} \bigr)^{\frac\theta2}
\,.
\end{equation}
We may, therefore, define the potential field
\begin{equation*} 
\nabla \phi_e := \nabla \phi_{m,e} + \sum_{n = m}^\infty \nabla v_{n,e} \,.
\end{equation*}
Notice that the choice is independent of~$m$ and coincides with the global~$\Zd$-stationary corrector field~$\nabla \phi_e$.
Moreover, in view of~\eqref{e.sumovescalesphi}, we obtain by the triangle inequality the following estimate, which compares the finite-volume correctors to the infinite-volume corrector: for every~$m\in\N$ with~$3^m\geq \X$ and~$r \in \left[\X, 3^{m-1} \right]$,
\begin{equation}
\label{e.finite.to.infinite.phi}
\| \nabla \phi_{m,e} - \nabla \phi_e \|_{\underline{L}^2(B_r)} 
\leq
C\delta ( \X 3^{-m} )^{\frac\theta2}\,.
\end{equation}
We deduce from the previous display and~\eqref{e.FV.sublinearity.again} that,
for every~$r \in \left[\X,\infty \right)$,
\begin{equation}
\label{e.minscale.minusone}
\frac1r \left\| \nabla \phi_e \right\|_{\Hminusul(B_{r})} 
+
\frac1r \left\| \a(e+ \nabla \phi_e) - \ahom e \right\|_{\Hminusul(B_{r})} \leq C\delta \Bigl( \frac{r}{\X}\Bigr)^{\!-\frac\theta2} 
\,.
\end{equation}
We next argue that the bounds~\eqref{e.finite.to.infinite.phi} and~\eqref{e.minscale.minusone} implies a sublinear estimate for the infinite-volume correctors and flux correctors: that is, for every~$r \in \left[\X,\infty \right]$,
\begin{equation}
\label{e.minscale.yes}
\frac1r \left\| \phi_e - (\phi_e)_{B_r} \right\|_{\underline{L}^2(B_{r})} 
+
\frac1r \left\| \bfs_e - (\bfs_e)_{B_r} \right\|_{\underline{L}^2(B_{r})} 
\leq 
C \delta \Bigl( \frac{r}{\X}\Bigr)^{\!-\frac\theta2} 
\,.
\end{equation}
Note that~\eqref{e.minscale.yes} would finish the proof since it implies the desired integrability for the random variable~$\X$ defined in~\eqref{e.XC11}.

\smallskip

To prove~\eqref{e.minscale.yes}, we first note that the estimate for~$\phi_e$ in~\eqref{e.minscale.yes} is immediate from~\eqref{e.L2toHminus} and~\eqref{e.minscale.minusone}. 
This also implies that, for every~$r \geq \X$ and~$|e|=1$,
\begin{equation}
\label{e.corrbounds}
\| \nabla \phi_e \|_{\underline{L}^2(B_r)}
\leq C\,.
\end{equation}
To get the bounds on the flux correctors~$\bfs_e$, define first 
\begin{equation*}  
w_{n,e,ij} = \s_{n+1,e,ij} - \s_{n,e,ij} 
\,, \quad 
\f_{n,e,ij} = e_i (\a\nabla( \phi_{n+1,e} {-} \phi_{n,e}) )_j  - e_j (\a\nabla( \phi_{n+1,e} {-} \phi_{n,e}) )_i \,.
\end{equation*}
Then, by the equation of~$\s_{n,e,ij}$ in~\eqref{e.fluxcorrect.sec24},~$w_{n,e,ij}$ solves the equation 
\begin{equation*}
\left\{
\begin{aligned}
&
-\Delta w_{n,e,ij} = \nabla \cdot \f_{n,e,ij}
\quad \mbox{in}  \ \R^d, \\
& w_{n,e,ij} , \f_{n,e,ij}  \ \mbox{are~$3^{n+1}\Z^d$--periodic and} \ ( w_{n,e,ij} )_{\cu_{n+1}} = 0.
\end{aligned}
\right. 
\end{equation*}
and we can  write the flux corrector, by taking~$m\in \N$ such that~$r \in [3^{m-1},3^m)$, as
\begin{equation*}  
\bfs_{e,ij} = \bfs_{m,e,ij} +  \sum_{n=m}^\infty w_{n,e,ij} \,.
\end{equation*}
The first term on the right is bounded by~\eqref{e.dualest.forflux} and~\eqref{e.FV.sublinearity.again}. To bound the summands, we observe that~$w_{n,e,ij}$ has the formula
\begin{align}
\label{e.w.present1}
w_{n,e,ij}(x) - (\Phi_r \ast w_{n,e,ij})(x) 
& 
=
\int_0^{r^2} ( \Phi(t,\cdot) \ast \nabla \cdot \f_{n,e,ij} )(x) \, dt
\notag \\
& 
=
- \sum_{k=1}^d \int_0^{r^2}  (\partial_{x_k} \Phi(t,\cdot) \ast  (\f_{n,e,ij})_k )(x) \, dt 
\,.
\end{align}
and 
\begin{equation}  \label{e.w.present2}
(\Phi_r \ast w_{n,e,ij})(x)  = - \sum_{k=1}^d \int_{r^2}^\infty  (\partial_{x_k} \Phi(t,\cdot) \ast  (\f_{n,e,ij})_k )(x) \, dt  \,.
\end{equation}
By~\eqref{e.FV.sublinearity.again}, Theorem~\ref{t.C01} and periodicity, we have that
\begin{equation}  \label{e.w.fbound}
\sup_{ s \geq \X } \| \f_{n,e,ij} \|_{\underline{L}^2(B_s) } \leq C \delta \bigl( \X 3^{-n} \bigr)^{\frac\theta2} \,.
\end{equation}
Therefore, integrating over~$x$, we see by H\"older's inequality, the previous display and since~$|\nabla^k \Phi(t,x)| \leq C(k,d)t^{-\nicefrac k2} \Phi(\frac t2,x)$ for every~$k \in \N$,~$x\in \R^d$ and~$t>0$,  that 
\begin{align*}  
\| w_{n,e,ij} - \Phi_r \ast w_{n,e,ij} \|_{\underline{L}^2(B_r)}^2
& 
\leq  
C r \int_0^{r^2} t^{-\nicefrac12}  \int_{\R^d} \Phi\bigl(\tfrac12 t,y\bigr) \| \f_{n,e,ij} \|_{\underline{L}^2(B_r(y))}^2 \, dy \, dt 
\\ 
&
\leq C r \sum_{k=0}^\infty \exp(-c 4^k) \int_0^{r^2} t^{-\nicefrac12}  \| \f_{n,e,ij}\|_{\underline{L}^2(B_{r + 2^k t^{\nicefrac12}})}^2 \, dt
\\ 
&
\leq C \delta  r \bigl( \X 3^{-n} \bigr)^{\theta}  \sum_{k=0}^\infty \exp(-c 4^k) \int_0^{r^2} t^{-\nicefrac12}  \, dt    
\\ 
&
\leq C \delta r^2 \bigl( \X 3^{-n} \bigr)^{\theta}
\,.
\end{align*}
By Jensen's inequality and the previous display, we also get that 
\begin{equation*}  
|(\Phi_r \ast w_{n,e,ij})_{B_r}  - (w_{n,e,ij})_{B_r}|  
\leq 
C \delta  r \bigl( \X 3^{-n} \bigr)^{\frac\theta2}\,.
\end{equation*}
We still need to control the oscillations of~$\Phi_r \ast w_{n,e,ij}$ in~$B_r$. To do this, we now use~\eqref{e.w.present2} to deduce that 
\begin{equation*}  
| \nabla (\Phi_r \ast w_{n,e,ij})(x)|  \leq \sum_{k=1}^d  \int_{r^2}^\infty   | \partial_{x_k} \nabla \Phi( t ,\cdot) \ast  (\f_{n,e,ij})_k|  (x) \, dt  \,.
\end{equation*}
Due to~$3^{n+1}\Z^d$-periodicity of~$\f_{n,e,ij}$ and~\eqref{e.w.fbound}, we have, for every~$t \geq r^2$ and~$x \in B_r$, that  
\begin{equation*}  
 | \partial_{x_k} \nabla \Phi( t ,\cdot) \ast  (\f_{n,e,ij})_k|  (x) \leq C \delta  t^{-1} \exp(-c 3^{-2n} t ) \bigl( \X 3^{-n} \bigr)^{\frac\theta2}  \,.
\end{equation*}
(See for instance \cite[Exercise 3.24]{AKMBook}.) By splitting the integration to intervals~$(r^2,3^{2n})$ and~$(3^{2n},\infty)$, we get that 
\begin{equation*}  
| \nabla (\Phi_r \ast w_{n,e,ij})(x)|  \leq  C \delta (n +1 -m) 3^{-\frac\theta2(n-m)} \bigl( \X 3^{-m} \bigr)^{\frac\theta2} \,.
\end{equation*}
It thus follows that
\begin{equation*}  
\bigl\| \Phi_r \ast w_{n,e,ij} - (\Phi_r \ast w_{e,n,ij})_{B_r}\bigr\|_{L^\infty(B_r)}  \leq C \delta  r (n +1 -m) 3^{-\frac\theta2(n-m)} 
\bigl( \X 3^{-m} \bigr)^{\frac\theta2} 
\,.
\end{equation*}
Combining the above displays yields, by the triangle inequality, that
\begin{equation*}  
\| w_{n,e,ij} - (w_{n,e,ij})_{B_r} \|_{\underline{L}^2(B_r)} 
\leq 
C \delta  r (n +1 -m) 3^{-\frac\theta2(n-m)}  \Bigl( \frac{r}{\X} \Bigr)^{\!-\frac\theta2} 
\,.
\end{equation*}
Therefore, summing over~$n$ from~$m$ to infinity and applying triangle inequality, we obtain
\begin{equation*}  
\frac1r \| \bfs_{e,ij} - (\bfs_{e,ij})_{B_r}\|_{\underline{L}^2(B_r)} \leq C \delta  \Bigl( \frac{r}{\X} \Bigr)^{\!-\frac\theta2} \,,
\end{equation*}
which shows~\eqref{e.minscale.yes} and finishes the proof. 
\end{proof}

We conclude this section with an extension of~Theorem~\ref{t.optimalstochasticintegrability} to equations with a right-hand side in divergence form. For this result, our probability measure~$\P$ is on the enlarged space described in Section~\ref{sss.formalisms.for.RHS}. 
The main result is an estimate on the global corrector with zero slope for the equation with a right-hand side~$\nabla \cdot \f$.
This is the unique element~$\nabla \psi_{\f} \in V^2_{\pot,0}(\P)$ for which~$\a \nabla \psi_{\f}  + \f \in V^2_\sol(\P)$. 
In other words,~$\psi_{\f}$ is defined up to a constant, its gradient is a stationary, random field with zero mean, and it satisfies the equation
\begin{equation*}
-\nabla \cdot \a\nabla\psi_{\f} = \nabla\cdot \f \quad \mbox{in} \ \Rd\,.
\end{equation*}
Its existence and uniqueness is immediately obtained by the Lax-Milgram lemma and Lemma~\ref{l.stat.HH}, just as in the proof of Proposition~\ref{p.qual.correctors}. 
 
\begin{theorem}
\label{t.f.optimalstochasticintegrability}
Suppose~$\P$ satisfies~\eqref{e.f.sec.ass}. Let~$\delta \in (0,1]$. 
Then there exists a constant vector~$\overline{\f}\in
\Rd$  and a random variable~$\X$ satisfying   
\begin{equation} \label{e.f.X.integrability}
\X^{\frac d2(1-\beta)} = \O_\Psi(C)
\end{equation}
such that, for every~$r \geq \X$, the zero-slope corrector~$\psi_{\f}$ satisfies the estimate
\begin{equation}  \label{e.f.optimalstochasticintegrability}
\frac1r \|\psi_{\f} - (\psi_{\f})_{B_r} \|_{\underline{L}^2(B_r)}
+
\frac1r \|\nabla \psi_{\f} \|_{\Hminusul(B_r)} +  \frac1r \|\a \nabla \psi_{\f} + \f - \overline{\f} \|_{\Hminusul(B_r)} 
\leq 
\delta \Bigl(\frac{\X}{r} \Bigr)^{-\theta}\,.
\end{equation}
\end{theorem}
\begin{proof}
The existence of the corrector is immediate by Proposition~\ref{p.qual.correctors} and Lemma~\ref{l.stat.HH}. To show the estimates, this time, we use Theorem~\ref{t.zeroslope} to deduce that 
\begin{equation*}  
3^{-m}\bigl\| \nabla \psi_{\f,m} \bigr\|_{\underline{H}^{-1}(\cu_m)} 
+
3^{-m} \bigl\| \a \nabla \psi_{\f,m} + \f - \overline{\f} \bigr\|_{\underline{H}^{-1}(\cu_m)} 
\leq 
C  \bigl( \X 3^{-m} \bigr)^{\theta}
\,,
\end{equation*}
where~$\psi_{\f,m}$ solves~$-\nabla \cdot \a \nabla \psi_{\f,m} = \nabla \cdot \f$ in~$\cu_m$ and~$\X^{\frac d2(1-\beta)} = \O_\Psi(C)$.  Considering the difference~$v_{m} := \psi_{\f,m+1}-\psi_{\f,m}$, which belongs to~$\A(\cu_m)$, we may then proceed as in the proof of Theorem~\ref{t.optimalstochasticintegrability} to show~\eqref{e.f.optimalstochasticintegrability}. 
\end{proof}

\subsection{Renormalization: regularity improves the convergence rate}
\label{ss.almostone} 

Homogenization improves the regularity of solutions on large scales. This is an important principle we first encountered in Theorem~\ref{t.C11.sharp}, and then in a more quantitative way in the previous section. 

\smallskip 

We will now explore the converse of this principle: \emph{large-scale regularity improves the rate of homogenization}. We will explore this idea with more detail and sophistication in Chapter~\ref{s.renormalization}, where we prove quantitative estimates on the homogenization error, which are \emph{optimal} in their dependence on the scale separation ratio.
In this section,\footnote{The results in this section are not used in subsequent sections, and it may be skipped on a first reading.} we will pursue the more modest goal of improving the exponent in Theorem~\ref{t.subadd.converge} and Proposition~\ref{p.algebraic.nosymm} from a small positive constant~$\alpha(\beta,d,\lambda,\Lambda)$ to any~$\alpha<1$ (in the case of strong mixing conditions). This is optimal for the coarse-grained quantities defined with respect to cubes. 

\begin{proposition}[{Improvement of exponent to any~$\alpha<1$}]
\label{p.alpha.almostone}
Assume~$\P$ is a~$\Zd$--stationary measure on~$(\Omega,\F)$ and satisfies~$\CFS(\beta,\Psi)$. 
Let~$0 < \alpha < \min\{ 1, d(1-\beta) \}$. 
There exists a constant~$C(\alpha,\dataref)<\infty$
such that, for every~$m \in\N$,
\begin{equation}
\label{e.improvedEconv}
| \ahom(\cu_m) - \ahom | 
+
|\ahom_*(\cu_m) - \ahom |
\leq 
C 3^{-m \alpha} \,.
\end{equation}
Consequently, for every~$m,n\in\N$ with~$\beta m < n < m$, 
\begin{equation}
\label{e.aastar.big.smash.improved}
| \a(\cu_m) - \ahom |
+
|\a_*(\cu_m) - \ahom |
\leq 
C 3^{-n \alpha}
+ 
\O_\Psi \bigl( C3^{-\frac d2(m-n)} \bigr) 
\,.
\end{equation}
\end{proposition}

Compared to the iteration in the symmetric case in Chapter~\ref{s.subadd}, the main step in the proof of Proposition~\ref{p.alpha.almostone} is the improvement of Lemma~\ref{l.flatness.rules}. To see why we need to improve this lemma, we can look back to the iteration argument that gave the proof of Proposition~\ref{e.EJtozero.rate}. The inequality we iterated is found in~\eqref{e.combo}. If we ask which of the terms on the right side of~\eqref{e.combo} saturate the error, we find that it is the last term: the discounted sum of the~$\tau_j$'s. Even if all of the other terms on the right side of~\eqref{e.combo} could be discarded, an iteration of the resulting inequality would still give us only a tiny exponent~$\alpha$ depending on the prefactor constant~$C$. We therefore need an argument that removes these~$\tau_j$'s. 

\smallskip

The~$\tau_j$'s found their way into the error term by way of~\eqref{e.flatness.rules}, and they arose in the latter inequality because we made a comparison between optimizing functions on different scales. They represent the ``copy and paste error." We therefore need a more precise way to compare minimizers on different scales, and we will do this by using the large-scale regularity theory. 
The improved estimate that we use in place of~\eqref{e.flatness.rules} is found below in~\eqref{e.baby.optimal.iter.basic}, near the end of the proof of Proposition~\ref{p.alpha.almostone}. 

\smallskip

We begin in the next lemma by bounding the quantity~$J(U,p,q)$ in a different way from Lemmas~\ref{l.flatness.rules} and~\ref{l.flatness.rules.nosymm}. 
Before presenting the statement, we recall the definitions of~$\a_*(U_j):= \s_*(U_j) - \k^t(U_j)$ and~$\ahom_*(U_j):= \shom_*(U_j) - \khom^t(U_j)$.

\begin{lemma}[Upper bound for~$J$]
\label{l.J.upper.bound}
For every bounded Lipschitz domain~$U\subseteq\Rd$, finite partition~$\{ U_1,\ldots,U_N\}$ of~$U$ satisfying~\eqref{e.U.partition} 
and~$p,q\in\Rd$, 
\begin{align}
\label{e.J.upper.bound}
J(U,p,q) 
\leq 
\sum_{j=1}^N\!
\frac{|U_j|}{|U|}
\biggl( 
p\cdot \fint_{U_j}\bigl( \a_*(U_j) - \a \bigr) \nabla v(\cdot,U,p,q) 
+
\frac12\bigl| \s_*^{-\nicefrac12}(U_j) ( q - \a_*^t (U_j) p ) \bigr|^2
\biggr)\,.
\end{align}

\end{lemma}
\begin{proof}
We use~\eqref{e.Jaas.nosymm} to obtain, for every~$u \in \mathcal{A}(U_j)$, 
\begin{align*}
\lefteqn{ 
\fint_{U_j} \left( -\frac12 \nabla u \cdot \a\nabla u -p\cdot \a\nabla u+ q\cdot \nabla u   \right)
} \qquad & 
\notag \\ & 
= 
\fint_{U_j} \left( -\frac12 \nabla u \cdot \a\nabla u - (q - \a_*^t(U_j) p) \cdot \nabla u   \right)
+
p\cdot\fint_{U_j} \bigl( \a_*(U_j) - \a \bigr) \nabla u 
\notag \\ & 
\leq
J(U_j,0,q - \a_*^t(U_j) p)
+
p\cdot\fint_{U_j} \bigl( \a_*(U_j) - \a \bigr) \nabla u 
\notag \\ & 
= \frac12 \bigl( q - \a_*^t(U_j) p \bigr) \cdot
\s_*^{-1} (U_j)
\bigl( q - \a_*^t(U_j) p \bigr) + p\cdot\fint_{U_j} \bigl( \a_*(U_j) - \a \bigr) \nabla u 
\,.
\end{align*}
Applying this inequality to~$u \in \A(U)$ and summing over~$j\in \{1,\ldots,N\}$ yields
\begin{align*}
\lefteqn{
\fint_{U} 
\left( -\frac12 \nabla u \cdot \a\nabla u -p\cdot \a\nabla u+ q\cdot \nabla u   \right)
} \qquad & 
\notag \\ & 
\leq 
\sum_{j=1}^N
\frac{|U_j|}{|U|}
\biggl( 
\frac12\bigl| \s_*^{-\nicefrac12}(U_j) ( q - \a_*^t (U_j) p ) \bigr|^2
+
p\cdot\fint_{U_j} \bigl( \a_*(U_j) - \a \bigr) \nabla u
\biggr)
\,.
\end{align*}
Taking~$u = v(\cdot,U,p,q)$  yields the lemma. 
\end{proof}

In order to make use of~\eqref{e.J.upper.bound}, we need to estimate the first term in the sum on the right side without comparing~$\nabla v(\cdot,U,p,q)$ to the smaller-scale optimizers~$\nabla v(\cdot,U_j,p,q)$ because this would come at the expense~$\tau_j$s. 
A clue for how we should proceed can be found in the estimate~\eqref{e.coarse.graining.nosymm}, which we restate here: for every~$u\in \mathcal{A}(\cu_n)$, 
\begin{equation}
\label{e.fluxmaps.again}
\biggl| 
\fint_{\cu_n} \bigl( \a_*(\cu_n) - \a \bigr) \nabla u
\biggr| ^2
\leq 
2\left| \s(\cu_n) - \s_*(\cu_n) \right| 
\biggl( \fint_{\cu_n}  \nabla u \cdot \a\nabla u  \biggr)
\,.
\end{equation}
This may, at first glance, not appear very promising. 
Indeed, if we apply~\eqref{e.fluxmaps.again} to the right side of~\eqref{e.J.upper.bound} with~$U=\cu_m$ and partition~$\{ z+\cu_n\,:\, z\in 3^n\Zd\cap \cu_m\}$, and then use~\eqref{e.diagonalset} to bound~$| \a(\cu_n) - \a_*(\cu_n)|$ and~\eqref{e.Jenergyv} to estimate~$\| \nabla v(\cdot,\cu_m,p,q)\|_{\underline{L}^2(\cu_n)}$, we would obtain, after setting~$q=\ahom_*^t(\cu_n)p$, 
\begin{align*}
J(\cu_m,p,\ahom_*^t(\cu_n)p ) 
&
\leq 
\frac 12\!\avsum_{z\in 3^n\Zd \cap\cu_m }
\bigl| \s_*^{-\nicefrac12}(z+\cu_n) (\a_*(z+\cu_n) - \ahom_*(\cu_n))^tp\bigr|^2
\notag \\ & \qquad 
+ 
4J(\cu_m,p,\ahom_*^t(\cu_n)p ) ^{\nicefrac12}
\biggl( \avsum_{z\in 3^n\Zd \cap\cu_m }
\bigl| (\s - \s_*) (z+\cu_n) \bigr|
\biggr)^{\!\nicefrac12}
\,.
\end{align*}
Needless to say, this is not a very useful estimate since the second term on the right side is larger than the left side. 

\smallskip

However, it \emph{almost} works. The above computation tells us that if we can improve the estimate~\eqref{e.fluxmaps.again} by reducing the second term on the right side from~$CJ(\cu_m,p,\ahom(\cu_n)p )$ to~$\frac12 J(\cu_m,p,\ahom(\cu_n)p )$, 
then we can reabsorb it on the left side of the inequality. We would then obtain an estimate similar to~\eqref{e.flatness.rules}, but much more favorable since there are no~$\tau_j$'s on the right side.

\smallskip

The improvement of~\eqref{e.fluxmaps.again} comes thanks to the large-scale~$C^{1,\gamma}$ estimate. The idea is that the error in~\eqref{e.fluxmaps.again} is identically zero in the case that~$u$ is equal to one of the minimizers~$v(\cdot,\cu_n,p,0)$ of~$J(\cu_n,p,0)$ for~$p\in\Rd$. But the regularity theory says that every solution (on a much larger cube) should be well-approximated by one of these functions!

\begin{lemma}[Improved coarse-graining inequality]
\label{l.fluxmaps.now.better}
Let~$\theta >0$ be as in the statement of Lemma~\ref{l.mathcalE.minscale}.
For each~$\delta \in (0,\tfrac12]$ and~$\gamma\in [\tfrac12,1)$, there exists a minimal scale~$\X_{\delta,\gamma}$ and constants~$C_1(\theta,K,\delta,\gamma,\dataref)<\infty$ and~$C(\gamma,d,\lambda,\Lambda)<\infty$
satisfying   
\begin{equation} 
\label{e.X.integrability.delta.gamma}
\X_{\delta,\gamma}^{\frac d2(1-\beta)} = \O_\Psi(C_1)\,,
\end{equation}
such that, for every~$m,n\in\N$ and~$u\in \A(\cu_m)$, 
\begin{align}
\label{e.fluxmaps.now.better}
\lefteqn{
\biggl| \fint_{\cu_n} \bigl( \a_*(\cu_n) - \a \bigr) \nabla u \biggr| 
} \quad & 
\notag \\ &
\leq
C | \s(\cu_n) - \s_*(\cu_n) |^{\nicefrac12}  \bigl( 3^{-\gamma (m-n)}
\| \nabla u \|_{\underline{L}^2(\cu_m)}
+ \bigl(  \delta( \X_{\delta,\gamma} 3^{-n})^{{\frac\theta 2}}\wedge 1\bigr) \bigl\| \nabla u \bigr\|_{\underline{L}^2(\cu_n)} \bigr) 
\,.
\end{align}
\end{lemma}
\begin{proof}
Let~$\X_{\delta,\gamma}$ be the maximum of the random variables in the statements of Theorems~\ref{t.C11.sharp} and~\ref{t.optimalstochasticintegrability} (we will also need estimates found in the proof of the latter). By Theorem~\ref{t.optimalstochasticintegrability}, the estimate~\eqref{e.X.integrability.delta.gamma} for~$\X_{\delta,\gamma}$ is valid. Note that~$\delta$ can be made small without loss of generality. 

\smallskip

Select~$m,n\in\N$ with~$n<m$ and~$u\in \A(\cu_m)$. 
In the case that~$3^n < \X_{\delta,\gamma}$, there is nothing to show as the inequality~\eqref{e.fluxmaps.now.better} follows from~\eqref{e.fluxmaps.again}. 
We therefore assume that~$3^n \geq \X_{\delta,\gamma}$. 
By~\eqref{e.a.formulas.nosymm} and~\eqref{e.fluxmaps.again}, we have that, for every~$e\in\Rd$, 
\begin{align}
\label{e.removep}
\biggl| \fint_{\cu_n} \bigl( \a_*(\cu_n) - \a \bigr) \nabla u \biggr| 
&
=
\biggl| \fint_{\cu_n} \bigl( \a_*(\cu_n) - \a \bigr) \bigl( \nabla u - \nabla v(\cdot,\cu_n,0,e) \bigr) 
\biggr| 
\notag \\ & 
\leq
2^{\nicefrac12} \bigl| \s(\cu_n) - \s_*(\cu_n) \bigr|^{\nicefrac12}  \bigl\| \s^{\nicefrac12} (\nabla u - \nabla  v(\cdot,\cu_n,0,e)) \bigr\|_{\underline{L}^2(\cu_n)} 
\,.
\end{align}
By Theorem~\ref{t.C11.sharp}, the assumption~$3^n \geq \X_{\delta,\gamma}$ implies the existence of~$e\in\Rd$ such that 
\begin{align*}
\lefteqn{ 
\bigl\| \nabla u - \nabla  v(\cdot,\cu_n,0,\ahom^t e)  \bigr\|_{\underline{L}^2(\cu_n)} 
} \qquad & 
\notag \\ &
\leq 
\bigl\| \nabla u - ( e + \nabla \phi_e)  \bigr\|_{\underline{L}^2(\cu_n)} 
+
\bigl\| (e+\nabla \phi_e) - \nabla  v(\cdot,\cu_n,0,\ahom^t e)  \bigr\|_{\underline{L}^2(\cu_n)}  
\notag \\ & 
\leq 
C 3^{-\gamma (m-n)}
\| \nabla u \|_{\underline{L}^2(\cu_m)}
+
\bigl\| (e+\nabla \phi_e) - \nabla  v(\cdot,\cu_n,0,\ahom^t e)  \bigr\|_{\underline{L}^2(\cu_n)}  \,.
\end{align*}
By~\eqref{e.Jenergyv.nosymm}, Remark~\ref{r.localX.nonsymm} and~\eqref{e.finite.to.infinite.phi}, we have 
\begin{align*}
\lefteqn{ 
\bigl\| (e+\nabla \phi_e) - \nabla  v(\cdot,\cu_n,0,\ahom^t e)  \bigr\|_{\underline{L}^2(\cu_n)}
} \qquad & 
\notag \\ &
\leq 
\bigl\| (e+\nabla \phi_e) - \nabla  v(\cdot,\cu_n,-e,0)  \bigr\|_{\underline{L}^2(\cu_n)}
+
\bigl\| \nabla  v(\cdot,\cu_n,e,\ahom^t e)  \bigr\|_{\underline{L}^2(\cu_n)}
\notag \\ &
\leq 
\bigl\| (e+\nabla \phi_e) - \nabla  v(\cdot,\cu_n,-e,0)  \bigr\|_{\underline{L}^2(\cu_n)}
+
CJ(\cdot,\cu_n,e,\ahom^t e)^{\nicefrac12}
\notag \\ &
\leq 
C\delta|e|\bigl( \X_{\delta,\gamma} 3^{-n} \bigr)^{{\frac\theta 2}}
\leq 
C\delta\bigl( \X_{\delta,\gamma} 3^{-n} \bigr)^{{\frac\theta 2}} \bigl\| \nabla  v(\cdot,\cu_n,0,\ahom^te)  \bigr\|_{\underline{L}^2(\cu_n)}
\,.
\end{align*}
Assuming that~$\delta>0$ is sufficiently small, depending only on~$(d,\lambda,\Lambda)$, we obtain
\begin{equation*}
\bigl\| \nabla u - \nabla  v(\cdot,\cu_n,0,\ahom^t e)  \bigr\|_{\underline{L}^2(\cu_n)} 
\leq
C 3^{-\gamma (m-n)}
\| \nabla u \|_{\underline{L}^2(\cu_m)}
+ C\delta\bigl( \X_{\delta,\gamma} 3^{-n} \bigr)^{{\frac\theta 2}} 
\bigl\| \nabla u \bigr\|_{\underline{L}^2(\cu_n)} 
\,.
\end{equation*}
Combining the above yields the lemma. 
\end{proof}

We are now ready for the proof of Proposition~\ref{p.alpha.almostone}. Most of the argument involves using Lemma~\ref{l.fluxmaps.now.better} to estimate the first term on the right side of~\eqref{e.J.upper.bound}.

\begin{proof}[{Proof of Proposition~\ref{p.alpha.almostone}}]
Fix~$\alpha \in (0 , \min\{ 1, d(1-\beta) \})$, and let~$\sigma := \min\{ 1, d(1-\beta) \} - \alpha$. We allow the constants in the proof to depend on~$\sigma$ and~$\dataref$.

\smallskip

For readability, we will write the argument using notation from the symmetric case. The adaptations required for the nonsymmetric case are minimal, and we indicate them in the proof. 

\smallskip

\emph{Step 1.} We prove that there exists a constant~$C(d,\lambda,\Lambda)<\infty$ such that, for every~$m,n \in \N$ with~$m \geq n$,
\begin{align} 
\label{e.baby.optimal.iter.basic}
\sup_{|e| \leq 1} \E\bigl[ J(\cu_m, e, \ahom_*(\cu_m) e ) \bigr]
&
\leq 
C  \bigl( 3^{-(m-n)} + \delta^2 \bigr)  \sup_{|e| \leq 1} \E\bigl[ J(\cu_n, e, \ahom_*(\cu_n) e )\bigr]
\notag \\ & \qquad  + C \var\bigl[\a_*(\cu_n)\bigr] + C \P\bigl[\X_\delta > 3^n \bigr]
 \,.
\end{align}
Here~$\X_\delta := \X_{\delta,\nicefrac23}$ is as in Lemma~\ref{l.fluxmaps.now.better} and we denote by~$\X_\delta(z)$,~$z \in \Zd$, its~$\Zd$-stationary extension. 

\smallskip

In the general nonsymmetric case, the estimate~\eqref{e.baby.optimal.iter.basic} must be replaced by one which involves the sum of~$J$ and~$J^*$, as follows:
\begin{align}
\label{e.baby.optimal.iter.basic.nonsymm}
\lefteqn{
\sup_{|e| \leq 1} \E\bigl[ J(\cu_m, e, (\ahom -\khom)(\cu_m) e ) +  J^*(\cu_m, e, (\ahom +\khom)(\cu_m) e ) \bigr]
} \qquad & 
\notag \\ &
\leq 
C  \bigl( 3^{-(m-n)} + \delta^2 \bigr)  \sup_{|e| \leq 1} \E\bigl[ J(\cu_n, e, (\ahom -\khom)(\cu_n) e ) +  J^*(\cu_n, e, (\ahom +\khom)(\cu_n) e )\bigr]
\notag \\ & \qquad  
 + C \var\bigl[\bfA_*(\cu_n)\bigr] + C \P\bigl[\X_\delta > 3^n \bigr]
 \,.
\end{align}
Continuing with the proof of~\eqref{e.baby.optimal.iter.basic}, let us denote~$v = v(\cdot,\cu_m,e, \ahom_*(\cu_n) e)$ for each~$e\in\Rd$. For each~$j,n \in \N$ with~$j\geq n$, let~$P_j$ be the partition of cubes of size~$3^n$ in the boundary layer of width~$3^j$ in~$\cu_m$:
\begin{equation*} 
P_{j} := \{ z \in 3^{n}\Zd  \cap \cu_m \, : \, z + \cu_j \subset \cu_m\,, (z + \cu_{j+1}) \cap \partial \cu_m \neq \emptyset    \} \,.
\end{equation*}
There are at most~$3^{-(m-j)} 3^{d(m-n)}$ many members in~$P_j$. We have
\begin{equation*} 
\avsum_{z\in 3^{n} \Zd  \cap \cu_m} 
\fint_{z+\cu_{n}} \bigl( \a_*(z+\cu_{n}) - \a \bigr) \nabla v  
= \sum_{j=n}^{m-1}
\sum_{z\in 3^{n} \Zd \cap P_j}
\frac{|\cu_n|}{|\cu_m|}  \fint_{z+\cu_{n}} \bigl( \a_*(z+\cu_{n}) - \a \bigr) \nabla v  
\,.
\end{equation*}
For every~$z \in P_j$, we get, by Lemma~\ref{l.fluxmaps.now.better}, 
\begin{align*}
\lefteqn{
\biggl| \fint_{z+\cu_{n}} \bigl( \a_*(z+\cu_{n}) - \a \bigr) \nabla v  \biggr| 
}   & 
\notag \\ &
\leq
C \bigl| (\a-\a_*)(z{+}\cu_{n}) \bigr|^{\nicefrac12}  
\Bigl( 3^{-\frac 23(j-n) }
\| \nabla v \|_{\underline{L}^2(z+\cu_{j})}
{+} \bigl(  \delta( \X_{\delta}(z) 3^{-n})^{\theta}\wedge 1\bigr) \bigl\| \nabla v \bigr\|_{\underline{L}^2(z+\cu_{n})} \Bigr) 
\,.
\end{align*}
Summing over~$z\in 3^{n} \Zd \cap P_j$, and using Cauchy-Schwarz and~$\Zd$-stationarity, we get 
\begin{align*} 
\lefteqn{
\sum_{z\in 3^{n} \Zd \cap P_j}
\frac{|\cu_n|}{|\cu_m|}   \E \Biggl[ \biggl| \fint_{z+\cu_{n}} \bigl( \a_*(z+\cu_{n}) - \a \bigr) \nabla v  \biggr| 
\Biggr]
} \quad &
\notag \\ &
\leq 
C \E\bigl[ | \a(\cu_{n}) - \a_*(\cu_{n}) | \bigr]^{\nicefrac12}  
\notag \\ & \qquad \times 
\sum_{z\in 3^{n} \Zd \cap P_j} \Bigl( 3^{-\frac23(j-n)} \E\bigl[\| \nabla v \|_{\underline{L}^2(z+\cu_{j})}^2 \bigr]^{\nicefrac12} + \delta \E\bigl[\| \nabla v \|_{\underline{L}^2(z+\cu_{n})}^2 \bigr]^{\nicefrac12}  \Bigr)
\notag \\ & \qquad 
+ C \P\bigl[\X_\delta > 3^n \bigr]^{\nicefrac12}  \sum_{z\in 3^{n} \Zd \cap P_j} \E\bigl[\| \nabla v \|_{\underline{L}^2(z+\cu_{n})}^2 \bigr]^{\nicefrac12} 
\notag \\ &
\leq 
C \biggl( 3^{-(m-j)} 
\E\bigl[\| \nabla v \|_{\underline{L}^2(\cu_{m})}^2 \bigr] \Bigl(  \bigl(3^{-(j-n)}+\delta^2 \bigr) \E\bigl[ | \a(\cu_{n}) - \a_*(\cu_{n}) | \bigr] +   \P\bigl[\X_\delta > 3^n \bigr] \Bigr)\biggr)^{\!\nicefrac12}
\,.
\end{align*}
It thus follows, by summing over~$j$ and applying Young's inequality, that
\begin{align*} 
\lefteqn{
 \E \Biggl[ \biggl| \avsum_{z\in 3^{n} \Zd  \cap \cu_m} 
\fint_{z+\cu_{n}} \bigl( \a_*(z+\cu_{n}) - \a \bigr) \nabla v    \biggr|\Biggr]
} \qquad &
\notag \\ &
\leq \frac12 \E\bigl[\| \nabla v \|_{\underline{L}^2(\cu_{m})}^2 \bigr] + C \bigl( 3^{-(m-n)} + \delta^2 \bigr) \E\bigl[ | \a(\cu_{n}) - \a_*(\cu_{n}) | \bigr] 
+ C \P\bigl[\X_\delta > 3^n \bigr]
\,.
\end{align*}
By~\eqref{e.diagonalset} (or using~\eqref{e.diagonalset.bigA} in the general nonsymmetric case) , we have
\begin{equation*} 
\E\bigl[ | \a(\cu_{n}) - \a_*(\cu_{n}) | \bigr] 
\leq C \sup_{|e| \leq 1} \E\bigl[ J(\cu_n, e, \ahom_*(\cu_n) e ) \bigr]\,.
\end{equation*}
Applying Lemma~\ref{l.J.upper.bound},~\eqref{e.Jenergyv} (using~\eqref{e.Jenergyv.nosymm} in the general nonsymmetric case) and the above two inequalities, then reabsorbing and taking supremum over~$|e| \leq 1$ of the resulting estimate, we get
\begin{align*} 
\sup_{|e| \leq 1} \E\bigl[ J(\cu_m, e, \ahom_*(\cu_n) e ) \bigr]
&
\leq 
C  \bigl( 3^{-(m-n)} + \delta^2 \bigr)  \sup_{|e| \leq 1} \E\bigl[ J(\cu_n, e, \ahom_*(\cu_n) e )\bigr]
\notag \\ & \qquad  + C \var\bigl[\a_*(\cu_n)\bigr] + C \P\bigl[\X_\delta > 3^n \bigr]
 \,.
\end{align*}
Since~$q = \ahom_*(\cu_m) e$ minimizes~$q \mapsto \E\bigl[ J(\cu_m, e, q ) \bigr]$, we then arrive at~\eqref{e.baby.optimal.iter.basic}.
 
\smallskip

\emph{Step 2.}  We give the proof of~\eqref{e.improvedEconv}. First, Lemma~\ref{l.flatness} (or Lemma~\ref{l.flatness.nosymm}) implies that, for every~$\beta n < k < n$, 
\begin{equation}
\label{e.baby.optimal.variance.pre}
\var \bigl[  \a(\cu_{n}) \bigr] 
\leq
C
\sup_{|e|=1}
\E \left[ J(\cu_{k},e, \ahom_*(\cu_{k})  e) \right]^2
+
C
3^{-d(n-k)}
\,.
\end{equation}
By~\eqref{e.X.integrability.delta.gamma}, we deduce that
\begin{equation} 
\label{e.baby.optimal.minscale.pre}
\P\bigl[\X_\delta > 3^n \bigr] 
\leq \frac1 {\Psi(c 3^{\frac d2(1-\beta) n}) }  
\leq C 3^{d(1-\beta) n}
\,.
\end{equation}
Recall here that we assume that~$\Psi(\cdot)$ has at least quadratic growth, provided by~\eqref{e.Young.growth}. 
Fix~$\theta := \max\{ \beta \,, \frac{d}{d+1} \} $ and recall that~$\sigma := \min\{1,d(1-\beta) \} - \alpha$. Let~$h \in \N$ be the smallest integer such that~$C_{\eqref{e.baby.optimal.iter.basic}}  3^{-\sigma h} \leq \nicefrac18$ and then fix~$\delta$ so small that~$C_{\eqref{e.baby.optimal.iter.basic}} 3^{\alpha h} \delta^2 \leq \nicefrac18$. We have
\begin{equation} 
\label{e.baby.optimal.h.pre}
C_{\eqref{e.baby.optimal.iter.basic}} 3^{\alpha h} \bigl( 3^{-h} + \delta^2 \bigr) \leq \frac14
\,.
\end{equation} 
This way both~$h$ and~$\delta$ depend only on~$(\sigma,d,\lambda,\Lambda)$. In what follows, we absorb constants depending on~$h$ to the constant~$C$. By Proposition~\ref{p.algebraicrate.E}, we find constants~$\alpha_0(\beta,d,\lambda,\Lambda) \in (0,1)$ and~$C(\dataref)<\infty$ such that, for every~$k\in\N$, 
\begin{equation} 
\label{e.baby.optimal.iter.init}
\E \left[ J(\cu_{k},e, \ahom_*(\cu_{k})  e) \right] \leq C 3^{-\alpha_0 k }
\,.
\end{equation}
We let~$n_0,n_1 \in \N$ to be auxiliary parameters to be selected below so that both~$(1-\beta)n_1 \geq n_1 - n_0-1$ and~$n_0 \geq \theta(n_1+h)+2$ are valid. Denote, for every~$k \in \N$,\footnote{In the general nonsymmetric case, set~$F_k:=\sup_{|e| \leq 1} \E [ J(\cu_{k},e, (\ahom -\khom)(\cu_{k})  e)+ J^*(\cu_{k},e, (\ahom +\khom)(\cu_{k})  e)]$.} 
\begin{equation*} 
F_k := \sup_{|e| \leq 1}\E \left[ J(\cu_{k},e, \ahom_*(\cu_{k})  e) \right]
\qand
\tilde F_k := \indc_{\{k \geq n_1+h\}}\sum_{j=n_1+h}^k 3^{\alpha(j-n_1)} F_j 
 \,.
\end{equation*}
For~$m\in \N$ with~$m\geq n_1+h$, we define~$k_m := \lceil \theta (m -n_1-h) \rceil + n_0$. We have~$k_m \geq n_0$ and~$\beta m < k_m < m$ whenever~$m\geq n_1+h$ because~$n_0 \geq \theta(n_1+h)+2$. We then insert~\eqref{e.baby.optimal.variance.pre},~\eqref{e.baby.optimal.minscale.pre},~\eqref{e.baby.optimal.h.pre} and~\eqref{e.baby.optimal.iter.init} into~\eqref{e.baby.optimal.iter.basic} and get, for every~$m \in \N$ with~$m\geq n_1 + h$, 
\begin{equation*} 
F_m 
\leq 
\frac14 \cdot 3^{-\alpha h} F_{m-h}
+ C F_{k_m}^2 +  C \bigl( 3^{-d(n_1-n_0)} + 3^{-d(1-\beta) n_1}\bigr)3^{-d(1-\theta)(m-n_1)}
\,.
\end{equation*}
The above estimate can be written, using the condition~$(1-\beta)n_1 \geq n_1 - n_0 -1$, as
\begin{equation*} 
F_m 
\leq 
\frac14 \cdot 3^{-\alpha h} F_{m-h}
+ C F_{k_m}^2 +  C 3^{-d(n_1- n_0) } 3^{-d(1-\theta)(m-n_1)}\,.
\end{equation*}
The summation then yields, by~\eqref{e.baby.optimal.iter.init}, for every~$m\in\N$ with~$m \geq n_1 + h$, 
\begin{equation*} 
\tilde{F}_{m+h} 
\leq
\frac14 \tilde{F}_{m} + C 3^{-\alpha_0 n_0}
+C\sum_{n=n_1+h}^m 3^{\alpha (n-n_1)} F_{k_n}^2
+ C 3^{-d(n_1- n_0) }\,.
\end{equation*}
We estimate the third term in the above display. For~$n_2:= n_1+h + \lceil \theta^{-1} (n_1 + h - n_0) \rceil $, we get, by~\eqref{e.baby.optimal.iter.init}, that
\begin{equation*} 
\sum_{n=n_1+h}^{n_2} 3^{\alpha (n-n_1)} F_{k_n}^2
\leq C 3^{-2\alpha_0 n_0} 3^{\alpha \theta^{-1} (n_1-n_0)}
\,.
\end{equation*}
For~$n > n_2$, since~$\theta \geq \nicefrac12$, we use
\begin{align*} 
n - 2k_n 
\leq n_2 -2k_{n_2}  
& \leq n_1 + h + \theta^{-1}(n_1 + h - n_0)  - 2(n_1+h-n_0) -2n_0
\notag \\ &
= - n_1 + h(\theta^{-1}-1) + \theta^{-1}(n_1  - n_0)  \,, 
\end{align*}
and get
\begin{align*} 
\sum_{n=n_2+1}^m 3^{\alpha (n-n_1)} F_{k_n}^2 
&
\leq
3^{-\alpha n_1}\Bigl( \max_{n \in \N \cap [n_2,m]} 3^{\alpha( n -2 k_n)}\Bigr)
\Bigl( \max_{n \in \N \cap [n_2,m] } 
3^{\alpha k_n} F_{k_n} \Bigr)
\sum_{n=n_2+1}^m 3^{\alpha k_n} F_{k_n}
\notag \\ & 
\leq 
C3^{\alpha \theta^{-1}(n_1-n_0) }\biggl( \sum_{n=n_2+1}^m 3^{\alpha (k_n-n_1)} F_{k_n}\biggr)^2
\leq
C3^{\alpha \theta^{-1}(n_1-n_0) }\tilde{F}_{m}^2
\,.
\end{align*}
Combining the previous four displays gives us
\begin{equation*} 
\tilde{F}_{m+h} \leq \biggl(\frac14   + C 3^{\alpha \theta^{-1} (n_1-n_0)}  \tilde{F}_{m} \biggr) \tilde{F}_{m} +  C \Bigl(3^{-\alpha_0 n_0} + 3^{-2\alpha_0 n_0} 3^{\alpha \theta^{-1} (n_1-n_0)}  + 3^{d(n_1-n_0)} \Bigr)\,.
\end{equation*} 
We then assume inductively that there exists constants~$\kappa_1 \in (0,\alpha_0]$ and~$K \in [1,\infty)$ to be selected below such that
\begin{equation*} 
\max_{n \in \{ n_1+h,\ldots,m\} } \tilde{F}_{n}  
\leq K 3^{-\kappa_1 n_0} 
\,.
\end{equation*}
By~\eqref{e.baby.optimal.iter.init}, this is valid with~$m \in \N \cap [n_1+h, n_1 + 2h]$ and~$K \geq C_{\eqref{e.baby.optimal.iter.init}}3^{2\alpha h +1}$. Thus, we obtain
\begin{align*} 
\tilde{F}_{m+h} & \leq \biggl(\frac14 + C K 3^{\alpha \theta^{-1} (n_1-n_0)- \kappa_1 n_0}  \biggr) K 3^{-\kappa_1 n_0} 
\notag \\ & \qquad 
+ \frac{C}{K}\Bigl( 1 +3^{ \alpha \theta^{-1} (n_1-n_0)-\kappa_1 n_0} + 3^{-d(n_1-n_0)+\kappa_1 n_0} \biggr) K 3^{-\kappa_1 n_0} 
\,.
\end{align*}
Define now auxiliary parameters
\begin{equation*} 
\kappa_2 := \frac14 \min\{ 1-\theta , \alpha_0 \}\,, \quad 
\kappa_1 := \frac12 (2\alpha \theta^{-1} + 1)\kappa_2
\qand
n_1 : =  \bigl\lceil (1+ \kappa_2 ) n_0 \bigr\rceil 
\,.
\end{equation*}
Notice that the conditions~$(1-\beta)n_1 \geq n_1 - n_0-1$ and~$n_0 \geq \theta(n_1+h)+2$ are valid provided that~$n_0 \geq C$. We also have
\begin{equation*} 
\left\{
\begin{aligned}
& \alpha \theta^{-1} (n_1-n_0)- \kappa_1 n_0 \leq 2 - \frac {\kappa_2}{2} n_0 \,,
\\
& 
-d(n_1-n_0)+\kappa_1 n_0 \leq  \bigl(-d + \alpha \theta^{-1} + \tfrac12 \bigr) \kappa_2n_0 
\leq
\bigl(-d + \tfrac{d+1}{d} + \tfrac12 \bigr) \kappa_2n_0 \leq 0
\,.
\end{aligned}
\right.
\end{equation*}
We then obtain
\begin{equation} 
\label{e.baby.bootstrap.almost.done}
\tilde{F}_{m+h}  \leq \biggl( \frac14 + CK3^{- \frac{1}{2} \kappa_2 n_0} + \frac{C}{K}\biggr) K 3^{-\kappa_1 n_0} \,.
\end{equation}
By taking~$K$ so large that~$ K \geq 2 C_{\eqref{e.baby.bootstrap.almost.done}}$ and then~$n_0$ so large that~$C_{\eqref{e.baby.bootstrap.almost.done}}K3^{- \frac{1}{2} \kappa_2 n_0} \leq \nicefrac14$, we obtain the induction step, proving also that~$\tilde F_{m} \leq C$ for every~$m\in\N$.
Since~$F_m \leq 3^{-\alpha (m-n_1)} \tilde{F}_m$ for every~$m\geq n_1+h$, and~$n_1$ and~$h$ were chosen to be constants above, we deduce that 
\begin{equation}
\label{e.good.J.bound.bootbark}
\sup_{|e| \leq 1}\E \left[ J(\cu_{m},e, \ahom_*(\cu_{m})  e) \right]
\leq 
C 3^{-\alpha m}\,, \quad \forall m\in\N\,.
\end{equation}
This yields~\eqref{e.improvedEconv} by~\eqref{e.magic.b.plug.E.ahom.pre} (or~\eqref{e.sss.magic.iden.nosymm.homs} in the nonsymmetric case). 

\smallskip

\emph{Step 3.} 
To conclude the proof, we next use~\eqref{e.wrap.app} with~$\b = \ahom$, which we rewrite here for convenience: for every~$m,n\in\N$ with~$\beta m < n<m$, 
\begin{align}
\label{e.wrap.app.bark}
\biggl| \a(\cu_m) -  \avsum_{z\in 3^n\Zd\cap\cu_m}
\a(z+\cu_n) \biggr|
\leq 
C  \avsum_{z\in 3^n\Zd\cap\cu_m}
\sum_{i=1}^d
J(z+\cu_n,e_i, \ahom e_i)
\,.
\end{align}
Using the~$\CFS(\beta,\Psi)$ condition, we get 
\begin{equation}
\label{e.fluct.splitting.one}
\biggl| \a(\cu_m) -  \avsum_{z\in 3^n\Zd\cap\cu_m}
\a(z+\cu_n) \biggr|
\leq
C 3^{-\alpha n} +\O_\Psi\bigl( C 3^{- \frac d2(m-n)} \bigr)
\end{equation}
and
\begin{equation}
\label{e.fluct.splitting.two}
\biggl|\avsum_{z\in 3^n\Zd\cap\cu_m}
\a(z+\cu_n) - \ahom \biggr|
\leq
C 3^{-\alpha n} +\O_\Psi\bigl( C 3^{- \frac d2(m-n)} \bigr)
\,.
\end{equation}
Combining the previous two displays completes the proof.   
\end{proof}

\begin{exercise}[Improved fluctuation bound]
\label{exercise.improved.fluct}
If~$\P$ satisfies the condition~$\CFS(\beta,\Psi,\Psi(M\cdot))$\footnote{Recall that~$\CFS(\beta,\Psi,\Psi')$ is defined in Definition~\ref{d.CFS.strong}. As explored in Section~\ref{ss.examples}, coefficient fields satisfying~$\CFS(\beta,\Psi,\Psi(M\cdot))$ and~$\CFS(\beta,\Psi,\Psi(M\cdot),\gamma)$ include many natural example finite range fields, approximate finite range fields, and local functions of Gaussian fields.} for some~$M\in[1,\infty)$---a stronger condition than~$\CFS(\beta,\Psi)$---then the fluctuation bound in~\eqref{e.aastar.big.smash.improved} can be improved to the following statement. For every~$\theta < \min\{ 1,\frac d2(1-\beta) \}$, there exists~$C(\theta,M,\data)<\infty$ such that
\begin{equation}
\label{e.fluctuations.up.to.one}
| \a(\cu_m) - \ahom |
+
|\a_*(\cu_m) - \ahom |
\leq 
\O_\Psi \bigl( C3^{-\theta m} \bigr)
\,, \qquad 
\forall m\in\N
\,.
\end{equation}
Here is a sketch of the proof. Suppose that~$\theta>0$ and~$\mathsf{K} \in[1,\infty)$ are such that 
\begin{equation}
\label{e.fluct.ass.theta}
| \a(\cu_m) - \ahom |
+
|\a_*(\cu_m) - \ahom |
\leq 
\O_\Psi \bigl( \mathsf{K} 3^{-\theta m} \bigr) 
\,, \quad \forall m\in\N\,.
\end{equation}
Show that, in place of~\eqref{e.fluct.splitting.one} and~\eqref{e.fluct.splitting.two}, we obtain, for each~$\alpha < \min\{ 1, d(1-\beta) \}$, the existence of a constant~$C(M,\alpha,\data)<\infty$ such that
\begin{equation}
\label{e.fluct.splitting.one.CFS.Psi.Psi}
\biggl| \a(\cu_m) -  \avsum_{z\in 3^n\Zd\cap\cu_m}
\a(z+\cu_n) \biggr|
\leq
C 3^{-\alpha n} +\O_\Psi\bigl( C \mathsf{K} 3^{- \frac d2(m-n) - \theta n} + C3^{-\frac d2(1-\beta)m} \bigr)
\end{equation}
and
\begin{equation}
\label{e.fluct.splitting.two.CFS.Psi.Psi}
\biggl|\avsum_{z\in 3^n\Zd\cap\cu_m}
\a(z+\cu_n) - \ahom \biggr|
\leq
C 3^{-\alpha n} +\O_\Psi\bigl( C\mathsf{K} 3^{- \frac d2(m-n)-\theta n} + C3^{-\frac d2(1-\beta)m} \bigr)
\,.
\end{equation}
Deduce that, if~$\theta< \min\{ \alpha, \frac d2(1-\beta) \}$, then an appropriate choice of~$n$ gives, for some~$\ep>0$, 
\begin{equation*}
| \a(\cu_m) - \ahom |
+
|\a_*(\cu_m) - \ahom |
\leq 
\O_\Psi \bigl( C\mathsf{K} 3^{-(\theta+\ep) m} \bigr) 
\end{equation*}
and therefore the exponent~$\theta$ in~\eqref{e.fluct.ass.theta} self-improves provided that~$\theta< \min\{ 1, \frac d2(1-\beta) \}$.
\end{exercise}

\begin{exercise}[Improved fluctuation bound II]
\label{exercise.improved.fluct.2}
If~$\P$ satisfies~$\CFS(\beta,\Psi,\Psi(M\cdot),\gamma)$ for some~$M\in[1,\infty)$ and~$\gamma>0$---a slightly stronger condition than~$\CFS(\beta,\Psi, \Psi(M\cdot))$---then the  exponent~$\theta$ obtained in the previous exercise can be improved to~$\theta \in (0,1) \cap (0, \frac d2(1-\beta) ]$, with the constants~$C$ in~\eqref{e.fluctuations.up.to.one} now depending additionally on a lower bound for~$\gamma$. This turns out to be the \emph{optimal} fluctuation bound in the case that~$\frac d2(1-\beta) < 1$, and it leads to optimal quantitative homogenization estimates.
\end{exercise}

\begin{exercise}
\label{exercise.corrector.bounds}
Using the result of Proposition~\ref{p.alpha.almostone} and the previous two exercises, obtain bounds on the first-order correctors under three different assumptions on~$\P$, namely that~$\P$ satisfies~(i) $\CFS(\beta,\Psi)$,~(ii) $\CFS(\beta,\Psi,\Psi(M\cdot))$, and~(iii) $\CFS(\beta,\Psi,\Psi(M\cdot),\gamma)$ for~$\gamma>0$. 
Hint: obtain a bound on~$\mathcal{E}(m)$ defined in~\eqref{e.mcE.0.nonsymm}, and then apply~\eqref{e.FCV.nonsymm.grad} and~\eqref{e.FCV.nonsymm.flux}. 
\end{exercise}

\subsection{Higher-order regularity and Liouville theorems}
\label{ss.higherorder.reg}

We next present the higher-order Liouville theorem and large-scale~$C^{k,1}$ regularity theory estimate. Except for the presence of the parameter~$\beta$, the more general mixing condition, and the slightly sharper integrability of~$\X$ in~\eqref{e.X}, the statement of the following theorem is a cut-and-paste from~\cite[Theorem 3.8]{AKMBook} apart from the weak norm bounds.

\smallskip

Analogously to~\eqref{e.Ak.def}, we define 
\begin{equation}  
\label{e.Ahomk.def}
\Ahom_k  = \Bigl \{ u \in H^1_{\mathrm{loc}}(\R^d) \, : \, -\nabla \cdot \ahom \nabla u = 0, \; \lim_{r\to \infty} r^{-k-1} \left\| u \right\|_{\underline{L}^2 \left( B_{r} \right)} = 0  \Bigr \} \,.
\end{equation}
By the classical Liouville theorem, this is the space of~$\ahom$-harmonic polynomials of degree at most~$k$.
To see this, one may observe that~$\nabla^{k+1} u$ must vanish for each~$u \in \Ahom_k$ by the classical pointwise bounds on the derivatives of~$\ahom$-harmonic functions:~$\|\nabla^{k+1} u\|_{L^\infty(B_{r})} \leq C(k,d,\lambda,\Lambda)r^{-k-1} \left\| u \right\|_{\underline{L}^2 \left( B_{2r} \right)}$.

\begin{theorem}[{Large-scale~$C^{k,1}$ regularity}]
\label{t.Ck1}
Assume that~$\P$ is a~$\Zd$--stationary measure on~$(\Omega,\F)$ and satisfies~$\CFS(\beta,\Psi)$. 
There exist~$C(\dataref)<\infty$, an exponent~$\alpha(\beta,d,\lambda,\Lambda)>0$ and a random variable~$\X$ satisfying
\begin{equation}
\label{e.X}
\X^{\frac d2(1-\beta)  } 
= 
\O_{\Psi}(C) 
\end{equation}
such that the following statements hold, for every~$k\in\N$ and~$x \in \cu_1$:
\begin{enumerate}
\item[{$\mathrm{(i)}_k$}] There exists~$C(k,\beta,d,\lambda,\Lambda)<\infty$ such that, for every~$u \in \A_k$, there exists~$p\in \overline{\A}_k$ such that, for every~$R\geq \X$,
\begin{multline} \label{e.liouvillec}
\left\| u - p \right\|_{\underline{L}^2(B_R(x))} 
+
\left\|  \nabla u - \nabla p \right\|_{\underline{H}^{-1}(B_R(x))} 
+
\left\|  \a \nabla u - \ahom \nabla p \right\|_{\underline{H}^{-1}(B_R(x))} 
\\ 
\leq 
C \Bigl( \frac{R}{\X} \Bigr)^{-\alpha}
\left\| p \right\|_{\underline{L}^2(B_R(x))}
\,.
\end{multline}

\item[{$\mathrm{(ii)}_k$}]For every~$p\in \overline{\A}_k$, there exists~$u\in \A_k$ satisfying~\eqref{e.liouvillec} for every~$R\geq \X$. 

\item[{$\mathrm{(iii)}_k$}]
There exists~$C(k,\beta,d,\lambda,\Lambda)<\infty$ such that, for every~$R\geq \X$ and~$u\in \A(B_R)$, there exists~$\phi \in \A_k$ such that, for every~$r \in \left[ \X,  R \right]$, we have the estimate
\begin{equation}
\label{e.intrinsicreg}
r \left\| \nabla u - \nabla \phi \right\|_{\underline{L}^2(B_{r/2}(x))}
\leq 
C \left\| u - \phi \right\|_{\underline{L}^2(B_r(x))} 
\leq 
C \Bigl( \frac r R \Bigr)^{\!k+1} 
\left\| u \right\|_{\underline{L}^2(B_R(x))}.
\end{equation}
\end{enumerate}
In particular,~$\P$-almost surely, we have that\footnote{The second equality of~\eqref{e.dimensionofAk} holds in the case of equations. For systems we only have the first equality.}, for every~$k\in\N$,
\begin{equation} 
\label{e.dimensionofAk}
\dim(\A_k) = \dim(\Ahom_k) =  \binom{d+k-1}{k} + \binom{d+k-2}{k-1}.
\end{equation}
\end{theorem}

The proof of Theorem~\ref{t.Ck1} we give follows closely the presentation of~\cite[Proof of Theorem 3.8]{AKMBook}, although the argument is much more succinct here. Some small technical details that are omitted can be found in~\cite{AKMBook}.

\begin{proof} 
Without loss of generality, we may take~$x=0$. The proof is via induction, the case~$k=0$ being valid by Theorem~\ref{t.C01}. We therefore assume that for some~$k\in\N$, we have~$\mathrm{(i)}_{k-1}$,~$\mathrm{(ii)}_{k-1}$ and slightly weaker version of~$\mathrm{(iii)}_{k-1}'$, namely
\begin{enumerate}
\item[{$\mathrm{(iii)}_k'$}]
There exists~$C(k,\beta,d,\lambda,\Lambda)<\infty$ such that, for every~$R\geq \X$ and~$u\in \A(B_R)$, there exists~$\phi \in \A_k$ such that, for every~$r \in \left[ \X, R \right]$, we have the estimate
\begin{equation}
\label{e.intrinsicreg.prime}
r \left\| \nabla u - \nabla \phi \right\|_{\underline{L}^2(B_{r/2})}  
+ 
\left\| u - \phi \right\|_{\underline{L}^2(B_r)} 
\leq 
C \left( \frac r R \right)^{k+1-\alpha} 
\left\| u \right\|_{\underline{L}^2(B_R)} \, .
\end{equation}
\end{enumerate}
The parameter~$\alpha(\beta,d,\lambda,\Lambda) \in (0,\frac12)$ will be fixed in~\eqref{e.Ckone.fixalpha}.  

\smallskip

The induction step breaks into the following three implications:
\begin{align}   
\label{e.ind1}
& \mathrm{(i)}_{k-1}, \ \mathrm{(ii)}_{k-1} \ \ \text{and} \ \ \mathrm{(iii)}_{k-1}'  \implies \mathrm{(ii)}_{k}
\,,
\\ 
\label{e.ind2}
& \mathrm{(i)}_{k-1} \ \ \text{and} \ \ \mathrm{(ii)}_{k}   \implies \mathrm{(i)}_{k}
\,, \\ 
\label{e.ind3}
& \mathrm{(i)}_{k}\ \ \text{and} \ \ \mathrm{(ii)}_{k}  \implies \mathrm{(iii)}_{k}'
\,.
\end{align}
Once the induction loop is closed, we upgrade~$\mathrm{(iii)}_{k}'$ to the stronger statement~$\mathrm{(iii)}_{k}$ by appealing to a fourth implication, namely 
\begin{align}
\label{e.ind4}
\mathrm{(i)}_{k+1}, \ \mathrm{(ii)}_{k+1} \ \ \text{and} \ \ \mathrm{(iii)}_{k+1}'  \implies \mathrm{(iii)}_{k}
\,.
\end{align}
Notice that, by the inequality 
\begin{align*}  
\| \nabla f \|_{\underline{H}^{-1}(B_R)} \leq C  \|  f \|_{\underline{L}^{2}(B_R)} 
\,,
\end{align*}
it is enough to prove statement~\eqref{e.liouvillec} only for the~$L^2$-term and for the fluxes. 

\smallskip

\emph{Step 1. Proof of~\eqref{e.ind1}.}  Set~$r_m := 2^m \X$. By combining a rescaled version of Theorems~\ref{t.quant.DP} and~\ref{t.quant.DP.nosymm},  we find small~$\alpha(\beta,d,\lambda,\Lambda) \in (0,\frac12)$ such that, for every~$q \in \Ahom_k$ and~$m \in \N$ there exists~$v_m \in \A(B_{r_m})$ such that
\begin{equation} \label{e.Ckone.fixalpha}
\left\| v_{m} - q \right\|_{\underline{L}^2(B_{r_m})} 
+ 
\left\|  \a \nabla v_m - \ahom \nabla q \right\|_{\underline{H}^{-1}(B_{r_m})} 
\leq 
C_k \Bigl( \frac{r_m}{\X} \Bigr)^{-2\alpha} \left\| q \right\|_{\underline{L}^2(B_{r_m})}
\, .
\end{equation}
Notice that this choice of~$\alpha$ fixes the parameter in~$\mathrm{(iii)}_{k-1}'$ as well. 
Indeed, we take boundary values~$f = q$ in Theorem~\ref{t.quant.DP}, we see by the uniqueness of the Dirichlet problem that~$\bar u = q$. To conclude with~\eqref{e.Ckone.fixalpha}, we then use~$r \left\| \nabla q \right\|_{L^\infty(B_r)} + r^2 \left\| \nabla^2 q \right\|_{L^\infty(B_r)} \leq C_k \left\| q \right\|_{\underline{L}^2(B_r)}$ to estimate the terms on the right in~\eqref{e.quant.DP}. 
Next, it follows from~\eqref{e.Ckone.fixalpha}, by the triangle inequality and~$\mathrm{(iii)}_{k-1}'$, that, for every~$n \in\N$, there exists~$\phi_{n} \in \A_{k-1}$ such that, for every~$r \in  [\X , r_n]$,   
\begin{align*}  
\left\| v_{n+1} {-} v_{n} {-} \phi_{n} \right\|_{\underline{L}^2(B_r)}   
\leq 
C
\Bigl(\frac{r}{r_n}\Bigr)^{k-\alpha}  \left\| v_{n+1} {-} v_{n} \right\|_{\underline{L}^2(B_{r_n})} 
\leq 
C \Bigl( \frac{r_n}{\X} \Bigr)^{-\alpha}  \left\| q \right\|_{\underline{L}^2(B_r)}
\,.
\end{align*}
Using~$\mathrm{(i)}_{k-1}$, the above two displays and the triangle inequality together with the homogeneity of~$q$, we get, for every~$r \in [r_n \vee C\X,\infty)$,  
\begin{align*}  
\| \phi_{n} \|_{\underline{L}^2(B_r)} 
& 
\leq 
C\Bigl( \frac{r}{r_n} \Bigr)^{k-1} \| \phi_{n}  \|_{\underline{L}^2(B_{r_n})} 
\\ &
\leq 
C \Bigl( \frac{r}{r_n} \Bigr)^{k-1}  \Bigl( \frac{r_n}{\X} \Bigr)^{-\alpha}  \left\| q \right\|_{\underline{L}^2(B_{r_n})}
=
C \Bigl( \frac{r}{\X} \Bigr)^{-\alpha}  \Bigl( \frac{r_n}{r} \Bigr)^{1-\alpha}  \left\| q \right\|_{\underline{L}^2(B_{r})}
\,.
\end{align*}
Setting~$u_n = v_n - \sum_{j=1}^{n-1} \phi_j$, so that~$u_n = v_m +  \sum_{j=m}^{n-1}(v_{j+1} - v_j -\phi_j) + \sum_{j=1}^{m-1} \phi_j$,  the triangle inequality yields, for every~$m,n \in \N$ with~$m \leq n$, that
\begin{align*}  
\| u_n - q \|_{\underline{L}^2(B_{r_m})}  
& \leq 
\| v_m - q \|_{\underline{L}^2(B_{r_m})} 
+ \sum_{j=m}^{n-1}  \| v_{j+1} {-} v_j {-} \phi_j \|_{\underline{L}^2(B_{r_m})}
+ \sum_{j=1}^{m-1}\| \phi_j \|_{\underline{L}^2(B_{r_m})}
\notag \\ &
\leq 
C\biggl(1+ \sum_{j=m}^{n-1} \Bigl( \frac{r_m}{r_j}\Bigr)^{\! \alpha} +  \sum_{j=1}^{m-1} \Bigl( \frac{r_j}{r_m} \Bigr)^{1-\alpha}    \biggr)  \Bigl( \frac{r_m}{\X} \Bigr)^{-\alpha}   \left\| q \right\|_{\underline{L}^2(B_{r_m})}
\notag \\ &
\leq
C \Bigl( \frac{r_m}{\X} \Bigr)^{-\alpha}   \left\| q \right\|_{\underline{L}^2(B_{r_m})}
\,.
\end{align*}
Very similarly, using also the Caccioppoli inequality, 
\begin{align}  
\lefteqn{
\left\|  \a \nabla u_n - \ahom \nabla q \right\|_{\underline{H}^{-1}(B_{r_m})} 
} \quad &
\notag \\ & 
\leq
\left\|  \a \nabla v_m {-} \ahom \nabla q \right\|_{\underline{H}^{-1}(B_{r_m})} 
 + \sum_{j=m}^{n-1}  \| \a \nabla (v_{j+1} {-} v_j {-} \phi_j) \|_{\underline{H}^{-1}(B_{r_m})} 
+  \sum_{j=1}^{m-1 } \| \a \nabla \phi_j \|_{\underline{H}^{-1}(B_{r_m})} 
\notag \\ & 
\leq
\left\|  \a \nabla v_m {-} \ahom \nabla q \right\|_{\underline{H}^{-1}(B_{r_m})} 
 + C r_m\biggl( \sum_{j=m}^{n-1}   \|\nabla (v_{j+1} {-} v_j {-} \phi_j) \|_{\underline{L}^2(B_{r_m})}
+  \sum_{j=1}^{m-1 }  \| \nabla \phi_j \|_{\underline{L}^2(B_{r_m})} \biggr)
\notag \\ & 
\leq 
C \Bigl( \frac{r_m}{\X} \Bigr)^{-\alpha}   \left\| q \right\|_{\underline{L}^2(B_{r_m})}
\,.
\end{align}
By compactness, the limit~$u = \lim_{n \to \infty} u_n$ exists in~$\mathcal{A}(\R^d)$, up to a subsequence,  and~$u$ is the desired function satisfying the estimate in~\eqref{e.liouvillec} by the above bounds. 

\smallskip

\emph{Step 2. Proof of~\eqref{e.ind2}.} We first show the following weaker statement. There exists a constant~$C_k(\dataref)<\infty$ such that, for every~$u\in \A_k$ there exists~$p \in \Ahom_k$ satisfying, for every~$R\geq \X$,
\begin{equation}  \label{e.ind2.pre}
\left\| u - p \right\|_{\underline{L}^2(B_R)} 
\leq 
C \Bigl( \frac{R}{\X} \Bigr)^{\!-\alpha}
\left\| p \right\|_{\underline{L}^2(B_R)} \,.
\end{equation}
The full statement~\eqref{e.ind2} follows then from the above estimate and~$\mathrm{(ii)}_k$ as follows. 
Take~$p \in \overline{\A}_k$ be as in~\eqref{e.ind2.pre}. Assuming~$R \geq C_k \X$ for some~$C_k(\dataref)<\infty$, by~$\mathrm{(ii)}_k$ there exists~$\tilde u \in  \A_k$ such that
\begin{equation}  \label{e.ind2.prepre}
\left\| \tilde u - p \right\|_{\underline{L}^2(B_R)} 
+
\left\|  \a \nabla \tilde u - \ahom \nabla p \right\|_{\underline{H}^{-1}(B_R)} 
\leq
C \Bigl( \frac{R}{\X} \Bigr)^{\!-\alpha}
\left\| p \right\|_{\underline{L}^2(B_R)} 
\,.
\end{equation}
Thus, by the triangle inequality and the Caccioppoli inequality,
\begin{align*}  
\left\|  \a \nabla u - \ahom \nabla p \right\|_{\underline{H}^{-1}(B_R)} 
& 
\leq
\left\|  \a \nabla ( u - \tilde u )\right\|_{\underline{H}^{-1}(B_R)} 
+ 
\left\|  \a \nabla \tilde u - \ahom \nabla p \right\|_{\underline{H}^{-1}(B_R)} 
\\ & 
\leq 
C 
\left\| u - \tilde u \right\|_{\underline{L}^2(B_{2R})}  
+ 
\left\|  \a \nabla \tilde u - \ahom \nabla p \right\|_{\underline{H}^{-1}(B_R)}
 \,,
\end{align*}
which, together with~\eqref{e.ind2.pre},~\eqref{e.ind2.prepre} and the triangle inequality, proves~\eqref{e.ind2} by giving up a volume factor if~$R \in [\X,C_k^{\nicefrac1\alpha}\X]$. 

\smallskip

To prove~\eqref{e.ind2.pre}, we will employ the harmonic approximation. 
By Proposition~\ref{p.harm.we.have}, for every~$n \in \N$ we find~$\theta(d,\lambda,\Lambda) \in (0,\tfrac12]$~and~$C_k(d,\lambda,\Lambda)<\infty$ such that, for every~$r \in [ C_n \X,\infty)$, 
\begin{equation*}  
E_{n}(\theta r)
\leq 
\frac16 \theta^{1-\alpha} E_{n}(r) 
+ \frac16  \Bigl( \frac{r}{\X} \Bigr)^{-\alpha} r^{-k} \left\| p_{r,k} \right\|_{\underline{L}^2(B_r)} \, , \qquad  E_{n}(r) := \inf_{p \in \Ahom_n} r^{-n} \| u - p \|_{\underline{L}^2(B_{r})} 
\,.
\end{equation*}
Here~$p_{r,n} \in \Ahom_n$ is the~$\ahom$-harmonic polynomial minimizing~$E_{n}(r)$.
We integrate the above inequality with~$n= k+1$ against the Haar measure~$\frac{dr}{r}$ and change variables to obtain, for every~$r,R \in \R$ with~$r \in [ C_n \X,\infty)$ and~$R \geq \theta^{-1} r$,
\begin{align*} 
\int_{\theta r}^{\theta R} E_{k+1}(t) \, \frac{dt}{t}
&
=
\int_{r}^{R} E_{k+1}(\theta t) \, \frac{dt}{t}
\notag \\ &
\leq 
\frac12 \int_{r}^{\theta R} E_{k+1}(t) \, \frac{dt}{t}
+ 
\frac12 \int_{\theta R}^R E_{k+1}(t) \, \frac{dt}{t}
+ 
\frac12 \int_{r}^{R}  \Bigl( \frac{t}{\X} \Bigr)^{-\alpha} \left\| p_{t,k+1} \right\|_{\underline{L}^2(B_t)} \, \frac{dt}{t^{k+2}}
\,.
\end{align*}
Thus, after reabsorption and observing that the second term converges to zero as~$R \to \infty$ since~$t^{-(k+1)} \| u \|_{\underline{L}^{2}(B_t)} \to 0$ as~$t \to \infty$, we deduce that
\begin{equation} \label{e.Ekplusone.est.pre}
\int_{\theta r}^\infty E_{k+1}(t) \, \frac{dt}{t} 
\leq 
 \int_{r}^\infty  \Bigl( \frac{t}{\X} \Bigr)^{-\alpha} \left\| p_{t,k+1} \right\|_{\underline{L}^2(B_t)} \, \frac{dt}{t^{k+2}}
 \,.
\end{equation}
Next, for~$j,k \in\N$ and~$p \in \Ahom_k$, let~$\pi_j p(x) := \sum_{|\alpha|=j} \frac1{\alpha!}(\partial_{\alpha} p) (0) x^{\alpha}$ be the~$j$th degree part of~$p$. We set
\begin{equation*} 
\omega_j(r):= 
\int_{r}^\infty \Bigl( \frac{r}{t} \Bigr)^{k+1+\alpha} \left\| \pi_j p_{t,k} \right\|_{\underline{L}^2(B_t)} \, \frac{dt}{t} 
\qand 
\omega(r) 
:=
\sum_{j=0}^k 
\omega_j(r)
\,.
\end{equation*}
By the triangle inequality and~$\ahom$-harmonicity of~$p_{t,k}$, we have 
\begin{equation*} 
\int_{r}^\infty \Bigl( \frac{r}{t} \Bigr)^{k+1+\alpha} \left\| p_{t,k} \right\|_{\underline{L}^2(B_t)} \, \frac{dt}{t}  
\leq 
\omega(r) 
\leq 
C_k \int_{r}^\infty \Bigl( \frac{r}{t} \Bigr)^{k+1+\alpha} \left\| p_{t,k} \right\|_{\underline{L}^2(B_t)} \, \frac{dt}{t}  
 \,.
\end{equation*}
With this notation,~\eqref{e.Ekplusone.est.pre} implies that
\begin{equation} \label{e.Ekplusone.est} 
\int_{\theta r}^\infty E_{k+1}(t) \, \frac{dt}{t} 
\leq  
\Bigl( \frac{r}{\X} \Bigr)^{-\alpha} r^{-(k+1)} \omega(r)
+
C \int_{r}^\infty \Bigl( \frac{t}{\X} \Bigr)^{-\alpha} \bigl | \nabla^{k+1} p_{t,k+1} \bigr | \, \frac{dt}{t} 
\,.
\end{equation}
We next argue that the last term on the right is small. For this, notice first that, for every~$s,t \in \R$ with~$0< t \leq s \leq 2t$, 
\begin{equation*} 
\bigl| \nabla^{k+1} p_{s,k+1} - \nabla^{k+1}  p_{t,k+1} \bigr|
\leq 
Cs^{-(k+1)}\| p_{s,k+1} - p_{t,k+1} \|_{\underline{L}^2(B_t)} 
\leq
C E_{k+1}(2t) \,.
\end{equation*}
Now,~$|\nabla^{k+1} p_{r,k+1}| \to 0$ as~$r\to \infty$ since~$r^{-k-1} \| u \|_{\underline{L}^{2}(B_r)} \to 0$ as~$r \to \infty$. Thus, we also obtain, again by the triangle inequality and telescope summation, that
\begin{equation*}  
\sup_{t \in [r ,\infty) }|\nabla^{k+1} p_{t,k+1}| 
\leq 
C \int_{r/2}^\infty \sup_{s \in (t,2t)} |\nabla^{k+1} (p_{s,k+1} - p_{t,k+1}) | \, \frac{dt}{t}    
\leq 
C \int_{\theta r}^\infty E_{k+1}(t) \, \frac{dt}{t}
\,.
\end{equation*} 
Combining the above displays, we deduce that
\begin{equation} \label{e.Ekplusone.est.again}
\int_{\theta r}^\infty E_{k+1}(t) \, \frac{dt}{t}  
\leq 
C \int_{\theta r}^\infty \Bigl( \frac{t}{\X} \Bigr)^{-\alpha}  \, \frac{dt}{t}  
\int_{\theta r}^\infty E_{k+1}(t) \, \frac{dt}{t}  
+
\Bigl( \frac{r}{\X} \Bigr)^{-\alpha} r^{-(k+1)} \omega(r)
\,,
\end{equation}
and we may reabsorb the first term on the right using the lower bound~$r \geq C_k \X$. Collecting estimates so far, we have shown that
\begin{equation}  
\int_{\theta r}^\infty E_{k+1}(t) \, \frac{dt}{t} + \sup_{t \in [r ,\infty) }|\nabla^{k+1} p_{t,k+1}| 
\leq   
C \Bigl( \frac{r}{\X} \Bigr)^{-\alpha} r^{-(k+1)} \omega(r)  \,.
\end{equation}
The next step is to estimate the size of~$\omega(r)$, and for that we estimate homogeneous parts separately. By changing variables, using the above display, homogeneity and the triangle inequality, we deduce that
\begin{align} \notag 
\omega_j(2r)  & 
= 
\int_{2r}^\infty \Bigl( \frac{2r}{t} \Bigr)^{k+1+\alpha} \left\| \pi_j p_{t,k} \right\|_{\underline{L}^2(B_t)} \, \frac{dt}{t} 
\notag \\ &
=
2^j 
\int_{r}^\infty \Bigl( \frac{r}{t} \Bigr)^{k+1+\alpha} \left\| \pi_j p_{2 t,k} \right\|_{\underline{L}^2(B_{t})} \, \frac{dt}{t} 
\notag \\ &
\leq 
2^j \omega_j(r) 
+ 
C 
\int_{r}^\infty \Bigl( \frac{r}{t} \Bigr)^{k+1+\alpha} \left\| p_{2 t,k} - p_{t,k} \right\|_{\underline{L}^2(B_{t})} \, \frac{dt}{t} 
\notag \\ &
\leq 
2^j \omega_j(r) 
+
C r^{k+1}
\int_{r}^\infty  \Bigl( \frac{r}{t} \Bigr)^{\! \alpha} \Bigl( E_{k+1}(t)  +  \sup_{t \in [r ,\infty) }|\nabla^{k+1} p_{t,k+1}| \Bigr) \, \frac{dt}{t}
\notag \\ &
\leq 
2^j \omega_j(r)  + C \Bigl( \frac{r}{\X} \Bigr)^{-\alpha} \omega(r)
\notag
\,.
\end{align}
Thus, after summation and a simple iteration argument, we deduce that, for every~$r,t \in [C_k\X,\infty)$ with~$t\geq r$, 
\begin{equation}  \label{e.omega.bounds}
 \omega(t) \leq C \Bigl( \frac{t}{r} \Bigr)^k  \omega(r) 
 \qand
 \sum_{j=0}^{k-1} \omega_j(t) 
 \leq
 C \Bigl( \frac{t}{r} \Bigr)^{k-\alpha} \omega(r) 
 \,. 
\end{equation} 
Next, since~$E_k(r) \leq Cr \int_{r}^{2r} E_{k+1}(t) \, \frac{dt}{t} + C r \sup_{t \in [r ,2r] }|\nabla^{k+1} p_{t,k+1}|$, 
we get by the previous display and~\eqref{e.Ekplusone.est.again} that
\begin{equation}  \label{e.Ek.est}
\int_{r}^\infty E_k(t) \, \frac{dt}{t} 
\leq 
C \int_{r}^\infty \Bigl( \frac{t}{\X} \Bigr)^{-\alpha} t^{-k} \omega(t) \, \frac{dt}{t}  
\leq  
C \Bigl( \frac{r}{\X} \Bigr)^{-\alpha} r^{-k} \omega(r) 
\end{equation}
Comparing this to~\eqref{e.Ekplusone.est}, we see how using~$r^{-k-1} \| u \|_{\underline{L}^{2}(B_r)} \to 0$ as~$r \to \infty$ allowed us to lower the degree from~$k+1$ to~$k$.

Next, applying~\eqref{e.Ek.est} it is possible to identify a homogeneous polynomial~$\tilde p_{k} \in \Ahom_k$ by showing that~$\{ \nabla^k p_{t,k}\}_t$ is a Cauchy sequence. Indeed, by the triangle inequality, we have as above that, for every~$r\in [C_k \X,\infty)$, 
\begin{align*} 
\lefteqn{\sup_{t \in (2r,\infty)}
\bigl| \nabla^{k} p_{t,k} - \nabla^{k}  p_{r,k} \bigr|
} \qquad &
\notag \\  & 
\leq 
C \int_{r}^\infty \sup_{s \in (t,2t)} |\nabla^{k} (p_{s,k} - p_{t,k}) | \, \frac{dt}{t}    
\leq 
C \int_{r}^\infty E_{k}(t) \, \frac{dt}{t}
\leq
C \Bigl( \frac{r}{\X} \Bigr)^{-\alpha} r^{-k} \omega(r) \,.
\end{align*}
The left side tends to zero as~$r\to \infty$ and, hence, there exists a homogeneous~$\tilde p_k \in \Ahom_k$ such that, for every~$r\in [C_k \X,\infty)$, 
\begin{equation*} 
\bigl| \nabla^{k} p_{r,k} - \nabla^{k}  \tilde p_{k} \bigr| 
\leq 
C \Bigl( \frac{r}{\X} \Bigr)^{-\alpha} r^{-k} \omega(r) \,.
\end{equation*}
That~$\tilde p_k$ belongs to~$\Ahom_k$ is a consequence of the fact that~$\pi_k p_{r,k} \in \Ahom_k$. 
Combining the above display with~\eqref{e.Ek.est} and~\eqref{e.omega.bounds} yields 
\begin{align*} 
\lefteqn{
r^{-k} \left\| u - \tilde p_{k} \right\|_{\underline{L}^2(B_r)}  
\leq
C
\int_{r}^{2r} t^{-k} \left\| u - \tilde p_{k} \right\|_{\underline{L}^2(B_t)}\, \frac{dt}{t} 
} \qquad &
\notag\\ &
\leq
C 
\int_{r}^{2r} E_k(t) \, \frac{dt}{t} 
+
C \int_{r}^{2r} \bigl | \nabla^k p_{t,k} - \nabla^k \tilde p_{k} \bigr| \, \frac{dt}{t} 
+ 
C \sum_{j=0}^{k-1} \int_{r}^{2r}  t^{-k} \left\| \pi_j p_{t,k} \right\|_{\underline{L}^2(B_t)} \, \frac{dt}{t} 
\notag\\ &
\leq
C \Bigl( \frac{r}{\X} \Bigr)^{-\alpha} r^{-k} \omega(r) \,.
\end{align*}
Finally, given~$\tilde p_k$ as above and using~$\mathrm{(ii)}_{k}$, there exists a corrector~$\tilde \phi_k \in \A_k$ such that, for every~$r \in [C \X,\infty)$, we have~$\| \tilde \phi_k -\tilde p_{k} \|_{\underline{L}^2(B_r)} \leq \frac12 ( \frac{r}{\X} )^{-\alpha}  \| \tilde p_{k} \|_{\underline{L}^2(B_r)}$ and, by the previous display, also~$u - \tilde \phi_k \in \A_{k-1}$. Therefore, by~$\mathrm{(i)}_{k-1}$, we find~$\tilde p_{k-1} \in \Ahom_{k-1}$ such that, for every~$r \in [\X,\infty)$,~$\| u - \tilde \phi_k -\tilde p_{k-1} \|_{\underline{L}^2(B_r)} \leq C ( \frac{r}{\X} )^{-\alpha} \| \tilde p_{k-1} \|_{\underline{L}^2(B_r)}$. Then~\eqref{e.ind2.pre} follows by the triangle inequality with the polynomial~$p = \tilde p_k + \tilde p_{k-1}$ using the equivalence of norms as~$\sum_{j=0}^k \| \pi_j p \|_{\underline{L}^2(B_r)} \leq C(k,d,\lambda,\Lambda) \| p \|_{\underline{L}^2(B_r)}$, which is a consequence of the Caccioppoli estimate for~$p$. 

\smallskip

\emph{Step 3. Proof of~\eqref{e.ind3}.} Set~$r_j = C_k \theta^{-j} \X$. As in Step 2, using harmonic approximation, we find, for every~$u_{j+1} \in \A(B_{r_{j+1}})$, that there exists~$p_j \in \Ahom_k$ such that 
\begin{equation*}  
\inf_{p \in \Ahom_k} \| u_{j+1} - p \|_{\underline{L}^2(B_{r_{j}})}  
=
\| u_{j+1} - p_{j} \|_{\underline{L}^2(B_{r_{j}})} 
\leq 
\frac23 
 \theta^{k+1-\alpha} 
  \| u_{j+1} \|_{\underline{L}^2(B_{r_{j+1}})} 
\,.
\end{equation*}
By~$\mathrm{(ii)}_{k}$, we then find~$\phi_j \in \A_{k}$ such that~$\| \phi_j -p_j  \|_{\underline{L}^2(B_{r_{j}})} \leq C ( \frac{r_j}{\X} )^{-\alpha} \| u_{j+1}  \|_{\underline{L}^2(B_{r_{j+1}})}$. Thus, setting~$u_{j} = u_{j+1} - \phi_{j}$, we obtain that 
\begin{equation*}  
\| u_{j} \|_{\underline{L}^2(B_{r_{j}})} 
\leq 
\left( \frac{r_j}{r_{j+1}} \right)^{k+1-\alpha}
\| u_{j+1} \|_{\underline{L}^2(B_{r_{j}})} 
\,.
\end{equation*}
Letting~$J \in \N$ be such that~$3^{J} C \X \leq R < 3^{J+1} C \X$ (if such~$J$ does not exist, the result follows trivially), we select~$u_J = u$ so that~$u_j = u - \sum_{i=j}^{J} \phi_j$. Therefore, the above inequality yields that, for every~$r\in [\X,R]$, there exists~$\phi_r \in \A_k$ such that 
\begin{equation*}  
 \| u - \phi_r \|_{\underline{L}^2(B_{r})} 
\leq C 
\left( \frac{r}{R} \right)^{k+1-\alpha}
\| u \|_{\underline{L}^2(B_{R})} 
.
\end{equation*}
Finally, applying both~$\mathrm{(i)}_{k}$ and~$\mathrm{(ii)}_{k}$, together with the triangle inequality, give that, for every~$t \in [\X/2,2r]$, 
\begin{equation*}  
\left\| \phi_{2t} - \phi_t \right\|_{\underline{L}^2(B_{r})} 
\leq 
C \Bigl(\frac{r}{t}\Bigr)^{k} 
\Bigl( \frac{t}{R} \Bigr)^{k+1-\alpha} \| u \|_{\underline{L}^2(B_{R})} 
=
C 
\Bigl( \frac{t}{r} \Bigr)^{1-\alpha}
\Bigl( \frac{r}{R} \Bigr)^{k+1-\alpha} \| u \|_{\underline{L}^2(B_{R})} ,
\end{equation*}
and defining~$\phi = \phi_\X$ leads to the result since we may sum over~$t$ to get~$\int_{\X/2}^{2r} ( \frac{t}{r})^{1-\alpha} \, \frac{dt}{t} \leq C$, so that, for every~$r \in [\X,R]$, 
\begin{equation*}  
\left\| u - \phi \right\|_{\underline{L}^2(B_{r})} \leq C \int_{\X/2}^{2r} \left\| \phi_{2t} - \phi_{t} \right\|_{\underline{L}^2(B_{r})} \, \frac{dt}{t}+ \left\| u - \phi_r \right\|_{\underline{L}^2(B_{r})}
\leq C 
\Bigl( \frac{r}{R} \Bigr)^{k+1-\alpha} 
\| u \|_{\underline{L}^2(B_{R})} 
\,,
\end{equation*}
which was to be proven. 

\smallskip

\emph{Step 4. Proof of~\eqref{e.ind4}.} By~$\mathrm{(i)}_{k+1}$,~$\mathrm{(ii)}_{k+1}$ and~$\mathrm{(iii)}_{k+1}'$ we find~$\phi \in \A_k$ and~$\tilde \phi \in \A_{k+1}$ such that, for every~$r \in [\X,R]$, 
\begin{equation*}  
\| u - \phi - \tilde \phi \|_{\underline{L}^2(B_{r})} 
\leq 
C 
\Bigl( \frac{r}{R} \Bigr)^{k+2-\alpha} 
\| u \|_{\underline{L}^2(B_{R})} 
\quad \mbox{and} \quad 
\| \tilde \phi \|_{\underline{L}^2(B_{r})} 
\leq 
C \left( \frac{r}{R} \right)^{k+1} 
\| u \|_{\underline{L}^2(B_{R})} 
\,.
\end{equation*}
The proof can now be completed by the triangle inequality. 
\end{proof}

The previous theorem has a localized version. Let~$U$ be an open and bounded set.  We collect in~$\A_k (U)$ solutions with~$\ahom$-harmonic boundary values, that is, 
\begin{equation}
\label{e.localAk}
\A_k (U) := \left\{ w \in \A(U) \, : \, w = p + H_0^1(U)  \mbox{ for some } p\in \Ahom_k \right\}.
\end{equation}
Observe that~$\dim(\A_k(U)) = \dim(\A_k)  = \dim(\Ahom_k)$. We canonically extend every~$w\in \A_k (U)$ with boundary values~$p\in \Ahom_k$ to the whole~$\R^d$ as~$p$. 
Notice that~$\A_k$ can be seen as the limit of~$\A_k (B_R)$ as~$R\to \infty$. As a matter of fact, this is precisely how we constructed~$\A_k$ in the proof of~\eqref{e.ind1}. The proof is a verbatim repetition of Theorem~\ref{t.Ck1}, using this time the minimal scale from the Lemma~\ref{l.localX} or Remark~\ref{r.localX.nonsymm}, using now a localized version of the harmonic approximation in Proposition~\ref{p.harm.we.have}. 

\begin{theorem}[{Localized large-scale~$C^{k,1}$ regularity}]
\label{t.Ck1.local}
Fix~$R \geq 1$ and let~$\cu = t \cu_0$. Assume that~$\P$ satisfies~$\CFS(\beta,\Psi)$. 
There exist a constant~$C(\dataref)<\infty$, an exponent~$\alpha(\beta,d,\lambda,\Lambda)>0$ and~$\F(\cu)$-measurable random variable~$\X_t$ satisfying
\begin{equation}
\label{e.X.local}
\X_t^{\frac d2(1-\beta)  } 
= 
\O_{\Psi}(C) 
\qand 
|\partial_{\a(\cu)} \X_t |
\leq C(1+\X_t) 
\end{equation}
such that~$\mathrm{(i)}_k$,~$\mathrm{(ii)}_k$ and~$\mathrm{(iii)}_k$ all hold, for every~$k\in\N$ and~$x \in \cu_1$, with~$\X$ replaced by~$\X_t$ and~$\A_k$ replaced by~$\A_k (\cu)$, provided that~$R \in [\X_t,t]$. Finally,~$\cu$ can be replaced by any Lipschitz domain~$U$ with the boundary of~$R^{-1} U$ having a bounded Lipschitz constant. Then the constant~$C_k$ in~$\mathrm{(i)}_k$,~$\mathrm{(ii)}_k$ and~$\mathrm{(iii)}_k$ depend additionally on~$U$. 
\end{theorem}

\subsection*{Historical remarks and further reading}

The large-scale regularity theory in homogenization first appeared in the periodic setting in the work of Avellaneda and Lin~\cite{AL1,AL2,AL3}. The extension to non-periodic media required a more quantitative version of the Campanato iteration they introduced, and this first appeared in~\cite{AS}. After that paper, the regularity theory was quickly developed in a series of works, including~\cite{GNO3,FO,AKM1,AKM,AKMBook}, and became an important cornerstone of the theory. For more references and a more thorough presentation of the large-scale regularity theory, see~\cite[Chapter 3]{AKMBook}. 

\smallskip

The presentation in Section~\ref{ss.almostone} is a new, simpler proof of the result which first appeared in~\cite{AKM1}. The latter paper essentially proved Proposition~\ref{p.alpha.almostone} for any exponent~$\alpha < 1$, using the full statement of Theorem~\ref{t.Ck1} for large~$k$, under the assumption of finite range of dependence setting.

\section{Renormalization theory and optimal quantitative estimates}
\label{s.renormalization}

In this chapter, we turn our attention to the derivation of homogenization error estimates which are optimal in the asymptotic limit of large scale separation. A recurring theme throughout the text, emphasized as early as Section~\ref{s.periodic}, is that quantitative homogenization estimates for very general problems can be reduced by deterministic arguments using two-scale expansions to quantitative estimates on the first-order correctors. To be more precise, what is needed is a rate of convergence for the qualitative limit asserted in~\eqref{e.corrector.qualbound.L2}, which we reproduce here: 
\begin{equation}
\label{e.corrector.qualbound.L2.in.ch7}
\limsup_{m\to \infty}
3^{-m} 
\bigl(
\left\|  \phi_e - (\phi_e)_{\cu_m} \right\|_{\underline{L}^{2}(\cu_m)}
+
\left\|  \bfs_e - (\bfs_e)_{\cu_m} \right\|_{\underline{L}^{2}(\cu_m)}
\bigr) = 0,
\ \ \mbox{$\P$--a.s.}
\end{equation}
As we have seen already many times, this reduces by Poincar\'e's inequality to quantifying the convergence rate---in  Sobolev norms of negative regularity---of the gradient and flux of the correctors: for instance, 
\begin{equation}
\label{e.corrector.qualbound.Hm1.in.ch7}
\limsup_{m\to \infty}
3^{-m} 
\bigl(
\| \nabla \phi_{e} \|_{\Hminusul(\cu_m)} 
+
\| \a (e +\nabla \phi_{e}) - \ahom e \|_{\Hminusul(\cu_m)} 
\bigr) = 0,
\quad \mbox{$\P$--a.s.}
\end{equation}
Up to this point, we have used~$H^{-1}$ norms to quantify this weak convergence, but it is possible to obtain finer estimates using~$H^{-s}$ norms for general~$s>0$, as we will see. 

\smallskip

We will focus our attention in this chapter on obtaining the sharp convergence rate for the limit in~\eqref{e.corrector.qualbound.Hm1.in.ch7}. Consequences of this estimate, including sharp error estimates for the homogenization of boundary value problems as well as for homogenization of parabolic and elliptic Green functions  can be found in~\cite[Chapters 6, 8 \& 9]{AKMBook}. 

\smallskip

We have shown previously, under the assumption of~$\CFS(\beta,\Psi)$, that the quantity under the~$\limsup$ on the left side of~\eqref{e.corrector.qualbound.Hm1.in.ch7} is at most of order~$3^{-m\theta}$ for some positive exponent~$\theta>0$ which depends on~$(\beta,d,\lambda,\Lambda)$: see Theorem~\ref{t.optimalstochasticintegrability}. This exponent was later improved in Section~\ref{ss.almostone}, using a bootstrap argument, to any~$\theta < \frac{\alpha d}{d+2\alpha}$ for~$\alpha = \min\{ 1 , \frac d2(1-\beta)\}$.

\smallskip

The main result of this chapter is an improvement of these estimates to the optimal quantitative estimates for the first-order correctors. The statement is presented in the  following theorem. Throughout this chapter, we use the notation~$\Psi_\sigma(t) := \Psi(t^{1-\sigma})$ and we write~$\O^{1-\delta}_\Psi$ in place of~$\O_{\Psi_\delta}$. 

\begin{theorem}[First-order correctors, optimal quantitative estimates]
\label{t.corr.optimal}
Select parameters
\begin{equation*}
\delta \in (0,1)\,, \quad \gamma \in (0,\infty) 
\,, \quad 
\beta \in [0,1)
\,, \quad 
M, N \in [1,\infty)
\quad \mbox{and} \quad
p \in \Bigl(\frac{2}{1-\beta},\infty\Bigr)\,.
\end{equation*}
Let~$\Psi:[1,\infty) \to [0,\infty)$ be an increasing function satisfying
\begin{equation}  
\label{e.Psi.pgrowth.theorem.corr}
t^p \leq N \frac{\Psi(t s)}{\Psi(s)} 
\,, \qquad  
\forall 
s, t \in [1,\infty)
\,. 
\end{equation}
Let~$\P$ be a~$\Zd$--stationary measure on~$(\Omega,\F)$ satisfying: 
\begin{itemize}
\item If~$\beta >0$, then~$\P$ satisfies~$\CFS(\beta,\Psi,\Psi(M\cdot))$.
\item If~$\beta = 0$, then~$\P$ satisfies~$\CFS(0,\Psi_\sigma,\Psi_\sigma(M\cdot),\gamma)$ for each~$\sigma \in \{0,\delta\}$.
\end{itemize}
Then there exists~$C(\delta,\gamma,M,N,p,\dataref) < \infty$ such that, for every~$|e|=1$ and~$\ep \in (0,1)$, 
\begin{multline}
\label{e.corr.bounds}
\Big( \big\| \nabla \phi_e\bigl( \tfrac{\cdot}{\ep} \bigr)  \bigr\|_{H^{-s}(B_1)}  + \big\| \a\bigl( \tfrac{\cdot}{\ep} \bigr)\bigl( e + \nabla \phi_e\bigl( \tfrac{\cdot}{\ep} \bigr) \bigr) - \ahom e \bigr\|_{H^{-s}(B_1)} \Big)\indc_{\{ \X \leq \ep^{-1} \}} 
\\
= 
\left\{
\begin{aligned}
& \O_{\Psi}^{1-\delta} (C \ep^s) & \mbox{if} & \ 0 < s < \frac{d}{2}(1-\beta)\,,\\
&\O_{\Psi}^{1-\delta}  \bigl( C \ep^{\frac d2(1-\beta)} | \log \ep|^{\nicefrac12} \bigr) & \mbox{if} & \ s = \frac{d}{2}(1-\beta) \,, \\
&\O_{\Psi}^{1-\delta} \bigl( C \ep^{\frac d2(1-\beta)} \bigr) & \mbox{if} & \ s > \frac{d}{2}(1-\beta) \,,
\end{aligned}
\right.
\end{multline}
where~$\X$ is the minimal scale in the statement of Theorem~\ref{t.Ck1}.
\end{theorem}

Let us unpack the statement of the theorem. 
We recall that the~$\CFS$ mixing conditions are defined in Definition~\ref{d.CFS.strong}. Up to this point, with the exception of Exercises~\ref{exercise.improved.fluct} and \ref{exercise.improved.fluct.2}, we have not used the general condition~$\CFS(\beta,\Psi,\Psi',\gamma)$ and have formalized the all of our quantitative homogenization results with the much simpler and weaker condition~$\CFS(\beta,\Psi)$. Recall from~\eqref{e.yes.weaker} that the latter is weaker, so the hypotheses of the theorem above does imply that~$\P$ satisfies $\CFS(\beta,\Psi)$. The stronger~$\CFS$ assumption is needed to propagate the kind of bootstrap arguments made in this chapter. As proved in Chapter~\ref{s.CFS}, all of the canonical examples of random coefficients fields---such as local functions of Gaussian fields and fields which can be approximated by finite range fields---satisfy the stronger condition. 

\smallskip

Since the conclusion of Theorem~\ref{t.Ck1} is valid under the assumptions of Theorem~\ref{t.corr.optimal}, the minimal scale~$\X$ appearing in~\eqref{e.corr.bounds} satisfies~\eqref{e.X}, which we reproduce here:
\begin{equation}
\label{e.X.again.ch7}
\X^{\frac d2(1-\beta)  } 
= 
\O_{\Psi}(C) 
\,.
\end{equation}
Using~\eqref{e.X.again.ch7}, we can remove the indicator~$\indc_{\{ \X \leq \ep^{-1} \}}$ on the left as follows. 
By~\eqref{e.corr.grad.bound.pre} below, we have that
\begin{multline}  
\label{e.byebyeX}
\Bigl( \big\| \nabla \phi_e\bigl( \tfrac{\cdot}{\ep} \bigr)  \bigr\|_{H^{-s}(B_1)}  + \big\| \a\bigl( \tfrac{\cdot}{\ep} \bigr)\bigl( e + \nabla \phi_e\bigl( \tfrac{\cdot}{\ep} \bigr) \bigr) - \ahom e \bigr\|_{H^{-s}(B_1)}
\Bigr) 
\indc_{\{ \X > \ep^{-1} \} } 
\\ 
\leq 
C(\ep \X)^{\nicefrac d2-\eta}
\bigl( \| \nabla \phi_e\|_{\underline{L}^2(B_{\X})} + |e| \bigr)
\leq 
C |e| (\ep \X)^{\nicefrac d2-\eta}
\,,
\end{multline}
where~$\eta(d,\Lambda/\lambda)>0$ is the small exponent obtained in the standard ``hole-filling'' argument; alternatively, we can obtain the second inequality in~\eqref{e.byebyeX} by applying Meyers' estimate and then Holder inequality, as in~\cite[Lemma 3.17]{AKMBook}.
For~$\beta\leq \nicefrac{2\eta}{d}$, we can therefore remove the indicator function on the left side of~\eqref{e.corr.bounds}. 
For larger (non-tiny) values of~$\beta$, 
the stochastic integrability of~$\X^{\nicefrac d2-\eta}$ is worse than the one reported on the right side of~\eqref{e.corr.bounds}, and the leading-order contribution to the tail behavior of the random variable
\begin{equation*}
\big\| \nabla \phi_e\bigl( \tfrac{\cdot}{\ep} \bigr)  \bigr\|_{H^{-s}(B_1)}  + \big\| \a\bigl( \tfrac{\cdot}{\ep} \bigr)\bigl( e + \nabla \phi_e\bigl( \tfrac{\cdot}{\ep} \bigr) \bigr) - \ahom e \bigr\|_{H^{-s}(B_1)} 
\end{equation*}
is governed by the unlikely event that the minimal scale~$\X$ is extremely large.

\smallskip

It is possible to extend the estimate stated in~\eqref{e.corr.bounds} to stronger spatial integrability exponents by using~$W^{-s,p}$ norms rather than~$H^{-s}$: see~\cite[Theorem 4.24]{AKMBook}.

\smallskip

The growth condition~\eqref{e.Psi.pgrowth.theorem.corr} on~$\Psi$ implies both~\eqref{e.Young.growth} and~\eqref{e.weak.triangle.ass}. This stronger condition is needed for a technical reason---it is used in~\eqref{e.weakint.ave.applied}, below, which ensures that the minimal scale satisfies a certain spatial integrability property. 

\smallskip

The arguments in this chapter leading to the proof of Theorem~\ref{t.corr.optimal} are closely related to those of Section~\ref{ss.subadd.conv} and are even more closely foreshadowed in Section~\ref{ss.almostone}. As in those sections, we do not make the estimates for the correctors the main object of study. We instead deduce corrector estimates as a consequence of the speed of convergence for certain \emph{coarse-grained coefficient fields} as the coarsening scale tends to infinity. 

\smallskip

To begin thinking about the proof of Theorem~\ref{t.corr.optimal}, we should ask ourselves where the approach in Section~\ref{ss.almostone} falls short of optimality. There we proved that the exponent~$\alpha$ appearing in Theorem~\ref{t.subadd.converge} or Proposition~\ref{p.algebraic.nosymm},  can be improved to any~$\alpha < \min\{ 1 , \frac d2(1-\beta)\}$. It follows from this, Lemma~\ref{l.mathcalE.minscale} and~\eqref{e.FV.sublinearity} that the speed of convergence of the limit in~\eqref{e.corrector.qualbound.Hm1.in.ch7} is at most~$3^{-m\theta}$, for any~$\theta < \frac{\alpha d}{d+2\alpha}$. There are two main sources of suboptimality in this latter estimate:

\begin{itemize}

\item The exponent~$\theta$ is smaller than~$\alpha$ because some information was lost when we upgraded the stochastic integrability in Lemma~\ref{l.mathcalE.minscale}. To overcome this problem, we need to keep track of stronger moments in our induction up the scales rather than just keep track of an expectation as we do in the previous three chapters. 

\item The exponent~$\alpha$ cannot improve beyond~$\alpha=1$ if we use a method based on the coarse-grained quantities introduced in Chapters~\ref{s.subadd} and~\ref{s.nonsymm}. In this sense, the exponent~$\alpha$ in Proposition~\ref{p.alpha.almostone} is actually sharp because the coarse-grained matrices pick up boundary layer errors which are of order~$3^{-m}$ and no smaller. If~$\frac d2(1-\beta) > 1$, then we should expect a better exponent for the first-order correctors, but this will require a more sensitive quantity that does not display this as a boundary layer artifact. 

\end{itemize}

The first issue is a technical problem which is not difficult to fix. The second issue is the more serious one because it requires us to modify the coarse-grained coefficients we have worked with in previous chapters. Indeed, our best effort to build a renormalization scheme based on these quantities was in Section~\ref{ss.almostone}, where we got stuck at exponent~$\alpha =1$ because of boundary layer errors. Note that this is not an artifact of the argument there, as the leading order difference between~$\a(\cu_n)$ and~$\a_*(\cu_n)$ will be due to the presence of boundary layers in the solutions of the Dirichlet and Neumann problems. 

\smallskip

The coarse-graining procedure is analogous to convolution in many ways; to obtain a better rate, we must use a smoother kernel. This leads us to the idea of defining the coarse-grained coefficients in terms of profiles of the heat kernels rather than cubes. 
The smoothness of the heat kernel will prevent the build-up of boundary layer errors and allow us to prove optimal quantitative bounds on the coarse-grained coefficients. In other words, these coarse-grained coefficients behave much better, at least when seeking very fine estimates and faster convergence rates. 

\smallskip

This change, however, does come at a cost: we lose the strict locality of our quantities (they depend now on the coefficients in the whole space), and we also have no analog of~\eqref{e.diagonalset}, around which we based our arguments in Section~\ref{ss.subadd.conv}. We will, therefore, spend considerable effort on obtaining the \emph{approximate additivity} and, especially,  \emph{approximate locality} of our coarse-grained coefficients. On the other hand, we also have advantages that we did not possess in the context of Chapter~\ref{s.subadd}: namely, the regularity theory in Chapter~\ref{s.regularity}, which was built upon it. 

\smallskip

The coarse-grained coefficients~$\a(U)$ and~$\a_*(U)$ defined in Chapter~\ref{s.subadd} had many interesting and useful properties (see Lemma~\ref{l.J.basicprops}), but one of the most important was the identities in~\eqref{e.a.astar.formulas}: they each exactly mapped the spatial averages of gradients to spatial averages of fluxes for an important~$d$-dimensional subspace of solutions, which turn out to be finite-volume versions of the first-order correctors. We will make this the defining property of the coarse-grained coefficients we use in this chapter, with two differences. First, the~$d$-dimensional subspace of solutions we use will be~$\mathcal{A}_1$, in other words, the corrected affine functions corresponding to the infinite-volume first-order correctors. Second, as previously mentioned, the spatial averages will be taken not with respect to a bounded domain~$U\subseteq\Rd$, but rather with respect to the profile of the heat kernel at a fixed time and possibly translated in space.

\subsection{Coarse-graining with the heat kernel}

\subsubsection{Definition of the coarse-grained coefficients~$\b_r(x)$}
\label{ss.coarsegrain.with.tempo}
We denote the standard heat kernel by 
\begin{equation*}
\Phi(t,x) := (4\pi t)^{-\frac d2} \exp\biggl( -\frac{|x|^2}{4t} \biggr)\,.
\end{equation*}
We denote the profile of the heat kernel at time~$t=r^2$ by~$\Phi_r:= \Phi(r^2,\cdot)$.

\smallskip

Throughout this chapter, we fix a parameter~$\delta\in (0,\frac12]$, which must be sufficiently small for our arguments to go through. It will depend on~$(d,\lambda,\Lambda)$ and the constants in the~$\CFS$ condition. 

\smallskip

We let~$x \mapsto \X(x)$ be the~$\Zd$-stationary extension of the minimal scale~$\X$ in Theorem~\ref{t.Ck1}, which is taken to be a constant in every cube of the form~$z+\cu_0$ with~$z\in \Zd$. In other words,~$\X(x):= T_{[x]}\X$, where~$[x]$ is the closest lattice point of~$\Zd$ to~$x$ with the lexicographical ordering used as a tiebreaker if this point is not unique.
By decomposing~$\Phi$ into a sum of functions supported in annular domains with radii~$2^k r$, 
we deduce by~\eqref{e.liouvillec} that, for every~$e \in \R^d$ and~$r> 0$, there exists a function~$\psi_e \in \A_1$ such that
\begin{equation*} 
\frac1r \left\| \psi_e  -  \ell_e \right\|_{\underline{L}^2(B_r(x))} 
+
 | (\nabla \psi_e \ast \Phi_r)(x) - e | 
+
\left|  \bigl( \a \nabla \psi_e \ast \Phi_r \bigr)(x) - \ahom e \right|
\leq 
C\Bigl( \frac{r \wedge \X(x)} {r}\Bigr)^{\! \alpha} |e|
\,.
\end{equation*}
We do not need~$r\geq \X$ because the inequality is trivially valid for~$r<\X$ (just take the zero solution). 
By replacing the minimal scale~$\X$ by~$(c \delta)^{-\nicefrac 1\alpha}\X$, we get the same inequality as above but with a~$\delta$ on the right-hand side:
for every~$e \in \R^d$ and~$r>0$, there exists a function~$\psi_e \in \A_1$ such that
\begin{equation} \label{e.liouvillec.applied}
\frac1r \left\| \psi_e  -  \ell_e \right\|_{\underline{L}^2(B_r(x))} 
+
 | (\nabla \psi_e \ast \Phi_r)(x) - e | 
+
\left|  \bigl( \a \nabla \psi_e \ast \Phi_r \bigr)(x) - \ahom e \right|
\leq 
\delta \Bigl( \frac{r \wedge \X(x)} {r}\Bigr)^{\! \alpha} |e|
\,.
\end{equation}
We recognize~$\psi_e$, up to an additive constant, as~$\ell_e + \phi_e$, where~$\phi_e$ is the first-order corrector appearing in Chapter~\ref{s.qualitativetheory}. The mapping~$e \mapsto \nabla \psi_e$ is linear and hence, by Theorem~\ref{t.Ck1},
\begin{align}  \label{e.gradA1}
\nabla \A_1 := \{ \nabla \psi \, : \, \psi \in \A_1 \} = \{ \nabla \psi_e \, : \, e \in \R^d \}
\end{align}
almost surely.  Therefore, taking smaller~$\delta$ by means of~$\dataref$ in~\eqref{e.liouvillec.applied},
we see that, for every~$r \geq \X(x)$, 
\begin{equation} 
\label{e.b.representation}
\bigl| \bigl((\nabla \psi_{e}) \ast \Phi_{r}\bigr)(x)   - e \bigr| 
+ 
\bigl| \bigl( (\a \nabla \psi_{e}) \ast \Phi_{r}\bigr)(x)  -\ahom e \bigr| 
\leq 
\delta \Bigl( \frac{r \wedge \X(x)} {r}\Bigr)^{\! \alpha} 
\,.
\end{equation}

\smallskip

We define the matrix-valued, stationary random field~$\hat \b_r(\cdot)$, which is characterized, for every~$x \in \Rd$ and~$r \geq \X(x)$, by 
\begin{equation}
\label{e.coarsened.fist}
\hat \b_r(x) \!
\int_{\Rd} 
\Phi_r(y{-}x) \nabla \psi(y)\,dy 
=
\int_{\Rd} 
\Phi_r(y{-}x) \a(y) \nabla \psi(y)\,dy,
\quad 
\forall \psi \in \A_1, \, r\in [ \X(x),\infty).
\end{equation}
The definition is reasonable in view of~\eqref{e.gradA1} and~\eqref{e.b.representation} by enlarging~$\X(x)$ by a uniform factor. We also extend the definition for the event~$\{r < \X(x) \}$ as  
\begin{align}  \label{e.b.deg}
\b_r(x) =  \chi_r(x) \hat \b_r(x) + (1-\chi_r(x)) \ahom
\end{align}
with
\begin{equation}
\label{e.minscalefct.def}
\minscale_r(x) :=  \min\bigl\{ 2 \bigl(1 - r^{-1} \X(x) \bigr)_+\,,  1 \bigr\}
\,.
\end{equation}
The random field~$\minscale_r$ is an approximation of~$\indc_{\{r\geq  \X(\cdot)\}}$.  
By~\eqref{e.b.representation} we see that~$|\b_r(x) - \ahom| \leq \delta$ and, in the open set~$\{ x \,:\,  \X(x) <  \frac12 r \}$, the map~$x\mapsto \b_r(x)$ is real analytic; in particular it is continuous, and it varies on a length scale of order~$r$. Moreover, we have the estimate 
\begin{equation}  \label{e.coarsen.smallness}
| \b_r(x)  - \ahom | \leq \delta \Bigl( \frac{r \wedge \X(x)} {r}\Bigr)^{\! \alpha} 
\,.
\end{equation}

\smallskip
 
The stationary random field~$\b_r$ is the coarse-grained coefficient field at  scale~$r>0$, which is the primary object of study in this chapter.
To see why we give it this name, using~\eqref{e.coarsened.fist}, we get an equation for~$\psi \ast\Phi_r$, for every~$\psi\in\A_1$ and~$r>0$:
\begin{equation}
\label{e.coarsened.equation}
-\nabla \cdot \b_r \nabla (\psi \ast \Phi_r)
=
\nabla \cdot 
\bigl( (1-\minscale_{r} ) \bigl( (\a-\ahom)\nabla \psi \bigr)\ast \Phi_r
   \bigr)
\,.
\end{equation}
We stress that the second term on the right of~\eqref{e.coarsened.equation} is extremely small because it is smaller than anything we need to worry about in our arguments. Therefore, we see that the operator~$\nabla \cdot \b_r \nabla$ is a more precise, finite-scale version of the homogenized operator (at least for the first-order correctors). By convolving with~$\Phi_r$, we are rubbing away all the oscillations in the solution below scale~$r$ and focusing on the larger scales. We can thus consider it the ``coarse-grained'' or ``renormalized'' operator at scale~$r$. 

\smallskip

\subsubsection{Optimal quantitative estimates on the coarse-grained coefficients}

The following theorem gives optimal bounds on the size of the fluctuations of the stationary random field~$\b_r$. 
This can be regarded as the main result of the section, and we will show later that Theorem~\ref{t.corr.optimal} is essentially a corollary of it. 

\begin{theorem}[Optimal quantitative estimates for the coarse-grained coefficients~$\b_r$]
\label{t.optimal}
\textbf{} \\
Let~$\delta,\gamma> 0$ and~$\beta \in [0,1)$,~$M, N \in [1,\infty)$ and~$p \in (\frac{2}{1-\beta},\infty)$.  
Let~$\Psi:[1,\infty) \to [0,\infty)$ be an increasing function satisfying
\begin{equation}  \label{e.Psi.pgrowth.theorem}
t^p \leq N \frac{\Psi(t s)}{\Psi(s)}\,,  
\qquad \forall s, t \in [1,\infty)
\,. 
\end{equation}
We also suppose that~$\P$ is a~$\Zd$--stationary measure on~$(\Omega,\F)$ satisfying: \begin{itemize}
\item If~$\beta \in (0,1)$, then~$\P$ satisfies~$\CFS(\beta,\Psi,\Psi(M\cdot),0)$.
\item If~$\beta = 0$, then~$\P$ satisfies~$\CFS(0,\Psi_\sigma,\Psi_\sigma(M\cdot),\gamma)$ for each~$\sigma \in \{0,\delta\}$.
\end{itemize}
Then there exists~$C(\delta,\gamma,M,N,p,\dataref) < \infty$ such that, for every~$x\in \R^d$ and~$r \in [1,\infty)$, 
\begin{equation} 
\label{e.optimal.estimate}
\b_{r}(x)  - \ahom  =  \O_{\Psi}^{1-\delta} \bigl(C r^{-\frac d2(1-\beta)}\bigr)
\,.
\end{equation}
\end{theorem}

Recall from the definition of the~$\O_\Psi$ notation in~\eqref{e.weakcondOPsi} that the estimate~\eqref{e.optimal.estimate} means simply that, for every~$r,t\geq 1$, 
\begin{equation} 
\label{e.optimal.estimate.Cheby}
\b_{r}(x)  - \ahom  =  \O_{\Psi}^{1-\delta} \bigl(C r^{-\frac d2(1-\beta)}\bigr) 
\quad \Longleftrightarrow \quad
\P \Biggl[ \frac{| \b_{r}(x)  - \ahom |}{Cr^{-\frac d2(1-\beta)}} > t \Biggr] 
\leq 
\frac{1}{\Psi(t^{1-\delta})}
\,.
\end{equation}
In other words,~$| \b_{r}(x)  - \ahom |$ is at most of order~$r^{-\frac d2(1-\beta)}$, with stochastic integrability given by~$\Psi$. Comparing this to~\eqref{e.CFS}, we see that~$\b_{r}$ has fluctuations which are essentially no longer than the spatial average of the coefficient field~$\a(\cdot)$ itself under the~$\CFS$ assumption.

\smallskip

It turns out that the case that~$\beta = 0$ is somewhat more subtle than~$\beta\in (0,1)$, and requires the assumption that the parameter~$\gamma$ is positive in the~$\CFS$ condition. In fact, for~$\beta\in (0,1)$, the estimate~\eqref{e.optimal.estimate} can be slightly improved in terms of stochastic integrability: see the statement of Proposition~\ref{p.psipsi.coarse} below. 

\smallskip

The proof of Theorem~\ref{t.optimal} is completed in Section~\ref{s.fluctboot} (see Propositions~\ref{p.psipsi.coarse} and~\ref{p.psipsi.coarse.optimal}). The argument is essentially a variation of the approach found in~\cite{AKM} and~\cite[Chapter 4]{AKMBook}, but adapted to a more general setting. As in those works, we begin with the suboptimal estimate~\eqref{e.b.representation}, which serves to anchor the induction argument by providing crucial initial smallness.

\subsubsection{An overview of the proof of Theorem~\ref{t.optimal}}

The idea behind the proof of Theorem~\ref{t.optimal} is to \emph{accelerate} the above convergence of~$\b_r$ to~$\ahom$, as~$r$ tends to infinity, by showing that the estimate for~$\big| \b_r(x) - \ahom \big|$ is ``self-improving'' as we pass to larger scales. The proof of this self-improvement breaks into three steps: 

\begin{enumerate}
\item[(i)] \emph{Additivity estimate}: the additivity defect is bounded by the \emph{square} of the fluctuation size. 

\item[(ii)] \emph{Localization estimate}: the coarse-grained coefficients can be approximated by a ``local'' coefficient field with an error that is slightly better than the size of the fluctuations. 

\item[(iii)] \emph{Fluctuation estimate}: The fluctuations are controlled by the sum of (a) the localization error, (b) the additivity error, and (c) what the~$\CFS$ assumption gives.
 
\end{enumerate}
By feeding these statements into one another, one can quickly see that the estimate for the fluctuations will ``self-improve'' until they saturate on ``what the~$\CFS$ assumption gives'' and the additivity and localization will be a bit better. 

\smallskip

We have seen this effect already in Sections~\ref{ss.subadd.conv}: compare the vague statement (i) to~\eqref{e.flatness.quenched}, recalling that~$J$ is a proxy for ``additivity error'' (see~\eqref{e.add.defect.a} and~\eqref{e.add.defect.astar}), and then compare~\eqref{e.variance.J2} to~(iii). Note that there was no localization error in the context of Chapter~\ref{s.subadd} since the coarse-grained matrices were perfectly local. 
However, we see by looking back at the estimates that there were other sources of error, such as the last term on the right of~\eqref{e.flatness.quenched}. These terms were harmless in the context of the proof of Proposition~\ref{p.algebraicrate.E} but would cause the ``self-improving'' property of the estimates to saturate before reaching the limit of the quantitative mixing assumption. As mentioned above, these errors cannot be removed entirely due to boundary layer effects.

\smallskip

Let us give the vague assertions~$\mathrm{(i)}$--$\mathrm{(iii)}$ slightly more precise (albeit still imprecise) statements in our context. 
We use the following decomposition for~$\b_R - \ahom$. For every~$R\geq 2$ and~$r \in [1, \frac12 R]$, we write   
\begin{align} 
\notag
\b_R - \ahom
&
=
\underbrace{
\Bigl( \b_R - \b_r \ast \Phi_{\sqrt{R^2-r^2}} \Bigr) 
}_{\mathrm{additivity \; defect}}
+
\underbrace{
\Bigl( (\b_r - \E [ \b_r ]) - (\b'_{r,t} - \E [ \b'_{r,t} ] )  \Bigr) \ast \Phi_{\sqrt{R^2-r^2}} 
}_{\mathrm{localization \; defect}}
\notag \\ & \qquad
+ \underbrace{
\bigl( \b'_{r,t} - \E [ \b'_{r,t} ] \bigr) \ast \Phi_{\sqrt{R^2-r^2}}
}_{\mathrm{fluctuations}}
+
\underbrace{
\bigl( \E [ \b_r ]- \ahom \bigr) \ast \Phi_{\sqrt{R^2-r^2}}
}_{\mathrm{systematic \;  error}}
\,.
\label{e.ahomr.splitting}
\end{align} 
In order to use the~$\CFS$ mixing conditions, we need to find a suitable \emph{localization} of~$\b_r$. Given a parameter~$t > r$, the localized coarsened matrix is denoted by~$\b'_{r,t}$ with the property that, for every~$x\in \R^d$ and~$1\leq r< t$, the matrix~$\b'_{r,t}$ is~$\F(B_{t}(x))$-measurable and that~$\b_r - \b'_{r,t}$ is small. This localized matrix will be introduced and studied in Section~\ref{ss.local}.

\smallskip

The three steps (i)--(iii) above amount to roughly showing that 
\begin{equation}
\label{e.rough.add}
\mathrm{additivity \; defect}(r\to R)
\lesssim
C\cdot \mathrm{fluctuations}(r)^2
\lesssim C r^{-\alpha} \mathrm{fluctuations}(r),
\end{equation}
\begin{equation}
\label{e.rough.loc}
\mathrm{localization \; defect} (r)
\lesssim
r^{-\alpha} \cdot \mathrm{fluctuations}(r),
\end{equation}
and
\begin{align}
\lefteqn{ \mathrm{fluctuations}(R) } \quad & 
\notag \\ & 
\lesssim
\mathrm{localization \; defect} (r)
+
\mathrm{additivity \; defect}(r\to R)
+ 
\mathrm{CFS \; scaling}(r\to R)\,.
\end{align}
Putting these together and performing an iteration up the scales, we will find that the error saturates on whatever the~$\CFS$ condition gives, which is (up to a slight loss of stochastic integrability) essentially~$\O_\Psi(Cr^{-\frac d2(1-\beta)})$.

\smallskip

The rigorous formulation of the additivity estimate~\eqref{e.rough.add}
is proved in the following section: see~\eqref{e.additivityestimate}. 
The localization estimate~\eqref{e.rough.loc} is formalized in~Section~\ref{ss.local}, see~\eqref{e.localizationestimate}. 
The fluctuations estimate and the iteration argument are presented in Section~\ref{s.fluctboot}.

\subsection{Additivity defect estimate}

One of the features of our approach in this chapter is that we have defined the coarse-grained coefficients~$\b_r$ in such a way that additivity is relatively easy to check. As we will find out in the next section, it is localization that is more difficult and technical. This is the reverse of the situation we encountered in Chapter~\ref{s.subadd}, where we had perfect locality wired into the definitions, and the hard part was to estimate the additivity defect. 

\smallskip

To prove the additivity estimate, we start with the semigroup property of the heat equation (in the form~$\Phi_R = \Phi_r \ast \Phi_{\sqrt{R^2-r^2}}$) and use the triangle inequality to obtain under the event~$\{R \geq \X\}$, in view of~\eqref{e.liouvillec.applied}, that 
\begin{align*}
\lefteqn{
\Bigl| 
 \b_R(0) - \b_r \ast \Phi_{\sqrt{R^2-r^2}} (0)
\Bigr|
} 
\qquad & 
\\ &
\leq
C \sup_{e \in B_1} \bigl|
\bigl( \bigl( ( \b_R(0) - \b_r(\cdot) ) \nabla ( \psi_e \ast \Phi_{R}) (0) \bigr)
\ast \Phi_{\sqrt{R^2-r^2}} \bigr) (0) \bigr|
\\ & 
\leq
C \sup_{e \in B_1} \bigl| \bigl( \bigl(
( \b_R (0) - \b_r (\cdot) )\bigl( \nabla ( \psi_e \ast \Phi_{R})(0) - \nabla ( \psi_e \ast \Phi_{r}) (\cdot)  \bigr) \bigr)
\ast \Phi_{\sqrt{R^2-r^2}} \bigr)(0) \bigr| 
\\ & \qquad 
+ C \sup_{e \in B_1} \bigl| \b_R(0) \nabla( \psi_e \ast \Phi_R)(0)-   \bigl( \b_r \nabla ( \psi_e \ast \Phi_{r}) \bigr) \ast\Phi_{\sqrt{R^2-r^2}}(0)
 \bigr| 
\,.
\end{align*}
The second term on the right side is small by~\eqref{e.coarsened.fist}. Indeed, if~$R \geq  \X$ and~$e \in B_1$, then
\begin{align} \notag
\b_R (0) \nabla ( \psi_e \ast \Phi_R) (0) 
&
=
\bigl( (\a\nabla \psi_e)\ast \Phi_R \bigr)(0)
\\ \notag & 
=
\bigl( 
(\a\nabla \psi_e)\ast \Phi_r 
\bigr) \ast\Phi_{\sqrt{R^2-r^2}} (0)
\\  &  
=
\Bigl( \b_r \nabla(  \psi_e \ast \Phi_r )   + (1-\minscale_{r} ) \bigl((\a - \ahom)\nabla \psi_e \bigr)\ast \Phi_r \Bigr) \ast\Phi_{\sqrt{R^2-r^2}}(0)
\label{e.semigroup.identity}
\end{align}
and, by H\"older's inequality, the last term on the right side can be estimated by
\begin{align} \notag  
\lefteqn{
\sup_{e\in B_1} \Bigl|
\Bigl( 
(1-\minscale_{r} ) \bigl((\a - \ahom)\nabla \psi_e \bigr)\ast \Phi_r \Bigr) \ast\Phi_{\sqrt{R^2-r^2}}(0)
\Bigr| \indc_{\{\X(0) \leq R \}}
} \qquad &
\\ 
\notag &
\leq C \| \indc_{\{\X(\cdot)>r\}} \|_{L^2(\Phi_{R})} 
\sup_{e\in B_1} 
\| \nabla \psi_e \|_{L^2(\Phi_{R})} \indc_{\{\X \leq R \}}
\\ \notag & 
\leq 
C \| \indc_{\{\X(\cdot)>r\}} \|_{L^2(\Phi_{R})} 
\,.
\end{align}
Here we are using the notation~$\| \cdot \|_{L^2(\Phi_R)}$ defined in~\eqref{e.L2.Psi}.
Using the above estimates and the boundedness of~$\b_r$ and~$\b_R$, we obtain
\begin{align}
\lefteqn{
\Bigl| 
 \b_R(0) - \b_r \ast \Phi_{\sqrt{R^2-r^2}} (0)
\Bigr|
} \quad & 
\notag \\ &
\leq
C \| \b_R(0)- \b_r\|_{L^2(\Phi_{\sqrt{R^2-r^2}})} 
\sup_{e \in B_1} 
\|  \nabla \psi_e \ast \Phi_{r}- \nabla (\psi_e \ast\Phi_R)(0) \|_{L^2(\Phi_{\sqrt{R^2-r^2}})} 
\indc_{\{\X \leq R \}}
\notag 
\\ & \qquad
+
C  \|   \indc_{\{\X(\cdot)>r\}} \|_{L^2(\Phi_{R})}  + C \indc_{\{\X > R \}}
\,.
\label{e.verygood.Jestimate}
\end{align}
The two terms on the last line of~\eqref{e.verygood.Jestimate} are very small, of order~$\O_\Psi(C r^{-\frac d2(1+\beta)})$, for instance, under the assumption of~$\CFS(\beta,\Psi)$. The first term on the right side of~\eqref{e.verygood.Jestimate}---the principal part of the error---certainly appears to be a ``quadratic'' error! We will show that it is. 


\smallskip

We would ideally like to estimate the second factor (involving~$\psi_e$) in the main error term by the first factor (involving the fluctuations of the field~$\b_r(x)$). Because we are working with elements of~$\A_1$ (in contrast to a finite-volume corrector), this is not quite possible: our estimate will necessarily involve all the length scales larger than~$r$. This is a minor technical point that does not disturb the main argument.

\smallskip

The reason that we should be able to control the spatial oscillations of~$\nabla (\psi\ast \Phi_r)$ is because it satisfies the equation~\eqref{e.coarsened.equation}. Note that the right side of~\eqref{e.coarsened.equation} is small and can almost be ignored; meanwhile, the coefficient field~$\b_r$ has small contrast because its spatial oscillations are small. We should, therefore, expect to be able to use elliptic regularity (harmonic approximation) to show that the solution~$\psi\ast \Phi_r$ is close to an affine function. This is indeed the case, as we show in the following lemma, which provides the link between quantitative bounds on the fluctuations of the coarse-grained coefficients and those of the spatial averages of the gradients and fluxes of the correctors. The statement and proof of the lemma closely follow~\cite[Lemma 4.15]{AKM}. 

\smallskip

We introduce the following stationary random field~$\mathcal{G}_r$, which measures the fluctuations of coarse-grained matrices together with an averaged minimal scale. 
For every~$r \geq 2$ and~$x \in \R^d$, we set
\begin{align}  
\label{e.uglyG}
\mathcal{G}_r(x) 
& 
: = 
\sup_{t>r} \Bigl(\frac rt \Bigr)^{\nicefrac 12} 
\Bigl( 
\| 
\b_r - \ahom \|_{\underline{L}^2 \left( B_{t}(x) \right)} 
+  r^{-\nicefrac d2} \left\| \X 
\right\|_{\underline{L}^d \left( B_{t}(x) \right)}^{\nicefrac d2}  
 \Bigr)
\,.
\end{align}
Notice that, for every~$s \geq r >0$, we trivially have that~$\mathcal{G}_r(x) \leq (\frac sr)^{\frac{d-1}{2}} \mathcal{G}_s(x)$, and we will be using this property in the sequel without explicitly mentioning it.

\begin{lemma}
\label{l.scalecomparison}
There exists a constant~$C(\dataref)<\infty$ such that, for every~$e \in \R^d$,~$x\in \R^d$ and~$r \geq 2$,  we have
\begin{align}
\label{e.scalecomparison}
\left\| \nabla \psi_e \ast \Phi_r  - \nabla \psi_e \ast \Phi_R(x)\right\|_{L^2 \left( \Phi_{x,\sqrt{R^2 - r^2}} \right)} 
& \leq 
C \mathcal{G}_r(x) |e|
\,.
\end{align}
\end{lemma}
\begin{proof}
Without loss of generality, we may assume that~$x=0$ and~$|e|=1$.  Fix~$R \geq 2$,~$r \in [1,\frac12 R]$  and set~$t = \sqrt{R^2-r^2}$. We first apply~\eqref{e.coarsened.equation} to obtain the following equation for the mollified function~$w = \psi_e \ast \Phi_r$:
\begin{equation*}  
-\nabla \cdot \ahom \nabla w = 
\nabla \cdot \bigl( \minscale_r(x)  (\b_r - \ahom) \nabla w \bigr) + \nabla \cdot \f_{r} 
\end{equation*}
with
\begin{equation*}  
\f_{r} := (1-\minscale_r(\cdot) ) \Bigl( [(\a - \ahom) \nabla \psi] \ast \Phi_{r} + (\b_r - \ahom) \nabla w\Bigr) 
.
\end{equation*}
By Theorem~\ref{t.C01} we get, for every~$x\in\R^d$ and~$|e|=1$, that
\begin{equation} \label{e.corr.grad.bound.pre}
\sup_{s \geq \X(x)}\left\| \nabla \psi_e \right\|_{\underline{L}^2 \left( B_{s}(x) \right)} 
\leq 
C
\end{equation}
and, consequently,
\begin{equation}  \label{e.corr.grad.bound}
\left\| \nabla \psi_e \right\|_{L^2(\Phi_{x,r})} 
\leq 
C  \biggl( \frac{\X(x) \vee r}{r} \biggr)^{\!\!\nicefrac d2} 
\,.
\end{equation}
Therefore, we also obtain the following pointwise bounds, for every~$x\in\R^d$, 
\begin{equation*}  
\minscale_r(x) | \nabla w(x)| \leq C \qand |\f_{r}(x)| \leq C  \indc_{\{r< \X(x)\}}  \biggl( \frac{\X (x)}{r} \biggr)^{\!\!\nicefrac d2} 
\,.
\end{equation*}
The result now follows by Lemma~\ref{l.harmdecay} below. 
\end{proof}

Above, we applied the following lemma, which gives an estimate for global solutions of constant coefficient equations with a divergence-form right-hand side. 

\begin{lemma} \label{l.harmdecay}
Suppose~$\f\in  L_{\mathrm{loc}}^2(\R^d ; \R^d)$ and~$w \in H_{\mathrm{loc}}^1(\R^d)$ satisfy~$-\nabla \cdot \ahom \nabla w = \nabla \cdot \f$ and
\begin{align}  \label{e.harmdecay.ass}
\lim_{r \to \infty} r^{-\nicefrac12}
\left\| \nabla w \right\|_{\underline{L}^2 \left( B_{r} \right)} = 0
 \,.
\end{align}
Then there exists~$C(d,\lambda,\Lambda)<\infty$ such that, for every~$r>0$, 
\begin{align}  \label{e.harmdecay}
\left\| \nabla w  -  (\nabla w)_{\Phi_r} \right\|_{\underline{L}^2 \left( \Phi_{r} \right)}  
\leq 
C \sup_{t \geq r}  \Bigl( \frac{r}{t} \Bigr)^{\!\nicefrac12} \left\| \f - (\f)_{B_t}  \right\|_{\underline{L}^2 \left( B_{t} \right)}  \,.
\end{align}
\end{lemma}

\begin{proof}
\emph{Step 1.} We first show that
\begin{align}  \label{e.harmdecay.pre}
\left\| \nabla w  -  (\nabla w)_{B_r} \right\|_{\underline{L}^2 \left(B_r \right)}  
\leq 
C \sup_{t \geq r}  \Bigl( \frac{r}{t} \Bigr)^{\!\nicefrac12} \left\| \f - (\f)_{B_t}  \right\|_{\underline{L}^2 \left( B_{t} \right)}  \,.
\end{align}
The proof is via harmonic approximation. Let~$t \geq r$ and let~$h_t \in H_0^1(B_t)$ satisfy
\begin{equation*}
\left\{
\begin{aligned}
&-\nabla \cdot \ahom \nabla h_t = 0 & \mbox{in} & \ B_t \,, \\
& h_t = w & \mbox{on} & \ \partial B_t \,.
\end{aligned}
\right.
\end{equation*}
Testing both this and the equation for~$w$ with~$h_t - w$, we obtain 
\begin{align*}  
\left\| \nabla w - \nabla h_t  \right\|_{\underline{L}^2 \left( B_{t} \right)} 
\leq
C 
\left\| \f - (\f)_{B_t}  \right\|_{\underline{L}^2 \left( B_{t} \right)} 
\,.
\end{align*}
Since~$h_t$ is~$\ahom$-harmonic, it has bounded second derivatives and, in particular, 
\begin{align*}  
\left\| \nabla h_t  -  (\nabla h_t)_{B_{\theta t}} \right\|_{\underline{L}^2 \left( B_{\theta t} \right)} 
\leq 
C \theta \left\| \nabla h_t  -  (\nabla h_t)_{B_{t}} \right\|_{\underline{L}^2 \left( B_{t} \right)} 
 \,.
\end{align*}
Thus, by the triangle inequality,
\begin{equation}  
\left\| \nabla w  -  (\nabla w)_{B_{\theta t}} \right\|_{\underline{L}^2 \left( B_{\theta t} \right)} 
\leq 
C \theta \left\| \nabla w  -  (\nabla w)_{B_{t}} \right\|_{\underline{L}^2 \left( B_{t} \right)} 
+ 
C \theta^{-\nicefrac d2} \left\| \f - (\f)_{B_t}  \right\|_{\underline{L}^2 \left( B_{t} \right)} 
 \,.
 \label{e.harm.appr.additivity.bootstrap}
 \end{equation}
Choosing~$\theta$ so small that~$C_{\eqref{e.harm.appr.additivity.bootstrap}} \theta^{\nicefrac12} \leq \frac12$, we obtain by the above estimate that
\begin{align*}  
\theta^{-\nicefrac12}  \left\| \nabla w  -  (\nabla w)_{B_{\theta t}} \right\|_{\underline{L}^2 \left( B_{\theta t} \right)} 
\leq 
\frac12 \| \nabla w  -  (\nabla w)_{B_{t}} \|_{\underline{L}^2 \left( B_{t} \right)} 
+ 
C  \| \f - (\f)_{B_t}  \|_{\underline{L}^2 \left( B_{t} \right)} 
\,.
\end{align*}
Thus, by iterating and using the fact that~$t^{-\nicefrac12}\left\| \nabla w \right\|_{\underline{L}^2 \left( B_{t} \right)} \to 0$ as~$t\to \infty$, we deduce that
\begin{align*}  
\left\| \nabla w  -  (\nabla w)_{B_{r}} \right\|_{\underline{L}^2 \left( B_{r} \right)}  
\leq 
C \sup_{t \geq r}  \Bigl( \frac{r}{t} \Bigr)^{\!\nicefrac12} \left\| \f - (\f)_{B_t}  \right\|_{\underline{L}^2 \left( B_{t} \right)}  \,,
\end{align*}
proving~\eqref{e.harmdecay.pre}.
\smallskip

\emph{Step 2.} We prove~\eqref{e.harmdecay}. By a simple layer-cake formula, we obtain, for every~$g \in L_{\mathrm{loc}}^2(\R^d)$,
\begin{align*}  
\| g \|_{L^2(\Phi_r)}^2 \leq C(d) \int_{1}^\infty \exp\bigl(-c t^2 \bigr) \| g \|_{\underline{L}^2(B_{tr})}^2 \,\frac{dt}{t}  \,.
\end{align*}
Since~$| (g)_{B_{tr}} - (g)_{B_r} | \leq  t^{\nicefrac d2} \| g - (g)_{B_{tr}}  \|_{\underline{L}^2(B_{tr})}$, we may decrease the constant~$c$ above and smuggle in the averages by the triangle inequality:
\begin{align*}  
\| g  -(g)_{B_{r}}  \|_{L^2(\Phi_r)}^2 
\leq 
C(d) \int_{1}^\infty \exp\bigl(-ct^2\bigr) \| g - (g)_{B_{tr}} \|_{\underline{L}^2(B_{tr})}^2 \,\frac{dt}{t}  \,.
\end{align*}
Applying the estimate for~$g = \nabla w$, and then appealing to~\eqref{e.harmdecay.pre} completes the proof.
\end{proof}

Combining~\eqref{e.verygood.Jestimate} with the result of Lemma~\ref{l.scalecomparison} and~\eqref{e.coarsen.smallness}, and performing some routine manipulations with the triangle inequality and Young's inequality, we arrive at the main result of this section, the following \emph{additivity defect estimate}: There exists~$\alpha(\beta,d,\lambda,\Lambda) \in (0,1)$ such that, for every~$r,R \in [1,\infty)$ with~$r \leq R$, 
\begin{align}
\label{e.additivityestimate}
\big|
\bigl( \b_R - \b_r \ast \Phi_{\sqrt{R^2-r^2}} \bigr) (x) 
\big|
\leq
C \biggl( \delta \biggl(\frac{\| \X \wedge r  \|_{\Phi_{x,r}} }{r} \biggr)^{\!\alpha} \wedge \mathcal{G}_r(x) \biggr)^{\!2}
\,.
\end{align}
This is the precise version of~\eqref{e.rough.add}.

\subsection{Localization defect estimate}
\label{ss.local}

In this section we ``localize'' the coarse-grained coefficient field~$\b_r(x)$ by replacing it with a field~$\b'_{r,t}$ which has the property that~$\b'_{r,t}(x)$ is~$\mathcal{F}(B_t(x))$--measurable, that is,~$\b'_{r,t}(x)$ depends only on the coefficients in~$B_t(x)$. By a \emph{localization estimate}, we mean an estimate on the size of the difference~$|\b_r - \b'_{r,t}|$. Of course, a very silly idea would be to localize by simply taking~$\b'_{r,t}: = \ahom$ to be deterministic; the localization estimate obtained would be the same as the fluctuation estimate. One would hope to do better than this! As it turns out, we need to do only a tiny bit better than this for the renormalization argument. We just to improve by a factor of~$r^{-\alpha}$ for a tiny~$\alpha>0$, to match~\eqref{e.rough.loc}. 

\smallskip

The tiny improvement will be due to the fact that the spatial averages of the gradient and flux of arbitrary solutions are related by~$\ahom$, up to a small error. Indeed, fixing~$t \in [1,\infty]$ and~$x \in \R^d$,  we obtain using the harmonic approximation in Proposition~\ref{p.harm.we.have} that if~$\X_t(x) \leq t$, then, for every~$\theta \in (0,\frac12]$,~$R \in [\X_t(x),t)$ and every solution~$w\in \A(B_R(x))$,  we have that
\begin{align}
\label{e.genfluxmaps.minscale}
\norm{ (\ahom - \a ) \nabla w }_{\underline{H}^{-1}(B_{(1-\theta)R}(x))}
\leq
C_\theta \Bigl( \frac{\X_t(x)}{R} \Bigr)^{\! \alpha} 
\| \nabla w \|_{\underline{L}^{2} \left( B_{R}(x) \right)}
\,.
\end{align}
Consequently, for every~$r\geq \X_t(x)$,~$t \geq r^{1+\ep}$ and~$w \in \A(B_t(x))$, we have (taking~$\delta$ smaller in the above estimate, if necessary, depending on~$\ep$) by a straightforward layer-cake argument that there exist constants~$C(\ep,\delta,\dataref) < \infty$ and~$c(d) \in (0,1)$ such that
\begin{multline}
\label{e.genfluxmaps.heatscale}
\bigl| 
(  \indc_{B_{t/2}(x)} (\ahom  - \a ) \nabla w ) \ast \Phi_{r}(x)   \bigr|
\\ 
\leq
C \Bigl( \frac{\X_t(x)}{r} \Bigr)^{\! \alpha}  \big\|   \indc_{B_{t/2}(x)}  \nabla w  \big\|_{L^2(\Phi_{x,2r})} 
+ C \exp(-c r^{2\ep}) \big\| \nabla w \big\|_{\underline{L}^2(B_{t/2}(x))} 
\,.
\end{multline}
We spare the reader the detailed proofs of~\eqref{e.genfluxmaps.minscale} and~\eqref{e.genfluxmaps.heatscale}, which can also be found in~\cite{AKMBook}. We insist, however, that~\eqref{e.genfluxmaps.minscale} is indeed a routine exercise using the three ingredients mentioned, and~\eqref{e.genfluxmaps.heatscale} is obtained by first using a smooth partition of unity to chop the integral into pieces supported in dyadic annuli and then applying~\eqref{e.genfluxmaps.minscale} in each annulus separately.

\smallskip

A rough sketch of the plan to obtain a localization estimate, closely based on the ideas of~\cite[Section 4.4]{AKMBook}, is as follows. The first and main step is to produce a \emph{local} version of the correctors, denoted by~$\mathcal{A}_1^{(x,r,t)}$, which possesses the following properties. If~$r \geq \X(x)$ and~$t\geq r^{1+\ep}$, then the dimension of~$\mathcal{A}_1^{(x,r,t)}$ is the same as~$\dim \A_1=d+1$ and it has a basis with every element belonging to~$\A(B_{t})$. Moreover,~$\mathcal{A}_1^{(x,r,t)}$ is~$\F(B_t(x))$--measurable and  approximates~$\A_1$ in the sense that, for~$t := r^{1+\rho}$, 
\begin{equation}
\label{e.vague.correctorloc}
\inf_{\psi \in  \mathcal{A}_1^{(x,r,t)}} \left\| \nabla \psi_{e} - \nabla \psi \right\|_{\underline{L}^2(B_{t}(x))} 
\lesssim \text{(size of the fluctuations of~$\b_r(x)$).}
\end{equation}
We will then define~$\b'_{r,t}$ to be the matrix which gives the gradient-to-flux map for the localized correctors, that is,
\begin{equation}
\label{e.local.coarsened}
\b'_{r,t} (x) (\nabla \psi   \ast \Phi_{r} )(x)
=
\bigl( \a \nabla \psi   \bigr)  \ast \Phi_{r}(x)
\,, 
\qquad 
\forall \psi \in \mathcal{A}_1^{(x,r,t)}\,.
\end{equation}
We then observe that, for suitably chosen~$e\in B_1$ and for every~$\psi \in \mathcal{A}_1^{(x,r,t)}$, 
\begin{align}
\label{e.localmagic}
\big| \b_r - \b'_{r,t} \big| (x) 
& 
\leq 
C
\Big| ( \b_r(x) {-} \b'_{r,t}(x)  )
(\nabla \psi_e  \ast \Phi_{r}) (x) \Big| \indc_{\{r \geq \X(x)\}} + C \indc_{\{r < \X(x)\}} 
\notag \\ & 
= 
C \Bigl| \Bigl(   \bigl(  ( \a {-} \b'_{r,t}(x) ) \nabla \psi_e \bigr) \ast \Phi_{r}  \Bigr) (x) \Bigr| \indc_{\{r \geq \X(x)\}} 
+ C \indc_{\{r < \X(x)\}}  
\notag \\ & 
=
C \Bigl| 
\Bigl(  \bigl( ( \a {-} \b'_{r,t}) ( \nabla \psi_{e} {-} \nabla \psi )\bigr) \ast \Phi_{r}  \Bigr)  (x) \Bigr|  \indc_{\{r \geq \X(x)\}} 
 + C \indc_{\{r < \X(x)\}} 
\,.
\end{align}
We then estimate the right side by applying~\eqref{e.genfluxmaps.heatscale} to the difference~$w=\psi_e - \psi$ with some~$\psi$ we may choose, and then use that~$\b'_{r,t}$ is close to~$\ahom$, and finally apply~\eqref{e.vague.correctorloc}. This yields the desired result, namely that~$\big| \b_r - \b'_{r,t} \big|$ is smaller than the fluctuations of~$\b_r$, with the additional smallness coming from the right side of~\eqref{e.genfluxmaps.heatscale}.

\smallskip

To make this precise, we must formalize the vague estimate~\eqref{e.vague.correctorloc}.
At first glance, it might seem that the best way to approximate the correctors using only the knowledge of the coefficients in~$B_t$, with~$t = r^{1+\ep}$, would be to simply solve the Dirichlet problem in~$B_t$ with affine boundary data on~$\partial B_t$. Unfortunately, this will lead to an estimate which is not close to good enough: there will be boundary layer errors of order~$t^{-1}$ (with~$t$ only slightly larger than~$r$), while the size of the fluctuations of~$\b_r$ may be much smaller, for instance~$O(r^{-\frac d2(1-\beta)})$, possibly with~$\beta=0$. The renormalization argument would saturate on the boundary layer errors rather than the much smaller error given by the mixing condition. To do better, we define our localized correctors by looking for solutions very close to an affine function \emph{in a weak norm} rather than a strong one. 
Specifically, we will look for the solution of the equation in~$B_t$ which---after mollification by the heat kernel on scale~$cr$---is closest to an affine function in~$L^2$. We then pass from estimates on mollified solutions to the strong~$L^2$ norm of the solutions themselves using the following lemma, which is based on the large-scale regularity theory.

\begin{lemma}
\label{l.MSP.heat}
Fix~$k \in \N$. Let~$\X$ and~$\X_t$ be the minimal scales from Theorems~\ref{t.Ck1} and~\ref{t.Ck1.local}, respectively. There exist constants~$c(k,\dataref)\in (0,1)$ and~$C(k,\dataref)<\infty$ such that if~$\sigma \in (0,c]$ and either~$v \in \A_k$ and~$r \geq C \X$ or~$v \in \A_k(B_t)$ and~$r \geq C \X_t$ for some~$t \geq r$, then we have that
\begin{align}  \label{e.MSP.heat}
r^2 \fint_{B_{r} } |\nabla v|^2
+
\fint_{B_{r} } |v|^2  \leq C \fint_{ B_r } | v \ast \Phi_{\sigma r}|^2 
.
\end{align}
\end{lemma}
\begin{proof}
The argument here is a concise version---just a sketch, really---of a much more detailed proof presented in~\cite{AKMBook}. 
Fix~$k \in \N$. We treat only the case~$v \in \A_k(B_t)$ and~$r \geq  C \X_t$ with~$t\geq r$. For this, we see, by using the Caccioppoli estimate together with Theorem~\ref{t.Ck1.local}, that there exists a constant~$C(d,\lambda,\Lambda)<\infty$ such that, for every~$s \geq C^{-1}r$, 
\begin{align}  \label{e.MSP.heat.step0}
s^2 \fint_{B_s} | \nabla v|^2  \leq C \left(\frac sr\right)^{2k} \fint_{B_{r}} |v - (v)_{B_r}|^2  .
\end{align}
Notice that~$v$ has the extension outside of~$B_t$ as a member of~$\Ahom_k$.

\smallskip
 
\emph{Step 1.}  Let~$\Theta_r(x) = r^{-d} \exp( r^{-1} |x|)$. We show that, for every~$w \in L^2(\Theta_r)$ and~$\sigma \in (0,1]$, we have that 
\begin{align}  \label{e.MSP.heat.step1}
\int_{\Theta_r} |w|^2 
\leq 
6 \int_{\Theta_r} | w \ast \Phi_{\sigma r}|^2 
+ 
6  \int_{0}^{\sigma r}  s \int_{\Theta_r} | \nabla w \ast \Phi_{s}|^2 \, ds
.
\end{align}
 This inequality, which can be thought of as a ``multiscale Poincar\'e inequality with heat kernels'' (cf. Proposition~\ref{p.MSP}), follows by testing the heat equation satisfied by~$w \ast \Phi_{\sqrt{t}}$ with~$w \ast \Phi_{\sqrt{t}} \Theta_r$. A detailed proof can be found in~\cite[Lemma 4.19]{AKMBook}. 
 
 \smallskip
 
\emph{Step 2.}
We next specialize to the case~$v \in \A_k(B_t)$ and~$r \geq  \X_t$ and show that for every~$k \in \N$ there exists~$C(k,\dataref)<\infty$ such that, for every~$\sigma \in (0,1]$ and~$r \geq \X_t$, we have
\begin{align} \label{e.MSP.heat.step2}
 \int_{0}^{\sigma r} s \int_{\Theta_r} | \nabla v \ast \Phi_{s}|^2 \, ds
\leq 
C \sigma^2 \fint_{B_r} |v|^2 
.
\end{align}
By Jensen's inequality and Fubini's theorem, we first see that there exists a constant~$C(d)<\infty$ such that, for every~$s \in (0,\sigma r]$, 
\begin{align*} 
\int_{\Theta_r} | \nabla v \ast \Phi_{s}|^2 
& 
\leq 
C  \int_{\R^d} \Theta_r(z) | \nabla v(z)|^2 \biggl( \int_{\R^d} s^{-d} \exp\Bigl( - \frac{|z-y|^2}{4s^2} + \frac{|z|-|y|}{r} \Bigr) \, dy \biggr) \, dz 
\\ \notag 
&
\leq
C \int_{\Theta_r} | \nabla v|^2 
.
\end{align*}
Using this togther with~\eqref{e.MSP.heat.step0} yields~\eqref{e.MSP.heat.step2}.

 \smallskip
 
\emph{Step 3.}
In the last step of the proof, we show~\eqref{e.MSP.heat}. Applying~\eqref{e.MSP.heat.step1} for~$w = v$ together with~\eqref{e.MSP.heat.step2}, we take~$c$ so small that~$c^2 C =\frac12$ and then reabsorb 
the last term in~\eqref{e.MSP.heat.step1} applied with~$w = v$ and obtain, for every~$\sigma \in (0,c]$, 
\begin{align*}  
\int_{\Theta_r} v^2  \leq C \int_{\Theta_r} | v \ast \Phi_{\sigma r}|^2 
.
\end{align*}
We may then apply~\eqref{e.MSP.heat.step0} once more to reabsorb the tail terms to the left. Finally, renaming the variables, inflating the constant~$C$ and multiplying the minimal scale with~$C$ yields~\eqref{e.MSP.heat}, finishing the proof.
\end{proof}

Using the previous lemma, we turn to the rigorous formulation of the vague inequality~\eqref{e.vague.correctorloc}. 
Note that the right side of~\eqref{e.correctorloc} can be bounded in terms of the fluctuations of~$\b_r$ by Lemma~\ref{l.scalecomparison}, and therefore we do obtain an estimate faithful to~\eqref{e.vague.correctorloc}.  This is the first \emph{localization estimate}. Recall that~$\A_k(U)$ is defined above in~\eqref{e.localAk} and that its elements are functions on the whole space~$\R$.

\begin{lemma}[Localization estimates]
\label{l.correctorloc}
Assume that~$\P$ is a~$\Zd$--stationary measure on~$(\Omega,\F)$ satisfying~$\CFS(\beta,\Psi)$ for some~$\beta \in [0,1)$. 
Fix~$x\in \R^d$,~$\ep \in (0,\frac12]$,~$r \in [1,\infty)$ and~$t>r$ such that~$t^{1-\ep} \geq r$. There exist constants~$\alpha(\beta,d,\lambda,\Lambda) \in (0,1)$ and~$C(\ep,\dataref)<\infty$ such that the following assertions are valid. 
\begin{itemize}

\item 
There exist~$k(\ep,d) \in \N$ and, for~$e \in \R^d$, an~$\F(B_{t}(x))$-measurable element~$\psi_e^{(x,r,t)} \in \A_k(B_{t/3d})$ satisfying the following properties.  The mapping~$e \mapsto \psi_{e}^{(x,r,t)}$ is linear and if~$r \geq \X(x)$ and~$t^{1-\ep} \geq C r$, then the vector space
\begin{align} \label{e.localA1}
\mathcal{A}_1^{(x,r,t)} := \big\{   \psi_{e}^{(x,r,t)} \, : \, e \in \R^d \big\} 
\end{align}
has the same dimension as that of~$\A_1$. Moreover, we have, for every~$e \in \R^d$, that
\begin{align}  \label{e.correctorloc.converse}
\inf_{\psi \in \mathcal{A}_1^{(x,r,t)}}
\big\|  \nabla \psi - \nabla \psi_e \big\|_{L^2(\Phi_{x,2r})}
\indc_{\{ r\geq \X(x) \}}
\leq 
C
\Bigl( \frac{r}{t^{1-\ep}} \Bigr)
\mathcal{G}_{t^{1-\ep}}(x)  |e| 
\,. 
\end{align}

\item 
The matrix defined according to~\eqref{e.local.coarsened} exists. Moreover, it satisfies, 
for every,~$r,t\geq 1$ with~$t^{1-\ep} \geq C r$ and~$r\geq \X_t(x)$,
\begin{equation}
\label{e.matrix.local}
\big| {\b}'_{r,t}(x) - \ahom \big| \leq  \delta \biggl( \frac{\X_t(x) \wedge r}{r} \biggr)^{\alpha} .
\end{equation}

\item
The local matrix~$\b'_{r,t}$ approximates~$\b_r~$in the sense that, for every 
$r,t\geq 1$ with~$t^{1-\ep} \geq Cr$,
\begin{equation} \label{e.localizationestimate}
\big| \bigl(\b_r - \b'_{r,t} \bigr)(x) \big| 
\indc_{\left\{ r \geq \X (x) \right\}}
\leq
C
\Bigl( \frac{r }{t^{1-\ep}} \Bigr)
\Bigl( \frac{\X(x) \wedge r}{r} \Bigr)^{\! \alpha} 
\mathcal{G}_{t^{1-\ep}}(x)
\,.
\end{equation}

\item 
The Malliavin derivative satisfies the following deterministic bound: If~$r,t\geq 1$ satisfy~$t^{1-\ep} \geq Cr$, then
\begin{equation}  \label{e.malliavin.local}
\bigl| \partial_{\a(B_{t}(x))}  \b'_{r,t}\bigr|  (x) \leq C \log t .
\end{equation}
\end{itemize}

\end{lemma}

\begin{proof}
Fix~$r \in [1,\infty)$,~$\ep \in (0,\frac12]$ and~$t \geq r$ such that~$t^{1-\ep} \geq 2r$, and denote
\begin{equation}  \label{e.randk.fixed}
s :=d^{-1}  t^{1-\ep}
\qand
k := \lceil 20d \ep^{-1}  \rceil.
\end{equation}
Throughout the proof, let~$\X(x)$ and~$\X_t(x)$ be the minimal scales from Theorems~\ref{t.Ck1} and~\ref{t.Ck1.local} multiplied with a multiplicative factor depending on~$k$ above. Notice that~$\X_t(x) \leq \X(x)$ and~$\X_t(x) \to \X(x)$ as~$t \to \infty$. Corresponding to~$k$ in~\eqref{e.randk.fixed}, an application of Lemma~\ref{l.MSP.heat} yields that there exist constants~$\sigma \in (0,1]$ and~$C<\infty$ such that if~$s \geq C\X(x)$ ($s \geq C\X_t(x)$)
\begin{align} \label{e.MSP.heat.really}
s^2 \fint_{B_{s}(x) } | \nabla  v|^2 + \fint_{B_{s}(x) } | v|^2 
\leq 
C \fint_{ B_{s}(x)} | v \ast \Phi_{\sigma s}|^2 
\qquad \forall v \in \A_k  \; \; (v \in \A_k(B_t(x)) ) 
\, .
\end{align}
Given the right side in~\eqref{e.MSP.heat.really}, define an inner product
\begin{align}  \label{e.innerproduct}
\langle v,w\rangle_x
:= 
s^{-2}
\fint_{ B_{s}(x) } (v \ast \Phi_{\sigma s}) (w \ast \Phi_{\sigma s})
\qand
E_x(w) = \| w\|_x^2 := \langle w,w\rangle_x
\,. 
\end{align}
For the rest of the argument, we may take, without loss of generality,~$x=0$ and suppress it from the notation. Denote 
$\A_k^{0}(B_t) = \{ v \in \A_k(B_t) \, : \, E(v) = 0 \}$ and notice that if~$s \geq C\X_t$, then~\eqref{e.MSP.heat.really} implies that 
$\A_k^{0}(B_t) \subseteq \{ v \in \A_k(B_t) \, : \, v = 0 \mbox{ a.e. in } B_s \}$. 

\smallskip

\emph{Step 1.} Construction of~$\phi_e^{(r,t)}$.  
Fix~$e \in \R^d$ and let~$w_e$ be the unique minimizer of~$E(\cdot-\ell_e)$ over the quotient space~$\A_{k}(B_t) \backslash \A_k^{0}(B_t)$, where~$\ell_e(x) := e \cdot x$. It is clear that~$w_e$ is~$\F(B_t)$-measurable. The uniqueness is easy to see since~$e \mapsto w_e$ is a linear map by the first variation of the quadratic problem. 
Set 
\begin{align}  \label{e.localcorrdef}
\phi_{e}^{(r,t)}  : = \minscale_{r/2, t} w_e  
\quad \mbox{with } 
\minscale_{r, t} := \min\big\{1, 2 \bigl(1 -   r^{-1} \X_t \bigr)_+ \big\} 
\,.
\end{align}
The prefactor of~$w_e$ above is independent of the spatial variable and thus, since~$w_e \in \A(B_t)$, also~$\phi_{e}^{(r,t)}$ belongs to~$\A(B_t)$, as claimed.  Notice also that~$\phi_{e}^{(r,t)}=0$ if~$r \leq \frac12 \X_t$. 

\smallskip

\emph{Step 2.} We will define the matrix~$\b_{r,t}'$ and show~\eqref{e.matrix.local}. To do that, we need some auxiliary tools. Theorem~\ref{t.Ck1} yields that, by making~$\X$ larger by a universal factor, 
we find, for every~$e \in \R^d$,~$\psi_e \in \A_1$ such that, for every~$r\geq 1$, 
\begin{align*}  
r^{-1} \| \psi_e - \ell_e \|_{L^2(\Phi_{2r})} \leq \delta^2 \Bigl( \frac{\X}{r} \Bigr)^{\! \alpha} |e| \,.
\end{align*}
On the other hand, as in the case of~\eqref{e.liouvillec.applied}, now applying Theorem~\ref{t.Ck1.local}, we find~$\theta_e \in \A_1(B_t)$ such that, for~$r\geq \frac12 \X_t$, 
\begin{align} \label{e.thetae.est}
r^{-1} \| \theta_e - \ell_e \|_{L^2(\Phi_{2r})} + |(\nabla \theta_e)\ast \Phi_r(0) - e| + |(\a \nabla \theta_e)\ast \Phi_r(0) - \ahom e|  \leq \delta^2 \Bigl( \frac{\X_t}{r} \Bigr)^{\! \alpha} |e| \,.
\end{align}
By the Caccioppoli estimate, we then obtain that, for~$r \geq \X$, 
\begin{align*}  
\| \nabla \theta_e - \nabla \psi_e \|_{L^2(\Phi_{2r})} \leq C  \delta^2 \Bigl( \frac{\X}{r} \Bigr)^{\! \alpha} |e| \,. 
\end{align*}
Using the Lipschitz bound, Lemma~\ref{l.MSP.heat}, and the minimality of~$w_e$, we also get
\begin{align}  \label{e.we.vs.thetae}
\big\| \nabla w_e - \nabla \theta_e \big\|_{L^2(\Phi_{2r})}^2  \leq C E(w_e - \theta_e ) \leq 
C E(\theta_e - \ell_e)  \leq C \delta^4  \Bigl( \frac{\X_t}{r} \Bigr)^{2\alpha} |e|^2  \,.
\end{align}
This, together with~\eqref{e.thetae.est} and the linearity of~$e \mapsto w_e$, implies that~$\dim\{\nabla w_e \}_{e \in \R^d} = d$ provided that~$r \geq \frac12 \X_t$. 
By taking~$\delta \in(0,1)$ small enough, this also gives us, for every~$e \in \R^d$ and~$r \geq \frac12 \X$,
\begin{equation*}  
\| \nabla w_e - \nabla \psi_e \|_{L^2(\Phi_{2r})} 
\leq 
\delta \bigl(  \| \nabla w_e \|_{L^2(\Phi_{2r})} \wedge \| \nabla \psi_e \|_{L^2(\Phi_{2r})} \bigr) \,.
\end{equation*}
Since both~$e \mapsto \psi_e$ and~$e \mapsto w_e$ are linear and~$\A_1 = \{ \psi_e \}_{e \in \R^d}$, we have by the previous display that~$\A_1$ and~$\A_1^{(r,t)}$ are isomorphic for~$r\geq 4 \X$. Therefore, in order to show~\eqref{e.correctorloc.converse}, it is enough to prove that, for every~$e \in \R^d$, 
\begin{align*}  
\inf_{\phi \in \A_1 }
\big\| \nabla \phi - \nabla w_e \big\|_{L^2(\Phi_{2r})}
\indc_{\{ r \geq \X \}} 
\leq 
C |e| \frac{r}{s} \mathcal{G}_{s}(0) 
\,.
\end{align*}

\smallskip

We then define the matrix~$\b_{r,t}'$. Let the matrices~$G$ and~$Q$ have entries 
\begin{align} \label{e.JandQ}
\notag
G_{ij} & := \bigl(  \minscale_{r,t}  \bigl((\nabla w_{e_i} - e_i) \ast \Phi_{r}\bigr)(0)+ e_i \bigl)_j \,, 
\\ 
Q_{ij} & := \bigl(  \minscale_{r,t}  \bigl((\a \nabla w_{e_i} - \ahom e_i) \ast \Phi_{r}\bigr)(0)+ \ahom e_i \bigl)_j
\end{align}
with~$\minscale_{r,t}$ defined by~\eqref{e.localcorrdef}. We obtain, by~\eqref{e.thetae.est} and~\eqref{e.we.vs.thetae}, that the matrix~$G$ is an approximation of unity and~$Q$ is an approximation of~$\ahom$:
\begin{align*}  
\big| G - \Id \big| \leq \delta  \Bigl( \frac{\X_t}{r} \Bigr)^{\! \alpha} 
\qand 
\big| Q - \ahom \big| \leq \delta  \Bigl( \frac{\X_t}{r} \Bigr)^{\! \alpha} \,.
\end{align*}
The invertibility of~$G$ is then immediate if~$\delta \leq \frac12$. We now define 
\begin{align*}  
\b_{r,t}' := Q G^{-1}  \,.
\end{align*}
Notice that if~$r \geq \frac12 \X_t$, then~\eqref{e.local.coarsened} holds by~\eqref{e.localcorrdef} and the linearity of the mapping~$e\mapsto w_e$. If~$r < \frac12 \X_t$, then~\eqref{e.local.coarsened} is valid trivially (at~$x = 0$) since~$\mathcal{A}_1^{(r,t)} = \{0\}$ in this case. 

\smallskip

\emph{Step 3.} We prove that,  for~$\phi_{e}^{(r,t)}$ constructed in~\eqref{e.localcorrdef}, we have that 
\begin{align}  \label{e.correctorloc}
\inf_{\phi \in \A_1 }
\big\| \nabla \phi - \nabla \phi_{e}^{(r,t)}\big\|_{L^2(\Phi_{2r})}
\indc_{\{ r \geq \X_t \}} 
\leq 
C |e| 
\biggl( 
 \frac{r}{s} \mathcal{G}_{s}(0) 
+
 \indc_{\{ s \leq C \X \}}
\biggr) 
\,.
\end{align}
As discussed in the previous step, this also proves~\eqref{e.correctorloc.converse} since~$\A_1$ and~$\A_1^{(r,t)}$ are isomorphic for~$r\geq 4\X$.  Assume that~$r \geq  \X_t$. First, if~$s\leq C \X$, we take~$\phi = 0$ and use the Lipschitz bound and~\eqref{e.MSP.heat.really} to obtain 
\begin{align*}  
\indc_{\{ r \geq \X_t \}} \big\| \nabla w_{e} \indc_{B_s} \big\|_{L^2(\Phi_{2r})} 
\leq 
C \indc_{\{ r \geq \X_t \}} \big\| \nabla w_{e}  \big\|_{\underline{L}^2(B_s)} 
\leq
C |e| 
\end{align*}
and, using also the growth of~$\A_k(B_t)$ functions and again~\eqref{e.MSP.heat.really}, 
\begin{align*}  
\indc_{\{ r \geq \X_t \}} \big\| \nabla w_{e}  \indc_{\R^d \setminus B_s} \big\|_{L^2(\Phi_{2r})} 
\leq 
C |e| \int_{s}^\infty \biggl(\frac{s'}{s} \biggr)^{\!\!k+d} \exp\biggl( - c  \biggl( \frac{s'}{r} \biggr)^{\!\!2} \, \biggr) \, ds'  
\leq 
C|e| 
\,.
\end{align*}
Thus, we obtain 
\begin{equation*}  
\inf_{\phi \in \A_1 }
\big\| \nabla \phi - \nabla \phi_{e}^{(r,t)}\big\|_{L^2(\Phi_{2r})}
\indc_{\{ r \geq \X_t \}} \indc_{\{ s \leq C \X \}} 
\leq 
C |e|
\indc_{\{ s \leq C \X \}} 
\,.
\end{equation*}
We then assume that~$s\geq C\X$. We let~$\psi_e \in \A_1$ be such that~$\nabla \psi_e \ast  \Phi_s (0)  = e$. This is possible since, by~\eqref{e.liouvillec} for~$k=1$ and the fact that~$s \geq C\X$, we have that~$|\nabla \psi_e \ast  \Phi_s (0) - e| \leq \frac12 |e|$, and hence~$e \mapsto \nabla \psi_e \ast  \Phi_s (0)$ is a linear bijection for~$s\geq C\X$. Moreover, we have, by Theorem~\ref{t.Ck1.local}, that~$\sup_{\tau \geq s} \| \nabla \psi_e \|_{L^2(\Phi_\tau)} \leq C|e|$. For convenience, we take the additive constant for~$\psi_e$ such that~$\fint_{B_s} \psi_e\ast \Phi_{\sigma s} = 0$. By applying Theorem~\ref{t.Ck1.local}, 
we find~$\eta_e \in  \A_{k}(B_t)$ such that, for every~$\tau  \in [r,t]$, 
\begin{align} \label{e.loc.corrclose2}
\frac1\tau \left\|  \psi_e -  \eta_e \right\|_{\underline{L}^2 \left( B_{\tau} \right)} 
+
\left\|  \nabla  \psi_e -  \nabla  \eta_e \right\|_{\underline{L}^2 \left( B_{\tau} \right)} 
\leq 
C  \left( \frac{\tau}{t}\right)^{\!k}  \left\| \nabla \psi_e \right\|_{\underline{L}^2 \left( B_{t} \right)}  
\leq 
C\left( \frac{\tau}{t}\right)^{\!k} |e|
\,.
\end{align}
Next, we compare the values of~$E(\cdot - \ell_e)$ for~$w_e$ and~$\eta_e$ to obtain an estimate for~$E(w_e-\eta_e)$.  First, by the minimality and the triangle inequality, we get
\begin{equation*}  
E(w_e-\eta_e) 
= 
E(\eta_e-\ell_e) - E(w_e-\ell_e) 
\leq 
2 E(\psi_e-\ell_e) +  2 E(\psi_e-\eta_e) 
\,.
\end{equation*}
On the one hand, since~$\fint_{B_s} (\psi_e - \ell_e) \ast \Phi_{\sigma s} = 0$ and~$\nabla \psi_e \ast  \Phi_s (0)  = e$, the Poincar\'e inequality yields
\begin{equation*}  
 E (\psi_e - \ell_e) 
 \leq 
 C \| \nabla \psi_e \ast \Phi_{\sigma s} - \nabla \psi_e \ast \Phi_{s}(0) \|_{L^2(\Phi_{(1-\sigma^2)^{\nicefrac 12} s })}^2 
 \,.
\end{equation*}
On the other hand, using~\eqref{e.loc.corrclose2}, we get
\begin{equation*}  
E(\psi_e-\eta_e)  \leq C |e|^2 t^{-20d} \,.
\end{equation*}
Therefore, by applying Lemma~\ref{l.scalecomparison}, we obtain 
\begin{equation*}  
E(w_e-\eta_e) 
\leq 
C\| \nabla \psi_e \ast \Phi_{\sigma s} - \nabla \psi_e \ast \Phi_{s}(0) \|_{L^2(\Phi_{(1-\sigma^2)^{\nicefrac 12} s })}^2 
 + C|e|^2 t^{-20d}  
 \leq 
 C|e|^2  \mathcal{G}_{s}^2(0)
\,.
\end{equation*}
Furthermore, the triangle inequality and~\eqref{e.MSP.heat.really} give us
\begin{equation*}  
\fint_{B_{s} } |\nabla \eta_e - \nabla w_e|^2 
\leq 
C E(\eta_e - w_e) 
\leq 
C |e|^2 \mathcal{G}_{s}^2(0)
\,.
\end{equation*}
Using this together with~\eqref{e.loc.corrclose2}, the triangle inequality,~$C^{1,1}$-regularity in Theorem~\ref{t.Ck1} and~$C^{0,1}$-regularity in Theorem~\ref{t.Ck1.local}, we find yet another corrector, say~$\hat\phi \in \A_1$, such that, for every~$\tau \in [r,s]$, 
\begin{equation*}  
\| \nabla w_e - \nabla \eta_e - \nabla \hat\phi \|_{\underline{L}^2(B_\tau)}  \leq \frac{C |e|}{s} \bigl( \tau \vee (\X \wedge s) \bigr) \mathcal{G}_{s}(0) .
\end{equation*}
Then, a simple layer-cake formula together with~\eqref{e.loc.corrclose2} yields~\eqref{e.correctorloc}. 

\smallskip

\emph{Step 4.}
We show the bounds for the Malliavin derivatives. For this, fix~$\tilde \a \in \Omega$ such that~$|\tilde \a -\a | \leq h \indc_{B_t}$ with small~$h>0$. 
We first show that the minimizers~$w$ and~$\tilde w$ corresponding coefficients~$\a$ and~$\tilde \a$, respectively, satisfy, for every~$r \geq \frac12 \X_t$,
\begin{equation}  \label{e.supplementary.w.vs.tildew}
\big\| \nabla w - \nabla \tilde w \big\|_{L^2(\Phi_{2r})}  \leq C h \log t\,.
\end{equation} 
For this, we will find a suitable orthonormal basis and first look for the elements to orthonormalize. Assume that~$r \geq \frac12 \X_t$, fix~$m\in \N$, and let~$p$ be a homogeneous polynomial in~$\Ahom_m$. By Theorem~\ref{t.Ck1.local} we find~$w_p \in \A_m(B_t)$ such  that, for every~$s \in [r,t]$,  
\begin{equation*}  
\left\| w_p - p \right\|_{\underline{L}^2 \left( B_{t} \right)} 
\leq 
\delta \left\| p \right\|_{\underline{L}^2 \left( B_{t} \right)},
\end{equation*}
Here~$\delta$ can be made arbitrarily small by enlarging~$\X_t$ by a multiplicative factor of~$C_\delta$. By taking~$h$ small enough, we can guarantee that~$\tilde \X_t \leq 2 \X_t \leq 4r$. We claim that there exists~$\tilde w_p \in \tilde \A_k(B_t)$ such that 
\begin{equation} \label{e.tildewp.vs.wp}
r\left\| \nabla w_p - \nabla \tilde w_p \right\|_{L^2(\Phi_{2r})} +  \left\| w_p - \tilde w_p \right\|
 \leq 
 C h \log t \left\| p \right\|
 .
 \end{equation}
Let~$\tilde w \in \tilde \A_k$ have the boundary values of~$w_p$ on~$\partial B_t$. Then~$\tilde \eta= \tilde w - w_p$ solves the equation 
\begin{equation*}  
\left\{
\begin{aligned}
& -\nabla \cdot \left( \tilde \a \nabla \tilde \eta  \right) 
= -\nabla \cdot \left( (\a - \tilde \a) \nabla w_p \right)  &  \mbox{in} & \ B_{t } , \\
& \tilde \eta  = 0 & \mbox{on} & \ \partial B_{t } .
\end{aligned}
\right.
\end{equation*}
Let~$\theta \in (0,\frac12)$ and let~$\tilde \phi_m \in \tilde \A_k(B_t)$ be such that 
\begin{equation*}  
\|\tilde \eta - \tilde \phi_m   \|_{\underline{L}^2 (B_{\theta^{m} t })} 
= 
\inf_{\tilde \phi \in \tilde \A_k(B_t)}\|\tilde \eta - \tilde \phi   \|_{\underline{L}^2 (B_{\theta^{m} t })} 
\,,
\end{equation*}
and let~$\tilde \eta_m$ solve
\begin{equation*}  
\left\{
\begin{aligned}
& -\nabla \cdot \left( \tilde \a \nabla \tilde \eta_m  \right) 
= 0 &  \mbox{in} & \ B_{\theta^{m} t } \, , \\
& \tilde \eta_m  =  \tilde \eta - \tilde \phi_m & \mbox{on} & \ \partial B_{\theta^{m} t } \,.
\end{aligned}
\right.
\end{equation*}
Then 
\begin{equation*}  
\|\tilde \eta_m - (\tilde \eta - \tilde \phi_m)     \|_{\underline{L}^2 (B_{\theta^{m} t }) } 
\leq 
C h  \left\| p \right\|_{\underline{L}^2 ( B_{\theta^{m} t } )} \,.
\end{equation*}
It follows by Theorem~\ref{t.Ck1.local} (recall that we assume that~$\tilde \X_t \leq 2\X_t \leq 2r$) that 
\begin{equation*}  
\inf_{\tilde \phi \in \tilde \A_k(B_t)} \|\tilde \eta_m - \tilde \phi   \|_{\underline{L}^2 (B_{\theta^{m+1} t })}
\leq
C \theta^{k+1}\| \tilde \eta_m  \|_{\underline{L}^2 (B_{\theta^{m} t }) }  
\end{equation*}
and, therefore, by the triangle inequality,
\begin{equation*}  
\|\tilde \eta - \tilde \phi_{m+1}    \|_{\underline{L}^2 (B_{\theta^{m+1} t }) } 
\leq
C \theta^{k+1}
\|\tilde \eta - \tilde \phi_{m}   \|_{\underline{L}^2 (B_{\theta^{m} t })} 
 + C_\theta h  \left\| p \right\|_{\underline{L}^2 ( B_{\theta^{m} t } )}  
 \,.
\end{equation*}
Divide by~$\theta^{k(m+1)}$ and choose~$\theta$ small to deduce that 
\begin{equation*}  
\theta^{-k(m+1)}\|\tilde \eta - \tilde \phi_{m+1}    \|_{\underline{L}^2 (B_{\theta^{m+1} t }) } 
\leq
\frac12 
\theta^{-km}
\|\tilde \eta - \tilde \phi_{m}   \|_{\underline{L}^2( B_{\theta^{m} t })} 
 + C h  \left\| p \right\|_{\underline{L}^2 ( B_{ t } )}  
 \,.
\end{equation*}
Thus, by iterating, we obtain that, for every~$s \in [r,t]$, 
\begin{equation*}  
\inf_{\tilde \phi \in \tilde \A_k(B_t)}\|\tilde \eta - \tilde \phi   \|_{\underline{L}^2 (B_{s})} \leq C h  \left\| p \right\|_{\underline{L}^2 \left( B_{ s } \right)}. 
\end{equation*}
By summing over the scales and using the Caccioppoli estimate, we find~$\tilde \phi$ such that 
\begin{equation*}  
s \|\nabla \tilde \eta - \nabla \tilde \phi   \|_{\underline{L}^2 (B_{s})} 
+
 \| \tilde \eta - \tilde \phi   \|_{\underline{L}^2 (B_{s})}
\leq C h \log t  \left\| p \right\|_{\underline{L}^2 \left( B_{ s } \right)}. 
\end{equation*}
Thus, we may take~$\tilde w_p = \tilde w - \tilde \phi$ to obtain~\eqref{e.tildewp.vs.wp} by a layer-cake formula. 

\smallskip

We next construct orthonormal bases for~$\A_k(B_t)$ and~$\tilde \A_k(B_t)$ which are close to each other. By Lemma~\ref{l.MSP.heat}, the inner product~$\langle \cdot,\cdot \rangle$ is a non-degenerate inner product for~$\Ahom_k$. We thus find an orthonormal basis for~$\Ahom_k$ with respect to the inner product~$\langle \cdot,\cdot \rangle$ by finding a representation matrix~$A(d,k,\ahom)$ with uniformly bounded entries with respect to~$r$, and an orthonormal basis~$\{q_j\}_j$ such that 
\begin{equation*}  
q_j = \sum_i A_{ji} p_i,
\end{equation*}
with suitably chosen homogeneous polynomials~$p_i$  in~$\Ahom_k$ with~$\| p_j \| = \langle p_j, p_j \rangle^{\nicefrac12} = 1$. Indeed, one can first orthonormalize the basis for a given degree of homogeneity and then proceed with these polynomials using the Gram-Schmidt process. For each~$p_j$ there are elements~$v_j \in \A(B_t)$ and~$\tilde v_j \in \tilde \A(B_t)$ such that 
\begin{equation*}  
\|v_j - p_j \|  \leq C \delta 
\qand
r\left\| \nabla v_j - \nabla \tilde v_j \right\|_{L^2(\Phi_{2r})}
+
\|v_j - \tilde v_j \| \leq C h \log t\,.
\end{equation*}
We then find another matrix~$\widehat{A}$, which is~$\delta$-perturbation of~$A$, such that~$\{w_j\}_j$ with~$w_j := \sum_{i} \widehat{A}_{ji} v_i$  is an orthonormal basis for~$\A_k(B_t)$. We then have 
 \begin{equation*}  
\|w_j - \widehat{A}_{ji} \tilde v_i\| \leq C h \log t \,, 
\end{equation*}
and from this and~\eqref{e.tildewp.vs.wp} we find orthonormal basis~$\{w_j\}_j$ and~$\{\tilde{w}_j\}_j$ for~$\A_k(B_t)$ and~$\tilde \A_k(B_t)$, respectively, which are at most~$C h \log t$ apart from each other:
\begin{equation} 
\label{e.supplementary.w.vs.tildew.grad.basis}
r\left\| \nabla w_j - \nabla \tilde w_j \right\|_{L^2(\Phi_{2r})}  +  \| w_j - \tilde w_j\| \leq C h \log t \,.
\end{equation}

\smallskip

We then show~\eqref{e.supplementary.w.vs.tildew}. Let~$\{w_j\}_j$ and~$\{\tilde w_j\}_j$ be orthonormal basis we found above for~$\A_k(B_t)$ and~$\tilde{\A}_k(B_t)$, respectively, with respect to the inner product~$\langle \, , \, \rangle$.  We show that the minimizers of~$E(\cdot-\ell)$ over~$\A_k(B_t)$ and~$\tilde{\A}_k(B_t)$, denoted by~$w$ and~$\tilde w$, respectively, satisfy the estimate
\begin{equation}  \label{e.wvswtilde}
\| w - \tilde w \|
\leq 
C_k \max_{j}  \| w_{j} - \tilde w_{j} \| 
\,.
\end{equation} 
We may express both~$w$ and~$\tilde w$ by means of basis functions:
\begin{align} 
\label{e.wwtilde.decomp}
w & =  \sum_j \alpha_{j} w_{j} 
= 
\tilde v + \sum_j \alpha_{j} (w_{j} - \tilde w_{j} ) , &
\tilde w  = 
 \sum_j \tilde \alpha_{j} \tilde w_{j} = 
v + \sum_j \tilde \alpha_{j} (\tilde w_{j} - w_{j}) .
\end{align}
Here~$\tilde v := \sum_j \alpha_{j} \tilde w_{j}  \in \tilde{\A}_k(B_R)$ and 
$v := \sum_j \tilde \alpha_{j} w_{j} \in \A_k(B_R)$. Clearly~$\sum_j |\alpha_j|^2 \leq \| w\|^2 \leq C$. 
The minimality provides us, for every~$v \in \A_k(B_R)$ and~$\tilde v \in \tilde \A_k(B_R)$,  that 
\begin{equation*}  
\langle w - v , w - \ell \rangle  = 0 = \langle \tilde w - \tilde v , \tilde w - \ell \rangle \,,
\end{equation*}
and it follows, by~\eqref{e.wwtilde.decomp}, that
\begin{align*}  
\| w - \tilde w \|^2
& =
\langle w -\ell , w - \tilde w \rangle  - \langle \tilde w -\ell  , w - \tilde w \rangle 
\\ & 
=
 \sum_j \tilde \alpha_j \langle w -\ell , \tilde w_j - w_j \rangle - \sum_j \alpha_j \langle \tilde w -\ell , \tilde w_j - w_j \rangle
\\ & 
=
\sum_j ( \tilde \alpha_j - \alpha_j) \langle w -\ell , \tilde w_j - w_j \rangle -  \sum_j \alpha_j \langle \tilde w - w , \tilde w_j - w_j \rangle 
\,.
\end{align*}
After decomposing as 
\begin{equation*}  
\tilde \alpha_j - \alpha_j 
= 
\langle \tilde w  , \tilde w_j  \rangle - \langle w , w_j  \rangle
=
\langle \tilde w - w , w_j  \rangle + \langle \tilde w , \tilde w_j - w_j   \rangle
\,,
\end{equation*}
 we may estimate, using Young's inequality, as
\begin{equation*}  
\| w - \tilde w \|^2 
\leq 
\frac12 \| w - \tilde w \|^2
+
C_k  
\sum_j 
(  \| w - \ell \|^2 + \alpha_j^2 + \| \tilde w_j - w_j \|^2 ) \| \tilde w_j - w_j \|^2 
\,,
\end{equation*}
and reabsorption then yields~\eqref{e.wvswtilde}. 

\smallskip

Next, using Lemma~\ref{l.MSP.heat} together with the decomposition in~\eqref{e.wwtilde.decomp}, we deduce by the triangle inequality that
\begin{align*}  
r \big\| \nabla w - \nabla v \big\|_{L^2(\Phi_{2r})} 
& \leq 
C \| w -v \|
\\ 
& \leq
C \| w- \tilde w \| +  C \sum_j |\tilde \alpha_{j}|  \| \tilde w_{j} - w_{j}\| 
 \leq
C \max_j \| \tilde w_{j} - w_{j}\| .
\end{align*}
On the other hand, by~\eqref{e.supplementary.w.vs.tildew.grad.basis}, 
\begin{align*}  
r \big\| \nabla v - \nabla \tilde w \big\|_{L^2(\Phi_{2r})} 
& \leq 
C r \sum_j |\tilde \alpha_{j}|  \| \nabla \tilde w_{j} - \nabla w_{j} \|_{L^2(\Phi_{2r})}  
\leq
C h \log t 
\,.
\end{align*}
The triangle inequality finishes the proof of~\eqref{e.supplementary.w.vs.tildew}.

\smallskip

Combining~\eqref{e.supplementary.w.vs.tildew} with the definition of~$\phi_{e}^{(r,\rho)}$ it is then easy to show the bounds for the Malliavin derivatives. Indeed, letting~$\tilde G$ and~$\tilde Q$ correspond the matrices in Step 3 for coefficients~$\tilde \a$, we have that, for~$r \geq \frac12 \X_t$,~$|G^{-1} - \tilde G^{-1} | 
\leq C|G - \tilde G | \leq C h\log t$ and~$|Q - \tilde Q |\leq Ch \log t$. implying that~$| Q G^{-1} -  \tilde Q \tilde G^{-1}| \leq C h \log t$ as well. Thus~$|\partial_{\a(B_{t})} G^{-1}|  +| \partial_{\a(B_{t})} Q| \leq C \log t$ for~$r \geq \frac12 \X_t$, and~\eqref{e.malliavin.local} follows by~\eqref{e.X.local} and the product rule. 

\smallskip 

\emph{Step 5.}
We finally conclude the localization estimate~\eqref{e.localizationestimate} by combining~\eqref{e.localmagic} with the bound~\eqref{e.genfluxmaps.heatscale} and~\eqref{e.correctorloc.converse}. The proof is complete.
\end{proof}

We next present an alternative localization result, which improves the estimate for the Malliavin derivative and removes the logarithm present in~\eqref{e.malliavin.local}. The proof of this lemma is somewhat tedious, but the lemma is crucial to arguments in the following two sections
for obtaining estimates with the optimal scaling (as opposed to nearly optimal scaling).

\smallskip

Before the proof, let us recall the following interpolation inequalities, which are both easy consequences of H\"older's inequality. For every~$\eta,\theta,K,M \in (0,\infty)$ with~$\eta \leq \theta$, we have
\begin{equation} 
\label{e.Psi.product}
\Y = \O_{\Psi}^\eta (1), \quad \mathcal{Z}  = \O_{\Psi}^\theta (1) 
\quad \implies \quad  
\Y \mathcal{Z} = \O_{\Psi}^{ \frac{\eta \theta}{\eta + \theta} }(2)
\end{equation}
and
\begin{equation}  
\label{e.Psi.interpolation}
 \Y \leq K, \quad \Y = \O_{\Psi}^\eta (M)   
\quad \implies \quad  
\Y = \O_{\Psi}^\theta ( K^{1-\nicefrac \eta\theta} M^{\nicefrac \eta\theta})  
\,.
\end{equation}

\begin{lemma}
\label{l.correctorloc.new}
There exists a constant~$\delta_0(\dataref) \in (0,1)$ such that, for every~$\delta \in (0,\delta_0]$,~$M>0$ and~$\theta \in (0,\frac d2(1-\beta)]$, there exists another constant~$C(\delta,\dataref)<\infty$ such that the following claim is valid.  Suppose that, for every~$y \in \R^d$ and~$r \geq 1$, we have
\begin{align} \label{e.local.uglyG.cond}
\mathcal{G}_r (y) 
= 
\O_{\Psi}( M r^{-\theta}) + \O_{\Psi}^{1-\beta}( M r^{-\nicefrac d2} ) 
\,.
\end{align}
Then, for every~$x \in \R^d$,~$R \geq 2$ and~$t = C R \log^{\nicefrac12}R$, there exists a~$\F(B_{t}(x))$-measurable matrix~$\b_{R}'(x)$ such that 
\begin{align} \label{e.loc.berror}
 \big| \b_R(x) - \b_{R}'(x)  \big| 
=
\O_{\Psi}^{1-\delta } \Bigl( C R^{-\theta(1+\delta + \delta^2)} \Bigr)  
\end{align}
and the Malliavin derivative of~$\b_{R}'(x)$ satisfies the following bound:
\begin{align}  \label{e.malliavin.loc.improved}
\bigl|  \partial_{\a(B_{t}(x) )} \b_{R}'\bigr| (x)  \leq C .
\end{align}

\end{lemma}

\begin{proof}
The proof is similar to the one of Lemma~\ref{l.correctorloc}, and we will be somewhat brief about the details. The major difference now is that we use the additivity estimate, which will give us a \emph{better localization} and also \emph{average out} the error for the Malliavin derivative. Without loss of generality, we may take~$y = 0$. Denote~$t = H R \log^{\nicefrac 12} R$ for some constant~$H(M,\dataref)<\infty$ to be fixed. We also look for~$\delta_0(M,\dataref) \in (0,\frac12)$ and, for every~$\delta \in (0,\delta_0]$, another parameters~$\gamma(\delta,\dataref),\rho(\delta,\dataref) \in (0,\frac12)$. Let~$r = s = R^{1-\rho}$. Recall the definitions in~\eqref{e.innerproduct}. 

\smallskip

\emph{Step 1. Localized matrix.} In this step, we construct the local approximation. First, by~\eqref{e.additivityestimate} (replacing~$\alpha$ with~$\alpha/2$ there), we deduce that, by choosing~$H$ large enough in the definition~$t = H R \log^{\nicefrac 12} R$, 
\begin{align}
\label{e.additivityestimate.again.prepre}
\big|
\bigl( \b_{R} - (\indc_{B_{t/2}}   \b_{r} ) \ast \Phi_{\sqrt{R^2-r^2}} \bigr)(0)
\big|
\leq
C
\Bigl( \frac{\X \wedge r}{r} \Bigr)^{2\alpha} 
 \mathcal{G}_{r}(0) 
\,.
\end{align}
The left side of the previous inequality is bounded, and thus~\eqref{e.local.uglyG.cond} and the interpolation inequalities~\eqref{e.Psi.product} and~\eqref{e.Psi.interpolation} yield that there exists~$\delta_0,\rho_0 \in (0,\frac12)$ such that, for every~$\delta \in (0,\delta_0]$ and~$\rho \in (0,\rho_0]$, 
\begin{equation}
\label{e.additivityestimate.again.pre}
\big|
\bigl( \b_{R} - (\indc_{B_{t/2}}   \b_{r} ) \ast \Phi_{\sqrt{R^2-r^2}} \bigr)(0)
\big|
=
\O_\Psi^{1-\delta} \bigl( C R^{-\theta(1+\delta + \delta^2)} \bigr)  
\,.
\end{equation}
We cover~$B_{t/2}$ 
with a family of balls~$\{B_R(z_k)\}_k$ with a finite overlap, and let 
\begin{equation} \label{e.localization.eta}
\eta_k : = \indc_{B_R(z_k) \cap B_t} \biggl( \sum_n \indc_{B_R(z_n) \cap B_t}  \biggr)^{-1}
\end{equation}
be a partition of unity subordinate to~$\{B_R(z_k)  \cap B_t\}_k$. For every ball~$B_R(z_k)$, we define the local approximation of~$\b_r$ to have better properties concerning the length scale~$r$, but with a loss of stochastic integrability. Let first~$\X_{R,k}(x)$ be the minimal scale corresponding~$x\in B_R(z_k)$ with respect to~$B_{2R}(z_k)$. It can be chosen so that~\eqref{e.thetae.est} and~\eqref{e.we.vs.thetae} are both valid with~$x$ and~$\X_{R,k}(x)$ instead of origin and~$\X_t$ in them. Define then, for~$\gamma \in (0,\frac12)$ to be fixed, 
\begin{equation*}  
\minscale_{k}(x) := \min\big\{1, \bigl(1  - 2 r^{-1+\gamma} \X_{R,k}(x) \bigr)_+ \big\} .
\end{equation*}
Let~$w_{x,e,k}$ minimize~$E_x( \cdot - \ell_{x,e})$ over~$\A_k(B_{2R}(z_k))$, where~$\ell_{x,e}(z) = e \cdot (z-x)$. Let the matrix~$G_{k}(x)$ have columns~$\{\minscale_{k}(x) ( (\nabla w_{x,e,k} - e_j) \ast \Phi_{r})(x) + e_j \}_j$ and the matrix~$Q_{k}(x)$ have columns~$\{\minscale_{k}(x)  ( (\a \nabla w_{x,e,k} - \ahom e_j) \ast \Phi_{r}) (x) + \ahom e_j\}_j$. We obtain, by estimates similar to~\eqref{e.thetae.est} and~\eqref{e.we.vs.thetae}, that
\begin{equation}  \label{e.Jk.Qk}
\big| G_{k}(x) - \Id \big| \leq \delta r^{- \gamma\alpha} 
\qand 
\big| Q_{k}(x) - \ahom \big| \leq \delta r^{- \gamma\alpha} 
\,.
\end{equation}
The invertibility of~$G_{k}(x)$ is clear by the above display. We then define 
\begin{equation*}  
\b_{r,k}''(x) := Q_{k}(x)  G_{k}^{-1}(x)
\qand
\b_R'(0) = \biggl( \indc_{B_t} \sum_{k} \eta_k(\cdot)  \b_{r,k}''(\cdot) \biggr)  \ast \Phi_{\sqrt{R^2-r^2}}(0)\,.
\end{equation*}
Now~\eqref{e.loc.berror} can be shown by following the computation for~\eqref{e.correctorloc}. Indeed, we may choose~$\rho$ and~$\gamma$ small enough and then use~\eqref{e.localmagic} similarly as in the proof Lemma~\ref{l.correctorloc}, together with~\eqref{e.local.uglyG.cond} and the interpolation inequalities~\eqref{e.Psi.product} and~\eqref{e.Psi.interpolation}.

\smallskip

\emph{Step 2. Malliavin derivative: Basic quantities.} In view of the definition of~$\b_R'(0)$, it is clearly enough to show that 
\begin{equation}  \label{e.malliavin.loc.improved.pre}
\biggl| \partial_{\a(B_t)} \fint_{B_{R}(z_k)} \b_{r,k}''(x) \, dx \biggr| = \biggl| \partial_{\a(B_{2R}(z_k))} \fint_{B_{R}(z_k)} \b_{r,k}''(x) \, dx \biggr|  \leq C  
\,.
\end{equation}
After translation, we can take~$z_k = 0$.  We proceed to show the bound~\eqref{e.malliavin.loc.improved.pre}. For this, fix~$\tilde \a \in \Omega$ such that~$|\tilde \a -\a | \leq h \indc_{B_{2R}}$ with small~$h>0$. First,~\eqref{e.X.local} yields that
\begin{equation*}  
| \minscale_{k}(x) - \tilde \minscale_{k}(x) | \leq Ch \,.
\end{equation*}
By proceeding as in the proof of Lemma~\ref{l.correctorloc}, we also obtain that 
\begin{equation}  \label{e.Jk.Qk.initial.pre}
(\indc_{\{\minscale_{k}(x) >0\}} + \indc_{\{\tilde \minscale_{k}(x) >0\}}) \sup_{e \in B_1} \| \nabla w_{x,e,k}  - \nabla \tilde w_{x,e,k}  \|_{L^2(\Phi_{x,r})} 
\leq C h \log r
\,,
\end{equation}
and, by the Lipschitz estimate, 
\begin{equation} \label{e.Jk.Qk.initial.prepre}
(\indc_{\{\minscale_{k}(x) >0\}} + \indc_{\{\tilde \minscale_{k}(x) >0\}})
\sup_{e \in B_1} \bigl\| | \nabla w_{x,e,k}| + | \nabla \tilde w_{x,e,k}| \bigr\|_{L^2(\Phi_{x,r})}
\leq C
 \,.
\end{equation}
We then get that 
\begin{align} \notag  \label{e.Jk.Qk.initial}
\big| G_{k}(x) - \tilde G_{k}(x) \big| + \big| Q_{k}(x) - \tilde Q_{k}(x) \big| 
&
\leq
C | \minscale_{k}(x) - \tilde \minscale_{k}(x) | \sup_{e \in B_1} \| \nabla \tilde w_{x,e,k} \|_{L^2(\Phi_{x,r})}
\\  
& \quad 
+ \minscale_{k}(x)  \sup_{e \in B_1} \| \nabla w_{x,e,k}  - \nabla \tilde w_{x,e,k}  \|_{L^2(\Phi_{x,r})}
\,.
\end{align}
Using this, we also deduce that
\begin{align} \label{e.bprime.vs.tildebprime}
\big| \b_{r,k}''(x)   - \tilde \b_{r,k}''(x) \big|
\leq 
C \big| G_{k}(x) - \tilde G_{k}(x) \big|  +   C \big| Q_{k}(x)  - \tilde Q_{k}(x) \big|
\,.
\end{align}
To see this, expand as
\begin{align*}  
\b_{r,k}''(x)   - \tilde \b_{r,k}''(x)
& =  
G_{k}^{-1}(x) (G_{k}(x) - \tilde G_{k}(x) ) \tilde G_{k}^{-1}(x) Q_{k}(x) + \tilde G_{k}^{-1} (x)( Q_{k} (x)- \tilde Q_{k}(x))
\\ 
& 
= 
G_{k}^{-1}(x) (G_{k}(x) - \tilde G_{k}(x) ) \tilde G_{k}^{-1}(x) Q_{k}(x)   + \tilde G^{-1}_{k}(x) (Q_{k}(x) - \tilde Q_{k}(x))\,,
\end{align*}
and then use~\eqref{e.Jk.Qk}. By~\eqref{e.Jk.Qk.initial} and~\eqref{e.bprime.vs.tildebprime}, our goal is then to show that 
\begin{align}  \label{e.localcor.goal.pre}
\sum_{j = 1}^d \fint_{B_R} \| \nabla w_{x,e_j,k}  - \nabla \tilde w_{x,e_j,k}  \|_{L^2(\Phi_{x,r})} \, dx
\leq 
C h 
\,.
\end{align}

\smallskip

\emph{Step 3. Malliavin derivative: Orthonormal basis.}
Let next~$\{v_{x,j}\}_j$ and~$\{\tilde v_{x,j}\}_j$ be orthonormal basis of~$\A_k(B_{2R})$ and~$\tilde{\A}_k(B_{2R})$, respectively, with respect to the inner product~$\langle \, ,  \rangle_x$ in such a way that~$\{v_{x,j}\}_{j=1}^d$ and~$\{\tilde v_{x,j}\}_{j=1}^d$  span~$\A_1(B_{2R})$ and~$\tilde \A_1(B_{2R})$, respectively, and~$\deg v_{x,i} \leq  \deg v_{x,j}$ if~$i\leq j$. Here,~$\deg v_{x,j}$ denotes the degree of the polynomial boundary values of~$v_{x,j}$. Set also
\begin{align} \label{e.piofx}
\pi(x) 
:=
\indc_{\{ s \geq 2\X_{R}(x) \}} 
\sum_{j=1}^d \inf_{\tilde v \in \tilde \A_1(B_{2R})} 
\|  \nabla v_{x,j} - \nabla \tilde v  \|_{L^2(\Phi_{x,2r})} .
\end{align}
Notice that~$v_{x,j} \in  \A_1(B_{2R})$ is a linear combination of basis elements~$\{v_j\}_{j=1}^d$ of~$\A_1(B_{2R})$ with respect to the inner product of~$\underline{H}^1(B_{2R})$ and, by the linear growth of~$v_{x,j}$, the corresponding representation matrix~$A_x$ in~$v_{x,j} = \sum_{k=1}^d A_{x,jk} v_k$ is bounded. Therefore, testing  the infimum in~\eqref{e.piofx} with~$\sum_{k=1}^d A_{x,jk} \tilde v_k$ for each~$j$ with~$\tilde v_k \in \tilde{\A}_1(B_{2R}) \cap (v_k + H_0^1(B_{2R}))$, we arrive at 
\begin{align}  \label{e.e.piofx.est}
\fint_{B_{R}} \pi(x) \, dx  
\leq 
C \sum_{j=1}^d \fint_{B_{R}} 
\|  \nabla v_{j} - \nabla \tilde v_j  \|_{L^2(\Phi_{x,2r})}  \, dx    \leq C h
\, .
\end{align}
Notice here that both~$v_{j}$ and~$\tilde v_j$ can be extended as a linear function outside of~$B_{2R}$, and thus~$\nabla v_{j} = \nabla \tilde v_j$ outside of~$B_{2R}$. 

\smallskip

Next, for every~$j \in \{1,\ldots,d \}$, we may choose~$\tilde v_{x,j} \in \tilde{\A}_1(B_{2R})$ in such a way that 
\begin{align*}  
\indc_{\{ s \geq 2\X_R(x) \}}  
\sum_{j=1}^d \Bigl( 
\| v_{x,j} - \tilde v_{x,j} \|_x 
+
\| \nabla v_{x,j} - \nabla \tilde v_{x,j} \|_{L^2(\Phi_{x,2r})}
\Bigr)
\leq 
C \pi(x) 
\,.
\end{align*}
Moreover, as in Step 4 of the proof of Lemma~\ref{l.correctorloc}, we may choose~$\tilde v_{x,j}$ for~$j>d$ so that 
\begin{align}  \label{e.Malliavin.basebound}
\indc_{\{ s \geq 2\X_t(x) \}}  
\bigl( 
\| v_{x,j} - \tilde v_{x,j} \|_x 
+ 
\| \nabla v_{x,j} - \nabla \tilde v_{x,j} \|_{L^2(\Phi_{x,2r})}  
\bigr)
\leq 
C h \log r
\,.
\end{align}
Letting~$\theta_{x,e} \in \A_1(B_{2R})$ be such that~$\| \theta_{x,e} - \ell_{x,e} \|_x \leq C \bigl(\frac{\X_t(x)}{r} \Bigr)^{\! \alpha}|e|$. We obtain by the minimality that, for every~$r$ such that~$r^{1-\gamma} \geq \frac12 \X_t(x)$,
\begin{align}  \label{e.Malliavin.evenmoresilly}
 \| w_{x,e} - \ell_{x,e}  \|_x  \leq  \| \theta_{x,e} - \ell_{x,e}  \|_x \leq C r^{-\gamma \alpha}
\end{align}
and~$\langle w_{x,e} - \ell_{x,e}  , v_{x,j}  \rangle_x  = 0$. Rewrite thus 
\begin{align*}  
\langle w_{x,e}  , v_{x,j}  \rangle_x 
& = 
\langle \ell_{x,e}  , v_{x,j}  \rangle_x
= \langle \theta_{x,e} , v_{x,j}  \rangle_x 
- \langle \theta_{x,e} - \ell_{x,e}  , v_{x,j}  \rangle_x 
\,.
\end{align*}
By orthogonality the first term is zero if~$j>d$ by the fact that~$\theta_{x,e} \in \spn \{v_{x,j}\}_{j=1}^d$. Thus we find the following bound for the coefficients in the basis~$\spn \{v_{x,j}\}_{j}$,  for every~$r$ such that~$r^{1-\gamma} \geq \frac12 \X_t(x)$:
\begin{equation*}  
w_{x,e} = \sum_{j} \gamma_j  v_{x,j} 
\quad \implies \quad 
| \gamma_j | 
\leq 
C \left\{
\begin{aligned}
& 1
, & &
j \in \{1,\ldots,d\},
\\ 
& r^{-\frac12 \gamma \alpha}, & & j > d . 
\end{aligned}
\right.
\end{equation*}
Consequently, we have, for every~$r^{1-\gamma} \geq \frac12 \X_t(x)$, that 
\begin{equation}  \label{e.Malliavin.silly}
\sum_{j} | \gamma_j | \bigl( \| v_{x,j} - \tilde v_{x,j} \|_x  +  \| \nabla (  v_{x,j} - \tilde v_{x,j})   \|_{L^2(\Phi_{x,2r})}   \bigr)
\leq 
C \bigl( \pi(x) + h \bigr)
\,.
\end{equation}
We can then repeat the computation in Step 4 of the proof of  Lemma~\ref{l.correctorloc}, but now using the improved bound~\eqref{e.Malliavin.silly} above, and deduce that 
\begin{equation*}  
\big\| \nabla w_x - \nabla \tilde w_x \big\|_{L^2(\Phi_{x,2r})} 
\leq 
C  \bigl( \pi(x) + h\bigr).
\end{equation*}
Therefore,~\eqref{e.e.piofx.est} yields~\eqref{e.localcor.goal.pre}, completing the proof. 
\end{proof}

\subsection{Closing the renormalization estimates}
\label{s.fluctboot}

In this section, we complete the proof of Theorem~\ref{t.optimal} by tying the estimates from the previous two sections together and applying the~$\CFS$ mixing assumptions to obtain fluctuation bounds. 
We split the main estimates into two cases: in Proposition~\ref{p.psipsi.coarse} for~$\beta \in (0,1)$ and in Proposition~\ref{p.psipsi.coarse.optimal} for~$\beta=0$. Taken together, these two propositions imply Theorem~\ref{t.optimal}. 

\begin{proposition}
\label{p.psipsi.coarse}
Let~$\beta \in (0,1)$,~$M, N \in [1,\infty)$ and~$p \in (\frac{2}{1-\beta},\infty)$.  Let~$\Psi:[1,\infty) \to [0,\infty)$ be an increasing function satisfying
\begin{equation}  \label{e.Psi.pgrowth}
t^p \leq N \frac{\Psi(t s)}{\Psi(s)}  \quad \mbox{for every~$s, t \in [1,\infty)$}\,. 
\end{equation}
Suppose that~$\P$ is a~$\Zd$--stationary measure on~$(\Omega,\F)$ satisfying~$\CFS(\beta,\Psi,\Psi(M\cdot),0)$. Then, for every~$\delta \in (0,1)$,  there exist constants~$C(\delta,p,M,N,\dataref) < \infty$ such that, for every~$R \geq 1$, 
\begin{align}
\label{e.psipsi.coarse}
\b_R(x) - \ahom 
=
\O_{\Psi} \bigl( C R^{-\frac d2(1-\beta)}  \bigr) 
+ 
\O_{\Psi}^{1-\delta} \bigl( C R^{-\frac d2(1-\beta)(1+\delta)} \bigr)
\,.
\end{align}
\end{proposition}
\begin{proof}
Fix~$\delta \in (0,\frac12)$,~$M,N \in (0,\infty)$ and~$p \in (\frac{2}{1-\beta},\infty)$. Throughout the proof, we extend the list of parameters in~\eqref{e.data.def} to
\begin{equation*}  
\data = (\delta,M,N,p,d,\lambda,\Lambda,\beta,\CFS,\Psi)
\,.
\end{equation*} 

\emph{Step 1. Stochastic integrability.} 
Let us first collect a few elementary results used many times in what follows. First, by~\eqref{e.Psi.pgrowth}, for every continuous (or piecewise constant) random field~$x \mapsto \Y(x)$, a straightforward real analysis argument (see Lemma~\ref{e.supp.you.up}) yields
\begin{equation}  \label{e.weakint.ave.applied}
\mathcal{Y}(x) = \O_{\Psi}^{\frac d2(1-\beta)}(C) \quad
\mbox{$\forall x \in \R^d$}
\quad \implies \quad  
\sup_{s \geq r} \Bigl( \frac{r}{s}\Bigr)^{\nicefrac12}  \left\| \Y \right\|_{\underline{L}^d \left( B_{s}(x) \right)}^{\nicefrac d2}
=
\O_{\Psi}^{1-\beta}(C) \,.
\end{equation}
By~\eqref{e.Psi.pgrowth}, we obtain 
\begin{equation}  \label{e.weakint.ave.expectation}
\mathcal{Y}(x) = \O_{\Psi}^{\frac d2(1-\beta)}(C) 
\quad \implies \quad  
\E[ \Y^2(x)]  \leq C \,.
\end{equation}
In particular, the above implications are valid for~$\Y = \X$. 

\smallskip

\emph{Step 2. Induction assumption.} Select parameters 
\begin{equation}  \label{e.varthetaupper}
\gamma, \rho \in (0,\nicefrac12), \qquad \theta \in \Bigl[ \alpha , \frac{d}{2}(1-\beta) \Big) . 
\end{equation}
Assume inductively that there exists a constant~$C(\theta,\data)<\infty$ such that, for every~$x \in \R^d$ and~$r \in [1,\infty)$, we have that 
\begin{equation} \label{e.psipsi.induction}
\b_r(x) - \ahom  = \O_{\Psi} ( C r^{-\theta} ). 
\end{equation} 
The initial assumption is provided by~\eqref{e.coarsen.smallness} with~$\theta = \alpha$ (in fact, with much better stochastic integrability). Recalling the definition of~$\mathcal{G}$ in~\eqref{e.uglyG}, this, together with~\eqref{e.weakint.ave.applied}, implies that, for every~$x \in \R^d$ and~$r \in [1,\infty)$, we have
\begin{equation} \label{e.psipsi.uglyG}
\mathcal{G}_r(x) 
= 
\O_{\Psi} ( C r^{-\theta} )
+
\O_{\Psi}^{1-\beta   } ( C r^{-\frac{d}{2}} )
\,.
\end{equation}
We now move on to bound different terms in~\eqref{e.ahomr.splitting} under the induction assumption~\eqref{e.psipsi.induction}.  Throughout, we, fix
\begin{equation*}  
t = r^{1+\rho} 
\qand
r = R^{1-\gamma}
\,.
\end{equation*}

\smallskip

\emph{Step 3. Bound for the localization error}. 
We show that there exist two small constants~$\alpha(\beta,d,\lambda,\Lambda),\delta_0(\data) \in (0,1)$ such that, for every~$\ep,\rho \in [0,\frac12)$,~$\delta \in [0,\delta_0]$ and
\begin{equation} \label{e.psipsi.gamma1}
\gamma 
\in 
\biggl( 0, \frac14 \frac{\ep \alpha + \rho \theta}{\theta + \rho \theta + \ep\alpha} \biggl]\,, 
\end{equation}
there exists another constant~$C(\ep,\rho,\data) < \infty$ such that
\begin{multline}  \label{e.psipsi.local.buckle}
\Bigl( (\b_r - \E [ \b_r ]) - (\b'_{r,r^{1+\rho} } - \E [ \b'_{r,r^{1+\rho} } ] )  \Bigr) \ast \Phi_{\sqrt{R^2-r^2}} 
\\ 
= 
\O_{\Psi} (C R^{- \theta - \frac12 \theta \rho  - \frac12 \ep \alpha }) +  \O_{\Psi}^{1-\delta}\biggl(C R^{- \frac{d}{2}(1-\beta)  \frac{(1-\ep)(1 - \gamma)}{1-\delta}}  \biggr) 
\,.
\end{multline}
The result follows from from the induction assumption~\eqref{e.psipsi.induction} by showing that
\begin{align} \label{e.psipsi.local}
\Big| 
\b_r(x)  - \b'_{r,r^{1+\rho}}(x)
\Big|
= 
\O_{\Psi} (C R^{- \theta - \frac12 \theta \rho  - \frac12  \ep \alpha}) +  \O_{\Psi}^{1-\delta}\biggl(C R^{- \frac{d}{2}(1-\beta)  \frac{(1-\ep)(1 - \gamma)}{1-\delta}} \biggr) 
\,.
\end{align}
On the one hand, we use the localization estimate~\eqref{e.localizationestimate} and~\eqref{e.psipsi.uglyG} to obtain, for every~$\ep \in [0,\frac12)$ and~$\nu = \frac{1}{4} \frac{\ep \alpha+\rho \theta}{\theta(1+\rho)} \wedge \frac \rho2$, 
\begin{align} \notag  
\Big| 
\b_r(x)  - \b'_{r,r^{1+\rho}}(x)
\Big| 
\indc_{\{ \X(x) \leq r^{1- \ep}   \} }
&
\leq
C r^{-\ep \alpha} \frac{r}{r^{(1+\rho)(1-\nu)}}\mathcal{G}_{r^{(1+\rho)(1-\nu)}}(x)
\\ 
\notag &
= 
\O_{\Psi} \Bigl(C_\nu r^{-\ep \alpha} r^{- \theta(1+\rho)(1-\nu)}\Bigr)
+ 
\O_{\Psi}^{1-\beta}\Bigl( C_\nu r^{-\ep \alpha -\frac d2}\Bigr)
\\ 
\notag &
= 
 \O_{\Psi} (C R^{- \theta - \frac12 \theta \rho  - \frac12 \ep \alpha }) 
 + 
\O_{\Psi}^{1-\beta}\Bigl( C R^{\frac d2 (1-\gamma)}\Bigr)
\end{align}
On the other hand,
\begin{equation*}  
\Big| 
\b_r(x)  - \b'_{r,r^{1+\rho}}(x)
\Big| 
\indc_{\{ \X(x) > r^{1- \ep}   \} }
= 
\O_{\Psi}^{1-\delta}  \Bigl( C r^{-\frac d2(1-\beta)\frac{1-\ep}{1-\delta}}\Bigr) 
\leq 
\O_{\Psi}^{1-\delta}  \Bigl( C R^{- \frac{d}{2}(1-\beta)   \frac{(1-\ep)(1 - \gamma)}{1-\delta}}  \Bigr) 
\,.
\end{equation*}
By interpolation, using~\eqref{e.Psi.product} and~\eqref{e.Psi.interpolation}, we arrive at~\eqref{e.psipsi.local}. 

\smallskip 

\emph{Step 4. Bound for the fluctuations}. We show that, for every~$\rho \in (0,\frac12)$ and~$\gamma$ such that 
\begin{equation} \label{e.psipsi.gamma2}
\gamma \geq \frac{2d}{d-2\theta} \rho \,,
\end{equation}
there exists constants~$C(\rho,\gamma,\data) < \infty$ such that for
\begin{equation}  \label{e.br.flucs}
\bigl( \b'_{r,r^{1+\rho}} - \E [ \b'_{r,r^{1+\rho}} ] \bigr) \ast \Phi_{\sqrt{R^2-r^2}} 
= 
 \O_{\Psi} \Bigl(  C R^{-\theta - \frac{\gamma}{4}(d- 2\theta)}  \log R  + C R^{-\frac d2(1-\beta) } \log R  \Bigr) 
\,.
\end{equation}
To show~\eqref{e.br.flucs}, the induction assumption~\eqref{e.psipsi.induction} and the localization error~\eqref{e.psipsi.local} (applied with~$\delta = \ep = 0$) give us that 
\begin{equation*}  
 \b'_{r,r^{1+\rho}} - \E [ \b'_{r,r^{1+\rho}} ]  
 =  
 \O_{\Psi} (C r^{- \theta})
 \,.
\end{equation*}
Using then~$\CFS(\beta,\Psi,\Psi(M \cdot),0)$ condition and the logarithmic bound for the Malliavin derivative given by~\eqref{e.malliavin.local}, we deduce, for~$r = R^{1-\gamma}$ with~$\ep,\gamma\in (0,1)$ small that
\begin{align*}  
\bigl( \b'_{r,r^{1+\rho}} - \E [ \b'_{r,r^{1+\rho}} ] \bigr) \ast \Phi_{\sqrt{R^2-r^2}}
 & = 
 \O_{\Psi} \Bigl( C r^{-\theta} \Bigl( \frac{R}{r^{1+\rho}} \Bigr)^{-\nicefrac d2} \log r + C  R^{-\frac d2(1-\beta) } \log r  \Bigr)
  \\ & 
 =
 \O_{\Psi} \Bigl(  C R^{-\theta + \gamma \theta - \frac{d}{2} (\gamma - \rho(1-\gamma) ) }  \log R  + C R^{-\frac d2(1-\beta) } \log R  \Bigr) 
 \,.
\end{align*}
Thus~\eqref{e.br.flucs} follows under the condition~\eqref{e.psipsi.gamma2}.

\smallskip 

\emph{Step 5. Bound for the additivity error}. We show that there exists~$\alpha(\beta,d,\lambda,\Lambda),\delta_0(\data) \in (0,1)$ such that, for every~$\ep,\rho \in [0,\frac12)$,~$\delta \in (0,\delta_0]$ and
\begin{equation} \label{e.psipsi.gamma3}
\gamma 
\in 
\biggl( 0,  \frac{\ep \alpha}{4\theta} \wedge \frac12 \biggl]\,, 
\end{equation}
there exists another constant~$C(\ep,\rho,\delta,\gamma,\data) < \infty$ such that
\begin{equation} \label{e.psipsi.add}
\big|
\bigl( \b_R - \b_r \ast \Phi_{\sqrt{R^2-r^2}} \bigr) (x) 
\big|
 = 
\O_\Psi(C R^{-\theta - \frac12 \ep \alpha} )  + \O_{\Psi}^{1-\delta}\biggl(C R^{- \frac{d}{2}(1-\beta)   \frac{(1-\ep)(1 - \gamma)}{1-\delta}}  \biggr) 
.
\end{equation}
To prove~\eqref{e.psipsi.add}, we notice that~\eqref{e.additivityestimate} and~\eqref{e.psipsi.uglyG} give us, for every~$\ep, \delta \in [0,\frac12)$, 
\begin{align*}  
\big|
\bigl( \b_R - \b_r \ast \Phi_{\sqrt{R^2-r^2}} \bigr) (x) 
\big|
& 
\leq
C \indc_{\{ \X(x) \geq r^{1-\ep} \} } +  C r^{-\ep \alpha }\mathcal{G}_r(x)  
\\ 
&
= 
\O_{\Psi}^{1-\delta}\biggl(C r^{- \frac{d}{2}(1-\beta)  \frac{1-\ep}{1-\delta}} \biggr) 
+ 
\O_\Psi(C r^{-\theta - \ep \alpha} ) 
+ 
\O_\Psi^{1-\beta}(C r^{-\nicefrac d2 - \ep \alpha} ) 
\,.
\end{align*}
By the interpolation inequalities~\eqref{e.Psi.product} and~\eqref{e.Psi.interpolation},  we obtain~\eqref{e.psipsi.add}.

\smallskip

\emph{Step 6. Bound for the systematic error}. We show, under the induction assumption~\eqref{e.psipsi.induction} and~\eqref{e.Psi.pgrowth},  that the systematic error in~\eqref{e.ahomr.splitting} satisfies
\begin{equation}  \label{e.subopt.syserr}
\Big| \bigl( \E [ \b_r ]- \ahom \bigr) \ast \Phi_{\sqrt{R^2-r^2}} \Big| \leq C r^{-2\theta}. 
\end{equation}
To see this, rewrite as
\begin{equation*}  
\bigl( \E [ \b_r ]- \ahom \bigr) \ast \Phi_{\sqrt{R^2-r^2}} 
=
\bigl( \E [ \b_r ] - \langle \b_r \rangle\bigr) \ast \Phi_{\sqrt{R^2-r^2}} 
+ 
\langle \b_r \rangle - \ahom 
\,.
\end{equation*}
Since~$x \mapsto \E [ \b_r(x)]$ is~$\Z^d$-periodic, we have that 
\begin{equation*}  
\Big| \bigl( \E [ \b_r ]- \langle\b_r\rangle \bigr) \ast \Phi_{\sqrt{R^2-r^2}} \Big| \leq C \exp(-c R) .
\end{equation*}
To bound the second term, notice first that 
\begin{equation*}  
\langle \b_{2r} \rangle - \langle \b_r \rangle  
=
\int_{\cu_0} \E\bigl[  \b_{2r}   - \b_r  \ast \Phi_{\sqrt{3 r^2}} \bigr]
+
\int_{\cu_0} \bigl( \E[\b_r]  - \langle \b_r\rangle \bigr) \ast \Phi_{\sqrt{3 r^2}} 
\end{equation*}
The last term in the above display is again bounded by~$C \exp(-cr)$. For the first term, we use the additivity estimate~\eqref{e.additivityestimate} together with~\eqref{e.psipsi.uglyG} and~\eqref{e.weakint.ave.expectation} to obtain that 
\begin{equation*}  
\Big| \E\bigl[  \b_{2r}   - \b_r  \ast \Phi_{\sqrt{3 r^2}} \bigr]\Big| \leq C r^{-2\theta} .
\end{equation*}
Combining the above estimates and summing over scales yields~\eqref{e.subopt.syserr}. 

\smallskip 

\emph{Step 7. Induction step}. 
We show that there exists~$\theta(\data) \in (0,1)$ such that if the induction assumption~\eqref{e.psipsi.induction} is valid for some~$\theta \in \bigl[ \alpha, \frac{d}{2}(1-\beta) - \kappa \bigr]$ for some~$\kappa > 0$, then 
\begin{equation} \label{e.psipsi.buckling}
\b_r - \ahom 
=  
\O_{\Psi} 
\bigl(
C r^{- \theta - \theta \kappa} 
\bigr)
\,.
\end{equation}
This is easy to verify by~\eqref{e.ahomr.splitting} using~\eqref{e.br.flucs},~\eqref{e.psipsi.local},~\eqref{e.psipsi.add} and~\eqref{e.subopt.syserr} with~$\delta = 0$ and suitably chosen small~$\gamma$ and~$\rho$.  We can thus iterate this finitely many times and obtain that, for every~$\theta \in (0,\frac d2(1-\beta))$, there exists a constant~$C(\theta,\data)<\infty$ such that 
\begin{equation} \label{e.psipsi.nearlythere}
\b_r(x) - \ahom  = \O_{\Psi} ( C r^{-\theta} ) \,.
\end{equation} 

\smallskip 

\emph{Step 8. Buckling the result}. 
Finally, to obtain~\eqref{e.psipsi.coarse}, we fix~$\delta \in (0,\delta_0]$ and take~$\theta = \frac d2(1-\beta)-\ep$ with~$\ep$ small. Using now Lemma~\ref{l.correctorloc.new} we find, for every~$x \in \R^d$ and~$r \geq 2$, a matrix~$\b'_{r}(x)$, which is~$\F(B_{C r \log^{\nicefrac12} r})(x))$-measurable, has bounded Malliavin derivative and that, for every~$R \geq 2r$, we have that 
\begin{equation}  \label{e.add.local}
\big|
\bigl( \b_R - \b_r \ast \Phi_{\sqrt{R^2-r^2}} \bigr) (x) 
\big|
+
\big| 
\b_r(x)  - \b'_{r}(x)
\big|   
= 
\O_\Psi^{1-\delta}(C r^{-\frac d2 (1-\beta)(1+\delta+\delta^2)})
\,,
\end{equation}
The systematic error is bounded by~$Cr^{-\frac d2 (1-\beta)}$. These computations are verbatim repetitions of Steps 3,5,6. By Step 4,  for small enough~$\rho$, we may take~$\theta$ close enough to~$\frac d2 (1-\beta)$, and obtain that
\begin{equation}  \label{e.br.flucs.again}
\bigl( \b'_{r} - \E [ \b'_{r} ] \bigr) \ast \Phi_{\sqrt{R^2-r^2}} 
= 
 \O_{\Psi} \Bigl( C R^{-\frac d2(1-\beta) } \Bigr) 
\,.
\end{equation}
Pick then small enough~$\ep,\rho,\gamma \in (0,\frac12)$ by means of~$\delta$ and use the~$\delta^2$ term to absorb them. This is indeed possible since we may take~$\theta$ as close to~$\frac d2 (1-\beta)$ as we please by means of~$\delta$. The proof is complete. 
\end{proof}

We next prove the result for the case~$\beta = 0$. The proof actually provides a very general tool for treating additive and localizable bounded fields. We will indeed apply the very same proof when we will treat the non-homogeneous problems later on. Below, we will denote
\begin{equation} 
\label{e.Psi.sigma.def}
\Psi_\sigma(t) := \Psi(t^{1-\sigma})\,, \quad \sigma \in [0,1)\,.
\end{equation}
   
\begin{proposition} 
\label{p.psipsi.coarse.optimal}
Let~$\delta \in (0,\frac12)$ and~$\gamma \in [0,\infty]$,~$M,N \in [1,\infty)$ and~$p \in (2,\infty)$. 
Let~$\Psi:[1,\infty) \to [0,\infty)$ be an increasing function satisfying
\begin{equation}  \label{e.Psi.pgrowth.again}
t^p \leq N \frac{\Psi(t s)}{\Psi(s)}  \quad \mbox{for every~$s, t \in [1,\infty)$}. 
\end{equation} 
We suppose that~$\P$ satisfies, for~$\sigma \in \{0,\delta\}$,~$\CFS(0,\Psi_\sigma,\Psi_\sigma(M\cdot),\gamma)$ condition,  according to Definition~\ref{d.CFS.strong}. Then there exists a constant~$C(\delta,\gamma,M,N,\dataref) < \infty$ such that, for every~$x\in \R^d$ and~$r \in [1,\infty)$,
\begin{equation} 
\label{e.betazero}
\b_{r}(x)  - \ahom  
=  
\left\{
\begin{aligned}
& \O_{\Psi}^{1-\delta} \bigl(C r^{-\nicefrac d2}\bigr) & & \mbox{ if } \gamma>0
\, ,  
\\
& 
 \O_{\Psi}^{1-\delta} \bigl(C r^{-\nicefrac d2} \log r \bigr) & & \mbox{ if } \gamma = 0
\, .
\end{aligned}
\right.
\end{equation}
\end{proposition}

\begin{proof}
Fix first~$\delta \in (0,\frac12)$,~$\gamma \in (0,\infty]$,~$M,N \in [1,\infty)$ and~$p \in (2,\infty)$.  

\smallskip 

\emph{Step 1. Additivity and localization defects.} 
Define, for every~$x \in \R^d$ and~$r \geq 1$, 
\begin{align*}  
\bhom_{r}(x) = \b_{r}(x) -  \E[\b_{r}(x)] \qand  \bhom_{r}'(x) = \b_{r}'(x) -  \E[\b_{r}'(x)] 
\,.
\end{align*}
By~\eqref{e.add.local} we then obtain, together with the systematic error~$|\E[\b_{r}(x)] - \ahom| \leq C r^{-\nicefrac d2}$, that   for every~$x \in \R^d$ and~$r \geq 2$,
\begin{equation} \label{e.betazero.add}
 \bhom_{r}(x) -  (\bhom_{r/\sqrt 2}  \ast \Phi_{r/\sqrt 2})(x)   = \O_{\Psi_\delta}  \bigl(C r^{-\frac{d}{2}(1+\delta + \delta^2)} \bigr) 
\end{equation}
and there is a matrix~$\b_r'(x)$, which is~$\mathcal{F}(B_{t}(x))$-measurable with~$t = H(\dataref) r \log^{\nicefrac 12} r$, has bounded Malliavin derivative and satisfies
\begin{equation}  \label{e.betazero.local}
(\bhom_{r} - \bhom_{r}')(x) = \O_{\Psi_\delta}  \bigl(C r^{-\frac{d}{2}(1+\delta+ \delta^2)} \bigr) 
\,.
\end{equation}
It turns out that, together with the~$\CFS(0,\Psi_\delta,\Psi_\delta(M\cdot),\gamma)$ condition, these two estimates are enough to prove the result. Indeed, we may apply a simple consequence of semigroup property for the heat kernels. For given~$j,k \in \N$ with~$j \leq k$ and~$r \geq 2$, set
\begin{equation*}  
r_j = 2^{\nicefrac {(j{-}k)}2} r , \quad t_j = H r_j \log^{\nicefrac12} r_j
\qand
s_j = \biggl( \sum_{n=j}^{k-1} r_n^2 \biggr)^{\!\nicefrac12}, \; s_k = 0
\,.
\end{equation*}
Then, by the semigroup property 
\begin{align}  \label{e.scalinglimit.id}
\notag
\bhom_{r}' (x)  
& = 
\bhom_{r_{k}}'  \ast \Phi_{s_{k}}(x) 
\\ \notag
& = 
\bhom_{r_{k-1}}'  \ast \Phi_{s_{k-1}}(x) 
+
\bigl(  \bhom_{r_{k}}' - \bhom_{r_{k-1}}' \ast \Phi_{r_{k-1}} \bigr) \ast \Phi_{s_{k}}(x) 
\\ \notag 
& = 
\bhom_{r_{k-2}}'  \ast \Phi_{s_{k-2} }(x) 
+  \sum_{j=k-1}^{k} \bigl(  \bhom_{r_j}' - \bhom_{r_{j-1}}' \ast \Phi_{r_{j-1}} \bigr) \ast \Phi_{s_j}(x)
\\ \notag
& =  \ldots 
\\ 
& = 
\bhom_{r_{0}}'  \ast \Phi_{s_{0}}(x)  
+ \sum_{j=0}^{k} \bigl(  \bhom_{r_j}' - \bhom_{r_{j-1}}' \ast \Phi_{r_{j-1}} \bigr) \ast \Phi_{s_j}(x)
\,.
\end{align}
The first sum is a convolution of local quantities (take~$k$ so that~$r_0 \in [1,2)$), and the summands in the sum are contributions of different length scales but with a subcritical size. We will analyze them in the next step.

\smallskip

\emph{Step 2.} We show that, for every~$r\geq 1$,~$s \in [r,\infty)$ and~$x \in \R^d$, 
\begin{equation}  \label{e.Fiter}
\bigl(  \bhom_{\sqrt{2} r}' - \bhom_{r}' \ast \Phi_{r} \bigr) \ast \Phi_{s}(x) 
= 
\O_\Psi^{1-\delta}(C s^{-\nicefrac d2} r^{-\nicefrac{ \delta d}2 \wedge \gamma}) \,.
\end{equation}
First, by~\eqref{e.betazero.add} and~\eqref{e.betazero.local}, we have that 
\begin{equation*}  
\bhom_{\sqrt{2} r}' - \bhom_{r}' \ast \Phi_{r} = \O_{\Psi}^{1-\delta}  \bigl(C r^{-\frac{d}{2}(1+\delta)} \bigr) \,.
\end{equation*}
Let~$t = H r \log^{\nicefrac 12} r$ and rewrite, for every~$y\in \R^d$, as 
\begin{equation*}  
(\bhom_{r}' \ast \Phi_{r})(y)  
- 
\bigl( \bhom_{r}' \indc_{B_{t}(y) } \bigr) \ast \Phi_{r}(y) 
= \sum_{k=0}^\infty \Bigl( \bhom_{r}' \bigl( \indc_{B_{2^{k+1} t}(y)}  - \indc_{B_{2^{k} t}(y)} \bigr)\Bigr)   \ast  \Phi_{r} (y)
\,.
\end{equation*}
The sum on the right is very small and easy to bound. Indeed, by choosing~$H(d)<\infty$ large enough in the definition of~$t$, we deduce that
\begin{equation*}  
\int_{\R^d} \Big| \bhom_{r}'(x) \bigl( \indc_{B_{2^{k+1} t}(y)}(x)  - \indc_{B_{2^{k} t}(y)}(x) \bigr) \Big|  \Phi_{r}(x-y) \, dx
\leq C r^{- 2d (1+4^k)}
 \,.
\end{equation*}
Since~$\bigl( \bhom_{r}' ( \indc_{B_{2^{k+1} t}(y)}  - \indc_{B_{2^{k} t}(y)} ) \bigr) \ast \Phi_{r}(y)$ is~$\F(y+ 2^{k+2} t \cu)$-measurable, we obtain, appealing to the~$\CFS(0,\Psi_\delta,\Psi_\delta(M\cdot),\gamma)$ condition, that 
\begin{align}  \label{e.betazero.tails}
\notag
\sum_{k=0}^\infty \Bigl( \bhom_{r}' \bigl( \indc_{B_{2^{k+1} t}(\cdot)}  - \indc_{B_{2^{k} t}(\cdot)} \bigr)  \ast  \Phi_{r} \Bigr) \ast \Phi_s(x)
& 
=
\sum_{k=0}^\infty  \O_{\Psi_\delta} \biggl( C \Bigr( \frac{s}{2^k t}\Bigr)^{-\nicefrac{d}{2}}  r^{- 2d (1+4^k)} \biggr)
\\ &
= 
 \O_{\Psi_\delta} \bigl( C s^{-\nicefrac{d}{2}} r^{-d} \bigr)
\,.
\end{align}
Furthermore,~$\bhom_{\sqrt{2} r}'(y)  -  \bigl( \bhom_{r}' \indc_{B_{t}(y) } \bigr) \ast \Phi_{r}(y)$ is~$\F(B_{2t}(y))$-measurable and, by the above estimates, it satisfies
\begin{equation*}  
\bhom_{\sqrt{2} r}'(y)  -  \bigl( \bhom_{r}' \indc_{B_{t}(y) } \bigr) \ast \Phi_{r}(y) 
=
\O_{\Psi}^{1-\delta}  \bigl(C r^{-\frac{d}{2}(1+\delta+\delta^2)} \bigr)\,.
\end{equation*}
Therefore,~$\CFS(0,\Psi_\delta,\Psi_\delta(M\cdot),\gamma)$ condition implies that 
\begin{equation*}  
\Bigl( \bhom_{\sqrt{2} r}'(x) -   \bigl( \bhom_{r}' \indc_{B_{t}(x) }  \bigr) \ast \Phi_{r} \Bigr) \ast \Phi_s 
= 
 \O_{\Psi_\delta} \Bigl(C s^{-\frac{d}{2}} ( r^{-\nicefrac{\delta d}{2}}+ r^{-\gamma} ) \Bigr) 
 \,.
\end{equation*}
Putting this together with~\eqref{e.betazero.tails} yields~\eqref{e.Fiter}. 

\smallskip

\emph{Step 3. Conclusion when~$\gamma>0$.} We apply~\eqref{e.Fiter}, together with~\eqref{e.scalinglimit.id} and~$\CFS$ condition for the first term in~\eqref{e.scalinglimit.id}, to obtain
\begin{equation*}  
\bhom_{r}' (x)  
=
\bhom_{r_{0}}'  \ast \Phi_{s_{0}}(x)   + \sum_{j=0}^k 
\O_\Psi^{1-\delta}\bigl(C s^{-\nicefrac d2} 2^{- j \frac{ \delta d}{4} \wedge \frac{\gamma}{2} } \bigr) 
=
\O_\Psi^{1-\delta}\bigl(C s^{-\nicefrac d2} \bigr) 
\,.
\end{equation*}

\smallskip

\emph{Step 4. Conclusion when~$\gamma=0$.} We can repeat the Step 2 and obtain, instead of~\eqref{e.Fiter}, that 
\begin{equation}  \label{e.Fiter.gammazero}
\bigl(  \bhom_{\sqrt{2} r}' - \bhom_{r}' \ast \Phi_{r} \bigr) \ast \Phi_{s}(x) 
= 
\O_\Psi^{1-\delta}(C s^{-\nicefrac d2}) 
\,.
\end{equation}
Summation then produces~$\log r$. 
\end{proof}

It is also possible to obtain optimal results under the much weaker assumption~$\CFS(\beta,\Psi)$ for a certain range of~$\beta$, but with a significant loss of stochastic integrability. We omit the proof of the following proposition, which is similar to that of Proposition~\ref{p.psipsi.coarse}.
 
\begin{proposition} 
\label{p.psi.coarse}
There exist constants~$C(\dataref) < \infty$ and~$c(d,\beta) \in (0,1)$ such that if~$\P$ satisfies~$\CFS(\beta,\Psi)$, with~$\Psi$ satisfying~\eqref{e.Psi.pgrowth} with~$p\geq C(d,\beta)$ and 
\begin{equation}  \label{e.Psibeta.beta.ass}
\beta \geq \frac{1}{d} \Bigl( d + 1 - \sqrt{d  + 1}  \Bigr), 
\end{equation}
then, for every~$r \geq 2$, 
\begin{equation}  \label{e.Psibeta.coarse}
\b_r(x) - \ahom 
= 
\O_{\Psi}^c \bigl( C  r^{-\frac d2(1-\beta)}   \bigr) . 
\end{equation}
\end{proposition}

We conjecture that the previous proposition can be extended to every~$\beta \in (\frac12 ,1 )$.  

\begin{conjecture}
Suppose that~$\P$ satisfies~$\CFS(\beta,\Psi)$ with~$\beta \in (\frac12 ,1 )$. Then there exist constants~$C(\dataref) < \infty$ and~$c(d,\beta) \in (0,1)$ such that 
\begin{equation*}
\b_r(x) - \ahom = \O_{\Psi}^c \bigl( C r^{-\frac d2(1-\beta)}  \bigr)\,.
\end{equation*}
\end{conjecture}

\subsection{Optimal quantitative estimates of the first-order correctors}
\label{ss.correctors.optimal}

In this section we complete the proof of Theorem~\ref{t.corr.optimal} by showing that quantitative estimates on the convergence of the coarse-grained coefficients~$\b_r$ to~$\ahom$ imply quantitative bounds on the first-order correctors. The argument for obtaining the latter from the former is essentially deterministic and was observed already in Lemma~\ref{l.scalecomparison} and used in the renormalization argument. 

\smallskip

In order to prove estimates in a weak Sobolev space, we observe that it suffices to obtain estimates on the \emph{convolution of against heat kernels}. That is, we just need to obtain estimates on the stationary random field~$\nabla \phi_e \ast \Phi_R$. The reader should agree that this is at least possible in principle since any stationary random field~$f$ with zero mean can be written in terms of its convolutions with the heat kernel using the identity
\begin{equation*}
f(x) = - \int_0^\infty (\nabla f \ast \nabla \Phi(t,\cdot) ) (x) \,dt\,.
\end{equation*}
For instance, one may use a similar identity and the semigroup property to obtain that
\begin{align}
\label{e.multiscale.HK}
\| \nabla \phi_e \|_{H^{-1}(B_R)}^2 
& 
\leq 
C
\bigl\| \phi_e - \phi_e\ast \Phi_R  \bigr\|_{L^2(B_R)}^2 
\notag \\ & 
= 
\int_{B_R} \biggl| 
\int_0^{R^2} (\nabla \phi_e \ast \nabla \Phi(t,\cdot) ) (x) \,dt
\biggr|^2\,dx
\notag \\ & 
\leq 
C\int_0^{R^2} 
t^{-\nicefrac12}
\int_{B_R} 
\bigl| (\nabla \phi_e \ast \Phi(\tfrac t2,\cdot) ) (x) \bigr|^2
\,dx \,dt\,.
\end{align}
We have encountered this already, and the above inequality is nearly identical to~\eqref{e.MSP.heat.step1} and is also similar to Proposition~\ref{p.MSP}. 

\smallskip 

Negative Sobolev norms are \emph{equivalent} to certain expressions involving only convolutions with the heat kernel in the whole space. We may also localize these results and obtain, for each~$s\in (0,\infty)$, that there exists~$C(s,d) <\infty$ such that
\begin{equation}
\label{e.negsob.sint}
\big\| f \bigr\|_{H^{-s}(B_1)} 
\leq C
\biggl( 
\int_0^1 t^{s}  \int_{\Rd}  \bigl| f\ast \Phi(t,\cdot) (x) \bigr|^2 \exp(-|x|) \, dx 
\,\frac{dt}{t} 
\biggr)^{\!\nicefrac12}\,.
\end{equation}
The proof of this, 
and generalizations for~$W^{-s,p}$-spaces, appear in~\cite[Appendix D]{AKMBook}.

\smallskip

With this in mind, we return to Lemma~\ref{l.scalecomparison} and, for fixed~$e\in\Rd$, we take~$\psi = \ell_e + \phi_e$ and~$R=2r$ in~\eqref{e.scalecomparison} to obtain, for every~$r\geq 2$ and~$x\in\Rd$,
\begin{equation}
\label{e.conv.rto2r.0}
\left\| \nabla \phi_e \ast \Phi_r  - \nabla \phi_e \ast \Phi_{2r}(x)\right\|_{L^2 \left( \Phi_{x,\sqrt{3}r} \right)} 
\leq 
C |e| \mathcal{G}_r(x) \,. 
\end{equation}
This implies that, for every~$r \in [1,\infty)$, 
\begin{equation}
\label{e.conv.rto2r}
\left| \nabla \phi_e \ast \Phi_r (x) - \nabla \phi_e \ast \Phi_{2r}(x)\right| 
\leq 
C |e| \mathcal{G}_r(x) \,. 
\end{equation}
Since~$\lim_{r\to \infty} \nabla \phi_e \ast \Phi_r (x) = 0$ almost surely, we may sum over the scales to obtain that 
\begin{equation}
\label{e.conv.r}
\left| \nabla \phi_e \ast \Phi_r (x) \right| 
\leq 
C |e| \sum_{j=0}^\infty \mathcal{G}_{2^jr}(x) \,. 
\end{equation}
Recall that the random variable~$\mathcal{G}_r(x)$, defined in~\eqref{e.uglyG}, encodes the rate of convergence of~$\b_r$ to~$\ahom$. For instance, if~$\beta\in (0,1)$ and the hypothesis of Proposition~\ref{p.psipsi.coarse} are valid, then the bound~\eqref{e.psipsi.coarse} yields, for every~$x\in\Rd$ and~$r \in [1,\infty)$, 
\begin{equation}
\label{e.uglyG.bound}
\mathcal{G}_r(x)
=
\O_{\Psi} \bigl( C r^{-\frac d2(1-\beta)} \bigr) 
+ 
\O_{\Psi}^{1-\delta} \bigl( C r^{-\frac d2(1-\beta) - \gamma} \bigr)
+
\O_{\Psi}^{1-\beta} \bigl( C r^{-\nicefrac d2} \bigr)
\,.
\end{equation}
Likewise, in the case~$\beta=0$, if the hypothesis of Proposition~\ref{p.psipsi.coarse.optimal} are valid, then the bound~\eqref{e.betazero} gives the estimate
\begin{equation}
\label{e.uglyG.bound.betazero}
\mathcal{G}_r(x)
=
\O_{\Psi}^{1-\delta} (C r^{-\nicefrac d2})
\,.
\end{equation}
Now, recalling that, for every~$r \geq \X(x)$, 
\begin{align*}  
\left| \nabla \phi_e \ast \Phi_r (x) \right|  \leq C|e| \,,
\end{align*}
we obtain by the previous two displays, interpolating as in~\eqref{e.Psi.interpolation}, 
that, for every~$e \in B_1$ and~$r \in [\X(x),\infty)$, 
\begin{equation}
\label{e.nablacorr.r}
\left| \nabla \phi_e \ast \Phi_r (x) \right| 
\leq 
\left\{
\begin{aligned}
& \O_{\Psi}^{1-\delta} (C   r^{-\frac d2}) & \mbox{if} & \ \beta = 0\,,\\
&\O_{\Psi} \bigl( C    r^{-\frac d2(1-\beta)} \bigr) + \O_{\Psi}^{1-\delta} \bigl( C r^{-\frac d2(1-\beta) - \gamma} \bigr)
 & \mbox{if} & \ \beta\in (0,1).
\end{aligned}
\right.
\end{equation}
These are sharp estimates for the spatial averages of the gradients of the first-order correctors against heat kernels. 

\smallskip

We can get a similar bound for the fluxes from the bound~\eqref{e.nablacorr.r} for the gradients. We simply apply~\eqref{e.coarsened.fist} and then use~\eqref{e.psipsi.coarse},~\eqref{e.betazero} and~\eqref{e.nablacorr.r} to discover that, for every~$e \in B_1$ and~$r \in [\X(x),\infty)$, 
\begin{align}
\label{e.fluxcorr.r}
\lefteqn{
\bigl| \bigl(  ( \a (e + \nabla \phi_e) )   \ast \Phi_r(\cdot)  \bigr) (x) - \ahom e  \bigr| 
} \qquad & 
\notag \\ & 
=
\bigl| \b_r(x) ( \nabla \phi_e \ast \Phi_r) (x) + (\b_r(x) - \ahom) e \bigr|
\notag \\ & 
\leq 
C \bigl| \nabla \phi_e \ast \Phi_r (x)  \bigr| 
+
\bigl| \b_r(x) - \ahom \bigr|
\notag \\ & 
=  
\left\{
\begin{aligned}
& \O_{\Psi}^{1-\delta} (C r^{-\nicefrac d2}) & \mbox{if} & \ \beta = 0\,,\\
&\O_{\Psi} \bigl( C r^{-\frac d2(1-\beta)} \bigr) + \O_{\Psi}^{1-\delta} \bigl( C r^{-\frac d2(1-\beta) - \gamma} \bigr)
 & \mbox{if} & \ \beta\in (0,1). \\
\end{aligned}
\right.
\end{align}
These are, again, sharp estimates for the spatial averages of the fluxes of the first-order correctors against heat kernels. A similar estimate can be deduced for the energy as well. 

\smallskip

Using~\eqref{e.nablacorr.r},~\eqref{e.fluxcorr.r} and the characterization of negative Sobolev norms in~\eqref{e.negsob.sint}, we are able to complete the proof of Theorem~\ref{t.corr.optimal}.

%

\begin{proof}[Proof of Theorem~\ref{t.corr.optimal}]
Fix~$e \in \R^d$,~$\delta \in (0,\frac12)$,~$s \in (0,\infty)$. Let  
\begin{align*}  
f_\ep(x) := \nabla \phi_e\bigl( \tfrac{x}{\ep} \bigr)  
\quad \mbox{or} \quad 
f_\ep(x) :=  \a\bigl( \tfrac{x}{\ep} \bigr) \bigl( e + \nabla \phi_e\bigl( \tfrac{x}{\ep} \bigr) \bigr) - \ahom e \,.
\end{align*}
Then, by changing variables, we have by~\eqref{e.nablacorr.r} and~\eqref{e.fluxcorr.r} that, for every~$t \geq \ep^2 \X^2(\tfrac x\ep)$, 
\begin{align*}  
f_\ep\ast \Phi(t,\cdot) (x) 
& 
= 
\int_{\R^d} \Phi\bigl( \tfrac{t}{\ep^2} , \tfrac{x}{\ep} - y \bigr) f\bigl( y \bigr) \, dy 
=
 \O_{\Psi}^{1-\delta} (C |e| \ep^{\frac d2(1-\beta)} t^{-\frac d4(1-\beta)})
 \,.
\end{align*}
It follows by integration that 
\begin{multline}  \label{e.corr.bounds.pre}
\biggl( 
\int_{\R^d} \int_{\ep^2 \X^2\bigl(\tfrac x\ep\bigr) \wedge 1}^1 t^s   \big| \bigl( f_\ep  \ast \Phi(t,\cdot) \bigr) (x) \big|^2 \, \frac{dt}{t} \, \exp(-|x|)dx
\Biggr)^{\!\!\nicefrac12} 
\\
= 
\left\{
\begin{aligned}
& \O_{\Psi}^{1-\delta} (C |e| \ep^s) & \mbox{if} & \ 0 < s < \frac{d}{2}(1-\beta)\,,\\
&\O_{\Psi}^{1-\delta}  \bigl( C  |e| \ep^{\frac d2(1-\beta)} | \log \ep|^{\nicefrac12} \bigr) & \mbox{if} & \ s = \frac{d}{2}(1-\beta) \,, \\
&\O_{\Psi}^{1-\delta} \bigl( C |e| \ep^{\frac d2(1-\beta)} \bigr) & \mbox{if} & \ s > \frac{d}{2}(1-\beta) \,.
\end{aligned}
\right.
\end{multline}
On the other hand, we trivially get that
\begin{align} \notag  
\lefteqn{
\int_{\R^d} \int_{0}^{\ep^2 \X^2\bigl(\tfrac x\ep\bigr) \wedge 1} t^s   \big| \bigl( f_\ep  \ast \Phi(t,\cdot) \bigr) (x) \big|^2 \, \frac{dt}{t} \, \exp(-|x|) \, dx 
} \qquad &
\\ 
\notag &
\leq 
\frac{\ep^{2s}}s \int_{\R^d} \sup_{t \in (0,1)}  \big| \bigl( f_\ep \ast \Phi(t,\cdot) \bigr) (x) \big|^2 \X^{2s}(\tfrac x\ep) 
\, \exp(-|x|) \, dx 
\,.
\end{align}
It is straightforward to show, using Vitaly's covering theorem, that the maximal function above is bounded in~$L^p$: For every~$p \in (1,\infty)$ there exists~$C(p,d)<\infty$ such that 
\begin{align*}  
\Big\| \sup_{t \in (0,1)}  \big| \bigl( f_\ep \ast \Phi(t,\cdot) \bigr) \Big\|_{L^p(\exp(-|\cdot|))} 
\leq 
C \big\| f_\ep \big\|_{L^p(\exp(-|\cdot|))} 
\,.
\end{align*}
By Meyer's estimate and the bound~\eqref{e.corr.grad.bound.pre} we deduce that there exist constants~$\alpha(d,\lambda,\Lambda) \in (0,1)$ and~$C(d,\lambda,\Lambda)<\infty$ such that 
\begin{align*}  
\big\| f_\ep \big\|_{L^{2(1+\alpha)}(\exp(-|\cdot|))} 
\leq 
C \big\| f_\ep \big\|_{L^2(\exp(-|\cdot|/2))}  
\leq
C  ( \ep\X \vee 1 )^{\nicefrac d2} 
\,.
\end{align*}
In particular, the above two displays together with H\"older's inequality yield that
\begin{align} \notag  
\lefteqn{
\biggl(  \int_{\R^d} \int_{0}^{\ep^2 \X^2\bigl(\tfrac x\ep\bigr) \wedge 1} t^s   \big| \bigl( f_\ep  \ast \Phi(t,\cdot) \bigr) (x) \big|^2 \, \frac{dt}{t} \, \exp(-|x|) \, dx 
\Biggr)^{\!\!\nicefrac12} 
} \qquad &
\\ 
\notag &
\leq 
C \big\| f_\ep \big\|_{L^{2(1+\alpha)}(\exp(-|\cdot|))} 
\big\| \X^s(\tfrac \cdot\ep) \big\|_{L^{\nicefrac{2(1{+}\alpha)}{\alpha} } (\exp(-|\cdot|)) }
\\ 
\notag &
\leq 
C|e| ( \ep\X \vee 1 )^{\nicefrac d2} 
\big\| \X^s(\tfrac \cdot\ep) \big\|_{L^{\nicefrac{2(1{+}\alpha)}{\alpha} } (\exp(-|\cdot|)) }
\,.
\end{align}
Finally, combining the above inequality with~\eqref{e.negsob.sint} and~\eqref{e.corr.bounds.pre}, noticing that the left-hand side is bounded in~\eqref{e.corr.bounds} allowing interpolation as in~\eqref{e.Psi.interpolation}, we conclude the proof. 
\end{proof}

\begin{remark}
If~$\beta \in (0,1)$, the estimate~\eqref{e.corr.bounds} can be slightly improved using~\eqref{e.psipsi.coarse}. 
\end{remark}

\subsection{The shortcut to optimal estimates under~\texorpdfstring{$\LSI$}{{LSI}} or~\texorpdfstring{$\SG$}{{SG}}-type assumptions}
\label{s.LSI}

We showed in Sections~\ref{ss.LSI} and~\ref{ss.SG} that an assumption that~$\P$ satisfies a logarithmic Sobolev or spectral gap inequality implies that it also satisfies a general~$\CFS$ condition. 
In particular,~$\LSI(\beta,\rho)$ implies~$\CFS(\beta,\Gamma_2)$, and an assumption of~$\SG(\beta,\rho)$ implies~$\CFS(\beta,\Gamma_1)$.
However, in this chapter, we have proved our optimal quantitative estimates on the first-order correctors under a stronger~$\CFS$ condition, which we needed to close the loop in the renormalization argument. These stronger~$\CFS$ conditions may not be satisfied, in general, if we only assume~$\LSI(\beta,\rho)$ or~$\SG(\beta,\rho)$.\footnote{This is really an academic point, in some sense, since essentially all of the practical examples do satisfy the stronger~$\CFS$ conditions.}

\smallskip

There is, however, another way to obtain estimates on the first-order correctors using assumptions like~$\LSI$ or~$\SG$, which avoids the renormalization argument and is quite quick. This ``shortcut'' was introduced by Naddaf and Spencer~\cite{NS2} in a paper that predates the rest of the quantitative homogenization theory. It was subsequently extended and generalized in a sequence of many works of Gloria, Otto and collaborators, beginning with~\cite{GO1,GO2}. The idea is very natural and straightforward: apply the nonlinear concentration inequality to quantities involving the corrector, and compute. This quickly reduces to gradient bounds on the corrector---which, in our setting, we have already estimated optimally in Chapter~\ref{s.regularity}. Note that all of the results of Chapter~\ref{s.regularity} are applicable in the case of~$\LSI$ or~$\SG$ since they were obtained under the more general~$\CFS(\beta,\Psi)$ condition. 

\smallskip

The catch is that the stochastic integrability of the estimates one obtains in this way is sub-optimal. For instance, as we will see below, under an~$\LSI$ assumption, in the very best case, we obtain only slightly more than an exponential moment, which is far from the Gaussian moments we expect under suitable versions of the strong~$\CFS$ condition. Under a~$\SG$ assumption, we get a stretched exponential moment with an exponent slightly more than~$\nicefrac13$. 
However, in certain situations, we may not care about obtaining optimal stochastic moments, in which case the argument using~$\LSI$ or~$\SG$ are quite convenient, as they often lead to a quicker and easier proof of corrector estimates with the optimal scaling in~$\ep$. 

\smallskip

In this section, we present this shortcut to the optimal estimates on the scaling of the first-order correctors, under the assumption that~$\P$ satisfies either~$\LSI(\beta,\rho)$ or~$\SG(\beta,\rho)$. We will consider the case~$\beta = 0$ for the simplicity of the presentation. This simplifies the argument because we do not need to worry about \emph{localizing} the random variables. Recall that~$\LSI(\beta,\rho)$ and~$\SG(\beta,\rho)$ assumptions we defined in Section~\ref{s.CFS} are more general than usual and require the random variables to be measurable with respect to a fixed, bounded cube. Taking~$\beta = 0$ allows us to dispense with this constraint. Adapting the argument to general~$\beta \in [0,1)$, even under our more general~$\LSI$ and~$\SG$ conditions, can be handled with the aid of Lemma~\ref{l.correctorloc}, above. 

\smallskip

\begin{theorem}[{Optimal quantitative estimates for~$\LSI(0,\rho)$~\&~$\SG(0,\rho)$}]
\label{t.wecan.do.SG.too}
Suppose~$\P$ is a~$\Zd$--stationary measure on~$(\Omega,\F)$ that satisfies either~$\LSI(0,\rho)$ or~$\SG(0,\rho)$ for some~$\rho>0$. 
Then there exist~$C(\rho, d,\lambda,\Lambda)<\infty$ and~$\delta(d,\lambda,\Lambda)>0$ such that, for every~$\f \in H^{\nicefrac d2}(B_1;\Rd)$ with~$\int_{B_1}\f =0$ and~$\| \f \|_{H^{\nicefrac d2}(B_1)} \leq 1$, if we define 
\begin{equation*}
\f_r(x) := r^{-d} \f\Bigl( \frac\cdot r\Bigr)\,\quad r>0,\ x\in\Rd\,,
\end{equation*}
then, for every~$|e|=1$ and~$r\geq 1$, we have the estimate
\begin{equation}
\label{e.cheating}
\biggl| \int_{\Rd} 
\f_r \cdot \nabla \phi_e \,dx\biggr|
+
\biggl| \int_{\Rd} 
\f_r  \cdot \bigl( \a  \bigl( e+ \nabla \phi_e\bigr) - \ahom\bigr) \,dx\biggr|
= 
\left\{
\begin{aligned}
& \O_{\Gamma_{1+\delta}} \bigl(C r^{-\frac d2} \bigr) & \mbox{if} & \ \LSI(0,\rho) \ \mbox{holds}\,, \\
& \O_{\Gamma_{\nicefrac13+\delta}} \bigl(C r^{-\frac d2} \bigr) & \mbox{if} & \ \SG(0,\rho)\ \mbox{holds}\,.
\end{aligned}
\right.
\end{equation}
Moreover, the coarse-grained coefficients~$\b_r$ defined in Section~\ref{ss.coarsegrain.with.tempo} satisfy, for every~$r\geq 1$,
\begin{equation}
\label{e.cheating.CG}
\b_r(x) - \ahom  
= 
\left\{
\begin{aligned}
& \O_{\Gamma_{1+\delta}} \bigl(C r^{-\frac d2} \bigr) & \mbox{if} & \ \LSI(0,\rho) \ \mbox{holds}\,, \\
& \O_{\Gamma_{\nicefrac13+\delta}} \bigl(C r^{-\frac d2} \bigr) & \mbox{if} & \ \SG(0,\rho)\ \mbox{holds}\,.
\end{aligned}
\right.
\end{equation}
\end{theorem}

The estimates of~Theorem~\ref{t.corr.optimal} are valid, by the same proof, under~$\LSI(0,\rho)$ or~$\SG(0,\rho)$, with~$\beta = 0$ and with the obvious weakening of the stochastic integrability inherited from the estimates in the statement above. 

\smallskip

To prove Theorem~\ref{t.wecan.do.SG.too}, 
we will apply the concentration inequalities to 
random variables measuring the spatial averages of the fields~$\nabla \phi_e$ and of~$\a(e+\nabla \phi_e) - \ahom e$. Given a vector field~$\f\in L^2(\Rd)$ with sufficiently fast decay (perhaps compact support), we consider the random variable 
\begin{equation}
\label{e.rvX}
X:=
\int_{\Rd} \f(x)\cdot \nabla \phi_e(x)\,dx
\,.
\end{equation}
In order to use~$\LSI(0,1)$ or~$\SG(0,1)$ to estimate the fluctuations of this random variable, we need to estimate the quantity 
\begin{equation}
\label{e.Zderv}
Z := \sum_{z\in 3^n\Zd} \left|  \partial_{\a(z+\cu_n)} \int_{\Rd} \f(x)\cdot \nabla \phi_e(x)\,dx
\right|^2
\,.
\end{equation}
As we will show, this boils down to applying the large-scale~$C^{0,1}$ estimate. After obtaining an appropriate estimate, we will then plug this estimate into any nonlinearity concentration inequality, such as Lemma~\ref{l.LSI.moments} or Lemma~\ref{l.SG.moments}, to obtain bounds on the fluctuations of~$X$ itself. 

\smallskip

We begin with the following deterministic lemma, highlighting the importance of the large-scale~$C^{0,1}$ estimate. 

\begin{lemma}
\label{l.compute.mall.maul}
There exists~$C(d,\lambda,\Lambda)<\infty$ such that, for~$e\in \Rd$,~$z\in\Zd$,~$n\in\N$ and~$\mathbf{f}\in L^2(\Rd;\Rd)$,
\begin{equation}
\label{e.malliavin.gradcorr}
\left| \partial_{\a(z+\cu_n)}
\int_{\Rd} \mathbf{f}(x)\cdot \nabla \phi_e(x)\,dx
\right| 
\leq 
\int_{z+\cu_n}
\left| \nabla v(x) \right| \left|e+ \nabla \phi_e(x) \right|\,dx \,,
\end{equation}
where~$v$ is the solution of 
\begin{equation}
-\nabla \cdot \a^t \nabla v = \nabla \cdot \mathbf{f} \quad \mbox{in} \ \Rd. 
\end{equation}
\end{lemma}
\begin{proof}
The computation is straightforward, and we follow~\cite[proof of Proposition 3]{GNO2}. We assume without loss of generality that~$z=0$. Let~$t>0$ and~$\a,\tilde{\a}\in \Omega$ with~$\left| \a- \tilde{\a}\right| \leq t \indc_{\cu_n}$. Let~$\tilde v$ solve~$-\nabla \cdot \tilde \a^t \nabla \tilde v = \nabla \cdot \mathbf{f}$ in~$\Rd$. Then~$\| \nabla v - \nabla \tilde v \|_{L^2(\R^d)}  \leq C t \|\f \|_{L^2(\R^d)}$.  Also fix~$e\in \Rd$ and let~$\tilde{\phi}_e$ denote the first-order corrector corresponding to~$\tilde{\a}$. We have that 
\begin{equation}
-\nabla \cdot \tilde{\a} \nabla \bigl( \phi_e-\tilde{\phi}_e \bigr) 
= 
\nabla \cdot\bigl( ( \a - \tilde{\a} )( e+ \nabla \phi_e ) \bigr) \quad \mbox{in} \ \Rd. 
\end{equation}
Testing this equation with~$\tilde v$, and the equation of~$\tilde v$ with~$\phi_e-\tilde{\phi}_e$, yields the identity
\begin{align*}
\int_{\Rd} \f \cdot \nabla \bigl( \phi_e-\tilde{\phi}_e \bigr) 
&
=
\int_{\Rd} \nabla \tilde v \cdot \tilde \a \nabla ( \phi_e-\tilde{\phi}_e )
=
\int_{\Rd} \nabla \tilde v \cdot \bigl( \a - \tilde{\a} \bigr) ( e+ \nabla  \phi_e ).
\end{align*}
Therefore
\begin{equation}
\left| 
\int_{\Rd} \f\cdot \nabla \phi_e - \int_{\Rd} \f \cdot \nabla \tilde{\phi}_e 
\right|
\leq 
t \int_{\cu_n} \left| \nabla \tilde v \right| \left| e+\nabla {\phi}_e \right|
\,.
\end{equation}
This yields the lemma since~$\| \nabla v - \nabla \tilde v \|_{L^2(\R^d)}  \leq C t \|\f \|_{L^2(\R^d)}$. 
\end{proof}

\begin{proof}[{Proof of Theorem~\ref{t.wecan.do.SG.too}}]
We begin with the input from the regularity theory. Let~$\X$ be the minimal scale for the large-scale~$C^{0,1}$ estimate, given by Theorem~\ref{t.C01}. 
For every~$n\in\N$ with~$3^n \geq \X$, we have 
\begin{equation}
\fint_{\cu_n} \left| e+ \nabla \phi_e(x) \right|^2\,dx
\leq 
C.
\end{equation}
Moreover, by the hole-filling estimate, there exists~$\alpha(d,\lambda,\Lambda)>0$ such that, for every~$n\in\N$ with~$3^n\geq \X(z)$,~$k\in\N$ with~$k\leq n$ and~$z\in 3^k\Zd\cap\cu_n$, we have
\begin{equation}
\int_{z+\cu_k} 
\left| e+ \nabla \phi_e(x) \right|^2\,dx
\leq 
3^{-d\alpha(n-k)}
\int_{z+\cu_{n+1}}  \left| e+ \nabla \phi_e(x) \right|^2\,dx
\leq 
C 3^{-\alpha(n-k)} \left| \cu_n \right|.
\end{equation}
In particular, for every~$z\in\Zd$ and~$k\in\N$, 
\begin{equation*}
\| e + \nabla \phi_e \|_{\underline{L}^2(z+\cu_k)}
\leq 
C \bigl( 1+ \X(z) 3^{-k} \bigr)^{\frac12(d-\alpha)}
\,.
\end{equation*}
Now let~$v$ be the function in Lemma~\ref{l.compute.mall.maul}. 
We may estimate the right side of~\eqref{e.malliavin.gradcorr} by 
\begin{align}
\label{e.thisisverydifficult}
\int_{z+\cu_n}
\left| \nabla v(x) \right| \left|e+ \nabla \phi_e(x) \right|\,dx
&
\leq 
| \cu_n | 
\| \nabla v \|_{{L}^2(z+\cu_n)} \| e+\nabla \phi_e \|_{{L}^2(z+\cu_n)}
\notag \\ &
\leq 
C |\cu_n|^{\nicefrac12} \bigl(1+\X(z) 3^{-n} \bigr)^{\frac12(d-\alpha)}
\| \nabla v \|_{{L}^2(z+\cu_n)} 
\notag \\ &
= 
C |\cu_n| \bigl(1+\X(z) 3^{-n} \bigr)^{\frac12(d-\alpha)}
\| \nabla v \|_{\underline{L}^2(z+\cu_n)} 
\,.
\end{align}
We will use deterministic elliptic estimates for the function~$v$.  
To work with a fixed length scale~$r\geq 1$, we fix~$\f_1\in L^{q}(\Rd)$ with support in~$B_1$, for some~$q>2$, and we stretch it by defining
\begin{equation}
\label{e.fr.LSI}
\f_r(x):=  r^{-d} \f_1 \Bigl( \frac xr \Bigr)\,, 
\quad x\in\Rd, \  r\in [1,\infty). 
\end{equation}
Let~$v_r$ be the solution of~$-\nabla \cdot \a^t \nabla v_r = \nabla \cdot \f_r$ in~$\Rd$. By standard elliptic estimates, 
\begin{equation*}
\left\| v_r \right\|_{\underline{L}^2(B_{2R}\setminus B_R )}
\leq 
CR^{1-d} \left\| \f_r \right\|_{\underline{L}^1(\Rd)} 
=
CR^{1-d} \left\| \f_1 \right\|_{L^1(B_1)} \,,
\quad \forall R\geq 2r
\end{equation*}
and
\begin{equation*}
\left\| v_r \right\|_{\underline{L}^2(B_{2r})}
\leq 
Cr  \| \f_r \|_{\underline{L}^2(B_r)}
\leq
Cr^{1-\frac d2}  \left\| \f_r \right\|_{L^2(\Rd)}
=
Cr^{1-d} \left\| \f_1 \right\|_{L^2(B_1)} .
\end{equation*}
Furthermore, by the interior Meyers estimate (see~\cite[Theorem C.1]{AKMBook}), we have that, assuming~$q \in [2, 2+c(d,\lambda,\Lambda)]$,  
\begin{equation}
\label{e.go.Meyersout}
\left\| \nabla v_r \right\|_{\underline{L}^q(B_{2R}\setminus B_R )}
\leq 
CR^{-d} \left\| \f_1 \right\|_{L^1(B_1)} \,,
\quad \forall R\geq 2r
\end{equation}
and
\begin{equation}
\label{e.go.Meyersin}
\left\| \nabla v_r \right\|_{\underline{L}^q(B_{4r})}
\leq 
C\| \f_r \|_{\underline{L}^q(B_r)}
=
C r^{-d} \left\| \f_1 \right\|_{L^q(B_1)} .
\end{equation}
Combining these with~\eqref{e.malliavin.gradcorr} and~\eqref{e.thisisverydifficult} and using H\"older's inequality, we obtain, for every exponent~$q\in (2,2+c)$, 
\begin{align*}
\lefteqn{ \sum_{z\in 3^n\Zd\cap B_{2r}} \left|  \partial_{\a(z+\cu_n)} \int_{\Rd} \f_r(x)\cdot \nabla \phi_e(x)\,dx
\right|^2
} \qquad & 
\notag \\ & 
\leq
C r^d |\cu_n| 
\avsum_{z\in 3^n\Zd\cap B_{2r}}
 \bigl(1+\X(z) 3^{-n} \bigr)^{\! d-\alpha}
\| \nabla v_r \|_{\underline{L}^2(z+\cu_n)}^2
\notag \\ & 
\leq 
Cr^d |\cu_n|
\biggl( \avsum_{z\in 3^n\Zd\cap B_{2r}}
 \bigl(1+\X(z) 3^{-n} \bigr)^{\! \frac{q(d-\alpha)}{q-2}}
\biggr)^{\!\!\frac{q-2}{q} }
\biggl( 
\avsum_{z\in 3^n\Zd\cap B_{2r}}
\| \nabla v_r \|_{\underline{L}^2(z+\cu_n)}^{q}
\biggr)^{\!\!\nicefrac 2q}
\notag \\ & 
\leq 
Cr^d |\cu_n| 
\| \nabla v_r \|_{\underline{L}^q(B_{2r})}^2 
\biggl( \avsum_{z\in 3^n\Zd\cap B_{2r}}
 \bigl(1+\X(z) 3^{-n} \bigr)^{\! \frac{q(d-\alpha)}{q-2}}
\biggr)^{\!\!\frac{q-2}{q} }
\notag \\ & 
\leq 
C r^{-d} |\cu_n| \| \f_1\|_{L^q(B_1)}^2
\biggl( \avsum_{z\in 3^n\Zd\cap B_{2r}}
 \bigl(1+\X(z) 3^{-n} \bigr)^{\! \frac{q(d-\alpha)}{q-2}}
\biggr)^{\!\!\frac{q-2}{q} }
\,.
\end{align*}
We proceed similarly for~$z$ outside of~$B_{2r}$, using~\eqref{e.go.Meyersout} instead of~\eqref{e.go.Meyersin}, to get, for every~$R\geq 2r$, 
\begin{align*}
\lefteqn{ \sum_{z\in 3^n\Zd\cap B_{2R} \setminus B_R} \left|  \partial_{\a(z+\cu_n)} \int_{\Rd} \f(x)\cdot \nabla \phi_e(x)\,dx
\right|^2
} \qquad & 
\notag \\ & 
\leq
CR^{-d}  |\cu_n|  \| \f_1\|_{L^1(B_1)} 
\biggl( \avsum_{z\in 3^n\Zd\cap B_{2R} \setminus B_R}
\bigl(1+\X(z) 3^{-n} \bigr)^{\! \frac{q(d-\alpha)}{q-2}}
\biggr)^{\!\!\frac{q-2}{q} }
\,.
\end{align*}
Combining these, we get
\begin{align*}
\lefteqn{ \sum_{z\in 3^n\Zd} \left|  \partial_{\a(z+\cu_n)} \int_{\Rd} \f(x)\cdot \nabla \phi_e(x)\,dx
\right|^2
} \quad & 
\notag \\ & 
\leq
C r^{-d} |\cu_n|  \| \f_1\|_{L^q(B_1)} 
\biggl( 
\frac{|\cu_n|}{r^d}
\sum_{z\in 3^n\Zd}
\biggl( 1+ 
\frac{|z|}{r} 
\biggr)^{\!\!-d-\frac{dq}{q-2}}
\bigl(1+\X(z) 3^{-n} \bigr)^{\! \frac{q(d-\alpha)}{q-2}}
\biggr)^{\!\!\frac{q-2}{q} }
\,.
\end{align*}
Since 
\begin{equation*}
\frac{|\cu_n|}{r^d}
\sum_{z\in 3^n\Zd}
\biggl( 1+ 
\frac{|z|}{r} 
\biggr)^{\!\!-d-\frac{dq}{q-2}}  \leq C\,,
\end{equation*}
if we define the random variable~$\mathcal{Z}_r$ by
\begin{equation*}
\mathcal{Z}_r
:= 
\biggl( 
\frac{|\cu_n|}{r^d}
\sum_{z\in 3^n\Zd}
\biggl( 1+ 
\frac{|z|}{r} 
\biggr)^{\!\!-d-\frac{dq}{q-2}}
\bigl(1+\X(z) 3^{-n} \bigr)^{\! \frac{q(d-\alpha)}{q-2}}
\biggr)^{\!\!\frac{q-2}{q} }\,,
\end{equation*}
then we observe have that, by~\eqref{e.LSI.okay},~\eqref{e.SG.okay} and~\eqref{e.XLipschitz},
\begin{equation*}
\left\{
\begin{aligned}
& \LSI(0,\rho) 
\implies \quad 
\mathcal{Z}_r^{\frac{d}{2(d-\alpha)}  } 
= 
\O_{\Gamma_2}(C)\,, \\
& 
\SG(0,\rho) 
\implies \quad 
\mathcal{Z}_r^{\frac{d}{2(d-\alpha)}  } 
= 
\O_{\Gamma_1}(C) \,.
\end{aligned}
\right.
\end{equation*}
In other words, we have just more than an exponential moment for~$\mathcal{Z}_r$ under~$\LSI(0,\rho)$, and just more than~$\nicefrac12$ of an exponential moment under~$\SG(0,\rho)$: for some~$\delta(d,\lambda,\Lambda)>0$,
\begin{equation*}
\left\{
\begin{aligned}
& \LSI(0,\rho) 
\implies \quad 
\mathcal{Z}_r = \O_{\Gamma_{1+\delta}}(C)\,, \\
& 
\SG(0,\rho) 
\implies \quad 
\mathcal{Z}_r
= 
\O_{\Gamma_{d/2(d-\alpha)}}(C)\,.
\end{aligned}
\right.
\end{equation*}
Therefore, we have shown that the random variable~$Z$ in~\eqref{e.Zderv} satisfies
\begin{equation*}
\left\{
\begin{aligned}
& \LSI(0,\rho) 
\implies \quad 
Z = \O_{\Gamma_{1+\delta}}
\bigl(Cr^{-d} |\cu_n| \| \f_1\|_{L^q(B_1)}\bigr)
\,, \\
& 
\SG(0,\rho) 
\implies \quad 
Z = \O_{\Gamma_{\nicefrac12+\delta}}
\bigl(Cr^{-d} |\cu_n| \| \f_1\|_{L^q(B_1)}\bigr)
\,.
\end{aligned}
\right.
\end{equation*}
We now use the~$\LSI(0,\rho)$ and~$\SG(0,\rho)$ assumptions a second time---using Lemma~\ref{l.LSI.moments} with~$s$ slightly larger than the one for~$\LSI$, and Lemma~\ref{l.SG.moments} with~$s$ slightly larger than~$\nicefrac13$ for~$\SG$---to obtain a fluctuation bound on the random variable~$X$ in~\eqref{e.rvX}: for some~$\delta>0$, 
\begin{equation*}
\left\{
\begin{aligned}
& \LSI(0,\rho) 
\implies \quad 
X - \E \bigl[ X \bigr] = \O_{\Gamma_{1+\delta}}(Cr^{-\frac d2})
\,, \\
& 
\SG(0,\rho) 
\implies \quad 
X - \E \bigl[ X \bigr] = \O_{\Gamma_{\nicefrac13+\delta}}(Cr^{-\frac d2})
\,.
\end{aligned}
\right.
\end{equation*}
This takes care of the fluctuations of the spatial average of the gradient of a corrector. The same estimate for the spatial averages of the flux of a corrector is obtained in the same way.

\smallskip

We next use the following simple fact: if~$g \in L^2_{\mathrm{loc}}(\Rd)$ is a periodic function with zero mean and~$f \in H^{s}(\Rd)\cap L^1(\Rd)$ such that~$\int_{\Rd} f = 0$, and we define~$f_r:= r^{-d} f(\frac\cdot r)$, then 
\begin{equation*}
\biggl| \int_{\Rd} f_r(x) g(x)\,dx \biggr| \leq Cr^{-s}\,.
\end{equation*}
We leave a proof of this fact to the reader. 
Using it, we can obtain that, in the case that~$\f_1$ has entries in~$H^{\nicefrac d2}(B_1)$ which have zero mean, then 
\begin{equation*}
\biggl| \E \biggl[ \int_{\Rd} \f_r(x)\cdot \nabla \phi_e(x)\,dx \biggr] \biggr| 
=
\biggl|  \int_{\Rd} \f_r(x)\cdot \E \bigl[\nabla \phi_e(x)\bigr] \,dx  \biggr| 
\leq Cr^{-\nicefrac d2}\,,
\end{equation*}
since~$x\mapsto \E \bigl[ \nabla \phi_e(x) \bigr]$ is a mean zero, it is a periodic function by the assumption of~$\Zd$--stationarity. 
Similarly, 
\begin{equation*}
\biggl| \E \biggl[ \int_{\Rd} \f_r(x)\cdot \bigl( \a(x)(e+ \nabla \phi_e(x)) -\ahom e \bigr) \,dx \biggr] \biggr| 
\leq 
Cr^{-\nicefrac d2}\,.
\end{equation*}
Using these together with the fluctuation bounds, we obtain the estimates~\eqref{e.cheating}. The estimate~\eqref{e.cheating.CG} for the coarse-grained coefficients~$\b_r(x)$ follows from~\eqref{e.cheating} immediately. 
This completes the proof. 
\end{proof}

\subsection{Renormalization with a divergence-form forcing term}
\label{ss.rhs.optimal}

In this section, we explain how to extend the renormalization argument presented here to the case of a divergence-form right-hand side. In particular, we obtain estimates on the ``zero-slope'' corrector~$\psi_{\f}$ in the statement of Theorem~\ref{t.f.optimalstochasticintegrability}. 

\smallskip

We work under a generalized~$\CFS$ condition defined in Section~\ref{ss.CFS.gen}. In particular, we use the definition of Malliavin derivative given in~\eqref{e.Mall.Maul.yes} with~$\b = (\a,\f)$ and the norm~$\vertiii\cdot \vertiii_n$ defined in~\eqref{e.triple.your.bar.meso} and rewritten here for convenience: 
\begin{equation}
\label{e.triple.your.bar.meso.again}
\vertiii (\a,\f) \vertiii_n:= \| \a \|_{L^\infty(\cu_n)} 
+
\sup_{z\in 3^{k_n} \Zd \cap \cu_n}
\| \f \|_{\underline{L}^2(z+\cu_{k_n})}
\,,
\end{equation}
where we take~$k_n =  \lceil n - (\log n)^{\nicefrac12} \rceil$.
This is slightly stronger than the~$\| \cdot \|_{\underline{L}^2(\cu_m)}$ norm used in Section~\ref{ss.rhs} because it requires that the~$L^2$ norm does not concentrate in any mesocube, with the mesoscale taken to be very large (almost the macroscopic scale). 
The reason for this mesoscale is technical, but it seems necessary. It has to do with obtaining estimates at the optimal scale~$r^{-\frac d2(1-\beta)}$: if we would be willing to settle for slightly suboptimal estimates (i.e.,~$r^{-\alpha}$ for every~$\alpha<\frac d2(1-\beta)$) then we would not need to fuss. 

\smallskip

As in Section~\ref{ss.rhs}, in addition to the~$\CFS$ condition, we also need to assume some quantitative version of the condition~``$\f \in L^2_{\mathrm{loc}}(\Rd)$,'' which is the one already stated in~\eqref{e.f.dumbbound}.

\smallskip

Our assumptions are, therefore, the following. 
Let~$\delta,\gamma> 0$ and~$\beta \in [0,1)$,~$M, N \in [1,\infty)$ and~$p \in (\frac{2}{1-\beta},\infty)$.  
Let~$\Psi:[1,\infty) \to [0,\infty)$ be an increasing function satisfying~\eqref{e.Psi.pgrowth.theorem}.
We assume that 
\begin{itemize}
\item If~$\beta \in (0,1)$, then~$\P$ satisfies~$\CFS(\beta,\Psi,\Psi(M\cdot),0)$.
\item If~$\beta = 0$, then~$\P$ satisfies~$\CFS(0,\Psi_\sigma,\Psi_\sigma(M\cdot),\gamma)$ for each~$\sigma \in \{0,\delta\}$.
\end{itemize}
Here, the~$\CFS$ conditions are to be understood in the sense of Sections~\ref{ss.CFS.gen} and~\ref{ss.rhs}, as explained above. 
We also require that~$\P$ satisfy~\eqref{e.f.dumbbound}, which we recall here for convenience: for every~$m\in\N$, 
\begin{equation}
\label{e.f.dumbbound.again}
\| \f \|_{\underline{L}^2(\cu_m)}^2 
\leq 
1 +  \O_\Psi\bigl(\mathsf{K}  3^{-\frac d2 (1-\beta) m} \bigr)
\,.
\end{equation}

Our main result is the following improvement of Theorem~\ref{t.f.optimalstochasticintegrability}. 

\begin{theorem} \label{t.f.optimal}
Suppose that~$\beta,\gamma,\delta,M,N,p,\Psi$ and~$\P$ are as above. Then there exists a constant vector~$\overline{\f}\in \Rd$, a constant~$C(\delta,M,N,p,\gamma,\dataref)<\infty$, and, for every~$x \in \R^d$, a random variable~$\X(x)$ satisfying~$\X(x) = \O_{\Psi}(C)$ such that
\begin{equation} \label{e.f.optimal}
\Bigl(
\bigl| \bigl(\Phi_r \ast \nabla \psi_{\f} \bigr)(x) \bigr| + 
\bigl| \bigl( \Phi_r \ast (\a \nabla \psi_{\f} + \f - \overline{\f})\bigr) (x) \bigr| 
\Bigr)
\indc_{\{r \geq \X(x)\}}
= 
\O_\Psi^{1-\delta}\bigl(C r^{\frac d2(1-\beta)}\bigr) \,.
\end{equation}
In the case~$\beta=0$ the estimate~\eqref{e.f.optimal} is valid for every~$r\geq 1$. 
\end{theorem}

\smallskip

Let~$\psi_{\f}$ be given by Theorem~\ref{t.f.optimalstochasticintegrability}. We define the coarsened force~$\f_r$, for every~$x \in \R^d$ and~$r >0$, as  
\begin{align}  \label{e.fr}
\f_r(x) & :=  \chi_r(x) \Phi_r \ast  \bigl( \a \nabla \psi_{\f} + \f - \overline{\f}  - \ahom \nabla \psi_{\f} \bigr) (x) 
+  \overline{\f}
\end{align}
with
\begin{equation*}  
\minscale_r(x) := \bigl(2 - r^{-1} \X(x) \bigr)_+ \wedge 1
\,.
\end{equation*}
Here the minimal-scale~$\X(\cdot)$ is piecewise constant in~$z + \cu_0$,~$z \in \Z^d$, and at~$z$ it is the maximum of corresponding random variables appearing  Theorem~\ref{t.Ck1} and Theorem~\ref{t.f.optimalstochasticintegrability} enlarged with a multiplicative constant prefactor.  The presence of the indicator function above guarantees, due to Theorem~\ref{t.zeroslope}, that~$\f_r$ is bounded. 

\smallskip

Following the outline introduced for the coarse-grained coefficients, we decompose~$\f_R$ as
\begin{align} 
\label{e.fr.splitting}
\notag
\f_R - \overline{\f}
&
=
\underbrace{
\bigl( \f'_{r,t} - \E [ \f'_{r,t} ] \bigr) \ast \Phi_{\sqrt{R^2-r^2}}
}_{\mathrm{fluctuations}}
+
\underbrace{
\Bigl( (\f_r - \E [ \f_r ]) - (\f'_{r,t} - \E [ \f'_{r,t} ] )  \Bigr) \ast \Phi_{\sqrt{R^2-r^2}} 
}_{\mathrm{localization \; defect}}
\\   & \qquad
+
\underbrace{
\Bigl( \f_R - \f_r \ast \Phi_{\sqrt{R^2-r^2}} \Bigr) 
}_{\mathrm{additivity \; defect}}
+
\underbrace{
\bigl( \E [ \f_r ]- \overline{\f} \bigr) \ast \Phi_{\sqrt{R^2-r^2}}
}_{\mathrm{systematic \;  error}}
\,,
\end{align}
where~$\f_{r,t}'(x)$ is a~$\F(B_t(x))$-measurable localization of the field~$\f_r(x)$.  
We next discuss the different terms that appear in the above identity. 

\smallskip

Let us first argue that the additivity error is negligible. Defining 
\begin{align*}  
\g_r(x) := (1-\chi_r(x)) \Phi_r \ast  \bigl( \a \nabla \psi_{\f} + \f - \overline{\f} - \ahom \nabla \psi_{\f}   \bigr) (x)   \,,
\end{align*}
we can write~$\f_r$ as
\begin{align}  \label{e.fr.rewrite}
\f_r(x) = \Phi_r \ast  \bigl( \a \nabla \psi_{\f} + \f - \ahom \nabla \psi_{\f}   \bigr) (x)  - \g_r(x)  \,.
\end{align}
The first term on the right~\eqref{e.fr.rewrite} is perfectly additive, and thus, the additivity error can be made as small as we please in the scale~$r$ with the cost of the stochastic integrability. Indeed, we obtain by Theorem~\ref{t.f.optimalstochasticintegrability} that
\begin{align} \notag \label{e.gr.est}
| \g_r(x) | 
& 
\leq 
C(1-\chi_r(x)) \Bigl( \bigl| \Phi_r \ast  \bigl( \a \nabla \psi_{\f} + \f - \overline{\f}  \bigr)(x)\bigr| + \bigl| \Phi_r \ast  \nabla \psi_{\f}(x)\bigr|  + | \overline{\f} |  \Bigr)
 \\  &
\leq
C \indc_{\{ r \leq \X(x)\} } \Bigl( \frac{\X(x)}{r} \Bigl)^{\nicefrac d2}
\end{align}
and, hence, for every~$\gamma \in (0,1]$, 
\begin{align}  \label{e.fr.add}
\big| \f_R - \f_r \ast \Phi_{\sqrt{R^2-r^2}} \big|
=
\big| \g_R - \g_r \ast \Phi_{\sqrt{R^2-r^2}} \big| 
= 
\O_\Psi^{\gamma (1-\beta)} \bigl( C r^{-\nicefrac{d}{(2\gamma)}} \bigr)  \,.
\end{align}
By interpolation~\eqref{e.Psi.interpolation} and boundedness of~$\f_r$, we thus get that, for every~$\gamma \in (0,1]$, 
\begin{equation}  \label{e.fr.additivity}
\big| \f_R - \f_r \ast \Phi_{\sqrt{R^2-r^2}} \big| = \O_\Psi^{\gamma} \bigl( C r^{-\frac{d}{2\gamma}(1-\beta)} \bigr) \,.
\end{equation}
Thus, we expect that the two main sources of the error in~\eqref{e.fr.splitting} come from the fluctuations and the localization error. As before, fluctuations will be controlled by the~$\CFS$-assumption. The localization error will also be shown to be subcritical. We  localize~$\f_r$, using the very good additivity, by making a simple ansatz that the localized force is of the form, for~$x,y \in \R^d$, 
\begin{equation*}  
\f'_{R}(x)  =  \bigl( \bigl( \indc_{B_{C R \log^{\nicefrac12} R }(x)} \f_r'' \bigr) \ast  \Phi_{\sqrt{R^2-r^2}} \bigr) (x)
\end{equation*}
with 
\begin{equation*}  
\f_r''(x) = \minscale_{r,R}(x) \Phi_r \ast  \bigl( \a \nabla \psi_{\f,x}^{(r)}  + \f - \ahom \nabla \psi_{\f,x}^{(r)} - \overline{\f}  \bigr) (x)  \,,
\end{equation*}
where~$\psi_{\f,x}^{(r)}$ is suitably chosen local quantity solving~$-\nabla \cdot \a \nabla \psi_{\f,t,y} = \nabla \cdot \f$ in a neighborhood of~$x$. To see the sub-criticality in scaling, we first estimate simply as
\begin{equation} \label{e.fvsfloc.pre}
| \f_r(x) - \f'_{r,t}(x) | 
\leq  
C \indc_{\{r \geq \X(x)\}} \Big| \Phi_r \ast  \bigl(  (\a - \ahom) \nabla(\psi_{\f} -  \psi_{\f,t,x}) \bigr)(x) \Big| 
+ 
C  \indc_{\{r < \X(x)\}}
\,,
\end{equation}
and then use the fact that~$\psi_{\f} + \psi - \psi_{\f,t,x}$ belongs to~$\A(B_t(x))$ for every~$\psi \in \A_1$, together with~\eqref{e.coarsened.fist} and~\eqref{e.genfluxmaps.heatscale}, and choose then~$\psi \in \A_1$ to get additional smallness. The details can be found in Step 2 of the proof of Lemma~\ref{l.fr.correctorloc}.

\smallskip

We now move on to formalize the discussion above. First, as in the case of coarse-grained coefficients, we can control the fluctuations of~$\Phi_r \ast \nabla \psi_{\f}$ by means of the fluctuations of the coarsened function~$\f_r$.

\begin{lemma} \label{l.fr.frtopsi}
There exists~$C(\datareff)<\infty$ such that if~$R \geq 2r > 0$, then 
\begin{equation}  \label{e.fr.frtopsi}
\left\| \Phi_r {\ast}  \nabla \psi_{\f}  - \Phi_R {\ast} \nabla \psi_{\f}  \right\|_{L^2 ( \Phi_{\sqrt{R^2-r^2}})}
\leq 
C \sup_{t \geq r} \Bigl( \frac{r}{t} \Bigr)^{\!\nicefrac12}  \Bigl( \| \f_r - \overline{\f} \|_{\underline{L}^2 \left( B_{t} \right)} 
+ 
 r^{-\nicefrac d2}  \| \X \|_{\underline{L}^d \left( B_{t} \right)}^{\nicefrac d2} \Bigr)
\,.
\end{equation}
\end{lemma}

\begin{proof}
First, by~\eqref{e.fr.rewrite} we deduce the following coarsened equation
\begin{equation}  \label{e.fr.coarse.eq}
- \nabla \cdot \ahom \nabla \Phi_r \ast  \psi_{\f}  
= 
\nabla \cdot \f_r + \nabla \cdot \g_r 
\,. 
\end{equation}
Hence, by Lemma~\ref{l.harmdecay}, using the Caccioppoli estimate with Theorem~\ref{t.f.optimalstochasticintegrability} to obtain the qualitative assumption~\eqref{e.harmdecay.ass}, we have that 
\begin{equation*}  
\left\| \Phi_r \ast  \psi_{\f}  - \Phi_R \ast  \psi_{\f}  \right\|_{L^2 ( \Phi_{\sqrt{R^2-r^2}})} 
\leq 
C \sup_{t \geq r}  \Bigl( \frac{r}{t} \Bigr)^{\!\nicefrac12} 
\Bigl(  \| \f_r - \overline{\f} \|_{\underline{L}^2 \left( B_{t} \right)}  +  \| \g_r \|_{\underline{L}^2 \left( B_{t} \right)} \Bigr) \,.
\end{equation*}
The first term on the right appears in~\eqref{e.fr.frtopsi}, and the second is estimated using~\eqref{e.gr.est}.  
\end{proof}

Given the previous lemma, we define a random variable controlling various errors in what follows:
\begin{align}  \label{e.uglyH}
\mathcal{H}_r(x) := 
\sup_{t \geq r}  \Bigl( \frac{r}{t} \Bigr)^{\!\nicefrac12} 
\Bigl( 
\| \f_r - \overline{\f}  \|_{\underline{L}^2 \left( B_{t}(x) \right)} 
+ 
\| \b_r - \ahom  \|_{\underline{L}^2 \left( B_{t}(x) \right)} 
+ 
r^{-\nicefrac d2}  \| \X \|_{\underline{L}^d \left( B_{t} \right)}^{\nicefrac d2} 
\Bigr)
\,.
\end{align}
Notice that by Theorem~\ref{t.optimal}, we already know that the second term on the right is of the size~$\O_\Psi^{1-\delta}(Cr^{-\frac d2(1-\beta)})$ and the third term is of the size~$\O_\Psi^{1-\beta}(Cr^{-\nicefrac d2})$. Actually, we could have done the bootstrap directly for coarse-grained coefficients and coarsened forcing, but we feel that our approach is pedagogically correct; forcing is a separate question from homogenization of coefficients and should be treated like that.

\begin{lemma}[Localization estimates]
\label{l.fr.correctorloc}
Fix~$\ep \in (0,\frac12]$ and~$\delta \in [0,\frac12]$. There exist constants~$\alpha(\beta,d,\lambda,\Lambda) \in (0,1)$ and~$C(\ep,\delta,\datareff)<\infty$ and, for every~$R \in [1,\infty)$, a random field~$\f_R'(\cdot)$ such that, for every~$x \in \R^d$,~$\f_R'(x)$ is~$\F(B_t(x))$-measurable with~$t = 20 d R \log^{\nicefrac 12} R$ satisfying
\begin{equation} \label{e.fr.correctorloc}
| \f_R(x) - \f'_{R}(x) |  
\leq 
C
\biggl( \frac{\X(x)}{R^{1-\ep}} \biggr)^{\alpha} 
\Bigl( 
\mathcal{H}_{R}(x)  + \big| \ahom - \b_{R}(x)\big|
\Bigr) 
+ 
C \indc_{\{ R^{1-\delta} \leq \X_{R}(x) \}}
\end{equation}
and
\begin{equation}  \label{e.fr.correctorloc.malliavin}
\bigl| \partial_{(\a,\f)(B_t(x))} \f'_{R}(x)\bigr| 
\leq 
C ( 1+  R^{-\alpha \delta} \log R ) \,.
\end{equation}
\end{lemma}

\begin{proof}
To begin with, fix~$R \geq 1$ and~$\ep,\delta \in (0,\frac12]$, and set~$r = R^{1-\ep}$ and~$t = 20 d R \log^{\nicefrac 12} R$. We let~$m$ be the largest integer such that~$\cu_m \subset B_R$ and let the following set be the approximation of~$B_{t/2}$:
\begin{equation*}  
U_t := \bigcup_{z \in 3^m \Z^d \cap B_{t/2}} (z + \cu_m)\,.
\end{equation*}
By~\eqref{e.fr.add} and boundedness of~$\f_r$, we have that
\begin{align*}  
\big| \f_R(0) - (\indc_{U_t} \f_r) \ast \Phi_{\sqrt{R^2-r^2}}(0) \big|   
= 
\O_\Psi^{1-\delta} \bigl( C r^{-\frac{d}{2}\frac{1-\beta}{1-\delta}} \bigr)\,.
\end{align*}
We look for the localization in the form
\begin{equation*}  
\f_R'(0) =  (\indc_{U_t} \f_r''(x)) \ast \Phi_{\sqrt{R^2-r^2}}(0)\,,
\end{equation*}
where~$\f_r''(x)$ is~$\F(x + \cu_{m+3})$-measurable, making~$\f_R'(0)$ clearly~$\F(B_t)$-measurable.

\smallskip

Next, we define the piecewise constant minimal scale as
\begin{equation*}  
\X_R(x) = H \sum_{z \in 3^m \Z^d \cap U_t} \sum_{y \in \Z^d \cap (z + \cu_m)} 
\indc_{y + \cu_0}(x) \Bigl( \X_{\f,m}(y) + \X_{m+1}(y) + \X_{\f,m+1}(z) \Bigr)\,.
\end{equation*}
Here~$H$ is a large constant to be fixed,~$\X_{m+1}(y)$ is the~$\Z^d$-stationary extension of the localized minimal scale from Theorem~\ref{t.Ck1.local}, and~$\X_{\f,m+1}(y)$,~$y \in \Z^d$,  is the~$\Z^d$-stationary extension of the localized minimal scale from Corollary~\ref{c.Jf.minsetE}. It is then~$\F(z + \cu_{m+2})$-measurable, and hence also~$B_{t}$-measurable. Moreover, its Malliavin derivative satisfies
\begin{equation*} 
|\partial_{(\a,\f)(z+\cu_{m+2})}\X_R(z)| \leq C(1 + \X_R(z)) \,.
\end{equation*}
The random minimal scale~$\X(x)$ is obtained similarly, but now by taking~$m = \infty$ in the above definition. We also have that
\begin{equation*}  
\X + \X_R 
= 
\O_\Psi^{\frac d2(1-\beta)}(C)
\,.
\end{equation*}
We define the approximation of the indicator function~$\indc_{\{ r^{1-\delta} \geq  \X_{R}(x)\}}$ as
\begin{equation}  \label{e.fr.localization.minscale}
\minscale_{r,R}(x) = \min\big\{1, \bigl(2 -    r^{\delta-1} \X_{R}(x) \bigr)_+ \big\} \,.
\end{equation}

\smallskip 

After translation, we may momentarily assume that~$z_k = 0$ and fix~$x \in \cu_m$.  We let~$u_{\f} = \psi_{\f,m+1}$, where~$\psi_{\f,m+1}$ is as in Theorem~\ref{t.zeroslope}, and extend it to the whole~$\R^d$ periodically. Letting~$\psi_{\f,y,n}$ stand for~$\Z^d$ stationary extension of~$\psi_{\f,n}$, we write  
\begin{equation*}  
u_{\f} 
= 
\psi_{\f,y,n} 
+ \sum_{k=n}^{m} (\psi_{\f,y,n+1}  - \psi_{\f,y,n}) 
+ \psi_{\f,m+1} - \psi_{\f,y,m+1}\,.
\end{equation*}
If~$R \geq \X_R(x)$, then, for every~$s \in [\X_R(x),R]$ and~$s \in [3^{n-1},3^{n})$, we obtain by Remark~\ref{r.zeroslope} and Corollary~\ref{c.Jf.minsetE} that 
\begin{equation}  \label{e.fr.localization.minscale.prop}
\frac1s \| u_{\f} - (u_{\f})_{B_s(x)}  \|_{\underline{L}^2(B_s(x))}  
\leq 
C 
\Bigl( \frac{\X_R(x)}{s}  \Bigr)^{\! \alpha}
\,.
\end{equation}

\smallskip

Next, for every~$x \in \cu_m$, we define the inner product
\begin{align}  \label{e.fr.innerproduct}
\langle v,w\rangle_x
:= 
r^{-2}
\fint_{ B_{r}(x) } (v \ast \Phi_{\sigma r}) (w \ast \Phi_{\sigma r})
\qand
E_x(w) = \| w\|_x^2 := \langle w,w\rangle_x
\,. 
\end{align}
Denote~$\A_{k,x}^{0}(\cu_{m+1}) = \{ v \in \A_k(\cu_{m+1}) \, : \, E_x(v) = 0 \}$. As in the proof of Lemma~\ref{l.correctorloc.new}, we let~$\{w_{j,x}\}_j$ be the orthonormal basis of~$\A_{k}(\cu_{m+1}) \backslash \A_{k,x}^0(\cu_{m+1})$ with respect to the inner product~$\langle \cdot , \cdot \rangle_x$ defined above. 

\smallskip

\emph{Step 1.} Construction of~$\psi_{\f,x}^{(r)}$.  Let~$w_{\f,x}$ be the unique minimizer of~$E_x(\cdot + u_{\f})$ over the quotient space~$\A_{k}(\cu_{m+1}) \backslash \A_{k,x}^0(\cu_{m+1})$.  Then~$w_{\f,x}$ is unique and~$\F(\cu_{m+1})$-measurable.  Set 
\begin{equation}  \label{e.fr.localcorrdef}
\psi_{\f,x}^{(r)}  : = \minscale_{r,R}(x)( w_{\f,x} + u_{\f}) 
\,,
\end{equation}
so that~$\psi_{\f,x}^{(r)}$ solves~$-\nabla \cdot \a \nabla \psi_{\f,x}^{(r)}  = \nabla \cdot \f$ in~$\cu_{m+1}$. Notice that~$\psi_{\f,x}^{(r)} = 0$ if~$r^{1-\delta} \leq 2\X_{R}(x)$. We set
\begin{equation*}  
\f_{r}''(x) 
:= 
\minscale_{r,R}(x) \Phi_r \ast  \bigl( \a \nabla \psi_{\f,x}^{(r)}  + \f - \ahom \nabla \psi_{\f,x}^{(r)} - \overline{\f}  \bigr) (x) \,.
\end{equation*}
Then
\begin{align*}  
\f_r(x) - \f_{r}''(x)
& = 
\minscale_{r}(x) \Phi_r \ast  \bigl( (\a - \ahom)( \nabla \psi_{\f}  -  \nabla \psi_{\f,x}^{(r)}  \bigr) (x) 
\\ & \quad 
+ 
(\minscale_{r}(x)-\minscale_{r,R}(x)) \Phi_r \ast  \bigl( \f - \overline{\f}  \bigr) (x)\,.
\end{align*}
In the next step, we estimate the first term on the right.

\smallskip

\emph{Step 2.} We prove that~$\psi_{\f,x}^{(r)}$ constructed in~\eqref{e.fr.localcorrdef} satisfies
\begin{equation}  \label{e.fr.correctorloc.pre}
\minscale_{r}(x) 
\bigl| 
\Phi_r \ast  \bigl( (\a - \ahom)( \nabla \psi_{\f}  -  \nabla \psi_{\f,x}^{(r)}  \bigr) (x) 
\bigr| 
\\ 
\leq
C \Bigl( \frac{\X(x)}{r} \Bigr)^{\! \alpha}  \mathcal{H}_r(x) 
+ C  \indc_{\{ r^{1-\delta} < \X_{R}(x) \}}
\,.
\end{equation}
Due to the presence of~$\minscale_{r}(x)$ above, we assume during the rest of the step that~$r \geq \X(x)$. 

\smallskip

Consider first the case~$r^{1-\delta} \geq \X_{R}(x)$. Then~$\psi_{\f,x}^{(r)} = w_{\f,x} + u_{\f}$.  We may use~\eqref{e.coarsened.fist} to get
\begin{multline}   \label{e.fr.correctorloc.pre0}
\minscale_{r}(x)  
\bigl| 
\Phi_r \ast  \bigl( (\a - \ahom)( \nabla \psi_{\f}  -  \nabla \psi_{\f,x}^{(r)}  \bigr) (x) 
\bigr| 
\\ 
\leq 
\inf_{\psi \in \A_1} 
\Bigl\{ 
\bigl| 
\Phi_r \ast  \bigl(  (\a - \ahom) \nabla (\psi_{\f}  + \psi -  \psi_{\f,x}^{(r)})  \bigr)(x) 
\bigr| 
+
\bigl|  (\ahom - \b_{r}'(x)) (\Phi_r \ast \nabla \psi)(x) \bigr| 
\Bigr\}
\,.
\end{multline}
Since~$\psi_{\f}  + \psi -  \psi_{\f,x}^{(r)} \in \A(\cu_{m+1})$,~\eqref{e.genfluxmaps.heatscale} yields that
\begin{align}  \label{e.fr.correctorloc.pre1}
\Bigl| 
\Phi_r \ast  \bigl(  (\a - \ahom) \nabla (\psi_{\f}  + \psi -  \psi_{\f,x}^{(r)})  \bigr)(x) 
\Bigr| 
& 
\leq 
C \Bigl( \frac{\X_R(x)}{r} \Bigr)^{\! \alpha} \big\| \nabla (\psi_{\f}  + \psi -  \psi_{\f,x}^{(r)} )\big\|_{L^2(\Phi_{x,2r})} 
\notag \\ 
&  
\quad 
+ 
C \exp\Bigl( - c r^{2\ep} \Bigr)\big\| \nabla (\psi_{\f}  + \psi -  \psi_{\f,x}^{(r)})\big\|_{\underline{L}^2(B_{R})} 
\,.
\end{align}
Our intermediate task is to estimate the first term on the right.  

\smallskip

We fix now~$\psi \in \A_1$ in~\eqref{e.fr.correctorloc.pre0} to be such that~$\fint_{B_r(x)} \Phi_r \ast (\psi  + \psi_{\f}) = 0$ and~$\nabla \psi  \ast \Phi_{s}(x) = - \nabla \psi_{\f} \ast \Phi_{s}(x)$, and consider~$v_{\f,x} = \psi  + \psi_{\f} - u_{\f,x}$. Then~$v_{\f,x} \in \A(\cu_{m+1})$.  We have, using the Caccioppoli estimate,~\eqref{e.MSP.heat.really},~\eqref{e.fr.localization.minscale.prop},  and~\eqref{e.f.optimalstochasticintegrability}, that
\begin{equation}  \label{e.fr.local.bnd}
\max_{\tau \in [\X(x) , 2R]}\big\| |\nabla w_{\f,x} | +  |\nabla \psi | + |\nabla \psi_{\f} - \nabla u_{\f,x}| \big\|_{\underline{L}^2(B_\tau(x) \cap \cu_{m+1})} 
\leq 
C 
\,.
\end{equation}
Notice that the above estimate continues to hold also in the case~$r^{1-\delta} < \X_{R}(x)$. 
By applying Theorem~\ref{t.Ck1.local}, 
we find~$\eta \in  \A_{k}(\cu_{m+1})$ such that, for every~$\tau  \in [r,2R]$, 
\begin{equation} \label{e.f.loc.corrclose2}
\frac1\tau \left\|  v_{\f,x} -  \eta \right\|_{\underline{L}^2 \left( B_{\tau}(x) \cap \cu_{m+1} \right)} 
+
\left\|  \nabla v_{\f,x} -   \nabla  \eta \right\|_{\underline{L}^2 \left( B_{\tau}(x) \cap \cu_{m+1} \right)} 
\leq 
C\left( \frac{\tau}{R}\right)^{\!k} .
\end{equation}
By minimality and the triangle inequality, we get, recalling that~$v_{\f,x} + u_{\f,x} = \psi  + \psi_{\f}$,
\begin{equation*}  
E_x(w_{\f}-\eta) 
= 
E_x(\eta + u_{\f}) - E_x(w_{\f} + u_{\f}) 
\leq 
2 E_x(\psi  + \psi_{\f}) +  2 E_x(\eta - v_{\f}) 
\,.
\end{equation*}
Since~$\fint_{B_r(x)} (\psi + \psi_{\f}) \ast \Phi_{\sigma r} = 0$ and~$\nabla (\psi  + \psi_{\f}) \ast \Phi_{r}(x) = 0$, the Poincar\'e inequality, the triangle inequality and~\eqref{e.fr.frtopsi} give us
\begin{equation*}  
E_x(\psi  + \psi_{\f}) 
\leq
C  \big\| \nabla (\psi  + \psi_{\f}) \ast \Phi_{\sigma r} \|_{L^2(\Phi_{x,(1-\sigma^2)^{\nicefrac12} r})}^2
\leq
C \mathcal{H}_r^2(x) \,.
\end{equation*}
Using~\eqref{e.f.loc.corrclose2}, on the other hand, we obtain 
\begin{equation*}  
E_x(\eta - v_{\f}) 
\leq 
C \Bigl( \frac st \Bigr)^k 
 \leq 
 C t^{-20d}
 \leq
C \mathcal{H}_r^2(x)
\,.
\end{equation*}
Since~$w_{\f} - \eta$ belongs to~$\A_k(\cu_{m+1})$,~\eqref{e.MSP.heat.really} implies that 
\begin{equation*}  
\big\| \nabla (w_{\f} - \eta)\big\|_{L^2(\Phi_{x,2r})} 
\leq 
C \mathcal{H}_r(x)
\,.
\end{equation*}
Appealing to the previous display,~\eqref{e.f.loc.corrclose2} and the triangle inequality, we get that 
\begin{equation*}  
\big\| \nabla (\psi_{\f} + \psi - \psi_{\f,x}^{(r)} )\big\|_{L^2(\Phi_{x,2r})}  
\leq 
C \mathcal{H}_r(x)
 \,.
\end{equation*}
This can be used to control the first term on the right in~\eqref{e.fr.correctorloc.pre1}. 
Moreover, by~\eqref{e.fr.local.bnd}, we also deduce that
\begin{equation*}  
 \exp\Bigl( - c r^{2\rho} \Bigr)\big\| \nabla (\psi_{\f}  + \psi -  \psi_{\f,t} )\big\|_{\underline{L}^2(B_{t/2})}  
 \leq 
 C t^{-20d} 
 \leq 
 C \Bigl( \frac{\X(x)}{r} \Bigr)^{\! \alpha}   \mathcal{H}_r(0) \,.
\end{equation*}
Hence, we conclude with
\begin{equation*}  
\Bigl| 
\Phi_r \ast  \bigl(  (\a - \ahom) \nabla (\psi_{\f}  + \psi -  \psi_{\f,x}^{(r)})  \bigr)(x) 
\Bigr| 
\leq 
 C \Bigl( \frac{\X(x)}{r} \Bigr)^{\! \alpha}   \mathcal{H}_r(0)
 \,.
\end{equation*}

We then estimate the last terms in~\eqref{e.fr.correctorloc.pre0}. Since~$r \geq \X(x)$, we have by Theorem~\ref{t.Ck1} that 
\begin{equation*}  
\big| \nabla  \psi \ast \Phi_{r}(x)  \big|  
=
\big| \nabla \psi_{\f} \ast \Phi_{r}(x) \big| 
\leq 
C \Bigl( \frac{\X(x)}{r} \Bigr)^{\! \alpha} 
\,.
\end{equation*}
Thus,
\begin{equation*}  
\Bigl|  (\ahom - \b_{r}(x)) (\Phi_r \ast \nabla \psi)(x) \Bigr|  \leq C | \ahom - \b_{r}(x) | \Bigl( \frac{\X(x)}{r} \Bigr)^{\! \alpha} \,.
\end{equation*}

Finally, we still need to consider the case~$r^{1-\delta} < \X_{R}(x)$. In this case have have the following bound due to~\eqref{e.fr.local.bnd} and Theorem~\ref{t.zeroslope}:
\begin{equation} 
\label{e.fr.correctorloc2}
\Bigl| 
\Phi_r \ast  \bigl(  (\a - \ahom) \nabla (\psi_{\f}  -  \psi_{\f,x}^{(r)})  \bigr)(x) 
\Bigr| \indc_{\{ r \geq \X(x)\} } \indc_{\{ r^{1-\delta} < \X_{R}(x) \}}
\leq 
C 
\indc_{\{ r^{1-\delta} < \X_{R}(x) \}}
\,.
\end{equation}
Putting pieces together yields~\eqref{e.fr.correctorloc.pre}.

\smallskip

\emph{Step 4.}
Bounds for the Malliavin derivatives. We take~$\a,\tilde \a$ and~$\f,\tilde \f$ such that
\begin{equation} \label{e.fr.Malliavin.ass}
\|  \a - \tilde \a \|_{L^\infty(B_t)}  +  \vertiii \f - \tilde \f \vertiii_{B_t}
\leq 
h 
\end{equation} 
for small~$h>0$.  Here we denote
\begin{equation*} 
\vertiii \f \vertiii_{B_t}:= 
\sup_{z\in 3^{k_t} \Zd \cap B_t}
\| \f \|_{\underline{L}^2(z+\cu_{k_t})} \,, \quad k_t := \lceil \log_3 t - (\log \log_3  t)^{\nicefrac12} \rceil\,.
\end{equation*}
Notice also the implications 
\begin{equation*}  
 \vertiii \f - \tilde \f \vertiii_{B_t} \leq  C (\log R)^{\nicefrac d2}  \| \f - \tilde \f \|_{\underline{L}^2(B_t)} 
\qand 
\| \f - \tilde \f \|_{L^2(\Phi_R)}  \leq C \vertiii \f - \tilde \f \vertiii_{B_t} \,,
\end{equation*}
so if we would compute the Malliavin derivative with respect to~$\underline{L}^2(B_t)$ norm, we would get a bound, but with a multiplicative prefactor~$C (\log R)^{\nicefrac d2} $. 

\smallskip

Let~$z \in 3^m \Z^d \cap U_t$, and let~$u_z = \psi_{\f,z,m+1}$ and~$\tilde u_z = \tilde \psi_{\tilde \f,z,m+1}$, with~$\tilde \psi_{\tilde \f,z,m+1}$ corresponding the pair~$(\tilde \a,\tilde \f)$. Let~$x\in z + \cu_{m+1}$, and suppress~$z$ from the notation. Following the computations in the proofs of Theorem~\ref{t.zeroslope} and Lemma~\ref{l.Malliavin.f.base}, we obtain that 
\begin{equation*}  
\| u  - \tilde u \|_{\underline{L}^2(z + \cu_{m+1})} 
\leq 
C \bigl( \| \f\|_{\underline{L}^2(z + \cu_{m+1})} \| \a - \tilde{\a} \|_{L^\infty(z + \cu_{m+1})}
+
\| \f - \tilde{\f} \|_{\underline{L}^2(z + \cu_{m+1})} 
\bigr).
\end{equation*}
Let the corresponding minimizer be~$w_{x}$ of~$E_x(\cdot + u)$ over~$\A_k(z + \cu_{m+1})$ and~$\tilde w_x$ of~$E_x(\cdot + \tilde u)$ over~$\tilde \A_k(z + \cu_{m+1})$, and define~$\phi_{\f,x}^{(r)}$ and~$\tilde \phi_{\f,x}^{(r)}$ as in~\eqref{e.fr.localcorrdef}.  
Then
\begin{align*}  
 \f_{r}''(x) - \tilde \f_{r}''(x) 
& = 
\chi_{r,R}(x) \Phi_r \ast  \Bigl( (\a - \tilde \a) \nabla \psi_{\f,x}^{(r)} + \f - \tilde \f  + \tilde \a \nabla (\psi_{\f,x}^{(r)} - \tilde \psi_{\f,x}^{(r)}) \Bigr) (x) 
\\ &
\quad + 
(\chi_{r,R}(x) - \tilde \chi_{r,R}(x))  \Phi_r \ast  \Bigl( (\tilde \a - \ahom) \nabla \tilde \psi_{\f,x}^{(r)}  + \tilde \f - \overline{\f}  \Bigr) (x) 
\,.
\end{align*}
We see from here that 
\begin{equation*}  
\big| \f_{r}''(x) - \tilde \f_{r}''(x)  \big| 
\leq 
C h 
+ \| \f - \tilde \f \|_{L^2(\Phi_{x,r})} 
+ 
\chi_{r}(x)\| \nabla \phi_{\f,x}^{(r)} - \nabla \tilde \phi_{\f,x}^{(r)} \|_{L^2(\Phi_{x,r})}\,.
\end{equation*}
Here we also used the combination of~\eqref{e.fr.local.bnd},~\eqref{e.fr.Malliavin.ass} and Theorem~\ref{t.Ck1.local} and Corollary~\ref{c.Jf.minsetE}, which give us 
\begin{equation*}  
\big| \chi_{r,R}(x) - \tilde \chi_{r,R}(x) \big| \leq C h \,.
\end{equation*}
To estimate the last term, we observe that
\begin{align*}  
\chi_{r}(x) \| \nabla \phi_{\f,x}^{(r)} - \nabla \tilde \phi_{\f,x}^{(r)} \|_{L^2(\Phi_{x,r})}
&
\leq
\chi_{r}(x)\| \nabla (w_x - \tilde w_x)\|_{L^2(\Phi_{x,r})}
+
\| \nabla (u - \tilde u)\|_{L^2(\Phi_{x,r})}
\,.
\end{align*}
Recall that in the proof of Lemma~\ref{l.correctorloc.new} we found the orthonormal basis~$\{w_{j,x}\}$ and~$\{\tilde w_{j,x}\}$ for~$\A_k(\cu_{m+1})$ and~$\tilde \A_k(\cu_{m+1})$, respectively, such that 
\begin{equation*}  
\| w_{j,x}  \|_x = \| \tilde w_{j,x}  \|_x = 1
\qand 
\| w_{j,x} - \tilde w_{j,x} \|_x 
+  \| \nabla w_{j,x} - \nabla \tilde w_{j,x}\|_{L^2(\Phi_{x,2r})} 
\leq 
C h \log R
\,.
\end{equation*}
We have, up to~$o(h)$-error in~$L^2(B_R)$-norm, that 
\begin{align*}  
\tilde w_{x} - w_x 
= 
\sum_{j} \Bigl(\langle w_{j,x} , u \rangle_x  (w_{j,x} {-} \tilde w_{j,x}) 
+  \langle w_{j,x} {-}  \tilde w_{j,x} , u \rangle_x  w_{j,x}   
+  \langle w_{j,x} , u {-}  \tilde u \rangle_x w_{j,x} \Bigr)
+ o(h)
\,.
\end{align*}
By~\eqref{e.fr.localization.minscale.prop}, we obtain, under the event~$r^{1-\delta} \geq \frac14 \X_{R}(x)$, that
\begin{equation*}  
| \langle w_{j,x} , u \rangle_x| + | \langle \tilde w_{j,x} , u \rangle_x|  
\leq  
2 \inf_{a \in \R}\| u - a \|_x 
\leq 
C \Bigl( \frac{\X_{R}(x)}{r} \Bigr)^{\! \alpha}  
\leq
C r^{-\alpha \delta}
\,.
\end{equation*}
Therefore,
\begin{equation*}  
\chi_{r}(x)\| \nabla (w_x - \tilde w_x)\|_{L^2(\Phi_{x,r})}
\leq
C r^{-\alpha \delta} \log R  + \| u {-}  \tilde u \|_x + o(h)\,.
\end{equation*}
Next, we integrate
\begin{multline*}  
\biggl| \int_{z + \cu_m} ( \f_{r}''(y) - \tilde \f_{r}''(y)) \Phi_{\sqrt{R^2-r^2}}(x-y) \, dy  \biggr|
\\
\leq 
\biggl| \int_{z + \cu_m} \bigl(Ch +  \| \nabla (u - \tilde u)\|_{L^2(\Phi_{y,r})} + \| u {-}  \tilde u \|_y \bigr)
\Phi_{\sqrt{R^2-r^2}}(x-y)  \, dy
\biggr|
\,.
\end{multline*}
Using H\"older's inequality and the fact that~$c \Phi_{\sqrt{R^2-r^2}}(x)  \leq \Phi_{\sqrt{R^2-r^2}}(z)  \leq C \Phi_{\sqrt{R^2-r^2}}(x)$ for every~$x \in z + \cu_{m+1}$, we obtain
\begin{align} \notag  
\lefteqn{
\biggl| \int_{z + \cu_m} \bigl( \| \nabla (u - \tilde u)\|_{L^2(\Phi_{y,r})} + \| u {-}  \tilde u \|_y \bigr)
\Phi_{\sqrt{R^2-r^2}}(x-y)  \, dy
} \qquad &
\\ 
\notag &
\leq 
C 
\Phi_{\sqrt{R^2-r^2}}(x-z)
\sum_{k=0}^\infty \exp(-c 4^k) \int_{z + \cu_m} 
\| \nabla (u - \tilde u)\|_{L^2(B_{2^k r}(y))} \, dy  
\\ 
\notag &
\leq 
C 
\Phi_{\sqrt{R^2-r^2}}(x-z) \| \nabla (u - \tilde u)\|_{L^2(z+\cu_{m+1})} 
\\ 
\notag &
\leq 
\biggl( \int_{z + \cu_{m+1}} ( t |\f(y)|^2 +   |\f(y) - \tilde \f(y)|^2) \Phi_{\sqrt{R^2-r^2}}(x-y) \, dy \biggr)^{\!\!\nicefrac12} 
\\ 
\notag & \qquad \times 
\biggl( \int_{z + \cu_m} \Phi_{\sqrt{R^2-r^2}}(x) \, dy
\biggr)^{\!\!\nicefrac12} 
\,.
\end{align}
Therefore, by H\"older's inequality, we get that by taking smaller~$\alpha$, if necessary, that
\begin{equation*}  
\sum_{z \in 3^m \Z^d \cap U_t} \biggl| \int_{z + \cu_m} ( \f_{r}''(y) - \tilde \f_{r}''(y)) \Phi_{\sqrt{R^2-r^2}}(x-y) \, dy  \biggr|
\leq 
C(1 + R^{-\alpha \delta} \log R)  h +  \|\f - \tilde \f \|_{L^2(\Phi_R)} 
\,.
\end{equation*}
This concludes the proof. 
\end{proof}

By repeating the proofs of Propositions~\ref{p.psipsi.coarse} and~\ref{p.psipsi.coarse.optimal}, we obtain the following two propositions for the case with a divergence-form right-hand side.

\begin{proposition}
\label{p.fr.coarse.positivebeta}
Assume~$\beta>0$. 
There exists a constant~$C(\delta,p,M,N,\datareff) < \infty$ such that, for every~$r \geq 1$, 
\begin{equation}
\label{e.fr.coarse.positivebeta}
\f_r(x) - \overline{\f} 
=
\O_{\Psi} \bigl( C r^{-\frac d2(1-\beta)} \bigr) 
+ 
\O_{\Psi}^{1-\delta} \bigl( C r^{-\frac d2(1-\beta)(1+\delta)} \bigr)
\,.
\end{equation}
\end{proposition}

\begin{proof}
The proof is almost a verbatim repetition of the proof of Proposition~\ref{p.psipsi.coarse} after changing notation and using additivity estimate~\eqref{e.fr.additivity}, Lemma~\ref{l.fr.correctorloc} together with the assumed mixing condition.

\smallskip

The only slightly different technical point is in the norm we use for the sensitivity of~$\f$ in the definition of the Malliavin derivative. We can use the weaker norm and corresponding bound with a logarithmic prefactor as in~\eqref{e.fr.correctorloc.malliavin} to obtain that
\begin{equation*}  
\f_R(x) - \overline{\f} 
=
\O_{\Psi} \bigl( C R^{-\frac d2(1-\beta)}   (\log R)^{\nicefrac d2}  \bigr) 
+ 
\O_{\Psi}^{1-\delta} \bigl( C R^{-\frac d2(1-\beta)(1+\delta)} \bigr).
\end{equation*}
On the other hand, if we have the stronger norm to control the Malliavin derivative, we can get on top of the exponent and obtain~\eqref{e.fr.coarse.positivebeta}. We omit the straightforward modifications in the proof of Proposition~\ref{p.psipsi.coarse}. 
\end{proof}

We also obtain the result for~$\beta=0$. We will allow~$\gamma=0$ in the statement since the proof gives an estimate for it. 

\begin{proposition}
\label{p.fr.coarse.zerobeta}
Assume~$\beta=0$. Then there exists~$C(\delta,\gamma,p,M,N,\datareff) < \infty$ such that, for every~$r \geq 1$, 
\begin{equation}
\label{e.fr.coarse.zerobeta}
\f_r(x) - \overline{\f} 
=
\left\{
\begin{aligned}
& \O_{\Psi}^{1-\delta} \bigl(C r^{-\nicefrac d2}\bigr) & & \mbox{ if } \gamma>0
\, ,  
\\
& 
 \O_{\Psi}^{1-\delta} \bigl(C r^{-\nicefrac d2} \log r\bigr) & & \mbox{ if } \gamma = 0
\, .
\end{aligned}
\right.
\end{equation}
\end{proposition}

The proof of Proposition~\ref{p.fr.coarse.zerobeta} is omitted. 

\smallskip

We are finally ready to give the proof of Theorem~\ref{t.f.optimal}.

\begin{proof}[Proof of Theorem~\ref{t.f.optimal}]
Appealing to~\eqref{e.fr.frtopsi} and Propositions~\ref{p.fr.coarse.positivebeta} and~\ref{p.fr.coarse.zerobeta}, we find that
\begin{equation*}  
\left\| \Phi_r {\ast}  \nabla \psi_{\f}  - \Phi_R {\ast} \nabla \psi_{\f}  \right\|_{L^2 ( \Phi_{\sqrt{R^2-r^2}})}
= \O_{\Psi}^{1-\delta}(C r^{-\frac d2(1-\beta)}) + \O_{\Psi}^{1-\beta}(C r^{-\nicefrac d2})
\end{equation*}
It follows that 
\begin{equation*}  
\bigl| \Phi_r {\ast}  \nabla \psi_{\f}   - \Phi_{2r} {\ast}  \nabla \psi_{\f}  \bigr|
= 
\O_{\Psi}^{1-\delta}(C r^{-\frac d2(1-\beta)}) + \O_{\Psi}^{1-\beta}(C r^{-\nicefrac d2})\,.
\end{equation*}
Thus, summing over the scales and interpolating then yields that 
\begin{equation}  \label{e.f.corr.optimal}
\minscale_r(x)  \bigl| \Phi_r {\ast}  \nabla \psi_{\f}   \bigr| 
= 
\O_{\Psi}^{1-\delta}(C r^{-\frac d2(1-\beta)}) 
\,.
\end{equation}
By the definition of~$\f_r$ in~\eqref{e.fr}, we see that 
\begin{equation*}  
\minscale_r(x) \bigl( \Phi_r \ast (\a \nabla \psi_{\f} + \f - \overline{\f})\bigr) (x)
=
\f_r(x) -  \overline{\f} + \minscale_r(x) \, \ahom (\Phi_r \ast  \nabla \psi_{\f})(x)
\,,
\end{equation*}
so~\eqref{e.f.optimal} follows for~$r\geq \X(x)$ by~\eqref{e.f.corr.optimal} and Propositions~\ref{p.fr.coarse.positivebeta} and~\ref{p.fr.coarse.zerobeta}. Finally, for~$\beta = 0$ we see that we can remove the condition~$r\geq \X(x)$ by giving up the volume factor, since~$\X(x)^{\nicefrac d2} = \O_\Psi(C)$ in this case. The proof is complete. 
\end{proof}

\subsection*{Historical remarks and further reading}

The first optimal quantitative estimates in stochastic homogenization (in dimension at least two) were proved by Naddaf and Spencer~\cite{NS2}.
They observed that very strong gradient bounds on the first-order correctors imply quantitative bounds with optimal scaling, following the argument we have presented in Section~\ref{s.LSI}. Their setting was that of the discrete lattice~$\Zd$ with~i.i.d.~conductances on the nearest-neighbor edges, and their argument was based on a spectral gap inequality. They essentially circumvented the need for large-scale~$C^{0,1}$ estimate by assuming that the gradient bounds they needed were satisfied by working in the regime of small ellipticity contrast: that is,~$\Lambda\leq \lambda (1+\delta)$ for a sufficiently small~$\delta>0$, which guarantee~$L^p$--type gradient bounds for solutions with large~$p$ by Calder\'on-Zygmund theory. 

\smallskip

A decade later, this work served as inspiration for Gloria, Otto, and collaborators, who used it to develop a pretty complete theory of stochastic homogenization under~$\LSI$ and~$\SG$ conditions. Their early papers were formalized in the setting of the discrete lattice~\cite{GO1,GO2,GNO,MO}, with extensions to the continuum coming later~\cite{GO3}. These works pre-dated~\cite{AS} and, therefore, the development of the large-scale regularity theory. Still, they managed to obtain~$L^p$-type estimates on the gradients of the correctors via the concentration inequalities, in a kind of proto-regularity theory.  

\smallskip

The paper~\cite{AS} was the first to propose the strategy followed in Section~\ref{ss.correctors.optimal}: namely, that we first obtain a large-scale~$C^{0,1}$-type estimate and then quickly thereafter apply concentration inequalities to the first-order correctors directly to obtain optimal estimates under~$\LSI$ and~$\SG$-type conditions. This idea was subsequently formalized in the earlier arXiv versions of~\cite{GNO3} before being eventually published in~\cite{GNO4}. 

\smallskip

The coarse-grained coefficient field~$\b_r$ defined in this chapter is a variant of the ones appearing in Chapter~\ref{s.subadd}, which were central to the analysis of~\cite{AS,AM,AKM1}. The use of the heat kernel for the coarse-graining and in the definition of the coefficients we call~$\b_r$ first appeared in~\cite{AKM}, and the renormalization argument presented in this chapter is a simplification of the original one in that paper and of our first rewriting of the argument in~\cite[Chapter 4]{AKMBook}.
After~\cite{AKM}, the coarse-grained coefficients were subsequently re-introduced in~\cite{DGO} under the name ``homogenization commutator,'' which became the basis of a series of works by these authors and their collaborators~\cite{DGO2,DFG} for analyzing equations satisfying~$\LSI$ or~$\SG$-type ergodic assumptions. Gloria and Otto~\cite{GO6} later obtained similar results to the ones in this chapter under the assumption of a finite range of dependence (in a substantially revised version of their earlier arXiv posting, which had obtained sub-optimal estimates but was posted before~\cite{AKM}). 
As mentioned previously, most of these works are sub-optimal in terms of stochastic integrability, even though they are optimal in terms of the scaling of the fluctuations. Prior to the present work, estimates optimal in both stochastic integrability and in the size of the fluctuations were only known in the case of finite range dependence~\cite{AKM,GO6} (see also~\cite{Fischer}). 

\smallskip

A topic we have not touched on in this manuscript is that of the scaling limit of the first-order correctors to a variant of the Gaussian free field. We refer to~\cite{AKM,DGO} for results in this direction.

\renewcommand{\sectiontitle}{References}
\renewcommand{\subsectiontitle}{References}
{\small
\bibliographystyle{alpha}
\bibliography{homogenization}
}

\end{document}